\documentclass{amsbook}
\usepackage{amsmath}
\usepackage{amsthm}
\usepackage{amssymb}
\usepackage{makecell} %%% for larger cells in tables
\usepackage{appendix}
\usepackage[german,english]{babel}
\usepackage[utf8]{inputenc}
\usepackage[new]{old-arrows}
\usepackage{epsfig,fancybox,color}
\usepackage{setspace}
\usepackage[linesnumbered,ruled]{algorithm2e}
\usepackage{tikz}
\usepackage{tikz-cd}
\usetikzlibrary{cd}
\usepackage[toc,page]{}
\usepackage{adjustbox}
\usepackage{todonotes}
\usepackage{yhmath}
\usepackage{mathtools}
\usepackage{stmaryrd}
\usepackage{fancyvrb} \VerbatimFootnotes
\usepackage{hyperref}
\usepackage{wasysym} %%% for card game symbols
\definecolor{darkred}{RGB}{160,0,0}
\definecolor{darkblue}{RGB}{0,0,160}
\hypersetup{
  colorlinks=true,
  citecolor=darkgreen,
  citecolor=darkblue,
  filecolor=black,
  linkcolor=darkblue,
  urlcolor=darkblue
}

\usepackage[margin=1.25in]{geometry} %%% change page-margin

\tikzcdset{scale cd/.style={every label/.append style={scale=#1},
    cells={nodes={scale=#1}}}}

\DeclareMathAlphabet\mathbfcal{OMS}{cmsy}{b}{n} % cal and bf at the same time

\DeclareFontFamily{U}{mathx}{\hyphenchar\font45}
\DeclareFontShape{U}{mathx}{m}{n}{
      <5> <6> <7> <8> <9> <10>
      <10.95> <12> <14.4> <17.28> <20.74> <24.88>
      mathx10
      }{}
\DeclareSymbolFont{mathx}{U}{mathx}{m}{n}
\DeclareMathSymbol{\bigtimes}{1}{mathx}{"91}

\newcommand{\excise}[1]{}%{$\star$\textsc{#1}$\star$}

\theoremstyle{plain}
\newtheorem{thm}{Theorem}[section]

\newtheorem{lem}[thm]{Lemma}

\newtheorem{cor}[thm]{Corollary}
\newtheorem{prop}[thm]{Proposition}

\newtheorem{introthm}{Theorem}[chapter]

\theoremstyle{definition}
\newtheorem{defn}[thm]{Definition}

\newtheorem{example}[thm]{Example}

\newtheorem{condition}[thm]{Condition}

\theoremstyle{remark}
\newtheorem{claim}[thm]{Claim}
\newtheorem{remark}[thm]{Remark}
\newtheorem{conv}[thm]{Convention}
\newtheorem{notation}[thm]{Notation}
\newtheorem{examplenotation}[thm]{Example}

\numberwithin{equation}{section}

\renewcommand{\>}{\rangle}
\newcommand{\<}{\langle}
\newcommand{\CC}{\mathbb{C}}
\newcommand{\NN}{\mathbb{N}}
\newcommand{\QQ}{\mathbb{Q}}
\newcommand{\RR}{\mathbb{R}}
\newcommand{\FF}{\mathbb{F}}

\newcommand{\A}{\mathbb{A}}

\renewcommand\iff{\Leftrightarrow}
\renewcommand\implies{\Rightarrow}

\newcommand{\maps}{\rightarrow}
\newcommand\set[1]{\{#1\}}
\DeclareMathOperator{\SNrm}{\mathbf{SNrm}} % Seminormed spaces
\DeclareMathOperator{\Nrm}{\mathbf{Nrm}} % Normed spaces
\DeclareMathOperator{\Ban}{\mathbf{Ban}} % Banach spaces
\DeclareMathOperator{\CBorn}{\mathbf{CBorn}} % CBorn and...
\DeclareMathOperator{\E}{\mathbf{E}} % Some quasi-abelian category
\DeclareMathOperator{\F}{\mathbf{F}} % Some quasi-abelian category or some functor

\DeclareMathOperator{\LH}{\mathbf{LH}} % left heart

\DeclareMathOperator{\TT}{\mathbb{T}} % Torus

\DeclareMathOperator{\FarguesFontaine}{FF} % distance

\DeclareMathOperator{\act}{a} % action of a ring on a module, \a already defined

\DeclareMathOperator{\sh}{sh} % morphism from a presheaf in its associated sheaf

\DeclareMathOperator{\swap}{swap} % swap entries of mult in sym-monoid-cat.

\DeclareMathOperator{\cone}{cone} % homotopy category

\DeclareMathOperator{\D}{\mathbf{D}} % derived category
\DeclareMathOperator{\Ho}{H} % Homology and Cohomology, \H already defined
\DeclareMathOperator{\LHo}{LH} % Homology and Cohomology from the left t-structure

\DeclareMathOperator{\IndBan}{\mathbf{IndBan}} %
\DeclareMathOperator{\Pro}{\mathbf{Pro}} % Pro-completion

\DeclareMathOperator{\id}{id} % identity map

\DeclareMathOperator{\Sp}{Sp} % Spektrum of some affinoind variety
\DeclareMathOperator{\Hom}{Hom} % Hom
\DeclareMathOperator{\intHom}{\underline{\Hom}} % internal Hom
\DeclareMathOperator{\shHom}{\underline{\mathcal{H}\kern -.5pt\textit{om}}} % internal sheaf Hom
\DeclareMathOperator{\RshHom}{R\underline{\mathcal{H}\kern -.5pt\textit{om}}} % derived internal sheaf Hom
\DeclareMathOperator{\shExt}{\underline{\mathcal{E}\kern -1,75pt\textit{xt}}} % internal sheaf Ext
\DeclareMathOperator{\Psh}{Psh} % The category of presheaves

\DeclareMathOperator{\psh}{psh} % This should be used for all indices!

\DeclareMathOperator{\Sh}{Sh} % The category of sheaves
\DeclareMathOperator{\Loc}{Loc} % localisation of a category

\DeclareMathOperator{\Mod}{\mathbf{Mod}} % category of modules

\DeclareMathOperator{\sep}{sep} % seperated bornological vector space

\DeclareMathOperator{\im}{im} % image, \ker is already defined
\DeclareMathOperator{\coker}{coker} % cokernel
\DeclareMathOperator{\coim}{coim} % coimage

\DeclareMathOperator{\p}{p} % morphism that sends a valuation to its prime filter
\DeclareMathOperator{\I}{\mathbf{I}} % From E to LH(E)...
\DeclareMathOperator{\C}{\mathbf{C}} % ...and back.

\DeclareMathOperator{\op}{op} % opposite category

\DeclareMathOperator{\Dcap}{\mathcal{\wideparen{D}}} % D-cap
\DeclareMathOperator{\Dcapbd}{\mathcal{\wideparen{D}}\kern 1,00pt^{b}} % D-cap as a bornological space

\DeclareMathOperator{\hO}{\cal{\widehat{O}}} % complete structure sheaf on X_proet

\DeclareMathOperator{\Spa}{Spa} % Spa
\DeclareMathOperator{\dR}{dR} % B_{dR}
\DeclareMathOperator{\pdR}{pdR} % B_{pdR}
\DeclareMathOperator{\proet}{\text{pro\'{e}t}} %
\DeclareMathOperator{\profet}{\text{prof\'{e}t}} %
\DeclareMathOperator{\affperfd}{\text{affperfd}} % affinoid perfectoid
\DeclareMathOperator{\fin}{\text{fin}} % finite
\DeclareMathOperator{\et}{\text{\'{e}t}} % étale
\DeclareMathOperator{\fet}{f\text{\'{e}t}} % étale
\DeclareMathOperator{\BB}{\mathbb{B}}
\DeclareMathOperator{\OB}{\cal{O}\kern -1,00pt\mathbb{B}} % OB_{dR}
\DeclareMathOperator{\normalOB}{O\kern -1,00pt B} % OB_{dR}
\DeclareMathOperator{\normalOA}{O\kern -1,00pt A} % OB_{dR}
\DeclareMathOperator{\normalOtildeA}{O\kern -1,00pt {\tilde{A}}} % OB_{dR}
\DeclareMathOperator{\OA}{\cal{O}\kern -1,00pt\mathbb{A}} % OA_{\dR}
\DeclareMathOperator{\OBcap}{\wideparen{\cal{O}\kern -1,00pt\mathbb{B}}} % OB_{dR}
\DeclareMathOperator{\OBdR}{\cal{O}\kern -1,00pt\mathbb{B}_{\dR}} % OB_{dR}
\DeclareMathOperator{\OXBdR}{\cal{O}_{X}\kern -1,00pt\mathbb{B}_{\dR}} % OB_{dR}
\DeclareMathOperator{\OXxXBdR}{\cal{O}_{X\times X}\kern -1,00pt\mathbb{B}_{\dR}} % OB_{dR}
\DeclareMathOperator{\OC}{\cal{O}\kern -0,90pt\mathbb{C}} %
\DeclareMathOperator{\CB}{\cal{C}\kern -1,00pt\mathbb{B}} % OB_{dR}

\DeclareMathOperator{\Otheta}{\cal{O}\kern -1,00pt\theta} % OB_{dR}
\DeclareMathOperator{\Ovartheta}{\cal{O}\kern -1,00pt\vartheta} % OB_{dR}
\DeclareMathOperator{\Oiota}{\cal{O}\kern -1,00pt\iota} % OB_{dR}

\DeclareMathOperator{\TB}{\mathcal{T}\kern -1,00pt\mathbb{B}} % TB_{dR}
\DeclareMathOperator{\OmegaB}{\Omega\kern -1,00pt\mathbb{B}} % TB_{dR}

\DeclareMathOperator{\DcapB}{\Dcap\kern -1,00pt\BB}
\DeclareMathOperator{\RB}{R\kern -1,00pt\BB}
\DeclareMathOperator{\PB}{P\kern -1,00pt\BB}
\DeclareMathOperator{\rhoB}{\rho\kern -1,00pt\BB}

\DeclareMathOperator{\SB}{S\kern -1,50pt\BB} % Spencer resolution

\DeclareMathOperator{\dRB}{\dR\kern -1,50pt\BB}
\DeclareMathOperator{\DDB}{\DD\kern -1,50pt\BB}

\DeclareMathOperator{\SolB}{\Sol\kern -0.5pt\BB} % sol functor
\DeclareMathOperator{\dRfunctorB}{\dRfunctor\kern -1.5pt\BB} % dr functor

\DeclareMathOperator{\Conn}{Conn} % local systems
\DeclareMathOperator{\LocB}{\Loc\kern -1,50pt\BB} % local systems
\DeclareMathOperator{\ConnOB}{\Conn\kern -1,50pt\OB} % integrable connections

\DeclareMathOperator{\Mat}{Mat} % Matrices

\DeclareMathOperator{\DB}{\mathcal{D}\kern -1,00pt\mathbb{B}} % DcapB_{dR}
\DeclareMathOperator{\EB}{\mathcal{E}\kern -0.40pt\mathbb{B}} % EB_{dR}

\DeclareMathOperator{\Fil}{Fil} % Fil

\DeclareMathOperator{\bd}{b} % bounded derived category
\DeclareMathOperator{\Sol}{\mathcal{S}\kern -1.0pt\textit{ol}} % Solution functor
\DeclareMathOperator{\nSolB}{Sol\kern -0,5pt\mathbb{B}} % Solution functor
\DeclareMathOperator{\Rec}{\mathcal{R}\kern -1.0pt\textit{ec}} % Solution functor
\DeclareMathOperator{\R}{R} % the derived version of \rho
\DeclareMathOperator{\solid}{solid} % solid vector spaces

\DeclareMathOperator{\gr}{gr} % associated graded
\newcommand{\heart}{\ensuremath\heartsuit}

\DeclareMathOperator{\DD}{\mathbb{D}} % unit disc

\DeclareMathOperator{\cont}{cts} % continuous maps
\DeclareMathOperator{\clcont}{clcts} % classical continuous cohomology

\DeclareMathOperator{\Vect}{\mathbf{Vec}} % category of vector spaces, \Vec already defined

\DeclareMathOperator{\Ind}{\mathbf{Ind}} % Ind-categories

\newcommand{\cal}[1]{\mathcal{#1}}

\DeclareMathOperator{\cond}{cond} % condensed

\DeclareMathOperator{\cts}{cts} % continuous

\DeclareMathOperator{\ZZ}{\mathbb{Z}} % The integers

\DeclareMathOperator{\Gal}{Gal} % Galois group

\DeclareMathOperator{\al}{al} % almost zero

\DeclareMathOperator{\linearspan}{span} % linear span

\DeclareMathOperator{\GL}{GL} % general linear

\DeclareMathOperator{\rL}{L} % \L already defined
\DeclareMathOperator{\rC}{C} % regular C

\DeclareMathOperator{\aug}{aug} % augmented \v{C}ech complex

\DeclareMathOperator{\limit}{lim} % homotopy limit

\DeclareMathOperator{\bA}{\mathbf{A}} % bold A

\DeclareMathOperator{\isomap}{\stackrel{\cong}{\longrightarrow}} % isomorphism

\DeclareMathOperator{\dRfunctor}{\textit{d}\kern +0.5pt\mathcal{R}} % Solution functor

\DeclareMathOperator{\Kos}{Kos} % Koszul complexes

\DeclareMathOperator{\lc}{lc} % locally constant

\DeclareMathOperator{\lh}{^{\small\heart}\kern -3.00pt}

\DeclareMathOperator{\pfsets}{pfsets}

\DeclareMathOperator{\triv}{triv} %trivial norm

\DeclareMathOperator{\dist}{dist} % distance

\DeclareMathOperator{\Prof}{Prof} % profinite

\begin{document}

\selectlanguage{english}

\title{Galois and Pro-\'etale Cohomology of Overconvergent de Rham Period Rings}

\author{Finn Wiersig}
\address{National University of Singapore}
\email{fwiersig@nus.edu.sg}
\date{\today}

\begin{abstract}
  Motivated by the theory of $p$-adic differential equations and
  $p$-adic geometric representation theory, we introduce overconvergent
  variants of Fontaine's classical period rings. In particular,
  we study the positive overconvergent de Rham period ring
  $B_{\dR}^{\dag,+}$, which is the stalk of
  the structure sheaf of the analytic Fargues--Fontaine curve at
  infinity. Our main results include the computation of the
  Galois cohomology of these overconvergent period rings, as well
  as the cohomology of the associated period sheaves
  and period structure sheaves.
\end{abstract}

\maketitle

%\newpage

\tableofcontents

%\newpage

%%%%%%%%%%%%%%%%%%%%%%%%%%%%%%%%%%%%%%%%%%%%%%%%%%%%%%%%%%%%
%%%%%%%%%%%%%%%%%%%%%%%%%%%%%%%%%%%%%%%%%%%%%%%%%%%%%%%%%%%%
%%%%%%%%%%%%%%%%%%%%%%%%%%%%%%%%%%%%%%%%%%%%%%%%%%%%%%%%%%%%
% Introduction
%%%%%%%%%%%%%%%%%%%%%%%%%%%%%%%%%%%%%%%%%%%%%%%%%%%%%%%%%%%%
%%%%%%%%%%%%%%%%%%%%%%%%%%%%%%%%%%%%%%%%%%%%%%%%%%%%%%%%%%%%
%%%%%%%%%%%%%%%%%%%%%%%%%%%%%%%%%%%%%%%%%%%%%%%%%%%%%%%%%%%%

\chapter{Introduction}
\label{ch:intro}

%%%%%%%%%%%%%%%%%%%%%%%%%%%%%%%%%%%%%%%%%%%%%%%%%%%%%%%%%%%%
%%%%%%%%%%%%%%%%%%%%%%%%%%%%%%%%%%%%%%%%%%%%%%%%%%%%%%%%%%%%
% Introduction
%%%%%%%%%%%%%%%%%%%%%%%%%%%%%%%%%%%%%%%%%%%%%%%%%%%%%%%%%%%%
%%%%%%%%%%%%%%%%%%%%%%%%%%%%%%%%%%%%%%%%%%%%%%%%%%%%%%%%%%%%

\section{Motivation}

%%%%%%%%%%%%%%%%%%%%%%%%%%%%%%%%%%%%%%%%%%%%%%%%%%%%%%%%%%%%
%%%%%%%%%%%%%%%%%%%%%%%%%%%%%%%%%%%%%%%%%%%%%%%%%%%%%%%%%%%%
% Introduction
%%%%%%%%%%%%%%%%%%%%%%%%%%%%%%%%%%%%%%%%%%%%%%%%%%%%%%%%%%%%
%%%%%%%%%%%%%%%%%%%%%%%%%%%%%%%%%%%%%%%%%%%%%%%%%%%%%%%%%%%%

\subsection{Riemann-Hilbert correspondences}

A \emph{Riemann-Hilbert correspondence} relates analytic data, such as
systems of differential equations, to topological data, such as monodromy
representations. In the complex analytic setting this philosophy goes back to
Hilbert's twenty-first problem. It was developed in depth by
Deligne~\cite{Equationsdifferentiellesapointssinguliersreguliers},
Kashiwara~\cite{KashiwaraKawaiholonomicIII,Kashiwaraconstructibilityholonomic,KashiwaraRHforholonomicsystems}.
and Mebkhout~\cite{Mebkhoutuneequivalence}.%; see also~\cite{MR1276272}.

\begin{example}\label{ex:mon-periods}
Let $X=\DD^{\times}=\left\{x\in\CC\colon 0<|x|<1\right\}$ be the punctured
open disc with coordinate $z$, and consider the differential equation
$zf^{\prime}=\lambda f$
for some $\lambda\in\CC$. The local solutions are of the form $cz^{\lambda}$.
If $\lambda\notin\ZZ$, then $z^{\lambda}$ is multivalued on $X$, and the
equation has nontrivial monodromy. After choosing a base point $x_{0}\in X$,
this monodromy is described by a representation
\[
  \rho_{\lambda}\colon
  \pi_{1}\left(X,x_{0}\right)\to\GL_{1}\left(\CC\right),
\]
where a generator of $\pi_{1}(X,x_{0})$ is sent to
$e^{2\pi i\lambda}\in\CC^{\times}$.
\end{example}

Now let $X$ be a smooth rigid-analytic variety over a complete discretely
valued field $k$ of mixed characteristic $(0,p)$ with perfect residue field.
There are several versions of a $p$-adic Riemann-Hilbert correspondence,
depending on the class of analytic objects and the class of topological objects
one wants to relate. Important contributions include
Scholze~\cite[Theorem 7.6]{Sch13pAdicHodge},
Liu--Zhu~\cite{LiuZhuRH2017}, and
Diao--Lan--Liu--Zhu~\cite{MR4536903}.

The direction relevant for this work is the following. We would like to study
analytic objects on $X$ by relating them to suitable topological objects through
a $p$-adic Riemann-Hilbert correspondence. For this purpose, Fontaine's
classical de Rham period ring $B_{\dR}$ and Scholze's corresponding period
sheaf $\OB_{\dR}$ are too large: they contain the correct completed period
information, but they do not retain the convergence conditions needed to study
$p$-adic differential equations.

This motivates the introduction of the overconvergent de Rham period ring
$B_{\dR}^{\dagger}\subset B_{\dR}$ and its relative period sheaves. The guiding
principle is that $B_{\dR}^{\dagger}$ should be large enough to contain the
periods needed for $p$-adic monodromy, but small enough to remember the
analytic convergence conditions which are lost after passing to the completed
ring $B_{\dR}$.

%%%%%%%%%%%%%%%%%%%%%%%%%%%%%%%%%%%%%%%%%%%%%%%%%%%%%%%%%%%%
%%%%%%%%%%%%%%%%%%%%%%%%%%%%%%%%%%%%%%%%%%%%%%%%%%%%%%%%%%%%
% Introduction
%%%%%%%%%%%%%%%%%%%%%%%%%%%%%%%%%%%%%%%%%%%%%%%%%%%%%%%%%%%%
%%%%%%%%%%%%%%%%%%%%%%%%%%%%%%%%%%%%%%%%%%%%%%%%%%%%%%%%%%%%

\subsection{Why $B_{\dR}$ is too large}

Let us return to the rank-one differential equation from
Example~\ref{ex:mon-periods}, now in the $p$-adic setting. Let
$X=\DD^{\times}=\Sp k\<z\>\setminus\{0\}$
and consider $zf^{\prime}=\lambda f$
for some $\lambda\in k$. As in the complex case, one expects the associated
monodromy to involve a $p$-adic analogue of $e^{2\pi i\lambda}$.

The role of $2\pi i$ is played by Fontaine's element $t\in B_{\dR}$. More
precisely, choose a compatible system
$\epsilon=(1,\zeta_{p},\zeta_{p^{2}},\dots)$
of $p$-power roots of unity in $C$,
the completion of a fixed algebraic closure $\overline{k}$ of $k$,
and write
$t=\log[\epsilon]$,
where $[\epsilon]$ denotes the Teichm\"uller lift of $\epsilon$. Since $t$ is
the $p$-adic analogue of $2\pi i$, one is led to the expression
\[
  e^{2\pi i\lambda}
  \quad\leadsto\quad
  [\epsilon]^{\lambda}
  =
  \sum_{i\geq 0}
  \binom{\lambda}{i}\left([\epsilon]-1\right)^{i}
  \in B_{\dR}.
\]
Thus $B_{\dR}$ supplies a natural coefficient ring for a first version of
$p$-adic monodromy:
\[
  \rho_{\lambda}\colon
  \ZZ_{p}(1)\to\GL_{1}\left(B_{\dR}\right),
  \qquad
  \epsilon\mapsto[\epsilon]^{\lambda}.
\]
A general formalism for such $p$-adic monodromy representations will be
developed elsewhere.

However, this construction also shows why $B_{\dR}$ is too large for the
applications envisioned here. The issue is not that $B_{\dR}$ lacks periods;
rather, it contains too many of them. Passing to $B_{\dR}$ completes away the
analytic convergence information which should distinguish well-behaved
$p$-adic differential equations from pathological ones. The overconvergent
period ring $B_{\dR}^{\dagger}$ is designed to retain precisely this missing
information.

%%%%%%%%%%%%%%%%%%%%%%%%%%%%%%%%%%%%%%%%%%%%%%%%%%%%%%%%%%%%
%%%%%%%%%%%%%%%%%%%%%%%%%%%%%%%%%%%%%%%%%%%%%%%%%%%%%%%%%%%%
% Introduction
%%%%%%%%%%%%%%%%%%%%%%%%%%%%%%%%%%%%%%%%%%%%%%%%%%%%%%%%%%%%
%%%%%%%%%%%%%%%%%%%%%%%%%%%%%%%%%%%%%%%%%%%%%%%%%%%%%%%%%%%%

\subsubsection{Non-Liouville numbers}
\label{subsubsec:nonliouville-periods}

The first indication that $B_{\dR}^{\dagger}$ is the correct replacement for
$B_{\dR}$ comes from the theory of $p$-adic differential equations. Recall that
a scalar $\lambda\in k$ is called \emph{$p$-adically Liouville} if it admits
rapid approximation by positive integers, that is if
$\liminf_{n\to+\infty}|\lambda-n|^{1/n}=0$.
This notion was introduced by Clark~\cite{Clark1966_padicconvergence} in his
study of $p$-adic convergence of solutions of linear differential equations.
For the rank-one equation $zf^{\prime}=\lambda f$,
the distinction between Liouville and non-Liouville exponents is a genuine
analytic condition: non-Liouville exponents are precisely the exponents for
which the associated differential equation has the expected finiteness and
convergence behaviour. A modern formulation of this phenomenon, closely related
to the overconvergent de Rham period ring, appears in
Gao--Wang~\cite[Theorem 10.1]{arXiv260421605}.

Fontaine's ring $B_{\dR}$ does not see this distinction. Indeed,
$[\epsilon]^{\lambda}$ belongs to $B_{\dR}$ for every $\lambda\in k$, including
Liouville values of $\lambda$. In contrast,
$[\epsilon]^{\lambda}\in B_{\dR}^{\dagger}$
if and only if $\lambda$ is non-Liouville.
Thus $B_{\dR}^{\dagger}$ provides a period-theoretic way to impose the
non-Liouville condition. From this point of view, the overconvergent period ring
is not merely a smaller subring of $B_{\dR}$; it is the subring which remembers
the analytic convergence condition naturally attached to the differential
equation.

%%%%%%%%%%%%%%%%%%%%%%%%%%%%%%%%%%%%%%%%%%%%%%%%%%%%%%%%%%%%
%%%%%%%%%%%%%%%%%%%%%%%%%%%%%%%%%%%%%%%%%%%%%%%%%%%%%%%%%%%%
% Introduction
%%%%%%%%%%%%%%%%%%%%%%%%%%%%%%%%%%%%%%%%%%%%%%%%%%%%%%%%%%%%
%%%%%%%%%%%%%%%%%%%%%%%%%%%%%%%%%%%%%%%%%%%%%%%%%%%%%%%%%%%%

\subsubsection{A Riemann-Hilbert correspondence for $\Dcap$-modules}
\label{subsubsec:Dcapmodules-periods}

The second motivation comes from $p$-adic geometric representation theory.
Ardakov and Wadsley~\cite{AW2024_GlobalsectionsoneqlinebundlesonOmega}
proved, by explicit local computations, admissibility results for certain
locally analytic representations arising from Drinfeld's first covering of the
$p$-adic upper half-plane. Their work uses the theory of coadmissible
$\Dcap$-modules, where $\Dcap$ denotes the sheaf of infinite-order differential
operators on a smooth rigid-analytic variety.

This suggests that a Riemann-Hilbert correspondence for coadmissible
$\Dcap$-modules should play a structural role in the study of such
representations. Scholze's $p$-adic Riemann-Hilbert correspondence uses the
period sheaf $\OB_{\dR}$, equipped with its connection $\nabla_{\dR}$, to relate
filtered vector bundles with flat connection to $\mathbb{B}_{\dR}^{+}$-local
systems~\cite[Theorem 7.6]{Sch13pAdicHodge}. This construction is fundamental,
but the completed period sheaf $\OB_{\dR}$ is too large for the
coadmissible $\Dcap$-module setting: it naturally carries the structure of a
$\cal{D}$-module, but not the required structure of a
$\Dcap$-module.

The overconvergent period structure sheaf
\[
  \OB_{\dR}^{\dagger}\subset\OB_{\dR}
\]
is introduced to remedy this defect. It is the relative analogue of the
inclusion $B_{\dR}^{\dagger}\subset B_{\dR}$ and is small enough to carry the
desired $\Dcap$-module structure. In the follow-up work
\cite{WiersigReconstruction}, this construction is used to obtain a
Riemann-Hilbert correspondence for $\Dcap$-modules.

The purpose of the present work is to establish the period-ring and
period-sheaf cohomology needed for this overconvergent theory. In particular,
we compute the continuous Galois cohomology of $B_{\dR}^{\dagger}$,
prove local descriptions of the corresponding relative
period structure sheaves, and compute their derived pushforwards along the
projection from the pro-\'etale site.

%%%%%%%%%%%%%%%%%%%%%%%%%%%%%%%%%%%%%%%%%%%%%%%%%%%%%%%%%%%%
%%%%%%%%%%%%%%%%%%%%%%%%%%%%%%%%%%%%%%%%%%%%%%%%%%%%%%%%%%%%
% Introduction
%%%%%%%%%%%%%%%%%%%%%%%%%%%%%%%%%%%%%%%%%%%%%%%%%%%%%%%%%%%%
%%%%%%%%%%%%%%%%%%%%%%%%%%%%%%%%%%%%%%%%%%%%%%%%%%%%%%%%%%%%

\section{Results}

Let $k$ be a complete discrete valuation field in mixed characteristic $(0,p)$ with perfect residue field.
Fix an algebraic closure $\overline{k}$ of $k$ with completion $C$.

%%%%%%%%%%%%%%%%%%%%%%%%%%%%%%%%%%%%%%%%%%%%%%%%%%%%%%%%%%%%
%%%%%%%%%%%%%%%%%%%%%%%%%%%%%%%%%%%%%%%%%%%%%%%%%%%%%%%%%%%%
% Introduction
%%%%%%%%%%%%%%%%%%%%%%%%%%%%%%%%%%%%%%%%%%%%%%%%%%%%%%%%%%%%
%%%%%%%%%%%%%%%%%%%%%%%%%%%%%%%%%%%%%%%%%%%%%%%%%%%%%%%%%%%%

\subsection{Galois cohomology}

The absolute period ring $B_{\dR}^{\dag,+}$ is the stalk $\cal{O}_{\FarguesFontaine,\infty}$
of the structure sheaf of the analytic Fargues-Fontaine curve at the point $\infty\in\FarguesFontaine$.
We define $B_{\dR}^{\dag}:=B_{\dR}^{\dag,+}[1/t]$
where $t$ is Fontaine's $2\pi i$. Our first main result is:

\begin{introthm}[Theorem~\ref{thm:galois-cohomology-of-solidBdRdagger-born}]
\label{thm:BdRdagplusplus-Galoiscohomology-introduction}
\hfill %$B_{\dR}^{\dag}$ denotes the overconvergent de Rham period ring.
\begin{equation*}
  \Ho_{\cont}^{i}\left(\Gal\left(\overline{k}/k\right),B_{\dR}^{\dag}\right)
  \cong\begin{cases}
    k & \text{ if } i=0,1 \\
    0 & \text{ if } i \geq 2.
  \end{cases}
\end{equation*}
\end{introthm}

We recall that Fontaine's ring $B_{\dR}$ is equipped with a filtration induced
by the $t$-adic filtration on $B_{\dR}^{+}$. Since $B_{\dR}$ is complete with
respect to this filtration, the continuous Galois cohomology of $B_{\dR}$ can
be computed in terms of the Galois cohomology of the associated graded pieces
$\gr^{n}B_{\dR}$, which are isomorphic to the Tate twists $C(n)$. The Galois cohomology of
$C(n)$ was computed by Tate~\cite{tatepdivisiblegroups}, yielding an explicit
description of the continuous Galois cohomology of $B_{\dR}$.

To prove Theorem~\ref{thm:BdRdagplusplus-Galoiscohomology-introduction}, we
follow the same guiding philosophy, but a new argument is required. The ring
$B_{\dR}^{\dagger}$ inherits the filtration induced by the inclusion
$B_{\dR}^{\dagger}\subset B_{\dR}$. However, $B_{\dR}^{\dagger}$ is not
complete with respect to this filtration; its completion is $B_{\dR}$. Thus
Theorem~\ref{thm:BdRdagplusplus-Galoiscohomology-introduction} is not a formal
consequence of the corresponding computation for $B_{\dR}$. In particular, one
cannot simply pass to the associated graded pieces and invoke Tate's computation
for $C(n)$ as in the classical complete setting: doing so would lose precisely
the overconvergence information encoded by $B_{\dR}^{\dagger}$.

Instead, we reduce the proof of
Theorem~\ref{thm:BdRdagplusplus-Galoiscohomology-introduction} to the
computation of the Galois cohomology of the Tate twists $\mathcal{O}_{C}(n)$ of
the ring of integers of $C$. This reduction is one of the main technical points
of the proof. It requires a torsion-sensitive analysis of the integral pieces
which occur before passing to the completion, and hence cannot be replaced by
the usual argument for $B_{\dR}$. See
\S\ref{subsec:on:prop:galois-cohomology-of-BdRdaggerplus-primitiverootofunity}
and
\S\ref{subsubsec:--prop:galois-cohomology-of-BdRdaggerplus-kprime-injectintotminusj--anditsproof-recpaper}
for an overview of the steps involved.

The required Galois cohomology of $\mathcal{O}_{C}(n)$ is due to
Barthel--Schlank--Stapleton--Weinstein~\cite[Theorems 4.0.4 and
4.0.5]{BSSW2024_rationalizationoftheKnlocalsphere}. Their result is a key
input in our proof of
Theorem~\ref{thm:BdRdagplusplus-Galoiscohomology-introduction}, because it
provides the necessary control of the torsion appearing in the integral Tate
twists.

As a byproduct, we also compute the continuous Galois cohomology of
$B_{\dR}^{\dagger,+}$
(Theorem~\ref{thm:galois-cohomology-of-solidBdRdaggerplus-born}). After the
present author's results were made public, Gao--Wang developed a Tate--Sen
formalism for representations over the convergent de Rham period ring
$B_{\dR}^{\dagger,+}$~\cite{arXiv260421605}. Their work gives, among
other things, a cohomology comparison theorem for
$B_{\dR}^{\dagger,+}$-representations under a $p$-adic non-Liouville condition
on the Sen weights. The case of the trivial representation recovers the
continuous Galois cohomology computation of $B_{\dR}^{\dagger,+}$ obtained
here, but by different methods. Thus the present argument and Gao--Wang's
Tate--Sen approach are complementary: the former proves the foundational
absolute computation by an integral, torsion-sensitive filtration argument,
whereas the latter places the resulting computation in a broader
representation-theoretic framework.

%%%%%%%%%%%%%%%%%%%%%%%%%%%%%%%%%%%%%%%%%%%%%%%%%%%%%%%%%%%%
%%%%%%%%%%%%%%%%%%%%%%%%%%%%%%%%%%%%%%%%%%%%%%%%%%%%%%%%%%%%
% Introduction
%%%%%%%%%%%%%%%%%%%%%%%%%%%%%%%%%%%%%%%%%%%%%%%%%%%%%%%%%%%%
%%%%%%%%%%%%%%%%%%%%%%%%%%%%%%%%%%%%%%%%%%%%%%%%%%%%%%%%%%%%

\subsection{Cohomology of period sheaves}

Fix a smooth rigid-analytic $k$-variety $X$. Let $X_{\proet}$ be Scholze's
pro-\'etale site with canonical projection
$\nu\colon X_{\proet}\to X$
as in~\cite{Sch13pAdicHodge}. We define an overconvergent de Rham period sheaf
$\BB_{\dR}^{\dag}$ on $X_{\proet}$. If $X=\Sp k$ is a point, then
$\BB_{\dR}^{\dag}$ recovers the absolute period ring $B_{\dR}^{\dag}$. As in
\S\ref{subsubsec:Dcapmodules-periods}, we also introduce the overconvergent
de Rham period structure sheaf $\OB_{\dR}^{\dag}$ on $X_{\proet}$. It is an
overconvergent completion of $\nu^{-1}\cal{O}\otimes\BB_{\dR}^{\dag}$.

Our main relative cohomology result is the following.

\begin{introthm}[Theorem~\ref{cor:borncohomologyOBdRdagX-reconstructionpaper}]
\label{thm:borncohomologyOBdRdagX-periods-intro}
\hfill
\begin{equation*}
  \R^{i}\nu_{*}\OB_{\dR}^{\dag}
  \cong
  \begin{cases}
    \cal{O} & \text{ if } i=0,1, \\
    0       & \text{ if } i \geq 2.
  \end{cases}
\end{equation*}
\end{introthm}

\begin{remark}
  If $X=\Sp k$ is a point, then
  Theorem~\ref{thm:borncohomologyOBdRdagX-periods-intro} recovers
  Theorem~\ref{thm:BdRdagplusplus-Galoiscohomology-introduction}.
\end{remark}

\begin{remark}
  Scholze~\cite[Proposition 6.16]{Sch13pAdicHodge} computed the corresponding
  higher direct images for the classical period structure sheaf $\OB_{\dR}$.
  In comparing with this result, one should use the corrected definition of
  $\OB_{\dR}$ and the completion convention explained in
  \cite[Remark 2.2.11]{MR4536903}; in particular, the classical sheaf has to
  be completed with respect to the Hodge filtration. For
  $\OB_{\dR}^{\dag}$ no further completion with respect to the Hodge filtration
  is taken. Instead, the local sections of $\OB_{\dR}^{\dag}$ can be written
  as filtered colimits of localisations of rings carrying complete auxiliary
  filtrations. These complete filtrations are the ones used in the proof of
  Theorem~\ref{thm:borncohomologyOBdRdagX-periods-intro}.
\end{remark}

We briefly indicate the strategy of the proof. Since the assertion is local on
$X$, we may assume that $X$ is affinoid and equipped with an \'etale morphism
$X\to\TT^{d}:=\Sp k\left\<T_{1}^{\pm},\dots,T_{d}^{\pm}\right\>$.
Let $\widetilde{X}\to X$ denote the standard toric pro-\'etale cover associated
with this map, and write $X_{C}$ and $\widetilde{X}_{C}$ for the corresponding
base changes to $C$.

The proof has two parts. First, one computes over $X_{C}$ by using the
pro-\'etale cover $\widetilde{X}_{C}\to X_{C}$. This cover is a
$\ZZ_{p}(1)^{d}$-torsor, so the relevant \v{C}ech complex can be identified
with a continuous group cohomology complex. The computation then becomes
explicit because of the following local description of $\OB_{\dR}^{\dag}$
(Theorem~\ref{thm:cohomology-OBdRdag-over-affperfd-reconstructionpaper}):
\begin{equation*}
  {\varinjlim}_{q}
  \BB_{\dR}^{\dag}|_{\widetilde{X}}
  \left\<\frac{Z_{1},\dots,Z_{d}}{p^{q}}\right\>
  \stackrel{\cong}{\longrightarrow}
  \OB_{\dR}^{\dag}|_{\widetilde{X}},
  \qquad
  Z_{i}\mapsto
  T_{i}\widehat{\otimes}1
  -
  1\widehat{\otimes}\left[T_{i}^{\flat}\right].
\end{equation*}
This yields an explicit description of
$\R\Gamma(X_{C},\OB_{\dR}^{\dag})$ in terms of
$\cal{O}(X)$ and the period ring $B_{\dR}^{\dag}$.

Second, one descends from $X_{C}$ to $X$. This step reduces the computation to
the continuous Galois cohomology of $B_{\dR}^{\dag}$ computed in
Theorem~\ref{thm:BdRdagplusplus-Galoiscohomology-introduction}. Combining the
explicit computation over $X_{C}$ with this absolute Galois cohomology
calculation gives
Theorem~\ref{thm:borncohomologyOBdRdagX-periods-intro}.

\begin{remark}\label{rem:whydiscvalued}
  Our proof of the local description of $\OB_{\dR}^{\dag}$
  (Theorem~\ref{thm:cohomology-OBdRdag-over-affperfd-reconstructionpaper})
  uses that $k$ is discretely valued; see
  Remark~\ref{rem:OBdRlocaldescription-discreteval}.
\end{remark}

\begin{remark}
  We also give an explicit local description of $\OB_{\dR}^{\dag,+}$
  (Theorem~\ref{cor:localdescription-of-subsections-OBla}). After completion,
  this recovers Scholze's local description of $\OB_{\dR}^{+}$
  \cite[Proposition 6.10]{Sch13pAdicHodge}.
\end{remark}

%%%%%%%%%%%%%%%%%%%%%%%%%%%%%%%%%%%%%%%%%%%%%%%%%%%%%%%%%%%%
%%%%%%%%%%%%%%%%%%%%%%%%%%%%%%%%%%%%%%%%%%%%%%%%%%%%%%%%%%%%
% Introduction
%%%%%%%%%%%%%%%%%%%%%%%%%%%%%%%%%%%%%%%%%%%%%%%%%%%%%%%%%%%%
%%%%%%%%%%%%%%%%%%%%%%%%%%%%%%%%%%%%%%%%%%%%%%%%%%%%%%%%%%%%

\subsection{The overconvergent Poincar\'e lemma}

Theorem~\ref{thm:borncohomologyOBdRdagX-periods-intro} shows that
$\R\nu_{*}\OB_{\dR}^{\dag}$ is not concentrated in degree zero: a single copy
of $\cal{O}$ survives in degree one. For applications to Riemann-Hilbert
correspondences, it is useful to remove this degree-one class.

Following an idea of Fontaine, we do this by adjoining $\log t$. In the
classical setting, adjoining $\log t$ kills the degree-one class in the
continuous Galois cohomology of $B_{\dR}$. The same mechanism works in the
overconvergent setting. More precisely, we introduce the overconvergent almost
de Rham period ring $B_{\pdR}^{\dag}$, the overconvergent analogue of
Fontaine's ring
$B_{\pdR}=B_{\dR}[\log t]$,
and we prove that its continuous Galois cohomology is concentrated in degree
zero; see Theorem~\ref{thm:galois-cohomology-of-BpdRdagger}.

Passing to the relative period structure sheaf gives:

\begin{introthm}[Poincar\'e lemma, Corollary~\ref{cor:borncohomologyOBpdRdagX-reconstructionpaper}]
\label{thm:pushforwardOBpdRdag-introduction}
  We have the canonical isomorphism
  \begin{equation*}
    \cal{O}\isomap\R\nu_{*}\OB_{\pdR}^{\dag}.
  \end{equation*}
\end{introthm}

In our work on the Riemann--Hilbert correspondence for $\Dcap$-modules
\cite{WiersigReconstruction,WiersigCauchy},
Theorem~\ref{thm:pushforwardOBpdRdag-introduction} plays the role of the
Poincar\'e lemma in the classical complex Riemann--Hilbert correspondence.
In particular, it allows us to recover a vector bundle with flat connection from 
its de Rham complex.

%%%%%%%%%%%%%%%%%%%%%%%%%%%%%%%%%%%%%%%%%%%%%%%%%%%%%%%%%%%%
%%%%%%%%%%%%%%%%%%%%%%%%%%%%%%%%%%%%%%%%%%%%%%%%%%%%%%%%%%%%
% P-ADIC GEOMETRY
%%%%%%%%%%%%%%%%%%%%%%%%%%%%%%%%%%%%%%%%%%%%%%%%%%%%%%%%%%%%
%%%%%%%%%%%%%%%%%%%%%%%%%%%%%%%%%%%%%%%%%%%%%%%%%%%%%%%%%%%%

\section{Notation and conventions}\label{subsec:notationconventions}

%%%%%%%%%%%%%%%%%%%%%%%%%%%%%%%%%%%%%%%%%%%%%%%%%%%%%%%%%%%%
%%%%%%%%%%%%%%%%%%%%%%%%%%%%%%%%%%%%%%%%%%%%%%%%%%%%%%%%%%%%
% P-ADIC GEOMETRY
%%%%%%%%%%%%%%%%%%%%%%%%%%%%%%%%%%%%%%%%%%%%%%%%%%%%%%%%%%%%
%%%%%%%%%%%%%%%%%%%%%%%%%%%%%%%%%%%%%%%%%%%%%%%%%%%%%%%%%%%%

\subsection{Ground fields and rings}\label{subsec:conventions-reconstructionpaper}
Throughout, $k$ denotes a complete discrete valuation field of mixed characteristic
$(0,p)$ with perfect residue field $\kappa$.
Fix a uniformiser $\pi\in k$ and write $k^{\circ}\subseteq k$
for the ring of power-bounded elements. Set $k_{0}:=W(\kappa)[1/p]$.
$C$ is the completion of a fixed algebraic closure of $\overline{k}$ of $k$.

\begin{remark}
  We require $k$ to be discretely valued for two reasons. For the first reason, see Remark~\ref{rem:whydiscvalued}.
  Secondly, it allows us to apply~\cite[Theorems 4.0.4 and 4.0.5]{BSSW2024_rationalizationoftheKnlocalsphere},
  which underlie our proof of Theorem~\ref{thm:BdRdagplusplus-Galoiscohomology-introduction}. Note
  that~\cite{BSSW2024_rationalizationoftheKnlocalsphere} uses the term
  \emph{local field} in a somewhat nonstandard manner, cf. \emph{loc. cit.} at the beginning of \S 4.
\end{remark}

%%%%%%%%%%%%%%%%%%%%%%%%%%%%%%%%%%%%%%%%%%%%%%%%%%%%%%%%%%%%
%%%%%%%%%%%%%%%%%%%%%%%%%%%%%%%%%%%%%%%%%%%%%%%%%%%%%%%%%%%%
% Summary
%%%%%%%%%%%%%%%%%%%%%%%%%%%%%%%%%%%%%%%%%%%%%%%%%%%%%%%%%%%%
%%%%%%%%%%%%%%%%%%%%%%%%%%%%%%%%%%%%%%%%%%%%%%%%%%%%%%%%%%%%

\subsection{Overview of period rings}

The absolute period ring $B_{\dR}^{\dag,+}$ is the
stalk $\cal{O}_{\FarguesFontaine,\infty}$
of the structure sheaf of the analytic Fargues-Fontaine curve at the point $\infty\in\FarguesFontaine$.
That is, $B_{\dR}^{\dag,+}=\varinjlim_{q}B_{\dR}^{q,+}$,
where $B_{\dR}^{q,+}$ is the ring of functions on the closed ball
$\{x\colon |x-\infty|\leq p^{-q}\}\subseteq\FarguesFontaine$.
Similarly, $B_{\dR}^{\dag,+}=\varinjlim_{q}B_{\dR}^{>q,+}$,
where $B_{\dR}^{>q,+}$ is the ring of functions on the open ball
$\{x\colon |x-\infty|< p^{-q}\}\subseteq\FarguesFontaine$.
Similarly, we obtain presentations
$B_{\dR}^{\dag}=\varinjlim_{q}B_{\dR}^{q}=\varinjlim_{q}B_{\dR}^{>q}$
and
$B_{\pdR}^{\dag}=\varinjlim_{q}B_{\pdR}^{q}=\varinjlim_{q}B_{\pdR}^{>q}$.

Here is an overview of all period rings considered:
\begin{equation*}
\begin{tikzcd}[sep=tiny]
  \empty &
  B_{\pdR}^{0} \arrow[phantom]{r}{\subset} &
  B_{\pdR}^{>0} \arrow[phantom]{r}{\subset} &
  \cdots \arrow[phantom]{r}{\subset} &
  B_{\pdR}^{q} \arrow[phantom]{r}{\subset} &
  B_{\pdR}^{>q} \arrow[phantom]{r}{\subset} &
  B_{\pdR}^{q+1} \arrow[phantom]{r}{\subset} &
  B_{\pdR}^{>q+1} \arrow[phantom]{r}{\subset} &
  \cdots \arrow[phantom]{r}{\subset} &
  B_{\pdR}^{\dag} \arrow[phantom]{r}{\subset} &
  B_{\pdR} \\
  \empty &
  B_{\dR}^{0} \arrow[phantom]{r}{\subset} \arrow[phantom]{u}{\cup} &
  B_{\dR}^{>0} \arrow[phantom]{r}{\subset} \arrow[phantom]{u}{\cup} &
  \cdots \arrow[phantom]{r}{\subset} &
  B_{\dR}^{q} \arrow[phantom]{r}{\subset} \arrow[phantom]{u}{\cup} &
  B_{\dR}^{>q} \arrow[phantom]{r}{\subset} \arrow[phantom]{u}{\cup} &
  B_{\dR}^{q+1} \arrow[phantom]{r}{\subset} \arrow[phantom]{u}{\cup} &
  B_{\dR}^{>q+1} \arrow[phantom]{r}{\subset} \arrow[phantom]{u}{\cup} &
  \cdots \arrow[phantom]{r}{\subset} &
  B_{\dR}^{\dag} \arrow[phantom]{r}{\subset} \arrow[phantom]{u}{\cup} &
  B_{\dR} \arrow[phantom]{u}{\cup} \\
  B_{\inf} \arrow[phantom]{r}{\subset} &
    B_{\dR}^{0,+} \arrow[phantom]{r}{\subset} \arrow[phantom]{u}{\cup} &
    B_{\dR}^{>0,+} \arrow[phantom]{r}{\subset} \arrow[phantom]{u}{\cup} &
    \cdots \arrow[phantom]{r}{\subset} &
    B_{\dR}^{q,+} \arrow[phantom]{r}{\subset} \arrow[phantom]{u}{\cup} &
    B_{\dR}^{>q,+} \arrow[phantom]{r}{\subset} \arrow[phantom]{u}{\cup} &
    B_{\dR}^{q+1,+} \arrow[phantom]{r}{\subset} \arrow[phantom]{u}{\cup} &
    B_{\dR}^{>q+1,+} \arrow[phantom]{r}{\subset} \arrow[phantom]{u}{\cup} &
    \cdots \arrow[phantom]{r}{\subset} &
    B_{\dR}^{\dag,+} \arrow[phantom]{r}{\subset} \arrow[phantom]{u}{\cup} &
    B_{\dR}^{+} \arrow[phantom]{u}{\cup} \\
  A_{\inf} \arrow[phantom]{r}{\subset} \arrow[phantom]{u}{\cup} &
    A_{\dR}^{0} \arrow[phantom]{r}{\subset} \arrow[phantom]{u}{\cup} &
    A_{\dR}^{>0} \arrow[phantom]{r}{\subset} \arrow[phantom]{u}{\cup} &
    \cdots \arrow[phantom]{r}{\subset} &
    A_{\dR}^{q} \arrow[phantom]{r}{\subset} \arrow[phantom]{u}{\cup} &
    A_{\dR}^{>q} \arrow[phantom]{r}{\subset} \arrow[phantom]{u}{\cup} &
    A_{\dR}^{q+1} \arrow[phantom]{r}{\subset} \arrow[phantom]{u}{\cup} &
    A_{\dR}^{>q+1} \arrow[phantom]{r}{\subset} \arrow[phantom]{u}{\cup} &
    \cdots \arrow[phantom]{r}{\subset} &
    A_{\dR}^{\dag} \arrow[phantom]{u}{\cup} &
    \empty
\end{tikzcd}
\end{equation*}
We obtain similar diagrams for the period sheaves and period structure sheaves.

\iffalse %%% SIMPLIFIED
\begin{equation*}
\begin{tikzcd}
  \empty & \BB_{\pdR}^{q} \arrow[phantom]{r}{\subset} & \BB_{\pdR}^{>q} \arrow[phantom]{r}{\subset} & \BB_{\pdR}^{\dag} \arrow[phantom]{r}{\subset} & \BB_{\pdR} \\
  \empty &
  \BB_{\dR}^{q} \arrow[phantom]{r}{\subset} \arrow[phantom]{u}{\cup} &
  \BB_{\dR}^{>q} \arrow[phantom]{r}{\subset} \arrow[phantom]{u}{\cup} &
  \BB_{\dR}^{\dag} \arrow[phantom]{r}{\subset} \arrow[phantom]{u}{\cup} &
  \BB_{\dR} \arrow[phantom]{u}{\cup} \\
  \BB_{\inf} \arrow[phantom]{r}{\subset} &
    \BB_{\dR}^{q,+} \arrow[phantom]{r}{\subset} \arrow[phantom]{u}{\cup} &
    \BB_{\dR}^{>q,+} \arrow[phantom]{r}{\subset} \arrow[phantom]{u}{\cup} &
    \BB_{\dR}^{\dag,+} \arrow[phantom]{r}{\subset} \arrow[phantom]{u}{\cup} &
    \BB_{\dR}^{+} \arrow[phantom]{u}{\cup} \\
  \A_{\inf} \arrow[phantom]{r}{\subset} \arrow[phantom]{u}{\cup} &
    \A_{\dR}^{q} \arrow[phantom]{r}{\subset} \arrow[phantom]{u}{\cup} &
    \A_{\dR}^{>q} \arrow[phantom]{r}{\subset} \arrow[phantom]{u}{\cup} &
    \A_{\dR}^{\dag} \arrow[phantom]{u}{\cup} &
    \empty
\end{tikzcd}
\end{equation*}
\fi %%% COMMENT ENDS

%%%%%%%%%%%%%%%%%%%%%%%%%%%%%%%%%%%%%%%%%%%%%%%%%%%%%%%%%%%%
%%%%%%%%%%%%%%%%%%%%%%%%%%%%%%%%%%%%%%%%%%%%%%%%%%%%%%%%%%%%
% Notation
%%%%%%%%%%%%%%%%%%%%%%%%%%%%%%%%%%%%%%%%%%%%%%%%%%%%%%%%%%%%
%%%%%%%%%%%%%%%%%%%%%%%%%%%%%%%%%%%%%%%%%%%%%%%%%%%%%%%%%%%%

\subsection{Miscellany}
\label{subsec:notation}

$=$ denotes an equality, $\cong$ denotes an isomorphism, and $\simeq$ denotes an equivalence.
When an equality, isomorphism, or equivalence follows from a specific result, the reference
is written above the symbol denoting the equality, isomorphism, or equivalence. For example,
$X\stackrel{\text{\ref{lem:solidOBpdRdag-directsum-of-solidOBdRdag}}}{\cong}Y$ means:
Lemma~\ref{lem:solidOBpdRdag-directsum-of-solidOBdRdag} implies that $X$ and $Y$ are isomorphic.

Let $\mathbf{F}\colon\C\to\D$ be an additive functor between two additive categories.
It extends to a functor between the associated categories of chain complexes
and cochain complexes. Abusing notation, we denote it again by $\F$.

Given a symmetric monoidal category containing a monoid object $S$,
$\Mod(S)$ denotes the category of $S$-module objects. It is again symmetric
monoidal if $S$ is commutative, cf.~\cite[\S 2.2]{Bo21}.
The term \emph{module} always means \emph{left module}, unless
explicitly stated otherwise.

The natural numbers are $\NN=\{0,1,2,3,\dots\}$. For any $n\in\NN$,
$\NN_{\geq n}:=\{n,n+1,n+2,n+3\dots\}$.

All filtrations are descending, if not specified otherwise.
Given an ideal $I\subseteq R$ in a commutative ring $R$, 
the \emph{$I$-adic filtration on $R$} is given by
$\Fil^{s}R=I^{s}$ for all $s\in\ZZ_{\geq0}$ and $\Fil^{s}R=R$ if $s\in\ZZ_{<0}$.

%Fix a prime number $p$ throughout this article.

All Huber pairs $\left(A,A^{+}\right)$ are complete, that is
both $A$ and $A^{+}$ are complete as topological rings.

%%%%%%%%%%%%%%%%%%%%%%%%%%%%%%%%%%%%%%%%%%%%%%%%%%%%%%%%%%%%
%%%%%%%%%%%%%%%%%%%%%%%%%%%%%%%%%%%%%%%%%%%%%%%%%%%%%%%%%%%%
% Acknowledgements
%%%%%%%%%%%%%%%%%%%%%%%%%%%%%%%%%%%%%%%%%%%%%%%%%%%%%%%%%%%%
%%%%%%%%%%%%%%%%%%%%%%%%%%%%%%%%%%%%%%%%%%%%%%%%%%%%%%%%%%%%

\section{Acknowledgements}

Parts of this monograph are based on the author’s PhD thesis.
I am deeply grateful to my PhD supervisors, Konstantin Ardakov and Kobi Kremnitzer,
for introducing me to this subject and for their guidance and support:
thank you Konstantin and Kobi.

I would also like to thank Konstantin Ardakov for his continued mentorship and for many insightful and helpful discussions following my PhD, especially regarding the proof of Theorem~\ref{thm:pushforwardOBpdRdag-introduction}. My gratitude also extends to David Hansen for bringing Fontaine’s $B_{\pdR}$ to my attention, suggesting Proposition~\ref{prop:underlyingspaceBdRdagplus-iso-uncompletedcolim-reconstructionpaper}, and providing valuable comments on earlier drafts of this article. I am grateful to Fernando Peña Vázquez for identifying
a mistake in an earlier version of this article.

I would like to thank Simon Wadsley and Jack Kelly for their numerous valuable suggestions
on a draft of this article, as well as Dustin Clausen and Peter Scholze for discussions on the functional analytic formalism.

Furthermore, I would like to thank
Tomoyuki Abe,
%Oren Ben-Bassat,
Andreas Bode,
Hui Gao,
David Hansen,
Arun Soor,
and
Gergely Zábrádi
for valuable conversations.
Part of this research was financially supported by a Mathematical Institute Award at Oxford University.

%%%%%%%%%%%%%%%%%%%%%%%%%%%%%%%%%%%%%%%%%%%%%%%%%%%%%%%%%%%%
%%%%%%%%%%%%%%%%%%%%%%%%%%%%%%%%%%%%%%%%%%%%%%%%%%%%%%%%%%%%
%%%%%%%%%%%%%%%%%%%%%%%%%%%%%%%%%%%%%%%%%%%%%%%%%%%%%%%%%%%%
% Bornological spaces
%%%%%%%%%%%%%%%%%%%%%%%%%%%%%%%%%%%%%%%%%%%%%%%%%%%%%%%%%%%%
%%%%%%%%%%%%%%%%%%%%%%%%%%%%%%%%%%%%%%%%%%%%%%%%%%%%%%%%%%%%
%%%%%%%%%%%%%%%%%%%%%%%%%%%%%%%%%%%%%%%%%%%%%%%%%%%%%%%%%%%%

\chapter{Categories and sheaves}
\label{ch:functional-analysis}

We discuss the functional analytic formalism underpinning this monograph.
%We follow~\cite{Sch99},~\cite[\S 5]{benbassat2024perspectivefoundationsderivedanalytic}, and~\cite[\S 3, \S 4]{Bo21}.

%%%%%%%%%%%%%%%%%%%%%%%%%%%%%%%%%%%%%%%%%%%%%%%%%%%%%%%%%%%%
%%%%%%%%%%%%%%%%%%%%%%%%%%%%%%%%%%%%%%%%%%%%%%%%%%%%%%%%%%%%
% Introduction
%%%%%%%%%%%%%%%%%%%%%%%%%%%%%%%%%%%%%%%%%%%%%%%%%%%%%%%%%%%%
%%%%%%%%%%%%%%%%%%%%%%%%%%%%%%%%%%%%%%%%%%%%%%%%%%%%%%%%%%%%

\section{Quasi-abelian categories}
\label{subsec:qabelian-cat-recpaper}

We follow~\cite{Sch99}. Let $\E$ be an additive category with kernels and cokernels.

\begin{defn}
  A morphism $f$ in $\E$ is \emph{strict} if the induced $\coim f \to \im f$ is
  an isomorphism.
\end{defn}

$\E$ is abelian if all its morphisms $f$ are strict.
In this article, we encounter examples of categories $\E$
which are not abelian. Instead, these are \emph{quasi-abelian}.
That is, given the commutative diagram~(\ref{cd:quasi-abelian-cat}),
the stability conditions (i) and (ii) below hold, cf.~\cite[Definition 1.1.3]{Sch99}.
\begin{equation}\label{cd:quasi-abelian-cat}
  \begin{tikzcd}
    E \arrow{r}{e}\arrow{d} & F \arrow{d} \\
    G \arrow{r}{g} & H.
  \end{tikzcd}
\end{equation}
\begin{itemize}
  \item[(i)] If~(\ref{cd:quasi-abelian-cat}) is a pullback square, then
    $e$ is a strict epimorphism if $g$ is a strict epimorphism.
  \item[(ii)] If~(\ref{cd:quasi-abelian-cat}) is a pushout square, then
    $g$ is a strict monomorphism if $e$ is a strict monomorphism.
\end{itemize}

For the remainder of \S\ref{subsec:qabelian-cat-recpaper}, we suppose that $\E$ is quasi-abelian.
As Schneiders~\cite{Sch99} explains, this suffices to do homological algebra.
For example, instead of consider exact sequences, we
have the following definition as in \emph{loc. cit.} Definition 1.1.9.

\begin{defn}
  Consider a sequence of maps
  $E^{\prime}\stackrel{e^{\prime}}{\longrightarrow} E \stackrel{e}{\longrightarrow} E^{\prime\prime}$
  in $\E$ such that $e\circ e^{\prime}=0$. It is \emph{strictly exact} if the induced
  $\ker e^{\prime} \to \ker e^{\prime\prime}$ is an isomorphism. More generally,
  a sequence of maps
  \begin{equation*}
    E^{\bullet}\colon
    \dots\stackrel{e^{i-1}}{\longrightarrow} E^{i} \stackrel{e^{i}}{\longrightarrow} E^{i+1} \stackrel{e^{i+1}}{\longrightarrow}\dots
  \end{equation*}
  in $\E$ is a chochain complex if $e^{i}\circ e^{i-1}=0$ for all $i\in\ZZ$.
  $E^{\bullet}$ is \emph{strictly exact} if for all $i\in\ZZ$,
  $E^{i-1}\stackrel{e^{i-1}}{\longrightarrow} E^{i} \stackrel{e^{i}}{\longrightarrow} E^{i+1}$
  is strictly exact.
\end{defn}

Throughout this article, we freely use the following Definition~\ref{defn:exactnessclasses-recpaper},
cf.~\cite[\S 1.1.5]{Sch99}.

\begin{defn}\label{defn:exactnessclasses-recpaper}
  Given quasi-abelian categories $\E_{1}$ and $\E_{2}$,
  an additive functor $\F\colon\E_{1}\to\E_{2}$ is
  \begin{itemize}
    \item[(i)] \emph{left exact} if it preserves kernels of strict morphisms,
    \item[(ii)] \emph{strongly left exact} if it preserves kernels of arbitrary morphisms,
    \item[(iii)] \emph{right exact} if it preserves cokernels of strict morphisms,
    \item[(iv)] \emph{strongly right exact} if it preserves cokernels of arbitrary morphisms,
    \item[(v)] \emph{exact} if it is both left exact and right exact,
    \item[(vi)] \emph{strongly exact} if it is both strongly left exact and strongly right exact,
    \item[(vii)] \emph{strictly exact} if it preserves arbitrary strictly exact sequences.
  \end{itemize}
\end{defn}

Any quasi-abelian category $\E$ admits a derived category $\D\left(\E\right)$,
cf.~\cite[Definition 1.2.16]{Sch99}. It is a triangulated
category, arising as a localisation of the category of cochain complexes in $\E$.
$\D\left(\E\right)$ admits a natural t-structure, the \emph{left t-structure}, and its heart is the
\emph{left heart} $\LH\left(\E\right)$ of $\E$; see~\cite[Definition 1.2.18]{Sch99} for details.
Consequently, $\D\left(\E\right)\simeq\D\left(\LH\left(\E\right)\right)$. For example, any object $E\in\E$,
viewed as a complex concentrated in degree zero, lies in $\LH(\E)$. This gives rise to a functor
\begin{equation*}
  \I\colon\E\to\LH(\E), E\mapsto \text{$E$, viewed as a complex concentrated in degree zero},
\end{equation*}
which makes $\E$ a reflexive subcategory of $\LH(\E)$,
cf.~\cite[Proposition 1.2.27]{Sch99}.

The left t-structure $\D\left(\E\right)$ gives rise to cohomology functors
\begin{equation*}
  \LHo^{i}\colon\D\left(\E\right)\to\LH\left(\E\right).
\end{equation*}
When we understand $\D\left(\E\right)$ as the derived category
of $\LH\left(\E\right)$, we simply denote each $\LHo^{i}$ by $\Ho^{i}$. 

\begin{lem}\label{lem:LH-vs-H-reconstructionpaper}
  Given a cochain complex $E^{\bullet}$ in $\E$, we have canonical isomorphisms
  \begin{equation*}
    \LHo^{i}\left(E^{\bullet}\right)\cong\Ho^{i}\left(\I\left(E^{\bullet}\right)\right).
  \end{equation*}
\end{lem}
  
\begin{proof}
  This follows directly from~\cite[Proposition 1.2.32]{Sch99}.
\end{proof}

\begin{defn}\label{defn:quasiabelian-quasiiso-recpaper}
  A morphism $E^{\bullet}\to F^{\bullet}$ of cochain complexes
  in $\E$ is a \emph{quasi-isomorphism} if the induced morphisms
  $\LHo^{i}\left(E^{\bullet}\right)\isomap\LHo^{i}\left(F^{\bullet}\right)$ are quasi-isomorphisms
  in $\LH\left(\E\right)$ for all $i\in\ZZ$.
\end{defn}

\begin{lem}\label{lem:quasiiso-if-cone-strictlyexact-reconstructionpaper}
  A morphism $f^{\bullet}\colon E^{\bullet}\to F^{\bullet}$ of cochain complexes is a quasi-isomorphism
  if and only if its mapping cone $\cone\left(f^{\bullet}\right)$ is strictly exact.
\end{lem}

\begin{proof}
  By Lemma~\ref{lem:LH-vs-H-reconstructionpaper},
  $f^{\bullet}$ is a quasi-isomorphism if and only if $\I\left(f^{\bullet}\right)$ is a quasi-isomorphism.
  This is the case precisely when $\cone\left(\I\left(f^{\bullet}\right)\right)=\I\left(\cone\left(f^{\bullet}\right)\right)$
  is exact. Now apply~\cite[Corollary 1.2.28]{Sch99}.
\end{proof}

Finally, we include the following Lemma~\ref{lem:limpreservestrictmonomorphisms}
for future reference.

\begin{lem}\label{lem:limpreservestrictmonomorphisms}
  Fix a small category $I$ and suppose $\E$ admits limits of
  $I$-shaped diagrams. Then these limits preserve
  strict monomorphisms.
\end{lem}
  
\begin{proof}
  Let $\E^{I}$ denote the category of
  $I$-shaped diagrams in $\E$. The limit functor $\varprojlim_{I}\colon\E^{I}\to\E$
  exists by the assumptions, and it
  is a right adjoint of the constant diagram functor $\E\to\E^{I}$.
  We can thus apply~\cite[Lemma 3.7]{BBKFrechetModulesDescent}.
\end{proof}

%%%%%%%%%%%%%%%%%%%%%%%%%%%%%%%%%%%%%%%%%%%%%%%%%%%%%%%%%%%%
%%%%%%%%%%%%%%%%%%%%%%%%%%%%%%%%%%%%%%%%%%%%%%%%%%%%%%%%%%%%
%%%%%%%%%%%%%%%%%%%%%%%%%%%%%%%%%%%%%%%%%%%%%%%%%%%%%%%%%%%%
% Banach modules
%%%%%%%%%%%%%%%%%%%%%%%%%%%%%%%%%%%%%%%%%%%%%%%%%%%%%%%%%%%%
%%%%%%%%%%%%%%%%%%%%%%%%%%%%%%%%%%%%%%%%%%%%%%%%%%%%%%%%%%%%
%%%%%%%%%%%%%%%%%%%%%%%%%%%%%%%%%%%%%%%%%%%%%%%%%%%%%%%%%%%%

\section{Seminormed, normed, and Banach modules}
\label{subsec:categories-of-normed-modules}

We follow~\cite[\S 5]{benbassat2024perspectivefoundationsderivedanalytic}.
%We follow~\cite{BBB16} and~\cite{BBKFrechetModulesDescent}.

%%%%%%%%%%%%%%%%%%%%%%%%%%%%%%%%%%%%%%%%%%%%%%%%%%%%%%%%%%%%
%%%%%%%%%%%%%%%%%%%%%%%%%%%%%%%%%%%%%%%%%%%%%%%%%%%%%%%%%%%%
% Banach modules
%%%%%%%%%%%%%%%%%%%%%%%%%%%%%%%%%%%%%%%%%%%%%%%%%%%%%%%%%%%%
%%%%%%%%%%%%%%%%%%%%%%%%%%%%%%%%%%%%%%%%%%%%%%%%%%%%%%%%%%%%

\subsubsection{Rings and modules}

\begin{defn}
  A \emph{(non-Archimedean) seminormed ring} is a
  unitial commutative ring $R$ equipped with a map
  $|\cdot|\colon R\to\RR_{\geq0}$ such that 
  \begin{itemize}
    \item $|0|=0$,
    \item $|r+s|\leq\max\{|r|,|s|\}$ for all $r,s\in R$, and
    \item there is a $C>0$ such that $|rs|\leq C|r||s|$
      for all $r,s\in R$.
\end{itemize}
  $R$ is a \emph{(non-Archimedean) normed ring} if
  the following implication holds for all $r\in R$: $|r|=0$
  implies $r=0$. A normed ring is a \emph{(non-Archimedean) Banach ring}
  if it is a complete metric space with respect to the metric
  $(r,s)\mapsto|r-s|$.
\end{defn}

The following construction supplies examples of seminormed rings.

\begin{defn}\label{defn:I-adic-seminorm-norm}
  $R$ denotes an abstract commutative ring and $I\subseteq R$ an ideal.
  Define the \emph{$I$-adic seminorm on $R$ (with base $p$)}:
  Set $|r|:=p^{-v}$ for every $r\in R$,
  where $v\in\NN\cup\{\infty\}$ is maximal with respect to
  the property that $r\in I^{v}$. Here $I^{\infty}:=\bigcap_{j=0}^{\infty}I^{j}$
  and $p^{-\infty}:=0$. This turns $R$ into a seminormed ring.
  It is a normed ring when $R$ is separated with respect to the $I$-adic topology.
  $R$ is a Banach ring if it is separated and complete with respect to the $I$-adic topology.
  
  For any $s\in R$, the \emph{$s$-adic seminorm} is the $(s)$-adic seminorm.
\end{defn}

\begin{lem}\label{lem:bounded-map-adic-rings}
  Consider a map $\phi\colon R\to S$ between two Banach rings that
  is a morphism of abstract rings, %that is multiplicative and additive,
  $R$ carries an $I$-adic norm and $S$ carries a $J$-adic norm, and
  $\phi(I)\subseteq J$.
  Then $|\phi(r)|\leq |r|$ for all $r\in R$.
\end{lem}

\begin{proof}
  Fix the notation from Definition~\ref{defn:I-adic-seminorm-norm}.
  Consider $r\in R$ with $|r|=p^{-v}$. Then $r\in I^{v}$,
  and $\phi(I)\subseteq J$ implies $\phi(r)\in J^{v}$. That is
  $|\phi(r)|\leq p^{-v} = |r|$.
\end{proof}

\begin{defn}\label{defn:seminormed-normed-banach-modules}
Fix a seminormed ring $R$.
A \emph{(non-Archimedean) seminormed $R$-module} is an $R$-module
$M$ equipped with a map $\|\cdot\|\colon M\to\RR_{\geq0}$
such that
\begin{itemize}
\item $\|m+n\|\leq\max\{\|m\|,\|n\|\}$ for all $m,n\in M$ and
\item there is a $C>0$ such that $\|rm\|\leq C|r|\|m\|$
  for all $r\in R$, $m\in M$.
\end{itemize}
$M$ is a \emph{(non-Archimedean) normed $R$-module} if
the following implication holds for all $m\in M$: $|m|=0$
implies $m=0$. A normed $R$-module is a \emph{(non-Archimedean) Banach $R$-module}
if it is a complete metric space with respect to the metric
$(m,n)\mapsto|m-n|$.
\end{defn}

%%%%%%%%%%%%%%%%%%%%%%%%%%%%%%%%%%%%%%%%%%%%%%%%%%%%%%%%%%%%
%%%%%%%%%%%%%%%%%%%%%%%%%%%%%%%%%%%%%%%%%%%%%%%%%%%%%%%%%%%%
% Banach modules
%%%%%%%%%%%%%%%%%%%%%%%%%%%%%%%%%%%%%%%%%%%%%%%%%%%%%%%%%%%%
%%%%%%%%%%%%%%%%%%%%%%%%%%%%%%%%%%%%%%%%%%%%%%%%%%%%%%%%%%%%

\subsubsection{Categories}

Fix a seminormed ring $R$. We define the categories
\begin{equation}\label{eq:SNrmR-NrmR-BanR}
  \Ban_{R} \subseteq \Nrm_{R} \subseteq \SNrm_{R}
\end{equation}
of seminormed $R$-modules, normed $R$-modules,
and $R$-Banach modules. The morphisms are the
$R$-linear maps $\phi\colon M\to N$ which are
\emph{bounded}, that is
\begin{equation*}
  \|\phi(m)\| \leq C \|m\|
\end{equation*}
for a constant $C=C(\phi)>0$ and
every $m\in M$.

\begin{prop}\label{prop:BanRNrmRSNrmR-quasiabelian}
  The categories $\Ban_{R}$, $\Nrm_{R}$, and $\SNrm_{R}$
  are quasi-abelian. They admit enough functorial projectives.
\end{prop}

\begin{proof}
  See~\cite[Propositions 5.1.5 and 5.1.10]{benbassat2024perspectivefoundationsderivedanalytic}.
\end{proof}

We the cite the following two Lemma from~\cite[Proposition 3.14]{BBB16}.
\emph{Loc. cit.} assumes that $R$ is a Banach ring, but the arguments apply
as well for seminormed rings.

\begin{lem}\label{lem:krecokerimcoim-BanFcirc}
  Let $f\colon M\to N$ be a morphism of $R$-Banach modules. Then
  \begin{itemize}
    \item[(i)] $\ker(f)=f^{-1}(0)$ with the restriction of the norm
    on $M$,
    \item[(ii)] $\coker(f)=N/\overline{f(M)}$ with the residue norm,
    \item[(iii)] $\im(f)=\overline{f(N)}$ with the restriction of the norm on $N$, and
    \item[(iv)] $\coim(f)=M/\ker(f)$ with the residue norm.
  \end{itemize}
\end{lem}

\begin{lem}\label{lem:BanFcirc-kercoker}
  Let $f\colon M\to N$ be a morphism of $R$-Banach modules.
  \begin{itemize}
    \item[(i)] It is a monomorphism if and only if it is injective.
    \item[(ii)] It is an epimorphism if and only if $f(M)\subseteq N$ is dense.
    \item[(iii)] It is a strict monomorphism if and only if it is injective, the norm on $M$ is equivalent
to the norm induced by $N$, and $f(M)$ is a closed subset of $N$.
    \item[(iv)] It is a strict epimorphism if and only if it is surjective and
    the residue norm on $M/\ker(f)$ is equivalent to the norm on $N$.
  \end{itemize}
\end{lem}

The inclusions~(\ref{eq:SNrmR-NrmR-BanR}) admit
left adjoints: the separation and completion functors.

\begin{defn}\label{defn:separated-completion}
  \hfill
  \begin{itemize}
    \item[(i)] The \emph{separation functor} $\SNrm_{R}\to\Nrm_{R}$
      is $M\mapsto M^{\sep}:=M/\overline{\left\{0\right\}}$,
      equipped with the quotient norm. That is, the norm
      of an element $n\in M^{\sep}$ is $\inf_{m}\|m\|$,
      where the infinum runs over all preimages $m\in M$ of $n$.
    \item[(ii)] The completion functor $\Nrm_{R}\to\Ban_{R}$
      sends a seminormed $R$-module $N$ to its completion
      $\widehat{N}$; see for example~\cite[\S 1.1.7, the proof of Proposition 5]{BGR84}.
    \item[(iii)] The \emph{separated completion functor}
      $\SNrm_{R}\to\Ban_{R}$ is the composition of the
      completion functor and the separation functor. Abusing notation,
      we denote it by $M\mapsto\widehat{M}:=\widehat{M^{\sep}}$.
  \end{itemize}
\end{defn}

\begin{lem}\label{lem:sepcompl-SNrmF-BanF-exact}
  The separated completion functor%, see Definition~\ref{defn:separated-completion},
  is exact.
\end{lem}

\begin{proof}
  See~\cite[Remark 5.1.7]{benbassat2024perspectivefoundationsderivedanalytic}
  %This follows from Lemma~\ref{lem:krecokerimcoim-BanFcirc}
  %and~\cite[subsection 1.1.9, Corollary 6]{BGR84}.
  %See also~\cite[Proposition 5.1.6]{benbassat2024perspectivefoundationsderivedanalytic}.
  %It is right exact because it is a left adjoint. Left exactness
  %follows from~\cite[Proposition 5 in subsection 1.1.9]{BGR84}
  %and Lemma~\ref{lem:krecokerimcoim-BanFcirc}(i).
\end{proof}

%%%%%%%%%%%%%%%%%%%%%%%%%%%%%%%%%%%%%%%%%%%%%%%%%%%%%%%%%%%%
%%%%%%%%%%%%%%%%%%%%%%%%%%%%%%%%%%%%%%%%%%%%%%%%%%%%%%%%%%%%
% Banach modules
%%%%%%%%%%%%%%%%%%%%%%%%%%%%%%%%%%%%%%%%%%%%%%%%%%%%%%%%%%%%
%%%%%%%%%%%%%%%%%%%%%%%%%%%%%%%%%%%%%%%%%%%%%%%%%%%%%%%%%%%%

\subsubsection{Closed symmetric monoidal structures}

$R$ continuous to denote a seminormed ring.
Given seminormed $R$-modules $M$ and $N$, equip
$M\otimes_{R}N$ with the seminorm
\begin{equation*}
  \|x\|:=\inf\left\{ \max_{i=1,\dots,n}\|m_{i}\|\|n_{i}\|\colon x = \sum_{i=1}^{n}m_{i}\otimes_{R}n_{i}\right\}
\end{equation*}
for all $x\in M\otimes_{R}N$.
This defines a bifunctor
\begin{equation*}
  \SNrm_{R} \times \SNrm_{R} \to \SNrm_{R}, (M,N) \mapsto M\otimes_{R}N.
\end{equation*}
The \emph{separated tensor product} is
\begin{equation*}
  \Nrm_{R} \times \Nrm_{R} \to \Nrm_{R},
  (M,N) \mapsto M\otimes_{R}^{\sep}N
  :=\left( M\otimes_{R} N\right)^{\sep},
\end{equation*}
and the \emph{completed tensor product} is
\begin{equation*}
  \Ban_{R} \times \Ban_{R} \to \Ban_{R},
  (M,N) \mapsto M\widehat{\otimes}_{R}N
  :=\widehat{\left( M\otimes_{R} N\right)}.
\end{equation*}

\begin{lem}\label{lem:BanFcirc-structureresult}
  $\left(\SNrm_{R},R,\otimes_{R}\right)$,
  $\left(\Nrm_{R},R,\otimes_{R}^{\sep}\right)$, and
  $\left(\Ban_{R},R,\widehat{\otimes}_{R}\right)$
  are closed symmetric monoidal categories.
\end{lem}  

\begin{proof}
  See the discussion in~\cite[\S 5.1.1.2]{benbassat2024perspectivefoundationsderivedanalytic}, especially
  \emph{loc. cit.} Corollary 5.1.15.
\end{proof}

\begin{defn}
  For any two seminormed $R$-modules $M$ and $N$,
  define the \emph{internal homomorphisms}
  $\intHom_{R}\left(M,N\right)$ to be the seminormed $R$-module
  of all $R$-linear bounded functions $\phi\colon M\to N$, together with 
  the seminorm
  \begin{equation*}
    \|\phi\|:=\sup_{\substack{m\in M \\ \|m\|\neq0}}\frac{\|\phi(m)\|}{\|m\|}.
  \end{equation*}
\end{defn}

\begin{lem}
  Fix $M\in\SNrm_{R}$, $\Nrm_{R}$, or $\Ban_{R}$. Then the assignment
  $N\mapsto\intHom_{R}\left(M,N\right)$ defines a right adjoint of
  the functors $-\otimes_{R}M$, $-\otimes_{R}^{\sep}M$, or
  $-\widehat{\otimes}_{R}M$, respectively.
\end{lem}

\begin{proof}
  See the~\cite[remark following Corollary 5.1.15]{benbassat2024perspectivefoundationsderivedanalytic}.
\end{proof}

\begin{notation}
  An \emph{$R$-Banach algebra} $S$
  is a possibly non-commutative monoid object in $\Ban_{R}$.
  An \emph{$S$-Banach module} is
  a left $S$-module object.
  $\Ban_{S}:=\Mod\left(S\right)$ is the \emph{category of $S$-Banach modules},
  cf. \S\ref{subsec:notation}.
\end{notation}

%%%%%%%%%%%%%%%%%%%%%%%%%%%%%%%%%%%%%%%%%%%%%%%%%%%%%%%%%%%%
%%%%%%%%%%%%%%%%%%%%%%%%%%%%%%%%%%%%%%%%%%%%%%%%%%%%%%%%%%%%
%%%%%%%%%%%%%%%%%%%%%%%%%%%%%%%%%%%%%%%%%%%%%%%%%%%%%%%%%%%%
% Ind-Banach spaces
%%%%%%%%%%%%%%%%%%%%%%%%%%%%%%%%%%%%%%%%%%%%%%%%%%%%%%%%%%%%
%%%%%%%%%%%%%%%%%%%%%%%%%%%%%%%%%%%%%%%%%%%%%%%%%%%%%%%%%%%%
%%%%%%%%%%%%%%%%%%%%%%%%%%%%%%%%%%%%%%%%%%%%%%%%%%%%%%%%%%%%

\section{Ind-Banach modules}
\label{subsec:indBanachmodules}

Fix a Banach ring $R$. $\Ban_{R}$ is neither complete nor cocomplete.
Therefore, we consider its ind-completion, cf.~\cite[chapter 6]{KashiwaraSchapira2006}.
Note that $\widehat{\otimes}_{R}$ extends to a bifunctor
\begin{equation*}
\begin{split}
  \widehat{\otimes}_{R}\colon\Ind\left(\Ban_{R}\right)\times\Ind\left(\Ban_{R}\right)
  &\to\Ind\left(\Ban_{R}\right) \\
  \left(\text{``}\varinjlim_{i\in I}\text{"}V_{i}\right)
  \widehat{\otimes}_{R}
  \left(\text{``}\varinjlim_{j\in J}\text{"}W_{j}\right)
  &:=  \text{``}\varinjlim_{\substack{i\in I \\ j\in J}}\text{"}V_{i}\widehat{\otimes}_{R}W_{j}.
\end{split}
\end{equation*}

\begin{lem}\label{lem:indBanFcirc-structureresult}
  $\left(\Ind\left(\Ban_{R}\right),R,\widehat{\otimes}_{R}\right)$
  is a closed symmetric monoidal elementary quasi-abelian category.
  It has enough flat projectives stable under the monoidal
  structure $\widehat{\otimes}_{R}$.
  Furthermore, it has all limits and colimits.  
\end{lem}

\begin{proof}
  The first sentence follows from
  Lemma~\ref{lem:BanFcirc-structureresult}
  and~\cite[Proposition 2.1.19]{Sch99}.
  The second sentence follows again from~\cite[Proposition 2.1.19]{Sch99},
  which applies
  by the~\cite[Propositions 5.1.16 and 5.1.17]{benbassat2024perspectivefoundationsderivedanalytic}.
  The last sentence follows from~\cite[Proposition 6.1.18]{KashiwaraSchapira2006}
  because $\Ban_{R}$ has finite limits, cf. Lemma~\ref{lem:BanFcirc-structureresult}
  and~\cite[\href{https://stacks.math.columbia.edu/tag/002O}{Tag 002O}]{stacks-project}.
\end{proof}

\begin{cor}\label{cor:filteredcol-inIndBan-stronglyexact}
  Filtered colimits in $\Ind\left(\Ban_{R}\right)$ are strongly exact.
\end{cor}

\begin{proof}
  This follows from Lemma~\ref{lem:indBanFcirc-structureresult},
  by~\cite[Proposition 2.1.16]{Sch99}.
\end{proof}

\begin{notation}
  An \emph{$R$-ind-Banach algebra} $S$
  \footnote{We decide against the usage of the term
   \emph{ind-$R$-Banach algebra}. If spoken out loud,
   it might be misunderstood as an ind-($R$-Banach algebra).}
  is a possibly non-commutative monoid object in $\Ind\left(\Ban_{R}\right)$.
  An \emph{$S$-ind-Banach module} is
  a left $S$-module object.
  $\IndBan_{S}:=\Mod\left(S\right)$ is the \emph{category of $S$-ind-Banach modules},
  cf. \S\ref{subsec:notation}.
  In particular, $\IndBan_{R}=\Ind\left(\Ban_{R}\right)$ is the category
  of \emph{$R$-ind-Banach modules}.
\end{notation}

\begin{lem}\label{lem:functor-oC}
  Fix a category $\C$ and consider the canonical functor
  $\C\to\Ind\left(\C\right)$. It commutes with finite colimits.
  If $\C$ has all finite limits, then $\Ind\left(\C\right)$ has all finite limits
  as well and the canonical functor commutes with those.
\end{lem}

\begin{proof}
  This is~\cite[Corollary 6.1.6 and 6.1.17]{KashiwaraSchapira2006}.
\end{proof}

\begin{lem}\label{lem:messmse-indcat}
  Let $\E$ be a quasi-abelian category and $f\colon M\to N$
  a morphism in $\Ind(\E)$. Then $f$ is $\mathbf{P}$
  if and only if $f=\text{``}\varinjlim_{i}\text{"}f_{i}$
  where each $f_{i}$ is $\mathbf{P}$. Here,
  \begin{equation*} 
    \mathbf{P}\in
    \left\{ \text{mono, epi, strict, strict mono, strict epi}\right\}.
  \end{equation*}
\end{lem}

\begin{proof}
  This is~\cite[Proposition 2.10]{BBB16}.
  Its proof implicitly uses
  Lemma~\ref{lem:functor-oC}.
\end{proof}

\begin{cor}\label{cor:tensorproductfieldexact}
  Let $F$ denote a field, complete with respect to a non-trivial non-Archimedean
  valuation. Then, for any $F$-ind-Banach module $V$,
  $-\widehat{\otimes}_{F}V\colon\IndBan_{F}\to\IndBan_{F}$
  is exact.
\end{cor}

\begin{proof}
  Applying Lemma~\ref{lem:functor-oC} and~\ref{lem:messmse-indcat},
  we may assume that $V$ is a $k$-Banach module.
  Thus the Corollary follows from~\cite[Theorem 3.50]{BBBK18}.
\end{proof}

Fix a field $F$, complete with respect to a non-trivial non-Archimedean valuation.

\begin{notation}
  An \emph{$F$-Banach space} is an $F$-Banach module.
\end{notation}

\begin{defn}
  An $F$-ind-Banach space is \emph{bornological} if it is isomorphic to an object $\text{``}\varinjlim\text{"}_{i}E_{i}$
  where all the structural maps $E_{i}\to E_{j}$ are injective. A \emph{complete bornological $F$-vector space}
  is a bornological $F$-ind-Banach space.
  $\CBorn_{F}$ denotes the full subcategory of $\IndBan_{F}$ of complete bornological $F$-vector spaces.
\end{defn}

\begin{notation}
  Consider a diagram $i\mapsto E_{i}$ of complete bornological $F$-vector spaces.
  Denote its limit in $\CBorn_{F}$, if it exists, by $\varprojlim_{i}^{\bd}E_{i}$.
  $\varprojlim_{i}E_{i}$ is its limit in $\IndBan_{F}$.
\end{notation}

\begin{lem}\label{lem:limitCBornIndBan}
  Given a diagram $i\mapsto E_{i}$ of complete bornological $F$-vector spaces,
  $\varprojlim_{i}^{\bd}E_{i}$ exists and coincides with $\varprojlim_{i}E_{i}$.
\end{lem}

\begin{proof}
  This follows from~\cite[Remark 3.44 and Proposition 3.60]{BBB16}.
\end{proof}

\begin{lem}\label{lem:CBorntoIndBan-exact}
  The functor $\CBorn_{F}\to\IndBan_{F}$ is exact.
\end{lem}

\begin{proof}
  This is~\cite[Proposition 4.22]{Bo21}.
  \iffalse %%% alternative argument
  By Corollary~\cite[Corollary 1.2.28]{Sch99}, it suffices to check that the induced functor
  \begin{equation*}
    \LH\left( \CBorn_{F} \right) \to \LH\left( \IndBan_{F} \right)
  \end{equation*}
  between the left hearts is exact, cf. \emph{loc. cit.} Definition 1.2.18.
  The result follows from~\cite[Proposition 5.16(b)]{PS01}.
  We remark that this reference operates in the archimedean context,
  but the proof goes through in our setting as well.
  \fi %%% comment ends
\end{proof}

$\CBorn_{F}$ carries a symmetric monoidal operation which we denote by $\widehat{\otimes}_{F}^{\bd}$.
In this article, we do not need its precise~\cite[Definition 3.57]{BBB16}
but only the following result.

\begin{lem}\label{lem:hotimesbd-is-hotimes}
  Consider two complete bornological $F$-vector spaces $V$ and $W$, which
  are inverse limits of Banach spaces. Then there is a functorial isomorphism
  \begin{equation*}
    V\widehat{\otimes}_{F}W \isomap V\widehat{\otimes}_{F}^{\bd}W
  \end{equation*}
  of $F$-ind-Banach spaces.
\end{lem}

\begin{proof}
  Both $V$ and $W$ are \emph{proper} as bornological spaces,
  cf.~\cite[Definition 3.62]{BBB16}. This follows from~\cite[Proposition 3.11]{BBBK18},
  together with the fact that Banach spaces are proper, which follows directly
  from the definition. Lemma~\ref{lem:hotimesbd-is-hotimes}
  thus follows from~\cite[Proposition 3.64]{BBB16}.
\end{proof}

Fix a \emph{pseudo-uniformiser}, that is an element $\pi\in F^{\circ}$
such that $0<|\pi|<1$.

\iffalse %%% COMMENT BEGINNS
\begin{lem}\label{lem:Andreaspowerseriesstronglyexactlemma}
  $-\widehat{\otimes}_{F}\varprojlim_{r}F\left\<\zeta_{1},\dots,\zeta_{d}\right\>
  \colon\IndBan_{F}\to\IndBan_{F}$
  is strongly exact, given formal variables $\zeta_{1},\dots,\zeta_{d}$.
\end{lem}

\begin{proof}
  It preserves cokernels because the monoidal category $\IndBan_{k}$
  is closed. To show that it preserves kernels of arbitrary maps,
  apply~\cite[Remark 2.3]{BBB16} and Corollary~\ref{cor:filteredcol-inIndBan-stronglyexact};
  thus it suffices to check that it preserves kernels of maps
  between $k$-Banach spaces. Given a $k$-Banach space $V$,
  we compute with the Lemma~\ref{lem:limitCBornIndBan} and~\ref{lem:hotimesbd-is-hotimes}:
  \begin{equation*}
    V\widehat{\otimes}_{F}\varprojlim_{r}F\left\<\zeta_{1},\dots,\zeta_{d}\right\>
    \cong
    V\widehat{\otimes}_{F}^{\bd}{\varprojlim_{r}}^{\bd}F\left\<\zeta_{1},\dots,\zeta_{d}\right\>.
  \end{equation*}
  As $\CBorn_{F}\hookrightarrow\IndBan_{F}$ preserves kernels,
  \cite[Corollary 5.36]{Bo21} gives the result.
\end{proof}
\fi %%% COMMENT ENDS

\begin{lem}\label{lem:Andreaspowerseriesstronglyexactlemma}
  $-\widehat{\otimes}_{F}\varprojlim_{r}F\left\<\pi^{r}\zeta_{1},\dots,\pi^{r}\zeta_{d}\right\>
  \colon\IndBan_{F}\to\IndBan_{F}$
  is strongly exact, given formal variables $\zeta_{1},\dots,\zeta_{d}$.
\end{lem}

\begin{proof}
  It preserves cokernels because the monoidal category $\IndBan_{k}$
  is closed. To show that it preserves kernels of arbitrary maps,
  apply~\cite[Remark 2.3]{BBB16} and Corollary~\ref{cor:filteredcol-inIndBan-stronglyexact};
  thus it suffices to check that it preserves kernels of maps
  between $k$-Banach spaces. Given a $k$-Banach space $V$,
  we compute with the Lemma~\ref{lem:limitCBornIndBan} and~\ref{lem:hotimesbd-is-hotimes}:
  \begin{equation*}
    V\widehat{\otimes}_{F}\varprojlim_{r}F\left\<\pi^{r}\zeta_{1},\dots,\pi^{r}\zeta_{d}\right\>
    \cong
    V\widehat{\otimes}_{F}^{\bd}{\varprojlim_{r}}^{\bd}F\left\<\pi^{r}\zeta_{1},\dots,\pi^{r}\zeta_{d}\right\>.
  \end{equation*}
  As $\CBorn_{F}\hookrightarrow\IndBan_{F}$ preserves kernels,
  \cite[Corollary 5.36]{Bo21} gives the result.
\end{proof}

\begin{lem}\label{lem:thenormonAhotimesB-Iadic-piadic-reconstructionpaper}
  Let $A$, $B$ be a commutative $F^{\circ}$-Banach algebras.
  $A$ carries the $I$-adic norm,
  for some ideal $I\subseteq A$, and $B$ carries
  the $\pi$-adic norm. Then the norm on
  $A\widehat{\otimes}_{F^{\circ}}B$ induces the
  $\left( \left\{ a\widehat{\otimes}1\colon a\in I \right\} \cup \left\{ 1 \widehat{\otimes} \pi \right\} \right)$-adic
  topology.
\end{lem}

\begin{proof}
  We check that
  $A\otimes_{F^{\circ}} B$ carries the
  $J:=\left( \left\{ a\otimes1\colon a\in I \right\} \cup \left\{ 1\otimes\pi \right\} \right)$-adic topology,
  using~\cite[Exercise 7.9]{Ei95}.
  In fact, for all $x\in A\otimes_{F^{\circ}} B$, we claim that $\|x\|\leq|\pi|^{v}$ if and only if $x\in J^{v}$.
  The direction $\Leftarrow$ is clear. To show $\implies$, assume $\|x\|\leq|\pi|^{v}$. Then we may write
  $x=\sum_{i=1}^{n}y_{i}\otimes z_{i}$ such that $\max_{i=1,\dots,n}\|y_{i}\|\|z_{i}\|\leq|\pi|^{v}$.
  For all $i=1,\dots,n$, this implies $\|y_{i}\|\|z_{i}\|\leq|\pi|^{v}$. That is, there exists a $w\in\NN$, $w\leq v$, such that
  $y_{i}\in I^{w}$ and $z_{i}\in\left(\pi\right)^{v-w}$. Consequently, $y_{i}\otimes z_{i}\in J^{v}$
  for all $i=1,\dots,n$, and we conclude $x\in J^{v}$.
\end{proof}

%%%%%%%%%%%%%%%%%%%%%%%%%%%%%%%%%%%%%%%%%%%%%%%%%%%%%%%%%%%%
%%%%%%%%%%%%%%%%%%%%%%%%%%%%%%%%%%%%%%%%%%%%%%%%%%%%%%%%%%%%
% FUNCTIONAL ANALYSIS
%%%%%%%%%%%%%%%%%%%%%%%%%%%%%%%%%%%%%%%%%%%%%%%%%%%%%%%%%%%%
%%%%%%%%%%%%%%%%%%%%%%%%%%%%%%%%%%%%%%%%%%%%%%%%%%%%%%%%%%%%

\section{Ind-$G$-$R$-Banach modules}

Fix a Banach ring $R$. Consider a profinite group $G$ and an $R$-ind-Banach module $M$.
Both live in different categories, therefore we find it hard to define a notion
of a $G$-action on $M$. Luckily, the following ad hoc definition suffices for our purposes.

\begin{defn}
  We define the category \emph{$\Ban_{R}\left(G\right)$ of $G$-$R$-Banach modules}.
  \begin{itemize}
    \item[(i)] Its objects are $R$-Banach modules $M$, together with continuous $G\times M\to M$-actions.
    \item[(ii)] The morphisms are the $G$-equivariant bounded $R$-linear maps.
  \end{itemize}
  Its ind-completion $\IndBan_{R}\left(G\right):=\Ind\left(\Ban_{R}\left(G\right)\right)$
  is the \emph{category of ind-$G$-$R$-Banach modules}.
\end{defn}

Given two $G$-$R$-Banach modules $M$ and $N$, the usual completed tensor product
$M\widehat{\otimes}_{R}N$ carries the continuous $G$-action determined by
\begin{equation*}
  g \cdot \left( m \widehat{\otimes} n \right)
  := (g\cdot m) \widehat{\otimes} (g\cdot n)
\end{equation*}
for $g\in G$, $m\in M$, and $n\in N$. Consequently,
\begin{equation*} 
  \left(\text{``}\varinjlim_{i\in I}\text{"}V_{i}\right)
  \widehat{\otimes}_{R}
  \left(\text{``}\varinjlim_{j\in J}\text{"}W_{j}\right)
  :=\text{``}\varinjlim_{\substack{i\in I \\ j\in J}}\text{"}V_{i}\widehat{\otimes}_{R}W_{j}
\end{equation*}
defines a monoidal operation on $\IndBan_{R}\left(G\right)$
with unit $R$, carrying the trivial $G$-action.  

\begin{defn}
  An \emph{ind-$G$-$R$-Banach algebra} is a monoid object in $\IndBan_{R}$ with respect to the
  completed tensor product $\widehat{\otimes}_{R}$ described above and unit $R$, equipped with
  the trivial action.
\end{defn}

%%%%%%%%%%%%%%%%%%%%%%%%%%%%%%%%%%%%%%%%%%%%%%%%%%%%%%%%%%%%
%%%%%%%%%%%%%%%%%%%%%%%%%%%%%%%%%%%%%%%%%%%%%%%%%%%%%%%%%%%%
%%%%%%%%%%%%%%%%%%%%%%%%%%%%%%%%%%%%%%%%%%%%%%%%%%%%%%%%%%%%
% Ind-Banach spaces
%%%%%%%%%%%%%%%%%%%%%%%%%%%%%%%%%%%%%%%%%%%%%%%%%%%%%%%%%%%%
%%%%%%%%%%%%%%%%%%%%%%%%%%%%%%%%%%%%%%%%%%%%%%%%%%%%%%%%%%%%
%%%%%%%%%%%%%%%%%%%%%%%%%%%%%%%%%%%%%%%%%%%%%%%%%%%%%%%%%%%%

\section{The left heart of the category of ind-Banach modules}
\label{subsec:LHindBanachmodules-reconstructionpaper}

In \S\ref{subsec:LHindBanachmodules-reconstructionpaper},
we fix a field $F=R$, complete with respect to nontrivial non-Archimedean valuation.
Furthermore, $\pi\in F$ denotes a pseudo-uniformiser, that is $0<|\pi|<1$.

\begin{notation}
  Write $\IndBan_{\I\left(F\right)}:=\LH\left(\IndBan_{F}\right)$
  for the left heart, cf. \S\ref{subsec:qabelian-cat-recpaper}.
\end{notation}

We refer the reader to~\cite[\S 4.2]{Bo21} an in-depth study
of $\IndBan_{\I\left(F\right)}$. \emph{Loc. cit.} Lemma 4.20
and~\cite[Proposition 2.1.12]{Sch99} imply that it is an elementary
abelian category. By~\cite[Proposition 2.1.15 and 2.1.16]{Sch99},
it is complete with exact products and cocomplete with
exact direct sums and filtered colimits. It has enough projectives
and enough injectives. Furthermore, $\IndBan_{\I\left(F\right)}$ is canonically
closed symmetric monoidal, cf.~\cite[Lemma 3.6 and the remark following it,
which applies thanks to Corollary 4.24]{Bo21}
The monoidal operation
and internal homomorphisms are given as follows:
for all $V,W\in\IndBan_{\I\left(F\right)}$,
\begin{equation}\label{eq:internaloperations-IndBanIF-recpaper}
\begin{split}
  V\widehat{\otimes}_{\I\left(F\right)}W
    &:= \LHo^{0}\left(V\widehat{\otimes}_{F}^{\rL}W\right), \text{ and} \\
  \intHom_{\I\left(F\right)}\left( V , W \right)
    &:= \LHo^{0}\left(\R\intHom_{F}\left( V , W \right)\right).
\end{split}
\end{equation}  
Here, we derive $\widehat{\otimes}_{F}$ and $\intHom_{F}$
via projectives resolutions, cf. Lemma~\ref{prop:BanRNrmRSNrmR-quasiabelian}.
$\I\left(F\right)$ is the unit.

\begin{lem}\label{lem:Ifunctor-stronglymonoidal-reconstructionpaper}
  For any two $F$-ind-Banach spaces $V$ and $W$, we have the canonical isomorphism
  \begin{equation*}
    \I\left(V\right) \widehat{\otimes}_{\I\left(F\right)} \I\left(W\right)
    \isomap \I\left( V\widehat{\otimes}_{F} W\right).
  \end{equation*}
\end{lem}

\begin{proof}
  This follows from Corollary~\ref{cor:tensorproductfieldexact}
  and the definition of $\widehat{\otimes}_{\I\left(F\right)}$.
\end{proof}

%%% COMMENT ENDS

%%%%%%%%%%%%%%%%%%%%%%%%%%%%%%%%%%%%%%%%%%%%%%%%%%%%%%%%%%%%
%%%%%%%%%%%%%%%%%%%%%%%%%%%%%%%%%%%%%%%%%%%%%%%%%%%%%%%%%%%%
%%%%%%%%%%%%%%%%%%%%%%%%%%%%%%%%%%%%%%%%%%%%%%%%%%%%%%%%%%%%
% Categories of sheaves
%%%%%%%%%%%%%%%%%%%%%%%%%%%%%%%%%%%%%%%%%%%%%%%%%%%%%%%%%%%%
%%%%%%%%%%%%%%%%%%%%%%%%%%%%%%%%%%%%%%%%%%%%%%%%%%%%%%%%%%%%
%%%%%%%%%%%%%%%%%%%%%%%%%%%%%%%%%%%%%%%%%%%%%%%%%%%%%%%%%%%%

\section{Categories of sheaves}
\label{subsec:functional-analysis-sheaves}

Fix a quasi-abelian category $\E$ and a site $X$. The following Definition~\ref{defn:sheaf-reconstructionpaper}
is a slight generalisation of the original definition in~\cite[\S 2.2.1]{Sch99}.

\begin{defn}\label{defn:sheaf-reconstructionpaper}
  An \emph{$\E$-presheaf} or \emph{presheaf with values in $\E$}
  is a functor
  \begin{equation*}
    \mathcal{F}\colon X^{\op}\maps\E.
  \end{equation*}
  An \emph{$\E$-sheaf} or \emph{sheaf with values in $\E$} is
  an $\E$-presheaf $\mathcal{F}$ such that for any open
  $U\in X$ and any covering $\mathfrak{U}$ of $U$,
  \begin{itemize}
    \item the products $\prod_{V\in\mathfrak{U}}\mathcal{F}(V)$ and
    $\prod_{W,W^{\prime}\in\mathfrak{U}}\mathcal{F}\left(W\times_{U} W^{\prime}\right)$
    exist, and
    \item $0\to\mathcal{F}(U)
     \to\prod_{V\in\mathfrak{U}}\mathcal{F}(V)
     \to\prod_{W,W^{\prime}\in\mathfrak{U}}\mathcal{F}\left(W\times_{U} W^{\prime}\right)$
    is strictly exact.
  \end{itemize}
\end{defn}

The morphisms in the category of presheaves $\Psh(X,\E)$ valued in $\E$
are the natural transformations. The category of sheaves $\Sh(X,\E)\subseteq\Psh(X,\E)$
is a full triangulated subcategory. Schneiders explains in~\cite[\S 2.2]{Sch99} that
these categories are well-behaved when $\E$ is elementary,
cf. \emph{loc. cit.} Definition 2.1.10. For example, if $\E$ is elementary,
$\Sh(X,\E)$ is a complete and cocomplete quasi-abelian category by~\cite[Propositions 2.2.7 and 2.2.8]{Sch99}.
Furthermore, the inclusion $\Psh(X,\E)\to\Sh(X,\E)$ admits a left adjoint $\cal{F}\mapsto\cal{F}^{\sh}$
which we call \emph{sheafification}, cf. \emph{loc. cit.} Proposition 2.2.6.

\begin{lem}\label{lem:calFsheaf-FcirccalF-sheaf-exactfunctor}
  Suppose that $X$ admits only finite coverings and
  consider a strongly left exact functor $\F\colon\E_{1}\to\E_{2}$
  between two quasi-abelian categories
  which admit all finite products.
  Then for any $\E_{1}$-sheaf $\mathcal{F}$, $\F\circ\mathcal{F}$ is a $\E_{2}$-sheaf.
\end{lem}

\begin{proof}
  This is because $\F$ commutes with finite products, cf.~\cite[Remark 1.1.13]{Sch99}.
\end{proof}

From now on, $\mathbf{E}$ is elementary.

\begin{lem}\label{lem:directsumsheaves-finitecoverings}
  Suppose $X$ admits only finite coverings.
  Consider an arbitrary collection of $\E$-valued sheaves $\cal{F}_{\alpha}$.
  Then we have the canonical isomorphism
  \begin{equation*}
    \bigoplus_{\alpha}\cal{F}_{\alpha}(U) \isomap\left(\bigoplus_{\alpha}\cal{F}_{\alpha}\right)(U).
  \end{equation*}
\end{lem}

\begin{proof}
  By~\cite[Lemma 3.15]{Bo21},
  we have to check that $U\mapsto\bigoplus_{\alpha}F_{\alpha}(U)$ is a sheaf.
  To do this, we consider a covering $\mathfrak{U}$ of $U\in X$
  and the associated strictly exact sequences
  \begin{equation}\label{eq:sectionsofunderlineF-exactsequence--lem:directsumsheaves-finitecoverings}
    0\to\mathcal{F}_{\alpha}(U)
     \to\prod_{V\in\mathfrak{U}}\mathcal{F}_{\alpha}(V)
     \to\prod_{W,W^{\prime}\in\mathfrak{U}}\mathcal{F}_{\alpha}\left(W\times_{U} W^{\prime}\right).
  \end{equation}
  $\mathfrak{U}$ is finite,
  and so are all the products in~(\ref{eq:sectionsofunderlineF-exactsequence--lem:directsumsheaves-finitecoverings}).
  Since finite products commute with arbitrary direct sums
  and direct sums are strongly exact in $\E$, cf.~\cite[Proposition 2.1.15(b)]{Sch99},
  \begin{equation*}
    0\to\bigoplus_{\alpha}\mathcal{F}_{\alpha}(U)
     \to\prod_{V\in\mathfrak{U}}\bigoplus_{\alpha}\mathcal{F}_{\alpha}(V)
     \to\prod_{W,W^{\prime}\in\mathfrak{U}}\bigoplus_{\alpha}\mathcal{F}_{\alpha}\left(W\times_{U} W^{\prime}\right)
  \end{equation*}
  is strictly exact, as desired.
\end{proof}

Now assume that $\E$ is closed symmetric monoidal with
unit $1\in\E$ and tensor product $\otimes$. This gives
a symmetric monoidal structure on $\Sh(X,\E)$ as follows.
$1_{X}$ is the constant sheaf on $X$.
Next, define for any two  presheaves
$\mathcal{F}$ and $\mathcal{G}$ a presheaf
\begin{equation*}
  \mathcal{F}\otimes_{\psh}\mathcal{G} \colon
  U\mapsto \mathcal{F}(U)\otimes\mathcal{G}(U).
\end{equation*}
If $\mathcal{F}$ and $\mathcal{G}$ are sheaves,
$\mathcal{F}\otimes\mathcal{G}:=
  \left(\mathcal{F}\otimes_{\psh}\mathcal{G}\right)^{\sh}$.
This gives a bifunctor
\begin{equation*}
  -\otimes-\colon
   \Sh(X,\E)\times\Sh(X,\E)\to\Sh(X,\E).
 \end{equation*}

\begin{lem}[{\cite[Lemma 2.15]{Bo21}}]\label{lem:ShXE-closedcat}
  $\left(\Sh(X,\E),1_{X},\otimes\right)$
  is closed symmetric monoidal.
\end{lem}

We use the following fact without further reference:
A monoid structure on an $\E$-sheaf $\mathcal{R}$ is
equivalent to the data of a monoid structure on
the sections of $\mathcal{R}$ such that the restriction
maps $\mathcal{R}(U)\to\mathcal{R}(V)$ are multiplicative.
Similarly, the structure which makes an $\E$-sheaf
an $\mathcal{R}$-module object is equivalent to
section-wise module structures which commute with
the restriction maps in the obvious way.

\begin{lem}\label{lem:sh-strongly-monoidal}
  $\left(\mathcal{F}\otimes_{\psh}\mathcal{G}\right)^{\sh}
    \stackrel{\sim}{\longrightarrow}\mathcal{F}^{\sh}\otimes\mathcal{G}^{\sh}$
  is an isomorphism for any two $\E$-presheaves $\mathcal{F}$ and
  $\mathcal{G}$ on $X$. That is sheafification
  is strongly monoidal.
\end{lem}

\begin{proof}
  See~\cite[Lemma 2.16]{Bo21}.
  %This follows straightforwardly from the adjunctions.
%  \begin{align*}
%   \Hom_{F}\left(\mathcal{F}^{\sh}\otimes_{F}\mathcal{G}^{\sh},-\right)
%   &=\Hom_{F}\left(\mathcal{F}^{\sh},\shHom_{K}\left(\mathcal{G}^{\sh},-\right)\right) \\
%    &=\Hom_{F}\left(\mathcal{F},\shHom_{K}\left(\mathcal{G},-\right)\right) \\
%    &=\Hom_{F}\left(\mathcal{F}\otimes_{F,\psh}\mathcal{G},-\right) \\
%    &=\Hom_{F}\left(\left(\mathcal{F}\otimes_{F,\psh}\mathcal{G}\right)^{\sh},-\right).
%  \end{align*}
\end{proof}

We keep our assumptions on $\E$ fixed and consider a morphism $f\colon X\to Y$
of sites, which is given by a functor $f^{-1}\colon Y\to X$ between the
underlying categories.
\begin{align*}
  f^{\psh,-1}\colon\Psh\left(Y,\E\right) &\to \Psh\left(X,\E\right) \\
  f^{\psh,-1}\left(\mathcal{F}\right)(U) &:= \varinjlim_{U\to f^{-1}(V)}\mathcal{F}(V)
\end{align*}
is the \emph{presheaf inverse image}.
The \emph{direct image functor} is
\begin{align*}
  f_{*}\colon\Psh\left(X,\E\right) &\to \Psh\left(Y,\E\right) \\
  f_{*}\left(\mathcal{F}\right)(U) &:= \mathcal{F}\left(f^{-1}(U)\right),
\end{align*}
as discussed in~\cite[\S 2.6]{Bo21}. $f_{*}$ sends sheaves
to sheaves but $f^{\psh,-1}$ does, in general, not. This is why we
define the direct image functor $f^{-1}:=\cdot^{\sh}\circ f^{\psh,-1}$.

\iffalse %%% not needed anymore
\begin{lem}\label{lem:pullback-is-monoidal}
  There is a natural isomorphism
  $f^{-1}\left(\mathcal{F}\otimes\mathcal{G}\right)\cong f^{-1}\mathcal{F}\otimes f^{-1}\mathcal{G}$
  for any two $\E$-sheaves $\mathcal{F}$ and $\mathcal{G}$ on $X$. That is,
  $f^{-1}$ is strongly monoidal.
\end{lem}

\begin{proof}
  See~\cite[Lemma 2.28]{Bo21}.
\end{proof}
\fi %%%comment ends

\begin{lem}\label{lem:sh-commuteswith-restriction}
  Sheafification commutes with restriction.
\end{lem}

\begin{proof}
This follows from the adjunctions, in particular $f^{-1}\dashv f_{*}$.
%  Consider an open $U \in X$ and fix an $\E$-sheaf $\mathcal{F}$
%  on $X$. We have a morphism
%  $\mathcal{F}|_{U}\maps\mathcal{F}^{\sh}|_{U}$, and the universal
%  property of sheafification yields a functorial map
%  $\left(\mathcal{F}|_{U}\right)^{\sh}\maps\mathcal{F}^{\sh}|_{U}$.
%  One sees from the adjunctions that it is
%  an isomorphism.
%  One sees that it is an isomorphism via the Yoneda lemma:
% \begin{align*}
%    \Hom_{F}\left(\left(\mathcal{F}|_{U}\right)^{\sh},-\right)
%    &=\Hom_{F}\left(\mathcal{F}|_{U},-\right) \\
%    &=\Hom_{F}\left(i^{-1}\mathcal{F},-\right) \\
%    &=\Hom_{F}\left(\mathcal{F},i_{*}-\right) \\
%    &=\Hom_{F}\left(\mathcal{F}^{\sh},i_{*}-\right) \\
%    &=\Hom_{F}\left(i^{-1}\mathcal{F}^{\sh},-\right) \\
%    &=\Hom_{F}\left(\mathcal{F}^{\sh}|_{U},-\right).
%  \end{align*}
\end{proof}

Lemma~\ref{lem:sh-commuteswith-restriction} implies the following.

\begin{lem}\label{lem:monoidal-structure-commutes-restriction}
  For any two $\E$-presheaves $\mathcal{F}$ and
  $\mathcal{G}$ on $X$, there is a functorial isomorphism
  $\mathcal{F}|_{U}\otimes\mathcal{G}|_{U}\cong\left(\mathcal{F}\otimes\mathcal{G}\right)|_{U}$
  for any $U\in X$.
\end{lem}

%%%%%%%%%%%%%%%%%%%%%%%%%%%%%%%%%%%%%%%%%%%%%%%%%%%%%%%%%%%%
%%%%%%%%%%%%%%%%%%%%%%%%%%%%%%%%%%%%%%%%%%%%%%%%%%%%%%%%%%%%
% Some category theory
%%%%%%%%%%%%%%%%%%%%%%%%%%%%%%%%%%%%%%%%%%%%%%%%%%%%%%%%%%%%
%%%%%%%%%%%%%%%%%%%%%%%%%%%%%%%%%%%%%%%%%%%%%%%%%%%%%%%%%%%%

\section{Sites with many quasi-compact open subsets}
\label{subsec:siteqc-reducesheafcondtofincov}

We fix again an elementary quasi-abelian category
$\E$ and a site $X$.

\begin{prop}\label{prop:siteqc-reducesheafcondtofincov}
  Suppose that any
  $U\in X$ is quasi-compact. Let $\mathcal{F}$
  denote an $\E$-presheaf on $X$ such that the sequence
  \begin{equation*}
    0\maps\mathcal{F}(U)\maps\prod_{V\in\mathfrak{U}}\mathcal{F}(V)\maps\prod_{W,W^{\prime}\in\mathfrak{U}}\mathcal{F}(W\times_{U}W^{\prime})
  \end{equation*}
  is strictly exact for every finite covering $\mathfrak{U}$ of any $U\in X$.
  Then $\mathcal{F}$ is a sheaf.
\end{prop}

By~\cite[Corollary 1.2.28]{Sch99},
Proposition~\ref{prop:siteqc-reducesheafcondtofincov} follows
from the following.

\begin{lem}\label{lem:siteqc-reducesheafcondtofincov}
  Proposition~\ref{prop:siteqc-reducesheafcondtofincov} holds
  for any elementary abelian category $\E=\bA$.
\end{lem}

\begin{proof}
  Let $\mathfrak{U}=\set{U_{i}\maps U}_{i\in I}$ denote a covering.
  Since $U$ is quasicompact, we find a finite subcovering
  $\widetilde{\mathfrak{U}}=\set{U_{\widetilde{i}}\maps U}_{\widetilde{i}\in\widetilde{I}}$
  where $\widetilde{I}\subseteq I$. Consider the commutative diagram
  \begin{equation}\label{cd:siteqc-reducesheafcondtofincov-1}
    \begin{tikzcd}
      0 \arrow{r} &
      \mathcal{F}(U) \arrow{r}{\alpha} \arrow[equal]{d} &
      \prod_{i\in I}\mathcal{F}(U_{i}) \arrow{r}{\beta} \arrow{d}{\pi_{\widetilde{I}}} &
      \prod_{i,j\in I}\mathcal{F}(U_{i}\times_{U}U_{j}) \arrow{d}{\pi_{\widetilde{I}\times\widetilde{I}}} \\
      0 \arrow{r} &
      \mathcal{F}(U) \arrow{r}{\widetilde{\alpha}} &
      \prod_{\widetilde{i}\in\widetilde{I}}\mathcal{F}(U_{\widetilde{i}}) \arrow{r}{\widetilde{\beta}} &
      \prod_{\widetilde{i},\widetilde{j}\in\widetilde{I}}\mathcal{F}(U_{\widetilde{i}}\times_{U}U_{\widetilde{j}})
    \end{tikzcd}
  \end{equation}
  where both $\pi_{\widetilde{I}}$ and $\pi_{\widetilde{I}\times\widetilde{I}}$
  are the projections. It follows from the commutativity of
  the left square that $\alpha$ is a monomorphism.
  It remains to show that the canonical morphism $\Psi\colon\im\alpha\to\ker\beta$
  is an isomorphism. We introduce some notation.
  \begin{itemize}
  \item Let $\widetilde{\Phi}$ denote the inverse of
    $\widetilde{\Psi}\colon\im\widetilde{\alpha}\to\ker\widetilde{\beta}$.
  \item $\pi_{i_0}$ is the projection $\prod_{i\in I}\mathcal{F}(U_{i})\maps\mathcal{F}(U_{i_0})$
  and $\alpha_{i_0}:=\pi_{i_0}\circ\alpha$ for every $i_{0}\in I$.
  \item $\alpha^{\prime-1}\colon\im\alpha\to\mathcal{F}(U)$
    is the inverse of the morphism
    $\alpha^{\prime}\colon\mathcal{F}(U)\to\im\alpha$
    induced by $\alpha$. It is an isomorphism because
    $\alpha$ is a monomorphism and $\bA$ is abelian.
    Similarly, $\widetilde{\alpha}^{\prime-1}\colon\im\widetilde{\alpha}\to\mathcal{F}(U)$
    denotes the inverse of the morphism
    $\widetilde{\alpha}^{\prime}\colon\mathcal{F}(U)\to\im\widetilde{\alpha}$
    induced by $\widetilde{\alpha}$.
  \item $\pi_{\widetilde{I}}$ restricts to maps
    $\pi_{\widetilde{I}}^{\im\alpha}\colon\im\alpha\to\im\widetilde{\alpha}$ and
    $\pi_{\widetilde{I}}^{\ker\beta}\colon\ker\beta\to\ker\widetilde{\beta}$.
  \item Write $\pi_{i_0}^{\im\alpha}$ for the composition
    $\im\alpha\hookrightarrow\prod_{i\in I}\mathcal{F}(U_{i})\stackrel{\pi_{i_0}}{\longrightarrow}\mathcal{F}(U_{i_{0}})$
    and $\pi_{i_0}^{\ker\beta}$ for the composition
    $\ker\beta\hookrightarrow\prod_{i\in I}\mathcal{F}(U_{i})\stackrel{\pi_{i_0}}{\longrightarrow}\mathcal{F}(U_{i_{0}})$.
  \end{itemize}
  
  \begin{lem}\label{lem:siteqc-reducesheafcondtofincov-1}
    We have identities
    \begin{align*}
      \pi_{i_0}^{\im\alpha}&=\alpha_{i_0}\circ\widetilde{\alpha}^{\prime-1}\circ\pi_{\widetilde{I}}^{\im\alpha}, \text{ and} \\
      \pi_{i_0}^{\ker\beta}&=\alpha_{i_0}\circ\widetilde{\alpha}^{\prime-1}\circ\widetilde{\Phi}\circ\pi_{\widetilde{I}}^{\ker\beta}
    \end{align*}
    for every $i_{0}\in I$.
  \end{lem}
  
  \begin{proof}
    Compose
    $\widetilde{\alpha}^{\prime} =\pi_{\widetilde{I}}^{\im\alpha}\circ\alpha^{\prime}$
    with $\widetilde{\alpha}^{\prime-1}$ on the left and
    $\alpha^{\prime-1}$ on the right to get
    \begin{equation*}
      \alpha^{\prime-1} =\widetilde{\alpha}^{\prime-1}\circ\pi_{\widetilde{I}}^{\im\alpha}.
    \end{equation*}
    Now we compose with $\alpha^{\prime}$ on the left
    and again with $\pi_{i_0}^{\im\alpha}$ on the left. This yields
    \begin{align*}
      \pi_{i_0}^{\im\alpha}&=
      \pi_{i_0}^{\im\alpha}\circ
      \left(\alpha^{\prime}\circ\widetilde{\alpha}^{\prime-1}\circ\pi_{\widetilde{I}}^{\im\alpha}\right) \\
      &=\left(\pi_{i_0}^{\im\alpha}\circ\alpha^{\prime}\right)\circ\widetilde{\alpha}^{\prime-1}\circ\pi_{\widetilde{I}}^{\im\alpha} \\
      &=\alpha_{i_0}\circ\widetilde{\alpha}^{\prime-1}\circ\pi_{\widetilde{I}}^{\im\alpha}
    \end{align*}
    which is the first identity stated above in Lemma~\ref{lem:siteqc-reducesheafcondtofincov-1}.

    To prove the second identity, we may assume without loss of generality
    that $\bA$ is small; otherwise we pass to a suitable subcategory.
    The Freyd-Mitchell
    Embedding Theorem~\cite[Theorem 1.6.1]{Wei94}
    gives that $\bA$ is the category of modules over some ring.
    Let $(s_{i})_{i\in I}\in\ker\beta$. Then
    $\left(\widetilde{\Phi}\circ\pi_{\widetilde{I}}^{\ker\beta}\right)\left(\left(s_{i}\right)_{i\in I}\right)=\left(s_{\widetilde{i}}\right)_{\widetilde{i}\in \widetilde{I}}$
    lies in the kernel of $\widetilde{\beta}$,
    thus it lies in the image of $\widetilde{\alpha}$. That is,
    there exists an $s\in\mathcal{F}(U)$ such that
    $\widetilde{\alpha}(s)=\left(s_{\widetilde{i}}\right)_{\widetilde{i}\in \widetilde{I}}$.
    With other words, $s=\left(\widetilde{\alpha}^{\prime-1}\circ\widetilde{\Phi}\circ\pi_{\widetilde{I}}^{\ker\beta}\right)\left(\left(s_{i}\right)_{i\in I}\right)$.
    We have to show that $\alpha_{i_0}(s)=s_{i_0}$ for all $i_{0}\in I$,
    since $\pi_{i_0}^{\ker\beta}\left(\left(s_{i}\right)_{i\in I}\right)=s_{i_0}$.

    Because $\widetilde{\mathfrak{U}}=\set{U_{\widetilde{i}}\maps U}_{\widetilde{i}\in\widetilde{I}}$
    is a finite covering of $U$,
    $\set{U_{\widetilde{i}}\times_{U}U_{i_0}\maps U_{i_0}}_{\widetilde{i}\in\widetilde{I}}$
    is a finite covering of $U_{i_0}$.
    Thus the sequence
    \begin{equation*}
    \adjustbox{scale=0.9,center}{%
      \begin{tikzcd}
        0 \arrow{r} &
        \mathcal{F}(U_{i_{0}}) \arrow{r} &
        \prod_{\widetilde{i}\in\widetilde{I}}
        \mathcal{F}(U_{\widetilde{i}}\times_{U}U_{i_0}) \arrow{r} &
        \prod_{\widetilde{i},\widetilde{j}\in\widetilde{I}}
        \mathcal{F}((U_{\widetilde{i}}\times_{U}U_{i_0})\times_{U_{i_{0}}}(U_{\widetilde{j}}\times_{U}U_{i_0}))
      \end{tikzcd} }
    \end{equation*}
    is exact. That is, the element $s_{i_0}$
    is uniquely determined by its restrictions to the open
    sets $U_{\widetilde{i}}\times_{U}U_{i_0}$. Therefore, the computation
    \begin{equation*}
      s_{i_{0}}|_{U_{\widetilde{i}}\times_{U}U_{i_0}}
      =s_{\widetilde{i}}|_{U_{\widetilde{i}}\times_{U}U_{i_0}}
      =\left(s|_{U_{\widetilde{i}}}\right)|_{U_{\widetilde{i}}\times_{U}U_{i_0}}
      =s|_{U_{\widetilde{i}}\times_{U}U_{i_0}}
    \end{equation*}
    implies $\alpha_{i_0}(s)=s|_{U_{i_0}}=s_{i_0}$.
    Note that the first equality
    $s_{i_{0}}|_{U_{\widetilde{i}}\times_{U}U_{i_0}}=s_{\widetilde{i}}|_{U_{\widetilde{i}}\times_{U}U_{i_0}}$
    follows from $(s_{i})_{i\in I}\in\ker\beta$.
    This finishes the proof.
  \end{proof}

  We claim that
  \begin{equation*}
    \Phi = \alpha^{\prime} \circ \widetilde{\alpha}^{\prime-1} \circ \widetilde{\Phi} \circ \pi_{\widetilde{I}}^{\ker\beta}
  \end{equation*}
  is a two-sided inverse of $\Psi$.  First, we will show that it is a
  right-inverse. Since $\ker\beta$ is a subobject of
  $\prod_{i\in I}\mathcal{F}(U)$, it suffices to show
  the commutativity of the diagrams
  \begin{equation*}
    \begin{tikzcd}
      \ker\beta
      \arrow{r}{\Psi\circ\Phi}\arrow[swap]{rd}{\pi_{i_0}^{\ker\beta}} &
      \ker\beta \arrow{d}{\pi_{i_0}^{\ker\beta}} \\
      \empty &
      \mathcal{F}(U_{i_0})
    \end{tikzcd}
  \end{equation*}
  for every $i_{0}\in I$. Now compute
  \begin{equation}\label{eq:siteqc-reducesheafcondtofincov-1}
    \pi_{i_0}^{\ker\beta}\circ\Psi\circ\alpha^{\prime}
    =\pi_{i_0}^{\im\alpha}\circ\alpha^{\prime}
    =\alpha_{i_0}
  \end{equation}
  and therefore
  \begin{align*}
    \pi_{i_0}^{\ker\beta}\circ\left(\Psi\circ\Phi\right)
    &=\pi_{i_0}^{\ker\beta}\circ \Psi\circ \alpha^{\prime}\circ \widetilde{\alpha}^{\prime-1} \circ \widetilde{\Phi} \circ \pi_{\widetilde{I}}^{\ker\beta} \\
    &\stackrel{\text{(\ref{eq:siteqc-reducesheafcondtofincov-1})}}{=}
    \alpha_{i_0}\circ\widetilde{\alpha}^{\prime-1}\circ\widetilde{\Phi}\circ\pi_{\widetilde{I}}^{\ker\beta} \\
    &\stackrel{\text{\ref{lem:siteqc-reducesheafcondtofincov-1}}}{=}
    \pi_{i_0}^{\ker\beta}.
  \end{align*}
  We find that $\Phi$ is a right-inverse of
  $\Psi$. It remains
  to compute that the diagrams
  \begin{equation*}
    \begin{tikzcd}
      \im\alpha
      \arrow{r}{\Phi\circ\Psi}\arrow[swap]{rd}{\pi_{i_0}^{\im\alpha}} &
      \im\alpha \arrow{d}{\pi_{i_0}^{\im\alpha}} \\
      \empty & \mathcal{F}(U_{i_0})
    \end{tikzcd}
  \end{equation*}
  are commutative for every $i_{0}\in I$.
  We have
  $\pi_{\widetilde{I}}^{\ker\beta}\circ\Psi=\widetilde{\Psi}\circ\pi_{\widetilde{I}}^{\im\alpha}$,
  so that composing with $\widetilde{\Phi}$ on the left yields
  \begin{equation}\label{eq:siteqc-reducesheafcondtofincov-2}
    \widetilde{\Phi}\circ\pi_{\widetilde{I}}^{\ker\beta}\circ\Psi=\pi_{\widetilde{I}}^{\im\alpha}.
  \end{equation}
  We get
  \begin{align*}
    \pi_{i_0}^{\im\alpha}\circ(\Phi\circ\Psi)
    &=\pi_{i_0}^{\im\alpha}\circ\alpha^{\prime} \circ \widetilde{\alpha}^{\prime-1} \circ \widetilde{\Phi} \circ \pi_{\widetilde{I}}^{\ker\beta}\circ\Psi \\
    &=\alpha_{i_0} \circ \widetilde{\alpha}^{\prime-1} \circ \widetilde{\Phi} \circ \pi_{\widetilde{I}}^{\ker\beta}\circ\Psi \\
    &\stackrel{\text{(\ref{eq:siteqc-reducesheafcondtofincov-2})}}{=}\alpha_{i_0} \circ \widetilde{\alpha}^{\prime-1} \circ\pi_{I}^{\im\alpha} \\
    &\stackrel{\text{\ref{lem:siteqc-reducesheafcondtofincov-1}}}{=}
    \pi_{i_0}^{\im\alpha}.
  \end{align*}
  That is, $\Phi$ is a left-inverse of
  $\Psi$ as well. We have thus shown that the
  top row of the diagram~(\ref{cd:siteqc-reducesheafcondtofincov-1})
  is exact. Since this choice of covering was arbitrary,
  $\mathcal{F}$ is a sheaf.
\end{proof}

%%%%%%%%%%%%%%%%%%%%%%%%%%%%%%%%%%%%%%%%%%%%%%%%%%%%%%%%%%%%
%%%%%%%%%%%%%%%%%%%%%%%%%%%%%%%%%%%%%%%%%%%%%%%%%%%%%%%%%%%%
%%%%%%%%%%%%%%%%%%%%%%%%%%%%%%%%%%%%%%%%%%%%%%%%%%%%%%%%%%%%
% Ind-Banach spaces
%%%%%%%%%%%%%%%%%%%%%%%%%%%%%%%%%%%%%%%%%%%%%%%%%%%%%%%%%%%%
%%%%%%%%%%%%%%%%%%%%%%%%%%%%%%%%%%%%%%%%%%%%%%%%%%%%%%%%%%%%
%%%%%%%%%%%%%%%%%%%%%%%%%%%%%%%%%%%%%%%%%%%%%%%%%%%%%%%%%%%%

\section{Sheaves of Banach modules}
\label{subsec:sheavesindbanspaces-reconstructionpaper}

Fix a Banach ring $R$ and a site $X$.

\begin{lem}\label{lem:defnsheaves-indBan}
  Suppose all coverings in $X$ are finite.
  A $\Ban_{R}$-presheaf on $X$ is a sheaf
  if and only if the following sequence is strictly exact
  for every open $U\in X$ and every covering $\mathfrak{U}$ of $U$:
  \begin{equation*}
    0\to\mathcal{F}(U)
     \to\prod_{V\in\mathfrak{U}}\mathcal{F}(V)
     \to\prod_{W,W^{\prime}\in\mathfrak{U}}\mathcal{F}\left(W\times_{U} W^{\prime}\right).
  \end{equation*}
\end{lem}

\begin{proof}
  The implication $\implies$ is clear. $\Leftarrow$
  follows because $\Ban_{R}$ has finite products.
\end{proof}

\begin{lem}\label{lem:banach-to-indbanach-sheaves}
  Given a sheaf $\mathcal{F}\colon X\to\Ban_{R}$,
  its composition with the canonical functor
  $\Ban_{R}\to\IndBan_{R}$ is a sheaf of $R$-ind-Banach modules.
\end{lem}

\begin{proof}
  Lemma~\ref{lem:functor-oC} applies because $\Ban_{R}$ has finite limits.
\end{proof}

\begin{notation}
  $\mathcal{R}$ is a monoid in the category of sheaves on $X$ with
  values in a closed symmetric monoidal quasi-abelian
  category $\E$. $\mathcal{M}$ is an $\mathcal{R}$-module object.
  \begin{itemize}
    \item[(i)] $\mathcal{R}$ is a \emph{sheaf of $S$-Banach algebras}
    if $\E=\Ban_{S}$, where $S$ is an
    $R$-Banach algebra.
    $\mathcal{M}$ is a \emph{sheaf of $\mathcal{R}$-Banach modules}.
    \item[(ii)] $\mathcal{R}$ is a \emph{sheaf of $S$-ind-Banach algebras}
    if $\E=\IndBan_{S}$, where $S$ is an
    $R$-ind-Banach algebra.
    $\mathcal{M}$ is a \emph{sheaf of $\mathcal{R}$-ind-Banach modules}.
  \end{itemize}
\end{notation}

%%%%%%%%%%%%%%%%%%%%%%%%%%%%%%%%%%%%%%%%%%%%%%%%%%%%%%%%%%%%
%%%%%%%%%%%%%%%%%%%%%%%%%%%%%%%%%%%%%%%%%%%%%%%%%%%%%%%%%%%%
%%%%%%%%%%%%%%%%%%%%%%%%%%%%%%%%%%%%%%%%%%%%%%%%%%%%%%%%%%%%
% Ind-Banach spaces
%%%%%%%%%%%%%%%%%%%%%%%%%%%%%%%%%%%%%%%%%%%%%%%%%%%%%%%%%%%%
%%%%%%%%%%%%%%%%%%%%%%%%%%%%%%%%%%%%%%%%%%%%%%%%%%%%%%%%%%%%
%%%%%%%%%%%%%%%%%%%%%%%%%%%%%%%%%%%%%%%%%%%%%%%%%%%%%%%%%%%%

\section{Sheaves of ind-Banach spaces}
\label{subsec:sheavesindbanspaces-reconstructionpaper}

We fix again a field $F=R$, complete with respect to nontrivial non-Archimedean valuation.
As $\IndBan_{\I\left(F\right)}$ is elementary abelian,
cf. the previous \S\ref{subsec:LHindBanachmodules-reconstructionpaper},
it is in particular elementary quasi-abelian. Therefore, Schneiders' formalism
as in~\cite[chapter 2]{Sch99} applies. Because of~\cite[Corollary 4.24]{Bo21}, we furthermore
have all the results in \emph{loc. cit.} \S 3 in our disposal. In particular,
given any site $X$, $\Sh\left(X,\IndBan_{\I\left(F\right)}\right)$ is an abelian category.
It is complete and cocomplete, where both limits and colimits can be described
can be described explicitly, cf.~\cite[Lemma 3.15]{Bo21}.
Writing down colimits requires a notion of sheafification 
$\cal{F}\mapsto\cal{F}^{\sh}$ which is well-developed,
cf. \emph{loc. cit.} \S 3.4.

As $\I\colon\IndBan_{F}\to\IndBan_{\I(F)}$ is right exact,
cf.~\cite[Proposition 1.2.27]{Sch99}, it induces the functor
\begin{equation*}
  \Sh\left(X,\IndBan_{F}\right)
  \to
  \Sh\left(X,\IndBan_{\I\left(F\right)}\right)
  \cal{F}\mapsto \I\left(\cal{F}\right)
\end{equation*}
where $\I\left(\cal{F}\right)$ is the sheaf $U\mapsto\I\left(\cal{F}(U)\right)$.

Let $\E\in\{\IndBan_{F},\IndBan_{\I(F)}\}$. There is a \emph{\v{C}ech complex}
\begin{equation*}
  \check{\rC}^{n}\left( \mathfrak{U},\cal{F} \right)=
  \prod_{i_{0}<\dots<i_{n}}\cal{F}\left(V_{i_{0},\dots,i_{n}} \right)
\end{equation*}
for any presheaf $\cal{F}\in\Psh(X,\E)$ and any covering
$\mathfrak{U}=\left\{ U_{i}\to U\right\}_{i\in I}$. Here, we fixed
an ordering on the index set $I$ and wrote
$V_{i_{0},\dots,i_{n}}:=V_{i_{0}}\cap\dots\cap V_{i_{n}}$.
The \emph{augmented \v{C}ech complex $\check{\rC}_{\aug}^{\bullet}\left( \mathfrak{U},\cal{F} \right)$}
is
\begin{equation}\label{eq:augmentedcechcplx-recpaper}
  0 \to \cal{F}(U)
  \to \check{\rC}^{0}\left( \mathfrak{U},\cal{F} \right)
  \to \check{\rC}^{1}\left( \mathfrak{U},\cal{F} \right)
  \to\dots
\end{equation}
where $\cal{F}(U)$ sits in cohomological degree $-1$.

\begin{defn}
  An $\E$-sheaf has \emph{vanishing higher \v{C}ech cohomology}
  if for every covering $\mathfrak{U}$, the associated \v{C}ech complex
  $\check{\rC}\left( \mathfrak{U},\cal{F} \right)^{\bullet}$ is strictly exact
  in strictly positive degrees.
\end{defn}

$\Sh\left(X,\IndBan_{\I(F)}\right)$ has enough injectives because
$\IndBan_{\I(F)}$ has enough injectives,
hence the following Lemma~\ref{lem:strictlyexactcechcomplexes-sheafcohomology-reconstructionpaper}
makes sense.

\begin{lem}\label{lem:strictlyexactcechcomplexes-sheafcohomology-reconstructionpaper}
  Let $\cal{F}^{\psh}$ denote a presheaf of $F$-ind-Banach spaces on $X$.
  $\cal{F}:=\left(\cal{F}^{\psh}\right)^{\sh}$ is its sheafification.
  Assume that, for every covering $\mathfrak{V}$
  of any $V\in X$, the augmented \v{C}ech complex $\check{\cal{C}}_{\aug}^{\bullet}\left(\mathfrak{V},\cal{F}^{\psh}\right)$
  is strictly exact. Then, for every $U\in X$, the canonical morphism
  \begin{equation*}
    \I\left(\cal{F}^{\psh}(U)\right)
    \isomap
    \R\Gamma\left( U , \I\left(\cal{F}\right) \right)
  \end{equation*}
  is an isomorphism in the derived category of $\IndBan_{\I(F)}$.
  In particular, $\cal{F}^{\psh}(U)\isomap\cal{F}(U)$.
\end{lem}

\begin{proof}
  $\I$ commutes with products because $\IndBan_{\I(F)}$ is elementary,
  cf.~\cite[Proposition 2.1.15(a)]{Sch99}.
  \begin{equation}\label{eq:lem:strictlyexactcechcomplexes-sheafcohomology-reconstructionpaper}
    \I\left(\check{\cal{C}}_{\aug}^{\bullet}\left(\mathfrak{V},\cal{F}^{\psh}\right)\right)
    \cong\check{\cal{C}}_{\aug}^{\bullet}\left(\mathfrak{V},\I\left(\cal{F}^{\psh}\right)\right)    
  \end{equation}
  follows, where $\I\left(\cal{F}^{\psh}\right)$ is the presheaf
  $U\mapsto\I\left(\cal{F}^{\psh}(U)\right)$.
  This complex~(\ref{eq:lem:strictlyexactcechcomplexes-sheafcohomology-reconstructionpaper})
  is exact by~\cite[Corollary 1.2.28]{Sch99}.
  Thus $\I\left(\cal{F}^{\psh}\right)$ is a sheaf with vanishing higher \v{C}ech cohomology.
  This gives $\Gamma\left( U , \I\left(\cal{F}\right)\right)\cong\I\left(\cal{F}^{\psh}\right)(U)\cong\I\left(\cal{F}^{\psh}(U)\right)$
  and the vanishing of $\Ho^{i}\left( U , \I\left(\cal{F}\right)\right)=0$ for $i>0$.
  The latter follows from a version of the Čech-to-cohomology spectral sequence
  as in~\cite[\href{https://stacks.math.columbia.edu/tag/03OW}{Tag 03OW}]{stacks-project}
  for $\LH\left(\E\right)$-valued sheaves.
\end{proof}

We have discussed in \S\ref{subsec:LHindBanachmodules-reconstructionpaper}
that $\IndBan_{\I\left(F\right)}$ carries a closed symmetric monoidal structure.
\cite[Lemma 3.19]{Bo21} explains
how to turn the category of $\Sh\left(X,\IndBan_{\I\left(F\right)}\right)$ into a closed
symmetric monoidal category with monoidal operation
$\widehat{\otimes}_{\I\left(F\right)}$ and internal homomorphisms
$\shHom_{\I\left(F\right)}$. The unit is the constant sheaf $\I\left(F\right)_{X}$.

%%% 

%%%%%%%%%%%%%%%%%%%%%%%%%%%%%%%%%%%%%%%%%%%%%%%%%%%%%%%%%%%%
%%%%%%%%%%%%%%%%%%%%%%%%%%%%%%%%%%%%%%%%%%%%%%%%%%%%%%%%%%%%
% FUNCTIONAL ANALYSIS
%%%%%%%%%%%%%%%%%%%%%%%%%%%%%%%%%%%%%%%%%%%%%%%%%%%%%%%%%%%%
%%%%%%%%%%%%%%%%%%%%%%%%%%%%%%%%%%%%%%%%%%%%%%%%%%%%%%%%%%%%

\section{Condensed mathematics}
\label{subsec:condensedmath-recpaper}

The language of ind-Banach spaces
is not well-suited for computations with continuous group cohomology.
We therefore make use of condensed mathematics~\cite{CondensedMaths}.
With the help of~\cite[\S 3.2]{BSSW2024_rationalizationoftheKnlocalsphere}
and~\cite[Appendix A]{Bosco21},
we see in \S\ref{subsec:contgrpcohomology} that this solves our issues.

We ignore set-theoretic issues and refer the reader to~\cite[Lecture I]{CondensedMaths}
and~\cite[Appendix A]{Bosco21} for details. Let $\Prof$ denote the site whose underlying
category is the category of profinite sets; the covers are the jointly surjective continuous maps.
A \emph{condensed set} is a sheaf of sets
of on $\Prof$. There is a functor
$X\mapsto\underline{X}$ from topological groups to condensed sets
via $\underline{X}(S):=\Hom_{\cont}\left(S,X\right)$, which denotes the
set of continuous maps $S\to X$. Similarly, 
a \emph{condensed abelian group} is a sheaf
of abelian groups on $\Prof$. The same recipe
$G\mapsto\underline{G}$, $\underline{G}(S):=\Hom_{\cont}\left(S,G\right)$
as above defines a functor from topological groups to condensed abelian groups.
As explained in~\cite[Lecture I, page 13]{CondensedMaths},
the category of condensed abelian groups has has a canonical symmetric monoidal tensor product $\otimes$ which
makes it a closed symmetric monoidal category.
We can therefore speak of the category of $\underline{k}$-module objects, where
$k$ is as in \S\ref{subsec:conventions-reconstructionpaper}.
This is the \emph{category of condensed $k$-vector spaces} $\Vect_{k}^{\cond}$.

We refer the reader to~\cite[Proposition A.2]{Bosco21} for the definition
of the full subcategory $\Vect_{\QQ_{p}}^{\solid}\subseteq\Vect_{\QQ_{p}}^{\cond}$
of \emph{solid $\QQ_{p}$-vector spaces}. \emph{Loc. cit.} explains that
it is abelian, and, as a full subcategory of $\Vect_{\QQ_{p}}^{\solid}\subseteq\Vect_{\QQ_{p}}^{\cond}$,
stable under all limits, colimits, and extensions. It carries
a symmetric monoidal tensor product $\otimes_{\QQ_{p}}^{\blacksquare}$;
we again refer the reader again to~\cite[Proposition A.2]{Bosco21} for its definition.

By~\cite[Proposition A.12]{Bosco21}, we have
\begin{equation*}
  \Ban_{\QQ_{p}}\to\Vect_{\QQ_{p}}^{\solid}, V\mapsto\underline{V}.
\end{equation*}
This functor is strongly monoidal by \emph{loc. cit.} Proposition A.25.
As $k$ is canonically a monoid in $\Ban_{\QQ_{p}}$,
$\underline{k}$ is a monoid in $\Vect_{\QQ_{p}}^{\solid}$. We denote by
\begin{equation*}
  \Vect_{k}^{\solid}:=\Mod\left(\underline{k}\right)
\end{equation*}
the categories of $\underline{k}$-module objects.
This is the category of \emph{solid $k$-vector spaces}.
Following the constructions in~\cite[\S 2]{BBK},
$\Vect_{k}^{\solid}$ becomes a closed symmetric monoidal abelian category
with unit $\underline{k}$. Denote the monoidal operation by $\otimes_{k}^{\blacksquare}$.
By Lemma~\ref{lem:Bank-is-Modk-recpaper} below, one gets the functor
\begin{equation}\label{eq:Bank-to-Vectksolid-recpaper}
  \Ban_{k}\to\Vect_{k}^{\solid}, V\mapsto \underline{V}
\end{equation}

\begin{lem}\label{lem:Bank-is-Modk-recpaper}
  For $k\in\Ban_{\QQ_{p}}$,
  $\Ban_{k}=\Mod\left(k\right)$ as monoidal categories.
\end{lem}

\begin{proof}
  Clearly, $\Ban_{k}$ and $\Mod\left(k\right)$ coincide as categories.
  It remains to compare their monoidal structures: This is~\cite[Lemma A.54]{BBK}.
%  From the definition of the monoidal operation on
%  $\Mod\left(k\right)$ as in~\cite[Definition 2.2]{BBK}, we observe that
%  one has to check the following: For every two $k$-Banach spaces $V$ and $W$,
%  \begin{equation*}
%    V\widehat{\otimes}_{\QQ_{p}}k\widehat{\otimes}_{\QQ_{p}}W
%    \stackrel{\phi}{\longrightarrow} V \widehat{\otimes}_{\QQ_{p}}W
%    \longrightarrow V\widehat{\otimes}_{k}W
%    \longrightarrow 0
%  \end{equation*}
%  is strictly exact where $\phi$ is given by
%  $v\widehat{\otimes}\lambda\widehat{\otimes}w\mapsto
%    \left(v\lambda\widehat{\otimes} w\right)-\left(v\widehat{\otimes}\lambda w\right)$.
%  etc
\end{proof}

We prove a few simple lemmata for future reference.

\begin{lem}\label{lem:Vectksolid-to-Vectkcond-cont-cocont-reconstructionpaper}
  The canonical functor $\Vect_{k}^{\solid}\to\Vect_{\QQ_{p}}^{\cond}$
  commutes with all limits and colimits.
\end{lem}

\begin{proof}
  $\Vect_{\QQ_{p}}^{\solid}\subseteq\Vect_{\QQ_{o}}^{\cond}$
  is stable under all limits and colimits, cf.~\cite[Proposition A.2(ii)]{Bosco21},
  and the forgetful functor $\Vect_{k}^{\solid}\to\Vect_{\QQ_{p}}^{\solid}$
  commutes with limits and colimits by~\cite[Lemma 2.3]{BBK}.
\end{proof}

\begin{lem}\label{lem:colimitsVectkcond-exact}
  Filtered colimits in $\Vect_{\QQ_{p}}^{\cond}$ are left exact.
\end{lem}

\begin{proof}
  Given a filtered systems $\left\{\phi_{i}\colon W_{i}\to Z_{i}\right\}_{i}$ of condensed $k$-vector spaces,
  we claim that the canonical morphism
  %\begin{equation*}
    $\ker\varinjlim_{i}\phi_{i}
    \to
    \varinjlim_{i}\ker\phi_{i}$
  %\end{equation*}
  is an isomorphism. This is the case if it is an isomorphism of condensed abelian groups,
  which is well-known, cf.~\cite[Theorem 2.2]{CondensedMaths}.
\end{proof}

\begin{lem}\label{lem:colimitsVectksolid-exact}
  Filtered colimits in $\Vect_{k}^{\solid}$ are exact.
\end{lem}

\begin{proof}
  They commute with cokernels. It remains to check that given a filtered system
  $\left\{\phi_{i}\colon W_{i}\to Z_{i}\right\}_{i}$ of solid $k$-vector spaces,
  the canonical morphism
  %\begin{equation*}
    $\ker\varinjlim_{i}\phi_{i}
    \to
    \varinjlim_{i}\ker\phi_{i}$
  %\end{equation*}
  is an isomorphism. If $\iota$ denotes the canonical functor $\Vect_{k}^{\solid}\to\Vect_{\QQ_{p}}^{\cond}$, then
  one may check that
  %\begin{equation*}
    $\ker\varinjlim_{i}\iota\left(\phi_{i}\right)
    \to
    \varinjlim_{i}\ker\iota\left(\phi_{i}\right)$
  %\end{equation*}  
  is an isomorphism, cf. Lemma~\ref{lem:Vectksolid-to-Vectkcond-cont-cocont-reconstructionpaper}.
  This follows from Lemma~\ref{lem:colimitsVectkcond-exact}.
\end{proof}

\begin{defn}\label{defn:underlineHomctsSV-recpaper}
  Given a profinite set $S$ and a $k$-Banach space $V$, $\intHom_{\cont}\left(S,V\right)$
  denotes the seminormed $k$-vector space of all continuous maps $S\to V$
  equipped with the supremum norm.
\end{defn}

With the notation as in Definition~\ref{defn:underlineHomctsSV-recpaper},
$\intHom_{\cont}\left(S,V\right)$ is a $k$-Banach space because $S$ is compact.

\begin{lem}\label{lem:HomcontTtimessV-is-HomcontTHomcontSV-reconstructionpaper}
  Let $S$ and $T$ be profinite sets, $V$ is a $k$-Banach space. Then the canonical map
  \begin{equation*}
    \Hom_{\cont}\left(T\times S , V\right)
    \longrightarrow
    \Hom_{\cont}\left(T,\intHom_{\cont}\left(S,V\right)\right)
  \end{equation*}
  is a bijection.
\end{lem}

\begin{proof}
  This is well-known, since the norm on $\intHom_{\cont}\left(S,V\right)$
  induces the compact-open topology.
\end{proof}

Define internal homomorphisms as follows:
For all condensed sets $X,Y$ and $S\in\Prof$,
\begin{equation*}
  \intHom(X,Y)(S):=\Hom(X\times S,Y).
\end{equation*}

\begin{lem}\label{lem:intHomunderlineSunderlineV-is-underlineHomcontSV-reconstructionpaper}
  Let $S$ be a profinite set, $V$ is a $k$-Banach space. Then, functorially,
  \begin{equation*}
    \intHom\left(\underline{S},\underline{V}\right)
    \cong\underline{\intHom_{\cont}\left(S,V\right)}.
  \end{equation*}
\end{lem}

\begin{proof}
 The following direct computation
  \begin{align*}
    \intHom\left(\underline{S},\underline{V}\right)(T)
    &=\Hom\left(\underline{S\times T},\underline{V}\right) \\
    &=\underline{V}\left(S\times T\right) \\
    &=\Hom_{\cont}\left(S\times T,V\right) \\
    &\stackrel{\text{\ref{lem:HomcontTtimessV-is-HomcontTHomcontSV-reconstructionpaper}}}{=}
      \Hom_{\cont}\left(T,\intHom_{\cont}\left(S,V\right)\right) \\
    &=\intHom_{\cont}\left( S , V \right)(T)
  \end{align*}
  for every profinite set $T$ gives the result.
\end{proof}

Now we relate ind-Banach spaces to solid vector spaces.
$\Vect_{k}^{\solid}$ is cocomplete
since $\Vect_{\QQ_{p}}^{\solid}$ is cocomplete, cf.~\cite[Proposition A.2(ii)]{Bosco21}. Thus
(\ref{eq:Bank-to-Vectksolid-recpaper}) extends to the functor
\begin{equation}\label{eq:defn-solidificationfunctor-recpaper}
  \IndBan_{k}\to\Vect_{k}^{\solid},
  \text{``}\varinjlim_{i}\text{"}V_{i}
  \mapsto\underline{\text{``}\varinjlim_{i}\text{"}V_{i}}:=\text{``}\varinjlim_{i}\text{"}\underline{V_{i}}.
\end{equation}

\begin{notation}
  (\ref{eq:defn-solidificationfunctor-recpaper}) is the \emph{solidification functor}.
\end{notation}

\begin{lem}\label{lem:intHomunderlineSunderlineV-is-underlineHomcontSV-indBansetting-reconstructionpaper}
  Let $S$ be a profinite set, $V$ is a $k$-ind-Banach space. Then, functorially,
  \begin{equation*}
    \intHom\left(\underline{S},\underline{V}\right)
    \cong\underline{\intHom_{\cont}\left(S,V\right)}.
  \end{equation*}
\end{lem}

\begin{proof}
  Write $V=\text{``}\varinjlim_{i}\text{"}V_{i}$.
  Given a profinite set $L$, we claim that the canonical morphism
  \begin{equation}\label{eq:--lem:intHomunderlineSunderlineV-is-underlineHomcontSV-indBansetting-reconstructionpaper}
    \varinjlim_{i}\left(\underline{V_{i}}(L)\right)\isomap\underline{V}(L)
  \end{equation}
  is an isomorphism of $k$-vector spaces.
  This follows from~\cite[\href{https://stacks.math.columbia.edu/tag/0738}{Tag 0738}]{stacks-project}.
  \emph{Loc. cit.} applies because condensed sets are sheaves on sites of profinite sets,
  which are quasi-compact. Now compute
  \begin{equation*}
    \intHom\left(\underline{S},\underline{V}\right)(T)
    =\Hom\left(\underline{S\times T},\underline{V}\right)
    =\underline{V}(S\times T)
    \stackrel{\text{\ref{eq:--lem:intHomunderlineSunderlineV-is-underlineHomcontSV-indBansetting-reconstructionpaper}}}{\cong}
    \varinjlim_{i}\left(\underline{V_{i}}(S\times T)\right)
    \cong\varinjlim_{i}\left(\intHom\left(\underline{S},\underline{V_{i}}\right)(T)\right).
  \end{equation*}
  for every profinite set $T$.
  That is,
  $\intHom\left(\underline{S},\underline{V}\right)\cong\varinjlim_{i}\intHom\left(\underline{S},\underline{V}_{i}\right)$.
  Using this fact, one deduces
  Lemma~\ref{lem:intHomunderlineSunderlineV-is-underlineHomcontSV-indBansetting-reconstructionpaper}
  directly from Lemma~\ref{lem:intHomunderlineSunderlineV-is-underlineHomcontSV-reconstructionpaper}.
\end{proof}

\begin{lem}\label{lem:intHomunderlineSunderlineV-is-underlineHomcontSV-reconstructionpaper}
  Let $S$ be a profinite set, $V$ is a $k$-Banach space. Then, functorially,
  \begin{equation*}
    \intHom\left(\underline{S},\underline{V}\right)
    \cong\underline{\intHom_{\cont}\left(S,V\right)}.
  \end{equation*}
\end{lem}

\begin{proof}
 The following computation
  \begin{align*}
    \intHom\left(\underline{S},\underline{V}\right)(T)
    &=\Hom\left(\underline{S\times T},\underline{V}\right) \\
    &=\underline{V}\left(S\times T\right) \\
    &=\Hom_{\cont}\left(S\times T,V\right) \\
    &\stackrel{\text{\ref{lem:HomcontTtimessV-is-HomcontTHomcontSV-reconstructionpaper}}}{=}
      \Hom_{\cont}\left(T,\intHom_{\cont}\left(S,V\right)\right) \\
    &=\intHom_{\cont}\left( S , V \right)(T)
  \end{align*}
  for every profinite set $T$ gives the result.
\end{proof}

\begin{lem}\label{lem:IndBan-to-Solid-stronglylefteexact-reconstructionpaper}
  The solidification functor $\IndBan_{k}\to\Vect_{k}^{\solid}$, $V\mapsto\underline{V}$ is strongly left exact.
\end{lem}

\begin{proof}
  By~\cite[Remark 2.3]{BBB16}
  and Corollary~\ref{cor:filteredcol-inIndBan-stronglyexact},
  it suffices to check that $V\mapsto\underline{V}$ preserves kernels of
  bounded linear maps between $k$-Banach spaces.
  We may compute the kernels in the categories $\Ban_{\QQ_{p}}$
  and $\Vect_{\QQ_{p}}^{\solid}$, respectively. Thus~\cite[Lemma A.15]{Bosco21} applies.
\end{proof}

\begin{lem}\label{lem:IndBan-to-Solid-strictlyexact-reconstructionpaper}
  The solidification functor $\IndBan_{k}\to\Vect_{k}^{\solid}$, $V\mapsto\underline{V}$ is strictly exact.
\end{lem}

\begin{proof}
  Given a strictly exact sequence
  $V^{\prime}\stackrel{e^{\prime}}{\longrightarrow}V\stackrel{e}{\longrightarrow}V^{\prime\prime}$
  of $k$-ind-Banach spaces, we have to check that
  $\underline{V^{\prime}}\stackrel{\underline{e^{\prime}}}{\longrightarrow}
    \underline{V}\stackrel{\underline{e}}{\longrightarrow}\underline{V^{\prime\prime}}$
  is exact. That is,
  $\underline{e}^{\prime}\colon\underline{V}^{\prime}\to \ker\underline{e}$ has to be an epimorphism.
  Since $e^{\prime}\colon V\to\ker e$ is a strict epimorphism, and by
  Lemma~\ref{lem:IndBan-to-Solid-stronglylefteexact-reconstructionpaper},
  it suffices to check that $V\mapsto\underline{V}$ preserves strict epimorphisms.
  By~\cite[Remark 2.3]{BBB16}
  and Corollary~\ref{cor:filteredcol-inIndBan-stronglyexact},
  it suffices to check that $V\mapsto\underline{V}$ preserves strict
  epimorphisms between $k$-Banach spaces. But maps in
  $\Ban_{k}$, respectively $\Vect_{k}^{\solid}$, are epimorphisms if and only if they are
  epimorphisms in $\Ban_{\QQ_{p}}$, respectively $\Vect_{\QQ_{p}}^{\solid}$.
  Thus we can invoke the arguments in the~\cite[proof of Lemma A.15]{Bosco21}.
\end{proof}

\begin{lem}\label{lem:IndBan-to-Solid-preservesdirectsums-reconstructionpaper}
  The solidification functor $\IndBan_{k}\to\Vect_{k}^{\solid}$, $V\mapsto\underline{V}$ preserves direct sums.
\end{lem}

\begin{proof}
  $\Ban_{k}\to\Vect_{k}^{\solid}$ preserves finite direct sums. Thus
  Lemma~\ref{lem:IndBan-to-Solid-preservesdirectsums-reconstructionpaper} follows
  because the solidification functor $\IndBan_{k}\to\Vect_{k}^{\solid}$ preserves filtered colimits by construction
  and direct sums are filtered colimits of finite direct sums.
\end{proof}

In the following, we show that the converse of Lemma~\ref{lem:IndBan-to-Solid-strictlyexact-reconstructionpaper}
holds in one particular instance. %We start with
%Lemma~\ref{lem:indBan-solid-underlyingabstractvs-thesame-reconstructionpaper} below.

\begin{notation}
  Given a $k$-Banach space $V$, denote by $|V|$ its underlying abstract $k$-vector space.
  This construction extends to a functor
  \begin{equation*}
    \IndBan_{k} \to \Mod(|k|), \text{``}\varinjlim_{i}\text{"}V_{i}\mapsto\varinjlim_{i}|V_{i}|.
  \end{equation*}
  On the other hand, given a condensed $k$-vector space $W$, we get the abstract $k$-vector space
  $W(*)$.
\end{notation}

\begin{lem}\label{lem:indBan-solid-underlyingabstractvs-thesame-reconstructionpaper}
  For any $k$-ind-Banach space $V$, we have the canonical isomorphism
  $|V|\cong\underline{V}(*)$.
  %of abstract $k$-vector spaces.
\end{lem}

\begin{proof}
  If $V\in\Ban_{k}$, then the isomorphism is
  %\begin{equation*}
    $|V|\isomap\Hom_{\cont}\left(*,V\right)=\underline{V}(*)$.
  %\end{equation*}
  If $V=\text{``}\varinjlim_{i}\text{''}V_{i}$ is an ind-Banach space,
  then we find from the above that it suffices to check that the canonical map
  \begin{equation*}
    \varinjlim_{i}\underline{V_{i}}(*) \to \left(\varinjlim_{i}\underline{V_{i}}\right)(*)
  \end{equation*}
  is an isomorphism. As $*$ is extremally disconnected, this follow from
  the~\cite[first paragraph of the proof of Theorem 2.2]{CondensedMaths}.
\end{proof}

\begin{lem}\label{lem:condensed-takingunderlyingabstractmodule-is-exact}
  The functor $\Vect_{k}^{\solid}\to\Mod(|k|)$, $W\mapsto W(*)$ is exact.
\end{lem}

\begin{proof}
  By~\cite[Proposition A.2(ii)]{Bosco21}, we may show that
  $\Vect_{k}^{\solid}\to\Mod(|k|)$, $W\mapsto W(*)$ is exact.
  As $*$ is extremally disconnected, invoke
  the~\cite[first paragraph of the proof of Theorem 2.2]{CondensedMaths}.
\end{proof}

\begin{defn}\label{defn:bornologicalspace-reconstructionpaper}
  An $F$-ind-Banach space is \emph{bornological} if it is isomorphic to an object $\text{``}\varinjlim\text{"}_{i}E_{i}$
  where all the structural maps $E_{i}\to E_{j}$ are injective. A \emph{complete bornological $F$-vector space}
  is a bornological $F$-ind-Banach space.
  $\CBorn_{F}$ denotes the full subcategory of $\IndBan_{F}$ of complete bornological $F$-vector spaces.
\end{defn}

\begin{defn}\label{defn:bornology-countable-basis-reconstructionpaper}
  Given a complete bornological space $k$-vector space $E$,
  we say that its \emph{bornology has countable basis} if it is isomorphic
  to a $k$-ind-Banach space $\text{``}\varinjlim\text{''}_{n\in\NN}E_{n}$
  where all the structural maps $E_{n}\to E_{m}$
  for $m\geq n$ are injective.
\end{defn}

Here is the aforementioned converse of Lemma~\ref{lem:IndBan-to-Solid-strictlyexact-reconstructionpaper}:

\begin{prop}\label{prop:solidexact-implies-indBanachstrictlyexact-reconstructionpaper}
  Consider a sequence $E^{\prime}\to E\to E^{\prime\prime}$ of complete bornological
  $k$-vector space such that the bornology of $E^{\prime}$ has countable basis.
  If $\underline{E^{\prime}}\to\underline{E}\to\underline{E^{\prime\prime}}$ is exact
  as a sequence of solid $k$-vector spaces, then $E^{\prime}\to E\to E^{\prime\prime}$ is strictly exact.
\end{prop}

\begin{proof}
  If $\underline{E^{\prime}}\to\underline{E}\to\underline{E^{\prime\prime}}$
  is exact as a sequence of solid $k$-vector spaces, then it is exact as a sequence
  of condensed $k$-vector spaces by~\cite[Proposition A.2(ii)]{Bosco21}.
  Lemma~\ref{lem:condensed-takingunderlyingabstractmodule-is-exact}
  gives that $\underline{E^{\prime}}(*)\to\underline{E}(*)\to\underline{E^{\prime\prime}}(*)$
  is an exact sequence of abstract $k$-vector spaces.
  By Lemma~\ref{lem:indBan-solid-underlyingabstractvs-thesame-reconstructionpaper},
  $|E^{\prime}| \to |E| \to |E^{\prime\prime}|$ is exact.
  Now apply a version of the open mapping theorem,
  cf.~\cite[Theorem 4.9]{Ba15}, to deduce that $E^{\prime}\to E\to E^{\prime\prime}$ is strictly exact.
  This last step used that the bornology of $E^{\prime}$ has countable basis.
\end{proof}

We continue with further properties of the solidification functor $V\mapsto\underline{V}$.

\begin{lem}\label{lem:IndBan-to-Solid-canonicallystronglymonoidal-reconstructionpaper}
  The solidification functor $\IndBan_{k}\to\Vect_{k}^{\solid}$, $V\mapsto\underline{V}$ is strongly monoidal.
\end{lem}

\begin{proof}
  Both categories are closed, thus their monoidal structures commute with
  colimits. Since $V\mapsto\underline{V}$ commutes with filtered colimits,
  it suffices to check that its restriction $\Ban_{k}\to\Vect_{k}^{\solid}$
  is strongly monoidal. To show this, let $V,W\in\Ban_{k}$ and consider
  the sequence of $k$-Banach modules
  \begin{equation}\label{eq:1--lem:IndBan-to-Solid-canonicallystronglymonoidal-reconstructionpaper}
    k\widehat{\otimes}_{\QQ_{p}}k\widehat{\otimes}_{\QQ_{p}}W
    \stackrel{\phi}{\longrightarrow} k\widehat{\otimes}_{\QQ_{p}}W
    \stackrel{\psi}{\longrightarrow} W
    \longrightarrow 0.
  \end{equation}
  Here, $\phi$ is
  $\mu\widehat{\otimes}\lambda\widehat{\otimes}w\mapsto\mu\lambda\widehat{\otimes}w-\mu\widehat{\otimes}\lambda w$
  and $\psi$ is $\lambda\widehat{\otimes}w\mapsto\lambda w$.
  (\ref{eq:1--lem:IndBan-to-Solid-canonicallystronglymonoidal-reconstructionpaper})
  is strictly exact by~\cite[Lemma 2.9]{BBBK18}.
  By \emph{loc. cit.} Theorem 3.50,
  applying $V\widehat{\otimes}_{k}-$ gives a strictly exact sequence
  \begin{equation*}
    V\widehat{\otimes}_{\QQ_{p}}k\widehat{\otimes}_{\QQ_{p}}W
    \stackrel{\phi}{\longrightarrow} V\widehat{\otimes}_{\QQ_{p}}W
    \stackrel{\psi}{\longrightarrow} V\widehat{\otimes}_{k}W
    \longrightarrow 0.
  \end{equation*}
  Thus Lemma~\ref{lem:IndBan-to-Solid-strictlyexact-reconstructionpaper}
  applies in the following computation:
  \begin{equation}\label{eq:2--lem:IndBan-to-Solid-canonicallystronglymonoidal-reconstructionpaper}
    \underline{V\widehat{\otimes}_{k}W}
    \stackrel{\text{\ref{lem:IndBan-to-Solid-strictlyexact-reconstructionpaper}}}{\cong}
    \coker\left(
      \underline{V}\otimes_{\QQ_{p}}^{\blacksquare}
      \underline{k}\widehat{\otimes}_{\QQ_{p}}^{\blacksquare}
      \underline{W}
      \to\underline{V}\otimes_{\QQ_{p}}^{\blacksquare}\underline{W}
      \right)
    =\underline{V}\otimes_{k}^{\blacksquare}\underline{W}.
  \end{equation}
  Here, we also used that $\Ban_{\QQ_{p}}\to\Vect_{\QQ_{p}}^{\solid}$
  is strongly monoidal, cf.~\cite[Proposition A.25]{Bosco21}. The equality at the right-hand side
  of~(\ref{eq:2--lem:IndBan-to-Solid-canonicallystronglymonoidal-reconstructionpaper})
  holds by definition, as the monoidal structure on $\Vect_{k}^{\solid}=\Mod\left(\underline{k}\right)$
  is the one we obtain from $\otimes_{\QQ_{p}}^{\blacksquare}$
  by following the construction as in~\cite[Definition 2.2]{BBK}.
\end{proof}

Thanks to Lemma~\ref{lem:IndBan-to-Solid-strictlyexact-reconstructionpaper},
\cite[Proposition 1.2.34]{Sch99} implies the existence of a functor
$\IndBan_{\I\left(k\right)}\to\Vect_{k}^{\solid}$, $E\mapsto\underline{E}$
which is exact and, for all $V\in\IndBan_{k}$,
\begin{equation}\label{eq:underlineIVisunderlineV-reconstructionpaper}
  \underline{\I\left(V\right)}=\underline{V}.
\end{equation}
From now on, we fix such a functor.

\begin{lem}
  Given a cochain complex $V^{\bullet}$ of $k$-ind-Banach spaces, for all $i\in\ZZ$,
  \begin{equation*}
    \Ho^{i}\left(\underline{V^{\bullet}}\right)
    \cong\underline{\LHo^{i}\left(V^{\bullet}\right)}.
  \end{equation*}
\end{lem}

\begin{proof}
  Compute
  $\Ho^{i}\left(\underline{V^{\bullet}}\right)
    \stackrel{\text{(\ref{eq:underlineIVisunderlineV-reconstructionpaper})}}{\cong}
      \Ho^{i}\left(\underline{\I\left(V^{\bullet}\right)}\right)
    \cong\underline{\Ho^{i}\left(\I\left(V^{\bullet}\right)\right)}
    \stackrel{\text{\ref{lem:LH-vs-H-reconstructionpaper}}}{\cong}\underline{\LHo^{i}\left(V^{\bullet}\right)}$.
\end{proof}

Let $X$ denote a site.

\begin{defn}\label{defn:solidificationssheaves}
  Given a sheaf $\cal{F}$ of $k$-ind-Banach spaces on $X$,
  its \emph{solidification}
  $\underline{\cal{F}}$ is the sheafification of the presheaf $U\mapsto\underline{\cal{F}(U)}$.
\end{defn}

\begin{lem}\label{lem:sectionsofunderlineF}
  Let $\cal{F}$ denote a sheaf of $k$-ind-Banach spaces on
  $X$, and assume every covering in $X$ is finite. Then, for all $U\in X$,
  we have the canonical isomorphism
  \begin{equation*}
    \underline{\cal{F}(U)}\isomap\underline{F}(U).
  \end{equation*}
\end{lem}

\begin{proof}
  We have to check that $U\mapsto\underline{F(U)}$ is a sheaf.
  To do this, we consider a covering $\mathfrak{U}$ of $U\in X$
  and the associated strictly exact sequences
  \begin{equation}\label{eq:sectionsofunderlineF-exactsequence}
    0\to\mathcal{F}(U)
     \to\prod_{V\in\mathfrak{U}}\mathcal{F}(V)
     \to\prod_{W,W^{\prime}\in\mathfrak{U}}\mathcal{F}\left(W\times_{U} W^{\prime}\right).
  \end{equation}
  $\mathfrak{U}$ is finite,
  and so are all the products in~(\ref{eq:sectionsofunderlineF-exactsequence}).
  Therefore,
  \begin{equation*}
    0\to\underline{\mathcal{F}(U)}
     \to\prod_{V\in\mathfrak{U}}\underline{\mathcal{F}(V)}
     \to\prod_{W,W^{\prime}\in\mathfrak{U}}\underline{\mathcal{F}\left(W\times_{U} W^{\prime}\right)}
  \end{equation*}
  is exact by Lemma~\ref{lem:IndBan-to-Solid-strictlyexact-reconstructionpaper}, as desired.
\end{proof}

The construction $\cal{F}\mapsto\underline{\cal{F}}$ is functorial.

\begin{lem}\label{lem:underlineF-commuteswithdirectsums-finitecov}
  $\Sh\left( X , \IndBan_{k} \right)\to \Sh\left(X , \Vect_{k}^{\solid} \right)$,
  $\cal{F}\mapsto\underline{\cal{F}}$ commutes with arbitrary direct sums
  if $X$ admits only finite coverings.
\end{lem}

\begin{proof}
  This follows from Lemma~\ref{lem:directsumsheaves-finitecoverings}
  since the solidification functor is cocontinuous.
\end{proof}

%%%%%%%%%%%%%%%%%%%%%%%%%%%%%%%%%%%%%%%%%%%%%%%%%%%%%%%%%%%%
%%%%%%%%%%%%%%%%%%%%%%%%%%%%%%%%%%%%%%%%%%%%%%%%%%%%%%%%%%%%
%%%%%%%%%%%%%%%%%%%%%%%%%%%%%%%%%%%%%%%%%%%%%%%%%%%%%%%%%%%%
% Bornological underline{F}-vector spaces
%%%%%%%%%%%%%%%%%%%%%%%%%%%%%%%%%%%%%%%%%%%%%%%%%%%%%%%%%%%%
%%%%%%%%%%%%%%%%%%%%%%%%%%%%%%%%%%%%%%%%%%%%%%%%%%%%%%%%%%%%
%%%%%%%%%%%%%%%%%%%%%%%%%%%%%%%%%%%%%%%%%%%%%%%%%%%%%%%%%%%%

\chapter{Miscellaneous}
\label{ch:miscellaneous}

We discuss technical results which will be used later on.
We recommend to skip these on a first reading.

%%%%%%%%%%%%%%%%%%%%%%%%%%%%%%%%%%%%%%%%%%%%%%%%%%%%%%%%%%%%
%%%%%%%%%%%%%%%%%%%%%%%%%%%%%%%%%%%%%%%%%%%%%%%%%%%%%%%%%%%%
%%%%%%%%%%%%%%%%%%%%%%%%%%%%%%%%%%%%%%%%%%%%%%%%%%%%%%%%%%%%
% Bornological underline{F}-vector spaces
%%%%%%%%%%%%%%%%%%%%%%%%%%%%%%%%%%%%%%%%%%%%%%%%%%%%%%%%%%%%
%%%%%%%%%%%%%%%%%%%%%%%%%%%%%%%%%%%%%%%%%%%%%%%%%%%%%%%%%%%%
%%%%%%%%%%%%%%%%%%%%%%%%%%%%%%%%%%%%%%%%%%%%%%%%%%%%%%%%%%%%

\section{Algebraic Localisations}
\label{subsec:localisations}

Fix a seminormed ring $R$ and an element $r\in R$.

\begin{notation}\label{notation:localisation}
  Given a seminormed $R$-module $M$,
  equip $M[1/r]$ with the seminorm $\|n\|:=\inf\|m\|/|r^{i}|$.
  The infinum varies along expressions
  $n=m/r^{i}$ with $m\in M$.
\end{notation}

%Although $R[1/r]$ is not necessarily a Banach ring,
%Definition~\ref{defn:seminormed-normed-banach-modules}
%still makes sense.

\begin{lem}\label{lem:localisation-normedotimes}
  $M\otimes_{R}R[1/r]\isomap M[1/r]$ as seminormed
  $R[1/r]$-modules for any $R$-Banach module $M$.
\end{lem}

\begin{proof}
  Denote the canonical map $M\to M[1/r]$ by $\phi$.
  We aim to exploit Yoneda's Lemma to show that it is an isomorphism: We claim
  \begin{equation*}
    \Hom_{R[1/r]}\left(M[1/r],V\right)\to\Hom_{R}\left(M,V\right), f \mapsto f\circ\phi
  \end{equation*}
  is a bijection for every seminormed $R[1/r]$-module $V$. This is clear when we consider
  the $\Hom$ in the category of abstract modules. It remains to check that the $\Hom$
  also coincide when they only capture the bounded linear maps. That is, we have to
  check the following for every $R[1/r]$-linear map $f\colon M[1/r]\to V$:
  $f$ is bounded if and only if $f\circ\phi$ is bounded.
  
  The implication $\implies$ follows because $\phi$ is bounded. To prove the converse,
  fix $n\in M[1/r]$ and a presentation $n=m/r^{i}$. $C>0$ is a bound on
  the scalar multiplication $R[1/r]\times V\to V$.
  Then
  \begin{equation*}
    \|f(n)\|
    =\|\frac{(f\circ\phi)(m)}{r^{i}}\|
    \leq C|\frac{1}{r^{i}}|\|(f\circ\phi)(m)\|
    \leq C|1|\|f\circ\phi\|\frac{\|m\|}{|r^{i}|}
  \end{equation*}
  Taking the infinum over all such expressions $n=m/r^{i}$, we find
  $\|f\|\leq C|1|\|f\circ\phi\|$.
\end{proof}

\begin{lem}\label{lem:SNrm-to-SNrm-localisation-preserves-certain-kernels}
  $-\otimes_{R}R[1/r]\colon\SNrm_{R}\to\SNrm_{R}$ preserves kernels
  of maps $\phi\colon M\to N$ when $N$ is $r$-torsion free.
\end{lem}

\begin{proof}
  There is a canonical linear map $\tau\colon\left(\ker\phi\right)[1/r]\to\ker\left(\phi[1/r]\right)$,
  which is bijective. It remains to check that the seminorms on both sides coincide. Recall that
  for every $x\in\left(\ker\phi\right)[1/r]$,
  \begin{align*}
    \|x\| &= \inf_{\substack{x=m/r^{i} \\ m\in\ker\phi}}\frac{\|m\|}{|r^{i}|} \\
    \|\tau(x)\| &= \inf_{\substack{x=m/r^{i} \\ m\in M}}\frac{\|m\|}{|r^{i}|}.
  \end{align*}
  We may therefore compare the infinums' indexing sets
  \begin{align*}
    S_{\ker\phi} := \left\{ \left(m,r^{i}\right) \colon x= m/r^{i} \text{ and } m\in\ker\phi\right\}, \\
    S_{M}         := \left\{ \left(m,r^{i}\right) \colon x= m/r^{i} \text{ and } m\in M\right\}.
  \end{align*}
  We have to check that both sets coincide. The inclusion $S_{\ker\phi}\subseteq S_{M}$
  is clear. Consider $\left(m,r^{i}\right)\in S_{M}$ to prove $\supseteq$. Now pick some
  $\left(\widetilde{m},r^{\widetilde{i}}\right)\in S_{\ker\phi}$.
  The equality $m/r^{i}=x=\widetilde{m}/r^{\widetilde{i}}$ implies $m/1\in\ker\left(\phi[1/r]\right)$.
  That is $\phi(m)/1=0$, thus $\phi(m)\in N$ is killed by some power of $r$. Because
  $N$ is $r$-torsion free, this implies $\phi(m)=0$. We find $\left(m,r^{i}\right)\in S_{\ker\phi}$,
  as desired.
\end{proof}

\begin{lem}\label{em:commute-htensor-with-completion}
  The canonical morphism
  %\begin{equation*}
    $\widehat{M \otimes_{R} N} \stackrel{\cong}{\longrightarrow} \widehat{M} \widehat{\otimes}_{R} \widehat{N}$
  %\end{equation*}
  is an isomorphism for any two
  normed $R$-modules $M$ and $N$.
\end{lem}

\begin{proof}
  This follows from the adjunctions.
\end{proof}

\begin{cor}\label{cor:completed-localisation-strictlyexact}
  Consider a strictly exact complex
  $M^{\prime}\stackrel{\phi^{\prime}}{\longrightarrow}M\stackrel{\phi}{\longrightarrow}M^{\prime\prime}$
  of $R$-Banach modules. If $M^{\prime\prime}$ is $r$-torsion free, then
  \begin{equation}\label{eq:completed-localisation-strictlyexact}
    M^{\prime}\widehat{\otimes}_{R}\widehat{R[1/r]}
    \stackrel{\phi^{\prime}\widehat{\otimes}_{R}\id}{\longrightarrow}
    M\widehat{\otimes}_{R}\widehat{R[1/r]}
    \stackrel{\phi\widehat{\otimes}_{R}\id}{\longrightarrow}
    M^{\prime\prime}\widehat{\otimes}_{R}\widehat{R[1/r]}
  \end{equation}
  is a strictly exact complex of $\widehat{R[1/r]}$-Banach modules.
\end{cor}

\begin{proof}
  By assumption, $\phi^{\prime}$ is strict and the canonical morphism
  $\iota\colon\im\phi^{\prime}\to\ker\phi$
  is an isomorphism. We apply $-\widehat{\otimes}_{R}R[1/r]$ to $\iota$:
  \begin{itemize}
    \item Regarding its domain,
    \begin{align*}
      \left(\im\phi^{\prime}\right)\otimes_{R}R[1/r]
      &=\ker\left(M\to\coker\phi^{\prime}\right)\otimes_{R}R[1/r] \\
      &\stackrel{\text{\ref{lem:SNrm-to-SNrm-localisation-preserves-certain-kernels}}}{\cong}
      \ker\left(M\otimes_{R}R[1/r]\to \coker\left(\phi^{\prime}\right)\otimes_{R}R[1/r]\right) \\
      &\cong
      \ker\left(M\otimes_{R}R[1/r]\to \coker\left(\phi^{\prime}\otimes_{R}\id_{R[1/r]}\right)\right) \\
      &=\im\left(\phi^{\prime}\otimes_{R}\id_{R[1/r]}\right).
    \end{align*}
    The application of Lemma~\ref{lem:SNrm-to-SNrm-localisation-preserves-certain-kernels} requires that
    the canonical morphism $M\to\coker\phi^{\prime}$ is strict and $\coker\phi^{\prime}$ does not have
    $r$-torsion. The strictness is formal, see~\cite[Remark 1.1.2(a)]{Sch99}. Regarding the torsion,
    consider $[m]\in\coker\phi^{\prime}=M/\ker\phi$ such that $r[m]=0$. That is $rm\in\ker\phi$.
    This implies $r\phi(m)=\phi(rm)=0$, thus $\phi(m)=0$ because $M^{\prime\prime}$ is $r$-torsion free.
    Therefore, $[m]=0$.
    \item Regarding its codomain,
    \begin{equation*}
      \left(\ker\phi\right)\otimes_{R}R[1/r]
      \cong\ker\left(\phi\otimes_{R}\id_{R[1/r]}\right)
    \end{equation*}
    by Lemma~\ref{lem:SNrm-to-SNrm-localisation-preserves-certain-kernels}.
  \end{itemize}
  It follows that $\iota\otimes_{R}\id_{R[1/r]}$ coincides with
  the canonical morphism
  $\im\left(\phi^{\prime}\otimes_{R}\id_{R[1/r]}\right)
  \to\ker\left(\phi\otimes_{R}\id_{R[1/r]}\right)$, which is thus an isomorphism.
  That is
  \begin{equation*}
    M^{\prime}\otimes_{R}R[1/r]
    \stackrel{\phi^{\prime}\otimes_{R}\id}{\longrightarrow}
    M\otimes_{R}R[1/r]
    \stackrel{\phi\widehat{\otimes}_{R}\id}{\longrightarrow}
    M^{\prime\prime}\otimes_{R}R[1/r]
  \end{equation*}
  is strictly exact. Now apply the separated completion functor; the
  Lemmata~\ref{em:commute-htensor-with-completion}
  and~\ref{lem:sepcompl-SNrmF-BanF-exact} imply
  that~(\ref{eq:completed-localisation-strictlyexact}) is strictly exact.
\end{proof}

Let  $F$ be a field which is complete with respect to a non-trivial
non-Archimedean valuation. Fix a \emph{pseudo-uniformiser} $\pi\in F$,  that
is $0<|\pi|<1$.

The Banach ring of \emph{power-bounded elements} is
$F^{\circ}=\{x\in F \colon |x|\leq 1\}$. Consider $R=F^{\circ}$
and $r=\pi$.

\begin{lem}\label{lem:localise-torsionfree-modules}
  Let $M$ denote an $F^{\circ}$-Banach module carrying
  the $\pi$-adic norm. Assume further that $M$ has no $\pi$-torsion.
  Then $M[1/\pi]$ is an $F$-Banach space, in particular
  $M\widehat{\otimes}_{F^{\circ}}F\isomap M[1/\pi]$.
\end{lem}

\begin{proof}[Proof of Lemma~\ref{lem:localise-torsionfree-modules}]
  $M$ is the unit ball of $M[1/\pi]$ because it is $\pi$-torsion free. Any given
  Cauchy sequence in $M[1/\pi]$ is bounded, thus we can assume it lies in
  $M$. Since $M$ is complete, it follows that such a sequence
  converges in $M$, therefore it converges in $M[1/\pi]$.
  This shows that $M[1/\pi]$ is complete.
  Finally, apply Lemma~\ref{lem:localisation-normedotimes}.  
\end{proof}

\begin{cor}\label{cor:completed-loc-exact-overF}
  Consider a strictly exact complex
  $M^{\prime}\stackrel{\phi^{\prime}}{\longrightarrow}M\stackrel{\phi}{\longrightarrow}M^{\prime\prime}$
  of $F^{\circ}$-Banach modules. If $M^{\prime\prime}$ is $\pi$-torsion free, then
  \begin{equation}\label{eq:completed-localisation-strictlyexact}
    M^{\prime}\widehat{\otimes}_{F^{\circ}}F
    \stackrel{\phi^{\prime}\widehat{\otimes}_{R}\id}{\longrightarrow}
    M\widehat{\otimes}_{F^{\circ}}F
    \stackrel{\phi\widehat{\otimes}_{R}\id}{\longrightarrow}
    M^{\prime\prime}\widehat{\otimes}_{F^{\circ}}F
  \end{equation}
  is a strictly exact complex of $F$-Banach spaces.
\end{cor}

\begin{proof}
  $\widehat{F^{\circ}[1/p]}\isomap F$ by Lemma~\ref{lem:localise-torsionfree-modules},
  thus Corollary~\ref{cor:completed-localisation-strictlyexact} applies.
\end{proof}

\begin{lem}\label{lem:completed-localisation-preserves-fiberproducts}
  Consider $F^{\circ}$-Banach modules $M\stackrel{\phi}{\rightarrow} T \stackrel{\psi}{\leftarrow} N$.
  If both $M$ and $N$ are $\pi$-torsion free,
  \begin{equation*}
    \left( M \times_{T} N\right) \widehat{\otimes}_{F^{\circ}}F
    \isomap \left(M\widehat{\otimes}_{F^{\circ}}F\right)
    \times_{\left(T\widehat{\otimes}_{F^{\circ}}F\right)}
    \left(N\widehat{\otimes}_{F^{\circ}}F\right),
  \end{equation*}
\end{lem}

\begin{proof}
  $\widehat{F^{\circ}[1/p]}\isomap F$ by Lemma~\ref{lem:localise-torsionfree-modules},
  thus Corollary~\ref{cor:completed-localisation-strictlyexact} applies
  again as
  $M \times_{T} N \cong \ker\left( M\times N \to T, (m,n)\mapsto \phi(m)-\psi(n) \right)$.
\end{proof}

\begin{lem}\label{lem:widehatotimesFcircF-preserves-strictmonos}
  Let $\phi\colon M\to N$ denote a morphism of $F^{\circ}$-Banach modules
  such that $N$ and $\coker\phi$ are $\pi$-torsion free.
  If $\phi$ is a strict monomorphism, then $\phi\widehat{\otimes}_{F^{\circ}}\id_{F}$
  is a strict monomorphism.
\end{lem}

\begin{proof}
  We fit $\phi$ into the strictly exact sequence $0\to M\to N\to\coker\phi\to0$.
  Now apply Corollary~\ref{cor:completed-loc-exact-overF}
  to find that
  $0\to M\widehat{\otimes}_{F^{\circ}}F \to N\widehat{\otimes}_{F^{\circ}}F \to\coker\phi\widehat{\otimes}_{F^{\circ}}F \to0$
  is a strictly exact sequence. In particular, $\phi\widehat{\otimes}_{F^{\circ}}\id_{F}$
  is a strict monomorphism.
\end{proof}

\begin{lem}\label{lem:HomcontS-commutes-completed-localisation-somenorms}
  Consider an $F^{\circ}$-Banach module $M$
  such that $\|\pi m\|=|\pi|\|m\|$ for all $m\in M$.
  Then the norm on $M\otimes_{F^{\circ}}F$ is equivalent to the norm given by
  $\|m\otimes 1/\pi^{i}\|=\|m\| / |\pi^{i}|$ for all $m\otimes 1/\pi^{i}\in M\otimes_{F^{\circ}}F$.
\end{lem}
  
\begin{proof}
  By Lemma~\ref{lem:localisation-normedotimes}, we may
  identify $\phi\colon M\otimes_{F^{\circ}}F\isomap M[1/\pi]$ as normed $k_{0}$-vector spaces.
  Here, the latter carries the following norm:
  $\|m/\pi^{i}\|$ is the infinum of $\|n\|/|\pi^{j}|$, varying over all expression $n/\pi^{j}=m/\pi^{i}$.
  The pullback of this norm along $\phi$ gives
  \begin{equation*}
    \|\frac{m}{\pi^{i}}\|
    =\inf_{n/\pi^{j}=m/\pi^{i}}\frac{\|n\|}{|\pi^{j}|}
    =\inf_{s\in\NN}\frac{\|\pi^{s}m\|}{|\pi^{s+i}|}
    =\inf_{s\in\NN}\frac{|\pi^{s}|\|m\|}{|\pi^{s+i}|}
    =\frac{\|m\|}{|\pi^{i}|},
  \end{equation*} 
  which is the desired norm.
\end{proof}

%%%%%%%%%%%%%%%%%%%%%%%%%%%%%%%%%%%%%%%%%%%%%%%%%%%%%%%%%%%%
%%%%%%%%%%%%%%%%%%%%%%%%%%%%%%%%%%%%%%%%%%%%%%%%%%%%%%%%%%%%
% Tubular neighborhoods
%%%%%%%%%%%%%%%%%%%%%%%%%%%%%%%%%%%%%%%%%%%%%%%%%%%%%%%%%%%%
%%%%%%%%%%%%%%%%%%%%%%%%%%%%%%%%%%%%%%%%%%%%%%%%%%%%%%%%%%%%

\section{Banach and ind-Banach modules of power series}
\label{subsec:Tubular}

Fix a Banach ring $R$.

\begin{notation}\label{notation:indexing-sets}
  Let $\Omega$ denote a possibly infinite set.
  %$\NN^{(\Omega)}$ is the set of maps $\Omega\to\NN$
  %with finite support.
  $\NN^{(\Omega)}$ is the set of tuples
  $\alpha=\left( \alpha_{\omega} \right)_{\omega\in\Omega}\in\NN^{\Omega}$
  such that $|\alpha|:=\sum_{\omega\in\Omega}\alpha_{\omega}$ is finite. Write
  %we use the multi-index notation
  \begin{equation*}
    \zeta^{\alpha}:=\prod_{\omega\in\Omega}\zeta_{\omega}^{\alpha_{\omega}}
  \end{equation*}
  for a given set of formal variables $\left\{ \zeta_{\omega} \colon \omega\in\Omega\right\}$
  and all $\alpha=\left( \alpha_{\omega} \right)_{\omega\in\Omega}\in\NN^{(\Omega)}$.
\end{notation}

%%%%%%%%%%%%%%%%%%%%%%%%%%%%%%%%%%%%%%%%%%%%%%%%%%%%%%%%%%%%
% restricted power series
%%%%%%%%%%%%%%%%%%%%%%%%%%%%%%%%%%%%%%%%%%%%%%%%%%%%%%%%%%%%

%\subsubsection{Spaces of restricted power series}

%In the following discussion, we frequently consider an $R$-Banach module $M$
%together with a set $\Omega$. We usually assume $M$ and $\Omega$ to be disjoint.
%This allows us to differentiate these constructions from a later setup where
%disjointness is not assumed, cf.~\ref{}.

\begin{defn}\label{defn:restrictedpowerseries}
  Let $M$ denote an $R$-Banach module and $\Omega$ a fixed set.
  An element of the $R$-module $M\left\<\zeta_{\omega} \colon \omega\in \Omega\right\>$
  is a formal expression
  \begin{equation*}
    \sum_{\alpha\in\NN^{(\Omega)}}m_{\alpha}\zeta^{\alpha}
    \in \prod_{\alpha\in\NN^{(\Omega)}}M\zeta^{\alpha},
  \end{equation*}
  such that the sets
  $\left\{\alpha\in\NN^{(\Omega)} \colon m_{\alpha}\geq\epsilon \right\}$
  are finite for all $\epsilon>0$.
  We equip $M\left\<\zeta_{\omega} \colon \omega\in \Omega\right\>$ with the norm
  \begin{equation*}
    \|\sum_{\alpha\in\NN^{(\Omega)}}m_{\alpha}\zeta^{\alpha}\|
    :=\sup_{\alpha\in\NN^{(\Omega)}}\|m_{\alpha}\|.
  \end{equation*}
\end{defn}

\begin{remark}\label{rem:restrictedpowerseries-is-c0space}
  When $R=F$ is a field, complete with respect to a non-trivial non-Archimedean
  valuation, there is the canonical isomorphism
  \begin{align*}
    c_{0}\left(\NN^{(\Omega)}\right) &\isomap M\left\<\zeta_{\omega} \colon \omega\in \Omega\right\>, \\
    \left(\phi\colon\NN^{(\Omega)}\to F\right)&\mapsto\sum_{\alpha\in\NN^{(\Omega)}}\phi\left(\alpha\right)\zeta^{\alpha},
  \end{align*}
  cf.~\cite[\S 3]{Sch02}.
  We chose the notation $M\left\<\zeta_{\omega} \colon \omega\in \Omega\right\>$
  to emphasise that Definition~\ref{defn:restrictedpowerseries} is a generalisaton of Tate algebras,
  in view of \S\ref{sec:BanachandindBanachcompletions}.
\end{remark}

\begin{notation}
  $\Ban_{R}^{\leq1}$ denotes the category whose objects are $R$-Banach modules
  and whose morphisms are the $R$-linear maps which are bounded by $1$.
\end{notation}

Let $\left\{M_{i}\right\}_{i\in I}$ denote a family of $R$-Banach modules. Their
coproduct exists in $\Ban_{R}^{\leq1}$,
cf.~\cite[\S 5.1.1.1]{benbassat2024perspectivefoundationsderivedanalytic}.
Denote it by $\coprod_{i\in I}^{\leq1}M_{i}$. It is the
\emph{non-expanding coproduct}.

\begin{lem}\label{lem:infiniteTatealgebrasascontractingcoproductandc0}
  Let $M$ denote an $R$-Banach module and $\Omega$ a fixed set. Then
  the universal property of the non-expanding coproduct induces an isomorphism
  \begin{equation}\label{eq:infiniteTatealgebrasascontractingcoproductandc0-contractingcoproduct}
    {\coprod_{\alpha\in\NN^{(\Omega)}}}^{\leq 1} M \zeta^{\alpha}
    \isomap
    M\left\<\zeta_{\omega} \colon \omega\in \Omega\right\>
  \end{equation}
  of $R$-Banach modules.
\end{lem}

\begin{proof}
  By \cite[\S 5.1.1.1]{benbassat2024perspectivefoundationsderivedanalytic},
  $\coprod_{\alpha\in\NN^{(\Omega)}}^{\leq 1} M \zeta^{\alpha}$
  is the completion of $\bigoplus_{\alpha\in\NN^{(\Omega)}} M \zeta^{\alpha}$
  equipped with the norm
  \begin{equation*}
    \|\left( m_{\alpha} \zeta^{\alpha}\right)_{\alpha\in\NN^{(\Omega)}}\|:=\sup_{\alpha\in\NN^{(\Omega)}}\|m_{\alpha}\|.
  \end{equation*}
  On the other hand, $M\left\<\zeta_{\omega} \colon \omega\in \Omega\right\>$
  is the completion of $M\left[\zeta_{\omega} \colon \omega\in \Omega\right]$, the
  subspace of all finite sums $\sum_{\alpha\in\NN^{(\Omega)}}m_{\alpha}\zeta^{\alpha}$ carrying the
  induced norm
  \begin{equation*}
    \|\sum_{\alpha\in\NN^{(\Omega)}}m_{\alpha}\zeta^{\alpha}\|
    =\sup_{\alpha\in\NN^{(\Omega)}}\|m_{\alpha}\|.
  \end{equation*}
  The morphism~(\ref{eq:infiniteTatealgebrasascontractingcoproductandc0-contractingcoproduct})
  thus arises as the completion of the isomorphism
  \begin{equation*}
    \bigoplus_{\alpha\in\NN^{(\Omega)}} M \zeta^{\alpha}
    \isomap M\left[\zeta_{\omega} \colon \omega\in \Omega\right]
  \end{equation*}
  of normed $R$-modules.
\end{proof}

\begin{lem}\label{lem:extend-maps-to-restrictedpowerseries-modules}
  $M$ is an $R$-Banach module and $\Omega$ a fixed set.
  Given any $m:=\left(m_{\alpha}\right)_{\alpha\in\NN^{(\Omega)}}\subseteq M$
  with $\|rm_{\alpha}\|=|r|\|m_{\alpha}\|$ for all $r\in R$ and $\|m_{\alpha}\|\leq 1$
  for all $\alpha\in\NN^{(\Omega)}$,
  \begin{equation*}
    \sum_{\alpha\in\NN^{(\Omega)}}r_{\alpha}\zeta^{\alpha}
    \mapsto\sum_{\alpha\in\NN^{(\Omega)}}r_{\alpha}m_{\alpha}
  \end{equation*}  
  defines a morphism
  $R\left\< \zeta_{\omega}\colon\omega\in\Omega\right\> \maps M$
  of $R$-Banach modules.
\end{lem}

\begin{proof}
  By Lemma~\ref{lem:infiniteTatealgebrasascontractingcoproductandc0},
  we can exploit the universal property of the non-expanding coproduct.
  Thus we may check that the maps
  $R \to T$, $r\mapsto rm_{\alpha}$
  are bounded by $1$ for all $\alpha\in\NN^{\Omega}$: for all $r\in R$,
  $\|rm_{\alpha}\|\leq|r|\|m_{\alpha}\|\leq|r|\cdot 1$.
\end{proof}

Lemma~\ref{lem:infiniteTatealgebrasascontractingcoproductandc0}(i) implies that
the operation $M\mapsto M\left\<\zeta_{\omega} \colon \omega\in \Omega\right\>$
defines a functor
%\begin{equation*}
  $\Ban_{R} \to \Ban_{R}$.
%\end{equation*}
By~\cite[Proposition 6.1.9]{KashiwaraSchapira2006}, it extends canonically to a functor
%\begin{equation*}
  $\IndBan_{R} \to \IndBan_{R}$,
%\end{equation*}
which we denote again by
$M_{\bullet}\mapsto M_{\bullet}\left\<\zeta_{\omega} \colon \omega\in \Omega\right\>$.
\emph{Loc. cit.} explains that this is not an abuse of notation, as
the canonical morphisms
\begin{equation*}
  \text{``}\varinjlim_{i\in I}\text{"}\left(M_{i}\left\<\zeta_{\omega} \colon \omega\in \Omega\right\>\right)
  \stackrel{\cong}{\longrightarrow}
  M_{\bullet}\left\<\zeta_{\omega} \colon \sigma\in \Omega\right\>
\end{equation*}
are isomorphisms for all
$M_{\bullet}=\text{``}\varinjlim\text{"}_{i\in I}M_{i}\in\IndBan_{R}$.

\begin{lem}\label{lem:restrictedpowerseries-from-hotimes}
  Fix an $R$-ind-Banach module $M_{\bullet}=\text{``}\varinjlim\text{"}_{i\in I}M_{i}$
  and a set $\Omega$. Then
  \begin{equation}\label{eq:restrictedpowerseries-from-hotimes}
    M_{\bullet} \widehat{\otimes}_{R} R\left\<\zeta_{\omega} \colon \omega\in \Omega\right\>
    \stackrel{\cong}{\longrightarrow} M_{\bullet}\left\<\zeta_{\omega} \colon \omega\in \Omega\right\>,
  \end{equation}
  the canonical morphism, is an isomorphism.
\end{lem}

\begin{proof}
  %Because $\widehat{\otimes}$ commutes with colimits,
  We may assume without
  loss of generality that $M_{\bullet}=M$ is an $R$-Banach module.
  The Lemma then follows from Lemma~\ref{lem:infiniteTatealgebrasascontractingcoproductandc0}
  and the~\cite[last sentence in the proof of Proposition 5.1.16]{benbassat2024perspectivefoundationsderivedanalytic}.
\end{proof}

\begin{cor}\label{cor:restrictedpowerseries-exact}
  $M_{\bullet}\mapsto M_{\bullet}\left\<\zeta_{\omega} \colon \omega\in\Omega\right\>$
  is strongly exact for every set $\Omega$.
\end{cor}

\begin{proof}
  This follows from Lemma~\ref{lem:restrictedpowerseries-from-hotimes}
  and~\cite[Proposition 5.1.16]{benbassat2024perspectivefoundationsderivedanalytic}.
\end{proof}

\begin{notation}
  $\beta+\gamma:=\left( \beta_{\omega} + \gamma_{\omega} \right)_{\omega\in\Omega} \in \NN^{(\Omega)}$
  for
  $\beta=\left( \alpha_{\omega} \right)_{\omega\in\Omega},
  \gamma=\left( \gamma_{\omega} \right)_{\omega\in\Omega}
  \in\NN^{(\Omega)}$.
\end{notation}

\begin{lem}\label{lem:mult-on-S-restrictedpowerseries}
  Let $S$ denote an $R$-Banach algebra and $\Omega$ a fixed set. Then
  \begin{equation*}
    \left(\sum_{\alpha\in\NN^{(\Omega)}}s_{1\alpha}\zeta^{\alpha}\right)
    \cdot\left(\sum_{\alpha\in\NN^{(\Omega)}}s_{2\alpha}\zeta^{\alpha}\right):=
    \sum_{\alpha\in\NN^{(\Omega)}}\left(\sum_{\alpha=\beta+\gamma}s_{1\beta}s_{2\gamma}\right)\zeta^{\alpha}
  \end{equation*}
  makes $S\left\<\zeta_{\omega} \colon \omega\in \Omega\right\>$
  an $S$-Banach algebra.
\end{lem}

\begin{proof}
  This is obvious.
\end{proof}

\begin{lem}\label{cor:extend-maps-to-restrictedpowerseries}
  Consider a morphism $\phi\colon S\to T$ of $R$-Banach algebras and a set $\Omega$.
  Given any tuple $t:=\left( t_{\omega} \right)_{\omega\in\Omega}\in T^{\Omega}$
  with $\|rt^{\alpha}\|=|r|\|t^{\alpha}\|$ for all $r\in R$ and $\|t^{\alpha}\|\leq 1$ for all
  $t\in\NN^{(\Omega)}$,
  \begin{equation*}
    \sum_{\alpha\in\NN^{(\Omega)}}s_{\alpha}\zeta^{\alpha}
    \mapsto\sum_{\alpha\in\NN^{(\Omega)}}\phi\left(s_{\alpha}\right)t^{\alpha}
  \end{equation*}
  defines a morphism
  $S\left\< \zeta_{\omega}\colon\omega\in\Omega\right\> \maps T$
  of $R$-Banach algebras.
\end{lem}

\begin{proof}
  See Lemma~\ref{lem:extend-maps-to-restrictedpowerseries-modules}
  for the construction of the morphism
  \begin{equation*}
    R\left\< \zeta_{\omega}\colon\omega\in\Omega\right\> \maps T,
    \sum_{\alpha\in\NN^{(\Omega)}}r_{\alpha}\zeta^{\alpha}
    \mapsto\sum_{\alpha\in\NN^{(\Omega)}}rt^{\alpha}
  \end{equation*}
  of $R$-Banach modules. 
  Apply Lemma~\ref{lem:restrictedpowerseries-from-hotimes}
  and~\cite[Proposition 1.5.2]{Sch99} to lift it to the desired morphism.
  It is bounded, $S$-linear, and multiplicative by construction.
\end{proof}

Fix a pseudo-uniformiser
$\pi\in R$, that is $0<|\pi|<1$.

\begin{notation}\label{defn:stalks-rigidanalyticfunctions-smaller-discs}
  Let $M$ be an $R$-Banach module, $\Omega$ a fixed set, and $q\in\NN$. Define
  \begin{equation*}
    M\left\< \frac{\zeta_{\omega}}{\pi^{q}} \colon \omega\in\Omega \right\>
    := M\left\< \zeta_{\omega} \colon \omega\in\Omega \right\>
         \left\< \eta_{\omega} \colon \omega\in\Omega \right\>/
         \overline{\left( \pi^{q}\eta_{\omega} - \zeta_{\omega} \colon \omega\in\Omega\right)},
\end{equation*}
an $R$-Banach module,
where the $\eta_{\omega}$ denote formal variables. We write
$\zeta_{\omega}/p^{q}$ for the images of the $\eta_{\omega}$ in
$M\left\< \frac{\zeta_{\omega}}{\pi^{q}} \colon \omega\in\Omega \right\>$.

If $M=S$ is an $R$-Banach algebra, we view
$S\left\< \frac{\zeta_{\omega}}{\pi^{q}}\colon\omega\in\Omega\right\>$
as an $R$-Banach algebra where the multiplication is induced by
Lemma~\ref{lem:mult-on-S-restrictedpowerseries}.
\end{notation}

Fix $q\in\NN$. The operation $M\mapsto M\left\<\frac{\zeta_{\omega}}{\pi^{q}} \colon \omega\in \Omega\right\>$
defines a functor
%\begin{equation*}
  $\Ban_{R} \to \Ban_{R}$.
%\end{equation*}
By~\cite[Proposition 6.1.9]{KashiwaraSchapira2006}, it extends canonically to a functor
%\begin{equation*}
  $\IndBan_{R} \to \IndBan_{R}$,
%\end{equation*}
which we denote again by
$M_{\bullet}\mapsto M_{\bullet}\left\<\frac{\zeta_{\omega}}{\pi^{q}} \colon \omega\in \Omega\right\>$.
\emph{Loc. cit.} explains that this is not an abuse of notation, as
the canonical morphisms
\begin{equation*}
  \text{``}\varinjlim_{i\in I}\text{"}\left(M_{i}\left\<\frac{\zeta_{\omega}}{\pi^{q}} \colon \omega\in \Omega\right\>\right)
  \stackrel{\cong}{\longrightarrow}
  M_{\bullet}\left\<\frac{\zeta_{\omega}}{\pi^{q}} \colon \sigma\in \Omega\right\>
\end{equation*}
are isomorphisms for all
$M_{\bullet}=\text{``}\varinjlim\text{"}_{i\in I}M_{i}\in\IndBan_{R}$.

\begin{lem}\label{lem:restrictedpowerseries-exact-strongerconvergence}
  $M_{\bullet}\mapsto M_{\bullet}\left\<\frac{\zeta_{\omega}}{\pi^{q}} \colon \omega\in\Omega\right\>$
  is strongly exact for every $q\in\NN$ and set $\Omega$.
\end{lem}

\begin{proof}
  Fix an $R$-Banach module $M$ and consider the morphism
  \begin{equation*}
    M\left\<\eta_{\omega}\colon\omega\in\Omega\right\>
    \to M\left\<\frac{\zeta_{\omega}}{\pi^{q}}\colon\omega\in\Omega\right\>,
    \eta_{\omega}\mapsto\frac{\zeta_{\omega}}{\pi^{q}}.
  \end{equation*}
  It is an isomorphism, as one directly constructs a two-sided inverse via
  the universal property of the non-expanding coproduct, see
  Lemma~\ref{lem:infiniteTatealgebrasascontractingcoproductandc0}.
  Lemma~\ref{lem:restrictedpowerseries-exact-strongerconvergence}
  thus follows from Corollary~\ref{cor:restrictedpowerseries-exact}.
\end{proof}

\begin{lem}\label{lem:diagram-formal-restricted-powerseries}
  Let $M$ denote an $R$-Banach module, $\Omega$ a fixed set, and $q\in\NN$.
  Then
  \begin{equation*}
    M\left\< \zeta_{\omega} \colon \omega\in\Omega \right\>
    \to M\left\< \frac{\zeta_{\omega}}{\pi^{q+1}} \colon \omega\in\Omega \right\>
  \end{equation*}
  lifts canonically to a morphism of $R$-Banach modules as follows:
  \begin{equation*}
    M\left\< \frac{\zeta_{\omega}}{\pi^{q}} \colon \omega\in\Omega \right\>
    \to M\left\< \frac{\zeta_{\omega}}{\pi^{q+1}} \colon \omega\in\Omega \right\>
  \end{equation*}
  If $M=S$ is an $R$-Banach algebra,
  this defines a morphism of $R$-Banach algebras.
\end{lem}

\begin{proof}
  Exploit the universal property of the non-expanding coproduct,
  cf. Lemma~\ref{lem:infiniteTatealgebrasascontractingcoproductandc0},
  to write down the morphism
  \begin{equation*}
    M\left\< \zeta_{\omega} \colon \omega\in\Omega \right\>\left\<\eta_{\omega}\colon\omega\in\Omega\right\>
    \to M\left\< \frac{\zeta_{\omega}}{\pi^{q+1}} \colon \omega\in\Omega \right\>,
  \end{equation*}
  given by $\eta_{\omega}^{\alpha}\mapsto \pi^{|\alpha|}\left(\zeta_{\omega} / \pi^{q}\right)^{\alpha}$
  for all $\alpha\in\NN^{(\Omega)}$. It factors through
  \begin{equation*}
    M\left\< \zeta_{\omega} \colon \omega\in\Omega \right\>
         \left\< \eta_{\omega} \colon \omega\in\Omega \right\>/
         \left( \pi^{q}\eta_{\omega} - \zeta_{\omega} \colon \omega\in\Omega\right)
    \to M\left\< \frac{\zeta_{\omega}}{\pi^{q+1}} \colon \omega\in\Omega \right\>.
  \end{equation*}
  Complete to get the desired map. The last sentence of Lemma~\ref{lem:diagram-formal-restricted-powerseries} is clear.
\end{proof}

Lemma~\ref{lem:diagram-formal-restricted-powerseries} furnishes a commutative diagram
of $R$-Banach modules:
\begin{equation}\label{eq:formal-restricted-powerseries}
  M\left\< \zeta_{\omega} \colon \omega\in\Omega \right\> \longrightarrow
  M\left\< \frac{\zeta_{\omega}}{\pi} \colon \omega\in\Omega \right\> \longrightarrow
  M\left\< \frac{\zeta_{\omega}}{\pi^{2}} \colon \omega\in\Omega \right\> \longrightarrow
  \dots.
\end{equation}

\begin{defn}
  Let $M$ denote an $R$-Banach module and $\Omega$ a fixed set.
  We denote the formal colimit of the diagram~(\ref{eq:formal-restricted-powerseries}) by
  \begin{align*}
    M\left\< \frac{\zeta_{\omega}}{\pi^{\infty}}\colon\omega\in\Omega\right\>
    &:=\text{``}\varinjlim_{q\in\NN}\text{"}M\left\< \frac{\zeta_{\omega}}{\pi^{q}} \colon \omega\in\Omega\right\>.
  \end{align*}
  This is, by definition, an $R$-ind-Banach module. If $M=S$ is an $R$-Banach algebra,
  we view
  $S\left\< \frac{\zeta_{\omega}}{\pi^{\infty}}\colon\omega\in\Omega\right\>$
  as an $R$-ind-Banach algebra. %where the multiplication is induced by
  %Lemma~\ref{lem:mult-on-S-restrictedpowerseries}.
\end{defn}

Following previous discussion,
$M\mapsto M\left\<\frac{\zeta_{\omega}}{\pi^{\infty}} \colon \omega\in \Omega\right\>$
extends to a functor
\begin{equation*}
  \IndBan_{R} \to \IndBan_{R},
  \text{``}\varinjlim_{i\in I}\text{"}M_{i}
  \mapsto
  \varinjlim_{i\in I}M_{i}\left\<\frac{\zeta_{\omega}}{\pi^{\infty}}\colon\omega\in\Omega\right\>.
\end{equation*}
which we denote by
$M_{\bullet}\mapsto M_{\bullet}\left\<\frac{\zeta_{\omega}}{\pi^{\infty}} \colon \omega\in \Omega\right\>$.
\emph{Loc. cit.} explains that this is not an abuse of notation, as
the canonical morphisms
\begin{equation*}
  \text{``}\varinjlim_{i\in I}\text{"}\left(M_{i}\left\<\zeta_{\omega} \colon \omega\in \Omega\right\>\right)
  \stackrel{\cong}{\longrightarrow}
  M_{\bullet}\left\<\zeta_{\omega} \colon \sigma\in \Omega\right\>
\end{equation*}
are isomorphisms for all
$M_{\bullet}=\text{``}\varinjlim\text{"}_{i\in I}M_{i}\in\IndBan_{R}$.

\begin{lem}\label{lem:restrictedpowerseries-exact-strongerconvergenceinfty}
  $M_{\bullet}\mapsto M_{\bullet}\left\<\frac{\zeta_{\omega}}{\pi^{\infty}} \colon \omega\in\Omega\right\>$
  is strongly exact for every set $\Omega$.
\end{lem}

\begin{proof}
  By Lemma~\ref{lem:messmse-indcat}, it suffices to check the lemma
  for the restriction of the functor to $\Ban_{R}$. Now the result follows from
  Lemma~\ref{lem:restrictedpowerseries-exact-strongerconvergence}
  and Corollary~\ref{cor:filteredcol-inIndBan-stronglyexact}.
\end{proof}

%%%%%%%%%%%%%%%%%%%%%%%%%%%%%%%%%%%%%%%%%%%%%%%%%%%%%%%%%%%%
%%%%%%%%%%%%%%%%%%%%%%%%%%%%%%%%%%%%%%%%%%%%%%%%%%%%%%%%%%%%
% Completions
%%%%%%%%%%%%%%%%%%%%%%%%%%%%%%%%%%%%%%%%%%%%%%%%%%%%%%%%%%%%
%%%%%%%%%%%%%%%%%%%%%%%%%%%%%%%%%%%%%%%%%%%%%%%%%%%%%%%%%%%%

\section{Banach and ind-Banach completions}
\label{sec:BanachandindBanachcompletions}

We continue to fix a Banach ring $R$, together with a pseudo-uniformiser $\pi\in R$.

Recall Notation~\ref{defn:stalks-rigidanalyticfunctions-smaller-discs}.

\begin{defn}\label{defn:Banach-completions}
  Fix a commutative $R$-Banach algebra $S$ and a subset $\Sigma\subseteq S$. Set
  \begin{align*}
     S\left\< \frac{\Sigma}{\pi^{q}}\right\>&:=
     \left. S\left\< \frac{\zeta_{\sigma}}{\pi^{q}} \colon \sigma\in\Sigma \right\> \middle/
     \overline{\left( \sigma - \pi^{q}\frac{\zeta_{\sigma}}{\pi^{q}} \colon \sigma\in\Sigma \right)} \right. ,
   \end{align*}
   for all $q\in\NN$. This is, by definition, an $R$-Banach algebra.
   We write $\sigma/p^{q}$ for the images of the $\zeta_{\sigma}/\pi^{q}$ in
   $S\left\< \frac{\Sigma}{\pi^{q}}\right\>$.
 \end{defn}
      
\begin{notation}
%For $m=0$, we usually omit the denominator $\pi^{0}$, that is
%\begin{equation*}
%  S\left\< \Sigma\right\>:=
%  S\left\< \frac{\Sigma}{\pi^{0}}\right\>.
%\end{equation*}
When $\Sigma=\left\{ s_{1},\dots,s_{d}\right\}$ is finite,
abbreviate
\begin{equation*}
  S\left\< \frac{s_{1},\dots,s_{d}}{\pi^{q}}\right\>:=
  S\left\< \frac{\Sigma}{\pi^{q}}\right\>.
\end{equation*}
%When $m=0$, we would here again usually omit the denominator $\pi^{0}$.
\end{notation}

\begin{remark}%\label{rem:defn-tubular-neighborhood}
  The ideal
  $\left( \sigma - p^{q}\frac{\zeta_{\sigma}}{\pi^{q}} \colon \sigma\in\Sigma \right)$
  in Definition~\ref{defn:Banach-completions} is
  not closed in general, see for example~\cite[Proposition 5.7]{BK20}.
\end{remark}

\begin{lem}\label{lem:diagram-formal-restricted-powerseries-factor-through}
  For a commutative $R$-Banach algebra $S$, a subset $\Sigma\subset S$,
  and $q\in\NN$,
  \begin{equation*}
    \delta^{q}\colon S\left\< \frac{\zeta_{\omega}}{\pi^{q}} \colon \omega\in\Omega \right\>
    \to S\left\< \frac{\zeta_{\omega}}{\pi^{q+1}} \colon \omega\in\Omega \right\>
  \end{equation*}
  denotes the map constructed by Lemma~\ref{lem:diagram-formal-restricted-powerseries}.
  It factors through a morphism
  \begin{equation*}
    S\left\< \frac{\Sigma}{\pi^{q}} \right\> \to S\left\<\frac{\Sigma}{\pi^{q+1}} \right\>
  \end{equation*}
  of $R$-Banach algebras.
\end{lem}

\begin{proof}
  %All $\delta^{q}$ are multitplicative.
  %It remains to compute that they factor through the ideals
  %$\overline{\left( \sigma - \pi^{q}\frac{\zeta_{\sigma}}{\pi^{q}} \colon \sigma\in\Sigma \right)}$.
  %This follows from the computation
  The computations
  \begin{equation*}
    \delta^{q}\left( \sigma - \pi^{q}\frac{\zeta_{\sigma}}{\pi^{q}} \right)
    = \sigma - \pi^{q+1}\frac{\zeta_{\sigma}}{\pi^{q+1}}
  \end{equation*}
  for all $\sigma\in\Sigma$ imply
  \begin{equation*}
    \delta^{q}\left( \left( \sigma - \pi^{q}\frac{\zeta_{\sigma}}{\pi^{q}} \colon \sigma\in\Sigma \right) \right)
    \subseteq \left( \sigma - \pi^{q+1}\frac{\zeta_{\sigma}}{\pi^{q+1}} \colon \sigma\in\Sigma \right).
  \end{equation*}
  The $\delta^{q}$ are bounded, thus the
  statement for the closures of the ideals follows.
\end{proof}

Lemma~\ref{lem:diagram-formal-restricted-powerseries-factor-through}
gives a commutative diagram
\begin{equation}\label{eq:formal-restricted-powerseries-SinSigma}
  S\left\< \Sigma \right\> \longrightarrow
  S\left\< \frac{\Sigma}{\pi^{1}} \right\> \longrightarrow
  S\left\< \frac{\Sigma}{\pi^{2}} \right\> \longrightarrow
  \dots.
\end{equation}

\begin{defn}
  Let $S$ denote a commutative $R$-Banach algebra and $\Sigma\subset S$ a fixed subset.
  We denote the formal colimit of the diagram~(\ref{eq:formal-restricted-powerseries}) by
  \begin{align*}
    S\left\< \frac{\Sigma}{\pi^{\infty}}\right\>:=\text{``}\varinjlim_{q\in\NN}\text{"}S\left\< \frac{\Sigma}{\pi^{q}}\right\>.
  \end{align*}
  This is, by definition, an $R$-ind-Banach algebra.
\end{defn}

\begin{notation}
  When $\Sigma=\left\{ s_{1},\dots,s_{d} \right\}\subseteq S$ is finite,
  abbreviate
  \begin{equation*}
    S\left\< \frac{s_{1},\dots,s_{d}}{\pi^{\infty}}\right\>:=
    S\left\< \frac{\Sigma}{\pi^{\infty}}\right\>
    \end{equation*}
\end{notation}

\begin{defn}
  An $R$-Banach module $M$ is \emph{bounded} if
  \begin{equation*}
    \sup_{m\in M}\|m\| < C
  \end{equation*}
  for some constant $C=C(M)>0$.
  In this case, $M$ is \emph{bounded by $C$}.
\end{defn}

\begin{example}
  $R=F$ is a field, complete with respect to a non-Archimedean
  non-trivial valuation. $S=F\<T,L\>$ is the Tate algebra in two variables
  and $\Sigma$ is its ideal generated by $T-L$.
  Then the canonical morphism
  \begin{equation*}
    F\<T,L\>\left\<\frac{T-L}{\pi}\right\>
    \stackrel{\not\cong}{\longrightarrow}F\<T,L\>\left\<\frac{\Sigma}{\pi}\right\>
  \end{equation*}
  is not injective: $T-L$ is a non-zero element of the domain,
  but it vanishes in $F\<T,L\>\left\<\Sigma/\pi\right\>$. Indeed,
  for all $i\in\NN$,
  \begin{equation*}
    |\pi^{i}|\|T-L\|
    =\|\pi^{i}\left(T-L\right)\|
    \leq\|\zeta_{\pi^{i}\left(T-L\right)}\|
    \leq 1.
  \end{equation*}
  This implies $\|T-L\|=0$, thus $T-L=0$.
  
  We resolve this issue by considering
  $R=F^{\circ}$, $S=F^{\circ}\left\<T,L\right\>$,
  and $\Sigma=(T-L)$ instead. Then
  $F^{\circ}\left\<T,L\right\>$ is bounded by $1$,
  thus Lemma~\ref{lem:Banach-completions-extend} applies.
\end{example}

\begin{defn}\label{defn:submultiplicative}
  A Banach algebra is \emph{submultiplicative}
  if for every two elements $x,y$ of this algebra,
  $\|xy\|\leq\|x\|\|y\|$.
\end{defn}

\begin{lem}\label{lem:Banach-completions-extend}
  Fix two commutative $R$-Banach algebras $S$ and $T$
  as well as a subset $\Sigma\subseteq S$.
  We further assume that $T$ is bounded by $1$ and for all $t_{1},t_{2}\in T$, $\|t_{1}t_{2}\|\leq\|t_{2}\|\|t_{2}\|$.
  Also, $\pi=\pi\cdot 1\in T$ is not a zero-divisor. Then, for every $q\in\NN$,
  there is a bijection between the set of morphisms
  $\phi\colon S\left\< \Sigma / \pi^{q} \right\>\to T$ of $R$-Banach algebras
  and the set of morphisms of $\psi\colon S\to T$ of $R$-Banach algebras
  such that
  \begin{itemize}
    \item[(i)] $\pi^{q}$ divides $\phi(\sigma)\in T$ for all $\sigma\in\Sigma$ and
    \item[(ii)] $\|\phi(\sigma)\|\leq \|\pi^{q}\|$ for all $\sigma\in\Sigma$. Here,
    $\|\cdot\|$ denotes the norm on $T$.
  \end{itemize}
  The bijection is given by $\phi \mapsto \psi:=\phi|_{S}$.
\end{lem}

\begin{proof}
  Firstly, we show that the assignment $\phi \mapsto \phi|_{S}$
  defines the desired map. Indeed, given a morphism
  $\phi\colon S\left\< \Sigma / \pi^{q} \right\>\to T$
  of $R$-Banach algebras, we find that
  \begin{itemize}
    \item[(i)] $\phi(\sigma)=\pi^{q}\phi\left( \sigma / \pi^{q} \right)\in T$. That is,
    $\pi^{q}$ divides $\phi(\sigma)$.
    \item[(ii)] Furthermore,
    \begin{equation*}
      \|\phi(\sigma)\|\leq\|\pi^{q}\phi\left( \sigma / \pi^{q} \right)\|
      \leq \|\pi^{q}\|\|\phi\left( \sigma / \pi^{q} \right)\|
      \leq \|\pi^{q}\|1
      = \|\pi^{q}\|.
    \end{equation*}
    This computation relies on the assumptions on $T$.
  \end{itemize}
  
  Secondly, we check that $\phi \mapsto \phi|_{S}$ is injective. Pick two maps
  $\phi,\phi^{\prime}\colon S\left\< \frac{\Sigma}{\pi^{q}} \right\>\to T$
  which agree on $S$. It suffices to show that they agree on elements
  of the form $\sigma/\pi^{q}$ where $\sigma\in\Sigma$. But we compute
  \begin{equation*}
    \pi^{q}\phi\left( \frac{\sigma}{\pi^{q}} \right)
    =\phi\left(\sigma\right)
    =\phi^{\prime}\left( \sigma \right)
    =\pi^{q}\phi^{\prime}\left( \frac{\sigma}{\pi^{q}} \right).
  \end{equation*}
  Since $\pi\in T$ is not a zero-divisor, we find
  $\phi\left( \sigma / \pi^{q} \right)=\phi^{\prime}\left( \sigma / \pi^{q} \right)$.

  Thirdly, we prove that $\phi \mapsto \phi|_{S}$
  is surjective. That is, we have to extend a given $\psi\colon S\to T$
  satisfying (i) and (ii) to a morphism
  $S\left\< \Sigma / \pi^{q} \right\>\to T$ of $R$-Banach algebras.
  Extend $\psi$ with Corollary~\ref{cor:extend-maps-to-restrictedpowerseries}
  to a morphism
  \begin{equation*}
    S\left\< \frac{\zeta_{\sigma}}{\pi^{q}} \colon \sigma\in\Sigma \right\>\to T,
    \frac{\zeta_{\sigma}}{\pi^{q}} \mapsto \frac{\phi(\sigma)}{\pi^{q}}
  \end{equation*}
  of $R$-Banach algebras.
  It vanishes on the ideal generated by all the $\sigma-\pi^{q}\zeta_{\sigma} / \pi^{q}$,
  and thus it vanishes on the closure. Therefore, the map above factors
  through a morphism
  \begin{equation*}
    \phi\colon S\left\< \frac{\Sigma}{\pi^{q}} \right\>\to T
  \end{equation*}
  of $R$-Banach algebras. By construction, $\phi|_{S}=\psi$.
\end{proof}

\begin{lem}\label{lem:completionsalongideals-generators}
  Consider a commutative $R$-Banach algebra $S$ which is bounded by $1$
  and for all $s_{1},s_{2}\in S$, $\|s_{1}s_{2}\|\leq\|s_{1}\|\|s_{2}\|$.
  Fix a subset $\Sigma\subseteq S$. $I:=\left(\Sigma\right)$
  is the ideal generated by this subset. Then, for every $q\in\NN$,
  the canonical morphism
  \begin{equation}\label{eq:completionsalongideals-generators}
    S\left\< \frac{\Sigma}{\pi^{q}} \right\>
    \stackrel{\cong}{\longrightarrow}
    S\left\< \frac{I}{\pi^{q}} \right\>
    \end{equation}
  is an isomorphism of $S$-Banach algebras.
  It is an isomorphisms of $S$-ind-Banach algebras
  for $q=\infty$.
\end{lem}

\begin{proof}
  Assume $q<\infty$ without loss of generality and
  construct the morphism
  \begin{equation*}
    S\left\< \frac{\zeta_{i}}{\pi^{q}}\colon i\in I\right\>
    \maps S\left\< \frac{\Sigma}{\pi^{q}} \right\>,
    \frac{\zeta_{i}}{\pi^{q}}
    \mapsto
    \frac{i}{\pi^{q}}
  \end{equation*}
  of $S$-Banach algebras with Lemma~\ref{cor:extend-maps-to-restrictedpowerseries}.
  It factors through the desired two-sided inverse
  of~(\ref{eq:completionsalongideals-generators}). 
\end{proof}

\begin{prop}\label{prop:adictop-ringsofpowerseries}
  Fix a commutative ring $S$ containing a finitely generated ideal $I$, such that $S$ is $I$-adically
  separated and complete. Equip $S$ with the $I$-adic norm. Let $\zeta$
  denote a formal variable. Then the topological algebra underlying the
  Banach algebra $S\left\<\zeta\right\>$ carries the $(I)$-adic topology,
  where $(I)\subseteq S\left\<\zeta\right\>$ is the ideal generated
  by the image of $I$ in $S\left\<\zeta\right\>$.
  
  Now fix an element $s\in S$ and a natural number $q$. Then
  the topological algebra underlying the Banach algebra $S\left\< s / \pi^{q} \right\>$ carries the
  $(I)$-adic topology, where $(I)\subseteq S\left\<s / \pi^{q}\right\>$ is the ideal generated
  by the image of $I$ in $S\left\<s / \pi^{q}\right\>$.
\end{prop}

We need the following two Lemmata
in order to prove Proposition~\ref{prop:adictop-ringsofpowerseries}.

\begin{lem}\label{lem:images-of-adic-topologies}
  Fix a surjective map $\phi\colon A\maps B$ of abstract commutative rings,
  together with an ideal $J\subseteq A$ such that $A$ is $J$-adically separated
  and complete. Equip $B$ with the quotient topology. Then $B$ carries the
  $\phi(J)$-adic topology.
\end{lem}

\begin{proof}
  A subset $U\subseteq B$ is open if its preimage $\phi^{-1}(U)\subseteq A$
  is open. In this case, there exists an $n\in\NN$ such that $J^{n}\subseteq\phi^{-1}(U)$,
  thus $\phi(J)^{n}=\phi\left(J^{n}\right)\subseteq U$.
  
  It remains to show that for each $n\in\NN$, $\phi(J)^{n}\subseteq B$ is open.
  Again, this is the case once its preimage is open. But its preimage $\phi^{-1}\left( \phi(J)^{n} \right)$
  is an ideal in $A$ and it contains $J^{n}$. This implies that it is open.
\end{proof}

\begin{lem}\label{lem:Banach-localisation-in-one-variable-description}
  Let $S$ commutative $R$-Banach algebra, $s\in S$, and $q\in\NN$. Then
  \begin{equation}\label{eq:Banach-localisation-in-one-variable-description}
    S\left\<\zeta\right\>/\overline{\left(\pi^{q}\zeta-s\right)}
    \isomap S\left\<\frac{s}{\pi^{q}}\right\>,
    \zeta\mapsto\frac{s}{\pi^{q}}
  \end{equation}
  is an isomorphism of $S$-Banach algebras.
\end{lem}

\begin{proof}
  One constructs the map~(\ref{eq:Banach-localisation-in-one-variable-description})
  and its two-sided inverse via Lemma~\ref{cor:extend-maps-to-restrictedpowerseries}.
\end{proof}

\begin{proof}[Proof of Proposition~\ref{prop:adictop-ringsofpowerseries}]
  Fix a power series $f=\sum_{\alpha\geq0}f_{\alpha}\zeta^{\alpha}\in S\left\< \zeta \right\>$.
  We have, by definition, $\|f\|\leq p^{-r}$ if and only if $\|f_{\alpha}\|\leq p^{-r}$ for all $\alpha\geq 0$.
  In order to prove the first statement, we therefore have to show
  \begin{equation}\label{eq:adictop-ringsofpowerseries}
    \|f_{\alpha}\|\leq p^{-r} \text{ for all } \alpha\geq 0
    \iff f\in(I)^{r}.
  \end{equation}
  The direction $\Leftarrow$ is clear. It remains to check $\implies$.
  Write $J:=I^{r}$ and fix a finite generating set
  $\left(x_{1},\dots,x_{n}\right)=J$. Define
  $e_{\alpha}:=\sup_{f_{\alpha}\in J^{e}}e$ for all $\alpha\geq0$.
  Note that for all $\alpha\geq0$, the assumption $\|f_{\alpha}\|\leq p^{-r}$
  implies $f\in I^{r}=J$, which gives
  $e_{\alpha}\geq 1$. Therefore, we can write, for all $\alpha\geq0$,
  \begin{equation*}
    f_{\alpha}=\sum_{i=1}^{n}f_{i\alpha}x_{i}
  \end{equation*}
  for certain $f_{i\alpha}\in J^{e_{\alpha}-1}$.
  Also, $f_{\alpha}\to0$ for $\alpha\to\infty$ implies
  \begin{equation}\label{eq:adictop-ringsofpowerseries-divergence}
  e_{\alpha}\to\infty \text{ for }
  \alpha\to\infty.
  \end{equation}
  Thus
  %\begin{equation*}
    $\|f_{i\alpha}\|
    \leq p^{-(e_{\alpha}-1)}
    \to 0 \text{ for } \alpha\to\infty$,
  %\end{equation*}
  and the formal power series
  $f_{i}:=\sum_{\alpha\geq0}f_{i\alpha}\zeta^{\alpha}$
  define elements of $S\<\zeta\>$ for all $i=1,\dots,n$.
  But then
  \begin{align*}
    f
    = \sum_{\alpha\geq0}f_{\alpha}\zeta^{\alpha}
    = \sum_{\alpha\geq0}\sum_{i=1}^{n}f_{i\alpha}x_{i}\zeta^{\alpha}
    = \sum_{i=1}^{n}\left(\sum_{\alpha\geq0}f_{i\alpha}\zeta^{\alpha}\right)x_{i}
    = \sum_{i=1}^{n}f_{i}x_{i}
    \in (J) = (I)^{r}.
  \end{align*}
  This finishes the proof of $\implies$ in~(\ref{eq:adictop-ringsofpowerseries}),
  and we get the first half of Proposition~\ref{prop:adictop-ringsofpowerseries}.
  Apply Lemma~\ref{lem:images-of-adic-topologies}
  to
  %\begin{equation*}
    $S\left\< \zeta \right\>\to S\left\<\zeta\right\>/\overline{\left(\pi^{q}\zeta-s\right)}
    \stackrel{\text{\ref{lem:Banach-localisation-in-one-variable-description}}}{\cong} S\left\< s / \pi^{q}\right\>$
  %\end{equation*}
  for the second half.
\end{proof}

Given an ideal $I\subseteq R$ in a commutative ring,
an element $r\in R$ is invertible in the $I$-adic completion
of $R$ if and only if its image in $R/I$ is invertible.
We regard the following Lemma~\ref{lem:unit-analytic-completion-if-unit-quotient}
as an analytic version of this fact.

\begin{lem}\label{lem:unit-analytic-completion-if-unit-quotient}
  Consider a commutative $R$-Banach algebra $S$
  and an ideal $I\subseteq S$. Fix $r\in R$ and $q\in\NN$ such
  that the following two conditions hold:
  \begin{itemize}
    \item[(i)] There exists an $s\in S$ with $sr-\pi^{q}\in I$ in $S$.
    \item[(ii)] The element $\pi\in S\left\< I / \pi^{q+1}\right\>$ is not a zero-divisor.
  \end{itemize}
  Then the image of $r$ in $S\left\< I / \pi^{q+1}\right\>[1/\pi]$ is a unit.
\end{lem}

\begin{proof}
  Condition (i) implies that the image of $sr-\pi^{q}$
  lies in the ideal $\pi^{q+1} I/\pi^{q+1}\subseteq S\left\< I / \pi^{q+1}\right\>$.
  By (ii), we can divide by $\pi^{q}$ and get
  $\pi^{-q}sr-1\in \pi I/\pi^{q+1}\subseteq S\left\< I / \pi^{q+1}\right\>$.
  Therefore, $\pi^{-q}sr\equiv 1$ modulo $\pi$
  in $S\left\< I / \pi^{q+1}\right\>$. Because
  $S\left\< I / \pi^{q+1}\right\>$ is $\pi$-adically complete,
  \cite[\href{https://stacks.math.columbia.edu/tag/05GI}{Tag 05GI}]{stacks-project}
  implies that $\pi^{-q}sr\in S\left\< I / \pi^{q+1}\right\>$
  is a unit. This implies that
  $sr$ is invertible in $S\left\< I / \pi^{q+1}\right\>[1/\pi]$,
  and so is $r$.
\end{proof}

%%%%%%%%%%%%%%%%%%%%%%%%%%%%%%%%%%%%%%%%%%%%%%%%%%%%%%%%%%%%
%%%%%%%%%%%%%%%%%%%%%%%%%%%%%%%%%%%%%%%%%%%%%%%%%%%%%%%%%%%%
%%%%%%%%%%%%%%%%%%%%%%%%%%%%%%%%%%%%%%%%%%%%%%%%%%%%%%%%%%%%
% Categories of sheaves
%%%%%%%%%%%%%%%%%%%%%%%%%%%%%%%%%%%%%%%%%%%%%%%%%%%%%%%%%%%%
%%%%%%%%%%%%%%%%%%%%%%%%%%%%%%%%%%%%%%%%%%%%%%%%%%%%%%%%%%%%
%%%%%%%%%%%%%%%%%%%%%%%%%%%%%%%%%%%%%%%%%%%%%%%%%%%%%%%%%%%%

\section{A variant of Hensel's Lemma}
\label{subsec:Hensel}

Here in \S\ref{subsec:Hensel}, we formulate and prove Theorem~\ref{thm:Hensel-analytic}.
We view this result as a more general version of Hensel's lemma,
which is tailored to the specific situation in \S\ref{subsec:proof-localdescrption-of-OBla}.

We start with notation.

Assume that $R=F^{\circ}$ is the subring of power-bounded elements
of a field $F$, complete with respect to a non-Archimedean
non-trivial valuation. Fix a fixed pseudo-uniformiser $\pi\in F^{\circ}$.

Furthermore, we fix a commutative $F^{\circ}$-Banach algebra
$S$ which is bounded by $1$ and submultiplicative.
Let $I\subseteq S$
be a proper ideal which is closed. For all $q\in\NN$, we write
  $S^{(q)} := S\left\< I / \pi^{q} \right\>$ and
  $T^{(q)} := S^{(q)}\widehat{\otimes}_{F^{\circ}} F$.
The reduction-mod-$I$ map induces morphisms of
$F^{\circ}$- respectively $F$-Banach algebras
  $\phi^{(q)}\colon S^{(q)} \to S/I$ and
  $\psi^{(q)}\colon T^{(q)} \to \left(S/I\right)\widehat{\otimes}_{F^{\circ}} F$.
We denote their kernels by
  $I^{(q)} \subseteq S^{(q)}$ and
  $J^{(q)} \subseteq T^{(q)}$, respectively.

Next, we fix formal variables $L=\left(L_{1},\dots,L_{n}\right)$.
Let $H=\left(H_{1},\dots,H_{n}\right)$ denote a tuple of
elements of $S\<L\>=S\left\<L_{1},\dots,L_{n}\right\>$.
The associated Jacobian is the $n\times n$-matrix
\begin{equation*}
  D_{H}:=
  \begin{pmatrix}
    \frac{\partial H_{1}}{\partial L_{1}} & \cdots & \frac{\partial H_{1}}{\partial L_{n}} \\
      \vdots & \ddots & \vdots \\
      \frac{\partial H_{n}}{\partial L_{1}} & \cdots & \frac{\partial H_{n}}{\partial L_{n}}
  \end{pmatrix}
  \in \Mat_{n}\left(S\left\< L \right\>\right)
\end{equation*}
with entries in $S\left\< L \right\>$.
Its determinant is $J_{H}:=\det D_{H}\in S\left\< L \right\>$.

\begin{condition}\label{condition:Hensel-analytic}
  Given $q\in\NN$ and
  $s=\left(s_{1},\dots,s_{n}\right)\in S^{n}$, we list the following conditions:
  \begin{itemize}
    \item[(a)] There exists an $r\in S$
      such that $J_{H}(s) r - \pi^{q} \in I$ in $S$.
    \item[(b)] The element $\pi\in S^{\left(\widetilde{q}\right)}$ is not a zero-divisor
      for all $\widetilde{q}>q$.
    \item[(c)] We have $H_{i}(s) \equiv 0 \mod I$ for all $i=1,\dots,n$.
  \end{itemize}
\end{condition}

From now on, we fix $q\in\NN$ and
$s=\left(s_{1},\dots,s_{n}\right)\in S^{n}$
such that Condition~\ref{condition:Hensel-analytic} is satisfied.

In the following, we define a certain constants $\widetilde{q}\in\NN$ and $\lambda\in\RR_{\geq0}$.

\begin{notation}
  Given a Banach space $V$, we denote its norm by $\|\cdot\|_{V}=\|\cdot\|$.
  
  We fix the usual norm
  \begin{equation*}
    \| \left( v_{1},\dots,v_{n}\right) \|_{V^{n}}
    := \max\left\{ \| v_{1} \|_{V} , \dots , \| v_{n} \|_{V} \right\}
  \end{equation*}
  on the direct sum $V^{n}$ of finitely many copies of $V$.
  
  Given a Banach ring $A$, we have the following norm on the matrix ring $\Mat_{n}(A)$:
  \begin{equation*}
    \|M\|_{\Mat_{n}(A)}:=\max_{i,j=1,\dots,n}\| m_{ij} \|_{A}
    \quad \text{ where } M=\left(m_{ij}\right)_{ij=1,\dots,n}.
  \end{equation*}
  %\todo[inline]{
  If $A$ is submultiplicative and bounded by $1$,
  then this coincides with the norm of the operator
  $M\colon A^{n}\to A^{n}$.
  %}
\end{notation}

\begin{remark}
  Since $S$ is bounded by 1, $\|H_{i}\|\leq 1$
  for all $i=1,\dots,n$.
\end{remark}

\iffalse %%% THE FOLLOWING IS NOT NEEDED ANYMORE, SEE THE REMARK ABOVE
\begin{conv}
  In order to prove Theorem~\ref{thm:Hensel-analytic},
  we may assume without loss of generality that $\|H_{i}\|\leq 1$
  for all $i=1,\dots,n$, which we do. Indeed,
  Condition~\ref{condition:Hensel-analytic} allows
  us to multiply by arbitrary powers of $\pi$.
\end{conv}
\fi %%% COMMENT ENDS

\begin{lem}\label{lem:DHs-invertible--thm:Hensel-analytic}
  The matrix
  $D_{H}(s)\in \Mat_{n}\left( T^{\left(q+1\right)} \right)$
  is invertible.
\end{lem}

\begin{proof}
  The conditions (i) and (ii) in Lemma~\ref{lem:unit-analytic-completion-if-unit-quotient}
  are satisfied because of Condition~\ref{condition:Hensel-analytic}(a) and (b).
  This implies that the image of $J_{H}(s)$ in $S^{(q+1)}[1/\pi]$
  is a unit. Since morphisms of rings such as
  $S^{(q+1)}[1/\pi]\to T^{\left(q+1\right)}$
  preserve units, it follows that the image of
  $J_{H}(s)$ in $T^{\left(q+1\right)}$ is a unit.
  Because $J_{H}(s)=\det D_{H}(s)$,
  and $T^{\left(q+1\right)}$ is a $\QQ$-module,
  Cramer's rule implies that
  $D_{H}(s)$ is invertible in $\Mat_{n}\left( T^{\left(q+1\right)} \right)$.
\end{proof}

Following the notation at the beginning of this section,
we write $\phi^{(q+1)}$ for the map
$S^{(q+1)}\to S/I$ which is induced by the reduction-mod-$I$ map.
Abusing notation, we denote the induced map
$\Mat_{n}\left(S^{(q+1)}\right) \to \Mat_{n}(S/I)$
again by $\phi^{(q+1)}$.

Thanks to Lemma~\ref{lem:DHs-invertible--thm:Hensel-analytic}, we have the
matrix $D_{H}(s)^{-1}\in \Mat_{n}\left( T^{\left(q+1\right)} \right)$.
In particular, we get
\begin{equation*}
  \phi^{(q+1)}\left(D_{H}(s)^{-1}\right) \in \Mat_{n}(S/I).
\end{equation*}
Now equip $S/I$ with the quotient norm and $\Mat_{n}(S/I)$ with the induced operator norm.
This makes the following Definition~\ref{defntildeq:Hensel-analytic} meaningful:

\begin{defn}\label{defntildeq:Hensel-analytic}
  Set
  \begin{equation*}
    \widetilde{q}:=\max\left\{ q+1 ,
      \left\lceil \log_{|\pi|} \left( \left\| \phi^{(q+1)}\left(D_{H}(s)^{-1}\right)\right\|_{\Mat_{n}(S/I)}^{-2} \right)+1
      \right\rceil \right\} \in \NN.
  \end{equation*}
\end{defn}

By definition, $\widetilde{q}>q$. Therefore, we have a map
$\Mat_{n}\left( T^{\left(q+1\right)} \right)\to\Mat_{n}\left( T^{\left(\widetilde{q}\right)} \right)$.
It is a map of rings, and therefore Lemma~\ref{lem:DHs-invertible--thm:Hensel-analytic}
implies that the matrix $D_{H}(s)\in \Mat_{n}\left( T^{\left(\widetilde{q}\right)} \right)$
is invertible. We can therefore define the map
\begin{equation*}
  \varphi\colon \left(T^{\left(\widetilde{q}\right)}\right)^{n} \to \left(T^{\left(\widetilde{q}\right)}\right)^{n},
    x \mapsto x - D_{H}(s)^{-1}H(x).
\end{equation*}
We would like to show that $\varphi$ is a contraction between metric spaces,
so that we can construct $\widetilde{s}$ via the Banach fixed-point theorem.
However, we are only able to show that this statement is correct once we restrict
to a small ball
\begin{equation*}
  B_{s}(\lambda):=\left\{
    x\in \left(T^{\left(\widetilde{q}\right)}\right)^{n}\colon \|x-s\|\leq \lambda
    \right\}
\end{equation*}
of radius $\lambda$ around $s$.
Here, $\lambda$ is as in the following Definition~\ref{defn:lambda--Hensel-analytic}.

\begin{defn}\label{defn:lambda--Hensel-analytic}
  Define
  \begin{equation*}
    \lambda
    :=\|D_{H}(s)^{-1}\|_{\Mat_{n}\left( T^{\left(\widetilde{q}\right)} \right)}
      \|H(s)\|_{\left(T^{\left(\widetilde{q}\right)}\right)^{n}}\in\RR_{\geq0}.
  \end{equation*}
\end{defn}

Here is our analytic version of Hensel's lemma:

\begin{thm}\label{thm:Hensel-analytic}
  We continue to fix $q\in\NN$ and
  $s=\left(s_{1},\dots,s_{n}\right)\in S^{n}$
  such that Condition~\ref{condition:Hensel-analytic} is satisfied.
  Then, for every $\widetilde{q}>q$ large enough, there exists a unique tuple
  $\widetilde{s}
    =\left(\widetilde{s}_{1},\dots,\widetilde{s}_{n}\right)
    \in \left( T^{\left( \widetilde{q} \right)} \right)^{n}$
  such that
  \begin{itemize}
    \item[(i)] $\widetilde{s}_{j}\equiv s_{j} \mod J^{\left(\widetilde{q}\right)}$ for all $j=1,\dots,n$,
    \item[(ii)] $H_{i}\left( \widetilde{s} \right)=0$ for all $i=1,\dots,n$, and
    \item[(iii)] $\|\widetilde{s}-s\|_{\left( T^{\left( \widetilde{q} \right)} \right)^{n}}\leq\lambda$.
  \end{itemize}
  These $\widetilde{s}_{i}$ are furthermore power-bounded.
\end{thm}

We spend the remainder of \S\ref{subsec:Hensel} on
a proof of Theorem~\ref{thm:Hensel-analytic}.
We fix all the notation and start with preliminary considerations.

\begin{remark}
  A standard proof of Hensel's lemma is via Newton's method.
  Unfortunately, we were unable to it to prove
  Theorem~\ref{thm:Hensel-analytic}; the reason is that Newton's method
  requires the division by several elements, but we were unable to
  show that these elements are units in our setting.
  Our proof of Theorem~\ref{thm:Hensel-analytic} therefore
  follows a slightly nonstandard proof of Hensel's classical lemma,
  which we learned from~\cite[\S6]{KConradHensel}.
  %We refer the reader to \emph{loc. cit.} for further details.
\end{remark}

\begin{lem}\label{lem:lambdaDhsinverse-less-1--Hensel-analytic}
  The following inequality holds:
  $\lambda \|D_{H}(s)^{-1}\|_{\Mat_{n}\left( T^{\left(\widetilde{q}\right)} \right)} < 1$.
\end{lem}

\begin{proof}
  We deduce
  \begin{align*}
    \|H(s)\|_{\left(T^{\left(\widetilde{q}\right)}\right)^{n}}
    \leq |\pi|^{\widetilde{q}}
    &< \| \phi^{(q+1)}\left(D_{H}(s)^{-1}\right) \|_{\Mat_{n}\left(S/I\right)}^{-2} \\
    &= \| \phi^{(\widetilde{q})}\left(D_{H}(s)^{-1}\right) \|_{\Mat_{n}\left(S/I\right)}^{-2}
    \leq \|D_{H}(s)^{-1}\|_{\Mat_{n}\left(T^{\left(\widetilde{q}\right)}\right)}^{-2}
  \end{align*}
  where the first inequality comes from Condition~\ref{condition:Hensel-analytic}(c),
  the strict inequality follows from the Definition~\ref{defntildeq:Hensel-analytic}
  of $\widetilde{q}$, the equality follows from the equality
  $\phi^{(q+1)}\left(D_{H}(s)^{-1}\right)=\phi^{(\widetilde{q})}\left(D_{H}(s)^{-1}\right)$,
  and the final inequality comes from the fact that the operator norm
  of the map $\Mat_{n}\left(T^{\left(\widetilde{q}\right)}\right)\to\Mat_{n}(S/I)$
  is bounded by $1$. This implies
  \begin{align*}
    \lambda \|D_{H}(s)^{-1}\|_{\Mat_{n}\left( T^{\left(\widetilde{q}\right)} \right)}
    &= \|D_{H}(s)^{-1}\|_{\Mat_{n}\left( T^{\left(\widetilde{q}\right)} \right)}^{2}
      \|H(s)\|_{\left(T^{\left(\widetilde{q}\right)}\right)^{n}} \\
    &< \|D_{H}(s)^{-1}\|_{\Mat_{n}\left( T^{\left(\widetilde{q}\right)} \right)}^{2}
        \|D_{H}(s)^{-1}\|_{\Mat_{n}\left( T^{\left(\widetilde{q}\right)} \right)}^{-2}
      =1,
  \end{align*}
  using the Definition~\ref{defn:lambda--Hensel-analytic} of $\lambda$.
\end{proof}

\begin{lem}\label{lem:lambda-less-1--Hensel-analytic}
  The following inequality holds:
  $\lambda  < 1$.
\end{lem}

\begin{proof}
  Using the submultiplicativity and Lemma~\ref{lem:lambdaDhsinverse-less-1--Hensel-analytic}, we deduce
  \begin{equation*}
    \lambda \|D_{H}(s)\|_{\Mat_{n}\left( T^{\left(\widetilde{q}\right)} \right)}^{-1}
    \leq \lambda \|D_{H}(s)^{-1}\|_{\Mat_{n}\left( T^{\left(\widetilde{q}\right)} \right)}
    < 1.
  \end{equation*}
  Because $S$ is bounded by $1$, we find $\|D_{H}(s)\|_{\Mat_{n}\left( T^{\left(\widetilde{q}\right)} \right)}\leq 1$,
  and therefore
  \begin{equation*}
    \lambda
    < \|D_{H}(s)\|_{\Mat_{n}\left( T^{\left(\widetilde{q}\right)} \right)}
    \leq 1,
  \end{equation*}
  as desired.
\end{proof}

\begin{lem}\label{lem:S-to-Ttq-bounded1--Hensel-analytic}
  The canonical map $S\to T^{\left(\widetilde{q}\right)}$ is bounded by $1$.
\end{lem}

\begin{proof}
  This follows from the definitions.
\end{proof}

\begin{lem}\label{lem:Ttq-submultiplicative--Hensel-analytic}
  $T^{\left(\widetilde{q}\right)}$ is submultiplicative.
\end{lem}

\begin{proof}
  This follows from the submultiplicativity of $S$.
\end{proof}

\begin{lem}\label{lem:Taylor--Hensel-analytic}
  For all $x\in \left(T^{\left(\widetilde{q}\right)}\right)^{n}$,
  \begin{equation*}
    H(x) =
    H(s) + D_{H}(s)(x-s) + z_{x}.
  \end{equation*} 
  Here, $z_{x}\in \left(T^{\left(\widetilde{q}\right)}\right)^{n}$ is of the form
  \begin{equation*}
    z_{x} = \left( \sum_{\alpha\in\NN^{n} , |\alpha|\geq2} h_{i\alpha}\left( x-s \right)^{\alpha} \right)_{i=1,\dots,n}.
  \end{equation*}
  If $x\in B_{s}(\lambda)$, this implies implies
  $\|z_{x}\|_{\left(T^{\left(\widetilde{q}\right)}\right)^{n}}\leq\|x-s\|_{\left(T^{\left(\widetilde{q}\right)}\right)^{n}}^{2}$.
\end{lem}

\begin{proof}
  This follows by looking at the Taylor expansion. The finally inequality follows because
  the coefficients $h_{i\alpha}$ lie in $S$. Since $S$ is bounded by $1$ and thanks
  to Lemma~\ref{lem:S-to-Ttq-bounded1--Hensel-analytic}, we have
  \begin{equation*}
    \| h_{i\alpha} \|_{T^{\left(\widetilde{q}\right)}}
    \leq \| h_{i\alpha} \|_{S}
    \leq 1.
  \end{equation*}
  Now Lemma~\ref{lem:Ttq-submultiplicative--Hensel-analytic} implies:
  \begin{equation*}
    \|z_{x}\|_{\left(T^{\left(\widetilde{q}\right)}\right)^{n}}
    =\max_{i=1,\dots,n}\| \sum_{\alpha\in\NN^{n} , |\alpha|\geq2} h_{i\alpha}\left( x-s \right)^{\alpha} \|_{\left(T^{\left(\widetilde{q}\right)}\right)^{n}}
    \leq\max_{\alpha\in\NN^{n} , |\alpha|\geq2}\|(x-s)^{\alpha}\|_{\left(T^{\left(\widetilde{q}\right)}\right)^{n}}
    \leq\|x-s\|_{\left(T^{\left(\widetilde{q}\right)}\right)^{n}}^{2}.
  \end{equation*}
  The last inequality here follows because $x\in B_{s}(\lambda)$.
  Indeed, this implies $\|x-s\|_{\left(T^{\left(\widetilde{q}\right)}\right)^{n}}\leq\lambda<1$,
  cf. Lemma~\ref{lem:lambda-less-1--Hensel-analytic}.
\end{proof}

\begin{lem}\label{lem:varphi-preserves-B--Hensel-analytic}
  If $x\in B_{s}(\lambda)$, then $\varphi(x)\in B_{s}(\lambda)$.
\end{lem}

\begin{proof}
  To simplify notation, we denote all norms by $\|\cdot\|$.
  Let $x\in B_{s}(\lambda)$. We aim to
  check $\sigma(x)\in B_{s}(\lambda)$, that is $\|\sigma(x)-s\|\leq \lambda$.
  Fix the notation as in Lemma~\ref{lem:Taylor--Hensel-analytic} and
  compute:
  \begin{align*}
    \| D_{H}(s)^{-1} H(x) \|
    &=\| D_{H}(s)^{-1} \left( H(s) + D_{H}(s)(x-s) + z_{x} \right) \| \\
    &=\| D_{H}(s)^{-1} H(s) + (x-s) + D_{H}(s)^{-1} z_{x} \| \\
    &\leq \max\left\{ \| D_{H}(s)^{-1} H(s) \| , \| x-s \| , \| D_{H}(s)^{-1} z_{x} \| \right\} \\
    &\leq \max\left\{ \lambda , \| D_{H}(s)^{-1} z_{x} \| \right\};
  \end{align*}
  here, the last inequality uses $\| D_{H}(s)^{-1} H(s) \|\leq\| D_{H}(s)^{-1}\| \| H(s) \|\leq\lambda$,
  which follows from the submultiplicativity of $S$ the definition of $\lambda$.
  Because
  \begin{equation*}
    \| D_{H}(s)^{-1} z_{x} \|
    \leq \| D_{H}(s)^{-1} \| \| z_{x} \|
    \stackrel{\text{\ref{lem:Taylor--Hensel-analytic}}}{\leq} \| D_{H}(s)^{-1} \| \| x-s \|^{2}
    \stackrel{\text{\ref{lem:lambdaDhsinverse-less-1--Hensel-analytic}}}{\leq} \| D_{H}(s)^{-1} \| \lambda^{2}
    < \lambda,
  \end{equation*}
  this implies
  \begin{equation*}
    \| D_{H}(s)^{-1} H(x) \|
    \leq \max\left\{ \lambda , \| D_{H}(s)^{-1} z_{x} \| \right\}
    = \lambda.
  \end{equation*}
  Consequently,
  \begin{equation*}
    \| \varphi(x) - s \|
    =\| (x-s) - D_{H}(s)^{-1} H(x) \|
    \leq \max\left\{ \|x-s\| , \|D_{H}(s)^{-1} H(x)\| \right\}
    \leq\lambda.
  \end{equation*}
  With other words, $\varphi(x)\in B_{s}(\lambda)$.
\end{proof}

Lemma~\ref{lem:varphi-preserves-B--Hensel-analytic} implies that
$\varphi$ restricts to the endomorphism
\begin{equation*}
  \varphi_{|B_{s}(\lambda)}\colon B_{s}(\lambda)\to B_{s}(\lambda)
\end{equation*}
of complete metric spaces. We continue by showing that it is contracting.

\begin{lem}\label{lem:Xalpha-Yalpha-Hensel}
  Consider the ring $\ZZ\left[ X_{1},\dots,X_{n},Y_{1},\dots,Y_{n}\right]$,
  where $X_{i}, Y_{i}$ denote formal variables.
  We use multi-index notation $X=\left(X_{1},\dots,X_{n}\right)$
  and $Y=\left(Y_{1},\dots,Y_{n}\right)$, and for $\alpha=\left(\alpha_{1},\dots,\alpha_{n}\right)\in\NN^{n}$
  with $\sum_{i=1}^{n}\alpha_{i}\geq2$, the polynomial $X^{\alpha}-Y^{\alpha}$
  lies in the following product of ideals:
  \begin{equation*}
    X^{\alpha}-Y^{\alpha}
    \in\left(X_{1}-Y_{1},\dots,X_{n}-Y_{n}\right)\left(X_{1},\dots,X_{n},Y_{1},\dots,Y_{n}\right).
  \end{equation*}
\end{lem}

\begin{proof}
  We argue via induction in $n$.
  If $n=1$ and so that $\alpha\in\NN$, $\alpha\geq2$, we have:
  \begin{equation*}
    X^{\alpha}-Y^{\alpha} = (X-Y)\sum_{i=0}^{\alpha-1}X^{i}Y^{\alpha-1-i}
    \in(X-Y)(X,Y)
  \end{equation*}
  writing $X=X_{1}$, $Y=Y_{1}$. Now assume the result holds for fixed
  $n \geq 2$. Then:
  \begin{align*}
    X^{\alpha}-Y^{\alpha}
    =\left( X_{1}^{\alpha_{1}} - Y_{1}^{\alpha_{1}} \right)
        X_{2}^{\alpha_{2}} \cdots X_{n+1}^{\alpha_{n+1}}
      + Y_{1}^{\alpha_{1}}
        \left( X_{2}^{\alpha_{2}} \cdots X_{n+1}^{\alpha_{n+1}} - Y_{2}^{\alpha_{2}}\cdots Y_{n}^{\alpha_{2}}\right)
  \end{align*}
  and the result follows from the induction hypothesis.
\end{proof}

\begin{lem}\label{lem:Banchfixedpt-Hensel}
  $\varphi_{|B_{s}(\lambda)}\colon B_{s}(\lambda)\to B_{s}(\lambda)$
  is a contracting map between metric spaces.
\end{lem}
  
\begin{proof}
  To simplify notation, we denote all norms by $\|\cdot\|$.
  Let $x,y\in B_{s}(\lambda)$ be arbitrary.
  By Lemma~\ref{lem:Taylor--Hensel-analytic}, we have expressions
  \begin{align*}
    H(x) &= H(s) + D_{H}(s)(x-s) + z_{x}, \\
    H(y) &= H(s) + D_{H}(s)(y-s) + z_{y},
  \end{align*}
  where
  \begin{equation*}
    z_{x} = \left( \sum_{\alpha\in\NN^{n} , |\alpha|\geq2} h_{i\alpha}\left( x-s \right)^{\alpha} \right)_{i=1,\dots,n}
    \quad\text{and}\quad
    z_{y} = \left( \sum_{\alpha\in\NN^{n} , |\alpha|\geq2} h_{i\alpha}\left( y-s \right)^{\alpha} \right)_{i=1,\dots,n}.
  \end{equation*}
  Because $\|h_{i\alpha}\|\leq 1$ for all $i,\alpha$,
  Lemma~\ref{lem:Xalpha-Yalpha-Hensel}
  (for $X=x-s$ and $Y=y-s$) implies
  \begin{equation*}
    \| z_{x} - z_{y} \|
    \leq \left( \sum_{\alpha\in\NN^{n} , |\alpha|\geq2}
      h_{i\alpha}\left( \left( x-s \right)^{\alpha}  - \left( y-s \right)^{\alpha} \right)
    \right)_{i=1,\dots,n}
    \leq \|x-y\| \max\left\{ \|x-s\| , \|y-s\| \right\}.
  \end{equation*}
  Now we have
  \begin{align*}
    \|\varphi(x)-\varphi(y)\|
    &= \|x-y-D_{H}(s)^{-1}(H(x)-H(y))\| \\
    &= \|D_{H}(s)^{-1}(z_{x}-z_{y})\| \\
    &\leq \|D_{H}(s)^{-1}\| \|z_{x}-z_{y}\| \\
    &\leq \|D_{H}(s)^{-1}\| \|x-y\| \max\left\{ \|x-s\| , \|y-s\| \right\} \\
    &\leq \|D_{H}(s)^{-1}\|\lambda \|x-y\|
  \end{align*}
  The result follows because
  $\|D_{H}(s)^{-1}\|\lambda<1$, by Lemma~\ref{lem:lambdaDhsinverse-less-1--Hensel-analytic}.
\end{proof}

\begin{proof}[Proof of Theorem~\ref{thm:Hensel-analytic}]
  Thanks to Lemma~\ref{lem:Banchfixedpt-Hensel},
  we can apply the Banach fixed-point theorem to the contracting
  map $\varphi_{|B_{s}(\lambda)}\colon B_{s}(\lambda)\to B_{s}(\lambda)$.
  This yields a unique fixed point $\widetilde{s}\in B_{s}(\lambda)$, that is
  $\varphi\left(\widetilde{s}\right)=\widetilde{s}$.
  From the definition of $\varphi$, this implies $H_{i}\left( \widetilde{s} \right)=0$
  for all $i=1,\dots,n$; this is Theorem~\ref{thm:Hensel-analytic}(ii).
  Note that (iii) holds since $\widetilde{s}\in B_{s}(\lambda)$.
  
  We continue with checking (i).
  To do this, we need the explicit description
  $\widetilde{s}=\lim_{j\to\infty}\varphi(s)$, which is also given by the Banach
  fixed-point theorem. We aim to show:
  $\widetilde{s}-s\in \left(J^{\left(\widetilde{q}\right)}\right)^{n}$.
  Since $\left(J^{\left(\widetilde{q}\right)}\right)^{n}$ is closed,
  because $J^{\left(\widetilde{q}\right)}$ is the kernel of a bounded linear map,
  it suffices to check $\varphi^{j}(s)-s\in\left(J^{\left(\widetilde{q}\right)}\right)^{n}$
  for all $j\in\NN_{\geq1}$. This will follows from the following fact
  \begin{equation}\label{eq-proof:thm:Hensel-analytic}
    H\left( \varphi^{j}\left(s\right) \right) \in \left(J^{\left(\widetilde{q}\right)}\right)^{n}
  \end{equation}
  which we check by induction along $j$. For $j=1$, this is Condition~\ref{thm:Hensel-analytic}(c).
  Now suppose~(\ref{eq-proof:thm:Hensel-analytic})
  is known for fixed $j$. Then:
  \begin{equation*}
    H\left( \varphi^{j+1}\left(s\right) \right)
    =H\left( \varphi^{j}\left(s\right) - D_{H}(s)^{-1}H\left(\varphi^{j}(s) \right) \right)
    \stackrel{\clubsuit}{\equiv} H\left( \varphi^{j}\left(s\right) - 0 \right)
    = H\left( \varphi^{j}\left(s\right) \right)
    \stackrel{\clubsuit}{\equiv} 0
    \mod \left(J^{\left(\widetilde{q}\right)}\right)^{n},
  \end{equation*}
  where the $\clubsuit$ follow from the induction hypothesis.
  This confirms~(\ref{eq-proof:thm:Hensel-analytic}). As a consequence, we find
  \begin{equation*}
    \varphi^{j+1}(s)-\varphi^{j}(s)
    =-D_{H}(s)^{-1}H\left(\varphi^{j}(s)\right)
    \in\left(J^{\left(\widetilde{q}\right)}\right)^{n}
  \end{equation*}
  for all $j\in\NN$. This implies
  \begin{equation*}
    \varphi^{j}(s)-s
    =\sum_{i=1}^{j}\varphi^{j}(s)-\varphi^{j-1}(s)
    \in\left(J^{\left(\widetilde{q}\right)}\right)^{n},
  \end{equation*}
  completing the proof of (i).
   
  Uniqueness is clear from the Banach fixed-point theorem.
  
  Finally, we check that the $\widetilde{s}_{i}$ are power-bounded.
  $\widetilde{s}\in B_{s}(\lambda)$ implies
  \begin{equation*}
    \|\widetilde{s}_{i}\|
    \leq\|\widetilde{s}\|
    =\|\widetilde{s}-s+s\|
    \leq\max\left\{ \| \widetilde{s} - s \| , \| s \| \right\}
    \leq\max\left\{\lambda,1\right\}
    \stackrel{\text{\ref{lem:lambda-less-1--Hensel-analytic}}}{=}1.
  \end{equation*}
  The second to last inequality follows
  because $S$ is bounded by $1$, therefore
  $\|s\|\leq 1$. Because the norm on $S$ is submultiplicative,
  this implies that $\widetilde{s}_{i}$ is power-bounded.
\end{proof}

%%%%%%%%%%%%%%%%%%%%%%%%%%%%%%%%%%%%%%%%%%%%%%%%%%%%%%%%%%%%
%%%%%%%%%%%%%%%%%%%%%%%%%%%%%%%%%%%%%%%%%%%%%%%%%%%%%%%%%%%%
% Useful results
%%%%%%%%%%%%%%%%%%%%%%%%%%%%%%%%%%%%%%%%%%%%%%%%%%%%%%%%%%%%
%%%%%%%%%%%%%%%%%%%%%%%%%%%%%%%%%%%%%%%%%%%%%%%%%%%%%%%%%%%%

\section{The $\eta$-operator}
\label{subsec:the-eta-operator}

The $\eta$-operator, as introduced by Berthelot-Ogus~\cite{BerthelotOgus78Notesoncrystallinecoh},
kills torsion in the cohomology of a given cochain complex.
Our exposition follows~\cite[\S 6]{BhattMorrowScholze2018}

\begin{defn}\label{defn:algebraic-decalage}
  Let $S$ be an abstract commutative ring, $s\in S$ is not a zero-divisor. Define,
  for every cochain complex $N^{\bullet}$ of $S$-Banach modules concentrated in degrees
  $\geq0$ with differentials $d$,
  \begin{equation*}
    \left(\eta_{s}N^{\bullet}\right)^{i}
    :=\left\{n\in s^{i}N^{i}\colon d(n)\in s^{i+1}N^{i+1}\right\}.
  \end{equation*}
  This defines a subcomplex $\eta_{s}N^{\bullet}\subseteq N^{\bullet}$.
\end{defn}

The following is a special case of~\cite[Lemma 6.4]{BhattMorrowScholze2018}.

\begin{lem}\label{lem:eta-operator-kills-torsion}
  Fix the notation from Definition~\ref{defn:algebraic-decalage}. One computes, for all $i\in\ZZ$,
  \begin{equation*}
    \Ho^{i}\left(\eta_{s}N^{\bullet}\right)
    \cong
    \Ho^{i}\left( N^{\bullet} \right) / \left(\text{$s$-torsion}\right)
  \end{equation*}
\end{lem}

Now fix a Banach ring $R$. Let $r\in R$, which is not a zero-divisor.
We give a variant of the $\eta$-operator for complexes of $R$-Banach modules.

\begin{condition}\label{cond:Banachmoduledecalage-reconstructionpaper}
  Consider a cochain complex $M^{\bullet}$ of $R$-Banach modules.
  \begin{itemize}
    \item[(i)] $M^{i}=0$ for all $i<0$.
    \item[(ii)] $r^{i}M^{i}\subseteq M^{i}$ is a closed subset for all $i\geq0$.
    \item[(iii)] $M^{i}$ does not have $r$-torsion for all $i\geq0$.
  \end{itemize}
\end{condition}

\begin{defn}\label{defn:decalage-for-Banachmodules}
  Given a cochain complex $M^{\bullet}$ of $R$-Banach modules,
  we equip $\eta_{r}M^{\bullet}\subseteq M^{\bullet}$
  with the induced norms.
\end{defn}

\begin{lem}\label{lem:decalage-defined-reconstructionpaper}
  Consider a cochain complex $M^{\bullet}$ of $R$-Banach modules satisfying
  Condition~\ref{cond:Banachmoduledecalage-reconstructionpaper}.
  Then $\eta_{r}M^{\bullet}\subseteq M^{\bullet}$, equipped with the induced norms,
  is also cochain complex of $R$-Banach modules.
\end{lem}

\begin{proof}
  The following commutative diagrams are pullback diagrams of normed $R$-modules:
  \begin{equation}\label{cd:dacalagefunctor-degrees-pullbackdiagram}
    \begin{tikzcd}
      \left(\eta_{r}M^{\bullet}\right)^{i}
        \arrow[hook]{r}\arrow{d} &
      r^{i}M^{i}
        \arrow{d}{d} \\
      r^{i+1}M^{i+1}
        \arrow[hook]{r} &
      M^{i+1}.
    \end{tikzcd}
  \end{equation}
  Condition~\ref{cond:Banachmoduledecalage-reconstructionpaper}(ii) implies that all
  $\left(\eta_{r}M^{\bullet}\right)^{i}$ are $R$-Banach modules.
\end{proof}

We study the behaviour of the $\eta$-operator with respect to localisation.

\begin{notation}\label{notation:indbanach-invert-r}
  Introduce the following $R$-ind-Banach algebra
  \begin{align*}
    \text{``}\varinjlim_{r\times}\text{"}R
    &:=
    \text{``}\varinjlim\text{"}\left(
      \dots\stackrel{r\times}{\longrightarrow} R
        \stackrel{r\times}{\longrightarrow} R
        \stackrel{r\times}{\longrightarrow}\dots\right).
  \end{align*}
  The multiplication is
  \begin{equation*}
    \text{``}\varinjlim_{r\times}\text{"} R
    \widehat{\otimes}_{R}\text{``}\varinjlim_{r\times}\text{"} R
    \cong
    \text{``}\varinjlim_{2r\times}\text{"}R\widehat{\otimes}_{R} R
    \stackrel{\text{``}\varinjlim\text{"}_{2r\times}\mu}{\longrightarrow}
    \text{``}\varinjlim_{2r\times}\text{"} R
    \cong\text{``}\varinjlim_{r\times}\text{"} R,
  \end{equation*}
  where $\mu$ is the multiplication on $r$.
  The unit is induced by the unit $R\to R$, that is the identity.
\end{notation}

\begin{lem}\label{lem:decalage-indBanach-localisation-is-just-completed-localisation}
  Consider a cochain complex $M^{\bullet}$ of $R$-Banach modules satisfying
  Condition~\ref{cond:Banachmoduledecalage-reconstructionpaper}.
  Then the canonical morphism
  \begin{equation}\label{lem:decalage-indBanach-localisation-is-just-completed-localisation-theiso}
    \eta_{r}M^{\bullet} \widehat{\otimes}_{R} \text{``}\varinjlim_{r\times}\text{"} R
    \isomap
    M^{\bullet} \widehat{\otimes}_{R} \text{``}\varinjlim_{r\times}\text{"} R
  \end{equation}
  is an isomorphism of cochain complexes of $R$-ind-Banach modules.
\end{lem}

\begin{proof}
  First, we claim that the canonical
  morphisms~(\ref{eq:decalage-indBanach-localisation-is-just-completed-localisation-lemma})
  are isomorphism for all $i\geq0$:
  \begin{equation}\label{eq:decalage-indBanach-localisation-is-just-completed-localisation-lemma}
    r^{i}M^{i}\widehat{\otimes}_{R} \text{``}\varinjlim_{r\times}\text{"} R
    \isomap M^{i} \widehat{\otimes}_{R} \text{``}\varinjlim_{r\times}\text{"} R.
  \end{equation}
  These are colimits of
  the strict monomorphisms $r^{i}M^{i}\hookrightarrow M^{i}$ of $R$-Banach modules.
  Corollary~\ref{cor:filteredcol-inIndBan-stronglyexact}
  implies that~(\ref{eq:decalage-indBanach-localisation-is-just-completed-localisation-lemma})
  is a strict monomorphism. Next, observe that
  $r^{i}M^{i}\widehat{\otimes}_{R} \text{``}\varinjlim\text{"}_{r\times} R\cong\text{``}\varinjlim\text{"}_{r\times}r^{i}M^{i}$
  and $M^{i}\widehat{\otimes}_{R} \text{``}\varinjlim\text{"}_{r\times} R\cong\text{``}\varinjlim\text{"}_{r\times}M^{i}$
  are monomorphic ind-objects, see~\cite[Definition 3.30]{BBB16},
  by Condition~\ref{cond:Banachmoduledecalage-reconstructionpaper}(iii).
  Since the map between the underlying abstract $R$-modules
  \begin{equation*}
    \varinjlim_{r\times}|r^{i}M^{i}|
    =|r^{i}M^{i}\widehat{\otimes}_{R} \text{``}\varinjlim_{r\times}\text{"} R|
    \to |M^{i} \widehat{\otimes}_{R} \text{``}\varinjlim_{r\times}\text{"} R|
    =\varinjlim_{r\times}|M^{i}|
  \end{equation*}
  is surjective,~\cite[Corollary 3.32]{BBB16} implies
  that~(\ref{eq:decalage-indBanach-localisation-is-just-completed-localisation-lemma})
  is an epimorphism. It is thus an isomorphism.
  
  Now compute, using the pullback diagram~(\ref{cd:dacalagefunctor-degrees-pullbackdiagram}),
  \begin{equation*}
  \begin{split}
    \left(\eta_{r}M^{\bullet}\right)^{i}\widehat{\otimes}_{R} \text{``}\varinjlim_{r\times}\text{"} R
    &\stackrel{\text{(\ref{cd:dacalagefunctor-degrees-pullbackdiagram})}}{\cong}
      \left(r^{i+1}M^{i+1}\times_{M^{i+1}} r^{i}M^{i}\right)\widehat{\otimes}_{R}\text{``}\varinjlim_{r\times}\text{"} R \\
    &\stackrel{\text{\ref{cor:filteredcol-inIndBan-stronglyexact}}}{\cong}
      \left(r^{i+1}M^{i+1}\widehat{\otimes}_{R}\text{``}\varinjlim_{r\times}\text{"} R\right)
      \times_{M^{i+1}\widehat{\otimes}_{R}\text{``}\varinjlim_{r\times}\text{"} R}
      \left(r^{i}M^{i}\widehat{\otimes}_{R}\text{``}\varinjlim_{r\times}\text{"} R\right) \\
    &\stackrel{\text{(\ref{eq:decalage-indBanach-localisation-is-just-completed-localisation-lemma})}}{\cong}
      \left(M^{i+1}\widehat{\otimes}_{R}\text{``}\varinjlim_{r\times}\text{"} R\right)
      \times_{\left(M^{i+1}\widehat{\otimes}_{R}\text{``}\varinjlim_{r\times}\text{"} R\right)}
      \left(M^{i}\widehat{\otimes}_{R}\text{``}\varinjlim_{r\times}\text{"} R\right) \\
    &\cong
      M^{i}\widehat{\otimes}_{R}\text{``}\varinjlim_{r\times}\text{"} R
  \end{split}
  \end{equation*}
  for all $i\geq0$. This gives that~(\ref{lem:decalage-indBanach-localisation-is-just-completed-localisation-theiso})
  is an isomorphism.
\end{proof}

\begin{lem}\label{lem:decalage-completed-localisation-is-just-completed-localisation}
  Let $F$ denote a field complete with respect to a non-trivial non-Archimedean valuation.
  $F^{\circ}\subseteq F$ is the Banach ring of power-bounded elements,
  and $\pi\in F^{\circ}$ denotes a pseudo-uniformiser, that is $0<|\pi|<1$.
  Consider a cochain complex $M^{\bullet}$ of $F^{\circ}$-Banach modules satisfying
  Condition~\ref{cond:Banachmoduledecalage-reconstructionpaper} for $R=F^{\circ}$ and $r=\pi$.
  Then the canonical morphism
  \begin{equation}\label{lem:decalage-completed-localisation-is-just-completed-localisation}
    \eta_{\pi}M^{\bullet} \widehat{\otimes}_{F^{\circ}} F
    \isomap
    M^{\bullet} \widehat{\otimes}_{F^{\circ}} F
  \end{equation}
  is an isomorphism of cochain complexes of $R$-ind-Banach modules.
\end{lem}

\begin{proof}
  First, we claim that the canonical morphisms
  \begin{equation}\label{eq:decalage-completed-localisation-is-just-completed-localisation-lemma}
    \left(\pi^{i}M^{i}\right)\widehat{\otimes}_{F^{\circ}} F
    \isomap M^{i}\widehat{\otimes}_{F^{\circ}} F
  \end{equation}
  are isomorphisms for all $i\geq0$. Because $\pi^{i}M^{i}\hookrightarrow M^{i}$
  is a strict monomorphism,
  Lemma~\ref{lem:SNrm-to-SNrm-localisation-preserves-certain-kernels}
  implies that $\pi^{i}M^{i}\otimes_{F^{\circ}} F\to M^{i}\otimes_{F^{\circ}} F$ is a strict monomorphism;
  here \emph{loc. cit.} applies because of Condition~\ref{cond:Banachmoduledecalage-reconstructionpaper}(iii).
  As $\pi^{i}M^{i}\otimes_{F^{\circ}} F\to M^{i}\otimes_{F^{\circ}} F$ is surjective, it follows
  that it is an isomorphism of normed $F^{\circ}$-modules. Thus its
  completion~(\ref{eq:decalage-completed-localisation-is-just-completed-localisation-lemma})
  is an isomorphism of $F$-Banach spaces.
  
  Finally, the following computations for all $i\geq0$
  give that~(\ref{lem:decalage-completed-localisation-is-just-completed-localisation})
  is an isomorphism:
  \begin{equation*}
  \begin{split}
    \left(\eta_{\pi}M^{\bullet}\right)^{i}\widehat{\otimes}_{{F}^{\circ}}F
    &\stackrel{\text{(\ref{cd:dacalagefunctor-degrees-pullbackdiagram})}}{\cong}
      \left(\pi^{i+1}M^{i+1}\times_{M^{i+1}} \pi^{i}M^{i}\right)\widehat{\otimes}_{F^{\circ}}F \\
    &\stackrel{\text{\ref{lem:completed-localisation-preserves-fiberproducts}}}{\cong}
      \left(r^{i+1}M^{i+1}\widehat{\otimes}_{F^{\circ}}F\right)
      \times_{M^{i+1}\widehat{\otimes}_{F^{\circ}}F} \left(r^{i}M^{i}\widehat{\otimes}_{F^{\circ}}F\right) \\
    &\stackrel{\text{(\ref{eq:decalage-completed-localisation-is-just-completed-localisation-lemma})}}{\cong}
    \left(M^{i+1}\widehat{\otimes}_{F^{\circ}}F\right)
    \times_{M^{i+1}\widehat{\otimes}_{F^{\circ}}F} \left(M^{i}\widehat{\otimes}_{F^{\circ}}F\right) \\
    &\cong M^{i}\widehat{\otimes}_{F^{\circ}}F.
  \end{split}
  \end{equation*}
  Here, Lemma~\ref{lem:completed-localisation-preserves-fiberproducts} applies
  thanks to Condition~\ref{cond:Banachmoduledecalage-reconstructionpaper}(iii).
\end{proof}

We often investigate Banach modules via their associated graded modules.
Therefore, we introduce a variant of the $\eta$-operator for filtered modules.
From now on, $R$ denotes a filtered ring, and $r\in R$ is again not a zero-divisor.
All filtrations are descending.

\begin{condition}\label{cond:filteredmodulesdecalage-reconstructionpaper}
  Consider a cochain complex $M^{\bullet}$ of filtered $R$-modules.
  \begin{itemize}
    \item[(i)] $M^{i}=0$ for all $i<0$.
    \item[(ii)] $r^{i}M^{i}\subseteq M^{i}$ is a closed subset for all $i\geq0$ with respect
    to the filtration topology.
    \item[(iii)] $M^{i}$ does not have $r$-torsion for all $i\geq0$.
    \item[(iv)] $\gr M^{i}$ does not have $\sigma\left(r^{i}\right)$-torsion. Here,
      $\sigma\left(r^{i}\right)\in\gr R$ denotes the principal symbol of $r^{i}$.
  \end{itemize}
\end{condition}

\begin{lem}\label{lem:decalagefilteredmodules-defined}
  Consider a cochain complex $M^{\bullet}$ of filtered $R$-modules
  satisfying Conditions~\ref{cond:filteredmodulesdecalage-reconstructionpaper}(i), (ii), and (iii).
  Then $\eta_{r}M^{\bullet}\subseteq M^{\bullet}$, equipped with the induced filtrations,
  is a cochain complex of filtered $R$-modules as well. The filtrations are
  exhaustive (respectively separated, respectively complete) if the
  filtrations on the $M^{i}$ are exhaustive (respectively separated, respectively complete).
\end{lem}

\begin{proof}
  To show completeness, proceed as in the proof of
  Lemma~\ref{lem:decalage-defined-reconstructionpaper}.
  The filtrations are exhaustive (respectively separated)
  as they are induced by exhaustive (respectively separated) filtrations.
\end{proof}

\begin{lem}\label{lem:decalage-commutes-gr}
  Consider a cochain complex $M^{\bullet}$ of filtered $R$-modules
  satisfying Condition~\ref{cond:filteredmodulesdecalage-reconstructionpaper}.
  If $\sigma(r)\in\gr^{0}R$, then the grading on $\gr M^{\bullet}$ induces a grading on
  $\eta_{\sigma(r)}\gr M^{\bullet}$. In this case, we have
  \begin{equation}\label{eq:decalage-commutes-gr-theiso}
    \gr\eta_{r}M^{\bullet}\cong\eta_{\sigma(r)}\gr M^{\bullet},
  \end{equation}
  as complexes of graded $\gr R$-modules. This isomorphism is functorial.
\end{lem}

\begin{proof}
  Because $\sigma\left(r\right)\in\gr^{0} R$,
  we find $\eta_{\sigma(r)}\gr M^{\bullet}\cong \bigoplus_{n\in\ZZ}\eta_{\sigma(r)}\gr^{n} M^{\bullet}$.
  This gives the desired grading. Furthermore, for all $i\in\ZZ$, we have the sequence of filtered $R$-modules
  \begin{equation*}
    0 \longrightarrow
    M^{i} \stackrel{\iota}{\longrightarrow}
    M^{i}
  \end{equation*}
  where $\iota$ is the multiplication-by-$r^{i}$ map.
  It is strictly exact because its associated graded is exact by
  Condition~\ref{cond:filteredmodulesdecalage-reconstructionpaper}(iv),
  cf.~\cite[Chapter I, \S 4.1, page 31-32, Theorem 4]{HuishiOystaeyen1996}.
  Since passing to the associated graded
  transforms strictly exact sequences to exact sequence,
  cf. \emph{loc. cit.}, we find
  \begin{equation}\label{eq:1--lem:decalage-commutes-gr}
    \gr\left(r^{i}M^{i}\right)
    =\gr\left(\im\iota\right)
    =\im\left(\gr\iota\right)
    =\sigma\left(r^{i}\right)\gr\left(M^{i}\right).
  \end{equation}
  Now isomorphism~(\ref{eq:decalage-commutes-gr-theiso})
  follows from the computation
  \begin{equation*}
  \begin{split}
    \gr\left( \left(\eta_{r}M^{\bullet}\right)^{i}\right)
    &=\gr\left( r^{i}M^{i} \times_{d,M^{i+1}} r^{i+1}M^{i+1} \right) \\
    &\stackrel{\text{\ref{lem:gr:commutes-pullback}}}{\cong}
    \gr\left( r^{i}M^{i}\right) \times_{\gr\left(d\right),\gr\left(M^{i+1}\right)} \gr\left(r^{i+1}M^{i+1}\right) \\
    &\stackrel{\text{(\ref{eq:1--lem:decalage-commutes-gr})}}{\cong}
    \sigma\left(r^{i}\right)\gr\left(M^{i}\right)
    \times_{\gr\left(d\right),\gr\left(M^{i+1}\right)}
    \sigma\left(r^{i}\right)\gr\left(M^{i+1}\right) \\
    &=\left(\eta_{\sigma(r)}\gr M^{\bullet}\right)^{i}.
  \end{split}
  \end{equation*}
  Here we used Lemma~\ref{lem:gr:commutes-pullback} below.
  The functoriality is clear.
\end{proof}

\begin{lem}\label{lem:gr:commutes-pullback}
  We have the canonical isomorphism of graded $\gr R$-modules
  \begin{equation*}
    \gr\left(M\times_{T}N\right)\cong\gr M \times_{\gr T} \gr N
  \end{equation*}
  for every diagram $M\to T\leftarrow N$ of filtered $R$-modules.
\end{lem}

\begin{proof}
  Since
  $M \times_{T} N \cong \ker\left( M\times N \to T, (m,n)\mapsto \phi(m)-\psi(n) \right)$,
  \cite[Chapter I, \S 4.1, page 31-32, Theorem 4]{HuishiOystaeyen1996} applies.
\end{proof}

%%%%%%%%%%%%%%%%%%%%%%%%%%%%%%%%%%%%%%%%%%%%%%%%%%%%%%%%%%%%
%%%%%%%%%%%%%%%%%%%%%%%%%%%%%%%%%%%%%%%%%%%%%%%%%%%%%%%%%%%%
% FUNCTIONAL ANALYSIS
%%%%%%%%%%%%%%%%%%%%%%%%%%%%%%%%%%%%%%%%%%%%%%%%%%%%%%%%%%%%
%%%%%%%%%%%%%%%%%%%%%%%%%%%%%%%%%%%%%%%%%%%%%%%%%%%%%%%%%%%%

\section{Continuous group cohomology}\label{subsec:contgrpcohomology}

Fix a Banach ring $R$. $G$ denotes a profinite group.

%%%%%%%%%%%%%%%%%%%%%%%%%%%%%%%%%%%%%%%%%%%%%%%%%%%%%%%%%%%%
%%%%%%%%%%%%%%%%%%%%%%%%%%%%%%%%%%%%%%%%%%%%%%%%%%%%%%%%%%%%
% FUNCTIONAL ANALYSIS
%%%%%%%%%%%%%%%%%%%%%%%%%%%%%%%%%%%%%%%%%%%%%%%%%%%%%%%%%%%%
%%%%%%%%%%%%%%%%%%%%%%%%%%%%%%%%%%%%%%%%%%%%%%%%%%%%%%%%%%%%

\subsection{Ind-$G$-$k$-Banach spaces}

Recall the notion of Ind-$G$-$k$-Banach spaces as in~\S\label{sec:indGkbanachspaces}.

The following is needed for our definition of continuous group cohomology later on.

\begin{defn}
  Given a compact topological space $S$ and 
  an $R$-Banach module $M$,
  $\intHom_{\cont}\left(S,M\right)$ denotes the $R$-Banach module
  of all continuous functions $S\to M$ together with the supremum norm.
\end{defn}

We record the following results for future reference.

\begin{lem}\label{lem:koszulsequences-inHomcontSA}
  Let $S$ denote a compact topological space, and $A$ is an $R$-Banach algebra.
  View $\intHom_{\cont}\left(S,A\right)$ as an $R$-Banach algebra with pointwise
  multiplication. Then, given a regular sequence
  $a_{1},\dots,a_{n}\in A$, the constant functions $a_{1},\dots,a_{n}\in\intHom_{\cont}\left(S,A\right)$
  form a regular sequence.
\end{lem}

\begin{proof}
  As $a_{1},\dots,a_{n}\in A$ is a regular sequence, we find
  a nonzero element $[a]\in A/\left(a_{1},\dots,a_{n}\right)$. Thus 
  $\intHom_{\cont}\left(S,A\right)/\left(a_{1},\dots,a_{n}\right)$ is nonzero,
  as it contains the equivalence class of the constant map $s\mapsto a$.
  It remains to show that, for every $i=1,\dots,n$,
  $[a_{i}]\in\intHom_{\cont}\left(S,A\right)/\left(a_{1},\dots,a_{i-1}\right)$
  is nonzero. Suppose there exists an $i$ such that $a_{i}$ becomes zero
  in the quotient. Then we may write $a_{i}=\sum_{j=1}^{i-1}\phi_{j}a_{j}$,
  where the $\phi_{j}$ are continuous maps $S\to A$. But then, for any
  $s\in S$, $a_{i}=\sum_{j=1}^{i-1}\phi_{j}(s)a_{j}$. Thus
  $a_{i}\in\left(a_{1},\dots,a_{i-1}\right)\subseteq A$, which contradicts
  the regularity of $a_{1},\dots,a_{i-1}$.
\end{proof}

\begin{lem}\label{lem:aHomcontSA-iso-HomcontSaA-reconstructionpaper}
  Consider an $R$-Banach algebra $A$ and fix an element $a\in A$.
  The canonical map
  \begin{equation}\label{eq:aHomcontSA-iso-HomcontSaA-themap-reconstructionpaper}
    a\intHom_{\cont}\left(S , A \right) \isomap \intHom_{\cont}\left(S , aA \right)
  \end{equation}
  is an isomorphism of $R$-Banach modules if $A\to A$, $x\mapsto ax$ is a strict monomorphism.
\end{lem}

\begin{proof}
  The assumptions imply that $\alpha\colon A\isomap aA$, $x\mapsto ax$ is an isomorphism of
  $R$-Banach modules. We can thus define a two-sided inverse of
  (\ref{eq:aHomcontSA-iso-HomcontSaA-themap-reconstructionpaper}) via
  $\phi\mapsto a\left(\alpha^{-1}\circ\phi\right)$.
\end{proof}

$\intHom_{\cont}\left(S,-\right)$ defines a functor $\Ban_{R}\to\Ban_{R}$.
Apply~\cite[Proposition 6.1.9]{KashiwaraSchapira2006}
to lift it to a cocontinuous functor
$\IndBan_{R}\to\IndBan_{R}$ which we denote again by $\intHom_{\cont}\left(S,-\right)$.

\begin{lem}\label{lem:finitedirectsum-comutes-withHomcont-reconstructionpaper}
  $\intHom_{\cont}\left(S,-\right)$ commutes with direct sums of $R$-ind-Banach modules.
\end{lem}

\begin{proof}
  An arbitrary colimit is a filtered colimit of
  finite direct sums. Thanks to~\cite[Proposition 6.1.9(ii)]{KashiwaraSchapira2006},
  it remains to show that $\intHom_{\cont}\left(S,-\right)$ commutes
  with direct sums of finitely many $R$-Banach modules $M_{1},\dots,M_{n}$.
  There is an obvious morphism of $R$-Banach modules
  \begin{equation}\label{eq:--lem:finitedirectsum-comutes-withHomcont-reconstructionpaper}
    \bigoplus_{i=1}^{n}\intHom_{\cont}\left(S , M_{i}\right)
    \to
    \intHom_{\cont}\left(S , \bigoplus_{i=1}^{n} M_{i}\right).
  \end{equation}
  It is bijective, because the topological spaces underlying the direct sums
  in Banach modules are the products in the category of topological spaces, and
  $\Hom_{\cont}\left(S , - \right)$ is a restriction
  of the homomorphisms in the category of topological spaces,
  thus it commutes with limits. One checks directly from the definitions
  that~(\ref{eq:--lem:finitedirectsum-comutes-withHomcont-reconstructionpaper}) preserves the norms as well.
  Thus it is an isomorphism.
\end{proof}

\iffalse %%% do not need this anymore
\begin{prop}
  For any closed normal subgroup $H\subseteq G$ and any
  $M\in\IndBan_{R}\left(G\right)$,
  we have the following spectral sequence
  \begin{equation*}
    \Ho_{\cts}^{i}\left( G/H , \Ho_{\cts}^{j}\left( H , M \right) \right)
    \implies
    \Ho_{\cts}^{i+j}\left( G , M \right).
  \end{equation*}
\end{prop}

\begin{proof}
\end{proof}

\fi %%%

%%%%%%%%%%%%%%%%%%%%%%%%%%%%%%%%%%%%%%%%%%%%%%%%%%%%%%%%%%%%
% FUNCTIONAL ANALYSIS
%%%%%%%%%%%%%%%%%%%%%%%%%%%%%%%%%%%%%%%%%%%%%%%%%%%%%%%%%%%%

\subsection{Ind-$G$-$k$-Banach spaces as sheaves}

Recall~\cite[Proposition 3.5]{Sch13pAdicHodge}
and~\cite{Sch13pAdicHodgeErratum}:

\begin{defn}\label{defn:Gpfsets-reconstructionpaper}
  $G$-$\pfsets$ is the site whose underlying category is the
  category of profinite sets with continuous $G$-action, and a
  set of continuous $G$-equivariant maps $\left\{f_{i}\colon S_{i}\to S\right\}_{i}$
  is a covering if
  \begin{itemize}
    \item[(i)] each $S_{i}\to S$ is an inverse limit
      $\varprojlim_{\mu<\lambda}S_{i\mu}$ of finite sets with $S_{i0}=S$,
      having the following property: for all $\mu<\lambda$,
      $S_{\mu}\to\varprojlim_{\mu^{\prime}<\prime}S_{\mu^{\prime}}$
      is the pullback of a surjective map of finite sets.
    \item[(ii)] Furthermore, $S=\bigcup_{i}f_{i}\left(S_{i}\right)$.
  \end{itemize}
\end{defn}

We continue to fix a Banach ring $R$.

\begin{defn}\label{defn:HomcontG}
  Given $S\in G$-$\pfsets$ and a $G$-$R$-Banach module $M$,
  \begin{equation*}
    \intHom_{\cont,G}\left(S,M\right)
    \subseteq\intHom_{\cont}\left(S,M\right)
  \end{equation*}
  is the $R$-submodule of $G$-equivariant continuous maps $S \to M$
  equipped with the induced norm.
\end{defn}

\begin{lem}
  With the notation as in Definition~\ref{defn:HomcontG},
  $\intHom_{\cont,G}\left(S,M\right)$ is an $R$-Banach module.
\end{lem}

\begin{proof}
  $\intHom_{\cont,G}\left(S,M\right)$ is the intersection of the following sets
  \begin{equation*}
    E_{g}=\left\{
      \phi\in\intHom_{\cont}\left(S,M\right)\colon
      \phi(gs)=g\phi(s) \text{ for all } s\in S
    \right\} \text{ for all } g\in G.
  \end{equation*}
  These sets are closed. Indeed, one checks from the definition
  that the limit of any sequence $\left(\phi_{i}\right)_{i}\subseteq E_{g}$
  converging in $\intHom_{\cont,G}\left(S,M\right)$ lies again in $E_{g}$.
\end{proof}

$M\mapsto\intHom_{\cont,G}\left(S,M\right)$ defines a functor
$\Ban_{R}(G)\to\Ban_{R}$.
Apply~\cite[Proposition 6.1.9]{KashiwaraSchapira2006}
to lift it to a cocontinuous functor
$\IndBan_{R}(G)\to\IndBan_{R}$. Denote it again by $\intHom_{\cont,G}\left(S,-\right)$.
By definition, this construction is also functorial in $S$.

\begin{defn}\label{defn:FMfunctor-Gpfsets-reconstructionpaper}
  Given an ind-$G$-$R$-Banach module $M$, define
  the following presheaf on $G-\pfsets$:
  \begin{equation*}
    \cal{F}_{M}\colon S \mapsto \intHom_{\cont,G}\left(S,M\right),
  \end{equation*}
\end{defn}

From now on, we assume $R=F$ to be a field complete with respect to a non-trivial
non-Archimedean valuation. In particular, we have the open mapping theorem
in our disposal.

\begin{lem}\label{lem:FMfunctor-Gpfsets-givessheaves-reconstructionpaper}
  Given an ind-$G$-$F$-Banach space $M$, $\cal{F}_{M}$
  is a sheaf of $F$-ind-Banach spaces.
\end{lem}

\begin{proof}
  Given an arbitrary covering $\left\{f_{i}\colon S_{i}\to S\right\}$ in $G-\pfsets$,
  we have to show that
  \begin{equation}\label{eq:FMfunctor-Gpfsets-givessheaves-sheafexactsequence-reconstructionpaper}
    0
    \stackrel{}{\longrightarrow}    
    \intHom_{\cont,G}\left( S , M \right)
    \stackrel{}{\longrightarrow}
    \prod_{i}\intHom_{\cont,G}\left( S_{i} , M \right)
    \stackrel{}{\longrightarrow}
    \prod_{j,l}\intHom_{\cont,G}\left( S_{j}\times_{S}S_{l} , M \right)
  \end{equation}
  is strictly exact. We may assume that the covering is finite,
  by Proposition~\ref{prop:siteqc-reducesheafcondtofincov}.
  Indeed, since $S$ is compact, and Definition~\ref{defn:Gpfsets-reconstructionpaper}(i)
  implies that the $f_{i}$ are open, see~\cite[paragraph (1)]{Sch13pAdicHodgeErratum},
  we can replace the given covering by a finite subcovering.
  Next, we assume that $M$ is a $G$-$F$-Banach space. This is allowed,
  by Corollary~\ref{cor:filteredcol-inIndBan-stronglyexact}
  and because the products in~(\ref{eq:FMfunctor-Gpfsets-givessheaves-sheafexactsequence-reconstructionpaper})
  commute with filtered colimits as they are finite. Furthermore,
  we find with~\cite[Corollary 6.1.17]{KashiwaraSchapira2006}
  that the products in~(\ref{eq:FMfunctor-Gpfsets-givessheaves-sheafexactsequence-reconstructionpaper})
  are computed in $\Ban_{F}$. In this situation, it is well-known
  that~(\ref{eq:FMfunctor-Gpfsets-givessheaves-sheafexactsequence-reconstructionpaper})
  is exact, cf. the~\cite[discussion in \S 3]{Sch13pAdicHodge}.
  It is strictly exact by the open mapping theorem.
%%%
\end{proof}

%%%%%%%%%%%%%%%%%%%%%%%%%%%%%%%%%%%%%%%%%%%%%%%%%%%%%%%%%%%%
% FUNCTIONAL ANALYSIS
%%%%%%%%%%%%%%%%%%%%%%%%%%%%%%%%%%%%%%%%%%%%%%%%%%%%%%%%%%%%

\subsection{Continuous group cohomology}

$R$ is again a fixed Banach ring.
In the following, we view any ind-$G$-$R$-Banach module as an ind-$R$-Banach module
via the canonical $\IndBan_{R}\left(G\right)\to\IndBan_{R}$.

\begin{lem}\label{lem:contgpaction-banachmodule-hasboundedimage-reconstructionpaper}
  Given a $G$-$R$-Banach module, we find that the map
  \begin{equation*}
    b\colon G \to \intHom_{R}\left(M,M\right), \, g\mapsto \left( m\mapsto gm \right)
  \end{equation*}
  is continuous. As a consequence, $\sup_{g\in G}\|b(g)\|<\infty$.
\end{lem}

\begin{proof}
  The norm on $\intHom_{R}\left(M,M\right)$ induces the compact-open topology.
  Thus we have
  \begin{equation*}
    \Hom_{\cont}\left( G \times M , M \right)
    =\Hom_{\cont}\left( G , \intHom_{R}\left(M,M\right) \right).
  \end{equation*}
  The equality matches the action map $G\times M\to M$
  with the morphism $b$, which is thus continuous.
  This implies that the following composition is continuous:
  \begin{equation*}
    G \stackrel{b}{\longrightarrow} \intHom_{R}\left(M,M\right)
    \stackrel{\|\cdot\|}{\longrightarrow} \RR_{\geq0}.
  \end{equation*}
  Its image $\left\{\|b(g)\|\colon g\in G\right\}$ is compact because $G$ is compact,
  thus it is bounded.
\end{proof}

\begin{lem}\label{lem:defn:Banachmodule-Ccontcomplex-reconstructionpaper}
  Given a $G$-$R$-Banach module $M$, we find that
  \begin{align*}
    d^{0}\colon M &\to \intHom_{\cont}\left( G , M \right) \\
    m &\mapsto g(m)-m
  \end{align*}
  is a morphism of $R$-Banach modules. For all $i\geq1$,
  \begin{align*}
    d^{i}\colon \intHom_{\cont}\left( G^{i} , M \right) &\to \intHom_{\cont}\left( G^{i+1} , M \right) \\
    \phi &\mapsto\left(\left(g_{1},\dots,g_{i+1}\right)\mapsto
      \begin{array}{l}
        g_{1}\left(\phi\left(g_{2},\dots,g_{i+1}\right)\right) \\
        +\sum_{j=1}^{i}\left(-1\right)^{j}\phi\left(g_{1},\dots,g_{j}g_{j+1},\dots,g_{i+1}\right) \\
        +\left(-1\right)^{i+1}\phi\left(g_{1},\dots,g_{i}\right)
      \end{array}
    \right)
  \end{align*}
  are morphisms of $R$-Banach modules as well.
\end{lem}

\begin{proof}
  It is classical that these maps are well-defined, see for
  example~\cite{Tate1976K2andGaloisCohomology}
  or~\cite[\href{https://stacks.math.columbia.edu/tag/0DVG}{Tag 0DVG}]{stacks-project}.
  It remains to show that they are bounded. Recall
  Lemma~\ref{lem:contgpaction-banachmodule-hasboundedimage-reconstructionpaper}
  and compute, for all $i\geq0$,
  \begin{align*}
    \|d^{i}\left(\phi\right)\|
    =\sup_{g_{1},\dots,g_{i+1}\in G}d^{i}\left(\phi\right)\left( g_{1},\dots,g_{i+1} \right)
    \leq\max\left\{\sup_{g\in G}\|b(g)\|,1\right\}\|\phi\|.
  \end{align*}
  for any $\phi\in\intHom_{\cont}\left( G^{i} , M \right)$.
  That is, each $d^{i}$ is bounded by $\left\{\sup_{g\in G}\|b(g)\|,1\right\}$.
\end{proof}

\begin{defn}\label{defn:Banachmodule-Ccontcomplex-reconstructionpaper}
  Given a $G$-$R$-Banach module $M$, $C_{\cont}^{\bullet}\left(G,M\right)$ denotes the sequence of maps
  \begin{equation}\label{eq:contgpcoh-defn}
    M \stackrel{d^{0}}{\longrightarrow}
    \intHom_{\cont}\left( G , M\right)
    \stackrel{d^{1}}{\longrightarrow}
    \intHom_{\cont}\left( G \times G, M\right)
    \stackrel{d^{2}}{\longrightarrow} \dots
  \end{equation}
  where the differentials are as in Lemma~\ref{lem:defn:Banachmodule-Ccontcomplex-reconstructionpaper}.
  It is a cochain complex of abstract $R$-modules,
  see for example~\cite{Tate1976K2andGaloisCohomology}
  or~\cite[\href{https://stacks.math.columbia.edu/tag/0DVG}{Tag 0DVG}]{stacks-project}.
  It is a cochain complex of $R$-Banach modules by
  Lemma~\ref{lem:defn:Banachmodule-Ccontcomplex-reconstructionpaper}.
\end{defn}

\begin{defn}\label{defn:indBanachmodule-Ccontcomplex-reconstructionpaper}
  Given an ind-$G$-$R$-Banach module $M=\text{``}\varinjlim_{i}\text{''}M_{i}$,
  \begin{equation}\label{eq:contgpcoh-defn}
    C_{\cont}^{\bullet}\left(G,M\right):=\text{``}\varinjlim_{i}\text{''}C_{\cont}^{\bullet}\left(G,M_{i}\right)
  \end{equation}
  is a cochain complex of $R$-ind-Banach modules.
\end{defn}

\begin{defn}\label{defn:indBan-ctsRGamma-recpaper}
  %Fix an ind-$G$-$R$-Banach module $M$.
  The \emph{continuous group cohomology of $G$ with coefficients in $M\in\IndBan_{R}(G)$} is
  \begin{equation*}
    \R\Gamma_{\cont}\left( G , M\right) \in \D\left(\IndBan_{R}\right),
  \end{equation*}
  the image of the cochain complex $C_{\cont}^{\bullet}\left(G,M\right)$
  in the derived category.
  The \emph{$i$th continuous group cohomology of $G$ with coefficients in $M$} is
  \begin{equation*}
    \Ho_{\cont}^{i}\left( G , M\right)
    :=\LHo^{i}\left( C_{\cont}^{\bullet}\left( G , M\right) \right)
    =\LHo^{i}\left( \R\Gamma_{\cont}\left( G , M\right) \right)
    \in \LH\left(\IndBan_{R}\right).
  \end{equation*}
\end{defn}

\begin{lem}\label{lem:indbanachcontcohomology-commuteswith-filteredcolimits-reconstructionpaper}
  For any ind-$G$-$R$-Banach module $M=\text{``}\varinjlim_{j}\text{''}M_{j}$
  and all $i\in\ZZ$,
  \begin{equation*}
    \Ho_{\cont}^{i}\left( G , M\right)=\varinjlim_{j}\Ho_{\cont}^{i}\left( G , M_{j}\right).
  \end{equation*}
\end{lem}

\begin{proof}
  This follows directly from~(\ref{eq:contgpcoh-defn})
  and Corollary~\ref{cor:filteredcol-inIndBan-stronglyexact}.
  \iffalse %%%more details:
  \begin{equation*}
    \Ho_{\cont}^{i}\left( G , M\right)
    =\LHo^{i}\left( C_{\cont}^{\bullet}\left( G , M\right) \right)
    \cong
      \varinjlim_{j}\LHo^{i}\left( \R\Gamma_{\cont}\left( G , M_{j}\right) \right) 
    =\varinjlim_{j}\Ho_{\cont}^{i}\left( G , M\right).
  \end{equation*}
  \fi %%% comment ends
\end{proof}

We highlight again that, in general, each $\Ho_{\cont}^{i}\left( G , M\right)$ is not itself an ind-Banach space,
but an object in the left heart. However at least when $i=0$, we have
Lemma~\ref{lem:indBanachcontgpcohzero-is-invariance} below.

\begin{defn}\label{defn:Ginv-recpaper}
  Fix an ind-$G$-$R$-Banach module $M$. Its \emph{$G$-invariance} is the $R$-ind-Banach module
  \begin{equation*}
    M^{G}:=\ker\left( d^{0} \colon M \to \Hom_{\cont}\left(G,M\right)\right),
  \end{equation*}
  where $d^{0}$ is the zeroth differential of $C_{\cont}^{\bullet}\left(G,M\right)$.
\end{defn}

\begin{lem}\label{lem:indBanachcontgpcohzero-is-invariance}
  For any ind-$G$-$R$-Banach module $M$, we have the canonical isomorphism
  \begin{equation*}
    \Ho_{\cont}^{0}\left( G , M \right)
    \cong\I\left(M^{G}\right).
  \end{equation*}
\end{lem}

\begin{proof}
  %\begin{equation*}
    $\Ho^{0}\left( G , M \right)
    =\LHo^{0}\left(C_{\cont}^{\bullet}\left(G,M\right)\right)
    \stackrel{\text{\ref{lem:LH-vs-H-reconstructionpaper}}}{\cong}\Ho^{0}\left(\I\left(C_{\cont}^{\bullet}\left(G,M\right)\right)\right)
    =\ker\I\left(d^{0}\right)
    \cong\I\left(\ker d^{0}\right)
    =\I\left(M^{G}\right)$,
  %\end{equation*}
  where the last isomorphism of this computation follows from~\cite[Proposition 1.2.27]{Sch99}.
\end{proof}

\begin{lem}\label{lem:indBanachinvariance-cocont-leftexact}
  $M\mapsto M^{G}$ defines a cocountinuous left exact functor $\IndBan_{R}\left(G\right)\to\IndBan_{R}$.
\end{lem}

\begin{proof}
  The construction is clearly functorial and the cocontinuity follows from
  Corollary~\ref{cor:filteredcol-inIndBan-stronglyexact}.
  We have to check that $M\mapsto M^{G}$ preserves kernels of strict morphisms
  of ind-$G$-$R$-Banach modules. By~\cite[Proposition 2.10(3)]{BBB16}, it suffices to check that it preserves
  kernels of strict morphisms $\phi\colon M\to N$ of $G$-$R$-Banach modules. To do this,
  consider the commutative diagram
  \begin{equation*}
  \begin{tikzcd}
    M^{G} \arrow[hook]{r}{i} & M \\
    \ker\left(\phi^{G}\right) \arrow[hook]{u}{j} \arrow{r} & \left(\ker\phi\right)^{G} \arrow[hook]{u}.
  \end{tikzcd}
  \end{equation*}
  We have to check that the morphism at the bottom
  is an isomorphism. It is clearly bijective, thus it remains to show that it is strict.
  But both $i$ and $j$ are strict monomorphisms, thus their composition
  $i\circ j$ is a strict monomorphism by~\cite[Proposition 1.1.7]{Sch99}.
  Now apply \emph{loc. cit.} Proposition 1.1.8.
\end{proof}

The following discussion is almost identical to~\cite[\S 3.3]{BSSW2024_rationalizationoftheKnlocalsphere},
and we refer the reader to \emph{loc. cit.} for details.
Consider the solid abelian group $\underline{G}$.
We specialise to the case $R=k$, where
$k$ is as in \S\ref{subsec:conventions-reconstructionpaper}.
See \S\ref{subsec:condensedmath-recpaper}
for the definition of the category of solid $k$-vector spaces.
A $\underline{G}$-action on a solid $k$-vector space $V$ is a
morphism $\underline{G}\times V\to V$ satisfying the usual axioms.

\begin{lem}
  Let $V$ be an ind-$G$-$k$-Banach module.
  Then $\underline{V}$ is a solid $k$-vector space with $\underline{G}$-action.
\end{lem}

\begin{proof}
  $\underline{V}$ carries a $\underline{G}$-action
  by~\cite[Lemma 3.3.1]{BSSW2024_rationalizationoftheKnlocalsphere}.
  The result follows as $\underline{V}$ is also a $\underline{k}$-module.
\end{proof}

Let $\Vect_{k}^{\solid}\left(G\right)$ be the category of solid $k$-vector spaces
with $\underline{G}$-action. This category is abelian.

\begin{notation}\label{notation:solidctscoh-recpaper}
  The \emph{fixed points functor} $\Vect_{k}^{\solid}\left(G\right)\to\Vect_{k}^{\solid}$,
  denoted by $W\mapsto W^{G}$, is the right adjoint of the trivial action functor
  $\Vect_{k}^{\solid}\mapsto\Vect_{k}^{\solid}\left(G\right)$. Let
  $W^{\bullet}\to\R\Gamma_{\cont}\left(G,W^{\bullet}\right)$ be its derived functor
  $\D\left(\Vect_{k}^{\solid}\left(G\right)\right)\to\D\left(\Vect_{k}^{\solid}\right)$
  and let $\Ho_{\cont}^{i}\left(G,W^{\bullet}\right)=\R^{i}\Gamma_{\cont}\left(G,W^{\bullet}\right)$
  for all $i\in\ZZ$.
\end{notation}

\begin{lem}\label{lem:contgpcoh-indban-vs-solid-reconstructionpaper}
  Given an ind-$G$-k-ind-Banach space $V$, we have an isomorphism
  \begin{equation*}
    \R\Gamma_{\cont}\left(G,\underline{V}\right)\cong\underline{\R\Gamma_{\cont}\left(G,V\right)}
  \end{equation*}
  in $\D\left(\Vect_{k}^{\solid}\right)$.
  Consequently,
  $\Ho_{\cont}^{i}\left(G,\underline{V}\right)\cong\underline{\Ho^{i}_{\cont}\left(G,V\right)}$
  for all $i\in\ZZ$.
\end{lem}

\begin{proof}
  By the~\cite[proof of Lemma 3.3.2]{BSSW2024_rationalizationoftheKnlocalsphere},
  we find that $\R\Gamma_{\cont}\left(G,\underline{V}\right)$ is isomorphic to a complex
  \begin{equation*}
  0\to
  \underline{V} \to
  \intHom\left(\underline{G},\underline{V}\right) \to\dots 
  \intHom\left(\underline{G}^{2},\underline{V}\right) \to\dots.
  \end{equation*}
  By Lemma~\ref{lem:intHomunderlineSunderlineV-is-underlineHomcontSV-indBansetting-reconstructionpaper},
  this is the image of the complex~(\ref{eq:contgpcoh-defn})
  with respect to the functor $V^{\bullet}\mapsto\underline{V^{\bullet}}$.
  The second sentence follows from Lemma~\ref{lem:colimitsVectksolid-exact}.
\end{proof}

\begin{lem}\label{lem:formalcolimits-and-solid-group-cohomology-reconstructionpaper}
  Consider an ind-$G$-$k$-Banach space $V=\text{``}\varinjlim_{i}\text{''}V_{i}$. Then, for all $j\in\ZZ$,
  \begin{equation*}
    \Ho_{\cont}^{j}\left( G , \underline{V} \right)
    =\varinjlim_{i}\Ho_{\cont}^{j}\left( G , \underline{V_{i}} \right).
  \end{equation*}
\end{lem}

\begin{proof}
  Compute
  %\begin{equation*}
  $\Ho_{\cont}^{j}\left( G , \underline{V} \right)
    \stackrel{\text{\ref{lem:contgpcoh-indban-vs-solid-reconstructionpaper}}}{\cong}\underline{\Ho_{\cont}^{j}\left( G , V \right)}
    \stackrel{\text{\ref{lem:indbanachcontcohomology-commuteswith-filteredcolimits-reconstructionpaper}}}{\cong}
    \underline{\varinjlim_{i}\Ho_{\cont}^{j}\left( G , V_{i} \right)}
    \stackrel{\text{\ref{lem:contgpcoh-indban-vs-solid-reconstructionpaper}}}{\cong}\varinjlim_{i}\Ho_{\cont}^{j}\left( G , \underline{V_{i}} \right)$.
   %\end{equation*}
\end{proof}

The advantage of the solid formalism is the following result.

\begin{lem}\label{lem:solidHochschild-Serre-reconstructiontheorem}
  Given a solid $k$-vector space $W$ with continous $G$-action
  and a closed normal subgroup
  $H\subseteq G$, we have the Hochschild-Serre spectral sequence
  \begin{equation*}
    \Ho_{\cont}^{i}\left(G/H , \Ho_{\cont}^{j}\left( H , W \right) \right)
    \implies \Ho_{\cont}^{i+j}\left(W\right).
  \end{equation*}
\end{lem}

\begin{proof}
  It arises as a Grothendieck spectral sequence
  since $W^{G}=\left(W^{H}\right)^{G/H}$.
\end{proof}

We discuss the case $G=\ZZ_{p}^{d}$
which we have learned from~\cite[\S 7]{BhattMorrowScholze2018}
and~\cite[Proposition B.4]{Bosco21}.

\begin{defn}
  Let $W$ denote a solid $k$-vector space with commuting endomorphisms
  $f_{i}\colon V\to V$ for $i=1,\dots,d$. $\Kos_{V}\left( f_{1},\dots,f_{d}\right)$
  is the \emph{Koszul complex}, that is
  \begin{equation*}
    M
    \stackrel{\left(f_{1},\dots,f_{d}\right)}{\longrightarrow}
    \bigoplus_{1\leq i\leq d} M \longrightarrow
    \bigoplus_{1\leq i_{1}<i_{2}\leq d} M \longrightarrow
    \bigoplus_{1\leq i_{1}<\dots i_{l}\leq d} M \longrightarrow
    \dots,
  \end{equation*}
  where the differential from $M$ in spot
  $i_{1}<\dots< i_{l}$ to $M$ in spot $j_{1}<\dots< j_{l+1}$
  is non-zero only if
  $\left\{i_{1},\dots,i_{l}\right\}\subseteq\left\{j_{1},\dots,j_{l+1}\right\}$.
  In this case, it is given by $(-1)^{s-1}f_{j_{s}}$, where $s\in\left\{1,\dots,l+1\right\}$
  is the unique integer such that $j_{s}\not\in\left\{i_{1},\dots,i_{l}\right\}$.
\end{defn}

\begin{lem}\label{lem:solidcontinuous-cohomology-over-ZZpd-via-Koszul}
  Let $d\in\NN$. Given a solid $k$-vector space $W$ with
  $\underline{\ZZ_{p}}^{d}$-action, we have
  \begin{equation}\label{eq:iso--lem:solidcontinuous-cohomology-over-ZZpd-via-Koszul}
    \R\Gamma_{\cont}\left( \ZZ_{p}^{d} , W \right)
    \cong\Kos_{W}\left( \gamma_{1}-1, \dots, \gamma_{d}-1 \right).
  \end{equation}
  in the derived category of $\Vect_{k}^{\solid}$.
  Here, $\gamma_{1},\dots,\gamma_{d}$ denotes a basis of $\ZZ_{p}^{d}$.
\end{lem}

\begin{proof}
  See the~\cite[proof of Lemma 3.3.2]{BSSW2024_rationalizationoftheKnlocalsphere}
  for the construction of the solid Iwasawa algebra $\ZZ[G]^{\blacksquare}$,
  as well as a proof of the fact
  \begin{equation*}
    \R\Gamma_{\cont}\left(G,W\right)
    \cong
    \R\Hom_{\ZZ[G]^{\blacksquare}}\left(\ZZ , W \right).
  \end{equation*}
  By~\cite[Lemma B.5(i)]{Bosco21}
  \footnote{\emph{Loc. cit.} denotes $\ZZ[G]^{\blacksquare}$ by $\ZZ\llbracket G \rrbracket$},
  $\ZZ[G]^{\blacksquare}$ is the solidification of $\ZZ[G]$.  This implies
  \begin{equation*}
    \R\Gamma_{\cont}\left(G,W\right)
    \cong
    \R\Hom_{\ZZ[G]}\left(\ZZ , W \right).
  \end{equation*}
  Applying~\cite[Lemma B.5(i)]{Bosco21}
  gives~(\ref{eq:iso--lem:solidcontinuous-cohomology-over-ZZpd-via-Koszul}) in
  the desired category of solid abelian groups. The
  proof \emph{loc. cit.} implies that it is an isomorphism
  in the derived category of solid $k$-vector spaces as well.
\end{proof}

%%%%%%%%%%%%%%%%%%%%%%%%%%%%%%%%%%%%%%%%%%%%%%%%%%%%%%%%%%%%
% Some commutative algebra
%%%%%%%%%%%%%%%%%%%%%%%%%%%%%%%%%%%%%%%%%%%%%%%%%%%%%%%%%%%%

\section{Completeness of rings of formal power series}
\label{sec:formalpowerseries-cplt}

%Here in \S\ref{sec:formalpowerseries-cplt},
We prove the following
elementary result which we were unable to find in the literature.

\begin{prop}\label{prop:Scomplete-Spowerseriescomplete-ifKoszulregular}
  Fix a regular sequence $s_{1},\dots,s_{n}$
  in a commutative ring $S$. We consider the ideal $I:=\left(s_{1},\dots,s_{n}\right)$
  and pick an arbitrary $d\in\NN_{\geq 1}$.
  If $S$ is $I$-adically complete, then $S\llbracket X_{1},\dots,X_{d}\rrbracket$
  is $\left(X_{1},\dots,X_{d},s_{1},\dots,s_{n}\right)$-adically complete.
\end{prop}

\begin{lem}\label{lem:S-Icompl-iff-sicomplforalli}
  Let $S$ denote a commutative ring, fix a regular sequence $s_{1},\dots,s_{n}\in S$,
  and consider the ideal $I:=\left(s_{1},\dots,s_{n}\right)$. The following are equivalent.
  \begin{itemize}
    \item[(i)] $S$ is $I$-adically complete.
    \item[(ii)] $S$ is, for every $i=1,\dots,n$, $s_{i}$-adically complete.
  \end{itemize}
\end{lem}

\begin{proof}
  The direction (i) $\implies$ (ii) is proven in~\cite[\href{https://stacks.math.columbia.edu/tag/090T}{Tag 090T}]{stacks-project}.
  Now suppose (ii). \cite[\href{https://stacks.math.columbia.edu/tag/091T}{Tag 091T}]{stacks-project}
  implies that $S$ is, for every $i=1,\dots,n$, derived $s_{i}$-adically complete.
  It is derived $I$-adically complete
  by~\cite[\href{https://stacks.math.columbia.edu/tag/091Q}{Tag 091Q}]{stacks-project}.
  Now use the notation
  from~\cite[\href{https://stacks.math.columbia.edu/tag/0BKC}{Tag 0BKC}]{stacks-project},
  that is consider the Koszul complexes $K_{j}^{\bullet}=K_{\bullet}\left(S;s_{1}^{j},\dots,s_{n}^{j}\right)$
  for every $j\in\NN_{\geq1}$.
  Since $s_{1},\dots,s_{n}$ is a regular sequence, the sequences $s_{1}^{j},\dots,s_{n}^{j}$
  are regular. \cite[\href{https://stacks.math.columbia.edu/tag/062F}{Tag 062F}]{stacks-project}
  gives that $K_{j}^{\bullet}$ is quasi-isomorphic to $S/\left(s_{1}^{j},\dots,s_{n}^{j}\right)$,
  viewed as a complex concentrated in degree zero.
  \cite[\href{https://stacks.math.columbia.edu/tag/091Z}{Tag 091Z}]{stacks-project}
  implies that the canonical
  \begin{equation*}
  S\stackrel{\cong}{\longrightarrow}
  R\limit\left( S\otimes_{S}^{\rL} K_{j}^{\bullet}\right)\cong
  R\limit S/\left(s_{1}^{j},\dots,s_{n}^{j}\right)
  \end{equation*}
  is an isomorphism.
  Here, the operator $\limit$ denotes the homotopy limit
  along the maps $\dots\to K_{2}^{\bullet}\to K_{1}^{\bullet}$, see the discussion
  in~\cite[\href{https://stacks.math.columbia.edu/tag/0BKC}{Tag 0BKC}]{stacks-project}.
  We apply~\cite[\href{https://stacks.math.columbia.edu/tag/0941}{Tag 0941}]{stacks-project}
  \footnote{In the notation of~\cite[\href{https://stacks.math.columbia.edu/tag/0941}{Tag 0941}]{stacks-project},
  we choose $C$ to be the site associated
  to the topological space with one point.} to see that 
  the right-hand side coincides with
  $R\varprojlim_{t\in\NN} S/\left(s_{1}^{j},\dots,s_{n}^{j}\right)$,
  where $R\varprojlim_{j\in\NN}$ denotes the derived functor
  of the inverse limit. In particular, taking zeroth cohomology,
  $S\cong \varprojlim_{j\in\NN}S/\left(s_{1}^{j},\dots,s_{n}^{j}\right)$.
  It suffices to show that the canonical morphism
  $\varprojlim_{j\in\NN}S/\left(s_{1}^{j},\dots,s_{n}^{j}\right)\to\varprojlim_{j\in\NN}S/I^{j}$
  is an isomorphism. But for every $j\geq n$,
  $I^{j}\subseteq \left(s_{1}^{l},\dots,s_{n}^{l}\right)$
  for $l=\lfloor j/n \rfloor$, proving the claim.
\end{proof}

\begin{lem}\label{lem:S-sadiccomplete-then-SllbrXrrbr-sadiccomplete}
  Let $S$ denote a commutative ring containing an element $s\in S$ such
  that it is $s$-adically complete.
  Then $S\llbracket X_{1},\dots,X_{d}\rrbracket$ is $s$-adically complete as well.
\end{lem}

\begin{proof}
  We have to show that the following morphism is an isomorphism:
  \begin{equation*}
    \varphi\colon S\llbracket X_{1},\dots,X_{d}\rrbracket \maps
    \varprojlim_{j}S\llbracket X_{1},\dots,X_{d}\rrbracket/s^{j}.
  \end{equation*}
  Let $f\in\ker\varphi$. That is for every $j$ exists 
  an $f_{j}\in S\llbracket X_{1},\dots,X_{d}\rrbracket$ such
  that $f=s^{j}f_{j}$. Writing $f=\sum_{\alpha\in\NN^{d}}f_{\alpha}X^{\alpha}$
  and $f_{j}=\sum_{\alpha\in\NN^{d}}f_{j\alpha}X^{\alpha}$ for all $j$,
  this means that $f_{\alpha}=s^{j}f_{j\alpha}$ for all $\alpha$ and for all $j$.
  Since $S$ is $s$-adically separated this implies $f_{\alpha}=0$
  for all $\alpha$. This shows $f=0$.
  To prove surjectivity, pick
  $f_{j}\in S\llbracket X_{1},\dots,X_{d}\rrbracket$ for all $j$ such
  that
  \begin{equation*}
    \left(\overline{f_{j}}\right)_{j}\in
      \varprojlim_{j}S\llbracket X_{1},\dots,X_{d}\rrbracket/s^{j},
  \end{equation*}
  where $\overline{f_{j}}$ denotes the image of $f_{j}$ modulo $s^{j}$.
  This means that $f_{l}-f_{j}\in (s^{j})$ for every $l\geq j$.
  Writing $f_{j}=\sum_{\alpha\in\NN^{d}}f_{j\alpha}X^{\alpha}$ for all $j$, this implies
  that $f_{l\alpha}-f_{t\alpha}\in (s^{j})$ for every $l\geq j$. Therefore
  \begin{equation*}
    \left(\overline{f_{j\alpha}}\right)\in\varprojlim_{j} S/s^{j}.
  \end{equation*}
  Because $S$ is $s$-adically complete, we conclude that there exists
  an $f_{\alpha}\in S$ for every $\alpha$ such that
  $f_{\alpha}\equiv f_{j\alpha}\mod s^{j}$. Set $f=\sum_{\alpha\in\NN^{d}}f_{\alpha}X^{\alpha}$.
  We claim that $\varphi(f)=\left(\overline{f_{j}}\right)_{j}$.
  Indeed, we have
  \begin{equation*}
    f-f_{j}=\sum_{\alpha\in\NN^{d}}(f_{\alpha}-f_{j\alpha})X^{\alpha}\equiv 0 \mod s^{j}
  \end{equation*}
  for every $j$.
  This proves Lemma~\ref{lem:S-sadiccomplete-then-SllbrXrrbr-sadiccomplete}.
\end{proof}

\begin{proof}[Proof of Proposition~\ref{prop:Scomplete-Spowerseriescomplete-ifKoszulregular}]
  $S\llbracket X_{1},\dots,X_{d}\rrbracket\cong S\llbracket X_{1},\dots,X_{i-1},X_{i+1},\dots X_{d}\rrbracket\llbracket X_{i}\rrbracket$
  is $X_{i}$-adically complete for every $i=1,\dots,d$.
  By Lemma~\ref{lem:S-Icompl-iff-sicomplforalli}, $S$ is
  $s_{i}$-adically complete for every $i=1,\dots,n$, 
  therefore $S\llbracket X_{1},\dots,X_{d}\rrbracket$
  is $s_{i}$-adically complete for every $i=1,\dots,n$,
  see Lemma~\ref{lem:S-sadiccomplete-then-SllbrXrrbr-sadiccomplete}.
  But the sequence $X_{1},\dots,X_{d},s_{1},\dots,s_{n}$ is regular, thus
  we can apply Lemma~\ref{lem:S-Icompl-iff-sicomplforalli}
  to finish the proof of Proposition~\ref{prop:Scomplete-Spowerseriescomplete-ifKoszulregular}.
\end{proof}

%%%%%%%%%%%%%%%%%%%%%%%%%%%%%%%%%%%%%%%%%%%%%%%%%%%%%%%%%%%%
%%%%%%%%%%%%%%%%%%%%%%%%%%%%%%%%%%%%%%%%%%%%%%%%%%%%%%%%%%%%
%%%%%%%%%%%%%%%%%%%%%%%%%%%%%%%%%%%%%%%%%%%%%%%%%%%%%%%%%%%%
%%%%%%%%%%%%%%%%%%%%%%%%%%%%%%%%%%%%%%%%%%%%%%%%%%%%%%%%%%%%
% The sheaf $\OB_{\dR,X}^{\dag}$
%%%%%%%%%%%%%%%%%%%%%%%%%%%%%%%%%%%%%%%%%%%%%%%%%%%%%%%%%%%%
%%%%%%%%%%%%%%%%%%%%%%%%%%%%%%%%%%%%%%%%%%%%%%%%%%%%%%%%%%%%
%%%%%%%%%%%%%%%%%%%%%%%%%%%%%%%%%%%%%%%%%%%%%%%%%%%%%%%%%%%%
%%%%%%%%%%%%%%%%%%%%%%%%%%%%%%%%%%%%%%%%%%%%%%%%%%%%%%%%%%%%

\chapter{Period rings and period sheaves}
\label{ch:period-sheaves}

We introduce overconvergent period sheaves. Our constructions are
inspired by~\cite[\S 6]{Sch13pAdicHodge}.

%%%%%%%%%%%%%%%%%%%%%%%%%%%%%%%%%%%%%%%%%%%%%%%%%%%%%%%%%%%%
%%%%%%%%%%%%%%%%%%%%%%%%%%%%%%%%%%%%%%%%%%%%%%%%%%%%%%%%%%%%
%%%%%%%%%%%%%%%%%%%%%%%%%%%%%%%%%%%%%%%%%%%%%%%%%%%%%%%%%%%%
% Introduction
%%%%%%%%%%%%%%%%%%%%%%%%%%%%%%%%%%%%%%%%%%%%%%%%%%%%%%%%%%%%
%%%%%%%%%%%%%%%%%%%%%%%%%%%%%%%%%%%%%%%%%%%%%%%%%%%%%%%%%%%%
%%%%%%%%%%%%%%%%%%%%%%%%%%%%%%%%%%%%%%%%%%%%%%%%%%%%%%%%%%%%

\section{The pro-étale site}
\label{sec:proetalesite-recpaper}

\subsection{Definitions}

We recall the pro-étale site associated
to a smooth locally Noetherian adic space $X$ over $\Spa(k,k^{\circ})$
from~\cite{Sch13pAdicHodge}. Regarding the theory of adic spaces,
we follow the notation given in~\cite[Lecture 2 and 3]{ScholzeWeinstein2020}.
See the~\cite[beginning of \S 3]{Sch13pAdicHodge} for the
definition of \emph{locally Noetherian}.
$\Pro\left(X_{\et}\right)$ is the pro-completion of
the category $X_{\et}$ of adic spaces which are étale over $X$.
%Its objects are functors
%$I\maps X_{\et},i\mapsto U_{i}$ where $I$ is a small
%cofiltered category, which we denote by
%$U=\text{"}\varprojlim_{i\in I}\text{"}U_{i}$.
The underlying topological space of
$\text{``}\varprojlim\text{"}_{i\in I}U_{i}\in\Pro\left(X_{\et}\right)$ is
$\varprojlim_{i\in I}|U_{i}|$, where $|U_{i}|$ is the
underlying topological space of $U_{i}$.
$U\in\Pro\left(X_{\et}\right)$ is \emph{pro-étale over $X$}
if and only if it is isomorphic
to an object $\text{``}\varprojlim\text{"}_{i\in I}U_{i}$
such that all transition maps $U_{j}\maps U_{i}$ are finite
étale and surjective.

The \emph{pro-étale site $X_{\proet}$ of $X$} is the full
subcategory of $\Pro\left(X_{\et}\right)$ consisting of objects
which are pro-étale over $X$. A collection of maps
$\{f_{i}\colon U_{i}\maps U\}_{i}$ in $X_{\proet}$ is
a covering if and only if the collection
$\{|U_{i}|\maps|U|\}_{i}$ is a pointwise covering of the topological
space $|U|$, and a second set-theoretic condition is satisfied,
see~\cite{Sch13pAdicHodgeErratum}. %One can write $U_{i}\maps U$
%as an inverse limit $U_{i}=\varprojlim_{\mu<\lambda}U_{\mu}$
%of $U_{\mu}\in X_{\proet}$ along some ordinal $\lambda$
%such that $U_{0}\maps U$ is étale (i.e. the pullback of
%a map in $Y_{\et}$) and for all $\mu>0$ the map
%\begin{equation*}
%  U_{\mu} \maps \varprojlim_{\mu^{\prime}<\mu}U_{\mu^{\prime}}
%\end{equation*}
%is finite étale and surjective, that is the pullback of a finite
%étale and surjective map in $X_{\et}$.
In particular, there is a canonical morphism of sites
$\nu\colon X_{\proet} \maps X_{\et}$.

\begin{examplenotation}\label{examplenotation:profinitesets-in-Xproet}
  For $U\in X_{\proet}$ and a profinite set $S$, which is the
  cofiltered limit of finite sets $S_{j}$,
  \begin{equation*}
    U\times S
    := \text{``}\varprojlim_{j}\text{"} U \times S_{j}
    = \text{``}\varprojlim_{j}\text{"} \coprod_{S_{j}} U \in X_{\proet}.
  \end{equation*}
  This is indeed independent of the presentation $S=\varprojlim_{j}S_{j}$.
\end{examplenotation}

Recall~\cite[Definition 4.3]{Sch13pAdicHodge}.
$K$ is a perfectoid field of characteristic zero, 
cf.~\cite[Definition 3.1]{Sch12perfectoid}, containing
a ring of integral elements $K^{+}\subseteq K$,
cf.~\cite[Definition 2.2.12]{ScholzeWeinstein2020}. The
definition of the pro-étale site still makes sense for a given
locally Noetherian adic space $Y$ over $\Spa\left(K,K^{+}\right)$.
$U\in Y_{\proet}$ is
\emph{affinoid perfectoid} if it is isomorphic to
$\text{``}\varprojlim\text{"}_{i\in I}U_{i}$ with $U_{i}=\Spa(R_{i},R_{i}^{+})$
such that, denoting by $R^{+}$ the $p$-adic completion of
$\varinjlim_{i\in I}R_{i}^{+}$ and $R=R^{+}[1/p]$, the pair
$(R,R^{+})$ is a perfectoid affinoid $(K,K^{+})$-algebra,
cf.~\cite[Definition 5.1(i)]{Sch12perfectoid}.

\begin{notation}\label{notation:widehatU}
  With the notation from the previous paragraph,
  $\widehat{U}:=\Spa\left(R,R^{+}\right)$.
\end{notation}

Consider again $X$, the fixed smooth locally Noetherian adic space over
$\Spa\left(k,k^{\circ}\right)$.

\begin{notation}\label{notatio:base-change-in-proet}
  Let $K$ denote the completion of an algebraic extension of $k$ which is perfectoid,
  together with an open bounded valuation subring $K^{+}\subseteq K$.
  We consider $\Spa\left(K,K^{+}\right)$ as an object
  $V=\text{``}\varprojlim\text{"}_{i\in I}V_{i}\in\Spa\left( k , k^{\circ} \right)_{\proet}$
  with $V_{i}=\Spa(K_{i}, K_{i}^{+})$,
  such that $K^{+}$ is the $\pi$-adic completion of
  $\varinjlim_{i\in I}K_{i}^{+}$. By abuse of notation, define the base-change
  \begin{equation*}
    U_{K}=U_{\left(K,K^{+}\right)}:=U\times_{\Spa\left(k,k^{\circ}\right)}V\in X_{\proet}.
  \end{equation*}
\end{notation}

\begin{defn}[{\cite[Definition 4.3]{Sch13pAdicHodge}}]
  $U\in X_{\proet}$ is \emph{affinoid perfectoid}
  if the structural map $U \to X$ factors through a pro-étale map $X_{K}\to X$ such that
  $U\in X_{\proet}/X_{K}\simeq X_{K,\proet}$ is affinoid perfectoid;
  see~\cite[Proposition 3.15]{Sch13pAdicHodge} for the canonical identification
  of the sites. Here, $K$ is the completion of an algebraic extension
  of $k$ which is perfectoid.
\end{defn}

\begin{examplenotation}\label{examplenotation:torus-variable-index}
  Given an indexing set $\left\{n_{1}<\dots <n_{d}\right\}\subseteq\NN_{\geq1}$,
  $T_{n_{1}},\dots,T_{n_{s}}$ denote formal variables. Define the following
  tori for every $e\in\NN$:
  \begin{equation*}
    \TT_{e}^{\left\{n_{1},\dots,n_{d}\right\}}:=\Spa\left(k\left\<T_{n_{1}}^{\pm1/p^{e}},\dots,T_{n_{d}}^{\pm1/p^{e}}\right\>,
    k^{\circ}\left\<T_{n_{1}}^{\pm1/p^{e}},\dots,T_{n_{d}}^{\pm1/p^{e}}\right\>\right).
  \end{equation*}
  In particular, $\TT_{e}^{\emptyset}=\Spa\left(k,k^{\circ}\right)$
  and we write $\TT^{d}=\TT^{\{1,\dots,d\}}:=\TT_{0}^{\{1,\dots,d\}}$.
  Then
  \begin{equation*}
    \widetilde{\TT}^{\left\{n_{1},\dots,n_{d}\right\}}
    :=\text{``}\varprojlim_{e\in\NN}\text{"}\widetilde{\TT}^{\left\{n_{1},\dots,n_{d}\right\}}
    \in\TT_{\proet}^{\left\{n_{1},\dots,n_{d}\right\}}
  \end{equation*}  
  is not affinoid perfectoid, but its base-change to
  the completion of an algebraic extension of $k$ which is perfectoid is.
  Finally, write $\widetilde{\TT}^{d}=\widetilde{\TT}^{\{1,\dots,d\}}\in\TT_{\proet}^{d}$.
\end{examplenotation}

The following Lemma~\ref{lem:affperfd-basis} appears implicitly
in~\cite{Sch13pAdicHodge}. %For the convenience of the reader,
%we chose to provide the full proof here.

\begin{lem}\label{lem:affperfd-basis}
  The affinoid perfectoids form a basis for the site $X_{\proet}$.
\end{lem}

\begin{proof}
  Every finite field
  extension $k\subseteq k^{\prime}\subseteq C$ 
  induces a finite étale map $\Sp k^{\prime} \to \Sp k$ between
  the associated rigid-analytic spaces. Thus it induces a finite étale map
  between the associated adic spaces
  $\Spa \left(k^{\prime},k^{\prime\circ}\right) \to \Sp \left(k , k^{\circ}\right)$,
  cf.~\cite[Proposition 1.7.11]{Huber96Etale}. In particular,
  $V=\text{``}\varprojlim\text{"}_{[k^{\prime}:k]<\infty}\left(k^{\prime},k^{\prime\circ}\right)\in\Spa\left( k , k^{\circ} \right)_{\proet}$.
  Set $C^{+}:=\widehat{\varinjlim_{[k^{\prime}:k]<\infty}k^{\prime\circ}}\subseteq C$.
  $C$ is perfectoid. By the~\cite[remark proceeding Proposition 3.9]{HuberContVal1993},
  $C^{+}\subseteq C$ is a ring of integral elements.
  By construction, $X_{C}:=X_{\left(C,C^{+}\right)}\to X$ is a covering.
  Thus every covering of $X_{C}$ by affinoid perfectoids will give a covering
  of $X$ by affinoid perfectoids. Now apply~\cite[Corollary 4.7]{Sch13pAdicHodge}.
\end{proof}

%Let $K$ denote a \emph{perfectoid field}, as introduced in~\cite{Sch12perfectoid}.
%We assume that it has characteristic zero and fix an
%open and bounded valuation subring $K^{+}\subseteq K$.
%Consider a locally Noetherian adic space $Y$ over $\Spa(K,K^{+})$.
%An object $U\in Y_{\proet}$ is
%\emph{affinoid perfectoid}~\cite[Definition 4.3]{Sch13pAdicHodge}
%if it is isomorphic in $Y_{\proet}$ to an object
%$``\varprojlim"_{i\in I}U_{i}$ with $U_{i}=\Spa(R_{i},R_{i}^{+})$
%such that, denoting by $R^{+}$ the $p$-adic completion of
%$\varinjlim_{i\in I}R_{i}^{+}$ and $R=R^{+}[1/p]$, the pair
%$(R,R^{+})$ is a perfectoid affinoid $(K,K^{+})$-algebra.
%We write $\widehat{U}=\Spa(R,R^{+})$.

%\begin{lem}[{\cite[Corollary 4.7]{Sch13pAdicHodge}}]\label{lem:affperfd-basis}
%  When $Y$ is smooth over $\Spa(K,K^{+})$, then the
%  affinoid perfectoids $U\in Y_{\proet}$ form a basis
%  for the topology.
%\end{lem}

Consider the full subcategory
$X_{\proet,\affperfd}\subseteq X_{\proet}$
of affinoid perfectoids.

\begin{lem}\label{lem:affperfd-closed-fibreprod}
  $X_{\proet,\affperfd}$ is closed under fibre products.
\end{lem}

\begin{proof}
  Fix a diagram $U_{1}\to U_{2}\leftarrow U_{3}$ of affinoid perfectoids
  in $X_{\proet}$ that lives over a fixed perfectoid affinoid field
  $\left( K , K^{+} \right)$ of characteristic zero. Recall Notation~\ref{notation:widehatU}.
  The fibre product $\widehat{U_{1}}\times_{\widehat{U_{2}}}\widehat{U_{3}}$
  exists in the category of adic spaces and is again a perfectoid space,
  see~\cite[Proposition 6.18]{Sch12perfectoid}. On the other hand, $U_{1}\times_{U_{2}}U_{3}$ is perfectoid
  by~\cite[Lemma 4.6]{Sch13pAdicHodge}, and the universal property of the fibre product yields
  a map
  \begin{equation}\label{eq:affperfd-closed-fibreprod-1}
    \widehat{U_{1}\times_{U_{2}}U_{3}}\to\widehat{U_{1}}\times_{\widehat{U_{2}}}\widehat{U_{3}}.
  \end{equation}
  We claim that it is an isomorphism of adic spaces.
  
  Write $\widehat{U_{1}}=\Spa\left(R_{1},R_{1}^{+}\right)$,
  $\widehat{U_{2}}=\Spa\left(R_{2},R_{2}^{+}\right)$, $\widehat{U_{3}}=\Spa\left(R_{3},R_{3}^{+}\right)$,
  and $\widehat{U_{1}}\times_{\widehat{U_{2}}}\widehat{U_{3}}=\Spa(S,S^{+})$.
  Here $S=R_{1}\widehat{\otimes}_{R_{2}}R_{3}$ and $S^{+}$ is the completion of the
  integral closure of the image of $R_{1}^{+}\otimes_{R_{2}^{+}}R_{3}^{+}$ in $S$.
  On the other hand, write the diagram $U_{1}\to U_{2}\leftarrow U_{3}$ above, after suitable
  reindexing as explained in~\cite[page 54, remark 1.133]{Meyer2007Localanalyticcyclichomology},
  as an inverse limit $U_{1i}\to U_{2i} \leftarrow U_{3i}$
  of affinoids over a small cofiltered category $I$. Fix notation $U_{1i}=\Spa\left(R_{1i},R_{1i}^{+}\right)$,
  $U_{2i}=\Spa\left(R_{2i},R_{2i}^{+}\right)$, and $U_{3i}=\Spa\left(R_{3i},R_{3i}^{+}\right)$.
  Assume that all these Huber pairs are complete.
  Then each $U_{1i}\times_{U_{2i}}U_{3i}$ exists in the category of adic spaces,
  see~\cite[Proposition 1.2.2]{Huber96Etale}, and
  $U_{1}\times_{U_{2}}U_{3}=\text{``}\varprojlim\text{"}_{i}U_{1i}\times_{U_{2i}}U_{3i}$.
  Indeed, \emph{loc. cit.} says that $U_{1i}\times_{U_{2i}}U_{3i}=\Spa(S_{i},S_{i}^{+})$
  where $S_{i}=R_{1i}\otimes_{R_{2i}}R_{3i}$, $S_{i}^{+}$ is the integral closure
  of $R_{1i}^{+}\otimes_{R_{2i}^{+}}R_{3i}^{+}$ in $S_{i}$, and both $S_{i}$ and $S_{i}^{+}$
  carry suitabale topologies. Now write $S_{\infty}:=\varinjlim_{i}S_{i}$
  and $S_{\infty}^{+}:=\varinjlim_{i}S_{i}^{+}$. We turn $S_{\infty}$ into a Huber
  ring by declaring $S_{\infty}^{+}$ to be a ring of definition, which we equip with the
  $\pi$-adic topology.
  This gives a Huber pair $\left(S_{\infty},S_{\infty}^{+}\right)$.
  To check this, we have to verify three conditions:
  \begin{itemize}
    \item[(i)] $S_{\infty}^{+}\subseteq S_{\infty}$ is by definition open.
    \item[(ii)] $S_{\infty}^{+}\subseteq S_{\infty}^{\circ}$ follows from~\cite[Proposition 2.2.10.2]{ScholzeWeinstein2020}.
    \item[(iii)] It remains to show that $S_{\infty}^{+}$ is integrally closed
    in $S_{\infty}$. To show this, consider an element $s\in S_{\infty}$ such that there
    exists a monic polynomial $f=\sum_{j=0}^{n}f_{i}L^{i}\in S_{\infty}^{+}[L]$ with
    $f(s)=0$. Pick indices $l$ and $i_{0},\dots,i_{n}$ such that
    $s\in S_{l}$ and $f_{j}\in S_{i_{j}}^{+}$ for all $j=0,\dots,n$.
    Recall that the index category $I$ was assumed to be cofiltered.
    That is, $I^{\op}$ is filtered and thus we find an element $t\in I^{\op}$
    together with morphisms from $l$ and $i_{0},\dots,i_{n}$ to $t$.
    In particular, $s\in S_{t}$ and $f\in S_{t}^{+}[L]$. Because
    $S_{t}^{+}\subseteq S_{t}$ is integrally closed,
    $s\in S_{t}^{+}$ follows, and therefore $s\in S_{\infty}$.
  \end{itemize}
  %Indeed, $D_{\infty}^{+}$ is integrally closed because the colimit is filtered.
  We observe that the map~(\ref{eq:affperfd-closed-fibreprod-1})
  above is given by
  %\begin{equation*}
    $\Spa\left(\widehat{\left(S_{\infty},S_{\infty}^{+}\right)}\right)
    \to\Spa(S,S^{+})$.
 %\end{equation*} 
 It is an isomorphism by construction.
\end{proof}

\begin{defn}\label{defn:XproetaffperfdU}
  Fix $U\in X_{\proet}$.
  \begin{itemize}
    \item[(i)] $X_{\proet,\affperfd}/U$ is the full subcategory of
      $X_{\proet}/U$ whose objects are the maps
      $V\to U$ for affinoid perfectoid $V$.
      We equip it with the induced topology, and
      Lemma~\ref{lem:affperfd-closed-fibreprod}
      shows that it gives rise to a site.
    \item[(ii)] $X_{\proet,\affperfd}^{\fin}/U$ is the site
    whose underlying category is the category underlying
    $X_{\proet,\affperfd}/U$, but we consider only the finite
    coverings.
  \end{itemize}
  If $U=X$, write $X_{\proet,\affperfd}:=X_{\proet,\affperfd}/X$ and
  $X_{\proet,\affperfd}^{\fin}:=X_{\proet,\affperfd}^{\fin}/X$.
\end{defn}

\begin{lem}\label{lem:sheaves-on-Xproet-and-Xproetaffperfdfin}
Fix a covering $U\to X$ in $X_{\proet}$.
Then the morphisms of sites
\begin{equation*}
  \begin{tikzcd}
  X_{\proet} \arrow{r} &
  X_{\proet,\affperfd} \arrow{r} &
  X_{\proet,\affperfd}^{\fin} \\
  X_{\proet}/U \arrow{r}\arrow{u} &
  X_{\proet,\affperfd}/U \arrow{r}\arrow{u} &
  X_{\proet,\affperfd}^{\fin}/U \arrow{u}
  \end{tikzcd}
\end{equation*}
give rise to equivalences of categories
\begin{equation*}
  \begin{tikzcd}
  \Sh\left(X_{\proet},\mathbf{E}\right) \arrow{r}{\simeq}\arrow{d}{\simeq} &
  \Sh\left(X_{\proet,\affperfd},\mathbf{E}\right) \arrow{r}{\simeq}\arrow{d}{\simeq} &
  \Sh\left(X_{\proet,\affperfd}^{\fin},\mathbf{E}\right)\arrow{d}{\simeq} \\
  \Sh\left(X_{\proet}/U,\mathbf{E}\right) \arrow{r}{\simeq} &
  \Sh\left(X_{\proet,\affperfd}/U,\mathbf{E}\right) \arrow{r}{\simeq}&
  \Sh\left(X_{\proet,\affperfd}^{\fin}/U,\mathbf{E}\right)
  \end{tikzcd}
\end{equation*}
for any elementary quasi-abelian category $\mathbf{E}$.
\end{lem}

\begin{proof}
  The vertical morphism at the left-hand side is an equivalence because
  $U\to X$ is a covering. It remains to check that the horizontal arrows are
  equivalences. Since the row at the top is obtained from the row at the
  bottom by setting $U=X$, it suffices to check that the horizontal morphisms
  at the bottom are equivalences. Lemma~\ref{lem:affperfd-basis}
  gives the first equivalence
  $\Sh\left(X_{\proet}/U,\mathbf{E}\right)\simeq
  \Sh\left(X_{\proet,\affperfd}/U,\mathbf{E}\right)$. The second equivalence
  $\Sh\left(X_{\proet,\affperfd}/U,\mathbf{E}\right)\simeq
  \Sh\left(X_{\proet,\affperfd}^{\fin}/U,\mathbf{E}\right)$
  follows from Lemma~\ref{prop:siteqc-reducesheafcondtofincov},
  which applies because all affinoid perfectoids are quasicompact objects
  in $X_{\proet}$, cf.~\cite[Proposition 3.12]{Sch13pAdicHodge},
\end{proof}

%%%%%%%%%%%%%%%%%%%%%%%%%%%%%%%%%%%%%%%%%%%%%%%%%%%%%%%%%%%%
%%%%%%%%%%%%%%%%%%%%%%%%%%%%%%%%%%%%%%%%%%%%%%%%%%%%%%%%%%%%
% completed structure sheaf
%%%%%%%%%%%%%%%%%%%%%%%%%%%%%%%%%%%%%%%%%%%%%%%%%%%%%%%%%%%%
%%%%%%%%%%%%%%%%%%%%%%%%%%%%%%%%%%%%%%%%%%%%%%%%%%%%%%%%%%%%

\subsection{The pro-étale site of a point}

We recall Definition~\cite[Definition 3.3]{Sch13pAdicHodge}, see also~\cite{Sch13pAdicHodgeErratum}.
The category underlying the the pro-finite étale site $X_{\profet}$ is
$\Pro\left(X_{\fet}\right)$, and a collection of morphisms $\left\{U_{i}\to U\right\}_{i}$
is a covering if and only if it is a covering in $X_{\proet}$.

Now specialise to the case $X=*:=\Spa\left(k,k^{\circ}\right)$.
$\cal{G}:=\Gal\left(\overline{k}/k\right)$ is the absolute Galois group of $k$.
Recall the Definition~\ref{defn:Gpfsets-reconstructionpaper}
of the site of $\cal{G}$-profinite sets and
see~\cite[Proposition 3.5]{Sch13pAdicHodge} for the construction
of an equivalence of sites $*_{\profet}\simeq\cal{G}$-$\pfsets$.
Composition with the canonical map
$*_{\proet}\to*_{\profet}$ give a morphism of sites
\begin{equation*}
  \lambda\colon *_{\proet} \to \cal{G}-\pfsets.
\end{equation*}
Recall the functor $\IndBan_{k}\left(\cal{G}\right)\ni V\mapsto\cal{F}_{V}$ as in
Definition~\ref{defn:FMfunctor-Gpfsets-reconstructionpaper},
see also Lemma~\ref{lem:FMfunctor-Gpfsets-givessheaves-reconstructionpaper}.
Consider
\begin{equation}\label{eq:indBanachGaloisrep-toShonproetsite-reconstructionpaper}
  \IndBan_{k}\left(\cal{G}\right) \to \Sh\left( *_{\proet} , \IndBan_{k} \right),
  V \mapsto \lambda^{-1}\cal{F}_{V}.
\end{equation}

\begin{lem}\label{lem:basiscompareproetalepoint-and-Gpfsets-reconstructionpaper}
  $C$ is the completion of a fixed algebraic closure of $k$, cf. \S\ref{subsec:conventions-reconstructionpaper}.
  \begin{itemize}
    \item[(i)] $*_{\proet}$ admits a basis given by all affinoid perfectoids
      $\Spa\left(C,\cal{O}_{C}\right)\times S$, where $S$ denotes any profinite set.
    \item[(ii)] $\cal{G}-\pfsets$ admits a basis given by all products
      $\cal{G}\times S$,
      where $S$ denotes any profinite set carrying the trivial $\cal{G}$-action.
  \end{itemize}
\end{lem}

\begin{proof}
  Any basis for the site $*_{\proet}/C\simeq \Spa\left(C,\cal{O}_{C}\right)_{\proet}$, cf.~\cite[Proposition 3.15]{Sch13pAdicHodge},
  will descend to a basis of $*_{\proet}$. Thus (i) follows from~\cite[Remark 2.5]{Bosco21}.
  In particular, the affinoid perfectoids of the form $\Spa\left(C,\cal{O}_{C}\right)\times S$
  also form a basis of $*_{\profet}$. This gives (ii), as the equivalence of sites
  $*_{\profet}\simeq\cal{G}$-$\pfsets$ identifies $\Spa\left(C,\cal{O}_{C}\right)\times S$
  and $\cal{G}\times S$, cf. the~\cite[proof of Proposition 3.5]{Sch13pAdicHodge}.
\end{proof}

\begin{cor}\label{cor:comparepullbackalonglambda-reconstructionpaper}
  $\lambda^{-1}$ induces an equivalence
  $\Sh\left( *_{\profet} , \IndBan_{k} \right)\simeq\Sh\left( *_{\proet} , \IndBan_{k} \right)$.
  Furthermore, we have the functorial isomorphism for all profinite sets $S$
  \begin{equation*}
    \left(\lambda^{-1}\cal{F}\right)\left(\Spa\left(C,\cal{O}_{C}\right)\times S\right)
    \cong\cal{F}\left(\cal{G}\times S\right)
  \end{equation*}
\end{cor}

\begin{proof}
  This follows from Lemma~\ref{lem:basiscompareproetalepoint-and-Gpfsets-reconstructionpaper} and
  the~\cite[proof of Proposition 3.5]{Sch13pAdicHodge}
  as the equivalence $*_{\profet}\simeq\cal{G}$-$\pfsets$ identifies $\Spa\left(C,\cal{O}_{C}\right)\times S$
  and $\cal{G}\times S$.
\end{proof}

\begin{lem}\label{lem:rep-to-proetalesheaf-commutes-with-filtered-colimits}
  (\ref{eq:indBanachGaloisrep-toShonproetsite-reconstructionpaper}) commutes with filtered colimits.
\end{lem}

\begin{proof}
  By Corollary~\ref{cor:comparepullbackalonglambda-reconstructionpaper}, it remains to show
  that $V\mapsto\cal{F}_{V}$ commutes with filtered colimits. But
  \begin{equation*}
    V\mapsto\intHom_{\cont,G}\left( S , V \right)
  \end{equation*}
  commutes with filtered colimits for every $G$-profinite set $S$, by definition.
\end{proof}

\begin{lem}\label{lem:rep-to-proetalesheaf-laxmonoidal}
  (\ref{eq:indBanachGaloisrep-toShonproetsite-reconstructionpaper})
  is canonically lax monoidal. %That is, it sends monoids to monoids and modules to modules.
\end{lem}

\begin{proof}
  It remains to check that $V\mapsto\cal{F}_{V}$ is canonically lax monoidal as $\lambda^{-1}$
  is canonically strongly monoidal, cf.~\cite[Lemma 2.28]{Bo21}.
  As every operation commutes with filtered colimits, we may as well consider the restriction
  to the category of $G$-$R$-Banach modules. Here is a canonical morphism
  \begin{equation*}
    \intHom_{\cont,G}\left(S,V\right)\widehat{\otimes}_{R}\intHom_{\cont,G}\left(S,W\right)
    \to\intHom_{\cont,G}\left(S,V\widehat{\otimes}_{R}W\right)
  \end{equation*}
  of $k$-Banach spaces for every $S\in\cal{G}$-$\pfsets$ and $V,W\in\Ban_{k}\left(\cal{G}\right)$.
  Sheafify to get the desired functorial canonical morphism
  $\cal{F}_{V}\widehat{\otimes}_{k}\cal{F}_{W}\to\cal{F}_{V\widehat{\otimes}_{k}W}$.
\end{proof}

%%%%%%%%%%%%%%%%%%%%%%%%%%%%%%%%%%%%%%%%%%%%%%%%%%%%%%%%%%%%
%%%%%%%%%%%%%%%%%%%%%%%%%%%%%%%%%%%%%%%%%%%%%%%%%%%%%%%%%%%%
% completed structure sheaf
%%%%%%%%%%%%%%%%%%%%%%%%%%%%%%%%%%%%%%%%%%%%%%%%%%%%%%%%%%%%
%%%%%%%%%%%%%%%%%%%%%%%%%%%%%%%%%%%%%%%%%%%%%%%%%%%%%%%%%%%%

\subsection{Structure sheaves}

$X$ denotes again an arbitrary smooth locally Noetherian adic space over $\Spa(k,k^{\circ})$
and $\nu\colon X_{\proet} \maps X_{\et}$ is the canonical projection of sites.
See~\cite[Definition 4.1]{Sch13pAdicHodge} for the definition of the integral completed
structure sheaf $\widehat{\cal{O}}^{+}$ and the completed structure sheaf $\widehat{\cal{O}}$.
In the following, we promote these to sheaves of $k^{\circ}$-ind-Banach algebras and
$k$-ind-Banach algebras, respectively.
\begin{itemize}
\item[(i)] Let $U\in X_{\proet}$. By~\cite[Lemma 4.2(iii)]{Sch13pAdicHodge},
  $\varprojlim_{j\geq 1}\nu^{-1}\mathcal{O}_{X_{\et}}^{+}(U)/p^{j}$ is $p$-adically complete,
  thus it is $\pi$-adically complete.
  Equip it with the $\pi$-adic norm, cf. Definition~\ref{defn:I-adic-seminorm-norm}.
  This makes $\varprojlim_{j\geq 1}\nu^{-1}\mathcal{O}_{X_{\et}}^{+}(U)/p^{j}$
  a $k^{\circ}$-Banach algebra. It is thus a $k^{\circ}$-ind-Banach algebra.
  The sheafification of
  \begin{equation*}
    U \mapsto \varprojlim_{j\geq 1}\nu^{-1}\mathcal{O}_{X_{\et}}^{+}(U)/p^{j}
  \end{equation*}
  is the \emph{completed integral structure sheaf}
  $\widehat{\mathcal{O}}_{X_{\proet}}^{+}$.
  It is a sheaf of $k^{\circ}$-ind-Banach algebras because
  sheafification is strongly monoidal, cf. Lemma~\ref{lem:sh-strongly-monoidal}.
\item[(ii)] Consider the presheaf
  of $k$-Banach algebras
  \begin{align*}
    U \mapsto &\left(\varprojlim_{j\geq 1}\nu^{-1}\mathcal{O}_{X_{\et}}^{+}(U)/p^{j}\right)\widehat{\otimes}_{k^{\circ}}k
    \stackrel{\text{\ref{lem:localise-torsionfree-modules}}}{\cong}
    \left(\varprojlim_{j\geq 1}\nu^{-1}\mathcal{O}_{X_{\et}}^{+}(U)/p^{j}\right)[1/\pi].
  \end{align*}
  Sheafify it as a presheaf of $k$-ind-Banach algebras to get
  the \emph{completed structure sheaf} $\widehat{\mathcal{O}}_{X_{\proet}}$.
  It is a sheaf of $k$-ind-Banach algebras
  by Lemma~\ref{lem:sh-strongly-monoidal}.
\end{itemize}

\begin{notation}
  %When $X$ is understood,
  We may write
  $\widehat{\mathcal{O}}^{+}=\widehat{\mathcal{O}}_{X_{\proet}}^{+}$ and
  $\widehat{\mathcal{O}}=\widehat{\mathcal{O}}_{X_{\proet}}$.
\end{notation}

Here is a local description of the structure sheaves.

\begin{lem}\label{lem:sheavesonXproet-over-affinoidperfectoid}
  Assume that $U\in X_{\proet}$ is affinoid perfectoid
  with $\widehat{U}=\Spa(R,R^{+})$. Equip $R^{+}$ with the $p$-adic norm,
  giving $R=R^{+}[1/\pi]$ the structure of $k$-Banach algebra. Then
  $\widehat{\mathcal{O}}^{+}(U) \cong R^{+}$ and
  $\widehat{\mathcal{O}}(U) \cong R$.
\end{lem}

\begin{proof}
  Thanks to Lemma~\ref{lem:sheaves-on-Xproet-and-Xproetaffperfdfin}, we may
  view $\widehat{\mathcal{O}}^{+}$ and $\widehat{\mathcal{O}}$
  as sheaves on the site $X_{\proet,\affperfd}^{\fin}$.
  By \cite[Lemma 4.10(iii)]{Sch13pAdicHodge},
  it suffices to show that the presheaves
  $U\mapsto R^{+}$ and $U\mapsto R$
  are sheaves.
  
  \emph{Loc. cit.} says that $U\mapsto R$ is a
  sheaf of abstract $k$-algebras. The open mapping theorem
  implies that it is a sheaf of $k$-Banach algebras,
  and it is a sheaf of $k$-ind-Banach algebras by
  Lemma~\ref{lem:banach-to-indbanach-sheaves}.
  
  We know by~\cite[Lemma 4.10(iii)]{Sch13pAdicHodge}
  that $U\mapsto R^{+}$ is a sheaf of abstract $k^{\circ}$-algebras.
  $U\mapsto R$ being a sheaf of $k$-Banach algebras
  implies that $U\mapsto R^{+}$ is a sheaf of $k^{\circ}$-Banach algebras.
  Another application of Lemma~\ref{lem:banach-to-indbanach-sheaves}
  finishes the proof.
\end{proof}

We continue with technical results.

\begin{lem}\label{lem:hcalOUtimesS-is-HomctsScalO}
  We have the following functorial isomorphism of $k$-Banach spaces
  \begin{equation}\label{eq:hcalOUtimesS-is-HomctsScalO}
    \widehat{\cal{O}}(U\times S)
    \cong
    \intHom_{\cont}\left(S,\widehat{\cal{O}}(U)\right)
  \end{equation}
  for every affinoid perfectoid $U\in X_{\proet}$ and any profinite set $S$.
\end{lem}

\begin{proof}
  \cite[Corollary 6.6]{Sch13pAdicHodge} gives the desired map of abstract $k$-vector spaces
  and computes that it is an isomorphism. \emph{Loc cit.} also proves that
  there is an isomorphism, compatible with~(\ref{eq:hcalOUtimesS-is-HomctsScalO}),
  \begin{equation*}
    \widehat{\cal{O}}^{+}(U\times S)
    \cong
    \intHom_{\cont}\left(S,\widehat{\cal{O}}^{+}(U)\right).
  \end{equation*}
  Therefore,~(\ref{eq:hcalOUtimesS-is-HomctsScalO}),
  preserves the unit balls. It is thus an isomorphism of Banach spaces.
\end{proof}

$\pi\in k$ is a uniformiser as in \S\ref{subsec:conventions-reconstructionpaper}.

\begin{lem}\label{lem:pi-adicallycomplete-iso-if-grzeroiso}
  Consider a morphism $\phi$ of $\pi$-torsion free $k^{\circ}$-Banach rings
  carrying the $\pi$-adic topologies. It is an isomorphism if $\gr^{0}\phi$,
  with respect to the $\pi$-adic filtrations, is an isomorphism.
\end{lem}

\begin{proof}
  If $\gr^{0}\phi$ is an isomorphism, \cite[Exercise 17.16.a]{Ei95}
  implies that $\gr\phi$ is an isomorphism.
  Now apply~\cite[Chapter I, \S 4.2 page 31-32, Theorem 4(5)]{HuishiOystaeyen1996}.
\end{proof}

\begin{lem}\label{lem:Banachspace-iso-if-grzeroiso}
  Consider a morphism $\phi\colon V\to W$ of $k$-Banach spaces.
  Equip the Banach spaces with the filtration $\Fil^{n}:=\{x\colon\|x\|\leq|\pi|^{n}\}$,
  $n\in\ZZ$. Then $\phi$ is an isomorphism of Banach spaces
  if $\gr^{0}\phi$ is an isomorphism.
\end{lem}

\begin{proof}
  This follows from Lemma~\ref{lem:pi-adicallycomplete-iso-if-grzeroiso}.
\end{proof}

Fix $d\in\NN$.

\begin{lem}\label{lem:perfdtorus-tensor-products}
  We have the canonical isomorphism of $k$-Banach spaces
  \begin{equation*}
  \widehat{\cal{O}}\left(\widetilde{\TT}^{\left\{1\right\}}\right)
    \widehat{\otimes}_{k}
    \dots
    \widehat{\otimes}_{k}
    \widehat{\cal{O}}\left(\widetilde{\TT}^{\left\{d\right\}}\right)
    \isomap
  \widehat{\cal{O}}\left(\widetilde{\TT}^{d}\right).
  \end{equation*}
\end{lem}

\begin{proof}
  By Lemma~\ref{lem:Banachspace-iso-if-grzeroiso}, we may compute the
  zeroth pieces of the associated gradeds
  \begin{align*}
    \gr^{0}\left(\widehat{\cal{O}}\left(\widetilde{\TT}^{\left\{1\right\}}\right)
    \widehat{\otimes}_{k}
    \dots
    \widehat{\otimes}_{k}
    \widehat{\cal{O}}\left(\widetilde{\TT}^{\left\{d\right\}}\right)\right)
    &\cong
    \kappa\left[ T_{1}^{\pm\frac{1}{p^{\infty}}} \right]
    \otimes_{\kappa}
    \dots
    \otimes_{\kappa}
    \kappa\left[ T_{d}^{\pm\frac{1}{p^{\infty}}} \right], \text{ and}\\
    \gr^{0}\widehat{\cal{O}}\left(\widetilde{\TT}^{d}\right)
    &\cong
    \kappa\left[ T_{1}^{\pm\frac{1}{p^{\infty}}},\dots,
    T_{d}^{\pm\frac{1}{p^{\infty}}} \right],
  \end{align*}
  where $\kappa$ is the residue field of $k$ as in \S\ref{subsec:conventions-reconstructionpaper}.
  Lemma~\ref{lem:perfdtorus-tensor-products} follows.
\end{proof}

The following result appeared implicitly 
in the~\cite[proof of Lemma 8.7(ii)]{Sch13pAdicHodge}.

\begin{lem}\label{lem:perfdtorus-flat-over-torus}
  $\widehat{\cal{O}}\left(\widetilde{\TT}^{d}\right)$ is flat
  as an $\cal{O}\left(\TT^{d}\right)$-ind-Banach module.
\end{lem}

\begin{proof}
  The result is clear for $d=0$.
  Next, let $d=1$. Write
  $J_{r}:=\left\{m/p^{r}\colon m=0,\dots,p^{r}-1 %\text{ and } p^{r}\nmid m
  \right\}$
  for all $r\in\NN$ and $J:=\bigcup_{r\in\NN}J_{r}$.
  We claim that the canonical map
  \begin{equation}\label{eq:perfd-torus-flat-over-torus-1}
    \widehat{\bigoplus_{m/p^{r}\in J}
    k^{\circ}\left\<T^{\pm}\right\>\cdot T^{\frac{m}{p^{r}}}}
    \isomap k^{\circ}\left\<T^{\pm\frac{1}{p^{\infty}}}\right\>
  \end{equation}
  is an isomorphism of $k^{\circ}$-Banach modules,
  where the completion is the $\pi$-adic one.
  By Lemma~\ref{lem:pi-adicallycomplete-iso-if-grzeroiso},
  we may check that the zeroth graded piece
  with respect to the $\pi$-adic filtrations
  \begin{equation*}
    \psi\colon
    \bigoplus_{m/p^{r}\in J}
    \kappa\left[T^{\pm}\right]\cdot T^{\frac{m}{p^{r}}}
    \longrightarrow \kappa\left[T^{\pm\frac{1}{p^{\infty}}}\right]
  \end{equation*}
  is an isomorphism. This follows from (i) and (ii) below.
  \begin{itemize}
    \item[(i)] To prove injectivity,
      write $V_{r}
      :=\bigoplus_{m/p^{r}\in J_{r}}
      \kappa\left[T^{\pm}\right]\cdot T^{\frac{m}{p^{r}}}$
      for all $r\in\NN$. Every
      $\kappa\left[T^{\pm}\right]$-module $\psi\left(V_{r}\right)$
      is spanned by the set $\left\{T^{\frac{s}{p^{r}}}\colon p\nmid s\right\}$.
      These sets are pairwise disjoint, thus
      $\psi\left(V_{r}\right)\cap\psi\left(V_{r^{\prime}}\right)=\{0\}$
      for any two distinct $r,r^{\prime}\in\NN$. Therefore, it remains to check
      that $\psi|_{V_{r}}$ is injective for all $r\in\NN$.
      This follows from
      \begin{equation*}
        \psi\left( \kappa\left[T^{\pm}\right]\cdot T^{\frac{m}{p^{r}}}\right)
        =\linearspan_{\kappa}\left(\left\{T^{\frac{s}{p^{r}}} \colon s\equiv m\mod p^{r} \right\}\right),
      \end{equation*}
      which is implied by
     $\psi\left( T^{n} \cdot T^{\frac{m}{p^{r}}}\right)=T^{\frac{p^{r}n+m}{p^{r}}}$
     for all $n\in\ZZ$ and $m/p^{r}\in J$.
  \item[(ii)] $\psi$ is surjective: Given $T^{\frac{s}{p^{r}}}$ such that $p$ does not divide $s$,
  we may write $s=s^{\prime}p^{r}+m$ such that $p^{r}$ does not divide $m$. Then
  \begin{equation*}
    T^{\frac{s}{p^{r}}}
    =T^{\frac{s^{\prime}p^{r}+m}{p^{r}}}
    =T^{s^{\prime}}\cdot T^{\frac{m}{p^{r}}}
    =\psi\left(T^{s^{\prime}}\cdot T^{\frac{m}{p^{r}}}\right).
  \end{equation*}
  \end{itemize}
  Recall the definition of $c_{0}(-)$ as in~\cite[\S 3]{Sch02}.
  The isomorphism~(\ref{eq:perfd-torus-flat-over-torus-1}) gives
  \begin{equation}\label{eq:perfd-torus-flat-over-torus-2}    
    k\left\<T^{\pm\frac{1}{p^{\infty}}}\right\>
    \cong k\left\<T^{\pm}\right\>\widehat{\otimes}_{k}c_{0}(J).
  \end{equation}
  Now the flatness follows from
  Lemma~\ref{lem:Banachspace-iso-if-grzeroiso}.
  For $d>2$, proceed via Lemma~\ref{lem:perfdtorus-tensor-products}
  and~(\ref{eq:perfd-torus-flat-over-torus-2}).
\end{proof}

\iffalse %%% comment beginns.

\begin{lem}\label{lem:padiccompletion-and-tensor}
  For $p$-torsion free abstract $k^{\circ}$-modules $M$ and $N$,
  we have a canonical isomorphism
  \begin{equation}\label{eq:padiccompletion-and-tensor}
    M\widehat{\otimes}_{k^{0}} N \isomap  
    \widehat{M}\widehat{\otimes}_{k^{0}} \widehat{N}.
  \end{equation}
  Here, all completions are the $\pi$-adic ones.
\end{lem}

\begin{proof}
  Equip both the domain and the codomain of~(\ref{eq:padiccompletion-and-tensor})
  with the $\pi$-adic filtrations and compute
  \begin{equation}\label{eq:padiccompletion-and-colimits}
    \left(M\widehat{\otimes}_{k^{0}} N\right)/\pi
    \cong M/\pi\otimes_{k^{0}/\pi} N/\pi
    \cong\widehat{M}/\pi\otimes_{k^{0}/\pi} \widehat{N}/\pi
    \cong\left(\widehat{M}\otimes_{k^{0}} \widehat{N}\right)/\pi
    \cong\left(\widehat{M}\widehat{\otimes}_{k^{0}} \widehat{N}\right)/\pi.
  \end{equation}  
  Now apply Lemma~\ref{lem:pi-adicallycomplete-iso-if-grzeroiso}.
\end{proof}

\fi %%% comment ends

\iffalse %%% still need this?
\begin{lem}\label{lem:padiccompletion-and-colimits}
  Given a diagram $i\mapsto R_{i}$ of $p$-torsion free rings, we have the canonical isomorphism
  \begin{equation}\label{eq:padiccompletion-and-colimits}
    \widehat{\varinjlim_{i}R_{i}} \isomap \widehat{\varinjlim_{i}\widehat{R_{i}}}.
  \end{equation}
  Here, the completions are the $p$-adic ones.
\end{lem}

\begin{proof}
  Equip both the domain and the codomain of~(\ref{eq:padiccompletion-and-colimits})
  with the $p$-adic filtrations and compute
  \begin{equation}\label{eq:padiccompletion-and-colimits}
    \gr^{0}\widehat{\varinjlim_{i}R_{i}}
    \cong\varinjlim_{i}R_{i}/p
    =\gr^{0}\widehat{\varinjlim_{i}\widehat{R_{i}}}.
  \end{equation}  
  Now Lemma~\ref{lem:pi-adicallycomplete-iso-if-grzeroiso} for $R=\ZZ$ with the trvial
  norm and $r=p$ gives the result.
\end{proof}
\fi %%% comment ends

Fix the limit $S$ of a diagram of finite sets $j\mapsto S_{j}$ over a cofiltered index category $J$.

\begin{notation}
  Given a morphism $j\to l$ in $J$, consider the associated $\psi\colon S_{j}\to S_{l}$.
  Its dual gives
  \begin{equation*}
    \prod_{S_{l}}k^{\circ}
    = \Hom\left(S_{l}, k^{\circ} \right)
    \to\Hom\left(S_{j}, k^{\circ} \right)
    =\prod_{S_{j}}k^{\circ},
  \end{equation*}
  where the $\Hom$-sets are the ones in the category of sets.
  This defines a $J$-indexed diagram $j\mapsto\prod_{S_{j}}k^{\circ}$.
  $A_{S}$ is the $k$-Banach space whose unit ball $A_{S}^{\circ}$ is the $\pi$-adic
  completion of $\varinjlim_{j}\prod_{S_{j}}k^{\circ}$.
\end{notation}

Recall Example and Notation~\ref{examplenotation:profinitesets-in-Xproet}.

\begin{lem}\label{lem:sectionsoverUS-is-tensoring-with-AS}
  For any affinoid perfectoid $U$, we find the functorial isomorphism of $k$-Banach spaces
  \begin{equation*}
    A_{S}\widehat{\otimes}_{k}^{\rL}\widehat{\cal{O}}(U)
    \isomap\widehat{\cal{O}}(U\times S).
  \end{equation*}
  Similarly, we have the isomorphism
  $A_{S}^{\circ}\widehat{\otimes}_{k^{\circ}}\widehat{\cal{O}}^{+}(U)\isomap\widehat{\cal{O}}^{+}(U\times S)$
  of $k^{\circ}$-Banach modules.
\end{lem}

\begin{proof}
  The following completions are the $\pi$-adic ones.
  Writing $U$ as a cofiltered limit of affinoids $U_{i}$,
  \begin{equation}\label{eq:sectionsoverUS-is-tensoring-with-AS}
  \begin{split}
    \widehat{\cal{O}}^{+}(U\times S)
    &=\widehat{\varinjlim_{i,j}\cal{O}^{+}\left(U_{i}\times S_{j}\right)}
    \cong\widehat{\varinjlim_{i,j}\prod_{S_{j}}\cal{O}^{+}\left(U_{i}\right)}
    \cong\varinjlim_{j}\prod_{S_{j}}k^{\circ}\widehat{\otimes}_{k^{\circ}}\varinjlim_{j}\cal{O}^{+}\left(U_{i}\right), \\
    A_{S}^{\circ}\widehat{\otimes}_{k^{\circ}}\widehat{\cal{O}}^{+}(U)
    &\cong\widehat{\varinjlim_{j}\prod_{S_{j}}k^{\circ}}\widehat{\otimes}_{k^{\circ}}\widehat{\varinjlim_{j}\cal{O}^{+}\left(U_{i}\right)}.
    \end{split}
  \end{equation}  
  This clarifies how to write down the desired
  $\widehat{\cal{O}}^{+}(U\times S)\to A_{S}^{\circ}\widehat{\otimes}_{k^{\circ}}\widehat{\cal{O}}^{+}(U)$.
  It is an isomorphism as
  \begin{align*}
    \left(\widehat{\cal{O}}^{+}(U\times S)\right)/\pi
    &\stackrel{\text{(\ref{eq:sectionsoverUS-is-tensoring-with-AS})}}{\cong}
    \left(\varinjlim_{j}\prod_{S_{j}}k^{\circ}\widehat{\otimes}_{k^{\circ}}\varinjlim_{j}\cal{O}^{+}\left(U_{i}\right)\right)/\pi \\
    &\cong
    \left(\widehat{\varinjlim_{j}\prod_{S_{j}}k^{\circ}}\widehat{\otimes}_{k^{\circ}}\widehat{\varinjlim_{j}\cal{O}^{+}\left(U_{i}\right)}\right)/\pi
    \stackrel{\text{(\ref{eq:sectionsoverUS-is-tensoring-with-AS})}}{\cong}
    \left(A_{S}^{\circ}\widehat{\otimes}_{k^{\circ}}\widehat{\cal{O}}^{+}(U)\right)/\pi,
  \end{align*}
  thus Lemma~\ref{lem:pi-adicallycomplete-iso-if-grzeroiso} applies.
  This gives the isomorphism
  $A_{S}^{\circ}\widehat{\otimes}_{k^{\circ}}\widehat{\cal{O}}^{+}(U)\isomap\widehat{\cal{O}}^{+}(U\times S)$.  
  Furthermore,
  \begin{equation*}
    \widehat{\cal{O}}(U\times S)
    \cong\widehat{\cal{O}}^{+}(U\times S)\widehat{\otimes}_{k^{\circ}}k
    \cong
    A_{S}^{\circ}\widehat{\otimes}_{k^{\circ}}\widehat{\cal{O}}^{+}(U)\widehat{\otimes}_{k^{\circ}}k
    \cong A_{S}\widehat{\otimes}_{k}\widehat{\cal{O}}(U).
  \end{equation*}
  Here, we applied Lemma~\ref{lem:localise-torsionfree-modules}.
  Finally, use $A_{S}\widehat{\otimes}_{k}^{\rL}\widehat{\cal{O}}(U)=A_{S}\widehat{\otimes}_{k}\widehat{\cal{O}}(U)$
  by~\cite[Theorem 3.50]{BBBK18}.
\end{proof}

%\begin{lem}
%  $A_{S}^{\circ}$ is strongly flat as a $k^{\circ}$-ind-Banach space.
%\end{lem}

We fix the completion $K$ of an algebraic extension of $k$ which is perfectoid.

\begin{lem}\label{lem:hOtildeTKdtimesShotimesOTdOX-isomaphOtildeXKtimesS-reconstructionpaper}
  Suppose $X$ is affinoid and equipped with an étale morpism $X\to\TT^{d}$. Then
  \begin{equation*}
    \widehat{\cal{O}}\left(\widetilde{\TT}_{K}^{d}\times S\right)
    \widehat{\otimes}_{\cal{O}\left(\TT^{d}\right)}^{\rL}
    \cal{O}(X)
    \isomap\widehat{\cal{O}}\left(\widetilde{X}_{K}\times S\right),
  \end{equation*}
  the canonical morphism, is an isomorphism. Here, $\widetilde{X}:=X\times_{\TT^{d}}\widetilde{\TT}^{d}$.
\end{lem}

\begin{proof}
  By Lemma~\ref{lem:sectionsoverUS-is-tensoring-with-AS}, we may assume
  $S=*$. Now apply~\cite[Lemma 6.18]{Sch13pAdicHodge}
  and Lemma~\ref{lem:perfdtorus-flat-over-torus}.
  \iffalse %%% the following is too complicated
  Recall Examples~\ref{examplenotation:profinitesets-in-Xproet}
  and~\ref{examplenotation:torus-variable-index}. Then
  \begin{equation*}
    \widetilde{\TT}_{K}^{d}\times S
    =\text{``}\varprojlim_{e,L,j}\text{"}
    \widetilde{\TT}_{e,L}^{d} \times S_{j},
  \end{equation*}
  where the limit varies over $e\in\NN$, finite extensions $L\subseteq K$ of $k$,
  and $j\in J$. As explained in Notation~\ref{notatio:base-change-in-proet},
  we implicitly fix $L^{+}\subseteq L$.
  It follows that $\widehat{\cal{O}}^{+}\left(\widetilde{\TT}_{K}^{d}\times S\right)$
  is the $p$-adic completion of
  \begin{equation*}
    \varinjlim_{e,L,j}\prod_{S_{j}}L^{+}\left\<T_{d}^{\pm1/p^{e}},\dots,T_{d}^{\pm1/p^{e}}\right\>.
  \end{equation*}
  It suffices to compare
  \begin{equation*}
    \cal{O}^{+}\left(\widetilde{\TT}_{K}^{d}\times S\right)
    \widehat{\otimes}_{\cal{O}^{+}\left(\TT^{d}\right)}
    \cal{O}^{+}(X)
    \isomap\cal{O}^{+}\left(\widetilde{X}_{K}\times S\right),
  \end{equation*}
  which is clearly an isomorphism. Now $p$-adically complete
  and invert $p$.
  \fi %%% comment ends
\end{proof}

The following Definition~\ref{defn:proetaleGtorsor} is taken from~\cite[the paragraph preceding Proposition 3.6.3]{BSSW2024_rationalizationoftheKnlocalsphere}.

\begin{defn}\label{defn:proetaleGtorsor}
  Fix a profinite group $G$.
  A \emph{pro-étale $G$-torsor over $X$} is an object
  $Y\to X$ of $X_{\proet}$ admitting an action $G\times Y\to Y$ lying over the trivial
  action on $X$, such that the map $G\times Y\to Y\times_{X}Y$ given by
  $\left(g,y\right)\mapsto (y,gy)$ is an isomorphism.
\end{defn}

From now on, we assume that $K$ contains a compatible system
$1,\zeta_{p},\zeta_{p^{2}},\dots$ of $p$th power roots of unity, which
we fix. The following is~\cite[Lemma 5.5]{Sch13pAdicHodge}.
%Recall also Definition~\ref{defn:proetaleGtorsor}.

\begin{lem}\label{lem:ZpGaloiscoverings}
  Given a finite indexing set $I\subseteq\NN_{\geq1}$,
  we find $\widetilde{\TT}_{K}^{I}\to\TT_{K}^{I}$
  to be pro-étale $\ZZ_{p}(1)^{I}:=\prod_{I}\ZZ_{p}(1)$-torsor.
  The action $\ZZ_{p}(1)^{I}$ is as follows:
Fix a $\ZZ_{p}$-basis $\left\{\gamma_{i}\right\}_{i\in I}$ of $\ZZ_{p}(1)^{I}\cong\ZZ_{p}^{I}$.
Then
\begin{equation*}
  \gamma_{i}\cdot T^{\alpha}=\zeta^{\alpha_{i}}T^{\alpha}
\end{equation*}
for all $i\in I$ and $\alpha=\left(\alpha_{i}\right)_{i\in I}\in\NN^{I}$, where $T:=\left(T_{i}\right)_{i\in I}$.
Here, $\zeta^{\alpha_{i}}:=\zeta_{p^{j}}^{\alpha_{i}p^{j}}$ whenever $\alpha_{i}p^{j}\in\ZZ$.
\end{lem}

Let $j\in\{0,\dots,d\}$. Lemma~\ref{lem:ZpGaloiscoverings}
gives a $\ZZ_{p}(1)^{j}$-action
on $\widehat{\cal{O}}^{+}\left(\widetilde{\TT}_{K}^{j}\times S\right)$, which we restrict
to an action of $\ZZ_{p}^{\{j\}}$. But functoriality, this makes
\begin{equation*}
  \widehat{\cal{O}}^{+}\left(\widetilde{\TT}_{K}^{j} \times S\right)
  \widehat{\otimes}_{k^{\circ}}\cal{O}\left(\TT^{\left\{j+1,\dots,d\right\}}\right)^{\circ}
\end{equation*}
a continuous $\ZZ_{p}^{\{j\}}$-representation. %On the other hand,
%we equip $\widehat{\cal{O}}^{+}\left(\widetilde{\TT}_{K}^{j-1} \times S\right)
%\widehat{\otimes}_{k^{\circ}}\cal{O}\left(\TT^{\left\{j,\dots,d\right\}}\right)^{\circ}$
%with the trivial $\ZZ_{p}(1)^{\left\{j\right\}}$-action.
This sets the stage for

\begin{lem}\label{lem:Rt-S-ZZp-cohomology-i}
  The following map is injective and its cokernel is killed by $\zeta_{p}-1$:
  \begin{equation*}\label{eq:Rt-S-ZZp-cohomology-i}
      \widehat{\cal{O}}^{+}\left(\widetilde{\TT}_{K}^{j-1} \times S\right)
        \widehat{\otimes}_{k^{\circ}}\cal{O}\left(\TT^{\left\{j,\dots,d\right\}}\right)^{\circ}
    \to
    \Ho^{0}\left(\ZZ_{p}(1)^{\left\{j\right\}} , 
      \widehat{\cal{O}}^{+}\left(\widetilde{\TT}_{K}^{j} \times S\right)
        \widehat{\otimes}_{k^{\circ}}\cal{O}\left(\TT^{\left\{j+1,\dots,d\right\}}\right)^{\circ}
      \right).
  \end{equation*}
\end{lem}

\begin{proof}
  Fix an isomorphism $\ZZ_{p}(1)^{\{j\}}\cong\ZZ_{p}\gamma_{j}$.
  Consider the following cochain complex
  \begin{equation}\label{eq:Rt-S-ZZp-cohomology-i-thecomplex}
  \begin{split}
    0
    \longrightarrow
    \widehat{\cal{O}}^{+}\left(\widetilde{\TT}_{K}^{j-1}\times S\right)
      \widehat{\otimes}_{k^{\circ}}\cal{O}\left(\TT^{\left\{j,\dots,d\right\}}\right)^{\circ}
    &\longrightarrow
      \widehat{\cal{O}}^{+}\left(\widetilde{\TT}_{K}^{j}\times S\right)
        \widehat{\otimes}_{k^{\circ}}\cal{O}\left(\TT^{\left\{j+1,\dots,d\right\}}\right)^{\circ} \\
    &\stackrel{\gamma_{j}-1}{\longrightarrow}
      \widehat{\cal{O}}^{+}\left(\widetilde{\TT}_{K}^{j}\times S\right)
        \widehat{\otimes}_{k^{\circ}}\cal{O}\left(\TT^{\left\{j+1,\dots,d\right\}}\right)^{\circ},
      \end{split}
  \end{equation}
  concentrated in degrees $0$, $1$, and $2$.
  By~\cite[Lemma 7.3]{BhattMorrowScholze2018},
  we have to show that it is exact in degree $0$, and $\zeta_{p}-1$
  kills its first cohomology.
  By $p$-adic completeness, it suffices to consider its reductions modulo $\pi^{m}$ for all $m\in\NN$.
  We compute
  \begin{equation*}
  \begin{split}
    \left(\widehat{\cal{O}}^{+}\left(\widetilde{\TT}_{K}^{j-1}\times S\right)
      \widehat{\otimes}_{k^{\circ}}\cal{O}\left(\TT^{\left\{j,\dots,d\right\}}\right)^{\circ}\right)/\pi^{m}
      &\cong
    \widehat{\cal{O}}^{+}\left(\widetilde{\TT}_{K}^{j-1}\times S\right)/\pi^{m}
      \otimes_{k^{\circ}/\pi^{m}}\cal{O}\left(\TT^{\left\{j,\dots,d\right\}}\right)^{\circ}/\pi^{m} \\   
      &\stackrel{\text{\ref{lem:sectionsoverUS-is-tensoring-with-AS}}}{\cong}
    \varinjlim_{j}\prod_{S_{j}}\widehat{\cal{O}}^{+}\left(\widetilde{\TT}_{K}^{j-1}\right)/\pi^{m}
      \otimes_{k^{\circ}/\pi^{m}}\cal{O}\left(\TT^{\left\{j,\dots,d\right\}}\right)^{\circ}/\pi^{m} \\
      &\cong\varinjlim_{j}\prod_{S_{j}}
        \bigoplus_{\substack{\alpha_{1},\dots,\alpha_{j-1} \\ \in [0,1)\cap\ZZ[1/p]}}
        \left(K^{+}\left\<T^{\pm}\right\>/\pi^{m}\right)T_{1}^{\alpha_{1}}\cdots T_{j-1}^{\alpha_{j-1}}
  \end{split}
  \end{equation*}
  Similarly,
  \begin{equation*}
      \left(\widehat{\cal{O}}^{+}\left(\widetilde{\TT}_{K}^{j}\times S\right)
        \widehat{\otimes}_{k^{\circ}}\cal{O}\left(\TT^{\left\{j+1,\dots,d\right\}}\right)^{\circ}\right)/\pi^{m}
      \cong\varinjlim_{j}\prod_{S_{j}}
        \bigoplus_{\substack{\alpha_{1},\dots,\alpha_{j} \\ \in [0,1)\cap\ZZ[1/p]}}
        \left(K^{+}\left\<T^{\pm}\right\>/\pi^{m}\right)T_{1}^{\alpha_{1}}\cdots T_{j}^{\alpha_{j}}.
  \end{equation*}  
  That is, the reduction of~(\ref{eq:Rt-S-ZZp-cohomology-i-thecomplex}) modulo $\pi^{m}$ is
  \begin{align*}
    0
    &\longrightarrow
      \varinjlim_{j}\prod_{S_{j}}\bigoplus_{\substack{\alpha_{1},\dots,\alpha_{j-1} \\ \in [0,1)\cap\ZZ[1/p]}}\left(K^{+}\left\<T^{\pm}\right\>/\pi^{m}\right)T_{1}^{\alpha_{1}}\cdots T_{j-1}^{\alpha_{j-1}} \\
    &\longrightarrow
      \varinjlim_{j}\prod_{S_{j}}\bigoplus_{\substack{\alpha_{1},\dots,\alpha_{j} \\ \in [0,1)\cap\ZZ[1/p]}}\left(K^{+}\left\<T^{\pm}\right\>/\pi^{m}\right)T_{1}^{\alpha_{1}}\cdots T_{j}^{\alpha_{j}} \\
     &\stackrel{f\mapsto\zeta^{\alpha_{j}}f}{\longrightarrow}
         \varinjlim_{j}\prod_{S_{j}}\bigoplus_{\substack{\alpha_{1},\dots,\alpha_{j} \\ \in [0,1)\cap\ZZ[1/p]}}\left(K^{+}\left\<T^{\pm}\right\>/\pi^{m}\right)T_{1}^{\alpha_{1}}\cdots T_{j}^{\alpha_{j}}.
  \end{align*}
  As filtered colimits, products, and direct sums are exact, it suffices to consider
  the cohomology of
  \begin{multline*}
    0
    \longrightarrow
      \left(K^{+}\left\<T^{\pm}\right\>/\pi^{m}\right)T_{1}^{\alpha_{1}}\cdots T_{j-1}^{\alpha_{j-1}}
    \longrightarrow
      \left(K^{+}\left\<T^{\pm}\right\>/\pi^{m}\right)T_{1}^{\alpha_{1}}\cdots T_{j}^{\alpha_{j}} \\
     \stackrel{f\mapsto\zeta^{\alpha_{j}}f}{\longrightarrow}
         \left(K^{+}\left\<T^{\pm}\right\>/\pi^{m}\right)T_{1}^{\alpha_{1}}\cdots T_{j}^{\alpha_{j}},
  \end{multline*}
  for fixed $\alpha_{1},\dots,\alpha_{j}\in [0,1)\cap\ZZ[1/p]$. This complex is exact in degree zero,
  and its first cohomology is killed by $\zeta_{p}$; this has already been observed in
  the~\cite[proof of Lemma 5.5]{Sch13pAdicHodge}. Lemma~\ref{lem:Rt-S-ZZp-cohomology-i}
  follows.
\end{proof}

Equip
$\widehat{\cal{O}}^{+}\left(\widetilde{\TT}_{K}^{j-1} \times S\right)
\widehat{\otimes}_{k^{\circ}}\cal{O}\left(\TT^{\left\{j,\dots,d\right\}}\right)^{\circ}$
with the trivial $\ZZ_{p}(1)^{\left\{j\right\}}$-action.
The following continuous group cohomology refers to the classical one,
and not the one with values in the left heart.

\begin{lem}\label{lem:Rt-S-ZZp-cohomology-ii}
  The following map is injective with cokernel killed by $\zeta_{p}-1$:
  \begin{equation}\label{eq:Rt-S-ZZp-cohomology-ii-themap}
  \begin{split}
    &\Ho_{\cont}^{1}\left(\ZZ_{p}(1)^{\left\{j\right\}},\widehat{\cal{O}}^{+}\left(\widetilde{\TT}_{K}^{j-1} \times S\right)
        \widehat{\otimes}_{k^{\circ}}\cal{O}\left(\TT^{\left\{j,\dots,d\right\}}\right)^{\circ}\right) \\
    &\to
    \Ho_{\cont}^{1}\left(\ZZ_{p}(1)^{\left\{j\right\}} , 
      \widehat{\cal{O}}^{+}\left(\widetilde{\TT}_{K}^{j} \times S\right)
        \widehat{\otimes}_{k^{\circ}}\cal{O}\left(\TT^{\left\{j+1,\dots,d\right\}}\right)^{\circ}
      \right).
    \end{split}
    \end{equation}
\end{lem}
  
\begin{proof}
  \iffalse %%% the injectivity also follows from the argument below.
  $T_{l}^{a} \mapsto 0$ if $l=j$ and $a\not\in\ZZ$, otherwise $T_{l}^{a}\mapsto T_{l}^{a}$ defines
  defines a $\ZZ_{p}^{\{j\}}$-equivariant section of
    \begin{equation*}
    \widehat{\cal{O}}^{+}\left(\widetilde{\TT}_{K}^{j-1} \times S\right)
        \widehat{\otimes}_{k^{\circ}}\cal{O}\left(\TT^{\left\{j,\dots,d\right\}}\right)^{\circ}
      \to
      \widehat{\cal{O}}^{+}\left(\widetilde{\TT}_{K}^{j} \times S\right)
        \widehat{\otimes}_{k^{\circ}}\cal{O}\left(\TT^{\left\{j+1,\dots,d\right\}}\right)^{\circ}.
    \end{equation*}
    Its follows that~(\ref{eq:Rt-S-ZZp-cohomology-ii-themap}) is injective. Next,
    \fi %%% comment ends
    Fix $\ZZ_{p}(1)^{\{j\}}\cong\ZZ_{p}\gamma_{j}$ and
    compute with~\cite[Lemma 7.3]{BhattMorrowScholze2018}
  \begin{equation*}
  \begin{split}
    &\Ho_{\cont}^{1}\left(\ZZ_{p}(1)^{\left\{j\right\}},\widehat{\cal{O}}^{+}\left(\widetilde{\TT}_{K}^{j-1} \times S\right)
        \widehat{\otimes}_{k^{\circ}}\cal{O}\left(\TT^{\left\{j,\dots,d\right\}}\right)^{\circ}\right) \\
    &\cong\coker
    \left(
      \widehat{\cal{O}}^{+}\left(\widetilde{\TT}_{K}^{j-1} \times S\right)
        \widehat{\otimes}_{k^{\circ}}\cal{O}\left(\TT^{\left\{j,\dots,d\right\}}\right)^{\circ}
     \stackrel{\gamma_{j}-1}{\longrightarrow}   
      \widehat{\cal{O}}^{+}\left(\widetilde{\TT}_{K}^{j-1} \times S\right)
        \widehat{\otimes}_{k^{\circ}}\cal{O}\left(\TT^{\left\{j,\dots,d\right\}}\right)^{\circ}
    \right) \\
    &=\widehat{\cal{O}}^{+}\left(\widetilde{\TT}_{K}^{j-1} \times S\right)
      \widehat{\otimes}_{k^{\circ}}\cal{O}\left(\TT^{\left\{j,\dots,d\right\}}\right)^{\circ}.
    \end{split}
  \end{equation*}
  The last equality holds because $\gamma_{j}$ acts trivially
  on $\widehat{\cal{O}}^{+}\left(\widetilde{\TT}_{K}^{j-1} \times S\right)
      \widehat{\otimes}_{k^{\circ}}\cal{O}\left(\TT^{\left\{j,\dots,d\right\}}\right)^{\circ}$.
  It remains to understand
  \begin{equation*}
  \begin{split}
    &\Ho_{\cont}^{1}\left(\ZZ_{p}(1)^{\left\{j\right\}},\widehat{\cal{O}}^{+}\left(\widetilde{\TT}_{K}^{j} \times S\right)
        \widehat{\otimes}_{k^{\circ}}\cal{O}\left(\TT^{\left\{j+1,\dots,d\right\}}\right)^{\circ}\right) \\
    &\cong\coker
    \left(
      \widehat{\cal{O}}^{+}\left(\widetilde{\TT}_{K}^{j} \times S\right)
        \widehat{\otimes}_{k^{\circ}}\cal{O}\left(\TT^{\left\{j+1,\dots,d\right\}}\right)^{\circ}
     \stackrel{\gamma_{j}-1}{\longrightarrow}   
      \widehat{\cal{O}}^{+}\left(\widetilde{\TT}_{K}^{j} \times S\right)
        \widehat{\otimes}_{k^{\circ}}\cal{O}\left(\TT^{\left\{j+1,\dots,d\right\}}\right)^{\circ}
    \right).
    \end{split}
  \end{equation*}
    Proceed as in the proof of Lemma~\ref{lem:Rt-S-ZZp-cohomology-i}:
    reduce modulo $\pi^{m}$ for all $m$. Then we arrive, up to filtered colimits, products,
    and direct sums, again at the setting of~\cite[Lemma 5.5]{Sch13pAdicHodge}.
\end{proof}

%%%%%%%%%%%%%%%%%%%%%%%%%%%%%%%%%%%%%%%%%%%%%%%%%%%%%%%%%%%%
%%%%%%%%%%%%%%%%%%%%%%%%%%%%%%%%%%%%%%%%%%%%%%%%%%%%%%%%%%%%
% Period sheaves
%%%%%%%%%%%%%%%%%%%%%%%%%%%%%%%%%%%%%%%%%%%%%%%%%%%%%%%%%%%%
%%%%%%%%%%%%%%%%%%%%%%%%%%%%%%%%%%%%%%%%%%%%%%%%%%%%%%%%%%%%

\section{Period rings}
\label{subsec:locan-period-ring}

This \S\ref{subsec:locan-period-ring} follows the discussion at the beginning
of~\cite[\S 6]{Sch13pAdicHodge}. Our new contribution
is the definition of the positive overconvergent de Rham period ring,
and related constructions.
Fix the completion $K$ of an algebraic extension of $k$ which
is perfectoid. Pick a ring of integral elements $K^{+}\subseteq K$ 
containing $k^{\circ}$ and recall
\emph{tilting}, cf.~\cite[Lemma 3.4]{Sch12perfectoid}.
Fix an element $p^{\flat}\in K^{\flat}$ such that $\left(p^{\flat}\right)^{\sharp}/p\in (K^{+})^{\times}$.
Let $\left(R,R^{+}\right)$ denote an affinoid perfectoid
$\left(K,K^{+}\right)$-algebra. Its tilt is 
 $\left(R^{\flat},R^{\flat+}\right)$, cf.~\cite[Proposition 5.17 and Lemma 6.2]{Sch12perfectoid}.

%%%%%%%%%%%%%%%%%%%%%%%%%%%%%%%%%%%%%%%%%%%%%%%%%%%%%%%%%%%%
% Integral period rings
%%%%%%%%%%%%%%%%%%%%%%%%%%%%%%%%%%%%%%%%%%%%%%%%%%%%%%%%%%%%

\subsection{Integral period rings}
\label{subsubsection:integralperiodrings-solutionpaper}

%%%%%%%%%%%%%%%%%%%%%%%%%%%%%%%%%%%%%%%%%%%%%%%%%%%%%%%%%%%%
% Integral period rings
%%%%%%%%%%%%%%%%%%%%%%%%%%%%%%%%%%%%%%%%%%%%%%%%%%%%%%%%%%%%

\subsubsection{Infinitesimal period rings}
\label{subsubsection:Ainf}

\begin{equation*}
  \A_{\inf}\left(R,R^{+}\right)
  :=W\left(R^{\flat+}\right)
\end{equation*}
is the \emph{relative infinitesimal period ring}.
Here, the operator $W$ refers to the (always $p$-typical) Witt vectors.
We equip $\A_{\inf}\left(R,R^{+}\right)$ with the
$\left(p,\left[p^{\flat}\right]\right)$-adic seminorm,
cf. Definition~\ref{defn:I-adic-seminorm-norm}. That is,
$\|x\|:=p^{-v}$, where $v\in\NN$ is maximal with respect to
the property that $x\in\left(p,\left[p^{\flat}\right]\right)^{v}$;
$\|x\|:=0$ if $x\in\left(p,\left[p^{\flat}\right]\right)^{v}$ for all $v\in\NN$.
When the underlying perfectoid affinoid field
is understood, write $A_{\inf}:= \A_{\inf}\left(K,K^{+}\right)$.

\begin{lem}\label{Ainf-strictpring}
  The underlying abstract ring of
  $\A_{\inf}\left(R,R^{+}\right)$ is a strict $p$-ring.
\end{lem}

\begin{proof}
  $R^{\flat}$ is a perfectoid $K^{\flat}$-algebra, see
  the discussion in~\cite[\S 5]{Sch12perfectoid}.
  \emph{Loc. cit.} Proposition 5.9 gives that $R^{\flat +}$ is a perfect
  $\FF_{p}$-algebra and thus the lemma.
\end{proof}

Since $\kappa$ is perfect, we get the isomorphism at the left of
the composition
\begin{equation}\label{eq:kappa-to-Kflatplus}
  \kappa
  \cong \varprojlim_{\Phi} \kappa 
  = \varprojlim_{\Phi} k^{\circ}/\pi \to \varprojlim_{\Phi} K^{+}/\pi \cong K^{\flat +}.
\end{equation}
Here, $\Phi$ denotes the Frobenii. The isomorphism at the right-hand side
comes from~\cite[Lemma 3.4]{Sch12perfectoid}. This composition~(\ref{eq:kappa-to-Kflatplus}) is a
morphism of rings, thus it gives a map
\begin{equation*}
  W(\kappa) \to A_{\inf}
\end{equation*}
between the associated rings of Witt vectors. We find that
$A_{\inf}$ is a $W(\kappa)$-algebra, and so is $\A_{\inf}\left(R,R^{+}\right)$.
Once we equip $W(\kappa)$ with the $p$-adic norm, 
Lemma~\ref{lem:bounded-map-adic-rings} implies that
they are both $W(\kappa)$-Banach algebras.

On the other hand, $k^{\circ}$ is a $W(\kappa)$-algebra and so is $K^{+}$.
Lemma~\ref{lem:bounded-map-adic-rings} implies that
$K^{+}$ is a $W(\kappa)$-Banach algebra, and so is
$R^{+}$. We may now follow the~\cite[proof of Proposition 4.4.2]{BrinonConradpadicHodge}
to find that \emph{Fontaine's map}
\begin{equation*}
  \theta_{\inf}\colon\A_{\inf}\left(R,R^{+}\right) \to R^{+},
  \sum_{n\geq0}[a_{n}]p^{n} \mapsto \sum_{n\geq0}a_{n}^{\sharp}p^{n}.
\end{equation*}
is a morphism of $W\left(\kappa\right)$-algebras.
We will see shortly, cf. Lemma~\ref{theta:bd-strict}, that $\theta_{\inf}$
is a morphism of $W\left(\kappa\right)$-Banach algebras.

\begin{lem}\label{lem:kernel-theta}
  There is an element $\xi\in A_{\inf}$ that generates
  $\ker\theta_{\inf}$  and is not a zero-divisor in $\A_{\inf}(R,R^{+})$.
  It is of the form $\xi=\left[p^{\flat}\right]-ap$ for some unit $a\in A_{\inf}$.
\end{lem}

\begin{proof}
  Everything is proven in~\cite[Lemma 6.3]{Sch13pAdicHodge}, except
  that $a$ is a unit. Write $\xi=\left(\xi_{0},\xi_{1},\xi_{2},\dots \right)$
  as a Witt vector. Then $\xi_{1}$ is a unit in $K^{\flat+}$, see~\cite[Remark 3.11]{BhattMorrowScholze2018}.
  That is $\xi=\sum_{n\geq 0}\left[\xi_{n}^{1/p^{n}}\right]p^{n}$, thus
  $a=-\sum_{n\geq 0}\left[\xi_{n+1}^{1/p^{n+1}}\right]p^{n}$ is a unit modulo $p$.
  Because $A_{\inf}$ is $p$-adically complete, by
  Lemma~\ref{Ainf-strictpring}, the
  result follows with~\cite[\href{https://stacks.math.columbia.edu/tag/05GI}{Tag 05GI}]{stacks-project}.
\end{proof}

We note that the definition of $\xi$ requires a choice, see~\cite[Remark 3.11]{BhattMorrowScholze2018}.
%\todo[noline]{Explain later on what this means in terms of geometry.}

\begin{cor}\label{cor:Ainf-complete}
  $\A_{\inf}\left(R,R^{+}\right)$ is a $W(\kappa)$-Banach algebra.
\end{cor}

We learned the following proof from~\cite{437338}.

\begin{proof}[Proof of Corollary~\ref{cor:Ainf-complete}]
  We have to check that $\A_{\inf}\left(R,R^{+}\right)$ is complete with respect to the
  $\left(p,\left[p^{\flat}\right]\right)$-adic topology.
  But $\left(p,\left[p^{\flat}\right]\right)=(p,\xi)$, and
  $\xi,p$ is a regular sequence: $\xi$ is not a zero-divisor,
  by Lemma~\ref{lem:kernel-theta}, and
  $p$ is not a zero-divisor
  in $\A_{\inf}\left(R,R^{+}\right)/\xi\cong R^{+}$.
  Thus Lemma~\ref{lem:S-Icompl-iff-sicomplforalli}
  applies so that it suffices to check that
  $\A_{\inf}\left(R,R^{+}\right)$ is $\xi$-adically
  and $p$-adically complete. Now see the~\cite[proof of Proposition 15.3.4]{TCH2019}.
\end{proof}

\begin{cor}\label{theta:bd-strict}
  $\theta_{\inf}$ is a strict epimorphism of $W(\kappa)$-Banach algebras.
\end{cor}

\begin{proof}
  Surjectivity is clear. Lemma~\ref{lem:kernel-theta} implies
  %\begin{equation*}
    $\left(p,\left[p^{\flat}\right]\right)=(p,\xi)=(p)+\ker\theta_{\inf}$.
  %\end{equation*}
  Everything follows now from the Lemmata~\ref{lem:bounded-map-adic-rings}
  and~\ref{lem:BanFcirc-kercoker}.
\end{proof}

%The following Lemma will be useful~\ref{lem:Ainf-divisionbyxi-andp}.

\begin{lem}\label{lem:Ainf-divisionbyxi-andp}
  Fix an element $f\in\A_{\inf}\left(R,R^{+}\right)$.
    (i) If $\xi$ divides $fp$ then $\xi$ divides $f$.
    (ii) If $p$ divides $f\xi$ then $p$ divides $f$.
\end{lem}

\begin{proof}
  Let $\xi$ divide $fp$.
  We get $0=\theta_{\inf}\left(fp\right)=\theta_{\inf}(f)p\in R^{+}$.
  That is $\theta_{\inf}(f)=0$, giving $\xi|f$ and thus (i).
  To prove (ii), assume that $p$ divides $f\xi$.
  Consider
  $\sigma\colon\A_{\inf}(R,R^{+})\to\A_{\inf}(R,R^{+})/p\cong R^{\flat +}$.
  Then $0=\sigma(f\xi)=\sigma(f\left[p^{\flat}\right])=\sigma(f)p^{\flat}\in R^{\flat,+}$.
  Thus $\sigma(f)=0$, giving $p|f$.
\end{proof}

  \begin{lem}\label{lem:Ainf-normxi-normp}
    For every $f\in\A_{\inf}(R,R^{+})$,
    (i) $\|f\xi\|=\|f\|p^{-1}$,
    (ii) $\|fp\|=\|f\|p^{-1}$.
  \end{lem}

  \begin{proof}
    First, we prove (i).
    The estimate $\|f\xi\|\leq\|f\|p^{-1}$ is clear.
    Recall the description $\left(p,\left[p^{\flat}\right]\right)=(p,\xi)$ from Lemma~\ref{lem:kernel-theta}.
    To show $\geq$, assume that $f\in(p,\xi)^{v}$,
    and $v$ is maximal with respect to this property.
    Then we have to show that $v$ is maximal with respect to the
    property $f\xi\in(p,\xi)^{v+1}$.
    Suppose it is not, that is $f\xi\in(p,\xi)^{v+2}$.
    We may write
    %\begin{equation*}
      $f\xi
      = \sum_{i=0}^{v+2}f_{i}p^{i}\xi^{v+2-i}$
    %\end{equation*}
    for certain $f_{0},\dots,f_{v+2}\in\A_{\inf}\left( R , R^{+} \right)$.
    Then $f_{v+2}p^{v+2}$ is divisible by $\xi$.
    Write $f_{v+2}p^{v+2}=f_{v+2}^{\prime}p^{v+2}\xi$
    for some $f_{v+2}^{\prime}$
    with Lemma~\ref{lem:Ainf-divisionbyxi-andp}(i)
    such that
    %\begin{equation*}
      $f\xi
      = \sum_{i=0}^{v+1}f_{i}p^{i}\xi^{v+2-i} + f_{v+2}^{\prime}p^{v+2}\xi$.
    %\end{equation*}
    Since $\xi$ is not a zero-divisor, we find the following contradiction:
    \begin{equation*}
      f = \sum_{i=0}^{v+1}f_{i}p^{i}\xi^{v+1-i} + f_{v+2}^{\prime}p^{v+2}
      \in(p,\xi)^{v+1}.
    \end{equation*}
    The proof of (ii) is similar, but we apply Lemma~\ref{lem:Ainf-divisionbyxi-andp}(ii)
    and use that $p$ is not a zero-divisor, cf. Lemma~\ref{Ainf-strictpring}.
  \end{proof}

%%%%%%%%%%%%%%%%%%%%%%%%%%%%%%%%%%%%%%%%%%%%%%%%%%%%%%%%%%%%
% Integral period rings
%%%%%%%%%%%%%%%%%%%%%%%%%%%%%%%%%%%%%%%%%%%%%%%%%%%%%%%%%%%%

\subsubsection{Overconvergent integral period rings}
\label{subsubsection:oc-int-periods}

Recall Definition~\ref{defn:Banach-completions} and
Notation~\ref{notation:indexing-sets}.

\begin{defn}
  For all $q\in\NN$, define the $W(\kappa)$-Banach algebra
  \begin{equation*}
    \A_{\dR}^{q}\left( R,R^{+} \right)
    :=\A_{\inf}\left( R,R^{+} \right)\left\<\frac{\ker\theta_{\inf}}{p^{q}}\right\>.
  \end{equation*}
  For $q=\infty$, define the $W(\kappa)$-ind-Banach algebra
  \begin{equation*}
    \A_{\dR}^{\infty}\left( R,R^{+} \right)
    :=\A_{\inf}\left( R,R^{+} \right)\left\<\frac{\ker\theta_{\inf}}{p^{\infty}}\right\>
    =\text{``}\varinjlim_{q\in\NN}\text{"}\A_{\dR}^{q}\left( R,R^{+} \right).
  \end{equation*}
\end{defn}

\begin{notation}\label{notation:AdRdagger}
  Write
  $\A_{\dR}^{\dag}\left( R,R^{+} \right)
    :=\A_{\dR}^{\infty}\left( R,R^{+} \right)
    =\text{``}\varinjlim\text{"}_{q\in\NN}\A_{\dR}^{q}\left( R,R^{+} \right)$.
\end{notation}

\begin{remark}
  We highlight that $\A_{\inf}\left( R, R^{+} \right)$ carries the
  $\left( p , \left[ p^{\flat} \right]\right)$-adic topology, which is equivalent to the
  $\left( p , \xi\right)$-adic topology by Lemma~\ref{lem:kernel-theta}.
  For example, the power series $\sum_{\alpha\geq0}\xi^{\alpha}\left(\zeta/p^{q}\right)^{\alpha}$
  is an element of $\A_{\inf}\left( R, R^{+} \right)\left\<\zeta/p^{q}\right\>^{\alpha}$ with image
  \begin{equation*}
    \sum_{\alpha\geq0}\frac{\xi^{2\alpha}}{p^{q\alpha}}
    \in\A_{\dR}^{q}\left( R , R^{+} \right).
  \end{equation*}
%We hope that this example convinces the reader that the following
%Proposition~\ref{prop:Ageqm-Aleqm-iso} is correct.
\end{remark}

\begin{lem}\label{lem:Alam-local-complete-at-xi}
  The canonical morphisms
  \begin{align*}
    \A_{\inf}\left( R,R^{+} \right)\left\<\frac{\xi}{p^{q}}\right\>
    &\stackrel{\cong}{\longrightarrow}\A_{\dR}^{q}\left( R,R^{+} \right)
  \end{align*}
  are isomorphisms of $\A_{\inf}\left( R,R^{+} \right)$-Banach algebras
  for every $q\in\NN$. It is an isomorphism of
  $\A_{\inf}\left( R,R^{+} \right)$-ind-Banach algebras
  for $q=\infty$.
\end{lem}

\begin{proof}
  This follows from the Lemmata~\ref{lem:completionsalongideals-generators}
  and~\ref{lem:kernel-theta}.
\end{proof}

\begin{lem}\label{lem:Fontaines-map-for-Ala}
  $\theta_{\inf}\colon\A_{\inf}\left( R , R^{+} \right) \to R^{+}$
  factors through strict epimorphisms
  \begin{equation*}\label{eq:Fontaines-map-for-Ala}
      \theta_{\dR}^{q}\colon\A_{\dR}^{q}\left( R , R^{+} \right) \to R^{+} \\
  \end{equation*}
  of  $W(\kappa)$-Banach algebras for every $q\in\NN$.
  Their kernels are principal ideals generated by $\xi/p^{q}$.
  %where $\xi$ is chosen as in Lemma~\ref{lem:kernel-theta}.
  We exhibit a strict epimorphism
  \begin{equation*}
    \theta_{\dR}^{\infty}\colon\A_{\dR}^{\infty}\left( R , R^{+} \right) \to R^{+}
  \end{equation*}
  of $W(\kappa)$-ind-Banach algebras by passing to the colimit along $q\to\infty$.
\end{lem}

%Abusing notation, we denote the morphisms~(\ref{eq:Fontaines-map-for-Ala})
%again by $\theta$. We also refer to them as \emph{Fontaine's maps.}
Abusing notation, we refer to $\theta_{\dR}^{q}$
for $q\in\NN\cup\{\infty\}$ again as \emph{Fontaine's maps.}

\begin{notation}
  Write $\theta_{\dR}^{\dag}:=\theta_{\dR}^{\infty}$.
\end{notation}

\begin{proof}[Proof of Lemma~\ref{lem:Fontaines-map-for-Ala}]
  Fix $q$. Colimits preserve strict epimorphisms,
  cf.~\cite[Lemma 3.7]{BBKFrechetModulesDescent}, thus one may assume
  $q<\infty$. We get $\theta_{\dR}^{q}$ from an application of the
  Lemmata~\ref{lem:Banach-completions-extend}
  and~\ref{lem:Alam-local-complete-at-xi}. The proof of \emph{loc. cit.} also
  gives a commutative diagram
  \begin{equation}\label{cd:Fontaines-map-for-Ala-iota-thetainf-thetalaq}
    \begin{tikzcd}
    \A_{\dR}^{q}\left( R , R^{+} \right) \arrow{rd}{\theta_{\dR}^{q}} &  \\
    \A_{\inf}\left( R , R^{+} \right) \arrow{r}{\theta_{\inf}}\arrow{u}{\iota} & R^{+},
    \end{tikzcd}
  \end{equation}
  where $\iota$ is the canonical map.
  Since $\theta_{\inf}$ is a strict epimorphism, cf. Lemma~\ref{theta:bd-strict},
  \cite[Proposition 1.1.8]{Sch99} implies
  that $\theta_{\dR}^{q}$ is a strict epimorphism.

  Now compute the kernel. Let
  $a:=\sum_{\alpha\geq0}a_{\alpha}\left( \xi/p^{q} \right)^{\alpha}\in\ker\theta_{\dR}^{q}$
  with coefficients $a_{\alpha}\in\A_{\inf}\left( R, R^{+}\right)$ for all $\alpha\geq0$.
  Since $\theta_{\dR}^{q}$ is bounded and $\theta_{\inf}\left(\xi\right)=0$,
  \begin{equation}\label{eq:Fontaines-map-for-Ala-1}
  \begin{split}
    \theta_{\dR}^{q}\left( \sum_{\alpha\geq1}a_{\alpha}\left( \frac{\xi}{p^{q}} \right)^{\alpha} \right)
    =0.
  \end{split}
  \end{equation}
  Furthermore,
  \begin{equation*}
    \theta_{\inf}\left( a_{0} \right)
    \stackrel{\text{(\ref{cd:Fontaines-map-for-Ala-iota-thetainf-thetalaq})}}{=}
    \theta_{\dR}^{q}\left( \iota\left( a_{0}\right)\right)
    \stackrel{\text{(\ref{eq:Fontaines-map-for-Ala-1})}}{=}\theta_{\dR}^{q} \left( a \right)=0.
  \end{equation*}
  Lemma~\ref{lem:kernel-theta} implies
  %\begin{equation*}
    $a_{0} = \widetilde{a}_{0}\xi$
  %\end{equation*}    
  for some $\widetilde{a}_{0}\in\A_{\inf}\left( R , R^{+}\right)$. We find
  \begin{equation*}
    a
    = a_{0} + \sum_{\alpha\geq1}a_{\alpha}\left( \frac{\xi}{p^{q}} \right)^{\alpha}
    =\frac{\xi}{p^{q}}\left( \widetilde{a}_{0}p^{q} + \sum_{\alpha\geq1}a_{\alpha}\left( \frac{\xi}{p^{q}} \right)^{\alpha-1}\right)
    \in\left( \frac{\xi}{p^{q}} \right)
  \end{equation*}
  and $\ker\theta_{\dR}^{q}=\left( \xi/p^{q} \right)$ follows.
\end{proof}

\begin{lem}\label{lem:idealpmzeta-xi-closed-in-Ainflanglerzetaangle}
  The multiplication-by-$\left(p^{q}\zeta-\xi\right)$-map
  \begin{equation*}
    \A_{\inf}\left(R,R^{+}\right)\left\<\zeta\right\>\to\A_{\inf}\left(R,R^{+}\right)\left\<\zeta\right\>
  \end{equation*}
  is a strict monomorphism for every $q\in\NN_{\geq 2}$.
  Thus the
  $(p^{q}\zeta-\xi)\subseteq\A_{\inf}\left(R,R^{+}\right)\left\<\zeta\right\>$
  are closed and
  $\A_{\dR}^{q}\left(R,R^{+}\right)\cong\A_{\inf}\left(R,R^{+}\right)\left\<\zeta\right\>/\left(p^{q}\zeta-\xi\right)$.
\end{lem}

\begin{proof}
  We have to check that the morphism is injective,
  its image is closed,
  and it is open onto its image,
  cf. Lemma~\ref{lem:BanFcirc-kercoker}(iii).
  To show injectivity, note that
  \begin{equation*}
    \left(p^{q}\zeta-\xi\right)\sum_{\alpha\geq0}a_{\alpha}\zeta^{\alpha}=0
  \end{equation*}
  implies $-a_{0}\xi=0$ and $-a_{\alpha}\xi+p^{q}a_{\alpha-1}=0$
  for all $\alpha\geq 1$.
  By induction and because $\xi$ is not a zero-divisor, cf. Lemma~\ref{lem:kernel-theta},
  this gives $a_{\alpha}=0$ for all $\alpha$, thus injectivity.
  
  Next, we show that
  the ideal $(p^{q}\zeta-\xi)\subseteq\A_{\inf}\left(R,R^{+}\right)\left\<\zeta\right\>$
  is closed. If $f=\sum_{n\geq 0}f_{n}$ is convergent
  in $\A_{\inf}\left(S,S^{+}\right)\<\zeta\>$ with $f_{n}\in(p^{q}\zeta-\xi)$ for all $n\geq 0$,
  we can pick $g_{n}\in\A_{\inf}\left(S,S^{+}\right)\<\zeta\>$ such that
  $f_{n}=g_{n}\left(p^{q}\zeta-\xi\right)$. Because $q\geq 2$,
  $p^{-q}< p^{-1}$. Thus
  \begin{equation*}
    \|g_{n}p^{q}\zeta\|
    \leq \|g_{n}\|p^{-q}
    < \|g_{n}\|p^{-1}.
  \end{equation*}
  On the other hand, Lemma~\ref{lem:Ainf-normxi-normp} implies $\|g_{n}\xi\|=\|g_{n}\|p^{-1}$.
  This gives $\|g_{n}p^{q}\zeta\|<\|g_{n}\xi\|$, therefore~\cite[\S 2.1, Proposition 2]{Bo14}
  applies and we find
  \begin{equation*}
    \|f_{n}\|=\|g_{n}p^{q}\zeta - g_{n}\xi\| = \max\{\|g_{n}p^{q}\zeta\| , \|g_{n}\xi\|\} = \|g_{n}\|p^{-1}.
  \end{equation*}
  Since $\left(f_{n}\right)_{n\geq0}$ is a zero-sequence, this implies
  that $g_{n}\to0$ for $n\to\infty$. In particular,
  $f=\left(\sum_{n\geq0}g_{n}\right)\left(p^{q}\zeta-\xi\right)\in\left( p^{q}\zeta-\xi\right)$.
  
  Regarding openness, this follows from the following fact, which we have already
  proven above:
  $\|g\left(p^{q}\zeta-\xi\right)\|=\|g\|p^{-q}$ for all
  $g\in\A_{\inf}\left(R,R^{+}\right)\left\<\zeta\right\>$.
  Indeed, this
  implies that the
  ball of radius $p^{-N}$ in the image of the
  multiplication-by-$\left(p^{q}\zeta-\xi\right)$-map is contained in the image
  of the pall of radius $p^{-N+q}$. Finally,
  \begin{equation*}
    \A_{\dR}^{q}\left(R,R^{+}\right)
    \stackrel{\text{\ref{lem:Alam-local-complete-at-xi}}}{\cong}
    \A_{\inf}\left(R,R^{+}\right)\left\<\frac{\xi}{p^{q}}\right\>
    \stackrel{\text{\ref{lem:Banach-localisation-in-one-variable-description}}}{\cong}
    \A_{\inf}\left(R,R^{+}\right)\left\<\zeta\right\>/\overline{\left(p^{q}\zeta-\xi\right)}
  \end{equation*}
  implies the second sentence of Lemma~\ref{lem:idealpmzeta-xi-closed-in-Ainflanglerzetaangle}.
\end{proof}

\begin{lem}\label{lem:xi-over-pm-notzerodivisor}
  $\xi/p^{q}\in\A_{\dR}^{q}\left( R , R^{+} \right)$
  is not a zero-divisor for all $q\in\NN_{\geq 2}$.
\end{lem}

\begin{proof}
  The identification
  $\A_{\dR}^{q}\left(R,R^{+}\right)\cong\A_{\inf}\left(R,R^{+}\right)\left\<\zeta\right\>/\left(p^{q}\zeta-\xi\right)$
  from Lemma~\ref{lem:idealpmzeta-xi-closed-in-Ainflanglerzetaangle}
  implies that we have to check the following:
  Given
  $f=\sum_{\alpha\geq0}f_{\alpha}\zeta^{\alpha}\in\A_{\inf}\left( R , R^{+} \right)\left\<\zeta\right\>$,
  \begin{equation}\label{eq:xi-over-pm-notzerodivisor-implies}
    \zeta f\in \left( p^{q}\zeta - \xi \right) \implies
    f\in \left( p^{q}\zeta - \xi \right).
  \end{equation}
  We compute that~(\ref{eq:xi-over-pm-notzerodivisor-implies}) holds.
  Let $g=\sum_{\alpha\geq0}g_{\alpha}\zeta^{\alpha}$ such that
  \begin{equation*}
    \zeta f
    = g \left( p^{q}\zeta - \xi \right)
    =-\xi g_{0} + \sum_{\alpha\geq 1}\left( p^{q}g_{\alpha-1}-\xi g_{\alpha}\right)\zeta^{\alpha}.
  \end{equation*}
  Since $\xi$ is not a zero-divisor,  cf. Lemma~\ref{lem:kernel-theta}, $g_{0}=0$.
  Proceeding by induction, we find $g_{\alpha}=0$ for all $\alpha\geq0$,
  thus $f=0$. In particular, $f\in \left( p^{q}\zeta - \xi \right)$.
\end{proof}

\begin{lem}\label{AdR-ptorsionfree-reconstructionpaper}
  $p\in\A_{\dR}^{q}\left( R,R^{+} \right)$ is not a zero-divisor for all $q\in\NN_{\geq2}$.
\end{lem}

\begin{proof}
  The identification
  $\A_{\dR}^{q}\left(R,R^{+}\right)\cong\A_{\inf}\left(R,R^{+}\right)\left\<\zeta\right\>/\left(p^{q}\zeta-\xi\right)$
  from Lemma~\ref{lem:idealpmzeta-xi-closed-in-Ainflanglerzetaangle}
  implies that we have to check the following:
  Given
  $f=\sum_{\alpha\geq0}f_{\alpha}\zeta^{\alpha}\in\A_{\inf}\left( R , R^{+} \right)\left\<\zeta\right\>$,
  \begin{equation}\label{eq:xi-over-pm-notzerodivisor-implies-ptorsion}
    p f\in \left( p^{q}\zeta - \xi \right) \implies
    f\in \left( p^{q}\zeta - \xi \right).
  \end{equation}
  We compute that~(\ref{eq:xi-over-pm-notzerodivisor-implies-ptorsion}) holds.
  Let $g=\sum_{\alpha\geq0}g_{\alpha}\zeta^{\alpha}$ such that
  \begin{equation*}
    p f
    = g \left( p^{q}\zeta - \xi \right)
    =-\xi g_{0} + \sum_{\alpha\geq 1}\left( p^{q}g_{\alpha-1}-\xi g_{\alpha}\right)\zeta^{\alpha}.
  \end{equation*}
  Lemma~\ref{lem:Ainf-divisionbyxi-andp}(ii)
  implies that $p$ divides $g_{0}$. Induction gives that
  every $g_{\alpha}$ is divisible by $p$. This allows to define
  a formal power series $g^{\prime}$ such that $pg^{\prime}=g^{\prime}$.
  As $p\in\A_{\inf}\left( R , R^{+} \right)$ is not a zero-divisor,
  cf. Lemma~\ref{Ainf-strictpring},
  $f=g^{\prime}\left( p^{q}\zeta - \xi \right)$.
  Using Lemma~\ref{lem:Ainf-normxi-normp}, we find
  $g^{\prime}\in\A_{\inf}\left(R,R^{+}\right)\left\<\zeta\right\>$, as desired.
\end{proof}

%%%%%%%%%%%%%%%%%%%%%%%%%%%%%%%%%%%%%%%%%%%%%%%%%%%%%%%%%%%%
%%%%%%%%%%%%%%%%%%%%%%%%%%%%%%%%%%%%%%%%%%%%%%%%%%%%%%%%%%%%
% Inverting p
%%%%%%%%%%%%%%%%%%%%%%%%%%%%%%%%%%%%%%%%%%%%%%%%%%%%%%%%%%%%
%%%%%%%%%%%%%%%%%%%%%%%%%%%%%%%%%%%%%%%%%%%%%%%%%%%%%%%%%%%%

\subsubsection{Technical variants: the rings $\A_{\dR}^{>q}\left(R,R^{+}\right)$}
\label{subsubsection:Agreaterthanq}

The rings $\A_{\dR}^{q}\left(R,R^{+}\right)$ are not $\xi/p^{q}$-adically complete.
This motivates:

\begin{defn}
  For any $q\in\NN$, $\A_{\dR}^{>q}\left(R,R^{+}\right)$ is the
  completion of $\A_{\dR}^{q}\left(R,R^{+}\right)$, equipped with the
  $\left(p,\ker\theta_{\dR}^{q}\right)$-adic seminorm,
  cf. Definition~\ref{defn:I-adic-seminorm-norm}.
\end{defn}

\begin{remark}
  We think of $\A_{\dR}^{>q}\left( R,R^{+} \right)$
  as functions on an open tubular neighbourhood $U$ around
  the vanishing locus $\left\{\xi=0\right\}$, where $(\xi)=\ker\theta_{\inf}$.
  %%%cf.~\cite[Lemma 6.3]{Sch13pAdicHodge}.
  Intuitively, this neighbourhood has radius $|p|^{q}$,
  that is $U=\left\{x\colon \log_{p}\left(\dist\left(x,\xi\right)\right)> q \right\}$,
  motivating the superscript~$>q$.
\end{remark}

\iffalse %%%
\begin{remark}
  $\A_{\dR}^{q}\left(R,R^{+}\right)$ is not $\xi/p^{q}$-adically complete,
  but the rings $\A_{\dR}^{>q}\left(R,R^{+}\right)$ are,
  by the following Lemma~\ref{lem:AdRgreaterthanqRRplus-Banachalgebra}.
  This allows to argue via the associated graded with respect to the $\xi/p^{q}$-adic filtration,
  which becomes useful at many instances in this article.
\end{remark}
\fi %%%

\begin{lem}\label{lem:AdRgreaterthanqRRplus-Banachalgebra}
  $\A_{\dR}^{>q}\left(R,R^{+}\right)$ is a $W(\kappa)$-Banach algebra for all $q\in\NN$.
\end{lem}

\begin{proof}
  One has to check that the $\left(p,\ker\theta_{\dR}^{q}\right)$-adic completion
  is complete. This follows from~\cite[\href{https://stacks.math.columbia.edu/tag/05GG}{Tag 05GG}]{stacks-project},
  because $\ker\theta_{\dR}^{q}$ is a principal ideal
  by Lemma~\ref{lem:Fontaines-map-for-Ala}.
\end{proof}

\begin{lem}\label{lem:colimitAdRgreaterthanq}
  The $\A_{\dR}^{q}\left(R,R^{+}\right)\to\A_{\dR}^{q+1}\left(R,R^{+}\right)$
  factor canonically through the maps $\A_{\dR}^{q}\left(R,R^{+}\right)\to\A_{\dR}^{>q}\left(R,R^{+}\right)$.
  In particular, we get the canonical isomorphism
  of $W(\kappa)$-ind-Banach algebras
  \begin{equation*}
    \A_{\dR}^{\infty}\left(R,R^{+}\right)
    \isomap\text{``}\varinjlim_{q\in\NN}\text{"}\A_{\dR}^{>q}\left( R,R^{+} \right).
  \end{equation*}
\end{lem}

\begin{proof}
  We have to check that every element in the image of $\left(p,\ker\theta_{\dR}^{q}\right)$
  in $\A_{\dR}^{q+1}\left(R,R^{+}\right)$ is topologically nilpotent.
  By Lemma~\ref{lem:Fontaines-map-for-Ala},
  it suffices to check this for $p$ and $\xi/p^{q}$. This follows because
  $p$ is already topologically nilpotent in $\A_{\inf}\left(R,R^{+}\right)$ and
  $\xi/p^{q}=p \cdot \xi/p^{q+1}\in\A_{\dR}^{q+1}\left(R,R^{+}\right)$.
\end{proof}

\begin{lem}\label{lem:xipq-p-regularseq-inAdRq-recpaper}
  If $q\in\NN_{\geq2}$, then $\xi/p^{q},p\in\A_{\dR}^{q}\left( R,R^{+} \right)$
  is a regular sequence.
\end{lem}

\begin{proof}
  $\xi/p^{q}$ is not a zero-divisor by Lemma~\ref{lem:xi-over-pm-notzerodivisor}
  and the image of $p$ in $\A_{\dR}^{q}\left( R,R^{+} \right)/\left(\xi/p^{q}\right)\cong R^{+}$
  is not a zero-divisor as well, cf. Lemma~\ref{lem:Fontaines-map-for-Ala}.
\end{proof}

\begin{lem}\label{lem:assgr-AdRgreaterthanq-recpaper}
  Let $q\in\NN_{\geq2}$ and equip $\A_{\dR}^{>q}\left( R,R^{+} \right)$
  with the $\left(p,\xi/p^{q}\right)$-adic filtration. We compute
  \begin{equation*}
    \gr\A_{\dR}^{>q}\left( R,R^{+} \right)
    \cong \left(R^{+}/p\right)\left[\sigma\left(p\right),\sigma\left(\frac{\xi}{p^{q}}\right)\right],
  \end{equation*}
  where the principal symbols
  $\sigma\left(p\right)$ of $p$ and $\sigma\left(\xi/p^{q}\right)$ of $\xi/p^{q}$
  are homogenous of degree $1$.
\end{lem}

\begin{proof}
  Equip
  $\A_{\dR}^{q}\left( R,R^{+} \right)$ with the $\left(p,\xi/p^{q}\right)$-adic filtration
  and compute
  \begin{equation*}
    \gr\A_{\dR}^{>q}\left( R,R^{+} \right)
    \cong\gr\A_{\dR}^{q}\left( R,R^{+} \right)
    \cong \left(R^{+}/p\right)\left[\sigma\left(p\right),\sigma\left(\frac{\xi}{p^{q}}\right)\right].
  \end{equation*}
  Here, we used~\cite[Exercise 17.16.a]{Ei95}
  which applies thanks to
  Lemma~\ref{lem:xipq-p-regularseq-inAdRq-recpaper}.  
\end{proof}

The following three Lemma~\ref{lem:Agreaterthanq-multiplybyp-norm-reconstructionpaper},
\ref{lem:onAdR-multbyp-strictinjection-reconstructionpaper}, and
\ref{lem:AdRgreaterthanqptorsionfree} are immedeate consequences of
Lemma~\ref{lem:assgr-AdRgreaterthanq-recpaper}
and~\cite[Chapter I, \S 4.2 page 31-32, Theorem 4(5)]{HuishiOystaeyen1996}.

\begin{lem}\label{lem:Agreaterthanq-multiplybyp-norm-reconstructionpaper}
  For all $q\in\NN_{\geq2}$ and $a\in\A_{\dR}^{>q}\left( R,R^{+} \right)$,
  $\|pa\|=p^{-1}\|a\|$.
\end{lem}

%\begin{proof}
%  This follows from Lemma~\ref{lem:assgr-AdRgreaterthanq-recpaper}.
%\end{proof}

\begin{lem}\label{lem:onAdR-multbyp-strictinjection-reconstructionpaper}
  Let $N\in\NN$ be arbitrary and $q\in\NN_{\geq2}$. Then $\A_{\dR}^{>q}\left(R,R^{+}\right)\to\A_{\dR}^{>q}\left(R,R^{+}\right)$,
  $a\mapsto p^{N}a$ is a strict monomorphism of $W(\kappa)$-Banach modules.
\end{lem}

%\begin{proof}
%  Apply~\cite[Chapter I, \S 4.1, page 31-32, Theorem 4]{HuishiOystaeyen1996}
%  and Lemma~\ref{lem:assgr-AdRgreaterthanq-recpaper}.
%\end{proof}

\begin{lem}\label{lem:AdRgreaterthanqptorsionfree}
  Let $q\in\NN_{\geq 2}$. The ring $\A_{\dR}^{>q}\left( R,R^{+} \right)$ is $p$-torsion free.
\end{lem}

%\begin{proof}
%  We may check that the princial symbol $\sigma(p)\in\gr\A_{\dR}^{>q}\left( R,R^{+} \right)$
%  of $p$ is not a zero-divisor, where we consider the associated graded with respect to the
%  $\left(p,\xi/p^{q}\right)$-adic filtration,
%  cf.~\cite[Chapter I, \S 4.2 page 31-32, Theorem 4(5)]{HuishiOystaeyen1996}.
%  This follows directly from Lemma~\ref{lem:assgr-AdRgreaterthanq-recpaper}.
%\end{proof}

Here is a variant of Lemma~\ref{lem:Fontaines-map-for-Ala}:

\begin{lem}\label{lem:Fontaines-map-for-Alagreatq}
  $\theta_{\inf}\colon\A_{\inf}\left( R , R^{+} \right) \to R^{+}$
  factors through strict epimorphisms
  \begin{equation*}\label{eq:Fontaines-map-for-Ala}
      \theta_{\dR}^{>q}\colon\A_{\dR}^{>q}\left( R , R^{+} \right) \to R^{+} \\
  \end{equation*}
  of $W(\kappa)$-Banach algebras for all $q\in\NN$.
  If $q\geq2$, then their kernels are principal ideals generated by
  $\xi/p^{q}$, which are non-zero divisors.
\end{lem}

\begin{proof}
  The map $\theta_{\dR}^{>q}$ is the completion of bounded linear map
  \begin{equation*}
    \theta_{\dR}^{q}\colon\A_{\dR}^{q}\left( R , R^{+} \right) \to R^{+},
  \end{equation*}
  where $\A_{\dR}^{q}\left( R , R^{+} \right)$ carries the $\left(p,\xi/p^{q}\right)$-adic
  topology and $R^{+}$ is equipped with the $p$-adic topology.
  In particular, we get a commutative diagram
  \begin{equation}\label{cd:Fontaines-map-for-Ala-iota-thetainf-thetalaq}
    \begin{tikzcd}
    \A_{\dR}^{>q}\left( R , R^{+} \right) \arrow{rd}{\theta_{\dR}^{>q}} &  \\
    \A_{\inf}\left( R , R^{+} \right) \arrow{r}{\theta_{\inf}}\arrow{u}{\iota} & R^{+},
    \end{tikzcd}
  \end{equation}
  where $\iota$ is the canonical map.
  Since $\theta_{\inf}$ is a strict epimorphism,
  cf. Lemma~\ref{theta:bd-strict},
  \cite[Proposition 1.1.8]{Sch99} implies
  that $\theta_{\dR}^{>q}$ is a strict epimorphism.  
  To prove the second
  statement, consider
  \begin{equation*}
    0 \longrightarrow
    \A_{\dR}^{>q}\left( R , R^{+} \right) \stackrel{\xi/p^{q}\cdot}{\longrightarrow}
    \A_{\dR}^{>q}\left( R , R^{+} \right) \stackrel{\theta_{\dR}^{>q}}{\longrightarrow}
    R^{+} \longrightarrow
    0.
  \end{equation*}
  If $q\geq 2$, then
  by Lemma~\ref{lem:assgr-AdRgreaterthanq-recpaper},
  its associated graded is
  \begin{align*}
    0 \longrightarrow
    \left(R^{+}/p\right)\left[\sigma\left(p\right),\sigma\left(\frac{\xi}{p^{q}}\right)\right]
     \stackrel{\sigma\left(\xi/p^{q}\right)\cdot}{\longrightarrow}
    &\left(R^{+}/p\right)\left[\sigma\left(p\right),\sigma\left(\frac{\xi}{p^{q}}\right)\right] \\
     &\stackrel{\sigma\left(\xi/p^{q}\right)\mapsto0}{\longrightarrow}
    \left(R^{+}/p\right)\left[\sigma\left(p\right)\right] \longrightarrow
    0
  \end{align*}
  It is exact, thus~\cite[Chapter I, \S 4.2 page 31-32, Theorem 4(5)]{HuishiOystaeyen1996}
  gives the result.
\end{proof}

Here is a variant of Lemma~\ref{lem:Ainf-divisionbyxi-andp}(i):

\begin{lem}\label{lem:Ala-divisionbyxi}
  Fix an $a\in\A_{\dR}^{>q}\left(R,R^{+}\right)$, $q\in\NN_{\geq2}$.
  If $\xi/p^{q}$ divides  $ap$ then $\xi/p^{q}$ divides $a$.
\end{lem}

\begin{proof}
  We get $0=\theta_{\dR}^{>q}\left(ap\right)=\theta_{\dR}^{>q}(a)p\in R^{+}$.
  This gives $\theta_{\dR}^{>q}(a)=0$ and Lemma~\ref{lem:Fontaines-map-for-Alagreatq}
  implies that $\xi/p^{q}$ divides $a$.
\end{proof}

\begin{lem}\label{lem:AdRgreaterthanqUtimesS-isomapHomcontSAdRgreaterthanqU-assumptionsforlemmasatisfied-reconstructionpaper}
  $\A_{\dR}^{>q}\left(R,R^{+}\right)/\left(\xi/p^{q}\right)^{s}$
  has no $p$-power torsion for all $s\geq0$ if $q\geq2$.
\end{lem}

\begin{proof}
  Given $a\in\A_{\dR}^{>q}\left(R,R^{+}\right)$, suppose that $p^{m}a\in\left(\xi/p^{q}\right)^{s}$.
  Then Lemma~\ref{lem:Ala-divisionbyxi} implies that $\left(\xi/p^{q}\right)^{s}$ divides $a$, thus
  the image of $a$ in $\A_{\dR}^{>q}\left(R,R^{+}\right)/\left(\xi/p^{q}\right)^{s}$ is zero.
\end{proof}

Given a Banach ring $S$ and an $S$-Banach algebra $A$,
$|S|$ denotes the underlying ring and $|A|$
denotes the underlying $|S|$-algebra.
We recall that $\A_{\dR}^{>q}\left(R,R^{+}\right)$
carries the $p,\xi/p^{q}$-adic topology.

\begin{defn}\label{defn:Wkappatriv-tildeAdRgreaterthanq}
  $\widetilde{\A}_{\dR}^{>q}\left(R,R^{+}\right)$ is the
  $|W(\kappa)|$-algebra $|\A_{\dR}^{>q}\left(R,R^{+}\right)|$
  with the $\left(\xi/p^{q}\right)$-adic filtration.
\end{defn}

%\cite[\href{https://stacks.math.columbia.edu/tag/090T}{Tag 090T}]{stacks-project}
%implies the following Lemma~\ref{lem:tildeAdRgreaterthanq-is-Banach-reconstructionpaper}.
%
%\begin{lem}\label{lem:tildeAdRgreaterthanq-is-Banach-reconstructionpaper}
%  $\widetilde{\A}_{\dR}^{>q}\left(R,R^{+}\right)$ is a
%  $W(\kappa)^{\triv}$-Banach algebra.
%\end{lem}

\begin{lem}\label{lem:grAdRgreaterthanq-xioverpadicfiltration-reconstructionpaper}
  Let $q\in\NN_{\geq2}$
  and consider $\widetilde{\A}_{\dR}^{>q}\left(R,R^{+}\right)$.
  \begin{itemize}
    \item[(i)] The topology induced by the filtration is separated and complete.
    \item[(ii)] Fontaine's map induces the isomorphism
  \begin{equation*}
    \gr\widetilde{\A}_{\dR}^{>q}\left(R,R^{+}\right)\cong R^{+}\left[\sigma\left(\frac{\xi}{p^{q}}\right)\right],
  \end{equation*}
  \end{itemize}
  where the principal symbol $\sigma\left(\xi/p^{q}\right)$
  of $\xi/p^{q}$ is homogenous of degree one.
\end{lem}

\begin{proof}
  See~\cite[\href{https://stacks.math.columbia.edu/tag/090T}{Tag 090T}]{stacks-project}
  for (i). (ii) follows from Lemma~\ref{lem:Fontaines-map-for-Alagreatq}
  and~\cite[Exercise 17.16.a]{Ei95}.
\end{proof}

\begin{lem}\label{lem:AdRgreaterthanq-toAdRgreaterthanqplus1-mono}
  The morphisms
  $\A_{\dR}^{>q}\left(R,R^{+}\right)\to\A_{\dR}^{>q+1}\left(R,R^{+}\right)$
  are monomorphisms for all $q\geq 2$.
\end{lem}

\begin{proof}
  Similarly, we may consider
  $\widetilde{\A}_{\dR}^{>q}\left(R,R^{+}\right)\to\widetilde{\A}_{\dR}^{>q+1}\left(R,R^{+}\right)$,
  where $\widetilde{\A}_{\dR}^{>h}\left(R,R^{+}\right)$ carries the
  $\xi/p^{h}$-adic filtration for $h\in\left\{q,q+1\right\}$.  
  By Lemma~\ref{lem:grAdRgreaterthanq-xioverpadicfiltration-reconstructionpaper},
  which requires $q\geq2$,
  the induced morphism between the associated gradeds
  is the map
  \begin{equation*}
    R^{+}\left[\sigma\left(\frac{\xi}{p^{q}}\right)\right]
    \to
    R^{+}\left[\sigma\left(\frac{\xi}{p^{q+1}}\right)\right],
    \sigma\left(\frac{\xi}{p^{q}}\right)
    \mapsto
    p\sigma\left(\frac{\xi}{p^{q+1}}\right).
  \end{equation*}
  of $R^{+}$-algebras. It is injective.
  Now apply~\cite[Chapter I, \S 4.1, page 31-32, Theorem 4]{HuishiOystaeyen1996}.
\end{proof}

\begin{lem}\label{lem:galois-cohomology-of-Bla-pidealclosed-inAlaq}
  Let $N\in\NN$ be arbitrary and $q\in\NN_{\geq2}$. Then
  $\widetilde{\A}_{\dR}^{>q}\left(R,R^{+}\right)\to\widetilde{\A}_{\dR}^{>q}\left(R,R^{+}\right)$,
  $a\mapsto p^{N}a$ is a strict monomorphism. %%%of $W(\kappa)$-Banach modules.
\end{lem}

\begin{proof}
  This follows from~\cite[Chapter I, \S 4.1, page 31-32, Theorem 4]{HuishiOystaeyen1996}
  and Lemma~\ref{lem:grAdRgreaterthanq-xioverpadicfiltration-reconstructionpaper}.
  \iffalse %%% alternative proof
  The map $a\mapsto p^{N}a$ is injective by Lemma~\ref{lem:onAdR-multbyp-strictinjection-reconstructionpaper}.
  It remains to show that its image is closed.
  
  Equip $\widetilde{\A}_{\dR}^{>q}\left(R,R^{+}\right)$ with the $\xi/p^{q}$-adic filtration.
  Lemma~\ref{lem:grAdRgreaterthanq-xioverpadicfiltration-reconstructionpaper} implies
  that the principal symbol of $p$ in $\gr\widetilde{\A}_{\dR}^{>q}\left(R,R^{+}\right)$
  is a non-zerodivisor, homogenous of degree zero. That is,
  $\|pa\|=\|a\|$ for all $a\in \widetilde{\A}_{\dR}^{>q}\left(R,R^{+}\right)$, where
  $\|-\|$ denotes the $\xi/p^{q}$-adic norm. 
  By induction, $\|p^{i}a\|=\|a\|$.
  %Thus Lemma~\ref{lem:principleidealclosedinBanachring-reconstructionpaper}
  %applies, giving the desired result.
  
  Now we compute that $\left(p^{i}\right)\subseteq\widetilde{\A}_{\dR}^{>q}\left(R,R^{+}\right)$ is closed.
  Consider a sequence of $a_{n}\in\left(p^{i}\right)$
  converging to an $a\in \widetilde{\A}_{\dR}^{>q}\left(R,R^{+}\right)$ for $n\to\infty$. We may write $a_{n}=b_{n}p^{i}$
  for all $n$. But then
  \begin{equation*}
    \|b_{n}-b_{m}\|=\|\left(b_{n}-b_{m}\right)p^{i}\|=\|a_{n}-a_{m}\|
  \end{equation*}
  for all $n$ and $m$, by the
  above. Thus $\left(b_{n}\right)_{n}$ is a Cauchy sequence,
  and it converges to an element $b$.
  \begin{equation*}
    a
    = \lim_{n\to\infty} a_{n}
    = \lim_{n\to\infty} b_{n}p^{i}
    = bp^{i}
    \in\left(p^{i}\right)
  \end{equation*}
  follows, that is $\left(p^{i}\right)$ is closed.
  \fi %%%
\end{proof}

%%%%%%%%%%%%%%%%%%%%%%%%%%%%%%%%%%%%%%%%%%%%%%%%%%%%%%%%%%%%
%%%%%%%%%%%%%%%%%%%%%%%%%%%%%%%%%%%%%%%%%%%%%%%%%%%%%%%%%%%%
% Inverting p
%%%%%%%%%%%%%%%%%%%%%%%%%%%%%%%%%%%%%%%%%%%%%%%%%%%%%%%%%%%%
%%%%%%%%%%%%%%%%%%%%%%%%%%%%%%%%%%%%%%%%%%%%%%%%%%%%%%%%%%%%

\subsection{Inverting $p$}
\label{subsubsection:invertingp-solutionpaper}

%Apply Lemma~\ref{lem:monoid-on-tensorproduct} to get
Define the seminormed $k_{0}$-algebra
\begin{equation*}
  \BB_{\inf}\left( R,R^{+} \right)
  :=\A_{\inf}\left( R,R^{+} \right)\otimes_{W(\kappa)}k_{0}
\end{equation*}
and the $k_{0}$-Banach algebra
\begin{equation*}
  \widehat{\BB}_{\inf}\left( R,R^{+} \right)
  :=\A_{\inf}\left( R,R^{+} \right)\widehat{\otimes}_{W(\kappa)}k_{0}.
\end{equation*}

\begin{remark}
  $\BB_{\inf}\left( R,R^{+} \right)$ is not complete, as it
  does not contain $\sum_{n\geq0}\xi^{n+1}/p^{n}$.
\end{remark}

Similarly, we introduce the $\widehat{\BB}_{\inf}\left( R,R^{+} \right)$-Banach algebras
\begin{equation*}
  \BB_{\dR}^{q,+}\left( R,R^{+} \right):=\A_{\dR}^{q}\left( R,R^{+} \right)
  \widehat{\otimes}_{W(\kappa)}k_{0}
\end{equation*}
for all $q\in\NN$. For $q=\infty$, we have the
$\widehat{\BB}_{\inf}\left( R,R^{+} \right)$-ind-Banach algebras
\begin{equation*}
  \BB_{\dR}^{\infty,+}\left( R,R^{+} \right)
  :=\A_{\dR}^{\infty}\left( R,R^{+} \right)\widehat{\otimes}_{W(\kappa)}k_{0}
  =\A_{\dR}^{\dag}\left( R,R^{+} \right)\widehat{\otimes}_{W(\kappa)}k_{0},
\end{equation*}
cf. Notation~\ref{notation:AdRdagger}.
Write $\BB_{\dR}^{\dag,+}\left( R,R^{+} \right):=\BB_{\dR}^{\infty,+}\left( R,R^{+} \right)$
and note that
\begin{equation*}
  \BB_{\dR}^{\dag,+}\left( R,R^{+} \right)
  =\text{``}\varinjlim_{q\in\NN}\text{"}\BB_{\dR}^{q}\left( R,R^{+} \right).
\end{equation*}

\begin{defn}\label{defn:relative-la-period-ring}
  The $\widehat{\BB}_{\inf}\left( R,R^{+} \right)$-ind-Banach algebra
  $\BB_{\dR}^{\dag,+}\left( R,R^{+} \right)$ is the \emph{relative positive overconvergent de Rham period ring}.
  Whenever $\left( R, R^{+}\right)=\left(K,K^{+}\right)$ is a well-understood perfectoid field,
  we refer to $B_{\dR}^{\dag,+}:=\BB_{\dR}^{\dag.+}\left(K,K^{+}\right)$
  as the \emph{positive overconvergent de Rham period ring}.
\end{defn}

\begin{lem}\label{lem:Blaqplus-local-complete-at-xi}
  The morphisms
  \begin{equation*}
    \widehat{\BB}_{\inf}\left( R,R^{+} \right)\left\<\frac{\xi}{p^{q}}\right\>
    \stackrel{\cong}{\longrightarrow}\BB_{\dR}^{q,+}\left( R,R^{+} \right)
  \end{equation*}
  are isomorphisms of $\widehat{\BB}_{\inf}\left( R,R^{+} \right)$-Banach algebras
  for every $q\in\NN$. They are a isomorphisms of
  $\widehat{\BB}_{\inf}\left( R,R^{+} \right)$-ind-Banach algebras
  for $q=\infty$.
\end{lem}

\begin{proof}
  We may assume $q<\infty$ without loss of generality.
  Compute
  \begin{align*}
    &\widehat{\BB}_{\inf}\left( R,R^{+} \right)\left\<\frac{\xi}{p^{q}}\right\> \\
    &=\coker\left(\widehat{\BB}_{\inf}\left( R,R^{+} \right)\left\<\frac{\zeta}{p^{q}}\right\>
    \stackrel{\xi-p^{q}\frac{\zeta}{p^{q}}}{\longrightarrow}
    \widehat{\BB}_{\inf}\left( R,R^{+} \right)\left\<\frac{\zeta}{p^{q}}\right\>\right) \\
    &\cong\coker\left(\A_{\inf}\left( R,R^{+} \right)\left\<\frac{\zeta}{p^{q}}\right\>
    \stackrel{\xi-p^{q}\frac{\zeta}{p^{q}}}{\longrightarrow}
    \A_{\inf}\left( R,R^{+} \right)\left\<\frac{\zeta}{p^{q}}\right\>\right)
    \widehat{\otimes}_{W(\kappa)}k_{0} \\
    &\cong\A_{\inf}\left( R,R^{+} \right)\left\<\frac{\xi}{p^{q}}\right\>
    \widehat{\otimes}_{W(\kappa)}k_{0} \\
    &\cong\BB_{\dR}^{q,+}\left(R,R^{+}\right),
  \end{align*}
  where we have used Lemma~\ref{lem:Alam-local-complete-at-xi} in the last step.
\end{proof}

Recall Fontaine's maps $\theta_{\inf}$, $\theta_{\dR}^{q}$
for all $q\in\NN$, and $\theta_{\dR}^{\dag}$. They induce morphisms
\begin{equation*}%\label{eq:fontainesmaps-after-inverting-p}
  \begin{split}
    \widehat{\vartheta}_{\inf}\colon
    \widehat{\BB}_{\inf}\left( R,R^{+} \right)\stackrel{\theta_{\inf}\widehat{\otimes}_{W(\kappa)}\id_{k_{0}}}{\longrightarrow}
    R^{+}\widehat{\otimes}_{W(\kappa)}k_{0}
    \isomap R \\
    \vartheta_{\dR}^{q,+}\colon
    \BB_{\dR}^{q,+}\left( R,R^{+} \right)\stackrel{\theta_{\dR}^{q}\widehat{\otimes}_{W(\kappa)}\id_{k_{0}}}{\longrightarrow}
    R^{+}\widehat{\otimes}_{W(\kappa)}k_{0}
    \isomap R, \\
    \vartheta_{\dR}^{\dag,+}\colon
    \BB_{\dR}^{\dag,+}\left( R,R^{+} \right)\stackrel{\theta_{\dR}^{\dag}\widehat{\otimes}_{W(\kappa)}\id_{k_{0}}}{\longrightarrow}
    R^{+}\widehat{\otimes}_{W(\kappa)}k_{0}
    \isomap R
  \end{split}
\end{equation*}
of $k_{0}$-Banach, respectively $k_{0}$-ind-Banach algebras.
We refer to them again as \emph{Fontaine's maps}.

\begin{lem}\label{lem:fontainesmaps-after-inverting-p}
  The maps
  $\widehat{\vartheta}_{\inf}$, $\vartheta_{\dR}^{q,+}$, and $\vartheta_{\dR}^{\dag,+}$,
  are strict epimorphisms.
\end{lem}

\begin{proof}
  We have to check that the maps
  $\theta_{\inf}\widehat{\otimes}_{W(\kappa)}\id_{k_{0}}$,
  $\theta_{\dR}^{q,+}\widehat{\otimes}_{W(\kappa)}\id_{k_{0}}$, and
  $\theta_{\dR}^{\dag,+}\widehat{\otimes}_{W(\kappa)}\id_{k_{0}}$ are strict epimorphisms.
  But the completed tensor product preserves strict epimorphisms, cf.~\cite[Lemma 3.7]{BBKFrechetModulesDescent}.
  Now apply Corollary~\ref{theta:bd-strict} and Lemma~\ref{lem:Fontaines-map-for-Ala}.
\end{proof}

We spend the remainder of \S\ref{subsubsection:invertingp-solutionpaper} on the proof of the following result.

\begin{prop}\label{prop:BdRdagplus-bornology-countable-basis-reconstructionpaper}
  The canonical maps
  $\BB_{\dR}^{q,+}\left(R,R^{+}\right)\to\BB_{\dR}^{q+1,+}\left(R,R^{+}\right)$
  are injective if $q\geq 2$.
\end{prop}

We collect a few lemmata before we prove Proposition~\ref{prop:BdRdagplus-bornology-countable-basis-reconstructionpaper}.

\begin{lem}\label{lem:hBinf-multiplybyxi-norm1overp-reconstructionpaper}
  The norm on $\widehat{\BB}_{\inf}\left(R,R^{+}\right)$ is equivalent to a norm
  with the following property:
  For all $f\in\widehat{\BB}_{\inf}\left(R,R^{+}\right)$, (i) $\|f\xi\|=\|f\|p^{-1}$
  and (ii) $\|fp\|=\|f\|p^{-1}$.
\end{lem}

\begin{proof}
  Set $B:=\BB_{\inf}\left(R,R^{+}\right)$
  and $A:=\A_{\inf}\left(R,R^{+}\right)$.  
  By continuity, it suffices to check the statement of Lemma~\ref{lem:hBinf-multiplybyxi-norm1overp-reconstructionpaper}
  for $f\in B$. By Lemma~\ref{lem:localisation-normedotimes},
  we may equip $B$ with the norm
  \begin{equation*}
    \|f\|=\inf\left\{ \frac{\|a\|}{p^{r}} \colon \text{$a\in A$ and $f=\frac{a}{p^{r}}$}\right\}.
  \end{equation*}
  Firstly, we check (i). We have
  \begin{equation*}
    \|f\xi\|=\inf\left\{ \frac{\|c\|}{p^{r}} \colon \text{$c\in A$ and $f\xi=\frac{c}{p^{r}}$}\right\}.
  \end{equation*}
  We compare the indexing sets of both infinums as follows: We claim that
  \begin{equation}\label{eq:hBinf-multiplybyxi-norm1overp-compareindexsets-reconstructionpaper}
    \left\{ \left( a , p^{r} \right) \colon \text{$a\in A$ and $f=\frac{a}{p^{r}}$} \right\}
    \to\left\{ \left( c , p^{r} \right) \colon \text{$c\in A$ and $f\xi=\frac{c}{p^{r}}$} \right\},
    \left( a , p^{r} \right) \mapsto \left( a\xi , p^{r} \right)
  \end{equation}
  is a bijection. Firstly, it is injective because
  $\left( a_{1}\xi , p^{r_{1}} \right)=\left( a_{2}\xi , p^{r_{2}} \right)$ implies
  $a_{1}=a_{2}$ and $r_{1}=r_{2}$, as $\xi$ is not a zero-divisor, cf.~\cite[Lemma 6.3]{Sch13pAdicHodge}.
  To show surjectivity, fix an element $\left( c , p^{r} \right)$ codomain. Then
  $c/p^{r}\in\ker\vartheta_{\inf}$, where
  %\begin{equation*}
    $\vartheta_{\inf}:=\theta_{\inf}\otimes_{W(\kappa)}\id_{k_{0}}\colon B\to R$
  %\end{equation*}
  and $\theta\colon A=\A_{\inf}\left(R,R^{+}\right)\to R^{+}$ is Fontaine's map.
  Then $c\in\ker\theta_{\inf}$, because $1/p^{r}\in R$ is not a zero-divisor, thus
  $c\in(\xi)$ by~\cite[Lemma 6.3]{Sch13pAdicHodge}. Thu
  $\left( c , p^{r} \right)$ lies in the
  image of~(\ref{eq:hBinf-multiplybyxi-norm1overp-compareindexsets-reconstructionpaper}).
  We have thus checked that~(\ref{eq:hBinf-multiplybyxi-norm1overp-compareindexsets-reconstructionpaper})
  is bijective. This implies the second equality in the computation
  \begin{align*}
    \|f\xi\|
    &=\inf\left\{ \frac{\|c\|}{p^{r}} \colon \text{$c\in A$ and $f\xi=\frac{c}{p^{r}}$}\right\} \\
    &=\inf\left\{ \frac{\|a\xi\|}{p^{r}} \colon \text{$a\in A$ and $f=\frac{a}{p^{r}}$}\right\} \\
    &=\inf\left\{ \frac{\|a\|}{p^{r}} \colon \text{$a\in A$ and $f=\frac{a}{p^{r}}$}\right\}p^{-1}
    =\|f\|p^{-1}.
  \end{align*}
  The second equality in this computation follows from Lemma~\ref{lem:Ainf-normxi-normp}.
  We have thus proven (i).
  
  Proceed similarly in order to prove (ii):
  \begin{align*}
    \|fp\|
    &=\inf\left\{ \frac{\|a\|}{p^{r}} \colon \text{$a\in A$ and $fp=\frac{a}{p^{r}}$}\right\} \\
    &=\inf\left\{ \frac{\|a\|}{p^{r-1}} \colon \text{$a\in A$ and $f=\frac{a}{p^{r-1}}$}\right\}p^{-1}
    =\|f\|p^{-1}.
  \end{align*}
  We have thus finished the proof of Lemma~\ref{lem:hBinf-multiplybyxi-norm1overp-reconstructionpaper}.
\end{proof}

\begin{lem}\label{lem:hBBinfzetaoverpqmodxi-zeta-isBdRq-reconstructionpaper}
  For every $q\in\NN_{\geq2}$, we have the canonical isomorphism
  \begin{equation*}
    \widehat{\BB}_{\inf}\left(R,R^{+}\right)\left\<\frac{\zeta}{p^{q}}\right\>/
    \left(\xi-\zeta\right)
    \isomap
    \BB_{\dR}^{q,+}\left(R,R^{+}\right)
  \end{equation*}
  of seminormed $k_{0}$-vector spaces.
\end{lem}

\begin{proof}
  Lemma~\ref{lem:idealpmzeta-xi-closed-in-Ainflanglerzetaangle}
  says that 
  \begin{equation*}
    p^{q}\eta-\xi\colon\A_{\inf}\left(R,R^{+}\right)\left\<\eta\right\>\to\A_{\inf}\left(R,R^{+}\right)\left\<\eta\right\>
  \end{equation*}
  is a strict monomorphism with cokernel
  $\A_{\dR}^{q}\left(R,R^{+}\right)$. Because both
  $\A_{\inf}\left(R,R^{+}\right)\left\<\eta\right\>$
  and $\A_{\dR}^{q}\left(R,R^{+}\right)$ are $p$-torsion free,
  cf. Lemma~\ref{AdR-ptorsionfree-reconstructionpaper}, thus
  Lemma~\ref{lem:widehatotimesFcircF-preserves-strictmonos}
  implies that
  \begin{equation*}
  p^{q}\eta-\xi
    \colon
    \widehat{\BB}_{\inf}\left(R,R^{+}\right)\left\<\eta\right\>
    \to
    \widehat{\BB}_{\inf}\left(R,R^{+}\right)\left\<\eta\right\>
  \end{equation*}
  is a strict monomorphism.
  Substitute $\zeta/p^{q}$ for $\eta$ to find that
  \begin{equation*}
    \xi-\zeta
    \colon
    \widehat{\BB}_{\inf}\left(R,R^{+}\right)\left\<\frac{\zeta}{p^{q}}\right\>
    \to
    \widehat{\BB}_{\inf}\left(R,R^{+}\right)\left\<\frac{\zeta}{p^{q}}\right\>.
  \end{equation*}
  is strict monomorphism. By the proof of Lemma~\ref{lem:Blaqplus-local-complete-at-xi},
  $\BB_{\dR}^{q,+}\left(R,R^{+}\right)$ is the cokernel of this map.
  The strictness implies Lemma~\ref{lem:hBBinfzetaoverpqmodxi-zeta-isBdRq-reconstructionpaper}.
\end{proof}

\begin{lem}\label{keylem:prop:BdRdagplus-bornology-countable-basis-reconstructionpaper}
  Fix $q\in\NN_{\geq2}$ and consider a sequence
  $\left(g_{\alpha}\right)_{\alpha\in\NN_{\geq1}}\subseteq\widehat{\BB}_{\inf}\left(R,R^{+}\right)$
  such that
  \begin{itemize}
    \item[(i)] $p^{(q+1)\alpha}g_{\alpha}\to 0$ for $\alpha\to\infty$, and
    \item[(i)] $p^{q\alpha}f_{\alpha}\to0$ for $\alpha\to\infty$. Here,
      $f_{\alpha}:=g_{\alpha}\xi-g_{\alpha-1}$.
  \end{itemize}
  Then $p^{q\alpha}g_{\alpha}\to 0$ for $\alpha\to\infty$.
\end{lem}

\begin{proof}
  Write $\widehat{B}:=\widehat{\BB}_{\inf}\left(R,R^{+}\right)$.
  Throughout this proof, we fix the norm on $\widehat{B}$ as in
  Lemma~\ref{lem:hBinf-multiplybyxi-norm1overp-reconstructionpaper}.
  We use the identities $\|f\xi\|=\|f\|p^{-1}$ and $\|fp\|=\|f\|p^{-1}$
  for all $f\in\widehat{B}$ without further reference.

  Set $I:=\left\{ i\in\NN_{\geq1} \colon \|g_{i}\xi\| \neq \|g_{i-1}\| \right\}$.
  To streamline this proof of Lemma~\ref{keylem:prop:BdRdagplus-bornology-countable-basis-reconstructionpaper},
  we introduce the notation
  $g_{0}:=0$ and $I_{0}:=I\cup\{0\}$. One reason for this is that $I_{0}\neq\emptyset$,
  whilst $I$ can be empty.
  
  We start with the following observation:
  \begin{equation}\label{eq:keylem:prop:BdRdagplus-bornology-countable-basis-obs1-reconstructionpaper}
    \left( p^{qi} g_{i}\right)_{i\in I} \text{ is a zero sequence}.
  \end{equation}
  Suppose it is not a zero sequence. But then, for $i\in I$,
  \begin{align*}
    \|p^{qi}g_{\alpha}\|
    &\leq p\max\left\{ \| p^{qi} g_{i}\|p^{-1} , \|p^{qi} g_{i-1}\| \right\} \\
    &\stackrel{\text{\ref{lem:hBinf-multiplybyxi-norm1overp-reconstructionpaper}}}{=}
      p\max\left\{ \|p^{qi} g_{i}\xi\| , \|p^{qi} g_{i-1}\| \right\} \\
    &=p\|p^{qi} f_{i}\|\to0 \text{ for $i\to\infty$},
  \end{align*}
  contradaction. Thus~(\ref{eq:keylem:prop:BdRdagplus-bornology-countable-basis-obs1-reconstructionpaper})
  holds. Its proof also implies
  \begin{equation}\label{eq:keylem:prop:BdRdagplus-bornology-countable-basis-obs2-reconstructionpaper}
    \left( p^{qi} g_{i-1}\right)_{i\in I_{0}} \text{ is a zero sequence}.
  \end{equation}
  Next, we prove the following claim: Given $\alpha,\beta\in\NN_{\geq1}$
  such that $\alpha\geq\beta$ and $[\beta,\alpha]\cap I=\emptyset$,
  \begin{equation}\label{eq:keylem:prop:BdRdagplus-bornology-countable-basis-Afiniteclaim1-reconstructionpaper}
    \|g_{\alpha}\|=p^{\alpha-\beta}\|g_{\beta}\|.
  \end{equation}
  We prove~(\ref{eq:keylem:prop:BdRdagplus-bornology-countable-basis-Afiniteclaim1-reconstructionpaper})
  via induction on $\alpha$ for fixed $\beta$. Everything is clear for $\alpha=\beta$. Now suppose
  (\ref{eq:keylem:prop:BdRdagplus-bornology-countable-basis-Afiniteclaim1-reconstructionpaper})
  holds for fixed $\alpha\geq\beta$, and we check it for $\alpha+1\not\in I$ and $\beta$.
  Then, because $\alpha+1\not\in I$,
  \begin{equation*}
    \|g_{\alpha+1}\|
    \stackrel{\text{\ref{lem:hBinf-multiplybyxi-norm1overp-reconstructionpaper}}}{=}
    p\|g_{\alpha+1}\xi\|
    =p\|g_{\alpha}\|
    =pp^{\alpha-\beta}\|g_{\beta}\|
    =p^{(\alpha+1)-\beta}\|g_{\beta}\|,
  \end{equation*}
  Thus~(\ref{eq:keylem:prop:BdRdagplus-bornology-countable-basis-Afiniteclaim1-reconstructionpaper})
  holds by induction.
  
  We can now prove
  Lemma~\ref{keylem:prop:BdRdagplus-bornology-countable-basis-reconstructionpaper}.
  Distinguish two cases: (i) $A$ is finite and (ii) $A$ is infinite.
  \begin{itemize}
    \item[(i)] If $A$ is finite, we can fix an $\beta:=\sup A+1\in\NN$. Then, for all $\alpha\geq\beta$,
    \begin{equation*}
      \|p^{q\alpha}g_{\alpha}\|
      =p^{-q\alpha}\|g_{\alpha}\|
      \stackrel{\text{(\ref{eq:keylem:prop:BdRdagplus-bornology-countable-basis-Afiniteclaim1-reconstructionpaper})}}{=}
        p^{-q\alpha+\alpha-\beta}\|g_{\beta}\|
      \to 0 \text{ for } \alpha\to\infty,
    \end{equation*}
    as desired. Here we used $q\geq 2$.
    \item[(ii)] Let $A$ be infinite and fix $\epsilon>0$.
      By~(\ref{eq:keylem:prop:BdRdagplus-bornology-countable-basis-obs1-reconstructionpaper})
      and~(\ref{eq:keylem:prop:BdRdagplus-bornology-countable-basis-obs2-reconstructionpaper}),
      there exists an $N\in\NN$ such that
      $\|p^{qi}g_{i}\|<\epsilon$ and $\|p^{qi} g_{i-1}\|<\epsilon$ for all $i\in I_{0}$ such that $i\geq N$.
      Now consider an arbitrary $\alpha\in\NN$ such that $\alpha\geq N$.
      Pick $i\in\NN$ which is minimal with respect to the following property:
      $i\in I_{0}$ and $i>\alpha$. Then $qi-(i-1)=(q-1)i+1\geq (q-1)\alpha$, thus
      $qi-(i-1)+\alpha\geq q\alpha$. Therefore,
      \begin{equation}\label{eq:keylem:prop:BdRdagplus-bornology-countable-basis-randomintegerestimate-reconstructionpaper}
        -qi + (i-1 - \alpha) \leq -q\alpha
      \end{equation}
      Now compute
      \begin{equation*}
        \|p^{q\alpha}g_{\alpha}\|
        =p^{-q\alpha}\|g_{\alpha}\|
        \stackrel{\text{(\ref{eq:keylem:prop:BdRdagplus-bornology-countable-basis-randomintegerestimate-reconstructionpaper})}}{\leq}
          p^{-qi + (i-1 - \alpha)}\|g_{\alpha}\|
        \stackrel{\text{(\ref{eq:keylem:prop:BdRdagplus-bornology-countable-basis-Afiniteclaim1-reconstructionpaper})}}{=}
          p^{-qi}\|g_{i-1}\|
          =\|p^{qi}g_{i-1}\|<\epsilon       
      \end{equation*}
      Here,~(\ref{eq:keylem:prop:BdRdagplus-bornology-countable-basis-Afiniteclaim1-reconstructionpaper})
      applies by the minimality of $i$.
      
      We have thus checked that $\|p^{q\alpha}g_{\alpha}\|<\epsilon$
      for all $\alpha\in\NN$ such that $\alpha\geq N$.
      %That is, $p^{q\alpha}g_{\alpha}\to 0$ for $\alpha\to\infty$.
  \end{itemize}
  This completes the proof of Lemma~\ref{keylem:prop:BdRdagplus-bornology-countable-basis-reconstructionpaper}.
\end{proof}

\begin{proof}[Proof of Proposition~\ref{prop:BdRdagplus-bornology-countable-basis-reconstructionpaper}]
  Set $\widehat{B}:=\widehat{\BB}_{\inf}\left(R,R^{+}\right)$. Denote the canonical map
  $\widehat{B}\left(R,R^{+}\right)\left\<\zeta / p^{q}\right\>
    \to\widehat{B}\left(R,R^{+}\right)\left\<\zeta / p^{q+1}\right\>$
  by $j^{q}$. We use the description of the rings as in
  Lemma~\ref{lem:hBBinfzetaoverpqmodxi-zeta-isBdRq-reconstructionpaper}
  to see that the injectivity of $\BB_{\dR}^{q,+}\left(R,R^{+}\right)\to\BB_{\dR}^{q+1,+}\left(R,R^{+}\right)$
  is equivalent to the following:
  For every $f\in \widehat{B}\left\<\zeta/p^{q}\right\>$,
  \begin{equation}\label{eq:BdRdagplus-bornology-countable-basis-whattoshow-reconstructionpaper}
    j^{q}(f)\in(\xi-\zeta) \implies f\in(\xi-\zeta). 
  \end{equation}
  In order to check~(\ref{eq:BdRdagplus-bornology-countable-basis-whattoshow-reconstructionpaper}),
  we fix $f=\sum_{\alpha\geq0}f_{\alpha}\zeta^{\alpha}\in\widehat{B}\left\<\zeta/p^{q}\right\>$
  such that $j^{q}(f)\in(\xi-\zeta)$. We may then write $j^{q}(f)=g(\xi-\zeta)$ for some
  $g=\sum_{\alpha\geq0}g_{\alpha}\zeta^{\alpha}\in\widehat{B}\left\<\zeta/p^{q+1}\right\>$.
  A routine computation implies
  \begin{equation*}
    f_{0}=g_{0}\xi\, \text{ and }\, f_{\alpha}=g_{\alpha}\xi-g_{\alpha-1}\text{ for all $\alpha\geq1$.}
  \end{equation*}
  Now apply Lemma~\ref{keylem:prop:BdRdagplus-bornology-countable-basis-reconstructionpaper}
  to find $g\in\widehat{B}\left\<\zeta/p^{q+1}\right\>$. Thus $f=g(\xi-\zeta)\in (\xi-\zeta)$,
  giving~(\ref{eq:BdRdagplus-bornology-countable-basis-whattoshow-reconstructionpaper}).
\end{proof}

%%%%%%%%%%%%%%%%%%%%%%%%%%%%%%%%%%%%%%%%%%%%%%%%%%%%%%%%%%%%
%%%%%%%%%%%%%%%%%%%%%%%%%%%%%%%%%%%%%%%%%%%%%%%%%%%%%%%%%%%%
% Inverting t
%%%%%%%%%%%%%%%%%%%%%%%%%%%%%%%%%%%%%%%%%%%%%%%%%%%%%%%%%%%%
%%%%%%%%%%%%%%%%%%%%%%%%%%%%%%%%%%%%%%%%%%%%%%%%%%%%%%%%%%%%

\subsection{Inverting $t$, Fontaine's $2\pi i$}\label{subsubsec:invertt}

\begin{notation}\label{notation:Kadmtscompativlesystemetc}
Assume that $K$ admits a compatible system $1,\zeta_{p},\zeta_{p^{2}},\dots$
of primitive $p$th power roots of unity, that is $\zeta_{p^{n+1}}^{p}=\zeta_{p^{n}}$ for all $n\in\NN$.
Fix such a system. Since each $\zeta_{p^{n}}$ satisfies the monic polynomial
$L^{p^{n}}-1\in K^{+}[L]$ and $K^{+}$ is integrally closed, this system
$\left\{ \zeta_{p^{n}}\right\}_{n\in\NN}$ lies in $K^{+}$. Thus
$\epsilon:=\left(1,\zeta_{p},\zeta_{p^{2}},\dots\right)\in K^{\flat+}$.
Write
$\mu:=\left[\epsilon\right]-1\in\ker\theta_{\inf}\subseteq A_{\inf}:=\A_{\inf}\left( K,K^{+}\right)$ and
\begin{equation}\label{eq:Kadmtscompativlesystemetc-defnxi}
  \xi
  := 1 + \left[\epsilon^{1/p}\right] + \left[\epsilon^{1/p}\right]^{2} + \dots + \left[\epsilon^{1/p}\right]^{p-1}
  =\frac{\mu}{\varphi^{-1}\left(\mu\right)}.
\end{equation}
Here, $\varphi$ is the Frobenius on $A_{\inf}$ and $\xi$ generates $\ker\theta_{\inf}$,
cf.~\cite[Example 3.16 and \S 3.3]{BhattMorrowScholze2018}.
\end{notation}

Write $A_{\dR}^{1}:=\A_{\dR}^{1}\left( K,K^{+}\right)$.

\begin{defn}\label{defn:t}
  The following power series convergences because $\mu\in\ker\theta_{\inf}$:
  \begin{align*}
    t:=\log([\epsilon])
    =\log(1 + ([\epsilon]-1))
    &=\sum_{\alpha\geq 1} (-1)^{\alpha+1}\frac{([\epsilon]-1)^{\alpha}}{\alpha} \\
    &=\sum_{\alpha\geq 1} (-1)^{\alpha+1}\frac{p^{\alpha}}{\alpha}\left(\frac{[\epsilon]-1}{p}\right)^{\alpha}
    \in A_{\dR}^{1}.
  \end{align*}
  Here, we are also using that $\alpha$ divides $p^{\alpha}$ in $n$ in $\ZZ_{p}$.
\end{defn}

Following Notation~\ref{notation:indbanach-invert-r}, we have the
$A_{\dR}^{1}$-ind-Banach algebra
  \begin{align*}
    \text{``}\varinjlim_{t\times}\text{"}A_{\dR}^{1}
    &:=
    \text{``}\varinjlim\text{"}\left(
      \dots\stackrel{t\times}{\longrightarrow}A_{\dR}^{1}
        \stackrel{t\times}{\longrightarrow} A_{\dR}^{1}
        \stackrel{t\times}{\longrightarrow}\dots\right).
  \end{align*}
  \iffalse %%% not needed anymore
  The multiplication is
  \begin{equation*}
    \text{``}\varinjlim_{t\times}\text{"}A_{\dR}^{1}
    \widehat{\otimes}_{W(\kappa)}\text{``}\varinjlim_{t\times}\text{"}A_{\dR}^{1}
    \cong
    \text{``}\varinjlim_{2t\times}\text{"}A_{\dR}^{1}\widehat{\otimes}_{W(\kappa)}A_{\dR}^{1}
    \stackrel{\text{``}\varinjlim\text{"}_{2t\times}\mu}{\longrightarrow}
    \text{``}\varinjlim_{2t\times}\text{"}A_{\dR}^{1}
    \cong\text{``}\varinjlim_{t\times}\text{"}A_{\dR}^{1}
  \end{equation*}
  where $\mu$ is the multiplication on $A_{\dR}^{1}$.
  The unit is induced by the unit $W(\kappa)\to A_{\dR}^{1}$.
  \fi %%% comment ends

%Use Lemma~\ref{lem:monoid-on-tensorproduct} to
Define, for all $q\in\NN$, the $k_{0}$-ind-Banach algebras
\begin{align*}
  \BB_{\dR}^{q}\left( R,R^{+} \right)
  &:=\BB_{\dR}^{q,+}\left( R,R^{+} \right)
  \widehat{\otimes}_{A_{\dR}^{1}}
  \text{``}\varinjlim_{t\times}\text{"}A_{\dR}^{1} \text{ and} \\
  \BB_{\dR}^{>q}\left( R,R^{+} \right)
  &:=\BB_{\dR}^{>q,+}\left( R,R^{+} \right)
  \widehat{\otimes}_{A_{\dR}^{1}}
  \text{``}\varinjlim_{t\times}\text{"}A_{\dR}^{1}.
\end{align*}
For $q=\infty$, set
\begin{equation*}
  \BB_{\dR}^{\dag}\left( R,R^{+} \right)=
  \BB_{\dR}^{\infty}\left( R,R^{+} \right)
  :=\BB_{\dR}^{\infty,+}\left( R,R^{+} \right)
  \widehat{\otimes}_{A_{\dR}^{1}}
  \text{``}\varinjlim_{t\times}\text{"}A_{\dR}^{1}
  =\BB_{\dR}^{\dag,+}\left( R,R^{+} \right)
  \widehat{\otimes}_{A_{\dR}^{1}}
  \text{``}\varinjlim_{t\times}\text{"}A_{\dR}^{1}.
\end{equation*}

\begin{defn}
  $\BB_{\dR}^{\dag}\left( R,R^{+} \right)$ is the \emph{relative overconvergent de Rham period ring}.
  When $\left( R, R^{+}\right)=\left(K,K^{+}\right)$ is a well-understood perfectoid field,
  $B_{\dR}^{\dag}:=\BB_{\dR}^{\dag}\left(K,K^{+}\right)$
  is the \emph{overconvergent de Rham period ring}.
\end{defn}

The previous definitions depend a priori on the choice of $\epsilon$.

\begin{lem}\label{Bla-independent-of-epsilon}
  Fix another choice $\epsilon^{\prime}\in K^{\flat+}$
  of a compatible system of primitive $p$th power roots of unity.
  Then there exists a unique unit $u\in\A_{\inf}$
  such that $(\left[\epsilon\right]-1)=u\left(\left[\epsilon^{\prime}\right]-1\right)$.
  Writing $t^{\prime}:=\log\left(\left[\epsilon^{\prime}\right]\right)$,
  $u$ induces an isomorphism
  \begin{equation}\label{eq:Bla-independent-of-epsilon}
    \text{``}\varinjlim_{t\times}\text{"}A_{\dR}^{1}
    \cong\text{``}\varinjlim_{t^{\prime}\times}\text{"}A_{\dR}^{1}.
  \end{equation}  
  As a consequence, the definitions of $\BB_{\dR}^{q}\left( R,R^{+} \right)$ and
  $\BB_{\dR}^{>q}\left( R,R^{+} \right)$ for all $q\in\NN$ as well as
  $\BB_{\dR}^{\dag}\left( R,R^{+} \right)$ are independent of the choice of $\epsilon$.
  \iffalse\begin{alignat*}{2}
    \BB_{\dR}^{q}\left( R,R^{+} \right)
      &\cong\BB_{\dR}^{q,+}\left( R,R^{+} \right)
      \widehat{\otimes}_{A_{\dR}^{1}}
    \widehat{A_{\dR}^{1}\left[1/t^{\prime}\right]}, \text{ and} \\
    \BB_{\dR}\left( R,R^{+} \right)
      &\cong\BB_{\dR}^{+}\left( R,R^{+} \right)
   \widehat{\otimes}_{A_{\dR}^{1}}
  \widehat{A_{\dR}^{1}\left[1/t^{\prime}\right]}.
  \end{alignat*}
  \begin{align*}
    \BB_{\dR}^{q}\left( R,R^{+} \right)
      &\cong\BB_{\dR}^{q,+}\left( R,R^{+} \right)
      \widehat{\otimes}_{A_{\dR}^{1}}
    \widehat{A_{\dR}^{1}\left[1/t^{\prime}\right]} \text{ and} \\
    \BB_{\dR}^{>q}\left( R,R^{+} \right)
      &\cong\BB_{\dR}^{>q,+}\left( R,R^{+} \right)
      \widehat{\otimes}_{A_{\dR}^{1}}
    \widehat{A_{\dR}^{1}\left[1/t^{\prime}\right]}
  \end{align*}
  for all $q\in\NN\cup\{\infty\}$. In particular,
  \begin{equation*}
    \BB_{\dR}^{\dag}\left( R,R^{+} \right)
      \cong\BB_{\dR}^{\dag,+}\left( R,R^{+} \right)
   \widehat{\otimes}_{A_{\dR}^{1}}
  \widehat{A_{\dR}^{1}\left[1/t^{\prime}\right]}.
  \end{equation*}\fi
\end{lem}

\begin{proof}
  Write $\mu:=[\epsilon]-1$ and $\mu^{\prime}:=\left[\epsilon^{\prime}\right]-1$.
  The ideals $\left(\mu\right)=\left( \mu^{\prime}\right)\subseteq A_{\inf}$
  coincide by~\cite[Lemma 3.23]{BhattMorrowScholze2018}.
  \emph{Loc. cit.} Proposition 3.17(ii) furthermore
  says that $\mu$ and $\mu^{\prime}$ are non-zero-divisors, thus
  there exists a unit $u$ such that $\mu=u\mu^{\prime}$. Using again
  that $\mu$ and $\mu^{\prime}$ are non-zero-divisors, one deduces
  that $u$ is unique with respect to that property.
  Next, note that $\alpha+1$ divides $p^{\alpha}$ in $\ZZ_{p}$ and
  define
  \begin{equation*}
    v:=\sum_{\alpha\geq 0} (-1)^{\alpha}\frac{p^{\alpha}}{\alpha+1}\left(\frac{\mu}{p}\right)^{\alpha}
    \in A_{\dR}^{1},
  \end{equation*}
  satisfying $t=v\mu$.
  Note that $v\equiv 1 \mod\xi/p$, which is a unit. Since
  $A_{\dR}^{1}$ is complete with respect to the $\xi/p$-adic topology,
  cf. Lemma~\ref{lem:Alam-local-complete-at-xi}
    and Proposition~\ref{prop:adictop-ringsofpowerseries},
  \cite[\href{https://stacks.math.columbia.edu/tag/05GI}{Tag 05GI}]{stacks-project} implies that
  $v$ is a unit. Similarly, write
  %\begin{equation*}
    $t^{\prime} = \mu^{\prime} v^{\prime}$
  %\end{equation*}
  for some unit $v^{\prime}\in A_{\dR}^{1}$. In particular,
  %\begin{equation*}
    $t = \left( vuv^{\prime}\right) t^{\prime}$.
  %\end{equation*}
  Then $vuv^{\prime}$ is again a unit,
  giving rise to the isomorphism~(\ref{eq:Bla-independent-of-epsilon}).
\end{proof}

The previous definitions depend a priori on the choice of $\left(K,K^{+}\right)$.

\begin{lem}\label{lem:tensoruptoAinfK-or-AinfC}
  Let $\left( R,R^{+}\right)$ be an affinoid perfectoid
  $\left(K,K^{+}\right)$-algebra and an affinoid perfectoid
  $\left(L,L^{+}\right)$-algebra,
  where both $K$ and $L$ are completions of algebraic extension of $k$.
  Then
  \begin{align*}
    &\BB_{\dR}^{q,+}\left(R,R^{+}\right)
    \widehat{\otimes}_{\A_{\dR}^{1}\left( K,K^{+} \right)}\text{``}\varinjlim_{t_{K}\times}\text{"}\A_{\dR}^{1}\left( K,K^{+} \right) \\
    &\cong
    \BB_{\dR}^{q,+}\left(R,R^{+}\right)
    \widehat{\otimes}_{\A_{\dR}^{1}\left( L,L^{+} \right)}\text{``}\varinjlim_{t_{L}\times}\text{"}\A_{\dR}^{1}\left( L,L^{+} \right)
    \text{ and} \\
    &\BB_{\dR}^{>q,+}\left(R,R^{+}\right)
    \widehat{\otimes}_{\A_{\dR}^{1}\left( K,K^{+} \right)}\text{``}\varinjlim_{t_{K}\times}\text{"}\A_{\dR}^{1}\left( K,K^{+} \right) \\
    &\cong
    \BB_{\dR}^{>q,+}\left(R,R^{+}\right)
    \widehat{\otimes}_{\A_{\dR}^{1}\left( L,L^{+} \right)}\text{``}\varinjlim_{t_{L}\times}\text{"}\A_{\dR}^{1}\left( L,L^{+} \right).
    \intertext{This implies}
    &\BB_{\dR}^{\dag,+}\left(R,R^{+}\right)
    \widehat{\otimes}_{\A_{\dR}^{1}\left( K,K^{+} \right)}\text{``}\varinjlim_{t_{K}\times}\text{"}\A_{\dR}^{1}\left( K,K^{+} \right) \\
    &\cong
    \BB_{\dR}^{\dag,+}\left(R,R^{+}\right)
    \widehat{\otimes}_{\A_{\dR}^{1}\left( L,L^{+} \right)}\text{``}\varinjlim_{t_{L}\times}\text{"}\A_{\dR}^{1}\left( L,L^{+} \right)
  \end{align*}  
  for all $q\in\NN$.
  Here, both $t_{K}\in\A_{\dR}^{1}\left( K,K^{+} \right)$
  and $t_{L}\in\A_{\dR}^{1}\left( L,L^{+} \right)$
  are given as $t$ in Definition~\ref{defn:t}.
\end{lem}

\begin{proof}
  Suppose $K$ is the completion of $K^{\prime}$ and $L$ is the completion of $L^{\prime}$,
  where both $K^{\prime}/k$ and $L^{\prime}/k$ are algebraic extensions. Let
  $\overline{k}$ denote the algebraic closure which we fixed in
  \S\ref{subsec:conventions-reconstructionpaper}.
  By~\cite[\href{https://stacks.math.columbia.edu/tag/09GU}{Tag 09GU}]{stacks-project},
  we may assume without loss of generality that $K^{\prime}$ and $L^{\prime}$
  are contained in $\overline{k}$. In particular, we may assume $K,L\subseteq C$,
  where $C$ is the completion of $\overline{k}$. The commutative diagram
% https://q.uiver.app/#q=WzAsNCxbMSwyLCJrIl0sWzAsMSwiSyJdLFsyLDEsIkwiXSxbMSwwLCJDIl0sWzAsMywiIiwwLHsic3R5bGUiOnsidGFpbCI6eyJuYW1lIjoiaG9vayIsInNpZGUiOiJib3R0b20ifX19XSxbMCwxLCIiLDIseyJzdHlsZSI6eyJ0YWlsIjp7Im5hbWUiOiJob29rIiwic2lkZSI6ImJvdHRvbSJ9fX1dLFsxLDMsIiIsMix7InN0eWxlIjp7InRhaWwiOnsibmFtZSI6Imhvb2siLCJzaWRlIjoidG9wIn19fV0sWzAsMiwiIiwyLHsic3R5bGUiOnsidGFpbCI6eyJuYW1lIjoiaG9vayIsInNpZGUiOiJ0b3AifX19XSxbMiwzLCIiLDIseyJzdHlsZSI6eyJ0YWlsIjp7Im5hbWUiOiJob29rIiwic2lkZSI6ImJvdHRvbSJ9fX1dXQ==
\[\begin{tikzcd}
	& C \\
	K && L \\
	& k
	\arrow[hook, from=2-1, to=1-2]
	\arrow[hook', from=2-3, to=1-2]
	\arrow[hook', from=3-2, to=1-2]
	\arrow[hook', from=3-2, to=2-1]
	\arrow[hook, from=3-2, to=2-3]
\end{tikzcd}\]
  gives a summary of our setup. Now compute
  \begin{equation*}
  \begin{split}
  &\BB_{\dR}^{q,+}\left(R,R^{+}\right)\widehat{\otimes}_{\A_{\dR}^{1}\left( K,K^{+} \right)}\text{``}\varinjlim_{t_{K}\times}\text{"}\A_{\dR}^{1}\left( K,K^{+} \right)  \\
  &\cong
   \text{``}\varinjlim_{t_{K}\times}\text{"}\BB_{\dR}^{q,+}\left(R,R^{+}\right)\\
   &\cong
   \BB_{\dR}^{q,+}\left(R,R^{+}\right)\widehat{\otimes}_{\A_{\dR}^{1}\left( C,C^{+} \right)}\text{``}\varinjlim_{t_{K}\times}\text{"}\A_{\dR}^{1}\left( C,C^{+} \right).
   \end{split}
   \end{equation*}
   This is an isomorphism of $k$-ind-Banach algebras.
   We do this computation again but over $L$ to get
  \begin{equation*}
  \begin{split}
  &\BB_{\dR}^{q,+}\left(R,R^{+}\right)\widehat{\otimes}_{\A_{\dR}^{1}\left( L,L^{+} \right)}\text{``}\varinjlim_{t_{L}\times}\text{"}\A_{\dR}^{1}\left( L,L^{+} \right)  \\
  &\cong
   \text{``}\varinjlim_{t_{L}\times}\text{"}\BB_{\dR}^{q,+}\left(R,R^{+}\right)\\
   &\cong
   \BB_{\dR}^{q,+}\left(R,R^{+}\right)\widehat{\otimes}_{\A_{\dR}^{1}\left( C,C^{+} \right)}\text{``}\varinjlim_{t_{L}\times}\text{"}\A_{\dR}^{1}\left( C,C^{+} \right) 
   \end{split}
   \end{equation*}
   Similar computations work for $\BB_{\dR}^{>q,+}$.
   Now apply Lemma~\ref{Bla-independent-of-epsilon}.
   Compute the colimit along $q\to\infty$ to get the result
   for $\BB_{\dR}^{\dag,+}$.
\end{proof}

Write $A_{\dR}^{>1}:=\A_{\dR}^{>1}\left(K,K^{+}\right)$ and $B_{\dR}^{>1}:=\BB_{\dR}^{>1}\left(K,K^{+}\right)$.
We record Lemma~\ref{lem:BBdR-fromAAdR-and-BdRgreaterthan1-reconstructionpaper}
for future reference.

\begin{lem}\label{lem:BBdR-fromAAdR-and-BdRgreaterthan1-reconstructionpaper}
  We have the canonical isomorphism
  \begin{equation}\label{eq:BBdR-fromAAdR-and-BdRgreaterthan1-themap-reconstructionpaper}
    \A_{\dR}^{\dag}\left(R,R^{+}\right)\widehat{\otimes}_{A_{\dR}^{>1}}B_{\dR}^{>1}
    \isomap\BB_{\dR}^{\dag}\left(R,R^{+}\right).
  \end{equation}
\end{lem}

\begin{proof}
  Compute
  \begin{align*}
    \A_{\dR}^{>q}\left(R,R^{+}\right)\widehat{\otimes}_{A_{\dR}^{>1}}B_{\dR}^{>1,+}
    &=\A_{\dR}^{>q}\left(R,R^{+}\right)
      \widehat{\otimes}_{A_{\dR}^{>1}}\left(A_{\dR}^{>1}\widehat{\otimes}_{W(\kappa)}k_{0}\right) \\
    &\cong\A_{\dR}^{>q}\left(R,R^{+}\right)\widehat{\otimes}_{W(\kappa)}k_{0}
    =\BB_{\dR}^{>q,+}\left(R,R^{+}\right).
  \end{align*}
  Now pass to the colimit along $q\to\infty$ and
  apply $\varinjlim_{t\times}$.
\end{proof}

Recall Definitions~\ref{defn:bornologicalspace-reconstructionpaper}
and~\ref{defn:bornology-countable-basis-reconstructionpaper}.

\begin{thm}\label{thm:BdRdagRRplus-bornological}
  $\BB_{\dR}^{\dag}\left(R,R^{+}\right)$ is bornological. Its bornology has countable basis.
\end{thm}

We complete the proof of Theorem~\ref{thm:BdRdagRRplus-bornological}
on page~\pageref{proof:thm:BdRdagRRplus-bornological}.
We start with some lemmata.

Recall Definition~\ref{defn:Wkappatriv-tildeAdRgreaterthanq}.

\begin{lem}\label{lem:principlesymbol-of-t}
  Let $q\in\NN_{\geq2}$ %, equip $\widetilde{\A}_{\dR}^{>q}\left(R,R^{+}\right)$ with the $\xi/p^{q}$-adic filtration,
  and compute
  \begin{equation*}
    \sigma\left(\frac{t}{p^{q}}\right)
    =\left(\zeta_{p}-1\right)\sigma\left(\frac{\xi}{p^{q}}\right)
    \in R^{+}\sigma\left(\frac{\xi}{p^{q}}\right)
    \stackrel{\text{\ref{lem:grAdRgreaterthanq-xioverpadicfiltration-reconstructionpaper}}}{\cong}
    \gr^{1}\widetilde{\A}_{\dR}^{>q}\left(R,R^{+}\right).
  \end{equation*}
\end{lem}

\begin{proof}
  Recall the Definition~(\ref{eq:Kadmtscompativlesystemetc-defnxi}) of $\xi$.
  One finds $t/p^{q}=v \mu/p^{q}=v\varphi^{-1}\left(\mu\right)\xi/p^{q}$, cf.
  the proof of Lemma~\ref{Bla-independent-of-epsilon}
  and~\cite[\S 3.3]{BhattMorrowScholze2018}.
  Here, $v$ is congruent to $1$ modulo $\xi/p^{q}$. Thus we can compute
  \begin{equation*}
    \sigma\left(\frac{t}{p^{q}}\right)
    =\theta_{\dR}^{>q}\left( \varphi^{-1}\left(\mu\right) \right)\sigma\left( \frac{\xi}{p^{q}} \right)
    =\left(\zeta_{p}-1\right)\sigma\left(\frac{\xi}{p^{q}}\right),
  \end{equation*}
  as
  $R^{+}\sigma\left(\xi / p^{q}\right)\stackrel{\text{\ref{lem:grAdRgreaterthanq-xioverpadicfiltration-reconstructionpaper}}}{\cong}\gr^{1}\A_{\dR}^{>q}\left(R,R^{+}\right)$
  comes from Fontaine's map. Lemma~\ref{lem:grAdRgreaterthanq-xioverpadicfiltration-reconstructionpaper}
  required $q\geq2$.
\end{proof}

Lemma~\ref{lem:canetaoperatorAdRgreaterthanqRRplus-i} below
is not used in the proof of Theorem~\ref{thm:BdRdagRRplus-bornological}.
We still include it for future reference.

\begin{lem}\label{lem:canetaoperatorAdRgreaterthanqRRplus-i}
  Let $q\in\NN_{\geq2}$ and $i\in\NN$. Then
  $\left(t/p^{q}\right)^{i}\widetilde{\A}_{\dR}^{>q}\left(R,R^{+}\right)\subseteq\widetilde{\A}_{\dR}^{>q}\left(R,R^{+}\right)$
  is a closed subset.
\end{lem}

\begin{proof}
  We check the following stronger statement:
  $\widetilde{\A}_{\dR}^{>q}\left(R,R^{+}\right)\stackrel{t/p^{q}}{\to}\widetilde{\A}_{\dR}^{>q}\left(R,R^{+}\right)$
  is a strict monomorphism.
  By~\cite[Chapter I, \S 4.2 page 31-32, Theorem 4(5)]{HuishiOystaeyen1996},
  we can equip $\widetilde{\A}_{\dR}^{>q}\left(R,R^{+}\right)$ with the $\xi/p^{q}$-adic filtration
  and check that
  $\gr\widetilde{\A}_{\dR}^{>q}\left(R,R^{+}\right)\stackrel{\sigma\left(t/p^{q}\right)}{\to}\gr\widetilde{\A}_{\dR}^{>q}\left(R,R^{+}\right)$
  is injective.
  This follows from Lemma~\ref{lem:principlesymbol-of-t}
  and Lemma~\ref{lem:grAdRgreaterthanq-xioverpadicfiltration-reconstructionpaper}.
  \emph{Loc. cit.} required $q\geq2$.
\end{proof}

The proof of Lemma~\ref{lem:canetaoperatorAdRgreaterthanqRRplus-i}
implies Lemma~\ref{lem:canetaoperatorAdRgreaterthanqRRplus-ii} below,
which we use in the proof of Theorem~\ref{thm:BdRdagRRplus-bornological}.

\begin{lem}\label{lem:canetaoperatorAdRgreaterthanqRRplus-ii}
  $\widetilde{\A}_{\dR}^{>q}\left(R,R^{+}\right)$, and therefore
  $\A_{\dR}^{>q}\left(R,R^{+}\right)$, does not have $t/p^{q}$-torsion.
\end{lem}

\iffalse %%% i broke this into two lemmata

\begin{lem}\label{lem:canetaoperatorAdRgreaterthanqRRplus}
  Equip $\A_{\dR}^{>q}\left(R,R^{+}\right)$ with the $\xi/p^{q}$-adic topology.
  \begin{itemize}
    \item[(i)] $t^{i}\A_{\dR}^{>q}\left(R,R^{+}\right)\subseteq\A_{\dR}^{>q}\left(R,R^{+}\right)$
      is a closed subset for all $i\geq0$.
    \item[(ii)] $\A_{\dR}^{>q}\left(R,R^{+}\right)$ does not have $t$-torsion.
  \end{itemize}
\end{lem}

\begin{proof}
  Both (i) and (ii) are implied by the following claim:
  $\A_{\dR}^{>q}\left(R,R^{+}\right)\stackrel{t}{\to}\A_{\dR}^{>q}\left(R,R^{+}\right)$
  is a strict monomorphism.
  By~\cite[Chapter I, \S 4.2 page 31-32, Theorem 4(5)]{HuishiOystaeyen1996},
  we can equip $\A_{\dR}^{>q}\left(R,R^{+}\right)$ with the $\xi/p^{q}$-adic filtration
  and check the following:
  $\gr\A_{\dR}^{>q}\left(R,R^{+}\right)\stackrel{\sigma(t)}{\to}\gr\A_{\dR}^{>q}\left(R,R^{+}\right)$
  is injective.
  This follows from Lemma~\ref{lem:principlesymbol-of-t}.
\end{proof}

\fi %%%% comment ends.

\begin{prop}\label{prop:underlyingspaceBdRdagplus-iso-uncompletedcolim-reconstructionpaper}
  We have the canonical isomorphism of abstract $k_{0}$-algebras
  \begin{equation}\label{eq:themap--prop:underlyingspaceBdRdagplus-iso-uncompletedcolim-reconstructionpaper}
    \varinjlim_{q}|\A_{\dR}^{>q}\left(R,R^{+}\right)|[1/p]\isomap|\BB_{\dR}^{\dag,+}\left(R,R^{+}\right)|.
  \end{equation}
\end{prop}

\begin{proof}
  %Note that $\varinjlim_{q}|A_{\dR}^{>q}\left(R,R^{+}\right)|\cong\varinjlim_{q}|A_{\dR}^{q}\left(R,R^{+}\right)|$.
  Injectivity follows from
  Lemma~\ref{lem:underlyingspaceBdRdagplus-iso-uncompletedcolim-injective-reconstructionpaper}
  and surjectivity follows from
  Lemma~\ref{lem:underlyingspaceBdRdagplus-iso-uncompletedcolim-surjective-reconstructionpaper}
  below.
\end{proof}

\begin{lem}\label{lem:underlyingspaceBdRdagplus-iso-uncompletedcolim-injective-reconstructionpaper}
  $\A_{\dR}^{>q}\left(R,R^{+}\right)[1/p]\to \BB_{\dR}^{>q,+}\left(R,R^{+}\right)$
  is injective for all $q\in\NN$.
\end{lem}

\begin{proof}
  Write $A^{>q}:=\A_{\dR}^{>q}\left(R,R^{+}\right)$ and $B^{>q}:=\BB_{\dR}^{>q,+}\left(R,R^{+}\right)$.
  By Lemma~\ref{lem:localisation-normedotimes}, $B^{>q}$ is the completion
  of $A^{>q}[1/p]$ which carries the following seminorm:
  \begin{equation*}
    b
    =\inf_{\substack{b=a/p^{n} \\ a\in A^{>q}}}\|a\|/\|p^{n}\|.
  \end{equation*}
  Thus we have to check: For all $b\in A^{>q}[1/p]$ such that $\|b\|=0$, $b=0$.
  To do this, fix $b\in A^{>q}[1/p]$ with $\|b\|=0$. Then we find a sequence of expressions
  $b=a_{\alpha}/p^{n_{\alpha}}$ with $a_{\alpha}\in A^{>q}$, $\alpha\in\NN$, such that
  \begin{equation}\label{eq:underlyingspaceBdRdagplus-iso-uncompletedcolim-injective-convergence-reconstructionpaper}
    \|a_{\alpha}\|p^{n_{\alpha}}
    \stackrel{\text{\ref{lem:Agreaterthanq-multiplybyp-norm-reconstructionpaper}}}{=}
    \|a_{\alpha}\|/\|p^{n_{\alpha}}\|
    \to 0 \text{ for $\alpha\to\infty$.}
  \end{equation}
  Furthermore, $a_{\alpha}/p^{n_{\alpha}}=b=a_{\alpha+1}/p^{n_{\alpha+1}}$ implies
  $p^{n_{\alpha+1}-n_{\alpha}}a_{\alpha}=a_{\alpha+1}$ for all $\alpha\in\NN$.
  Via induction, we deduce $p^{n_{\alpha}-n_{0}}a_{0}=a_{\alpha}$ for all $\alpha\in\NN$.
  But then
  \begin{equation*}
    \|a_{\alpha}\|p^{n_{\alpha}}
    =\| p^{n_{\alpha}-n_{0}} a_{0} \|p^{n_{\alpha}}
    \stackrel{\text{\ref{lem:Agreaterthanq-multiplybyp-norm-reconstructionpaper}}}{=} p^{n_{\alpha} - n_{\alpha} +n_{0}} \| a_{0} \|
    =p^{n_{0}}\|a_{0}\|.
  \end{equation*}
%  for all $\alpha\in\NN$.
  (\ref{eq:underlyingspaceBdRdagplus-iso-uncompletedcolim-injective-convergence-reconstructionpaper})
  implies $\|a_{\alpha}\|=0$ for all $\alpha\in\NN$. Since $A^{>q}$ is normed, this gives
  $a_{\alpha}=0$ for all $\alpha\in\NN$.
\end{proof}

\begin{lem}\label{lem:underlyingspaceBdRdagplus-iso-uncompletedcolim-surjective-reconstructionpaper}
  Denote~(\ref{eq:themap--prop:underlyingspaceBdRdagplus-iso-uncompletedcolim-reconstructionpaper})
  in Proposition~\ref{prop:underlyingspaceBdRdagplus-iso-uncompletedcolim-reconstructionpaper} by
  $\psi$. For all $b\in \BB_{\dR}^{>q,+}\left(R,R^{+}\right)$, there exist
  $c\in \A_{\dR}^{>q+1}\left(R,R^{+}\right)[1/p]$ such that $\psi\left(c\right)=b$.
\end{lem}

\begin{proof}
  Write $A^{>q}:=\A_{\dR}^{>q}\left(R,R^{+}\right)$,
  $A^{>q+1}:=\A_{\dR}^{>q+1}\left(R,R^{+}\right)$, and $B^{>q}:=\BB_{\dR}^{>q,+}\left(R,R^{+}\right)$.
  By Lemma~\ref{lem:localisation-normedotimes}, $B^{>q}$ is the completion
  of $A^{>q}[1/p]$, which carries the following seminorm:
  \begin{equation}\label{eq:underlyingspaceBdRdagplus-iso-uncompletedcolim-surjective-norm-reconstructionpaper}
    b
    =\inf_{\substack{b=a/p^{n} \\ a\in A^{>q}}}\|a\|/\|p^{n}\|
    \stackrel{\text{\ref{lem:Agreaterthanq-multiplybyp-norm-reconstructionpaper}}}{=}
      \inf_{\substack{b=a/p^{n} \\ a\in A^{>q}}}\|a\|p^{n} 
  \end{equation}  
  Therefore, every element $b\in B^{>q}$ is of the form
  $b=\sum_{\alpha\geq0}b_{\alpha}$
  with $b_{\alpha}\in A^{>q}[1/p]$ for all $\alpha\in\NN$
  such that $b_{\alpha}\to0$ for $\alpha\to\infty$.
  Furthermore, the
  description~(\ref{eq:underlyingspaceBdRdagplus-iso-uncompletedcolim-surjective-norm-reconstructionpaper})
  implies that there exist expressions $b_{\alpha}=a_{\alpha}/p^{n_{\alpha}}$
  for all $\alpha\in\NN$ such that $\|a_{\alpha}\|p^{n_{\alpha}}\to 0$ for $\alpha\to\infty$.
  Set $s_{\alpha}:=-\log_{p}\|a_{\alpha}\|$ for all $\alpha\in\NN$.
  Then $b_{\alpha}\to0$ for $\alpha\to\infty$ implies
  \begin{equation}\label{eq:underlyingspaceBdRdagplus-iso-uncompletedcolim-surjective-conv2-reconstructionpaper}
    s_{\alpha}-n_{\alpha}\to\infty \text{ for } \alpha\to\infty.
  \end{equation}
  Therefore, we can fix an $\alpha_{0}\in\NN$ such that $s_{\alpha}-n_{\alpha}\geq0$
  for all $\alpha\geq\alpha_{0}$.

  Next, we observe that $a_{\alpha}\in\left(p,\xi/p^{q}\right)^{s_{\alpha}}$; this is because
  $A^{>q}$ carries the $\left(p,\xi/p^{q}\right)$-adic norm. Therefore, we can write
  \begin{equation*}
    a_{\alpha}=\sum_{i=0}^{s_{\alpha}}a_{\alpha,i}p^{s_{\alpha}-i}\left(\frac{\xi}{p^{q}}\right)^{i}
  \end{equation*}
  with $a_{\alpha,0},\dots,a_{\alpha,s_{\alpha}}\in A^{>q}$. Furthermore,
  \begin{equation}\label{eq:underlyingspaceBdRdagplus-iso-uncompletedcolim-surjective-someidentity-reconstructionpaper}
    \frac{a_{\alpha}}{p^{n_{\alpha}}}=\sum_{i=0}^{s_{\alpha}}a_{\alpha,i}p^{s_{\alpha}-n_{\alpha}}\left(\frac{\xi}{p^{q+1}}\right)^{i}
  \end{equation}
  is an element of $A^{>q+1}$ because
  $s_{\alpha}-n_{\alpha}\geq0$ as $\alpha\geq\alpha_{0}$. Furthermore,
  the power series
  \begin{equation*}
    c_{2}:=
      \sum_{\alpha\geq\alpha_{0}}\left( \sum_{i=0}^{s_{\alpha}} a_{\alpha,i}\left(\frac{\xi}{p^{q+1}}\right)^{i} \right)
      p^{s_{\alpha}-n_{\alpha}}  
  \end{equation*}
  converges in $A^{>q+1}$
  by~(\ref{eq:underlyingspaceBdRdagplus-iso-uncompletedcolim-surjective-conv2-reconstructionpaper}).
  Now define $c:=c_{1}+c_{2}\in A^{>q+1}[1/p]$ where $c_{1}$ is the image of
  \begin{equation*}
    c_{1}:=\sum_{\alpha=0}^{\alpha_{0}-1}\frac{a_{\alpha}}{p^{n_{\alpha}}}\in A^{>q}[1/p].
  \end{equation*}
  in $A^{>q+1}[1/p]$. We find
  \begin{equation*}
    \varphi(c)
    =\varphi\left(c_{1}\right)+\varphi\left(c_{2}\right)
%    =\sum_{\alpha=0}^{\alpha_{0}-1}\frac{a_{\alpha}}{p^{n_{\alpha}}}
%      +\sum_{\alpha\geq\alpha_{0}}\left( \sum_{i=0}^{s_{\alpha}} a_{\alpha,i}\left(\frac{\xi}{p^{q+1}}\right)^{i} \right)
%      p^{s_{\alpha}-n_{\alpha}}
    \stackrel{\text{(\ref{eq:underlyingspaceBdRdagplus-iso-uncompletedcolim-surjective-someidentity-reconstructionpaper})}}{=}
      \sum_{\alpha=0}^{\alpha_{0}-1}\frac{a_{\alpha}}{p^{n_{\alpha}}}
      +\sum_{\alpha\geq0}\frac{a_{\alpha}}{p^{n_{\alpha}}}
%    =\sum_{\alpha\geq0}b_{\alpha}
    =b,
  \end{equation*}
  as desired.
\end{proof}

%We care about Proposition~\ref{prop:underlyingspaceBdRdagplus-iso-uncompletedcolim-reconstructionpaper}
%because of the following result.

\begin{cor}\label{cor:tnonzerodivisor-on-BdRqplus-reconstructionpaper}
  The multiplication-by-$t$ map $\BB_{\dR}^{q,+}\left(R,R^{+}\right) \to \BB_{\dR}^{q,+}\left(R,R^{+}\right)$
  is injective for all $q\geq2$.
\end{cor}

\begin{proof}
  Consider the following diagram of abstract $k$-vector spaces
  \begin{equation}\label{eq:tnonzerodivisor-on-BdRqplus-reconstructionpaper}
    \begin{tikzcd}
      {|\BB_{\dR}^{\dag,+}\left(R,R^{+}\right)|}
      \arrow{r}{t\times} &
      {|\BB_{\dR}^{\dag,+}\left(R,R^{+}\right)|} \\
      {|\BB_{\dR}^{q,+}\left(R,R^{+}\right)|}
      \arrow[hook]{u}
      \arrow{r}{t\times} &
      {|\BB_{\dR}^{q,+}\left(R,R^{+}\right)|}
      \arrow[hook]{u}.      
    \end{tikzcd}
  \end{equation}
  The morphism at the top is injective by
  Lemma~\ref{lem:canetaoperatorAdRgreaterthanqRRplus-ii}, which applies
  thanks to Proposition~\ref{prop:underlyingspaceBdRdagplus-iso-uncompletedcolim-reconstructionpaper}.
  The vertical morphisms are injective
  by Proposition~\ref{prop:BdRdagplus-bornology-countable-basis-reconstructionpaper};
  here we use $q\geq2$.
  Thus Corollary~\ref{cor:tnonzerodivisor-on-BdRqplus-reconstructionpaper}
  follows from the commutativity of the diagram~(\ref{eq:tnonzerodivisor-on-BdRqplus-reconstructionpaper}).
\end{proof}

\begin{proof}[Proof of Theorem~\ref{thm:BdRdagRRplus-bornological}]\label{proof:thm:BdRdagRRplus-bornological}
  $\BB_{\dR}^{\dag}\left(R,R^{+}\right)$ is by definition the colimit of the diagram
  \begin{equation*}
    \BB_{\dR}^{2}\left(R,R^{+}\right)
    \stackrel{t\times}{\longrightarrow}
    \BB_{\dR}^{3}\left(R,R^{+}\right)
    \stackrel{t\times}{\longrightarrow}
    \BB_{\dR}^{4}\left(R,R^{+}\right)
    \stackrel{t\times}{\longrightarrow}
    \dots
  \end{equation*}
  where the maps in this diagram are injective
  by Proposition~\ref{prop:BdRdagplus-bornology-countable-basis-reconstructionpaper}
  and Corollary~\ref{cor:tnonzerodivisor-on-BdRqplus-reconstructionpaper}.
\end{proof}

We continue with a few lemmata which will be useful later on in this article.

\begin{notation}
  Let $\zeta_{1},\dots,\zeta_{n}$ be formal variables.
  $\A_{\dR}^{>q}\left(R,R^{+}\right)\left\llbracket \zeta_{1},\dots,\zeta_{n} \right\rrbracket$
  denotes the usual formal power series ring equipped with
  the $\left(p,\xi/p^{q},\zeta_{1},\dots,\zeta_{n}\right)$-adic seminorm.
  Note that $\zeta_{n},\dots,\zeta_{1},\xi/p^{q},p$ is a regular sequence
  by Lemma~\ref{lem:Fontaines-map-for-Ala}. Thus
  $\A_{\dR}^{>q}\left(R,R^{+}\right)\left\llbracket \zeta_{1},\dots,\zeta_{n} \right\rrbracket$
  is a $W(\kappa)$-Banach algebra
  by \emph{loc. cit.} Corollary~\ref*{prop:Scomplete-Spowerseriescomplete-ifKoszulregular}.
\end{notation}

The following two Lemma~\ref{lem:multbyxipq-onAdRgreaterthanq-strictmono-reconstructionpaper}
and~\ref{lem:multbytpq-onAdRgreaterthanq-strictmono-reconstructionpaper} apply in the case $n=0$ as well.

\begin{lem}\label{lem:multbyxipq-onAdRgreaterthanq-strictmono-reconstructionpaper}
  Let $\zeta_{1},\dots,\zeta_{n}$ be formal variables, $n\in\NN$.
  The multiplication-by-$\xi/p^{q}$ map
  \begin{equation*}  
    \xi/p^{q}\times\colon
    \A_{\dR}^{>q}\left(R,R^{+}\right)\left\llbracket \zeta_{1},\dots,\zeta_{n} \right\rrbracket
    \to
    \A_{\dR}^{>q}\left(R,R^{+}\right)\left\llbracket \zeta_{1},\dots,\zeta_{n} \right\rrbracket
  \end{equation*}
  is a strict monomorphism for all $q\in\NN_{\geq2}$.
\end{lem}

\begin{proof}
  By~\cite[Chapter I, \S 4.2 page 31-32, Theorem 4(5)]{HuishiOystaeyen1996},
  this follows once
  \begin{equation*}
    \sigma\left(\frac{\xi}{p^{q}}\right)\in\gr\A_{\dR}^{>q}\left(R,R^{+}\right)\left\llbracket \zeta_{1},\dots,\zeta_{n} \right\rrbracket
  \end{equation*}
  is not a zero-divisor. Here, $\A_{\dR}^{>q}\left(R,R^{+}\right)\left\llbracket \zeta_{1},\dots,\zeta_{n} \right\rrbracket$
  carries the $\left(p,\xi/p^{q},\zeta_{1},\dots,\zeta_{n}\right)$-adic filtration.
  This follows from the description of the associated graded
  \begin{equation*}
    \gr\A_{\dR}^{>q}\left(R,R^{+}\right)\left\llbracket \zeta_{1},\dots,\zeta_{n} \right\rrbracket
    \cong \left(R^{+}/p\right)\left[\sigma(p),\sigma\left(\frac{\xi}{p^{q}}\right),\sigma\left(\zeta_{1}\right),\dots,\sigma\left(\zeta_{n}\right)\right],
  \end{equation*}  
  which in turn follows from~\cite[Exercise 17.16.a]{Ei95}.
  Here, \emph{loc. cit.} applies because 
  $\zeta_{n},\dots,\zeta_{1},\xi/p^{q},p$ is a regular sequence,
  cf. Lemma~\ref{lem:Fontaines-map-for-Alagreatq}.
\end{proof}

\begin{lem}\label{lem:multbytpq-onAdRgreaterthanq-strictmono-reconstructionpaper}
  Let $\zeta_{1},\dots,\zeta_{n}$ be formal variables.
  The multiplication-by-$t/p^{q}$ map
  \begin{equation*}  
    t/p^{q}\times\colon
    \A_{\dR}^{>q}\left(R,R^{+}\right)\left\llbracket \zeta_{1},\dots,\zeta_{n} \right\rrbracket
    \to
    \A_{\dR}^{>q}\left(R,R^{+}\right)\left\llbracket \zeta_{1},\dots,\zeta_{n} \right\rrbracket
  \end{equation*}
  is a strict monomorphism for all $q\in\NN_{\geq2}$.
\end{lem}

\begin{proof}
  Recall $\mu=[\epsilon]-1$. Since
  $t/p^{q}$ and $\mu/p^{q}$
  coincide up to a unit, cf. the proof of Lemma~\ref{Bla-independent-of-epsilon},
  we may as well check that the multiplication-by-$\mu/p^{q}$ map
  is a strict monomorphism. Let $\varphi$ denote the Frobenius. By~\cite[Example 3.16]{BhattMorrowScholze2018},
  $\xi:=\mu / \varphi^{-1}\left(\mu\right)$ is a generator of the kernel
  of Fontaine's map $\theta_{\inf}\colon \A_{\inf}\left(K,K^{+}\right)\to K^{+}$.
  By Lemma~\ref{lem:Fontaines-map-for-Alagreatq},
  $\xi/p^{q}$ is a generator of the kernel of $\theta_{\dR}^{>q}\colon \A_{\dR}^{>q}\left(R,R^{+}\right)\to R^{+}$.
  Thus Lemma~\ref{lem:multbyxipq-onAdRgreaterthanq-strictmono-reconstructionpaper} applies.
  And because $\mu/p^{q}=\varphi^{-1}\left(\mu\right)\xi/p^{q}$, we may check that
  \begin{equation*}  
    \varphi^{-1}\left(\mu\right)\times\colon
    \A_{\dR}^{>q}\left(R,R^{+}\right)\left\llbracket \zeta_{1},\dots,\zeta_{n} \right\rrbracket
    \to
    \A_{\dR}^{>q}\left(R,R^{+}\right)\left\llbracket \zeta_{1},\dots,\zeta_{n} \right\rrbracket
  \end{equation*}  
  is a strict monomorphism, cf.~\cite[Proposition 1.1.7]{Sch99}.
  By~\cite[Proposition 1.1.8]{Sch99}, may check that
  \begin{equation*}  
    \varphi^{-1}\left(\mu\right)^{p-1}\times\colon
    \A_{\dR}^{>q}\left(R,R^{+}\right)\left\llbracket \zeta_{1},\dots,\zeta_{n} \right\rrbracket
    \to
    \A_{\dR}^{>q}\left(R,R^{+}\right)\left\llbracket \zeta_{1},\dots,\zeta_{n} \right\rrbracket
  \end{equation*}  
  is a strict monomorphism.
  In fact, we show that the principal symbol
  \begin{equation*}
    \sigma\left(\varphi^{-1}\left(\mu\right)^{p-1}\right)\in
    \gr\A_{\dR}^{>q}\left(R,R^{+}\right)\left\llbracket \zeta_{1},\dots,\zeta_{n} \right\rrbracket
  \end{equation*}
  is not a zero-divisor, cf.~\cite[Chapter I, \S 4.2 page 31-32, Theorem 4(5)]{HuishiOystaeyen1996}.
  Here, we equipped $\A_{\dR}^{>q}\left(R,R^{+}\right)\left\llbracket \zeta_{1},\dots,\zeta_{n} \right\rrbracket$
  with the $\left(p,\xi/p^{q},\zeta_{1},\dots,\zeta_{n}\right)$-adic filtration.
  We see from the description of the associated as in the proof of
  Lemma~\ref{lem:multbyxipq-onAdRgreaterthanq-strictmono-reconstructionpaper} that it is sufficient
  to check that
  \begin{equation*}
    \sigma\left(\varphi^{-1}\left(\mu\right)^{p-1}\right)\in
    \left(R^{+}/p\right)\left[\sigma(p),\sigma\left(\frac{\xi}{p^{q}}\right)\right]
    \stackrel{\text{\ref{lem:assgr-AdRgreaterthanq-recpaper}}}{\cong}
    \gr\A_{\dR}^{>q}\left(R,R^{+}\right)
  \end{equation*}
  is not a zero-divisor
  We will thus compute $\sigma\left(\varphi^{-1}\left(\mu\right)^{p-1}\right)$.
  Firstly, note
  \begin{equation}\label{eq:multbytpq-onAdRgreaterthanq-strictmono-somecomputation-reconstructionpaper}
    \theta_{\dR}^{>q}\left(\varphi^{-1}\left(\mu\right)^{p-1}\right)
    =\theta_{\dR}^{>q}\left( \left[ \epsilon^{1/p}\right] - 1\right)^{p-1}
    =\left(\zeta_{p}-1\right)^{p-1}
    =up,
  \end{equation}
  where $u\in K^{+}$ is a unit. Since $\xi/p^{q}$ generates
  $\ker\theta_{\dR}^{>q}$, cf. Lemma~\ref{lem:Fontaines-map-for-Alagreatq},
  $\varphi^{-1}\left(\mu\right)^{p-1}\in\left(p,\xi/p^{q}\right)$ follows.
  On the other hand, $\varphi^{-1}\left(\mu\right)^{p-1}\not\in\left(p,\xi/p^{q}\right)^{2}$.
  Indeed, if $\varphi^{-1}\left(\mu\right)^{p-1}\in\left(p,\xi/p^{q}\right)^{2}$, then
  \begin{equation*}
    up
    \stackrel{\text{(\ref{eq:multbytpq-onAdRgreaterthanq-strictmono-somecomputation-reconstructionpaper})}}{=}
    \theta_{\dR}^{>q}\left(\varphi^{-1}\left(\mu\right)^{p-1}\right)\in\left(p^{2}\right),
  \end{equation*}
  contradiction. These considerations imply
  \begin{equation*}
    \sigma\left(\varphi^{-1}\left(\mu\right)^{p-1}\right)
    \in\gr^{1}\A_{\dR}^{>q}\left(R,R^{+}\right).
  \end{equation*}
  Explicitly,~(\ref{eq:multbytpq-onAdRgreaterthanq-strictmono-somecomputation-reconstructionpaper})
  implies $\sigma\left(\varphi^{-1}\left(\mu\right)^{p-1}\right)=\tau(u)\sigma(p)+a\sigma\left(\xi/p^{q}\right)$,
  where $\tau\colon R^{+}\to R^{+}/p$ is the canonical projection and
  $a\in\gr^{0}\A_{\dR}^{>q}\left(R,R^{+}\right)$.
  Now~\cite[\S 7.2, exercise 2, page 238]{DummitFoote2004} implies
  that $\sigma\left(\varphi^{-1}\left(\mu\right)^{p-1}\right)$ is not a zero-divisor, using
  Lemma~\ref{lem:assgr-AdRgreaterthanq-recpaper}
  and that $\tau(u)\in R^{+}/p$ is a unit.
\end{proof}

%%%%%%%%%%%%%%%%%%%%%%%%%%%%%%%%%%%%%%%%%%%%%%%%%%%%%%%%%%%%
%%%%%%%%%%%%%%%%%%%%%%%%%%%%%%%%%%%%%%%%%%%%%%%%%%%%%%%%%%%%
% Inverting t
%%%%%%%%%%%%%%%%%%%%%%%%%%%%%%%%%%%%%%%%%%%%%%%%%%%%%%%%%%%%
%%%%%%%%%%%%%%%%%%%%%%%%%%%%%%%%%%%%%%%%%%%%%%%%%%%%%%%%%%%%

\subsubsection{Adding $\log t$}

We define an overconvergent version of Fontaine's almost de Rham period ring $B_{\pdR}$,
cf.~\cite[\S 4.3]{Fontaine2004Arithmetic}. For the sake of simplicity, we only work over the completion
$C$ of the algebraic closure $\overline{k}$ of $k$ which we fixed in
\S\ref{subsec:conventions-reconstructionpaper}. This suffices for the purposes of this article.

The action of the Galois group $\cal{G}:=\Gal\left(\overline{k} / k\right)$
on $\overline{k}$ extends by continuity to an action on $C$. Now restrict it to an action on $\cal{O}_{C}$.
By functoriality, it gives rise to a $\cal{G}$-action on $A_{\inf}=\A_{\inf}\left(C,\cal{O}_{C}\right)$. It preserves both
$p$ and the kernel of Fontaine's $A_{\inf}\to\cal{O}_{C}$, thus it lifts to continuous $\cal{G}$-actions
on $B_{\dR}^{q,+}:=\BB_{\dR}^{q}\left(C,\cal{O}_{C}\right)$ and $B_{\dR}^{>q,+}:=\BB_{\dR}^{>q}\left(C,\cal{O}_{C}\right)$,
for all $q\in\NN$.

\begin{notation}\label{notation:addingthelogtthin-reconstructionpaper}
  Let $B\in\left\{B_{\dR}^{q,+},B_{\dR}^{>q,+}\colon q\in\NN\right\}$ and $n\in\NN$.
  We consider the $k_{0}$-Banach space
  \begin{equation*}
    B\left[\log t\right]^{\leq n}:=\bigoplus_{i=0}^{n}B\left(\log t\right)^{i},
  \end{equation*}
  where $\log t$ is a formal symbol. %thus it is equipped with the supremum norm
  %$\|\sum_{i=0}^{n}b_{i}\left(\log t\right)^{i}\|$.
  Let $\chi\colon\cal{G}\to\ZZ_{p}^{\times}$ denote the cyclotomic character
  and equip $B\left[\log t\right]^{\leq n}$
  with the continuous $\cal{G}$-action
  \begin{equation*}
    g\left(\sum_{i=0}^{n}b_{i}\left(\log t\right)^{i}\right)
    :=\sum_{i=0}^{n}g\left(b_{i}\right)\left( \log\chi(g) + \log t\right)^{i}
    \text{ for all } g\in\cal{G}.
  \end{equation*}
  Finally, $B\left[\log t\right]:=\text{``}\varinjlim_{n}\text{"}B\left[\log t\right]^{\leq n}$
  is an ind-$\cal{G}$-$k_{0}$-Banach space.
\end{notation}

An \emph{ind-$\cal{G}$-$k_{0}$-algebra} is a monoid in $\IndBan_{k_{0}}\left(\cal{G}\right)$.
Lemma~\ref{lem:B[log t]-monoid-recpaper} is obvious.

\begin{lem}\label{lem:B[log t]-monoid-recpaper}
  Fix $B$ as in Notation~\ref{notation:addingthelogtthin-reconstructionpaper}.
  Then $B\left[\log t\right]$ becomes an ind-$\cal{G}$-$k_{0}$-algebra as follows:
  The multiplication is the colimit the maps
  \begin{equation*}
    B\left[\log t\right]^{\leq n}
    \widehat{\otimes}_{k_{0}}
    B\left[\log t\right]^{\leq n}
    \to
    B\left[\log t\right]^{\leq 2n},
  \end{equation*}
  which are given by the usual multiplication of polynomials in the variable
  $\log t$. Its unit is
  \begin{equation*}
    k_{0} \stackrel{1_{B}}{\to} B = B\left[\log t\right]^{\leq 0} \to B\left[\log t\right],
  \end{equation*}
  where $1_{B}$ denotes the unit of $B$.
\end{lem}

\begin{defn}
  For every $q\in\NN$,
  \begin{align*}
    B_{\pdR}^{q,+} &:=B_{\dR}^{q,+}\left[\log t\right], \\
    B_{\pdR}^{>q,+} &:=B_{\dR}^{>q,+}\left[\log t\right], \\
    B_{\pdR}^{\dagger,+} &:=\text{``}\varinjlim_{q}\text{"}B_{\pdR}^{q,+}
    \stackrel{\text{\ref{lem:colimitAdRgreaterthanq}}}{\cong}\text{``}\varinjlim_{q}\text{"}B_{\pdR}^{>q,+}, \\
    B_{\pdR}^{q} &:= B_{\pdR}^{q,+} \widehat{\otimes}_{A_{\dR}^{1}}
  \text{``}\varinjlim_{t\times}\text{"}A_{\dR}^{1} \\
    B_{\pdR}^{>q} &:=B_{\pdR}^{>q,+} \widehat{\otimes}_{A_{\dR}^{1}}
  \text{``}\varinjlim_{t\times}\text{"}A_{\dR}^{1}, \text{ and} \\
    B_{\pdR}^{\dagger} &:=B_{\pdR}^{\dag,+} \widehat{\otimes}_{A_{\dR}^{1}}
  \text{``}\varinjlim_{t\times}\text{"}A_{\dR}^{1}.
  \end{align*}
  We view all of these objects as ind-$\cal{G}$-$k_{0}$-Banach algebras,
  equipped with the induced multiplications.
\end{defn}

\begin{defn}
  $B_{\pdR}^{\dag,+}$ is the \emph{positive overconvergent almost de Rham period ring}.
  $B_{\pdR}^{\dag}$ is the \emph{overconvergent almost de Rham period ring}.
\end{defn}

Fix an affinoid perfectoid $\left(C,\cal{O}_{C}\right)$-algebra $\left(R,R^{+}\right)$.

\begin{defn}
  For every $q\in\NN$,
  \begin{align*}
    \BB_{\pdR}^{q,+}\left(R,R^{+}\right)
      &:= \BB_{\dR}^{q,+}\left(R,R^{+}\right) \widehat{\otimes}_{B_{\dR}^{q,+}} B_{\pdR}^{q,+} \\
    \BB_{\pdR}^{>q,+}\left(R,R^{+}\right)
      &:= \BB_{\dR}^{>q,+}\left(R,R^{+}\right) \widehat{\otimes}_{B_{\dR}^{>q,+}} B_{\pdR}^{>q,+} \\
    \BB_{\pdR}^{\dagger,+}\left(R,R^{+}\right)
      &:=  \BB_{\dR}^{\dagger,+}\left(R,R^{+}\right) \widehat{\otimes}_{B_{\dR}^{\dagger,+}} B_{\pdR}^{\dagger,+} \\    
    \BB_{\pdR}^{q}\left(R,R^{+}\right)
      &:= \BB_{\dR}^{q}\left(R,R^{+}\right) \widehat{\otimes}_{B_{\dR}^{q}} B_{\pdR}^{q} \\    
    \BB_{\pdR}^{>q}\left(R,R^{+}\right)
      &:= \BB_{\dR}^{>q}\left(R,R^{+}\right) \widehat{\otimes}_{B_{\dR}^{>q}} B_{\pdR}^{>q}, \text{ and} \\
    \BB_{\pdR}^{\dagger}\left(R,R^{+}\right)
      &:= \BB_{\dR}^{\dagger}\left(R,R^{+}\right) \widehat{\otimes}_{B_{\dR}^{\dag}} B_{\pdR}^{\dag}.
  \end{align*}
  We view all of these objects as ind-$\cal{G}$-$k_{0}$-Banach algebras,
  equipped with the induced multiplications.
\end{defn}

Recall Definition~\ref{defn:bornologicalspace-reconstructionpaper}
and~\ref{defn:bornology-countable-basis-reconstructionpaper}.

\begin{lem}\label{lem:BpdRdagplus-bornology-countable-basis-reconstructionpaper}
  $\BB_{\pdR}^{\dag,+}\left(R,R^{+}\right)$ is bornological.
  Its bornology has countable basis.
\end{lem}

\begin{proof}
  Write $B^{q,+}:=\BB_{\dR}^{q,+}\left(R,R^{+}\right)$.
  $\BB_{\pdR}^{\dag,+}\left(R,R^{+}\right)$ is by definition the colimit of the diagram
  \begin{equation*}
    B^{2,+}\left[ \log t \right]^{\leq2}
    \longrightarrow
    B^{3,+}\left[ \log t \right]^{\leq3}
    \longrightarrow
    B^{4,+}\left[ \log t \right]^{\leq4}
    \longrightarrow
    \dots
  \end{equation*}
  where the maps in this diagram are injective
  by Proposition~\ref*{prop:BdRdagplus-bornology-countable-basis-reconstructionpaper}.
\end{proof}

\begin{lem}\label{lem:BpdRdag-bornology-countable-basis-reconstructionpaper}
  $\BB_{\pdR}^{\dag}\left(R,R^{+}\right)$ is bornological.
  Its bornology has countable basis.
\end{lem}

\begin{proof}
  Write $B^{q}:=\BB_{\dR}^{q,+}\left(R,R^{+}\right)$.
  $\BB_{\pdR}^{\dag}\left(R,R^{+}\right)$ is by definition the colimit of the diagram
  \begin{equation*}
    B^{2}\left[ \log t \right]^{\leq2}
    \stackrel{ t \times }{\longrightarrow}
    B^{3}\left[ \log t \right]^{\leq3}
    \stackrel{ t \times }{\longrightarrow}
    B^{4}\left[ \log t \right]^{\leq4}
    \stackrel{ t \times }{\longrightarrow}
    \dots
  \end{equation*}
  where the maps in this diagram are injective
  by Proposition~\ref{prop:BdRdagplus-bornology-countable-basis-reconstructionpaper}
  and Corollary~\ref{cor:tnonzerodivisor-on-BdRqplus-reconstructionpaper}.
\end{proof}

%%%%%%%%%%%%%%%%%%%%%%%%%%%%%%%%%%%%%%%%%%%%%%%%%%%%%%%%%%%%
%%%%%%%%%%%%%%%%%%%%%%%%%%%%%%%%%%%%%%%%%%%%%%%%%%%%%%%%%%%%
% Inverting p
%%%%%%%%%%%%%%%%%%%%%%%%%%%%%%%%%%%%%%%%%%%%%%%%%%%%%%%%%%%%
%%%%%%%%%%%%%%%%%%%%%%%%%%%%%%%%%%%%%%%%%%%%%%%%%%%%%%%%%%%%

\section{Period sheaves}
\label{subsec:locan-period-sheaf}

%%%%%%%%%%%%%%%%%%%%%%%%%%%%%%%%%%%%%%%%%%%%%%%%%%%%%%%%%%%%
% Period rings and period sheaves
%%%%%%%%%%%%%%%%%%%%%%%%%%%%%%%%%%%%%%%%%%%%%%%%%%%%%%%%%%%%

\subsection{The positive overconvergent de Rham period sheaf}
\label{subsubsec:defn-periodsheaves-reconstructionpaper}

\subsubsection{Definitions}

Fix a locally Noetherian adic space $X$ over $\Spa(k,k^{\circ})$.
The constructions in the previous \S\ref{subsec:locan-period-ring} are functorial,
in the following sense. For an affinoid perfectoid $U\in X_{\proet}$ with
$\widehat{U}=\Spa\left(R,R^{+}\right)$,
Proposition~\ref{lem:sheavesonXproet-over-affinoidperfectoid}
gives $\left(R,R^{+}\right)=\left(\hO(U),\widehat{\mathcal{O}}^{+}(U)\right)$.
Thus on $X_{\proet,\affperfd}^{\fin}$ and for all $q\in\NN$, we get the presheaves
\begin{align*}
  \A_{\dR}^{q, \psh}\colon U &\mapsto \A_{\dR}^{q}\left(\hO(U),\widehat{\mathcal{O}}^{+}(U)\right), \\
  \A_{\dR}^{>q, \psh}\colon U &\mapsto \A_{\dR}^{>q}\left(\hO(U),\widehat{\mathcal{O}}^{+}(U)\right) \\
  \intertext{of $W(\kappa)$-Banach algebras,}
  \A_{\dR}^{\infty,\psh}\colon U &\mapsto \A_{\dR}^{\infty}\left(\hO(U),\widehat{\mathcal{O}}^{+}(U)\right)
\intertext{of $W(\kappa)$-ind-Banach algebras,}
  \BB_{\dR}^{q, +, \psh}\colon U &\mapsto \BB_{\dR}^{q, +}\left(\hO(U),\widehat{\mathcal{O}}^{+}(U)\right), \\
  \BB_{\dR}^{>q, +, \psh}\colon U &\mapsto \BB_{\dR}^{>q, +}\left(\hO(U),\widehat{\mathcal{O}}^{+}(U)\right) \\
\intertext{of $k_{0}$-Banach algebras, and}
  \BB_{\dR}^{\infty,+, \psh}\colon U &\mapsto \BB_{\dR}^{\infty,+}\left(\hO(U),\widehat{\mathcal{O}}^{+}(U)\right)
\end{align*}
of $k_{0}$-ind-Banach algebras. Write
$\A_{\dR}^{\dag,\psh}:=\A_{\dR}^{\infty,\psh}$ and
$\BB_{\dR}^{\dag,+, \psh}:=\BB_{\dR}^{\infty,+, \psh}$.

We view all the presheaves of Banach algebras
above as presheaves of ind-Banach algebras,
cf. Lemma~\ref{lem:banach-to-indbanach-sheaves}.
Now sheafification is allowed:
Denote the sheafifications of all the presheaves above by
%$\A_{\inf}$,
$\A_{\dR}^{q}$,
$\A_{\dR}^{>q}$,
$\A_{\dR}^{\dag}=\A_{\dR}^{\infty}$,
%$\widehat{\BB}_{\inf}$,
$\BB_{\dR}^{q, +}$, $\BB_{\dR}^{>q, +}$,
and $\BB_{\dR}^{\dag,+}=\BB_{\dR}^{\infty,+}$
%$\BB_{\dR}^{q}$, $\BB_{\dR}^{>q}$, and $\BB_{\dR}^{\dag}$
for $q\in\NN$,
respectively. Since sheafification is strongly monoidal,
see Lemma~\ref{lem:sh-strongly-monoidal}, all
of these sheaves are sheaves of $W(\kappa)$-ind-Banach
algebras and $k_{0}$-ind-Banach algebras, respectively.
By Lemma~\ref{lem:sheaves-on-Xproet-and-Xproetaffperfdfin},
these extend to sheaves of ind-Banach algebras on the pro-étale site $X_{\proet}$.

\begin{defn}\label{defn:Blaplus-Bla}
  $\BB_{\dR}^{\dag,+}$ is the
  \emph{positive overconvergent de Rham period sheaf}.
\end{defn}

We can make the sections of $\BB_{\dR}^{\dag,+}$ explicit.
We even have two proofs for this explicit description.
This is the content of the Theorems~\ref{thm:subsections-periodsheaves-affperfd}
and~\ref{thm:BdR>qplus+-sections-over-affperfd-recpaper} below.

%%%%%%%%%%%%%%%%%%%%%%%%%%%%%%%%%%%%%%%%%%%%%%%%%%%%%%%%%%%%
% Period rings and period sheaves
%%%%%%%%%%%%%%%%%%%%%%%%%%%%%%%%%%%%%%%%%%%%%%%%%%%%%%%%%%%%

\subsubsection{One the sections of $\BB_{\dR}^{\dag,+}$ and $\BB_{\dR}^{q,+}$}
\label{subsubsec:sectionsBBdRdag+-BBdRq+}

Here in \S\ref{subsubsec:sectionsBBdRdag+-BBdRq+}, we prove:

\begin{thm}\label{thm:subsections-periodsheaves-affperfd}
  Fix $q\in\NN_{\geq 2}$,
  together with an affinoid perfectoid $U\in X_{\proet}$.
  Let $\widehat{U}=\Spa\left( R , R^{+}\right)$, where
  $\left( R , R^{+}\right)$ denotes an affinoid perfectoid
  algebra over an affinoid perfectoid field $\left( K , K^{+}\right)$.
  Then the canonical morphism
  \begin{equation*}
    \BB_{\dR}^{q,+}\left( R, R^{+} \right) \stackrel{\cong}{\longrightarrow} \BB_{\dR}^{q,+}(U)
  \end{equation*}
  is an isomorphism of $k_{0}$-ind-Banach algebras. Consequently,
  \begin{equation*}
    \BB_{\dR}^{\dag,+}\left( R, R^{+} \right) \stackrel{\cong}{\longrightarrow} \BB_{\dR}^{\dag,+}(U)
  \end{equation*}
  is an isomorphism of $k_{0}$-ind-Banach algebras.
\end{thm}

From now on, we work towards the proof of Theorem~\ref{thm:subsections-periodsheaves-affperfd},
which we complete on page~\pageref{proof:thm:subsections-periodsheaves-affperfd}.
Our arguments rely on the following Proposition~\ref{prop:abstract-BdRqpluspsh-is-sheaf}.

\begin{prop}\label{prop:abstract-BdRqpluspsh-is-sheaf}
  For every $q\in\NN_{\geq2}$,  $\BB_{\dR}^{q,+,\psh}$
  is a sheaf of $k_{0}$-Banach spaces on $X_{\proet,\affperfd}^{\fin}$.
\end{prop}

Now we prepare the proof of Proposition~\ref{prop:abstract-BdRqpluspsh-is-sheaf},
which itself is presented on page~\pageref{proof:prop:abstract-BdRqpluspsh-is-sheaf}.
Until then, we forget about all norms and work in a purely algebraic setting.
Furthermore, we fix an affinoid perfectoid field
$\left(K,K^{+}\right)$ such that $K$ is the completion of an algebraic extension of
$k$. Define the localised 
site $Y:=X_{\proet,\affperfd}^{\fin}/X_{K}\cong X_{K,\proet,\affperfd}^{\fin}$,
cf.~\cite[Proposition 3.15]{Sch13pAdicHodge}.
Write $A_{\inf}:=\A_{\inf}\left(K,K^{+}\right)$.
Let $\mathfrak{m}\subseteq A_{\inf}$ denote the
ideal generated by all $\left[\left(p^{\flat}\right)^{1/p^{N}}\right]$.
By~\cite[\S 2.1.6]{GabberRameroalmostringtheory2003},
$\left(A_{\inf},\mathfrak{m}\right)$ is an \emph{almost setup}.
We refer the reader to appendix~\ref{subsec:almostmaths-appendix}
for a reminder of basic facts which we use
without further reference in the remainder of \S\ref{subsec:locan-period-sheaf}.

  From now on, we fix the setup $\left(A_{\inf},\mathfrak{m}\right)$.
    
  \begin{lem}\label{lem:abstract-Bla-is-sheaf-1}
    The presheaf $\A_{\inf}[\zeta]\colon U\mapsto \A_{\inf}(U)[\zeta]$
    is a sheaf of (abstract) $A_{\inf}$-algebras on $Y$ such that
    $\Ho^{i}\left(U,\A_{\inf}[\zeta]\right)$ is almost zero for every $i>0$ and for all $U\in Y$.
  \end{lem}
  
  \begin{proof}    
    By Lemma~\ref{lem:coprod-almostsheaves},
    $\A_{\inf}[\zeta]
    =\bigoplus_{\alpha\in\NN^{d}}\A_{\inf}\zeta^{\alpha}
    \cong\varinjlim_{I}\bigoplus_{\alpha\in I}\A_{\inf}\zeta^{\alpha}$,
    where the colimit runs over all finite subsets $I\subseteq\NN^{d}$.
    Thus
    \begin{equation*}
      \Ho^{i}\left(U,\A_{\inf}[\zeta]\right)^{\al}
      =\varinjlim_{I}\Ho^{i}\left(U,\bigoplus_{\alpha\in I}\A_{\inf}\zeta^{\alpha}\right)^{\al}
      =\varinjlim_{I}\bigoplus_{\alpha\in I}\Ho^{i}\left(U,\A_{\inf}\zeta^{\alpha}\right)^{\al}
      =0,
    \end{equation*}
    using~\cite[\href{https://stacks.math.columbia.edu/tag/0739}{Tag 0739}]{stacks-project}
    and~\cite[Theorem 6.5(ii)]{Sch13pAdicHodge}.
  \end{proof}

  We aim to reconstruct $\BB_{\dR}^{q,+,\psh}$ out of $\A_{\inf}[\zeta]$
  via the following result.
  
  \begin{lem}\label{lem:abstract-Bla-is-sheaf-2}
    Consider a sheaf $\cal{F}$ of almost $A_{\inf}$-modules on $Y$
    with $\Ho^{i}\left(U,\cal{F}\right)=0$ for every $i>0$ and for all $U\in Y$
    which carries a monomorphism $\phi\colon\cal{F}\to\cal{F}$. Then
    $U\mapsto\coker\phi(U)$ is a sheaf of almost $A_{\inf}$-modules
    with vanishing higher cohomology.
  \end{lem}
  
  \begin{proof}
    Define $\cal{G}$ to be the sheafification of the presheaf
    $U\mapsto\coker\phi(U)$. Because sheafification is exact,
    we get a short exact sequence
    $0\to\cal{F}\stackrel{\phi}{\to}\cal{F}\to\cal{G}\to0$. The result
    follows from the associated long exact sequence.
    \end{proof}
  
    Write $\A_{\inf}(U)\left[\xi/p^{q}\right]:=\A_{\inf}(U)[\zeta]/\left(p^{m}\zeta-\xi\right)$
    for every $U\in Y$.
      
  \begin{lem}\label{lem:abstract-Bla-is-sheaf-4}
    The canonical maps $\A_{\inf}(U)\left[\xi/p^{q}\right]\to\BB_{\inf}(U)$
    are injective.
  \end{lem}

  \begin{proof}
    We consider a polynomial $f=\sum_{\alpha\in\NN}f_{\alpha}\zeta\in\A_{\inf}(U)[\zeta]$
    such that $f\left(\xi / p^{q}\right)=0$ in $\BB_{\inf}(U)$. We have to show that
    there exists a polynomial $g$ with coefficients in $\A_{\inf}(U)$ such that
    $f=g(p^{q}\zeta-\xi)$.
    
    In a general commutative ring, the polynomial division algorithm only works
    when the divisor is monic. Because $(p^{q}\zeta-\xi)$ is not monic,
    we choose to work in the ring $\BB_{\inf}(U)[\zeta]$. It contains $\A_{\inf}(U)[\zeta]$
    as a subring, by Lemma~\ref{Ainf-strictpring}, so this is allowed. This is how we find
    a polynomial $g^{\prime}$ with coefficients in $\BB_{\inf}(U)$ such that
    $f=g^{\prime}(\zeta-\xi / p^{q})$. Thus it suffices to show that the
    coefficients $g_{\alpha}$ of $g:=p^{q}g^{\prime}$ lie $\A_{\inf}$. Indeed, one computes
    \begin{equation*}
      f_{0} = - g_{0}\frac{\xi}{p^{q}} \quad\text{and}\quad
      f_{\alpha} = g_{\alpha-1}-g_{\alpha}\frac{\xi}{p^{q}}
      \text{ for all $\alpha\in\NN_{\geq1}$.}
    \end{equation*}
    Suppose that $g_{0}$ does not lie
    in $\A_{\inf}(U)$. One may write $g_{0}=p^{-n}g^{\prime}$ with
    $g^{\prime}\in\A_{\inf}$, not divisible by $p^{n}$. Then
    $\xi g_{0}^{\prime}=-f_{0}p^{q+n}$ such that Lemma~\ref{lem:Ainf-divisionbyxi-andp}
    implies that $p^{q+n}$ divides $g_{0}^{\prime}$, contradiction. That is,
    $g_{0}\in\A_{\inf}(U)$. Using similar arguments, induction along
    $\alpha$ gives $g_{\alpha}\in\A_{\inf}(U)$ for all $\alpha\in\NN$.
  \end{proof}
    
  For any $s,l\in\NN$, we consider the following three presheaves of $A_{\inf}$-algebras on $Y$:
  \begin{equation*}
  \begin{split}
    \A_{\inf}\left[\frac{\xi}{p^{q}}\right]\colon
    U&\mapsto\A_{\inf}(U)\left[\frac{\xi}{p^{q}}\right], \\
    \left. \A_{\inf}\left[\frac{\xi}{p^{q}}\right] \middle/^{\psh} \xi^{s} \right.\colon
    U&\mapsto \A_{\inf}(U)\left[\frac{\xi}{p^{q}}\right]/\xi^{s}, \\
    \left. \left(\A_{\inf}\left[\frac{\xi}{p^{q}}\right] \middle/^{\psh} \xi^{s}\right)\middle/^{\psh}p^{q} \right.\colon
    U&\mapsto\left(\A_{\inf}(U)\left[\frac{\xi}{p^{q}}\right]/\xi^{s}\right)/p^{l}.
  \end{split}
  \end{equation*}
  The superscript $\psh$ highlights that we compute the quotients in
  the category of presheaves.
  
  \begin{lem}\label{lem1:apply-lem:abstract-Bla-is-sheaf-2-threetimes}
    $\A_{\inf}\left[\xi / p^{q}\right]^{\al}$ is a sheaf with vanishing higher cohomology.
  \end{lem}
  
  \begin{proof}
     $p^{q}\zeta-\xi\in\A_{\inf}(U)[\zeta]$
     is not a zero-divisor
     for any $U\in Y$. To prove this, let
     $f=\sum_{\alpha\in\NN}a_{\alpha}\zeta\in\A_{\inf}(U)[\zeta]$ such that
     $f\left(p^{q}\zeta-\xi\right)=0$. This implies $f_{0}p^{q}=0$,
     thus $f_{0}=0$, by Lemma~\ref{Ainf-strictpring}.
     Furthermore, $f_{\alpha-1}p^{q}-f_{\alpha}\xi=0$
     for all $\alpha\in\NN_{\geq1}$.
     Using that $\xi$ is not a zero-divisor, cf. Lemma~\ref{lem:kernel-theta},
     a simply induction gives $f_{\alpha}=0$ for all $\alpha\geq0$.
     That is $f=0$, as desired.
     
     We can now apply Lemma~\ref{lem:abstract-Bla-is-sheaf-2} to
     $\A_{\inf}[\zeta]^{\al}$ with
     $p^{q}\zeta-\xi\colon\A_{\inf}[\zeta]^{\al}\hookrightarrow\A_{\inf}[\zeta]^{\al}$,
     thanks to the previous paragraph.
  \end{proof}
  
  \begin{lem}\label{lem2:apply-lem:abstract-Bla-is-sheaf-2-threetimes}
    For any $s\in\NN$,
    $\left(\A_{\inf}\left[\xi / p^{q}\right] \middle/^{\psh} \xi^{s}\right)^{\al}$ is a sheaf
    with vanishing higher cohomology.
  \end{lem}
  
  \begin{proof}
    $\xi\in\BB_{\inf}$(U) is not a zero-divisor for any $U\in Y$,
    cf. Lemma~\ref{lem:kernel-theta}.
    Lemma~\ref{lem:abstract-Bla-is-sheaf-4} implies
    that $\xi^{s}\in\A_{\inf}(U)\left[ \xi / p^{m}\right]$
    is not a zero-divisor. Thanks to this and Lemma~\ref{lem1:apply-lem:abstract-Bla-is-sheaf-2-threetimes},
    we can now Lemma~\ref{lem:abstract-Bla-is-sheaf-2} to
    $\A_{\inf}\left[ \xi / p^{q}\right]^{\al}$ and
    $\xi^{s}\colon\A_{\inf}\left[\xi / p^{q}\right]^{\al}
      \hookrightarrow\A_{\inf}\left[ \xi / p^{q}\right]^{\al}$.
  \end{proof}  
  
  \begin{lem}\label{lem3:apply-lem:abstract-Bla-is-sheaf-2-threetimes}
    For any $s,l\in\NN$,
    $\left(\left(\A_{\inf}\left[\xi / p^{q}\right] \middle/^{\psh} \xi^{s}\right)\middle/^{\psh}p^{l}\right)^{\al}$
    is a sheaf with vanishing higher cohomology.
  \end{lem}
  
  \begin{proof}
    We claim that $p^{l}\in\A_{\inf}(U)\left[\frac{\xi}{p^{q}}\right]/\xi^{s}$
    is not a zero-divisor for any $U\in Y$. To prove this, we first note that
    $p$ is not a zero-divisor in $\A_{\inf}(U)\left[\frac{\xi}{p^{m}}\right]$,
    cf. Lemma~\ref{lem:abstract-Bla-is-sheaf-4}. Next,
    $\theta(p)=p\in\hO(U)$ is not zero, that is $\xi$ does not divide $p$.
    This implies the claim.
   
    We can now apply Lemma~\ref{lem:abstract-Bla-is-sheaf-2} to
    $\left(\A_{\inf}\left[\xi / p^{q}\right] \middle/^{\psh} \xi^{s}\right)^{\al}$
    carrying
    $p^{l}\colon
      \left(\A_{\inf}\left[\xi / p^{q}\right] \middle/^{\psh} \xi^{s}\right)^{\al}
      \hookrightarrow\left(\A_{\inf}\left[\xi / p^{q}\right] \middle/^{\psh} \xi^{s}\right)^{\al}$,
    thanks to the previous paragraph and Lemma~\ref{lem2:apply-lem:abstract-Bla-is-sheaf-2-threetimes}.
  \end{proof}  
  
  Define the presheaf
  \begin{equation*}
    \cal{F}:=
    \varprojlim_{s,l\in\NN}
    \left. \left(\A_{\inf}\left[\frac{\xi}{p^{q}}\right] \middle/^{\psh} \xi^{s}\right)\middle/^{\psh}p^{q} \right. .
  \end{equation*}
  Lemma~\ref{lem3:apply-lem:abstract-Bla-is-sheaf-2-threetimes} and~\ref{lem:lim-preserves-almost-sheafy}  
  imply that $\cal{F}^{\al}$is a sheaf of almost $A_{\inf}$-modules. This implies the second sentence
  in the following result:
  
  \begin{lem}\label{lemAdRqpshal-sheaf}
    We have the canonical isomorphism $\cal{F}\isomap\A_{\dR}^{q,\psh}$ of presheaves.
    Consequently, $\A_{\dR}^{q,\psh,\al}$ is a sheaf of almost $A_{\inf}$-modules on $Y$.
  \end{lem}
  
  \begin{proof}
    Fix $U\in Y$ and compute
    \begin{equation*}
      \cal{F}(U)
      \cong \varprojlim_{s\in\NN} \A_{\inf}(U)\left[\frac{\xi}{p^{q}}\right]/\left(\xi^{s},p^{s}\right)
      \cong \varprojlim_{s\in\NN} \A_{\inf}(U)\left[\frac{\xi}{p^{q}}\right]/\left(\xi,p\right)^{s}.
    \end{equation*}  
    We may therefore check that the following canonical map is an isomorphism:
    \begin{equation}\label{lem:abstract-Bla-is-sheaf-3-1}
      \A_{\dR}^{q,\psh}(U) \to
      \varprojlim_{s\in\NN} \A_{\inf}(U)\left[\frac{\xi}{p^{q}}\right]/\left(\xi,p\right)^{s}.
    \end{equation}  
    To do this, we consider the exact sequence
    \begin{equation}\label{eq:abstract-Bla-is-sheaf-3-2}
      \A_{\inf}(U)[\zeta]
      \to \A_{\inf}(U)[\zeta]
      \to \A_{\inf}(U)\left[\frac{\xi}{p^{q}}\right]
      \to 0
    \end{equation}
    of abstract $\A_{\inf}(U)$-modules. Equip
    $\A_{\inf}(U)[\zeta]$ with the
    $(p,\left[p^{\flat}\right])\stackrel{\text{(\ref{lem:kernel-theta})}}{=}\left(\xi,p\right)$-adic norm
    with base $p$, cf. Definition~\ref{defn:I-adic-seminorm-norm}.
    Let $\A_{\inf}(U)\left[\xi / p^{q}\right]$ carry
    the quotient seminorm. Thus, for $f\in\A_{\inf}(U)\left[\xi/p^{q}\right]$,
    \begin{equation*}
      \|f\|:=\inf \|\sum_{\alpha\in\NN}a_{\alpha}\zeta\|
    \end{equation*}
    where the infinum varies over all expressions
    $f=\sum_{\alpha\in\NN}a_{\alpha}\left(\xi / p^{m}\right)^{\alpha}$
    with $a_{\alpha}\in\A_{\inf}(U)$.
    It follows from~\cite[Lemma 3.9]{BBB16} that
    this turns~(\ref{eq:abstract-Bla-is-sheaf-3-2})
    into a costrictly exact complex of seminormed $\A_{\inf}$-modules.
    Therefore, in order to show that~(\ref{lem:abstract-Bla-is-sheaf-3-1})
    is an isomorphism, we may compute that
    the seminorm on $\A_{\inf}(U)\left[ \xi / p^{q}\right]$
    induces the $\left(\xi,p\right)$-adic topology.
    Indeed, applying the separated completion functor
    to~(\ref{eq:abstract-Bla-is-sheaf-3-2}) would then
    give the exact complex 
    \begin{equation*}
      \A_{\inf}(U)\<\zeta\>
      \to \A_{\inf}(U)\<\zeta\>
      \to 
      \begin{array}{l}
        \widehat{\A_{\inf}(U)\left[\frac{\xi}{p^{q}}\right]} \\
        \cong\varprojlim_{s\in\NN} \A_{\inf}(U)\left[\frac{\xi}{p^{q}}\right]/(\xi,p)^{s}
      \end{array}
      \to 0,
    \end{equation*}    
    and Lemma~\ref{lem:idealpmzeta-xi-closed-in-Ainflanglerzetaangle} would finish the proof.
    
    Now we check that the seminorm on $\A_{\inf}(U)\left[ \xi / p^{q}\right]$
    induces the $\left(\xi,p\right)$-adic topology.
    On $\A_{\inf}(U)\left[\xi / p^{q}\right]$, we denote the quotient norm
    by $\|\cdot\|^{\prime}$ and the $(\xi,p)$-adic norm with base $p$ by $\|\cdot\|$.
    Denote the epimorphism
    $\A_{\inf}(U)[\zeta]
      \to \A_{\inf}(U)\left[ \xi / p^{q}\right]$ by $\sigma$.
    Fix $f\in\A_{\inf}(U)\left[ \xi / p^{q}\right]$ with $\|f\|^{\prime}=p^{-s^{\prime}}$ and $\|f\|=p^{-s}$.
    In particular, $f\in(\xi,p)^{s}$, thus we may write
    $f=\sum_{i=0}^{s}f_{i}\xi^{i}p^{t-i}$. Pick lifts $\widetilde{f}_{i}$ of the
    $f_{i}$ along $\sigma$ such that $g:=\|\sum_{i=0}^{t}\widetilde{f}_{i}\xi^{i}p^{t-i}\|$
    is a lift of $f$. Then compute
    \begin{equation*}
      \|f\|^{\prime}
      \leq\|g\|
      \leq\max_{i=0,\dots,s}\|\widetilde{f}_{i}\xi^{i}p^{t-i}\|
      \leq p^{-s}\|\widetilde{f}_{i}\|
      \leq p^{-s}
      = \|f\|.
    \end{equation*}
    On the other hand, we may pick a lift $\widetilde{f}$
    of $f$ along $\sigma$ such that $\|\widetilde{f}\|=p^{s^{\prime}}$.
    Then the coefficients of $\widetilde{f}$ lie all in the ideal
    $\left(\xi,p\right)^{s^{\prime}}\subseteq\A_{\inf}(U)$ and we deduce
    that $\widetilde{f}$ lies in the ideal $\left(\xi,p\right)^{s^{\prime}}\subseteq\A_{\inf}(U)[\zeta]$.
    We find $\|f\|\leq p^{-s^{\prime}}=\|f\|$. This implies
    $\|f\|^{\prime}=\|f\|$, as desired.
  \end{proof}
  
  Having established Lemma~\ref{lemAdRqpshal-sheaf}, we now invert $p$.

  \begin{lem}\label{lemAdRqpshinvertp-sheaf}
    $\A_{\dR}^{q,\psh}[1/p]\colon U\mapsto\A_{\dR}^{q,\psh}(U)[1/p]$ is a sheaf of $B_{\inf}$-modules on $Y$.
  \end{lem}

  In the following proof, we write
  $B_{\inf}:=\BB_{\inf}\left(K,K^{+}\right)=A_{\inf}[1/p]$.
  
  \begin{proof}[Proof of Lemma~\ref{lemAdRqpshinvertp-sheaf}]
    We claim that $\left[p^{\flat}\right]$ divides $p$ in $\A_{\dR}^{q,\psh}(U)$.
    Without less of generality, it suffices to check that
    $\left[p^{\flat}\right]$ divides $p$ in the ring $A_{\dR}^{2}:=A_{\inf}\left\< \xi / p^{2}\right\>$.
    Fix $a\in A_{\inf}$ as in Lemma~\ref{lem:kernel-theta}
    such that $w:=pa+\eta$ is a lift of $\left[p^{\flat}\right]$
    along the map $A_{\inf}[\eta]\to A_{\inf}$, $\eta\mapsto\xi$.
    This $w$ is invertible in the ring $B_{\inf}\llbracket \eta \rrbracket$:
    \begin{equation*}
      w^{-1}
      =(pa)^{-1}\left(1+(pa)^{-1}\eta\right)^{-1}
      =(pa)^{-1}\sum_{i\geq0}(pa)^{-i}\eta^{i}.
    \end{equation*}
    Now rewrite
    $f:=\sum_{i\geq0}(pa)^{-i}\eta^{i}
    =\sum_{i\geq0}a^{-i}p^{i}\left( \eta / p^{2}\right)^{i}$
    and note that $\|a\|=1$ in $A_{\inf}$, because it is a unit.
    We observe that $f$ is an element in $A_{\inf}\left\< \eta / p^{m}\right\>$
    and we get that $[\pi]$ divides $p$ in this ring of convergent
    power series. Passing through the map
    $A_{\inf}\left\< \eta / p^{m} \right\>\to A_{\dR}^{2}$
    gives the claim.    
    
    Now Lemma~\ref{lemAdRqpshinvertp-sheaf} follows from
    Lemma~\ref{lem:ractusinvertibly-then-almostsheaf-implies-sheaf}.    
  \end{proof}

\begin{proof}[Proof of Proposition~\ref{prop:abstract-BdRqpluspsh-is-sheaf}]
\label{proof:prop:abstract-BdRqpluspsh-is-sheaf}
  Fix an affinoid perfectoid $U\in X_{\proet}$ together with
  a finite covering $\mathfrak{U}$ by affinoid perfectoids. We show that
  the following cochain complex of $k_{0}$-Banach spaces
  \begin{equation}\label{proof:prop:abstract-BdRqpluspsh-is-sheaf-toshow:exact}
    0 \to \BB_{\dR}^{q,+,\psh}\left(U\right)
    \to \prod_{V\in\mathfrak{U}}\BB_{\dR}^{q,+,\psh}\left(V\right)
    \to \prod_{W,W^{\prime}\in\mathfrak{U}}\BB_{\dR}^{q,+,\psh}\left(W\times_{U}W^{\prime}\right)
  \end{equation}
  is exact. Suppose that $\left(\hO(U),\widehat{\cal{O}}^{+}(U)\right)$ is a perfectoid affinoid
  $\left(K,K^{+}\right)$-algebra. We may then restrict to the localised
  site $Y=X_{\proet,\affperfd}^{\fin}/X_{K}$. Now Lemma~\ref{lemAdRqpshinvertp-sheaf}
  applies, giving that
  \begin{equation}\label{proof:prop:abstract-BdRqpluspsh-is-sheaf-toshow:exactdecompleted}
    0 \to \A_{\dR}^{q,\psh}\left(U\right)[1/p]
    \stackrel{\phi}{\to} \prod_{V\in\mathfrak{U}}\A_{\dR}^{q,\psh}\left(V\right)[1/p]
    \to \prod_{W,W^{\prime}\in\mathfrak{U}}\A_{\dR}^{q,\psh}\left(W\times_{U}W^{\prime}\right)[1/p]
  \end{equation}  
  is an exact complex of abstract $k_{0}$-vector spaces. As all products here are finite,
  since $\mathfrak{U}$ is a finite covering, (\ref{proof:prop:abstract-BdRqpluspsh-is-sheaf-toshow:exact})
  is the completion of~(\ref{proof:prop:abstract-BdRqpluspsh-is-sheaf-toshow:exactdecompleted}).
  By Lemma~\ref{lem:sepcompl-SNrmF-BanF-exact},
  we may therefore check that~(\ref{proof:prop:abstract-BdRqpluspsh-is-sheaf-toshow:exactdecompleted})
  is strictly exact as a cochain complex of normed $k$-vector spaces, which is what we do.
  
  Since~\ref{proof:prop:abstract-BdRqpluspsh-is-sheaf-toshow:exactdecompleted} is exact
  as a complex of abstract $k_{0}$-vector spaces, we only have to check that $\phi$ is strict.
  Its image is closed, as it is the kernel of a map. It remains to check that it is open.
  Similarly, we may check that the morphism $\varphi$
  in the following cochain complex is open:
  \begin{equation}\label{proof:prop:abstract-BdRqpluspsh-is-sheaf-toshow:exactdecompletedintegral}
    0 \to \A_{\dR}^{q,\psh}\left(U\right)
    \stackrel{\varphi}{\to} \prod_{V\in\mathfrak{U}}\A_{\dR}^{q,\psh}\left(V\right)
    \stackrel{\psi}{\to} \prod_{W,W^{\prime}\in\mathfrak{U}}\A_{\dR}^{q,\psh}\left(W\times_{U}W^{\prime}\right).
  \end{equation}    
  From the isomorphism~(\ref{lem:abstract-Bla-is-sheaf-3-1})
  and~\cite[\href{https://stacks.math.columbia.edu/tag/05GG}{Tag 05GG}]{stacks-project},
  we find that that $\A_{\dR}^{q}(U)$, $\prod_{V\in\mathfrak{U}}\A_{\dR}^{q}(V)$,
  and $\ker\psi$ carry the $\left(p,\xi\right)$-adic topology. Therefore,
  it suffices to compute that
  $\varphi\left(\left(p,\xi\right)\right)\subseteq\ker\psi$
  is open. This requires the following two
  Lemma~\ref{lem:open1-prop:abstract-BdRqpluspsh-is-sheaf}
  and~\ref{lem:open2-prop:abstract-BdRqpluspsh-is-sheaf}.
  
  \begin{lem}\label{lem:open1-prop:abstract-BdRqpluspsh-is-sheaf}
    $\left[p^{\flat}\right]^{2}\left(\left(p,\xi\right)\cap\ker\psi\right)
    \subseteq
    \varphi\left(\left(p,\xi\right)\right)$.
  \end{lem}

  \begin{proof}
    For any $Z\in Y$,
    the isomorphism~(\ref{lem:abstract-Bla-is-sheaf-3-1})
    and~\cite[\href{https://stacks.math.columbia.edu/tag/05GG}{Tag 05GG}]{stacks-project} give
    \begin{equation*}
      \A_{\dR}^{q}(Z)/(p,\xi) \cong \A_{\inf}(Z)\left[\frac{\xi}{p^{q}}\right]/(p,\xi)
    \end{equation*}
    Denote the projections $\A_{\dR}^{q}(Z)\to\A_{\inf}(Z)\left[ \xi / p^{q}\right]/(p,\xi)$
    by $\sigma$ and consider the commutative diagram
    \begin{equation*}
\adjustbox{scale= 0.9,center}{%
      \begin{tikzcd}
        0 \arrow{r} &
        \A_{\dR}^{q}(U) \arrow{r}{\varphi}\arrow{d}{\sigma} &
        \prod_{V\in\mathfrak{U}}\A_{\dR}^{q}(V) \arrow{r}{\psi}\arrow{d}{\sigma} &
        \prod_{W,W^{\prime}\in\mathfrak{U}}\A_{\dR}^{q}\left(W\times_{U} W^{\prime}\right)\arrow{d}{\sigma} \\
        0 \arrow{r} &
        \A_{\inf}(U)\left[\frac{\xi}{p^{q}}\right]/(p,\xi) \arrow{r} {\varphi^{\prime}}&
        \prod_{V\in\mathfrak{U}}\A_{\inf}(V)\left[\frac{\xi}{p^{q}}\right]/(p,\xi) \arrow{r}{\psi^{\prime}} &
        \prod_{W,W^{\prime}\in\mathfrak{U}}\A_{\inf}\left(W\times_{U} W^{\prime}\right)\left[\frac{\xi}{p^{q}}\right]/(p,\xi)
      \end{tikzcd}
}
    \end{equation*}
    Lemma~\ref{lemAdRqpshal-sheaf}
    and~\ref{lem3:apply-lem:abstract-Bla-is-sheaf-2-threetimes}
    say that its rows are almost exact with respect to a suitable almost setup
    $(A_{\inf},\mathfrak{m})$. Now let $f\in\left(p,\xi\right)\cap\ker\psi$
    be arbitrary. Then $\sigma(f)=0$, and there exists an element $\widetilde{f}\in\A_{\dR}^{q}(U)$
    such that $\varphi\left( \widetilde{f} \right) = \left[p^{\flat}\right]f$. We find
    $\sigma\left( \widetilde{f} \right)\in\ker\varphi^{\prime}$, thus the almost exactness
    of the diagram above implies $\left[p^{\flat}\right]\sigma\left( \widetilde{f} \right)=0$.
    We find $\left[p^{\flat}\right]\widetilde{f}\in\ker\sigma=(p,\xi)$ and
    \begin{equation*}
      \varphi\left(\left[p^{\flat}\right]\widetilde{f}\right)
      =\left[p^{\flat}\right]^{2}f,
    \end{equation*}
    as desired.
  \end{proof}
  
  \begin{lem}\label{lem:open2-prop:abstract-BdRqpluspsh-is-sheaf}
    The multiplication-by-$\left[p^{\flat}\right]$ map on $\prod_{V\in\mathfrak{U}}\A_{\dR}^{m}(V)$
    is open.
  \end{lem}
  
  \begin{proof}
    We claim that $(p,\xi)^{s} \subseteq \left[p^{\flat}\right](p,\xi)^{s}$.
    The ideal $(p,\xi)^{s+1}$ is generated by the elements
    $p^{(s+1)-i}\xi^{i}$ for $i=0,\dots,s+1$. As $\left[p^{\flat}\right]$ divides $p$,
    cf. the proof of Lemma~\ref{lemAdRqpshinvertp-sheaf}, we find
    $p^{(s+1)-i}\xi^{i}\in\left[p^{\flat}\right](p,\xi)^{s}$ for $i\neq0$.
    The case $i=0$ follows from the description of $\xi$ in
    Lemma~\ref{lem:kernel-theta} which gives
    $\xi^{s+1}=\left[p^{\flat}\right]\xi^{s}-ap\xi^{s}\in\left[p^{\flat}\right](p,\xi)^{s}$.
  \end{proof}
  
  Lemma~\ref{lem:open2-prop:abstract-BdRqpluspsh-is-sheaf} implies that
  $\left[p^{\flat}\right]^{2}\left(\left(p,\xi\right)\cap\ker\psi\right)\subseteq\ker\psi$
  is an open subset. Lemma~\ref{lem:open1-prop:abstract-BdRqpluspsh-is-sheaf} gives that
  $\varphi\left(\left(p,\xi\right)\right)$ is open.
  As explained above, this finishes the proof of Proposition~\ref{prop:abstract-BdRqpluspsh-is-sheaf}.
\end{proof}

\begin{proof}[Proof of Theorem~\ref{thm:subsections-periodsheaves-affperfd}]
\label{proof:thm:subsections-periodsheaves-affperfd}
  The first isomorphism comes directly from Proposition~\ref{prop:abstract-BdRqpluspsh-is-sheaf}.
  Together with Corollary~\ref{cor:filteredcol-inIndBan-stronglyexact}
  and Lemma~\ref{lem:calFsheaf-FcirccalF-sheaf-exactfunctor}, it also implies that
  $\BB_{\dR}^{\dag,+,\psh}$ is a sheaf of $k_{0}$-ind-Banach spaces
  on $X_{\proet,\affperfd}^{\fin}$.
  This gives the second isomorphism in Theorem~\ref{thm:subsections-periodsheaves-affperfd}.
\end{proof}

%%%%%%%%%%%%%%%%%%%%%%%%%%%%%%%%%%%%%%%%%%%%%%%%%%%%%%%%%%%%
% Period rings and period sheaves
%%%%%%%%%%%%%%%%%%%%%%%%%%%%%%%%%%%%%%%%%%%%%%%%%%%%%%%%%%%%

\subsubsection{One the sections of $\BB_{\dR}^{\dag,+}$ and $\BB_{\dR}^{>q,+}$}
\label{subsubsec:sectionsBBdRdag+-BBdRgreaterthanq+}

Here in \S\ref{subsubsec:sectionsBBdRdag+-BBdRgreaterthanq+}, we prove:

\begin{thm}\label{thm:BdR>qplus+-sections-over-affperfd-recpaper}
  Fix $q\in\NN_{\geq 2}$,
  together with an affinoid perfectoid $U\in X_{\proet}$.
  Let $\widehat{U}=\Spa\left( R , R^{+}\right)$, where
  $\left( R , R^{+}\right)$ denotes an affinoid perfectoid
  algebra over an affinoid perfectoid field $\left( K , K^{+}\right)$.
  Then we have the canonical isomorphisms
  \begin{align*}
    \A_{\dR}^{>q,+}\left( R, R^{+} \right) &\stackrel{\cong}{\longrightarrow} \A_{\dR}^{>q,+}(U) \text{ and}\\
    \BB_{\dR}^{>q,+}\left( R, R^{+} \right) &\stackrel{\cong}{\longrightarrow} \BB_{\dR}^{>q,+}(U)
  \end{align*}
  of $W(\kappa)$-ind-Banach algebras and $k_{0}$-ind-Banach algebras.
  Consequently,
  \begin{equation*}
    \BB_{\dR}^{\dag,+}\left( R, R^{+} \right) \stackrel{\cong}{\longrightarrow} \BB_{\dR}^{\dag,+}(U)
  \end{equation*}
  is an isomorphism of $k_{0}$-ind-Banach algebras.
\end{thm}

See page~\pageref{proof:prop:BdR>qplus+-sections-over-affperfd-recpaper}
below for the proof of Theorem~\ref{thm:BdR>qplus+-sections-over-affperfd-recpaper}.
We start with Definition~\ref{defn:poperator}.

\begin{defn}\label{defn:poperator}
  Given a filtered group $M$, $\p\left(M\right)$ is the topological group $M$
  with open neighbourhood basis
  \begin{equation*}
    \left\{ p^{n}M + \Fil^{n}M\right\}_{n\geq0}.
  \end{equation*}
  This construction is functorial. In particular, we may associate to a given cochain complex $M^{\bullet}$
  of filtered groups a cochain complex $\p\left(M^{\bullet}\right)$ of topological groups.
\end{defn}

\begin{lem}\label{lem:subsections-periodsheaves-affperfd-1}
  Consider a strictly exact sequence of filtered abelian groups
  \begin{equation*}
    M^{\bullet}\colon
    M^{\prime} \stackrel{d^{\prime}}{\longrightarrow}
    M \stackrel{d}{\longrightarrow}
    M^{\prime\prime}.
  \end{equation*}
  If
  %\begin{itemize}
    %\item[(i)] $M^{\prime\prime}$ has no $p$-torsion and
    %\item[(ii)]
  $M^{\prime\prime}/\Fil^{s}M^{\prime\prime}$ has no $p$-power torsion for all $s\geq0$,
  \begin{equation*}
    \p\left(M^{\bullet}\right)\colon
    \p\left(M^{\prime}\right) \stackrel{d^{\prime}}{\longrightarrow}
    \p\left(M\right) \stackrel{d}{\longrightarrow}
    \p\left(M^{\prime\prime}\right)
  \end{equation*}
  is strictly exact as a complex of topological abelian groups.
\end{lem}

%One part of Lemma~\ref{lem:filtration-strictly-exact-open-mapping-thm-reconstructiontheorem}
%has been proven in my thesis, see \emph{loc. cit.} Lemma 3.4.4.
%For the convenience of the reader, we include the full proof here.

\begin{proof}
    Since $M^{\bullet}$ is strictly exact as a complex of filtered abelian groups,
    we have the exact sequences
    \begin{equation*}
      M^{\bullet,s}\colon
      \frac{M^{\prime}}{\Fil^{s}M^{\prime}} \stackrel{d^{\prime,s}}{\longrightarrow}
      \frac{M}{\Fil^{s}M} \stackrel{d^{s}}{\longrightarrow}
      \frac{M^{\prime\prime}}{\Fil^{s}M^{\prime\prime}}
    \end{equation*}
    for every $s\geq0$. Equip every
    $N^{s}\in\left\{
    M^{\prime} / \Fil^{s}M^{\prime},
    M / \Fil^{s}M,
    M^{\prime\prime} / \Fil^{s}M^{\prime\prime}\right\}$
    with the $p$-adic filtration:
    \begin{equation*}
      \Fil^{l}N^{s}:=p^{l}N^{s} \text{ for all } l\in\NN.
    \end{equation*}
    Then $M^{\bullet,s}$ is strictly
    exact: This follows from the observations
    \begin{equation*}
      d^{s}\left(p^{s}\frac{M}{\Fil^{s}M}\right)
      =p^{s}d^{s}\left(\frac{M}{\Fil^{s}M}\right)
      =p^{s}\frac{M^{\prime\prime}}{\Fil^{s}M^{\prime\prime}}
    \end{equation*}
    and
    \begin{equation*}
      d^{\prime,s}\left( p^{s} \frac{M^{\prime}}{\Fil^{s}M^{\prime}} \right)
      =p^{s}d^{\prime,s}\left( \frac{M^{\prime}}{\Fil^{s}M^{\prime}} \right)
      =p^{s}\ker d^{s}
      =\left(p^{s}\frac{M}{\Fil^{s}M}\right)\ker d^{\prime,s}.
    \end{equation*}
    This implies that the complexes
    \begin{equation*}
      M^{\bullet,s}\colon
      \left. \frac{M^{\prime}}{\Fil^{s}M^{\prime}}\middle/p^{s} \right. \longrightarrow
      \left. \frac{M}{\Fil^{s}M}\middle/p^{s} \right. \longrightarrow 
      \left. \frac{M^{\prime\prime}}{\Fil^{s}M^{\prime\prime}}\middle/p^{s} \right.
    \end{equation*}
    are exact for all $s\geq0$. But these complexes are isomorphic to
    \begin{equation*}
      M^{\bullet,s}\colon
      \frac{M^{\prime}}{p^{s}M^{\prime}+\Fil^{s}M^{\prime}}\longrightarrow
      \frac{M}{p^{s}M+\Fil^{s}M} \longrightarrow
      \frac{M^{\prime\prime}}{p^{s}M^{\prime\prime}+\Fil^{s}M^{\prime\prime}}.
    \end{equation*}
    This gives
    \begin{equation*}
    \begin{split}
      d^{\prime}\left( p^{s}M^{\prime} + \Fil^{s}M^{\prime} \right)
      =d^{\prime}\left( M^{\prime} \right)
      \cap \left( p^{s}M + \Fil^{s}M\right) \text{ and} \\
      d\left( p^{s}M+\Fil^{s}M \right) \supset p^{s}M^{\prime\prime}+\Fil^{s}M^{\prime\prime}.
    \end{split}
    \end{equation*} 
    That is, $\p\left(M^{\bullet}\right)$ is strictly exact as a complex of topological abelian groups.
\end{proof}

In the following,
we often confuse $W(\kappa)$-Banach modules with their underlying topological groups.
Furthermore, we view every $\widetilde{\A}_{\dR}^{>q}\left(R,R^{+}\right)$ as a filtered algebra,
equipped with the $\xi/p^{q}$-adic filtration.

\begin{lem}\label{lem:pwidetildeAdRisAdR-reconstructionpaper}
  For any affinoid perfectoid $U$ with $\widehat{U}=\Spa\left( R , R^{+}\right)$,
  \begin{equation}\label{eq:pwidetildeAdRisAdR-reconstructionpaper}
    \p\left(\widetilde{\A}_{\dR}^{>q}\left(R,R^{+}\right)\right) = \A_{\dR}^{>q}\left(R,R^{+}\right)
  \end{equation}  
\end{lem}

\begin{proof}
  This is true by definition.
\end{proof}

\begin{proof}[Proof of Theorem~\ref{thm:BdR>qplus+-sections-over-affperfd-recpaper}]
\label{proof:prop:BdR>qplus+-sections-over-affperfd-recpaper}
  Assume that $\left(R,R^{+}\right)$ is an affinoid perfectoid $\left(K,K^{+}\right)$-algebra
  where $K$ is the completion of an algebraic extension of $k$ which is perfectoid.
  Then we may check the following fact: both $\A_{\dR}^{>q,\psh}$ and $\BB_{\dR}^{>q,\psh}$
  are sheaves on $X_{\proet,\affperfd}^{\fin}/X_{K}$, where we have used Notation~\ref{notatio:base-change-in-proet}.
  
  We start with the following claim:
  for every covering $\mathfrak{V}$ of a $V\in X_{\proet,\affperfd}^{\fin}/X_{K}$,
  \begin{equation}\label{eq:exactsequence--prop:BdR>qplus+-sections-over-affperfd-recpaper}
    0\to\widetilde{\A}_{\dR}^{>q,\psh}(V)
     \to\prod_{W\in\mathfrak{V}}\widetilde{\A}_{\dR}^{>q,\psh}(W)
     \to\prod_{Z,Z^{\prime}\in\mathfrak{V}}\widetilde{\A}_{\dR}^{>q,\psh}\left(Z\times_{V} Z^{\prime}\right)
  \end{equation}
  is a strictly exact sequence of filtered rings. Here, every ring carries
  the $\xi/p^{q}$-adic filtration, where $\xi\in\A_{\inf}\left(K,K^{+}\right)$ is a generator 
  of Fontaine's map.
  By~\cite[Chapter I, \S 4.2 page 31-32, Theorem 4(5)]{HuishiOystaeyen1996},
  which applies thanks to Lemma~\ref{lem:grAdRgreaterthanq-xioverpadicfiltration-reconstructionpaper}(i),
  we may check that its associated graded is exact. But this follows from
  Lemma~\ref{lem:grAdRgreaterthanq-xioverpadicfiltration-reconstructionpaper}(ii)
  and~\cite[Lemma 4.10(iii)]{Sch13pAdicHodge}.
  Thus~(\ref{eq:exactsequence--prop:BdR>qplus+-sections-over-affperfd-recpaper})
  is strictly exact.
  
  Now apply the operator $\p$ as in Definition~\ref{defn:poperator}
  to~(\ref{eq:exactsequence--prop:BdR>qplus+-sections-over-affperfd-recpaper}).
  By Lemma~\ref{lem:pwidetildeAdRisAdR-reconstructionpaper}, we get the complex
  \begin{equation}\label{eq:exactsequence2--prop:BdR>qplus+-sections-over-affperfd-recpaper}
    0\to\A_{\dR}^{>q,\psh}(V)
     \to\prod_{W\in\mathfrak{V}}\A_{\dR}^{>q,\psh}(W)
     \to\prod_{Z,Z^{\prime}\in\mathfrak{V}}\A_{\dR}^{>q,\psh}\left(Z\times_{V} Z^{\prime}\right).
  \end{equation}
  It is a strictly exact complex of topological abelian groups
  by Lemma~\ref{lem:subsections-periodsheaves-affperfd-1},
  which applies by
  Lemma~\ref{lem:AdRgreaterthanqUtimesS-isomapHomcontSAdRgreaterthanqU-assumptionsforlemmasatisfied-reconstructionpaper}.
  Hence~(\ref{eq:exactsequence2--prop:BdR>qplus+-sections-over-affperfd-recpaper})
  is a strictly exact complex of $W(\kappa)$-Banach modules.
  That is $\A_{\dR}^{>q,\psh}$ is a sheaf on $X_{\proet,\affperfd}^{\fin}$,
  as desired.
  Now apply $-\widehat{\otimes}_{W(\kappa)}k_{0}$
  to~(\ref{eq:exactsequence2--prop:BdR>qplus+-sections-over-affperfd-recpaper}).
  We get the complex
  \begin{equation*}
    0\to\BB_{\dR}^{>q,\psh}(V)
     \to\prod_{W\in\mathfrak{V}}\BB_{\dR}^{>q,\psh}(W)
     \to\prod_{Z,Z^{\prime}\in\mathfrak{V}}\BB_{\dR}^{>q,\psh}\left(Z\times_{V} Z^{\prime}\right)
  \end{equation*}  
  of $k_{0}$-Banach spaces.
  It is strictly exact by Corollary~\ref{cor:completed-localisation-strictlyexact},
  which applies by Lemma~\ref{lem:AdRgreaterthanqptorsionfree}.
  That is $\BB_{\dR}^{>q,\psh}$ is a sheaf on $X_{\proet,\affperfd}^{\fin}$,
  as desired.

  Fix the notation as in Theorem~\ref{thm:BdR>qplus+-sections-over-affperfd-recpaper}.
  As $\BB_{\dR}^{\dag,+}=\varinjlim_{q}\BB_{\dR}^{>q,+}$, we can deduce
  that the canonical
  \begin{equation*}
    \BB_{\dR}^{\dag,+}\left( R, R^{+} \right) \stackrel{\cong}{\longrightarrow} \BB_{\dR}^{\dag,+}(U)
  \end{equation*}
  is an isomorphism $k_{0}$-ind-Banach algebras.
\end{proof}

%%%%%%%%%%%%%%%%%%%%%%%%%%%%%%%%%%%%%%%%%%%%%%%%%%%%%%%%%%%%
%%%%%%%%%%%%%%%%%%%%%%%%%%%%%%%%%%%%%%%%%%%%%%%%%%%%%%%%%%%%
% Period rings and period sheaves
%%%%%%%%%%%%%%%%%%%%%%%%%%%%%%%%%%%%%%%%%%%%%%%%%%%%%%%%%%%%
%%%%%%%%%%%%%%%%%%%%%%%%%%%%%%%%%%%%%%%%%%%%%%%%%%%%%%%%%%%%

\subsection{Inverting $t$ and adding $\log t$}

Since $k^{\circ}$ does not contain
a compatible system of $p$th power roots of unity, the element $t$ does not exist
on the whole site $X_{\proet,\affperfd}^{\fin}$.
Therefore, we first work over the point $*:=\Spa\left(k,k^{\circ}\right)$.
Set
\begin{equation}\label{eq:defnBdRdag-BpdRdag+-Bpddag-overpt-recpaper}
  \BB_{\dR,*}^{\dag}
    :=\lambda^{-1}\cal{F}_{B_{\dR}^{\dag}},
  \BB_{\pdR,*}^{\dag,+}
    :=\lambda^{-1}\cal{F}_{B_{\pdR}^{\dag,+}}, \text{ and }
  \BB_{\pdR,*}^{\dag}
    :=\lambda^{-1}\cal{F}_{B_{\pdR}^{\dag}},
\end{equation}
with the notation as in~(\ref{eq:indBanachGaloisrep-toShonproetsite-reconstructionpaper})
on page~\pageref{eq:indBanachGaloisrep-toShonproetsite-reconstructionpaper}.
Also,
$B_{\dR}^{\dag}:=\BB_{\dR}^{\dag}\left(C,\cal{O}_{C}\right)$,
$B_{\pdR}^{\dag,+}:=\BB_{\pdR}^{\dag,+}\left(C,\cal{O}_{C}\right)$, and
$B_{\pdR}^{\dag}:=\BB_{\dR}^{\dag}\left(C,\cal{O}_{C}\right)$.
$\BB_{\dR,*}^{\dag}$, $\BB_{\pdR,*}^{\dag,+}$, and $\BB_{\pdR,*}^{\dag}$
are sheaves of $k_{0}$-ind-Banach algebras
on $*_{\proet}$. Next, denote the canonical projection $X_{\proet}\to *_{\proet}$ by $j$.
Then
\begin{equation}\label{eq:defnBdRdag-BpdRdag+-Bpddag-recpaper}
\begin{split}
  \BB_{\dR}^{\dag}&:=\BB_{\dR}^{\dag,+}
    \widehat{\otimes}_{j^{-1}\BB_{\dR,*}^{\dag,+}}j^{-1}\BB_{\dR,*}^{\dag} \\
  \BB_{\pdR}^{\dag,+}&:=\BB_{\dR}^{\dag,+}
    \widehat{\otimes}_{j^{-1}\BB_{\dR,*}^{\dag,+}}j^{-1}\BB_{\pdR,*}^{\dag,+}, \text{ and} \\
  \BB_{\pdR}^{\dag}&:=\BB_{\dR}^{\dag,+}
    \widehat{\otimes}_{j^{-1}\BB_{\dR,*}^{\dag,+}}j^{-1}\BB_{\pdR,*}^{\dag}
\end{split}
\end{equation}
are sheaves of $k_{0}$-ind-Banach algebras on $X_{\proet}$.

\begin{defn}\label{defn:namesfortheperiodsheaves-reconstructionpaper}
  $\BB_{\dR}^{\dag}$ is the
  \emph{overconvergent de Rham period sheaf}.
  $\BB_{\pdR}^{\dag,+}$ is the
  \emph{positive overconvergent almost de Rham period sheaf}.
  $\BB_{\pdR}^{\dag}$ is the
  \emph{overconvergent almost de Rham period sheaf}.
\end{defn}

\iffalse %%% do not need this anymore

\begin{lem}\label{lem:multonBdRdag-factors-throughBdRdagplus-reconstructionpaper}
  The multiplication $\BB_{\dR}^{\dag}\widehat{\otimes}_{k_{0}}\BB_{\dR}^{\dag}\to\BB_{\dR}^{\dag}$
  factors through a morphism
  \begin{equation}\label{eq:multonBdRdag-factors-throughBdRdagplus-reconstructionpaper}
    \BB_{\dR}^{\dag}\widehat{\otimes}_{\BB_{\dR}^{\dag,+}}\BB_{\dR}^{\dag}\to\BB_{\dR}^{\dag}.
  \end{equation}
\end{lem}

\begin{lem}\label{lem:multonBpdRdag-factors-throughBdRdagplus-reconstructionpaper}
  The multiplication $\BB_{\pdR}^{\dag}\widehat{\otimes}_{k_{0}}\BB_{\pdR}^{\dag}\to\BB_{\pdR}^{\dag}$
  factors through a morphism
  \begin{equation}\label{eq:multonBpdRdag-factors-throughBdRdagplus-reconstructionpaper}
    \BB_{\pdR}^{\dag}\widehat{\otimes}_{\BB_{\dR}^{\dag,+}}\BB_{\pdR}^{\dag}\to\BB_{\pdR}^{\dag}.
  \end{equation}
\end{lem}

Both Lemma~\ref{lem:multonBdRdag-factors-throughBdRdagplus-reconstructionpaper}
and~\ref{lem:multonBpdRdag-factors-throughBdRdagplus-reconstructionpaper} follow
directly from Lemma~\ref{lem:morphism-of-monoids-factorsthrough} below.

\fi %%% comment ends

\begin{lem}\label{lem:multonBpdRdag-factors-throughBdRdagplus-reconstructionpaper}
  The multiplication $\BB_{\pdR}^{\dag}\widehat{\otimes}_{k_{0}}\BB_{\pdR}^{\dag}\to\BB_{\pdR}^{\dag}$
  factors through a morphism
  \begin{equation}\label{eq:multonBpdRdag-factors-throughBdRdagplus-reconstructionpaper}
    \BB_{\pdR}^{\dag}\widehat{\otimes}_{\BB_{\dR}^{\dag,+}}\BB_{\pdR}^{\dag}\to\BB_{\pdR}^{\dag}.
  \end{equation}
\end{lem}

Lemma~\ref{lem:multonBpdRdag-factors-throughBdRdagplus-reconstructionpaper} follows
directly from Lemma~\ref{lem:morphism-of-monoids-factorsthrough} below.

\begin{lem}\label{lem:morphism-of-monoids-factorsthrough}
  Consider a morphism of commutative monoids $\phi\colon R\to S$ in a given additive symmetric
  monoidal category $\left(\C,\otimes,1\right)$. Then the multiplication
  $\mu\colon S\otimes S\to S$ factors through a morphism $S\otimes_{R}S\to S$.
\end{lem}

\begin{proof}
  Denote the induced $R$-module action map by $\act\colon R\otimes S\to S$.
  Unwrapping the definition of $S\otimes_{R}S$, we observe that one has to check
  %\begin{equation*}
    $\mu\circ\left( \left( \act \circ \swap \right) \otimes \id \right)
    =\mu\circ\left( \id \otimes \act \right)$,
  %\end{equation*}
  %as morphisms $S \otimes R \otimes S \to S$. Here,
  where $\swap\colon R\otimes S\isomap S\otimes R$ is the usual swap of entries.
  This is straightforward.
\end{proof}

%%%%%%%%%%%%%%%%%%%%%%%%%%%%%%%%%%%%%%%%%%%%%%%%%%%%%%%%%%%%
% Period rings and period sheaves
%%%%%%%%%%%%%%%%%%%%%%%%%%%%%%%%%%%%%%%%%%%%%%%%%%%%%%%%%%%%

\subsection{Sections over affinoid perfectoids $U\times S$}
\label{subsubsec:sectionsoveraffperfdsUtimesS-recpaper}

\begin{remark}
  All results in \S\ref{subsubsec:sectionsoveraffperfdsUtimesS-recpaper} include the case $d=0$.
\end{remark}

\begin{prop}\label{prop:AdRgreaterthanqUtimesS-isomapHomcontSAdRgreaterthanqU-reconstructionpaper}
  Fix $q\in\NN_{\geq 2}$. Suppose $X$ is affinoid and equipped with an étale morphism $X\to\TT^{d}$.
  Let $U\in X_{\proet}/\widetilde{\TT}^{d}$ denote an affinoid perfectoid.
  Then, for every profinite set $S$, functorially,
  \begin{equation*}
    \A_{\dR}^{>q}(U\times S)
    \isomap\intHom_{\cont}\left(S , \A_{\dR}^{>q}(U) \right).
  \end{equation*}
\end{prop}

\begin{proof}
  We have morphisms of $W(\kappa)$-Banach algebras, functorial in $U$ and $q$,
  \begin{equation*}
    \A_{\inf}(U\times S)
    \cong\Hom_{\cont}\left( S , \A_{\inf}(U) \right)
    \to\Hom_{\cont}\left(S , \A_{\dR}^{>q}(U)\right),
  \end{equation*}
  cf.~\cite[Corollary 6.6]{Sch13pAdicHodge}.
  By the definitions, these factor through the map
  \begin{equation*}
    \phi\colon\A_{\dR}^{>q}(U\times S)
    \to\intHom_{\cont}\left(S , \A_{\dR}^{>q}(U) \right).
  \end{equation*}
  Both the domain and codomain of $\phi$ carry the $\left(p,\ker\theta_{\dR}^{>q}\right)$-adic topology,
  and therefore it suffices to check that $\phi$ is an isomorphism of abstract rings.
  We may argue via the associated graded with respect to the $\xi/p^{q}$-adic filtration.
  Since $\xi/p^{q}$ is a regular element in both rings in question, we only have to consider
  the zeroth piece of the associated graded:
  \begin{equation*}
    \widehat{\cal{O}}^{+}(U\times S)
    \cong\A_{\dR}^{>q}(U\times S)/\frac{\xi}{p^{q}}
    \to\intHom_{\cont}\left(S , \A_{\dR}^{>q}(U) \right)/\frac{\xi}{p^{q}}
    \cong\intHom_{\cont}\left(S , \widehat{\cal{O}}^{+}(U) \right).
  \end{equation*}
  This is an isomorphism by~\cite[Corollary 6.6]{Sch13pAdicHodge}.
\end{proof}

%We continue with the following consequence of
%Proposition~\ref{prop:AdRgreaterthanqUtimesS-isomapHomcontSAdRgreaterthanqU-reconstructionpaper}.

\begin{lem}\label{lem:HomcontS-commutes-completed-localisation-overWkappa-reconstructionpaper}
  Consider a profinite $S$ and a $W(\kappa)$-Banach module $M$
  such that $\|pm\|=|p|\|m\|$ for all $m\in M$. Here, $\|-\|$
  denotes the norm on $M$ and $|-|$ is the absolut value on $k_{0}$.
  Then
  \begin{equation}\label{eq:HomcontS-commutes-completed-localisation-THEmap}
    \intHom_{\cont}\left( S , M \right) \widehat{\otimes}_{W(\kappa)} k_{0}
    \isomap \intHom_{\cont}\left( S , M \widehat{\otimes}_{W(\kappa)} k_{0} \right).
  \end{equation}
\end{lem}

\begin{proof}
  Consider the map
  \begin{equation*}
    \Phi\colon\Hom_{\lc}\left( S , M \right) \otimes_{W(\kappa)} k_{0}
    \isomap \Hom_{\lc}\left( S , M \otimes_{W(\kappa)} k_{0} \right).  
  \end{equation*}
  Here, $\Hom_{\lc}$ refers to the locally constant maps, equipped with the
  supremum norm. We claim that $\Phi$
  is an isomorphism of seminormed $k_{0}$-vector spaces.
    
  $M$ is $p$-torsion free because for every $m\in M$, $pm=0$ implies
  $0=\|pm\|=|p|\|m\|$, that is $m=0$. Thus $\Phi$ is injective.
  The compactness of $S$ implies the surjectivity of $\Phi$.
  $\Phi$ is by construction bounded. It remains to show that it is strict.
  Consider an arbitrary element
  $g\in\Hom_{\lc}\left( S , M \otimes_{W(\kappa)} k_{0}\right)$.
  Because $S$ is compact, we can assume without
  loss of generalist that $g=f\otimes 1/p^{j}$
  for some locally constant function $f\colon S \to M$ and $i\in\NN$.
  Fix the norm on $M\otimes_{W(\kappa)}k_{0}$ as in
  Lemma~\ref{lem:HomcontS-commutes-completed-localisation-somenorms}. Then,
  \begin{equation*}
    \|\Phi\left(g\right)\|
    =\sup_{s\in S}\frac{\|f(s)\|}{|p^{i}|}
    =\frac{\|f\|}{|p|^{i}}
    =\|g\|.
  \end{equation*}
  Applying Lemma~\ref{lem:HomcontS-commutes-completed-localisation-somenorms} again
  but to the module $\Hom_{\lc}\left( S , M \right) \otimes_{W(\kappa)} k_{0}$ implies that
  $\Phi$ is strict. It is thus an isomorphism of normed $k_{0}$-vector spaces.
  Now the completion of $\Phi$
  recovers~(\ref{eq:HomcontS-commutes-completed-localisation-THEmap}):
  \begin{equation*}
  \begin{split}
    \widehat{\Hom_{\lc}\left( S , M \right) \otimes_{W(\kappa)} k_{0}}
    \cong\widehat{\Hom_{\lc}\left( S , M \right)} \widehat{\otimes}_{W(\kappa)} k_{0}
    \cong \intHom_{\cont}\left( S , M \right) \widehat{\otimes}_{W(\kappa)} k_{0} \text{ and} \\
    \widehat{\Hom_{\lc}\left( S , M \otimes_{W(\kappa)} k_{0} \right)}
    \cong\intHom_{\cont}\left( S , \widehat{M \otimes_{W(\kappa)} k_{0}} \right)
    \cong\intHom_{\cont}\left( S , M \widehat{\otimes}_{W(\kappa)} k_{0} \right),
  \end{split}
  \end{equation*}
  where we have used the~\cite[proof of Proposition 8.19]{CSanalyticgeometry}.
  Therefore,~(\ref{eq:HomcontS-commutes-completed-localisation-THEmap})
  is an isomorphism.
\end{proof}

\begin{prop}\label{prop:BdRgreaterthanqUtimesS-isomapHomcontSAdRgreaterthanqU-reconstructionpaper}
  Fix $q\in\NN_{\geq 2}$. Suppose $X$ is affinoid and equipped with an étale morphism $X\to\TT^{d}$.
  Let $U\in X_{\proet}/\widetilde{\TT}^{d}$ denote an affinoid perfectoid.
  Then, for every profinite set $S$, functorially,
  \begin{equation*}
    \BB_{\dR}^{>q,+}(U\times S)
    \isomap\intHom_{\cont}\left(S , \BB_{\dR}^{>q,+}(U) \right).
  \end{equation*}
\end{prop}

\begin{proof}
  Invoke Proposition~\ref{prop:AdRgreaterthanqUtimesS-isomapHomcontSAdRgreaterthanqU-reconstructionpaper}
  and Lemma~\ref{lem:HomcontS-commutes-completed-localisation-overWkappa-reconstructionpaper},
  which applies by Lemma~\ref{lem:Agreaterthanq-multiplybyp-norm-reconstructionpaper}.
\end{proof}

\begin{prop}\label{prop:BdRdaggerplusUtimesS-isomapHomcontSAdRgreaterthanqU-reconstructionpaper}
  Fix $q\in\NN_{\geq 2}$. Suppose $X$ is affinoid and equipped with an étale morphism $X\to\TT^{d}$.
  Let $U\in X_{\proet}/\widetilde{\TT}^{d}$ denote an affinoid perfectoid.
  Then, for every profinite set $S$, functorially,
  \begin{equation*}
    \BB_{\dR}^{\dag,+}(U\times S)
    \isomap\intHom_{\cont}\left(S , \BB_{\dR}^{\dag,+}(U) \right).
  \end{equation*}
\end{prop}

\begin{proof}
  $\BB_{\dR}^{\dag,+}(V)=\text{``}\varinjlim\text{''}_{q}\BB_{\dR}^{>q,+}(V)$
  for all affinoid perfectoids $V\in X_{\proet}/\widetilde{\TT}^{d}$,
  cf. Theorem~\ref{thm:subsections-periodsheaves-affperfd},
  Theorem~\ref{thm:BdR>qplus+-sections-over-affperfd-recpaper},
  and Lemma~\ref{lem:colimitAdRgreaterthanq}.
  Since $\intHom_{\cont}\left(S,-\right)$ commutes with formal filtered
  colimits, the result follows from
  Proposition~\ref{prop:BdRgreaterthanqUtimesS-isomapHomcontSAdRgreaterthanqU-reconstructionpaper}.
\end{proof}

%%%%%%%%%%%%%%%%%%%%%%%%%%%%%%%%%%%%%%%%%%%%%%%%%%%%%%%%%%%%
% Period rings and period sheaves
%%%%%%%%%%%%%%%%%%%%%%%%%%%%%%%%%%%%%%%%%%%%%%%%%%%%%%%%%%%%

\subsection{Cohomology over affinoid perfectoids}
\label{subsubsec:cohoveraffperfd-recpaper}

Continue to fix the notation as in \S\ref{subsec:conventions-reconstructionpaper}.
In the following, we work over $C$ rather than more general perfectoid fields; this is sufficient
for the purposes of this article.
Furthermore, in \S\ref{subsubsec:cohoveraffperfd-recpaper},
we work in left hearts of quasi-abelian categories,
as this ensures the existence of enough injectives, cf.
\S\ref{subsec:LHindBanachmodules-reconstructionpaper}
and \S\ref{subsec:sheavesindbanspaces-reconstructionpaper}.

\begin{thm}\label{thm:derivedlocalsectionsofBdRdagplus-recostructionpaper}
  Fix an affinoid perfectoid $U\in X_{\proet}$.
  Let $\widehat{U}=\Spa\left( R , R^{+}\right)$, where
  $\left( R , R^{+}\right)$ denotes an affinoid perfectoid
  algebra over $\left( C , \cal{O}_{C}\right)$.
  Then the canonical morphism
  \begin{equation*}
    \I\left(\BB_{\dR}^{\dag}\left( R, R^{+} \right)\right)
    \stackrel{\cong}{\longrightarrow} \R\Gamma\left( U , \I\left(\BB_{\dR}^{\dag}\right)\right)
  \end{equation*}
  is an isomorphism in the derived category of $\IndBan_{\I\left(k_{0}\right)}$.
  In particular,
  $\BB_{\dR}^{\dag}\left( R, R^{+} \right)\isomap\BB_{\dR}^{\dag}(U)$.
\end{thm}

\begin{proof}%[Proof of Theorem~\ref{thm:derivedlocalsectionsofBdRdagplus-recostructionpaper}]
  Fix the notation as in~(\ref{eq:defnBdRdag-BpdRdag+-Bpddag-recpaper}).
  We may work on the localised site $X_{\proet}/X_{C}$ where we compute
  \begin{align*}
    \BB_{\dR}^{\dag}|_{X_{C}}
    &=\BB_{\dR}^{\dag,+}|_{X_{C}}
      \widehat{\otimes}_{j^{-1}\BB_{\dR,*}^{\dag,+}|_{X_{C}}} j^{-1}\BB_{\dR,*}^{\dag}|_{X_{C}} \\
    &\stackrel{\text{\ref{lem:rep-to-proetalesheaf-commutes-with-filtered-colimits}}}{\cong}
      \varinjlim_{t\times}
      \BB_{\dR}^{\dag,+}|_{X_{C}}
      \widehat{\otimes}_{j^{-1}\BB_{\dR,*}^{\dag,+}|_{X_{C}}} j^{-1}\BB_{\dR,*}^{\dag,+}|_{X_{C}}
    = \varinjlim_{t\times}\BB_{\dR}^{\dag,+}|_{X_{C}}.
  \end{align*}
  We may thus show that the presheaf $\BB_{\dR}^{\dag,\psh}$ given by
  \begin{equation*}
    V\mapsto\varinjlim_{t\times}\BB_{\dR}^{\dag,+,\psh}(V)=\BB_{\dR}^{\dag,\psh}(V)
  \end{equation*}
  is a sheaf on $X_{\proet,\affperfd}^{\fin}/X_{C}$ with vanishing higer \v{C}ech cohomology.
  Indeed, then Lemma~\ref{lem:strictlyexactcechcomplexes-sheafcohomology-reconstructionpaper}
  would imply Theorem~\ref{thm:derivedlocalsectionsofBdRdagplus-recostructionpaper}.  
  
  $\BB_{\dR}^{\dag,\psh}$ is a presheaf of bornological $k_{0}$-vector spaces
  whose bornology has countable basis, cf. Theorem~\ref{thm:BdRdagRRplus-bornological}.
  Thanks to a version of the open mapping theorem, cf.~\cite[Theorem 4.9]{Ba15},
  it suffices to check that the underlying presheaf $|\BB_{\dR}^{\dag,\psh}|$ of abstract $k_{0}$-vector spaces
  \begin{equation*}
    V\mapsto|\BB_{\dR}^{\dag,\psh}(V)|
  \end{equation*}
  is a sheaf on $X_{\proet}/X_{C}$ with vanishing higher \v{C}ech cohomology.
  In the following, we therefore only work with the underlying sheaves of abstract
  modules and omit the notation $|\cdot|$ for clarity.
  By Propositions~\ref{prop:underlyingspaceBdRdagplus-iso-uncompletedcolim-reconstructionpaper}
  and~\ref{thm:BdR>qplus+-sections-over-affperfd-recpaper},
  we may as well check that for every $q\in\NN_{\geq2}$,
  $\varinjlim_{t\times}\A_{\dR}^{>q}=\A_{\dR}^{>q}[1/t]$ has vanishing higher \v{C}ech cohomology
  as a sheaf on $X_{\proet}/X_{C}$.
  Here, $t$ is as in Definition~\ref{defn:t}.
  To do this, we fix a covering $\mathfrak{V}$ in $X_{\proet,\affperfd}^{\fin}/X_{C}$
  and consider the \v{C}ech complex
  \begin{equation*}
    C^{>q,\bullet}:=\check{C}^{\bullet}\left( \mathfrak{V} , \A_{\dR}^{>q} \right).
  \end{equation*}
  By the exactness of localisation and Lemma~\ref{lem:eta-operator-kills-torsion},
  it suffices to check that $\eta_{t/p^{q}}C^{>q,\bullet}$ is exact in strictly positive degrees.
  We would like to argue via a suitable associated graded. We consider
  $C^{>q,\bullet}$ again, but we equip the module in every degree
  with the $\xi/p^{q}$-adic filtration. In alignment with the notation as in
  Definition~\ref{defn:Wkappatriv-tildeAdRgreaterthanq},
  we denote the resulting complex by $\widetilde{C}^{>q,\bullet}$.
  That is, we have to check that $\eta_{t/p^{q}}\widetilde{C}^{>q,\bullet}$ is exact in strictly positive degrees.
  
  To do this, we first observe that
  Condition~\ref{cond:filteredmodulesdecalage-reconstructionpaper}
  is satisfied for the cochain complex of $W(\kappa)$-modules $\widetilde{C}^{>q,\bullet}$
  and $r=t/p^{q}$, where $W(\kappa)$ carries the trivial filtration:
  \begin{itemize}
    \item[(i)] It is concentrated in non-negative degrees.
    \item[(ii)] This is Lemma~\ref{lem:canetaoperatorAdRgreaterthanqRRplus-i},
    \item[(iii)] this is Lemma~\ref{lem:canetaoperatorAdRgreaterthanqRRplus-ii}, and
    \item[(iv)] this is Lemma~\ref{lem:principlesymbol-of-t}.
  \end{itemize}  
  Thus, Lemma~\ref{lem:decalagefilteredmodules-defined}
  and~\ref{lem:grAdRgreaterthanq-xioverpadicfiltration-reconstructionpaper}(i)
  imply that the induced filtration on
  $\eta_{t/p^{q}}\widetilde{C}^{>q,\bullet}$ is exhaustive, separated, and complete. To show its exactness,
  we may therefore show that its associated graded is exact,
  cf.~\cite[Chapter I, \S 4.1, page 31-32, Theorem 4]{HuishiOystaeyen1996}. Compute
  \begin{equation*}
    \gr\eta_{t/p^{q}}\widetilde{C}^{>q,\bullet}
    \stackrel{\text{\ref{lem:decalage-commutes-gr}}}{\cong}
      \eta_{\sigma\left(t/p^{q}\right)}\gr\widetilde{C}^{>q,\bullet}
    \underset{\text{\ref{lem:principlesymbol-of-t}}}{\stackrel{\text{\ref{lem:grAdRgreaterthanq-xioverpadicfiltration-reconstructionpaper}(ii)}}{\cong}}
      \eta_{\left(\zeta_{p}-1\right)\sigma\left(\xi/p^{q}\right)}
      \check{C}^{\bullet}\left( \mathfrak{V} ,
      \widehat{\cal{O}}^{+}\left[\sigma\left(\frac{\xi}{p^{q}}\right)\right] \right).
  \end{equation*}
  Thus the exactness of $\gr\eta_{t/p^{q}}\widetilde{C}^{>q,\bullet}$
  in strictly positive degrees follows
  from~\cite[Lemma 4.10(iii) and (v)]{Sch13pAdicHodge} and Lemma~\ref{lem:eta-operator-kills-torsion}.  
%%% comment ends
\end{proof}

In the following, we make implicit use of Notation~\ref{notatio:base-change-in-proet}.

\begin{cor}\label{cor:BdRdaggerUtimesS-isomapHomcontSAdRgreaterthanqU-reconstructionpaper}
  Fix $q\in\NN_{\geq 2}$. Suppose $X$ is affinoid and equipped with an étale morphism $X\to\TT^{d}$.
  Let $U\in X_{\proet}/\widetilde{\TT}_{C}^{d}$ denote an affinoid perfectoid.
  Then, for every profinite set $S$, functorially,
  \begin{equation*}
    \BB_{\dR}^{\dag}(U\times S)
    \isomap\intHom_{\cont}\left(S , \BB_{\dR}^{\dag}(U) \right).
  \end{equation*}
\end{cor}

\begin{proof}
  Theorem~\ref{thm:derivedlocalsectionsofBdRdagplus-recostructionpaper}
  implies $\BB_{\dR}^{\dag}(V)=\varinjlim_{t\times}\BB_{\dR}^{\dag,+}(V)$
  for all affinoid perfectoids $V$ over $\widetilde{\TT}_{C}^{d}$.
  As $\intHom_{\cont}\left(S,-\right)$ commutes with filtered
  colimits, the result follows from
  Proposition~\ref{prop:BdRdaggerplusUtimesS-isomapHomcontSAdRgreaterthanqU-reconstructionpaper}.
\end{proof}

\begin{lem}\label{lem:describeBpdRdagoverXC}
  We have the following isomorphism of sheaves of $k$-ind-Banach algebras
  \begin{equation}\label{eq:describeBpdRdagoverXC-themap}
    \bigoplus_{\alpha\in\NN}\BB_{\dR}^{\dag}|_{X_{C}}\left( \log t \right)^{\alpha}
    \isomap\BB_{\pdR}^{\dag}|_{X_{C}}.
  \end{equation}
\end{lem}

\begin{proof}
  By~\cite[Proposition 3.15]{Sch13pAdicHodge} and~\cite[Remark 2.5]{Bosco21},
  it suffices to compare the sections of the sheaves
  $\bigoplus_{\alpha\in\NN}\BB_{\dR}^{\dag}\left( \log t \right)^{\alpha}$
  and $\BB_{\pdR}^{\dag}$ over affinoid perfectoids
  of the form $\Spa\left(C,\cal{O}_{C}\right)\times S$ for any profinite set $S$,
  cf. Example~\ref{examplenotation:profinitesets-in-Xproet}. Here, we find
  with the notation as in~(\ref{eq:defnBdRdag-BpdRdag+-Bpddag-overpt-recpaper})
  on page~\pageref{eq:defnBdRdag-BpdRdag+-Bpddag-overpt-recpaper},
  \begin{equation}\label{lem:describeBpdRdagoverXC-intererstingcomputation}
  \begin{split}
    \BB_{\pdR}^{\dag}\left( \Spa\left(C,\cal{O}_{C}\right) \times S \right)
    &=\left(\lambda^{-1}\cal{F}_{B_{\pdR}^{\dag}}\right)\left( \Spa\left(C,\cal{O}_{C}\right) \times S \right) \\
    &\stackrel{\text{\ref{cor:comparepullbackalonglambda-reconstructionpaper}}}{\cong}
      \cal{F}_{B_{\pdR}^{\dag}}\left( \cal{G}\times S \right) \\
    &=\intHom_{\cont,\cal{G}}\left(\cal{G}\times S , B_{\pdR}^{\dag} \right) \\
    &=\intHom_{\cont}\left( S , B_{\pdR}^{\dag} \right) \\
    &\stackrel{\text{\ref{lem:finitedirectsum-comutes-withHomcont-reconstructionpaper}}}{\cong}
      \bigoplus_{\alpha\in\NN}\intHom_{\cont}\left( S , B_{\dR}^{\dag} \right)\left( \log t \right)^{\alpha} \\
    &\underset{\text{\ref{cor:BdRdaggerUtimesS-isomapHomcontSAdRgreaterthanqU-reconstructionpaper}}}{\stackrel{\text{\ref{thm:derivedlocalsectionsofBdRdagplus-recostructionpaper}}}{\cong}}
      \bigoplus_{\alpha\in\NN}\BB_{\dR}^{\dag}\left( \Spa\left(C,\cal{O}_{C}\right) \times S \right)\left( \log t \right)^{\alpha}
  \end{split}
  \end{equation}
  where $\cal{G}:=\Gal\left(\overline{k} / k \right)$.
  These isomorphisms are functorial, thus they give the desired~(\ref{eq:describeBpdRdagoverXC-themap}).
\end{proof}

\begin{cor}\label{cor:derivedlocalsectionsofBpdRdag-recostructionpaper}
  Fix an affinoid perfectoid $U\in X_{\proet}$.
  Let $\widehat{U}=\Spa\left( R , R^{+}\right)$, where
  $\left( R , R^{+}\right)$ denotes an affinoid perfectoid
  algebra over $\left( C , \cal{O}_{C}\right)$.
  Then the canonical morphism
  \begin{equation*}
    \I\left(\BB_{\pdR}^{\dag}\left( R, R^{+} \right)\right)
    \stackrel{\cong}{\longrightarrow} \R\Gamma\left( U , \I\left(\BB_{\pdR}^{\dag}\right)\right)
  \end{equation*}
  is an isomorphism in the derived category of $\IndBan_{\I\left(k_{0}\right)}$.
  In particular,
  $\BB_{\pdR}^{\dag}\left( R, R^{+} \right)\isomap\BB_{\pdR}^{\dag}(U)$.
\end{cor}

\begin{proof}
  Combining Lemma~\ref{lem:directsumsheaves-finitecoverings},
  Lemma~\ref{lem:describeBpdRdagoverXC}, and
  Theorem~\ref{thm:derivedlocalsectionsofBdRdagplus-recostructionpaper},
  we find the desired description of the sections of $\BB_{\pdR}^{\dag}$.
  From the proof of Theorem~\ref{thm:derivedlocalsectionsofBdRdagplus-recostructionpaper}
  and the exactness of direct sums, cf.~\cite[Proposition 2.1.15(b)]{Sch99}
  which applies by Lemma~\ref{lem:indBanFcirc-structureresult},
  we furthermore find that $\BB_{\pdR}^{\dag}$ has vanishing higher \v{C}ech cohomology
  as a sheaf on $X_{\proet,\affperfd}^{\fin}/X_{C}$.
  Now apply Lemma~\ref{lem:strictlyexactcechcomplexes-sheafcohomology-reconstructionpaper}.
\end{proof}

\begin{cor}\label{cor:BpdRdaggerUtimesS-isomapHomcontSAdRgreaterthanqU-reconstructionpaper}
  Fix $q\in\NN_{\geq 2}$. Suppose $X$ is affinoid and equipped with an étale morphism $X\to\TT^{d}$.
  Let $U\in X_{\proet}/\widetilde{\TT}_{C}^{d}$ denote an affinoid perfectoid.
  Then, for every profinite set $S$, functorially,
  \begin{equation*}
    \BB_{\pdR}^{\dag}(U\times S)
    \isomap\intHom_{\cont}\left(S , \BB_{\pdR}^{\dag}(U) \right).
  \end{equation*}
\end{cor}

\begin{proof}
  By Corollary~\ref{cor:derivedlocalsectionsofBpdRdag-recostructionpaper},
  $\BB_{\pdR}^{\dag}(V)
  =\bigoplus_{\alpha\geq0}\BB_{\dR}^{\dag}(V)\left(\log t\right)^{\alpha}$
  for all affinoid perfectoids $V$ over $\widetilde{\TT}_{C}^{d}$.
  Thus the result follows from
  Lemma~\ref{lem:finitedirectsum-comutes-withHomcont-reconstructionpaper} and
  Corollary~\ref{cor:BdRdaggerUtimesS-isomapHomcontSAdRgreaterthanqU-reconstructionpaper}.
\end{proof}

%%%%%%%%%%%%%%%%%%%%%%%%%%%%%%%%%%%%%%%%%%%%%%%%%%%%%%%%%%%%
% relative period rings
%%%%%%%%%%%%%%%%%%%%%%%%%%%%%%%%%%%%%%%%%%%%%%%%%%%%%%%%%%%%

\section{Comparison with $B_{\dR}^{+}$ and $\BB_{\dR}^{+}$}

\cite{Sch13pAdicHodge} works with the sheaf $\BB_{\dR}^{+}$.
For future reference, we explain that $\BB_{\dR}^{+}$
carries a canonical algebra structure from its overconvergent counterpart $\BB_{\dR}^{\dag,+}$.

%%%%%%%%%%%%%%%%%%%%%%%%%%%%%%%%%%%%%%%%%%%%%%%%%%%%%%%%%%%%
% relative period rings
%%%%%%%%%%%%%%%%%%%%%%%%%%%%%%%%%%%%%%%%%%%%%%%%%%%%%%%%%%%%

\subsection{Relative period rings}
\label{subsubsec:comp-scholzefunctorI-relativeperiodrings}

We recall constructions from~\cite[\S 6]{Sch13pAdicHodge}.
Consider the completion $K$ of an algebraic extension of $k$ which
is perfectoid. Pick a ring of integral elements $K^{+}\subseteq K$ 
containing $k^{\circ}$ and fix an affinoid perfectoid
$\left(K,K{+}\right)$-algebra $\left(R,R^{+}\right)$.
Consider Fontaine's map
$\theta_{\inf}\colon\A_{\inf}\left(R,R^{+}\right)\to R^{+}$
and invert $p$ to get $\vartheta_{\inf}\colon\BB_{\inf}\left(R,R^{+}\right)\to R$.
The~\emph{relative positive de Rham period ring} is
\begin{equation}\label{eq:Scholze-BdRplus}
  \BB_{\dR}^{+}\left(R,R^{+}\right)
  :=\varprojlim_{j\in\NN}\BB_{\inf}\left(R,R^{+}\right)/\left(\ker\vartheta_{\inf}\right)^{j}.
\end{equation}
\iffalse %%% not needed
If $K$ admits a compatible sequence of primitive
$p$th roots of unity $\epsilon\in\ K^{\flat}$,
\begin{equation}\label{eq:Scholze-BdR}
  \BB_{\dR}\left(R,R^{+}\right)
  :=\BB_{\dR}^{+}\left(R,R^{+}\right)\left[1/t\right]
\end{equation}
is the \emph{relative de Rham period ring}.
Here, $t$ is as in Definition~\ref{defn:t}.
\fi %%% comment ends
We aim to relate this period ring to the overconvergent one.

\begin{notation}
 $\A_{\inf}^{(p)}\left(R,R^{+}\right)$ is the seminormed $W(\kappa)$-algebra
 $\A_{\inf}\left(R,R^{+}\right)$, equipped with the $p$-adic
 seminorm. It is a Banach algebra by Lemma~\ref{Ainf-strictpring}.
\end{notation}

\begin{lem}\label{lem:Ainf-mod-xi-complete}
  $\A_{\inf}^{(p)}\left(R,R^{+}\right)/\left(\ker\theta_{\inf}\right)^{j}$ is complete for every $j\in\NN$.
\end{lem}

\begin{proof}
  By Lemma~\ref{lem:kernel-theta}, $\ker\theta_{\inf}=(\xi)$.
  Lemma~\ref{lem:Ainf-divisionbyxi-andp}(ii) implies
  \begin{equation}\label{eq:Ainf-mod-xi-complete-1}
    \|a\xi^{j}\|=\|a\|
  \end{equation}
  for all $a\in\A_{\inf}^{(p)}\left(R,R^{+}\right)$.
  We show that the ideal
  $\left(\xi^{j}\right)$ is closed.
  Consider a sequence $\left(f_{i}\xi^{j}\right)_{i\in\NN}\subseteq\left(\xi^{j}\right)$
  which converges to an element $h$. But then $\left( f_{i}\right)$ is Cauchy:
  \begin{equation*}
    \|f_{i}-f_{i+1}\|
    \stackrel{\text{(\ref{eq:Ainf-mod-xi-complete-1})}}{=}
    \|f_{i}\xi^{j} - f_{i+1}\xi^{j} \|
    \to 0 \text{ for $i\to\infty$},
  \end{equation*}
  and $h=\lim_{i\to\infty}\left(f_{i}\xi^{j}\right)=\left(\lim_{i\to\infty}f_{i}\right)\xi^{j}\in\left(\xi^{j}\right)$
  follows.
\end{proof}

\begin{notation}
 $\BB_{\inf}^{(p)}\left(R,R^{+}\right)$ is the seminormed $k_{0}$-algebra
 $\BB_{\inf}\left(R,R^{+}\right)$ with unit ball
 $\A_{\inf}\left(R,R^{+}\right)$. It is a Banach algebra by Lemma~\ref{Ainf-strictpring}.
\end{notation}

\begin{lem}\label{lem:Binfpmodkervarthetaj-complete}
  $\BB_{\inf}^{(p)}\left(R,R^{+}\right)/\left(\ker\vartheta_{\inf}\right)^{j}$
  is complete for every $j\in\NN$.
\end{lem}

\begin{proof}
  By Lemma~\ref{lem:localise-torsionfree-modules},
  \begin{equation*}
    \A_{\inf}^{(p)}\left(R,R^{+}\right)\widehat{\otimes}_{W(\kappa)}k_{0}
    \isomap\BB_{\inf}^{(p)}\left(R,R^{+}\right).
  \end{equation*}
  By Lemma~\ref{cor:completed-loc-exact-overF}, which applies because
  $\A_{\inf}^{(p)}\left(R,R^{+}\right)$ does not have $p$-torsion by Lemma~\ref{Ainf-strictpring},
  the result would follow once
  \begin{equation}\label{eq:Ainfp-xij-Ainfp}
    \A_{\inf}^{(p)}\left(R,R^{+}\right) \stackrel{\xi^{j}}{\longrightarrow}
    \A_{\inf}^{(p)}\left(R,R^{+}\right)
  \end{equation}
  is a strict monomorphism. It is injective by Lemma~\ref{lem:kernel-theta},
  has closed image by Lemma~\ref{lem:Ainf-mod-xi-complete} and is open onto its image
  because $\|a\xi\|=\|a\|$ for every $a\in\A_{\inf}^{(p)}\left(R,R^{+}\right)$,
  cf. Lemma~\ref{lem:Ainf-divisionbyxi-andp}(ii).
  (\ref{eq:Ainfp-xij-Ainfp}) is thus a strict monomorphism by Lemma~\ref{lem:BanFcirc-kercoker}(iii).
\end{proof}

$\A_{\inf}\left(R,R^{+}\right)$ still carries the $(p,\xi)$-adic topology.
Consider
\begin{equation*}
  \A_{\inf}\left(R,R^{+}\right)
  \to\A_{\inf}^{(p)}\left(R,R^{+}\right)/\left(\ker\theta_{\inf}\right)^{j}.
\end{equation*}
It map is a morphism of $W(\kappa)$-Banach algebras by
Lemma~\ref{lem:bounded-map-adic-rings}.
It lifts to a morphism
\begin{equation*}
  \BB_{\dR}^{q,+}\left(R,R^{+}\right)
  \to\BB_{\inf}^{(p)}\left(R,R^{+}\right)/\left(\ker\vartheta_{\inf}\right)^{j}
\end{equation*}
of $k_{0}$-Banach algebras for every $q\in\NN$.
Now compute the inverse limit along $j$ to get
\begin{equation*}
  |\BB_{\dR}^{q,+}\left(R,R^{+}\right)|
  \to\BB_{\dR}\left(R,R^{+}\right),
\end{equation*}
where $| - |$ refers to the underlying abstract ring.
Finally, we get the morphism
\begin{equation*}
  |\BB_{\dR}^{\dag,+}\left(R,R^{+}\right)|
  =\varinjlim|\BB_{\dR}^{q,+}\left(R,R^{+}\right)|
  \to\BB_{\dR}^{+}\left(R,R^{+}\right)
\end{equation*}
of $k_{0}$-algebras.

\begin{prop}\label{prop:BdRq+-to-BdR+-inj}
  The canonical map $|\BB_{\dR}^{q,+}\left(R,R^{+}\right)|\to\BB_{\dR}\left(R,R^{+}\right)$
  are injective if $q\geq 2$.
\end{prop}

\begin{proof}
  Because of Proposition~\ref{prop:underlyingspaceBdRdagplus-iso-uncompletedcolim-reconstructionpaper}
  and since
  $\A_{\dR}^{>q}\left(R,R^{+}\right)$ is $p$-torsion free by Lemma~\ref{lem:AdRgreaterthanqptorsionfree},
  it suffices to check that the composition
  \begin{equation*}
    |\A_{\dR}^{>q}\left(R,R^{+}\right)|\to|\BB_{\dR}^{q,+}\left(R,R^{+}\right)|\to\BB_{\dR}\left(R,R^{+}\right)
  \end{equation*}
  is injective. This follows by considering the associated gradeds with respect to the $\xi/p^{q}$-adic filtrations.
\end{proof}

%%%%%%%%%%%%%%%%%%%%%%%%%%%%%%%%%%%%%%%%%%%%%%%%%%%%%%%%%%%%
%%%%%%%%%%%%%%%%%%%%%%%%%%%%%%%%%%%%%%%%%%%%%%%%%%%%%%%%%%%%
% Almost mathematics
%%%%%%%%%%%%%%%%%%%%%%%%%%%%%%%%%%%%%%%%%%%%%%%%%%%%%%%%%%%%
%%%%%%%%%%%%%%%%%%%%%%%%%%%%%%%%%%%%%%%%%%%%%%%%%%%%%%%%%%%%

\begin{subappendices}

\section{Appendix: Almost mathematics}
\label{subsec:almostmaths-appendix}

Our proof of Theorem~\ref{prop:abstract-BdRqpluspsh-is-sheaf} relies
on the framework of \emph{almost ring theory} as developed
in~\cite{GabberRameroalmostringtheory2003}. Here in \S\ref{subsec:almostmaths-appendix},
we collect some general results.

We fix an \emph{almost setup} $(R,\mathfrak{m})$. That is,
$R$ denotes a commutative ring together with an ideal
$\mathfrak{m}\subseteq R$ such that $\mathfrak{m}^{2}=\mathfrak{m}$
and $\mathfrak{m}\otimes_{R}\mathfrak{m}$ is a flat
$R$-module. An $R$-module $M$ is \emph{almost zero}
if $\mathfrak{m}M=0$. The full subcategory $\Sigma$ of
$\Mod(R)$ consisting of all $R$-modules that are almost
zero is a Serre subcategory, hence we can form the quotient
category $\Mod\left(R^{\al}\right):=\Mod(R)/\Sigma$.
This is the category of almost $R$-modules. In this article,
we use the following Lemma~\ref{sec:almostmaths-appendix-1}
and its immediate consequences without further reference.

\begin{lem}\label{sec:almostmaths-appendix-1}
  The category of almost $R$-modules is a Grothendieck abelian
  category.
\end{lem}

\begin{proof}
  See~\cite[\S 2.2.16]{GabberRameroalmostringtheory2003}.
\end{proof}

Denote the canonical functor $\Mod(R)\to\Mod(R)/\Sigma$
by $M\to M^{\al}$.

\begin{lem}\label{lem:M-to-Mal}
 $M\to M^{\al}$ is exact and commutes with all colimits.
\end{lem}

\begin{proof}
  The exactness follows from abstract nonsense,
  see~\cite[\href{https://stacks.math.columbia.edu/tag/02MS}{Tag 02MS}]{stacks-project}.
  \cite[Proposition 2.2.13]{GabberRameroalmostringtheory2003}
  implies the second statement.
\end{proof}

We also use Lemma~\ref{lem:M-to-Mal} without further reference.

A chain complex $C^{\bullet}$ of $R$-modules
is \emph{almost exact} if $C^{\bullet,\al}$ is an exact complex of
almost $R$-modules, cf. the notation in \S\ref{subsec:notation}.

\begin{lem}\label{lem:ractusinvertibly-then-almostexact-implies-exact}
  Consider a complex $M_{1}\to M_{2} \to M_{3}$
  of $R$-modules such that an element $r\in R$
  acts as an automorphism on $M_{i}$, for each
  $i=1,2,3$. If this complex is almost exact, then it is exact. 
\end{lem}

\begin{proof}
  We may view the complex as a complex $R[1/r]$-modules,
  cf.~\cite[\href{https://stacks.math.columbia.edu/tag/07JY}{Tag 07JY}]{stacks-project},
  and the result follows from an elementary computation.
\end{proof}

In general, the functor $M\mapsto M^{\al}$ does not commute with limits.

\begin{lem}\label{lem:lim-preserve-leftalmostexact}
  We consider complexes
  \begin{equation*}
    C_{t}\colon
    0 \stackrel{\empty}{\longrightarrow}
    M_{t,1} \stackrel{\varphi_{t}}{\longrightarrow}
    M_{t,2} \stackrel{\psi_{t}}{\longrightarrow}
    M_{t,3}
  \end{equation*}
  of $R$-modules for every $t\in\NN$ together with maps
  $\omega_{t}\colon M_{t,i}\to M_{t-1,i}$ 
  for every $i=1,2,3$ and $t\geq 1$ such that the obvious squares commute.
  If those sequences $C_{t}$ are almost exact, then
  their inverse limit $\varprojlim_{t\in\NN}C_{t}$ is almost exact as well.
\end{lem}

\begin{proof}
  Write $M_{i}:=\varprojlim_{t\in\NN}M_{t,i}$ for $i=1,2,3$,
  $\varphi:=\varprojlim_{t\in\NN}\varphi_{t}$, and
  $\psi:=\varprojlim_{t\in\NN}\psi_{t}$. We show that
  \begin{equation*}
    0 \stackrel{\empty}{\longrightarrow}
    M_{1}^{\al} \stackrel{\varphi^{\al}}{\longrightarrow}
    M_{2}^{\al} \stackrel{\psi^{\al}}{\longrightarrow}
    M_{3}^{\al}
  \end{equation*}
  is exact.
  
  Let $m=\left(m_{t}\right)_{t\in\NN}\in\ker\varphi$.
  Then $m_{t}\in\ker\varphi_{t}$ for all $t\in\NN$.
  Thus $xm_{t}=0$ for all $x\in\mathfrak{m}$ and all
  $t\in\NN$ and we find
  $\mathfrak{m}\ker\varphi=0$. That is,
  $\ker\varphi^{\al}=(\ker\varphi)^{\al}=0$.
  
  Next, pick $m=\left(m_{t}\right)_{t\in\NN}\in\ker\psi$
  and fix an arbitrary $x\in\mathfrak{m}$. We have
  to show that $xm\in\im\varphi$. Indeed, this would imply
  $\mathfrak{m}\ker\psi/\im\varphi=0$, that is
  $\ker\psi^{\al}/\im\varphi^{\al}=\left(\ker\psi/\im\varphi\right)^{\al}=0$.
  Since $\mathfrak{m}^{2}=\mathfrak{m}$ we may write,
  without loss of generality, $x=yz$ for some $y,z\in\mathfrak{m}$.
  Because the complexes $C_{t}$ are almost exact,
  we find that $ym_{t}\in\im\varphi_{t}$, for all $t\in\NN$.
  Pick $n_{t}\in\varphi_{t}^{-1}\left(zm_{t}\right)$. Compute
  \begin{equation*}
    \varphi_{t}\left( \omega_{t+1}\left( n_{t+1} \right) - n_{t}\right)
    =\omega_{t+1}\left(\varphi_{t+1}\left( n_{t+1} \right) \right)- ym_{t}
    =ym_{t} - ym_{t}
    =0,
  \end{equation*}
  that is $\omega_{t+1}\left( n_{t+1} \right) - n_{t}\in\ker\varphi_{t}$,
  for every $t\in\NN$. But the $\varphi_{t}$ are almost injective,
  therefore $\omega_{t+1}\left( yn_{t+1} \right)=y\omega_{t+1}\left( n_{t+1} \right)=yn_{t}$.
  This implies that the sequence $n_{y}:=\left(yn_{t}\right)_{t\in\NN}$
  defines an element in $M_{1}$. It is clear from the definition that
  $\varphi\left( n_{y} \right)=\left( yzm_{t}\right)_{t\in\NN}=xm$.
  That is, $xm\in\im\varphi$.
\end{proof}

Now fix a site $X$ which admits \emph{only finite coverings}.
Because the category of almost $R$-modules is Grothendieck,
we get a well-behaved abelian category
$\Sh\left( X, \Mod\left(R^{\al}\right)\right)$ of sheaves of
almost $R$-modules on $X$. For example, we have an exact
sheafification functor $\cal{F}\mapsto\cal{F}^{\sh}$.

\begin{lem}\label{lem:coprod-almostsheaves}
  Consider a collection $\cal{G}_{i}$ of sheaves almost $R$-modules on $X$.
  Their coproduct is given by $U\mapsto\bigoplus_{i}\cal{G}_{i}(U)$.
\end{lem}

\begin{proof}
  Given $U\in X$ and any covering $\mathfrak{U}$ of $U$, we have to check that
  \begin{equation*}
    0\to\bigoplus_{i}\cal{G}_{i}(U)
     \to\prod_{V\in\mathfrak{U}}\bigoplus_{i}\cal{G}_{i}(V)
     \to\prod_{W,W^{\prime}\in\mathfrak{U}}\bigoplus_{i}\cal{G}_{i}\left(W\times_{U} W^{\prime}\right)
  \end{equation*}
  is exact. But because $\mathfrak{U}$ is finite,
  all the products are finite.
  Thus the operator $\bigoplus_{i}$ commutes
  with these products. Exactness now follows because
  the $\cal{G}_{i}$ are sheaves and because
  coproducts are exact in Grothendieck abelian categories.
\end{proof}

Given a presheaf $\cal{F}$ of $R$-modules, denote the
presheaf $U\mapsto\cal{F}(U)^{\al}$ of almost $R$-modules
by $\cal{F}^{\al}$.

\begin{lem}\label{lem:RGammaUFal-RGammaUFal}
  Consider $\cal{F}\mapsto\cal{F}^{\al}$.
  \begin{itemize}
    \item[(i)] It restricts to a functor
    $\Sh\left( X, \Mod\left( R \right)\right)\to\Sh\left( X, \Mod\left( R^{\al} \right)\right)$.
    \item[(ii)] Its restriction to the categories of sheaves is exact.
    \item[(iii)] For every sheaf $\cal{F}$ of $R$-modules,
      \begin{equation*}
        R\Gamma(U,\cal{F}^{\al})\cong R\Gamma(U,\cal{F})^{\al}
      \end{equation*}
      for all $U\in X$.
  \end{itemize}
\end{lem}

\begin{proof}
  $M\mapsto M^{\al}$ is exact and commutes with finite products
  because it commutes with all colimits. This gives (i) because $X$
  admits only finite coverings. Furthermore, the exactness of $M\mapsto M^{\al}$ implies
    the exactness of the functor $\cal{F}\mapsto\cal{F}^{\al}$ between the categories
    of presheaves. Now $\cal{F}\mapsto\cal{F}^{\al}$ commutes
    with the sheafification functor; this follows because
    $M\mapsto M^{\al}$ commutes with kernels, finite products, and filtered colimits. 
    This proves (ii). Finally, we consider (iii).
      Pick a flabby resolution $\cal{F}\to\cal{I}^{\bullet}$.
  Then $\cal{F}^{\al}\to\cal{I}^{\bullet,\al}$ is a flabby resolution as well,
  cf. the notation in \S\ref{subsec:notation},
  thus it computes the derived sections. Now compute
  \begin{align*}
    R\Gamma(U,\cal{F}^{\al})
    =R\Gamma(U,\cal{I}^{\al,\bullet})
    \cong R\Gamma(U,\cal{I}^{\bullet})^{\al}
    =R\Gamma(U,\cal{F})^{\al}
  \end{align*}
  with (ii).
\end{proof}

\begin{lem}\label{lem:ractusinvertibly-then-almostsheaf-implies-sheaf}
  Consider a presheaf $\cal{F}$ of $R$-modules on $X$ and
  suppose that there exists an element $r\in R$ which acts on
  $\cal{F}$ as an automorphism. If $\cal{F}^{\al}$ is a sheaf, then $\cal{F}$ is a sheaf.
\end{lem}  

\begin{proof}  
  Given $U\in X$ and any covering $\mathfrak{U}$ of $U$, we have to check that
  the following complex is exact:
  \begin{equation*}
    0\to\cal{F}(U)^{\al}
     \to\prod_{V\in\mathfrak{U}}\cal{F}(V)^{\al}
     \to\prod_{W,W^{\prime}\in\mathfrak{U}}\cal{F}\left(W\times_{U} W^{\prime}\right)^{\al}.
  \end{equation*}
  Since these products are finite, they commute with $M\mapsto M^{\al}$.
  Now apply Lemma~\ref{lem:ractusinvertibly-then-almostexact-implies-exact}.
\end{proof}
 
\begin{lem}\label{lem:lim-preserves-almost-sheafy}
  Consider
  an inverse system $\dots\to\cal{F}_{2}\to\cal{F}_{1}\to\cal{F}_{0}$
  of presheaves of $R$-modules on a site $X$ which
  admits only finite coverings. If $\cal{F}_{t}^{\al}$ is a sheaf of
  almost $R$-modules for every $t\in\NN$, then
  $\left(\varprojlim_{t\in\NN}\cal{F}_{t}\right)^{\al}$ is a sheaf
  of almost $R$-modules as well.
\end{lem}

\begin{proof}
  For any covering $\mathfrak{U}$ of an open
  $U\in X$, apply Lemma~\ref{lem:lim-preserve-leftalmostexact}
  to the almost exact sequences
  \begin{equation*}
    0
    \to\cal{F}_{t}(U)
    \to\prod_{V\in\mathfrak{U}}\cal{F}_{t}(V)
    \to\prod_{W,W^{\prime}\in\mathfrak{U}}\cal{F}_{t}\left(W\times_{U}W^{\prime}\right).
  \end{equation*}
  Here we use again that $M\mapsto M^{\al}$ commutes with finite products.
\end{proof}

\end{subappendices}

%%% COMMENT ENDS. DO NOT DELETE.

%%%%%%%%%%%%%%%%%%%%%%%%%%%%%%%%%%%%%%%%%%%%%%%%%%%%%%%%%%%%
%%%%%%%%%%%%%%%%%%%%%%%%%%%%%%%%%%%%%%%%%%%%%%%%%%%%%%%%%%%%
% Period structure sheaves
%%%%%%%%%%%%%%%%%%%%%%%%%%%%%%%%%%%%%%%%%%%%%%%%%%%%%%%%%%%%
%%%%%%%%%%%%%%%%%%%%%%%%%%%%%%%%%%%%%%%%%%%%%%%%%%%%%%%%%%%%

\chapter{Period structure sheaves}
\label{ch:period-structure-sheaves}

%%%%%%%%%%%%%%%%%%%%%%%%%%%%%%%%%%%%%%%%%%%%%%%%%%%%%%%%%%%%
%%%%%%%%%%%%%%%%%%%%%%%%%%%%%%%%%%%%%%%%%%%%%%%%%%%%%%%%%%%%
% OBla
%%%%%%%%%%%%%%%%%%%%%%%%%%%%%%%%%%%%%%%%%%%%%%%%%%%%%%%%%%%%
%%%%%%%%%%%%%%%%%%%%%%%%%%%%%%%%%%%%%%%%%%%%%%%%%%%%%%%%%%%%

\section{Definitions}
\label{subsec:OBla}

Keep the notation as in \S\ref{subsec:locan-period-sheaf} fixed.
$X$ denotes again a locally Noetherian adic space over $\Spa(k,k^{\circ})$.

For every $q\in\NN\cup\{\infty\}$,
%\footnote{The bound on $m$ is motivated by the Theorem~\ref{thm:Bla-is-sheaf}.},
we introduce the presheaves of $k$-ind-Banach algebras
\begin{align*}
  \widetilde{\OB}_{\dR}^{q,+,\psh}\colon
  U=\text{``}\varprojlim_{i\in I}\text{"}U
  &\mapsto
  \text{``}\varinjlim_{i\in I}\text{"}
  \left(\mathcal{O}^{+}(U_{i})\widehat{\otimes}_{W(\kappa)}\A_{\inf}(U)\right)
  \left\<\frac{\ker\Otheta_{\inf}}{p^{q}}\right\>\widehat{\otimes}_{k^{\circ}}k
\end{align*}
on $X_{\proet,\affperfd}^{\fin}$.
Here, $\Otheta_{\inf}$ denotes the composition
of the surjections
\begin{equation*}
  \mathcal{O}^{+}(U_{i})\widehat{\otimes}_{W(\kappa)}\A_{\inf}(U)
  \stackrel{\id\widehat{\otimes}\theta_{\inf}}{\longrightarrow}
  \mathcal{O}^{+}(U_{i})\widehat{\otimes}_{W(\kappa)}\widehat{\mathcal{O}}^{+}(U)
  \stackrel{\mu^{+}}{\longrightarrow}
  \widehat{\mathcal{O}}^{+}(U),
\end{equation*}
where $\mu^{+}$ denotes the multiplication.
$\mathcal{O}^{+}(U_{i})$ carries the $\pi$-adic seminorm.

\begin{remark}
  The kernels of the maps $\Otheta_{\inf}$ are finitely generated.
  This follows from Lemma~\ref{lem:Fontaines-map-for-Ala} and because the
  kernel of $\mu^{+}$ is finitely generated. Now Lemma~\ref{lem:completionsalongideals-generators}
  allows to give a more concrete definition of
  $\OB_{\dR}^{q,+,\psh}(U)$.
\end{remark}

Since sheafification is strongly monoidal, cf. Lemma~\ref{lem:sh-strongly-monoidal},
the sheafifications $\widetilde{\OB}_{\dR}^{q,+}$ of $\OB_{\dR}^{q,+,\psh}$
are sheaves of $k$-ind-Banach algebras.
They extend, by Lemma~\ref{lem:sheaves-on-Xproet-and-Xproetaffperfdfin},
to sheaves on the whole pro-étale site, which we denote again by $\widetilde{\OB}_{\dR}^{q,+}$.

We recall that in any quasi-abelian category, the coimage of a given morphism
$f\colon E\to F$ is $\coim f:= \coker\left(\ker(f)\to F\right)$. In general, it is not isomorphic
to its image $\im f$. This is the case, for example, for the morphism
$\widetilde{\OB}_{\dR}^{q,+}\to\widetilde{\OB}_{\dR}^{\infty,+}$ in the following
Definition~\ref{defn:OBlaplus}.

\begin{defn}\label{defn:OBlaplus}
  For every $q\in\NN\cup\{\infty\}$, we set
  \begin{equation*}
    \OB_{\dR}^{q,+}:=\coim\left( \widetilde{\OB}_{\dR}^{q,+}\to\widetilde{\OB}_{\dR}^{\infty,+} \right).
  \end{equation*}
  $\OB_{\dR}^{\dag,+}:=\OB_{\dR}^{\infty,+}\cong\widetilde{\OB}_{\dR}^{\infty,+}$
  is the \emph{positive overconvergent de Rham period structure sheaf}.
\end{defn}

\begin{remark}
  We chose this somewhat cumbersome Definition~\ref{defn:OBlaplus}
  of $\OB_{\dR}^{q,+}$ because we wanted to have
  a local description as in Theorem~\ref{thm:localdescription-of-OBqplus}.
  The proof of this theorem relies heavily on
  Proposition~\ref{prop:coordinates-germ-diagonal}.
  A similar local description does not hold for
  $\widetilde{\OB}_{\dR}^{q,+}$, because the corresponding version
  of Proposition~\ref{prop:coordinates-germ-diagonal} fails:
  that is, in the notation of \emph{loc. cit.}
  \begin{equation*}
    \tau_{i}\colon R_{i}\left\< \frac{Z_{1},\dots,Z_{d}}{p^{q}}\right\>
    \to
    \left(R_{i}\widehat{\otimes}_{k_{0}}R_{i}\right)
    \left\< \frac{s_{1},\dots,s_{n}}{p^{q}}\right\>
   \end{equation*}
  is usually not an isomorphism for $q<\infty$. 
\end{remark}

\begin{remark}
  Use Lemma~\ref{lem:banach-to-indbanach-sheaves} to
  view the structure sheaf $\mathcal{O}$ on $X$ as a sheaf of $k$-ind-Banach algebras.
  The canonical morphisms
  \begin{equation*}
    \nu^{-1}\mathcal{O} \widehat{\otimes}_{k_{0}}\BB_{\dR}^{q,+}
    \to \OB_{\dR}^{q,+}
  \end{equation*}
  are morphisms of sheaves of $k$-ind-Banach algebras, cf. Lemma~\ref{lem:sh-strongly-monoidal}.
  This makes $\OB_{\dR}^{q,+}$ a sheaf of
  $\nu^{-1}\mathcal{O}$-ind-Banach algebras, and a sheaf of
  $\BB_{\dR}^{q,+}$-ind-Banach algebras.
\end{remark}

\begin{remark}
  The maps $\Otheta_{\inf}$ induce morphisms
  \begin{equation*}
    \Ovartheta_{\dR}^{q,+}\colon\OB_{\dR}^{q,+} \to \hO
  \end{equation*}
  of sheaves of $k$-ind-Banach spaces for every $q\in\NN\cup\{\infty\}$,
  by Lemma~\ref{lem:Banach-completions-extend} and~\ref{lem:sh-strongly-monoidal}.
\end{remark}

We continue with the following definitions:
\begin{align*}
  \OB_{\dR}^{\dag}&:=\OB_{\dR}^{\dag,+}\widehat{\otimes}_{\BB_{\dR}^{\dag,+}}\BB_{\dR}^{\dag} \\
  \OB_{\pdR}^{\dag,+}&:=\OB_{\dR}^{\dag,+}\widehat{\otimes}_{\BB_{\dR}^{\dag,+}}\BB_{\pdR}^{\dag,+}, \text{ and} \\
  \OB_{\pdR}^{\dag}&:=\OB_{\dR}^{\dag,+}\widehat{\otimes}_{\BB_{\dR}^{\dag,+}}\BB_{\pdR}^{\dag}
\end{align*}
are sheaves of
$\nu^{-1}\cal{O}\widehat{\otimes}_{k_{0}}\BB_{\dR}^{\dag}$-,
$\nu^{-1}\cal{O}\widehat{\otimes}_{k_{0}}\BB_{\pdR}^{\dag,+}$-, and
$\nu^{-1}\cal{O}\widehat{\otimes}_{k_{0}}\BB_{\pdR}^{\dag}$-ind-Banach algebras on $X_{\proet}$.

\begin{defn}\label{defn:namesfortheperiodsheaves-reconstructionpaper}
  $\OB_{\dR}^{\dag}$ is the
  \emph{overconvergent de Rham period structure sheaf}.
  $\OB_{\pdR}^{\dag,+}$ is the
  \emph{positive overconvergent almost de Rham period structure sheaf}.
  $\OB_{\pdR}^{\dag}$ is the
  \emph{overconvergent almost de Rham period structure sheaf}.
\end{defn}

%%%%%%%%%%%%%%%%%%%%%%%%%%%%%%%%%%%%%%%%%%%%%%%%%%%%%%%%%%%%
%%%%%%%%%%%%%%%%%%%%%%%%%%%%%%%%%%%%%%%%%%%%%%%%%%%%%%%%%%%%
% OBla
%%%%%%%%%%%%%%%%%%%%%%%%%%%%%%%%%%%%%%%%%%%%%%%%%%%%%%%%%%%%
%%%%%%%%%%%%%%%%%%%%%%%%%%%%%%%%%%%%%%%%%%%%%%%%%%%%%%%%%%%%

\section{Local descriptions}
\label{section:localdescriptions-periodstructuresheaves}

%%%%%%%%%%%%%%%%%%%%%%%%%%%%%%%%%%%%%%%%%%%%%%%%%%%%%%%%%%%%
%%%%%%%%%%%%%%%%%%%%%%%%%%%%%%%%%%%%%%%%%%%%%%%%%%%%%%%%%%%%
% OBla
%%%%%%%%%%%%%%%%%%%%%%%%%%%%%%%%%%%%%%%%%%%%%%%%%%%%%%%%%%%%
%%%%%%%%%%%%%%%%%%%%%%%%%%%%%%%%%%%%%%%%%%%%%%%%%%%%%%%%%%%%

\subsection{Main results}

We aim to describe $\OB_{\dR}^{\dag,+}$ %$\OB_{\dR}^{+}$, $\OB_{\dR}^{+,\bd}$, $\OB_{\dR}$, and $\OB_{\dR}^{\bd}$
locally, in the spirit of~\cite[Proposition 6.10]{Sch13pAdicHodge}. Assume that $X$ is
affinoid and equipped with an étale map $X\to\TT^{d}:=\TT_{0}^{d}$. Here we define,
for all $e\in\NN$,
\begin{equation*}
  \TT_{e}^{d}:=\Spa\left(k\left\<T_{1}^{\pm1/p^{e}},\dots,T_{d}^{\pm1/p^{e}}\right\>,
  k^{\circ}\left\<T_{1}^{\pm1/p^{e}},\dots,T_{d}^{\pm1/p^{e}}\right\>\right)
\end{equation*}
Write $\widetilde{\TT}^{d}:=\text{``}\varprojlim_{e\in\NN}\text{"}\TT_{e}^{d}\in\TT_{\proet}^{d}$
and $\widetilde{X}:=X\times_{\TT^{d}}\widetilde{\TT}^{d}$.
Every $U=\text{``}\varprojlim_{i\in I}\text{"}U_{i}\in X_{\proet}/\widetilde{X}$ gives rise
to a morphism $U\to\widetilde{X}\to\widetilde{\TT}^{d}$ in
$\Pro\left(\TT_{\proet}^{d}\right)$. It is thus given by a compatible system
of étale maps $U_{i_{0}}\maps\TT_{e}^{d}$ for a fixed
$i_{0}\in I$ and varying $e$.

\begin{notation}\label{notation:i0}
  Fix such an $i_{0}\in I$.
\end{notation}

Consider the images of the elements $T_{1},\dots,T_{d}\in\mathcal{O}^{+}\left(\TT_{e}^{d}\right)$
in $\mathcal{O}^{+}\left(U_{i}\right)$ for all $i\geq i_{0}$
and $e$ large enough. Abusing notation, denote them again by $T_{1},\dots,T_{d}$. 
\begin{equation}\label{eq:defn-ul}
  u_{l}:=T_{l}\widehat{\otimes}1 - 1 \widehat{\otimes}\left[T_{l}^{\flat}\right]
  \in \mathcal{O}^{+}(U_{i})\widehat{\otimes}_{W(\kappa)}\A_{\inf}(U)
\end{equation}
for all $l=1,\dots,d$, every $U=\text{``}\varprojlim_{i\in I}\text{"}U_{i}\in X_{\proet,\affperfd}/\widetilde{X}$,
and $i\geq i_{0}$. Here
\begin{equation*}
  T_{l}^{\flat}:=\left(T_{l},T_{l}^{1/p},T_{l}^{1/p^{2}},\dots\right)\in\left(\widehat{\mathcal{O}}^{+}(U)\right)^{\flat}.
\end{equation*}
Since those $u_{l}$ lie in the kernel of $\Otheta_{\inf}$,
the canonical maps
\begin{equation*}
  \A_{\dR}^{q}(U) \to
  \left(\mathcal{O}^{+}(U_{i})\widehat{\otimes}_{W(\kappa)}\A_{\inf}(U)\right)\left\< \frac{\ker\vartheta}{p^{q}} \right\>
\end{equation*}
extend with Corollary~\ref{cor:extend-maps-to-restrictedpowerseries}
to the morphisms
\begin{equation*}
  \A_{\dR}^{q}(U)\left\<\frac{Z_{1},\dots,Z_{d}}{p^{q}}\right\>
  \to\left(\mathcal{O}^{+}(U_{i})\widehat{\otimes}_{W(\kappa)}\A_{\inf}(U)\right)\left\< \frac{\ker\vartheta}{p^{q}} \right\>,
  Z_{l}/p^{q} \mapsto u_{l}/p^{q}
  %\frac{Z_{l}}{p^{q}}&\mapsto \frac{u_{l}}{p^{q}}
\end{equation*}
of $\A_{\dR}^{q}(U)$-Banach algebras for every $q\in\NN$.
Here, the $Z_{l}$ denote formal variables
\footnote{\cite{Sch13pAdicHodge} denotes $Z_{l}$ as $X_{l}$.
We use different symbols to avoid confusion with the space $X$.}.
Invert $p$ and
pass to the colimit along $i\in I$ to get
\begin{equation*}
  \Phi^{q,+}(U)\colon
  \BB_{\dR}^{q,+}(U)\left\<\frac{Z_{1},\dots,Z_{d}}{p^{q}}\right\>
  \to\widetilde{\OB}_{\dR}^{q,+,\psh}(U).
\end{equation*}
The data of these maps $\Phi_{\dR}^{q,+}(U)$ define morphisms
\begin{equation*}
  \widetilde{\Phi}^{q, +,\psh}\colon
  \BB_{\dR}^{q, +}|_{\widetilde{X}}\left\<\frac{Z_{1},\dots,Z_{d}}{p^{q}}\right\>^{\psh}
  \to\widetilde{\OB}_{\dR}^{q, +,\psh}|_{\widetilde{X}}
\end{equation*}
of presheaves of $\BB_{\dR}^{q, +}|_{\widetilde{X}}$-ind-Banach algebras on $X_{\proet}/\widetilde{X}$,
where
\begin{equation*}
  \BB_{\dR}^{q,+}|_{\widetilde{X}}\left\<\frac{Z_{1},\dots,Z_{d}}{p^{q}}\right\>^{\psh}
  \colon
  V\mapsto\BB_{\dR}^{q,+}(V)\left\<\frac{Z_{1},\dots,Z_{d}}{p^{q}}\right\>.
\end{equation*}
By Lemma~\ref{lem:sh-strongly-monoidal}, its sheafification
\begin{equation*}
  \widetilde{\Phi}^{q,+}\colon
  \BB_{\dR}^{q,+}|_{\widetilde{X}}\left\<\frac{Z_{1},\dots,Z_{d}}{p^{q}}\right\>
  \to\widetilde{\OB}_{\dR}^{q,+}|_{\widetilde{X}}
\end{equation*}
is a morphism of sheaves of $\BB_{\dR}^{q,+}|_{\widetilde{X}}$-ind-Banach algebras.
Here,
\begin{equation*}
  \BB_{\dR}^{q,+}|_{\widetilde{X}}\left\<\frac{Z_{1},\dots,Z_{d}}{p^{q}}\right\>:=
  \left(\BB_{\dR}^{q,+}|_{\widetilde{X}}\left\<\frac{Z_{1},\dots,Z_{d}}{p^{q}}\right\>^{\psh}\right)^{\sh}.
\end{equation*}
Now pass to the colimit along
$q\to\infty$ to obtain the morphism
\begin{equation*}
  \widetilde{\Phi}^{\dag,+}\colon
  \BB_{\dR}^{\dag,+}|_{\widetilde{X}}\left\<\frac{Z_{1},\dots,Z_{d}}{p^{\infty}}\right\>
  \to\OB_{\dR}^{\dag,+}|_{\widetilde{X}}
\end{equation*}
of sheaves of $\BB_{\dR}^{+}|_{\widetilde{X}}$-ind-Banach algebras.

\begin{thm}\label{thm:localdescription-of-OBla}
  Let $X$ be affinoid and equipped with
  an étale map $X\to\TT^{d}$, giving rise to the
  pro-étale covering $\widetilde{X}\to X$. Then the morphism
  \begin{equation*}
    \Phi^{\dag,+}\colon
    \BB_{\dR}^{\dag,+}|_{\widetilde{X}}\left\<\frac{Z_{1},\dots,Z_{d}}{p^{\infty}}\right\>
    \stackrel{\cong}{\longrightarrow}\OB_{\dR}^{\dag,+}|_{\widetilde{X}}
  \end{equation*}
  is an isomorphism of sheaves of
  $\BB_{\dR}^{\dag,+}|_{\widetilde{X}}$-ind-Banach algebras.
\end{thm}

\begin{remark}
  Theorem~\ref{thm:localdescription-of-OBla} implies
  its version~\cite[Proposition 6.10]{Sch13pAdicHodge} for $\OB_{\dR}^{+}$
  via taking $t$-adic completions, locally on the pro-étale site.
\end{remark}

We prove Theorem~\ref{thm:localdescription-of-OBla}
in \S\ref{subsec:proof-localdescrption-of-OBla}.

\begin{cor}\label{cor:localdescription-of-subsections-OBla}
  Let $X$ be affinoid and equipped with
  an étale map $X\to\TT^{d}$, giving rise to the
  pro-étale covering $\widetilde{X}\to X$.
  For every affinoid perfectoid $U\in X_{\proet}/\widetilde{X}$
  with $\widehat{U}=\Spa\left( R,R^{+} \right)$,
  \begin{equation*}
  \begin{split}
    \BB_{\dR}^{\dag,+}\left( R , R^{+} \right)\left\< \frac{Z_{1},\dots,Z_{d}}{p^{\infty}}\right\>
    &\stackrel{\cong}{\longrightarrow}\OB_{\dR}^{\dag,+}(U) \\
    Z_{l}&\longmapsto u_{l}
  \end{split}
  \end{equation*}
  is an isomorphism
  of $\BB_{\dR}^{\dag,+}\left( R , R^{+} \right)$-ind-Banach algebras.
\end{cor}

\begin{proof}
  It suffices to check that
  \begin{equation}\label{eq:cor-localdescription-of-subsections-OBla}
    U\mapsto \BB_{\dR}^{\dag,+}\left(\hO(U),\widehat{\mathcal{O}}^{+}(U)\right)\left\<\frac{Z_{1},\dots,Z_{d}}{p^{\infty}}\right\>
  \end{equation}
  is a sheaf on $X_{\proet,\affperfd}^{\fin}$. Indeed,
  then we can apply the Theorems~\ref{thm:subsections-periodsheaves-affperfd} and~\ref{thm:localdescription-of-OBla}.
  But the sheafiness of~(\ref{eq:cor-localdescription-of-subsections-OBla})
  follows from Theorem~\ref{thm:subsections-periodsheaves-affperfd}
  together with Lemma~\ref{lem:calFsheaf-FcirccalF-sheaf-exactfunctor},
  which applies because of Corollary~\ref{lem:restrictedpowerseries-exact-strongerconvergenceinfty}.
\end{proof}

\begin{remark}\label{remark:subsectionsofOBdRdagpluspsh}
  Write $\OB_{\dR}^{\dag,+,\psh}:=\OB_{\dR}^{\infty,+,\psh}$.
  Then, for any $U\in X_{\proet}/\widetilde{X}$,
  \begin{equation*}
    \OB_{\dR}^{\dag,+,\psh}(U)
    \cong\BB_{\dR}^{\dag,\psh}(U)\left\<\frac{Z_{1},\dots,Z_{d}}{p^{\infty}}\right\>
    \cong\BB_{\dR}^{\dag}(U)\left\<\frac{Z_{1},\dots,Z_{d}}{p^{\infty}}\right\>
    \cong\OB_{\dR}^{\dag,+}(U)
  \end{equation*}
  by the proofs of Theorem~\ref{thm:localdescription-of-OBla}
  and Corollary~\ref{cor:localdescription-of-subsections-OBla}.
\end{remark}

Fix again $q\in\NN$. We continue by composing $\widetilde{\Phi}^{q,+}$
with the canonical morphism
\begin{equation*}
  \widetilde{\OB}_{\dR}^{q,+}
  \to\OB_{\dR}^{q,+}
\end{equation*}
to get the morphism $\Phi^{q,+}$ in the following Theorem~\ref{thm:localdescription-of-OBqplus}

\begin{thm}\label{thm:localdescription-of-OBqplus}
  Let $X$ be affinoid and equipped with
  an étale map $X\to\TT^{d}$, giving rise to the
  pro-étale covering $\widetilde{X}\to X$. Then, for $q\gg0$ large enough (depending on $X\to\TT^{d}$), the morphism
  \begin{equation*}
    \Phi^{q,+}\colon
    \BB_{\dR}^{q,+}|_{\widetilde{X}}\left\<\frac{Z_{1},\dots,Z_{d}}{p^{q}}\right\>
    \stackrel{\cong}{\longrightarrow}\OB_{\dR}^{q,+}|_{\widetilde{X}}
  \end{equation*}
  is an isomorphism of sheaves of
  $\BB_{\dR}^{q,+}|_{\widetilde{X}}$-ind-Banach algebras.
\end{thm}

We prove Theorem~\ref{thm:localdescription-of-OBqplus}
in \S\ref{subsec:proof-localdescrption-of-OBla-finitelevel}.

\begin{notation}\label{notation:qgbound}
  Given an affinoid $X$ equipped with
  an étale map $g\colon X\to\TT^{d}$, define $q_{g}\in\NN_{\geq2}$ to be the smallest
  number such that $\Phi^{q,+}$ in Theorem~\ref{thm:localdescription-of-OBqplus} is an isomorphism.
\end{notation}

\begin{cor}\label{cor:localdescription-of-subsections-OBqplus}
  Let $X$ be affinoid and equipped with
  an étale map $g\colon X\to\TT^{d}$, giving rise to the
  pro-étale covering $\widetilde{X}\to X$.
  Fix $q\geq g_{g}$.
  For every affinoid perfectoid $U\in X_{\proet}/\widetilde{X}$
  with $\widehat{U}=\Spa\left( R,R^{+} \right)$,
  \begin{equation*}
  \begin{split}
    \BB_{\dR}^{q,+}\left( R , R^{+} \right)\left\< \frac{Z_{1},\dots,Z_{d}}{p^{q}}\right\>
    &\stackrel{\cong}{\longrightarrow}\OB_{\dR}^{q,+}(U) \\
    Z_{l}&\longmapsto u_{l}
  \end{split}
  \end{equation*}
  is an isomorphism
  of $\BB_{\dR}^{q,+}\left( R , R^{+} \right)$-ind-Banach algebras.
\end{cor}

\begin{proof}
  It suffices to check that
  \begin{equation}\label{eq:cor-localdescription-of-subsections-OBqplus}
    U\mapsto \BB_{\dR}^{q,+}\left(\hO(U),\widehat{\mathcal{O}}^{+}(U)\right)\left\<\frac{Z_{1},\dots,Z_{d}}{p^{q}}\right\>
  \end{equation}
  is a sheaf on $X_{\proet,\affperfd}^{\fin}$. Indeed,
  then we can apply the Theorems~\ref{thm:subsections-periodsheaves-affperfd} and~\ref{thm:localdescription-of-OBqplus}.
  But the sheafiness of~(\ref{eq:cor-localdescription-of-subsections-OBqplus})
  follows from Theorem~\ref{thm:subsections-periodsheaves-affperfd},
  which requires $q\geq 2$,
  together with Lemma~\ref{lem:calFsheaf-FcirccalF-sheaf-exactfunctor},
  which applies because of Lemma~\ref{lem:restrictedpowerseries-exact-strongerconvergence}.
\end{proof}

\begin{remark}\label{rem:OBdRlocaldescription-discreteval}
  Theorem~\ref{thm:localdescription-of-OBla},
  and hence Corollary~\ref{cor:localdescription-of-subsections-OBla},
  Theorem~\ref{thm:localdescription-of-OBqplus}, and
  Corollary~\ref{cor:localdescription-of-subsections-OBqplus}, require
  $k$ to be discretely valued. Indeed, the proof of
  Theorem~\ref{thm:localdescription-of-OBla} relies on
  Proposition~\ref{prop:coordinates-germ-diagonal},
  which uses that $k$ is discretely valued,
  cf. Remark~\ref{rem:discreteval-prop:coordinates-germ-diagonal}.
\end{remark}

%%%%%%%%%%%%%%%%%%%%%%%%%%%%%%%%%%%%%%%%%%%%%%%%%%%%%%%%%%%%
%%%%%%%%%%%%%%%%%%%%%%%%%%%%%%%%%%%%%%%%%%%%%%%%%%%%%%%%%%%%
% OBla
%%%%%%%%%%%%%%%%%%%%%%%%%%%%%%%%%%%%%%%%%%%%%%%%%%%%%%%%%%%%
%%%%%%%%%%%%%%%%%%%%%%%%%%%%%%%%%%%%%%%%%%%%%%%%%%%%%%%%%%%%

\subsection{Miscellaneous consequences of Theorems~\ref{thm:localdescription-of-OBla} and~\ref{thm:localdescription-of-OBqplus}}

%Recall Definition~\ref{defn:constantsheaf-recpaper}.

\begin{prop}\label{prop:localdescription-of-subsections-OBla-reconstructionpaper}
  Suppose $X$ is affinoid and equipped with an étale morphism $g\colon X\to\TT^{d}$,
  giving rise to the pro-étale covering $\widetilde{X}\to X$. For every $q\geq q_{g}$,
  there is an isomorphism
  \begin{equation*}
    \BB_{\dR}^{q,+}|_{\widetilde{X}}
      \widehat{\otimes}_{k_{0}}k_{0}\left\<\frac{Z_{1},\dots,Z_{d}}{p^{q}}\right\>_{X_{\proet}/\widetilde{X}}
    \isomap \OB_{\dR}^{q,+}|_{\widetilde{X}}
  \end{equation*}
  of sheaves of $\BB_{\dR}^{q,+}$-ind-Banach algebras on $X_{\proet}/\widetilde{X}$.
  This implies
  \begin{equation*}
    \BB_{\dR}^{\dag,+}|_{\widetilde{X}}
      \widehat{\otimes}_{k_{0}}k_{0}\left\<\frac{Z_{1},\dots,Z_{d}}{p^{\infty}}\right\>_{X_{\proet}/\widetilde{X}}
    \isomap \OB_{\dR}^{\dag,+}|_{\widetilde{X}}.
  \end{equation*}
\end{prop}

\begin{proof}
  Let $k_{0}\left\<\frac{Z_{1},\dots,Z_{d}}{p^{q}}\right\>_{X_{\proet}/\widetilde{X},\psh}$
  denote the constant presheaf on $X_{\proet}/\widetilde{X}$ with value $k_{0}\left\<\frac{Z_{1},\dots,Z_{d}}{p^{q}}\right\>$.
  For any two presheaves $\cal{F}$ and $\cal{G}$ of $k_{0}$-ind-Banach spaces, we have the presheaf
  \begin{equation*}
    \cal{F}\widehat{\otimes}_{k_{0},\psh}\cal{G}\colon U\mapsto\cal{F}(U)\widehat{\otimes}_{k_{0}}\cal{G}(U)
  \end{equation*}
  of $k_{0}$-ind-Banach spaces. Finally, denote sheafification by $-^{\sh}$ and compute
  \begin{align*}
    \OB_{\dR}^{q,+}|_{\widetilde{X}}
    &\cong\left(\BB_{\dR}^{q,+}|_{\widetilde{X}}
      \widehat{\otimes}_{k_{0},\psh}k_{0}\left\<\frac{Z_{1},\dots,Z_{d}}{p^{q}}\right\>_{X_{\proet}/\widetilde{X},\psh}\right)^{\sh} \\
    &\cong\BB_{\dR}^{q,+}|_{\widetilde{X}}
      \widehat{\otimes}_{k_{0}}k_{0}\left\<\frac{Z_{1},\dots,Z_{d}}{p^{q}}\right\>_{X_{\proet}/\widetilde{X}}.
  \end{align*}
  Here we used Lemma~\ref{lem:sh-strongly-monoidal},
  Corollary~\ref{cor:localdescription-of-subsections-OBla}, and
  Corollary~\ref{cor:localdescription-of-subsections-OBqplus}.
\end{proof}

In the following, we work over $C$ rather than more general perfectoid fields; this is sufficient
for the purposes of this article.
Furthermore, in \S\ref{subsubsec:cohoveraffperfd-recpaper},
we work in left hearts of quasi-abelian categories,
as this ensures the existence of enough injectives, cf.
\S\ref{subsec:LHindBanachmodules-reconstructionpaper}
and \S\ref{subsec:sheavesindbanspaces-reconstructionpaper}.

\begin{thm}\label{thm:cohomology-OBdRdag-over-affperfd-reconstructionpaper}
  Suppose $X$ is affinoid and equipped with an étale morphism $X\to\TT^{d}$.
  Let $U\in X_{\proet}$ denote an affinoid perfectoid over $\widetilde{X}_{C}$,
  with $\widehat{U}=\Spa\left(R,R^{+}\right)$.
  Then there is an isomorphism
  \begin{equation*}
    \I\left(\BB_{\dR}^{\dag}\left(R,R^{+}\right)\left\<\frac{Z_{1},\dots,Z_{d}}{p^{\infty}}\right\>\right)
    \isomap \R\Gamma\left(U,\I\left(\OB_{\dR}^{\dag}\right)\right)
  \end{equation*}
  in $\D\left(\IndBan_{\I\left(k\right)}\right)$, functorial in $U$.
  In particular,
  $\BB_{\dR}^{\dag}\left( R, R^{+} \right)\left\<\frac{Z_{1},\dots,Z_{d}}{p^{\infty}}\right\>\isomap\OB_{\dR}^{\dag}(U)$.
\end{thm}

\begin{proof}
  Firstly, we compute
  \begin{align*}
    \OB_{\dR}^{\dag}|_{\widetilde{X}_{C}}
    &\cong
    \OB_{\dR}^{\dag,+}|_{\widetilde{X}_{C}}
      \widehat{\otimes}_{\BB_{\dR}^{\dag,+}|_{\widetilde{X}_{C}}}
      \BB_{\dR}^{\dag}|_{\widetilde{X}_{C}} \\
    &\stackrel{\text{\ref{prop:localdescription-of-subsections-OBla-reconstructionpaper}}}{\cong}
    \left(\BB_{\dR}^{\dag,+}|_{\widetilde{X}_{C}}
      \widehat{\otimes}_{k_{0}}k_{0}\left\<\frac{Z_{1},\dots,Z_{d}}{p^{\infty}}\right\>_{X_{\proet}/\widetilde{X}_{C}}\right)
      \widehat{\otimes}_{\BB_{\dR}^{\dag,+}|_{\widetilde{X}_{C}}}
      \BB_{\dR}^{\dag}|_{\widetilde{X}_{C}} \\
    &\cong
    \BB_{\dR}^{\dag}|_{\widetilde{X}_{C}}
      \widehat{\otimes}_{k_{0}}k_{0}\left\<\frac{Z_{1},\dots,Z_{d}}{p^{\infty}}\right\>_{X_{\proet}/\widetilde{X}_{C}}.
  \end{align*}
  Together with Lemma~\ref{lem:sh-strongly-monoidal},
  we find that $\OB_{\dR}^{\dag}|_{\widetilde{X}_{C}}$ is the sheafification of the presheaf
  \begin{equation*}
    V \mapsto \BB_{\dR}^{\dag}(V)\left\<\frac{Z_{1},\dots,Z_{d}}{p^{\infty}}\right\>
  \end{equation*}
  on the site $X_{\proet,\affperfd}^{\fin}/\widetilde{X}_{C}$.
  But this is a sheaf with vanishing higher \v{C}ech cohomology;
  this follows because $\BB_{\dR}^{\dag}$ is a sheaf with vanishing higher \v{C}ech cohomology
  on $X_{\proet,\affperfd}^{\fin}/\widetilde{X}_{C}$, cf. the
  proof of Theorem~\ref{thm:derivedlocalsectionsofBdRdagplus-recostructionpaper},
  and by Lemma~\ref{lem:restrictedpowerseries-exact-strongerconvergenceinfty}.
  Note that for the application of the aforementioned lemma, it is crucial that we only allow finite coverings as in this case,
  the the \v{C}ech complexes are build out of finite products, which commute with completed tensor products.
  Now Lemma~\ref{lem:strictlyexactcechcomplexes-sheafcohomology-reconstructionpaper}
  implies Theorem~\ref{thm:cohomology-OBdRdag-over-affperfd-reconstructionpaper}.
\end{proof}

Recall Definition~\ref{defn:solidificationssheaves} and
denote the solidification of $\OB_{\dR}^{\dag}$ by $\underline{\OB}_{\dR}^{\dag}:=\underline{\OB_{\dR}^{\dag}}$.

\begin{cor}\label{cor:cohomology-solidOBdRdag-over-affperfd-reconstructionpaper}
  Suppose $X$ is affinoid and equipped with an étale morphism $X\to\TT^{d}$.
  Let $U\in X_{\proet}$ denote an affinoid perfectoid over $\widetilde{X}_{C}$
  with $\widehat{U}=\Spa\left(R,R^{+}\right)$.
  Then there is an isomorphism
  \begin{equation*}
    \underline{\BB_{\dR}^{\dag}\left(R,R^{+}\right)\left\<\frac{Z_{1},\dots,Z_{d}}{p^{\infty}}\right\>}
    \isomap \R\Gamma\left(U,\underline{\OB}_{\dR}^{\dag}\right)
  \end{equation*}
  in the derived category of $\Vect_{k}^{\solid}$. It is functorial in $U$.
\end{cor}

\begin{proof}
  We work on the site $X_{\proet,\affperfd}^{\fin}$.
  Now the description of the sections of $\underline{\OB}_{\dR}^{\dag}$
  follows directly from
  Lemma~\ref{lem:sectionsofunderlineF} and
  Theorem~\ref{thm:cohomology-OBdRdag-over-affperfd-reconstructionpaper}.  
  Its proof establishes that $\OB_{\dR}^{\dag}$ has vanishing
  higher \v{C}ech cohmology on $X_{\proet,\affperfd}^{\fin}$.
  Therefore, Lemma~\ref{lem:IndBan-to-Solid-strictlyexact-reconstructionpaper}
  implies that $\underline{\OB}_{\dR}^{\dag}$ is a sheaf with vanishing
  higher \v{C}ech cohmology.
  This implies $ \Ho^{i}\left(U,\underline{\OB}_{\dR}^{\dag}\right)=0$
  for all $i>0$ and thus the Corollary~\ref{cor:cohomology-solidOBdRdag-over-affperfd-reconstructionpaper}.
\end{proof}

\begin{thm}\label{thm:cohomology-OBpdRdag-over-affperfd-reconstructionpaper}
  Suppose $X$ is affinoid and equipped with an étale morphism $X\to\TT^{d}$.
  Let $U\in X_{\proet}$ denote an affinoid perfectoid over $\widetilde{X}_{C}$
  with $\widehat{U}=\Spa\left(R,R^{+}\right)$.
  Then there is an isomorphism
  \begin{equation*}
    \I\left(\BB_{\pdR}^{\dag}\left(R,R^{+}\right)\left\<\frac{Z_{1},\dots,Z_{d}}{p^{\infty}}\right\>\right)
    \isomap \R\Gamma\left(U,\I\left(\OB_{\pdR}^{\dag}\right)\right)
  \end{equation*}
  in $\D\left(\IndBan_{\I\left(k\right)}\right)$, functorial in $U$.
  In particular,
  $\BB_{\pdR}^{\dag}\left( R, R^{+} \right)\left\<\frac{Z_{1},\dots,Z_{d}}{p^{\infty}}\right\>\isomap\OB_{\pdR}^{\dag}(U)$.\end{thm}

\begin{proof}
  Proceed as in the proof of Theorem~\ref{thm:cohomology-OBdRdag-over-affperfd-reconstructionpaper},
  applying Corollary~\ref{cor:derivedlocalsectionsofBpdRdag-recostructionpaper}
  instead of Theorem~\ref{thm:derivedlocalsectionsofBdRdagplus-recostructionpaper}.
\end{proof}

\begin{cor}\label{cor:localdescription--OBpdRdag-sheafy-reconstructionpaper}
  Suppose $X$ is affinoid and equipped with an étale morphism $X\to\TT^{d}$,
  giving rise to the pro-étale covering $\widetilde{X}\to X$.
  There is an isomorphism
  \begin{equation*}
    \BB_{\pdR}^{\dag}|_{\widetilde{X}}\widehat{\otimes}_{k_{0}}k_{0}\left\<\frac{Z_{1},\dots,Z_{d}}{p^{q}}\right\>_{X_{\proet}/\widetilde{X}}
    \isomap \OB_{\pdR}^{\dag}|_{\widetilde{X}}
  \end{equation*}
  of sheaves of $\BB_{\pdR}^{\dag}$-ind-Banach algebras on $X_{\proet}/\widetilde{X}$.
\end{cor}

\begin{proof}
  Proceed as in the proof of Proposition~\ref{prop:localdescription-of-subsections-OBla-reconstructionpaper},
  but apply Theorem~\ref{thm:cohomology-OBpdRdag-over-affperfd-reconstructionpaper} instead
  of Corollary~\ref{cor:localdescription-of-subsections-OBqplus} or
  Corollary~\ref{cor:localdescription-of-subsections-OBla}.
\end{proof}

Denote the solidification of $\OB_{\pdR}^{\dag}$ by $\underline{\OB}_{\pdR}^{\dag}:=\underline{\OB_{\pdR}^{\dag}}$.

\begin{cor}\label{cor:cohomology-solidOBpdRdag-over-affperfd-reconstructionpaper}
  Suppose $X$ is affinoid and equipped with an étale morphism $X\to\TT^{d}$.
  Let $U\in X_{\proet}$ denote an affinoid perfectoid over $\widetilde{X}_{C}$
  with $\widehat{U}=\Spa\left(R,R^{+}\right)$.
  Then there is an isomorphism
  \begin{equation*}
    \underline{\BB_{\pdR}^{\dag}\left(R,R^{+}\right)\left\<\frac{Z_{1},\dots,Z_{d}}{p^{\infty}}\right\>}
    \isomap \R\Gamma\left(U,\underline{\OB}_{\pdR}^{\dag}\right)
  \end{equation*}
  in the derived category of $\Vect_{k}^{\solid}$. It is functorial in $U$.
\end{cor}

\begin{proof}
  We work on the site $X_{\proet,\affperfd}^{\fin}$.
  Now the description of the sections of $\underline{\OB}_{\pdR}^{\dag}$
  follows directly from
  Lemma~\ref{lem:sectionsofunderlineF} and
  Theorem~\ref{thm:cohomology-OBpdRdag-over-affperfd-reconstructionpaper}.  
  Its proof establishes that $\OB_{\pdR}^{\dag}$ has vanishing
  higher \v{C}ech cohmology.
  Therefore, Lemma~\ref{lem:IndBan-to-Solid-strictlyexact-reconstructionpaper}
  implies that $\underline{\OB}_{\pdR}^{\dag}$ is a sheaf with
  higher \v{C}ech cohmology.
  This implies $ \Ho^{i}\left(U,\underline{\OB}_{\pdR}^{\dag}\right)=0$
  for all $i>0$ and thus the Corollary~\ref{cor:cohomology-solidOBdRdag-over-affperfd-reconstructionpaper}.
\end{proof}

We record the following consequence of Corollary~\ref{cor:cohomology-solidOBpdRdag-over-affperfd-reconstructionpaper}
for future reference.

\begin{lem}\label{lem:solidOBpdRdag-directsum-of-solidOBdRdag}
  Suppose $X$ is affinoid and equipped with an étale morphism $X\to\TT^{d}$.
  Let $U\in X_{\proet}$ denote an affinoid perfectoid over $\widetilde{X}_{C}$.
  %with $\widehat{U}=\Spa\left(R,R^{+}\right)$.
  Then there is an isomorphism
  \begin{equation*}
    \bigoplus_{\alpha\geq0}\underline{\OB}_{\dR}^{\dag}(U)\left(\log t\right)^{\alpha}
    \isomap \underline{\OB}_{\pdR}^{\dag}(U)
  \end{equation*}
  of solid $k$-vector spaces.
\end{lem}

\begin{proof}
  Fix the notation as in~(\ref{eq:defnBdRdag-BpdRdag+-Bpddag-recpaper})
  on page~\pageref{eq:defnBdRdag-BpdRdag+-Bpddag-recpaper} and compute
  \begin{equation*}
    \OB_{\pdR}^{\dag}|_{X_{C}}
    =\OB_{\dR}^{\dag,+}|_{X_{C}}
      \widehat{\otimes}_{j^{-1}\BB_{\dR,*}^{\dag,+}}|_{X_{C}} j^{-1}\BB_{\pdR,*}^{\dag}|_{X_{C}}
    \stackrel{\text{\ref{lem:describeBpdRdagoverXC}}}{\cong}
      \bigoplus_{\alpha\geq0}\OB_{\dR}^{\dag}|_{X_{C}}\left(\log t\right)^{\alpha}.
  \end{equation*}
  Now use Lemma~\ref{lem:sheaves-on-Xproet-and-Xproetaffperfdfin}
  to work on the site $X_{\proet,\affperfd}^{\fin}$ and apply Lemma~\ref{lem:directsumsheaves-finitecoverings}
  to find
  \begin{equation*}
    \bigoplus_{\alpha\geq0}\OB_{\dR}^{\dag}(U)\left(\log t\right)^{\alpha}
    \isomap \OB_{\pdR}^{\dag}(U).
  \end{equation*}
  Now Lemma~\ref{lem:solidOBpdRdag-directsum-of-solidOBdRdag} follows from
  Lemma~\ref{lem:IndBan-to-Solid-preservesdirectsums-reconstructionpaper}, thanks to
  the descriptions of the sections of the involved sheaves as in
  Corollary~\ref{cor:cohomology-solidOBdRdag-over-affperfd-reconstructionpaper}
  and~\ref{cor:cohomology-solidOBpdRdag-over-affperfd-reconstructionpaper}.
\end{proof}

%%%%%%%%%%%%%%%%%%%%%%%%%%%%%%%%%%%%%%%%%%%%%%%%%%%%%%%%%%%%
% Local descriptions
%%%%%%%%%%%%%%%%%%%%%%%%%%%%%%%%%%%%%%%%%%%%%%%%%%%%%%%%%%%%

\subsection{Sections over affinoid perfectoids $U\times S$}

\begin{prop}\label{prop:S-to-OBlaU-isOBlaUtimesS}
  Suppose $X$ is affinoid and equipped with an étale morphism $X\to\TT^{d}$.
  Let $U\in X_{\proet}$ denote an affinoid perfectoid over $\widetilde{\TT}^{d}$.
  Then, for every profinite set $S$, functorially,
  \begin{equation*}
    \OB_{\dR}^{\dag,+}(U\times S)
    \cong
    \intHom_{\cont}\left(S,\OB_{\dR}^{\dag,+}(U)\right).
  \end{equation*}
\end{prop}

In the following, we establish preliminaries which are required in our proof of
Proposition~\ref{prop:S-to-OBlaU-isOBlaUtimesS}. Along the
way, we freely use the explicit description of the sections of
$\A_{\dR}^{>q}$ as in Theorem~\ref{thm:BdR>qplus+-sections-over-affperfd-recpaper}.

\begin{notation}\label{notationAdRgreaterthanqllbracketrrbracket-reconstructionpaper}
  Given $q\in\NN_{\geq2}$ and any affinoid perfectoid $V\in X_{\proet}$ over $\widetilde{\TT}^{d}$, write
  \begin{equation*}
    \A_{\dR}^{>q,+}\left(V\right)\left\llbracket \frac{Z_{1},\dots,Z_{d}}{p^{q}} \right\rrbracket
  \end{equation*}  
  for the $\left( p , \ker\theta_{\dR}^{>q}, Z_{1}/p^{q} , \dots, Z_{d}/p^{q}\right)$-adic completion of
  $\A_{\dR}^{q,+}\left( V\right)\left\< \frac{Z_{1},\dots,Z_{d}}{p^{q}} \right\>$.
  We equip it with the $\left( p , \ker\theta_{\dR}^{>q}, Z_{1}/p^{q} , \dots, Z_{d}/p^{q}\right)$-adic norm,
  such that it becomes an $\A_{\dR}^{>q,+}\left(V\right)$-Banach algebra
  by~\cite[\href{https://stacks.math.columbia.edu/tag/05GG}{Tag 05GG}]{stacks-project}.
  Here, we are using that $\ker\theta_{\dR}^{>q}$ is a principle ideal,
  cf. Lemma~\ref{lem:Fontaines-map-for-Alagreatq}.
\end{notation}

\begin{lem}\label{lem:sections-OBdRdag-over-affperfd-formalpowerseries-reconstructionpaper}
  Given any affinoid perfectoid $V\in X_{\proet}/\widetilde{\TT}^{d}$,
  canonically, %of $k_{0}$-ind-Banach spaces
  \begin{equation*}
    \text{``}\varinjlim_{q\in\NN_{\geq2}}\text{"}
    \A_{\dR}^{>q,+}\left(V\right)\left\llbracket \frac{Z_{1},\dots,Z_{d}}{p^{q}} \right\rrbracket
    \widehat{\otimes}_{W(\kappa)}k_{0}
    \isomap \OB_{\dR}^{\dag,+}\left(V\right).
  \end{equation*}
\end{lem}

\begin{proof}
  Following the proof of Lemma~\ref{lem:colimitAdRgreaterthanq}
  and using Theorem~\ref{thm:subsections-periodsheaves-affperfd},
  we find:
  \begin{equation*}
    \text{``}\varinjlim_{q\in\NN_{\geq2}}\text{"}
    \A_{\dR}^{>q,+}\left(V\right)\left\llbracket \frac{Z_{1},\dots,Z_{d}}{p^{q}} \right\rrbracket
    \isomap
    \text{``}\varinjlim_{q\in\NN}\text{"}
    \A_{\dR}^{q,+}\left(V\right)\left\< \frac{Z_{1},\dots,Z_{d}}{p^{q}} \right\>.
  \end{equation*}
  Now Lemma~\ref{lem:sections-OBdRdag-over-affperfd-formalpowerseries-reconstructionpaper}
  follows with Theorem~\ref{thm:cohomology-OBdRdag-over-affperfd-reconstructionpaper}.
\end{proof}

\begin{lem}\label{lem:AdRgreaterthanqRRpluszeta1-dots-zetad-reconstructionpaper}
  Fix $q\in\NN_{\geq2}$.
  Given any affinoid perfectoid $V\in X_{\proet}$ over $\widetilde{\TT}^{d}$,
  let $\A_{\dR}^{>q}\left(V\right)\llbracket \zeta_{1},\dots,\zeta_{d}\rrbracket$ denote the
  formal power series ring with coefficients in $\A_{\dR}^{>q}\left(V\right)$ in the formal
  variables $\zeta_{1},\dots,\zeta_{d}$, equipped with the
  $\left(p,\ker\theta_{\dR}^{>q},\zeta_{1},\dots,\zeta_{d}\right)$-adic
  seminorm. Then there exists an isomorphism
  \begin{equation}\label{eq:themap--lem:AdRgreaterthanqRRpluszeta1-dots-zetad-reconstructionpaper}
    \A_{\dR}^{>q}\left(V\right)\llbracket \zeta_{1},\dots,\zeta_{d}\rrbracket
    \isomap\A_{\dR}^{>q}\left(V\right)\left\llbracket \frac{Z_{1},\dots,Z_{d}}{p^{q}} \right\rrbracket
  \end{equation}
  of seminormed $\A_{\dR}^{>q}\left(V\right)$-Banach algebras,
  determined by $\zeta_{i}\mapsto Z_{i}/p^{q}$ for all $i=1,\dots,d$.
\end{lem}

\begin{proof}
  Indeed, the desired morphism exists. We continue by equipping
  $\A_{\dR}^{>q}\left(V\right)\llbracket \zeta_{1},\dots,\zeta_{d}\rrbracket$
  with the $\left(p,\ker\theta_{\dR}^{>q},\zeta_{1},\dots,\zeta_{d}\right)$-adic filtration
  and $\A_{\dR}^{>q}\left(V\right)\left\llbracket \frac{Z_{1},\dots,Z_{d}}{p^{q}} \right\rrbracket$
  with the $\left( p , \ker\theta_{\dR}^{>q}, Z_{1}/p^{q} , \dots, Z_{d}/p^{q}\right)$-adic filtration.
  By Lemma~\ref{lem:Fontaines-map-for-Alagreatq},
  $\ker\theta_{\dR}^{>q}$ admits a generator $\xi/p^{q}$. Furthermore,
  \emph{loc. cit.} implies that $\zeta_{d},\dots,\zeta_{1},\xi/p^{q},p$ and
  $Z_{d}/p^{q} , \dots, Z_{q}/p^{q},\xi/p^{q},p$ are regular sequences.
  Thus we can compute with~\cite[Exercise 17.16.a]{Ei95}:
  \begin{equation}\label{eq:AdRgreaterthanqRRpluszeta1-dots-zetad-computationintheproof-reconstructionpaper}
  \begin{split}
    \gr\A_{\dR}^{>q}\left(V\right)\llbracket \zeta_{1},\dots,\zeta_{d}\rrbracket
    &\cong\left(\widehat{\cal{O}}^{+}\left(V\right)/p\right)\left[\sigma\left(p\right),\sigma\left(\frac{\xi}{p^{q}}\right),\sigma\left(\zeta_{1}\right),\dots,\sigma\left(\zeta_{d}\right)\right], \text{ and } \\
    \gr\A_{\dR}^{>q}\left(V\right)\left\llbracket \frac{Z_{1},\dots,Z_{d}}{p^{q}} \right\rrbracket
    &\cong\gr\A_{\dR}^{q}\left(V\right)\left\< \frac{Z_{1},\dots,Z_{d}}{p^{q}} \right\> \\
    &\cong\left(\widehat{\cal{O}}^{+}\left(V\right)/p\right)\left[\sigma\left(p\right),\sigma\left(\frac{\xi}{p^{q}}\right),\sigma\left(\frac{Z_{1}}{p^{q}}\right),\dots,\sigma\left(\frac{Z_{d}}{p^{q}}\right)\right].
  \end{split}
  \end{equation}
  Here, $\sigma\left(-\right)$ refers to the principal symbol of a given element.
  It follows that the associated graded
  of~(\ref{eq:themap--lem:AdRgreaterthanqRRpluszeta1-dots-zetad-reconstructionpaper})
  is an isomorphism.
  Now~\cite[Chapter I, \S 4.2 page 31-32, Theorem 4(5)]{HuishiOystaeyen1996}
  implies Lemma~\ref{lem:AdRgreaterthanqRRpluszeta1-dots-zetad-reconstructionpaper}.
  Indeed, \emph{loc. cit.} applies because
  $\A_{\dR}^{>q}\left(V\right)\llbracket \zeta_{1},\dots,\zeta_{d}\rrbracket$ is complete,
  cf. Proposition~\ref{prop:Scomplete-Spowerseriescomplete-ifKoszulregular}.
\end{proof}

\begin{lem}\label{lem:S-to-OBlaU-isOBlaUtimesS-checkthattheHomcommuteswithlocalisation-reconstructionpaper}
  Given $q\in\NN_{\geq2}$ and
  $f\in\A_{\dR}^{>q,+}\left(V\right)\left\llbracket \frac{Z_{1},\dots,Z_{d}}{p^{q}} \right\rrbracket$,
  $\|pf\|=|p|\|f\|$.  
\end{lem}

\begin{proof}
  This follows from the computation of the associated graded as in
  the proof of Lemma~\ref{lem:AdRgreaterthanqRRpluszeta1-dots-zetad-reconstructionpaper}.
\end{proof}

\begin{lem}\label{lem:S-to-OBlaU-isOBlaUtimesS-inverselimit-reconstructionpaper}
  For $q\in\NN_{\geq2}$ and any affinoid perfectoid $V$ over $\widetilde{\TT}^{d}$, we have the isomorphism
  \begin{equation*}
    \A_{\dR}^{>q,+}\left( V \right)\left\llbracket \frac{Z_{1},\dots,Z_{d}}{p^{q}} \right\rrbracket
    \isomap\varprojlim_{j}\bigoplus_{\substack{\alpha\in\NN^{d} \\ |\alpha|<j}}\A_{\dR}^{>q,+}\left( V \right)\left(\frac{Z}{p^{q}}\right)^{\alpha}
  \end{equation*}
  of topological groups. Here, the direct sum is the direct sum of topological abelian groups,
  which in this instance coincides with the direct sum of Banach modules.
\end{lem}

\begin{proof}
  Using the identification as in Lemma~\ref{lem:AdRgreaterthanqRRpluszeta1-dots-zetad-reconstructionpaper},
  the desired isomorphism is
  \begin{equation*}
    \A_{\dR}^{>q}\left(V\right)\llbracket \zeta_{1},\dots,\zeta_{d}\rrbracket
    \isomap\varprojlim_{j}\A_{\dR}^{>q}\left(V\right)\llbracket \zeta_{1},\dots,\zeta_{d}\rrbracket/\left(\zeta_{1},\dots,\zeta_{d}\right)^{j}.
  \end{equation*}
  This is an isomorphism of rings
  by Proposition~\ref{prop:Scomplete-Spowerseriescomplete-ifKoszulregular}.
  Note that \emph{loc. cit.} applies by the discussion in the proof of
  Lemma~\ref{lem:AdRgreaterthanqRRpluszeta1-dots-zetad-reconstructionpaper}.
  It remains to compare the topologies, which is straightforward.
\end{proof}

\begin{proof}[Proof of Proposition~\ref{prop:S-to-OBlaU-isOBlaUtimesS}]
  Compute
  \begin{align*}
    \A_{\dR}^{>q,+}\left( U \times S \right)\left\llbracket \frac{Z_{1},\dots,Z_{d}}{p^{q}} \right\rrbracket
    &\stackrel{\text{\ref{lem:S-to-OBlaU-isOBlaUtimesS-inverselimit-reconstructionpaper}}}{\cong}
      \varprojlim_{j}\bigoplus_{\substack{\alpha\in\NN^{d} \\ |\alpha|<j}}\A_{\dR}^{>q,+}\left( U \times S \right)\left(\frac{Z}{p^{q}}\right)^{\alpha} \\
    &\stackrel{\text{\ref{prop:AdRgreaterthanqUtimesS-isomapHomcontSAdRgreaterthanqU-reconstructionpaper}}}{\cong}
    \varprojlim_{j}\bigoplus_{\substack{\alpha\in\NN^{d} \\ |\alpha|<j}}\intHom_{\cont}\left(S,\A_{\dR}^{>q,+}\left( U \right)\right)\left(\frac{Z}{p^{q}}\right)^{\alpha} \\
    &\stackrel{\text{\ref{lem:finitedirectsum-comutes-withHomcont-reconstructionpaper}}}{\cong}
    \intHom_{\cont}\left(S,\varprojlim_{j}\bigoplus_{\substack{\alpha\in\NN^{d} \\ |\alpha|<j}}\A_{\dR}^{>q,+}\left( U \right)\left(\frac{Z}{p^{q}}\right)^{\alpha}\right) \\
    &\stackrel{\text{\ref{lem:S-to-OBlaU-isOBlaUtimesS-inverselimit-reconstructionpaper}}}{\cong}
    \intHom_{\cont}\left(S,\A_{\dR}^{>q,+}\left( U \right)\left\llbracket \frac{Z_{1},\dots,Z_{d}}{p^{q}} \right\rrbracket\right)
  \end{align*}
  Thus Proposition~\ref{prop:S-to-OBlaU-isOBlaUtimesS} follows
  from Lemma~\ref{lem:sections-OBdRdag-over-affperfd-formalpowerseries-reconstructionpaper}
  and~\ref{lem:HomcontS-commutes-completed-localisation-overWkappa-reconstructionpaper},
  which applies by Lemma~\ref{lem:S-to-OBlaU-isOBlaUtimesS-checkthattheHomcommuteswithlocalisation-reconstructionpaper}.
\end{proof}

We end this section with the following Corollaries~\ref{cor:OBdRdag-overUtimesS-reconstructionpaper}
and~\ref{cor:OBpdRdag-overUtimesS-reconstructionpaper}.

\begin{cor}\label{cor:OBdRdag-overUtimesS-reconstructionpaper}
  Suppose $X$ is affinoid and equipped with an étale morphism $X\to\TT^{d}$.
  Let $U\in X_{\proet}$ denote an affinoid perfectoid over $\widetilde{\TT}_{C}^{d}$.
  Then, for every profinite set $S$, functorially,
  \begin{equation*}
    \OB_{\dR}^{\dag}(U\times S)
    \cong
    \intHom_{\cont}\left(S,\OB_{\dR}^{\dag}(U)\right).
  \end{equation*}
\end{cor}

\begin{proof}
  For any affinoid perfectoid $V$ over $\widetilde{\TT}_{C}^{d}$,
  Theorem~\ref{thm:cohomology-OBdRdag-over-affperfd-reconstructionpaper}
  gives a functorial isomorphism
  \begin{equation*}
    \OB_{\dR}^{\dag}(V)
    \cong
    \varinjlim_{t\times}\OB_{\dR}^{\dag,+}(V).
  \end{equation*}  
  %of $k$-ind-Banach modules.
  Thus the result follows from
  Proposition~\ref{prop:S-to-OBlaU-isOBlaUtimesS}
  and the cocontinuity  of $\Hom_{\cont}\left(S,-\right)$.
\end{proof}

\begin{cor}\label{cor:OBpdRdag-overUtimesS-reconstructionpaper}
  Suppose $X$ is affinoid and equipped with an étale morphism $X\to\TT^{d}$.
  Let $U\in X_{\proet}$ denote an affinoid perfectoid over $\widetilde{\TT}_{C}^{d}$.
  Then, for every profinite set $S$, functorially,
  \begin{equation*}
    \OB_{\pdR}^{\dag}(U\times S)
    \cong
    \intHom_{\cont}\left(S,\OB_{\pdR}^{\dag}(U)\right).
  \end{equation*}
\end{cor}

\begin{proof}
  For any affinoid perfectoid $V$ over $\widetilde{\TT}_{C}^{d}$,
  Theorem~\ref{thm:cohomology-OBpdRdag-over-affperfd-reconstructionpaper}
  gives a functorial isomorphism
  \begin{equation*}
    \OB_{\pdR}^{\dag}(V)
    \cong
    \bigoplus_{\alpha\in\NN}\OB_{\dR}^{\dag}(V)\left(\log t\right)^{\alpha}.
  \end{equation*}  
  Thus the result follows from
  Corollary~\ref{cor:OBdRdag-overUtimesS-reconstructionpaper} and
  Lemma~\ref{lem:finitedirectsum-comutes-withHomcont-reconstructionpaper}.
\end{proof}

%%%%%%%%%%%%%%%%%%%%%%%%%%%%%%%%%%%%%%%%%%%%%%%%%%%%%%%%%%%%
%%%%%%%%%%%%%%%%%%%%%%%%%%%%%%%%%%%%%%%%%%%%%%%%%%%%%%%%%%%%
% Proof of local description of OBla
%%%%%%%%%%%%%%%%%%%%%%%%%%%%%%%%%%%%%%%%%%%%%%%%%%%%%%%%%%%%
%%%%%%%%%%%%%%%%%%%%%%%%%%%%%%%%%%%%%%%%%%%%%%%%%%%%%%%%%%%%

\section{Proofs of the local descriptions}

%%%%%%%%%%%%%%%%%%%%%%%%%%%%%%%%%%%%%%%%%%%%%%%%%%%%%%%%%%%%
%%%%%%%%%%%%%%%%%%%%%%%%%%%%%%%%%%%%%%%%%%%%%%%%%%%%%%%%%%%%
% Proof of local description of OBla
%%%%%%%%%%%%%%%%%%%%%%%%%%%%%%%%%%%%%%%%%%%%%%%%%%%%%%%%%%%%
%%%%%%%%%%%%%%%%%%%%%%%%%%%%%%%%%%%%%%%%%%%%%%%%%%%%%%%%%%%%

\subsection{Proof of Theorem~\ref{thm:localdescription-of-OBla}}
\label{subsec:proof-localdescrption-of-OBla}

Fix the setup as described in Theorem~\ref{thm:localdescription-of-OBla}.
%From the definition of $\Phi_{\dR}$, we see that
It suffices to check that
$\Phi_{\dR}^{\dag,+}$ is an isomorphism. We may show that the maps
\begin{equation*}
  \Phi_{\dR}^{\dag,+,\psh}\left(U\right)\colon
  \BB_{\dR}^{\dag,+}\left(U\right)\left\<\frac{Z_{1},\dots,Z_{d}}{p^{\infty}}\right\>
  \to\OB_{\dR}^{\dag,+,\psh}\left(U\right)
\end{equation*}
are isomorphisms for every affinoid perfectoid $U\in X_{\proet}/\widetilde{X}$.
Fix such a $U$ with $\widehat{U}=\Spa\left( R , R^{+} \right)$, together
with a pro-étale presentation
$U=\text{``}\varprojlim_{i\in I}\text{"}U_{i}\in X_{\proet}/\widetilde{X}$,
$U_{i}=\Spa\left(R_{i} , R_{i}^{+} \right)$ for all $i\in I$.
Let $i\geq i_{0}$ be arbitrary, cf. Notation~\ref{notation:i0}.
We will show that the morphism
\begin{equation}\label{eq:localdescription-of-OBla-phii}
  \phi_{i}\colon
  \BB_{\dR}^{\dag,+}\left( R,R^{+} \right)\left\<\frac{Z_{1},\dots,Z_{d}}{p^{\infty}}\right\>
  \to
  \left(
  %\mathcal{O}^{+}(U_{i})
  R_{i}^{+}\widehat{\otimes}_{W(\kappa)}\A_{\inf}\left( R,R^{+} \right)\right)
  \left\<\frac{\ker\Otheta_{\inf}}{p^{\infty}}\right\>\widehat{\otimes}_{k^{\circ}}k
\end{equation}
is an isomorphism. This suffices because
$\Phi_{\dR}^{\dag,+,\psh}\left(U\right)=\varinjlim_{i\geq i_{0}}\phi_{i}$.

Along the way, we will use the following
Lemma~\ref{lem:Ri-andRiplus-Banachalgebra}
without further reference.
         
\begin{lem}\label{lem:Ri-andRiplus-Banachalgebra}
  $R_{i}$ is an affinoid $k$-algebra.
\end{lem}
    
\begin{proof}
  The étale map $U_{i}\to \TT^{d}$ induces an étale morphism
  %\begin{equation*}
    $k\left\< T_{1}^{\pm} , \dots , T_{d}^{\pm}\right\> \to R_{i}$
  %\end{equation*}
  of Huber rings. Now apply~\cite[Proposition 1.7.1(iii) and Corollary 1.7.2(ii)]{Huber96Etale}.
  %In particular, $R_{i}^{\circ}$ is $p$-adically complete. Using arguments similar to
  %the~\cite[discussion at the end of the proof of Lemma 3.6.1]{BSSW2024_rationalizationoftheKnlocalsphere},
  %we find that $R_{i}^{+}$ is $p$-adically complete. It is separated
  %because $R_{i}$ is separated.
\end{proof}

\begin{lem}\label{lem:presentation--lem:Ri-andRiplus-Banachalgebra}
  There exists a presentation
  \begin{equation*}
    R_{i}
    =k_{0}\left\<T_{1}^{\pm},\dots,T_{d}^{\pm}, L_{i1},\dots,L_{in_{i}}\right\>
      / \left( P_{i1},\dots,P_{in_{i}}\right)
  \end{equation*}
  for some formal variables $L_{i1},\dots,L_{in_{i}}$ such that
  the Jacobian $J_{P_{i}}\in R_{i}$, that is the determinant of the $n\times n$-matrix
  \begin{equation*}
    D_{P_{i}}:=
    \begin{pmatrix}
      \frac{\partial P_{i1}}{\partial L_{i1}} & \cdots & \frac{\partial P_{i1}}{\partial L_{in}} \\
        \vdots & \ddots & \vdots \\
        \frac{\partial P_{in}}{\partial L_{i1}} & \cdots & \frac{\partial P_{in}}{\partial L_{in}}
    \end{pmatrix}
  \in \Mat_{n}\left(R_{i}\right).
\end{equation*}
\end{lem}

\begin{proof}
  This is because the composition
  $k_{0}\left\< T_{1}^{\pm} , \dots , T_{d}^{\pm}\right\>\to k\left\< T_{1}^{\pm} , \dots , T_{d}^{\pm}\right\> \to R_{i}$
  is étale.
\end{proof}

Lemma~\ref{lem:presentation--lem:Ri-andRiplus-Banachalgebra} makes Notation~\ref{notation:qi} meaningful.

\begin{notation}\label{notation:qi}
  We fix $q_{i}\in\NN_{\geq2}$ such that
  $p^{q_{i}}J_{P_{i}}^{-1}\in R_{i}^{+}$.
\end{notation}

\begin{lem}\label{lem:imagetildenu-Ainf}
  For $\widetilde{q}_{i}>q_{i}$ large enough, the assignment
  $T_{l} \mapsto [T_{l}^{\flat}]+Z_{l}$
  defines a unique morphism
  \begin{equation*}
    \epsilon_{i}
    \colon
    R_{i} \maps \BB_{\dR}^{\widetilde{q}_{i},+}\left(R,R^{+}\right)\left\< \frac{Z_{1},\dots,Z_{d}}{p^{\widetilde{q}_{i}}} \right\>.
  \end{equation*}
  of $k_{0}$-Banach algebras which fits into the commutative diagram
  \begin{equation}\label{cd:imagetildenu-Ainf}
    \begin{tikzcd}
      R_{i}
        \arrow{rd}\arrow{r}{\epsilon_{i}} &
      \BB_{\dR}^{\widetilde{q}_{i},+}\left(R,R^{+}\right)\left\< \frac{Z_{1},\dots,Z_{d}}{\widetilde{q}_{i}} \right\>
        \arrow{d}{\Ovartheta_{\dR}^{\widetilde{q}_{i},+,\prime}} \\
      \empty &
      R,
    \end{tikzcd}
  \end{equation}
  where $\Ovartheta_{\dR}^{\widetilde{q}_{i},+,\prime}$ is the map
  $\sum_{\alpha\in\NN^{d}}b_{\alpha}Z^{\alpha}\mapsto\vartheta_{\dR}^{\widetilde{q}_{i},+}\left(b_{0}\right)$.
\end{lem}

\begin{proof}
  We closely follow the~\cite[proof of Proposition 6.10]{Sch13pAdicHodge}.

  Fix a presentation as in Lemma~\ref{lem:presentation--lem:Ri-andRiplus-Banachalgebra}:
  \begin{equation*}
    R_{i}
    =k_{0}\left\<T_{1}^{\pm},\dots,T_{d}^{\pm}, L_{i1},\dots,L_{in_{i}}\right\>
      / \left( P_{i1},\dots,P_{in_{i}}\right).
  \end{equation*}
  We use multi-index notation
  $L_{i}:=\left(L_{i1},\dots,L_{in_{i}}\right)$ and
  $P_{i}:=\left(P_{i1},\dots,P_{in_{i}}\right)$, omitting the index
  $_{i}i$ for clarity.
  With loss of generality, we assume that
  \begin{equation*}
    P_{i1},\dots,P_{in_{i}}\in W(\kappa)\left\<T_{1}^{\pm},\dots,T_{d}^{\pm}, L_{i1},\dots,L_{in_{i}}\right\>;
  \end{equation*}
  otherwise, we rescale the $P_{ij}$ by a nonzero scalar in $k_{0}$.

  The $\left[T_{l}^{\flat}\right]+Z_{l}$ are invertible in
  $\A_{\inf}\left(R,R^{+}\right)\left\<Z_{1},\dots,Z_{d}\right\>$,
  with inverses
  $\sum_{n\geq 0}(-1)^{n}\left[T_{l}^{\flat}\right]^{-n-1}Z_{l}^{n}$.
  Therefore, we have a morphism
  \begin{equation}\label{eq:imagetildenu-Ainf-1}
    W(\kappa)\left\<T_{1}^{\pm},\dots,T_{d}^{\pm}, L_{i1},\dots,L_{in_{i}}\right\>
    \to
    \A_{\inf}\left(R,R^{+}\right)\left\<Z_{1},\dots,Z_{d},L_{i1},\dots,L_{in_{i}}\right\>
  \end{equation}
  of $W\left(\kappa\right)$-Banach algebras
  sending $T_{l}$ to $\left[T_{l}^{\flat}\right]+Z_{l}$ for all $l=1,\dots,d$.
  Let $P_{ij}^{\prime}$ denote the image of the restricted power series
  $P_{ij}$ with respect to the map~(\ref{eq:imagetildenu-Ainf-1}).
  
  Next, we pick and fix a arbitrary tuple
  \begin{equation*}
    L_{i}^{\prime}=\left(L_{i1}^{\prime},\dots,L_{in}^{\prime}\right)
    \in\prod_{j=1}^{n_{i}}\A_{\inf}\left(R,R^{+}\right)\left\<Z_{1},\dots,Z_{d}\right\>
  \end{equation*}
  such that $\Otheta_{\inf}^{\prime}\left(L_{ij}^{\prime}\right)= L_{ij}$ for all $j=1,\dots n_{i}$;
  this is possible because $\Otheta_{\inf}^{\prime}$ is surjective.
  
  \iffalse %%% MAYBE INCLUDE THIS LATER ON???
  Furthermore,
  \begin{equation*}
    \widetilde{q}_{i}:=\max\left\{ q_{i}+1 ,
      \left\lceil \log_{|p|} \left( \left\| \phi^{(q+1)}\left(D_{P_{i}}(s)^{-1}\right)\right\|_{\Mat_{n}(R)}^{-2} \right)+1
      \right\rceil \right\} \in \NN_{>2}.
  \end{equation*}
  \fi %%% COMMENT ENDS
  
  \begin{claim}\label{claim:applicationHensel}
    There exists a unique tuple
    \begin{equation*}
      \widetilde{L}_{i}^{\prime}=\left(\widetilde{L}_{i1}^{\prime},\dots,\widetilde{L}_{in}^{\prime}\right)
      \in\prod_{j=1}^{n_{i}}\BB_{\dR}^{\widetilde{q}_{i},+}\left(R,R^{+}\right)\left\< \frac{Z_{1},\dots,Z_{d}}{p^{\widetilde{q}_{i}}}\right\>
    \end{equation*}
    such that $P_{ij}^{\prime}\left(\widetilde{L}_{ij}^{\prime}\right)=0$ and
    $\Ovartheta_{\dR}^{\widetilde{q}_{i},+,\prime}\left(\widetilde{L}_{ij}^{\prime}\right)= L_{ij}$ for all $j=1,\dots n_{i}$.
    These $\widetilde{L}_{ij}^{\prime}$ are furthermore power-bounded.
  \end{claim}

  \begin{proof}[Proof of Claim~\ref{claim:applicationHensel}]
    We start with existence. We would like to use the methods in~\S\ref{subsec:Hensel}.
    For clarity, we compared \emph{loc. cit.} with the setup here in the proof of Claim~\ref{claim:applicationHensel}
    in the table~\ref{table:compare-notation} on page~\pageref{table:compare-notation}.
\begin{table}
  \centering
  \caption{Comparison of Notation}
  \label{table:compare-notation}
  \begin{tabular}{|l|l|}
    \hline
    \textbf{Notation in~\S\ref{subsec:Hensel}} &
    \textbf{Notation in the proof of Claim~\ref{claim:applicationHensel}} \\
    \hline
    $F$ & $k_{0}$ \\
    \hline
    $R=F^{\circ}$ & $W(\kappa)$ \\
    \hline
    $\pi\in F^{\circ}$ pseudo-uniformiser & $p\in W(\kappa)$ \\
    \hline
    $q\in\NN$ & $q_{i}$ as in Notation~\ref{notation:qi} \\
    \hline
    $S$ & $\A_{\inf}\left(R,R^{+}\right)\left\< Z_{1},\dots,Z_{d}\right\>$ \\
    \hline
    $s=\left(s_{1},\dots,s_{n}\right)\in S^{n}$ & $L_{i}^{\prime}=\left(L_{i1}^{\prime},\dots,L_{in_{i}}^{\prime}\right)\in\prod_{j=1}^{n_{i}}\A_{\inf}\left(R,R^{+}\right)\left\< Z_{1},\dots,Z_{d}\right\>$ \\
    \hline
    $I\subseteq S$ & $\left(\ker\theta_{\inf}, Z_{1},\dots,Z_{d}\right)\subseteq\A_{\inf}\left(R,R^{+}\right)\left\< Z_{1},\dots,Z_{d}\right\>$ \\
    \hline
    $S^{(q)}:=S\left\<I/\pi^{q}\right\>$ & $\A_{\dR}^{q_{i}}\left(R,R^{+}\right)\left\< \frac{Z_{1},\dots,Z_{d}}{p^{q_{i}}}\right\>$ \\
    \hline
    $T^{(q)}:= S^{(q)}\widehat{\otimes}_{F^{\circ}}F$ & $\BB_{\dR}^{q_{i}}\left(R,R^{+}\right)\left\< \frac{Z_{1},\dots,Z_{d}}{p^{q_{i}}}\right\>$ \\
    \hline
    $\phi^{(q)}\colon S^{(q)} \to S/I$ &
      \makecell[l]{$\begin{aligned}
        \Otheta_{\dR}^{q_{i},\prime}\colon\A_{\dR}^{q_{i}}\left(R,R^{+}\right)\left\< \frac{Z_{1},\dots,Z_{d}}{p^{q_{i}}}\right\>&\to R^{+}, \\
        {\sum}_{\alpha\in\NN^{d}}a_{\alpha}Z^{\alpha}&\mapsto\theta_{\dR}^{q_{i}}\left(a_{0}\right)
        \end{aligned}$} \\
    \hline
    $\psi^{(q)}\colon T^{(q)} \to \left(S/I\right)\widehat{\otimes}_{F^{\circ}} F$ &
    \makecell[l]{$\begin{aligned}
        \Ovartheta_{\dR}^{q_{i},+,\prime}\colon\BB_{\dR}^{q_{i},+}\left(R,R^{+}\right)\left\< \frac{Z_{1},\dots,Z_{d}}{p^{q_{i}}}\right\>&\to R, \\
        {\sum}_{\alpha\in\NN^{d}}b_{\alpha}Z^{\alpha}&\mapsto\vartheta_{\dR}^{q_{i},+}\left(b_{0}\right)
        \end{aligned}$} \\
    \hline
    $I^{(q)}:=\ker\phi^{(q)} \subseteq S^{(q)}$ & $\ker\Otheta_{\dR}^{q_{i},\prime}\subseteq\A_{\dR}^{q_{i}}\left(R,R^{+}\right)\left\< \frac{Z_{1},\dots,Z_{d}}{p^{q_{i}}}\right\>$ \\
    \hline
    $J^{(q)}:=\ker\psi^{(q)} \subseteq T^{(q)}$ & $\ker\Ovartheta_{\dR}^{q_{i},+\prime}\subseteq\BB_{\dR}^{q_{i},+}\left(R,R^{+}\right)\left\< \frac{Z_{1},\dots,Z_{d}}{p^{q_{i}}}\right\>$ \\
    \hline
    $L=\left(L_{1},\dots,L_{n}\right)$ & $L_{i}=\left(L_{i1},\dots,L_{in}\right)$  \\
    \hline
    \makecell[l]{$H=\left(H_{1},\dots,H_{n}\right)$, where \\
    $H_{j}\in S\<L\>=S\left\<L_{1},\dots,L_{n}\right\>$ for all $j=1,\dots,n$} &
      $P_{i}^{\prime}=\left(P_{i1}^{\prime},\dots,P_{in_{i}}^{\prime}\right)$\\
    \hline
    $D_{H}$ is the Jacobian of $H$ & $D_{P_{i}^{\prime}}$ is the Jacobian of $P_{i}^{\prime}$ \\
    \hline
    $J_{H}:=\det D_{H}\in S\left\< L \right\>$ & $J_{P_{i}^{\prime}}:=\det D_{P_{i}^{\prime}}\in\A_{\inf}\left(R,R^{+}\right)\left\< Z_{1},\dots,Z_{d},L_{i1},\dots,L_{in}\right\>$ \\
    \hline
  \end{tabular}
\end{table}
With this notation, we now check Condition~\ref{condition:Hensel-analytic} in
the setup of Claim~\ref{claim:applicationHensel}:
\begin{itemize}
  \item[(a)] Set $G:= p^{q_{i}}J_{P_{i}}^{-1}\in R_{i}^{+}$, cf. Notation~\ref{notation:qi}.
    Fix any $G^{\prime}\in\A_{\inf}\left(R,R^{+}\right)\left\<Z_{1},\dots,Z_{d}\right\>$
    such that $\Otheta_{\inf}\left(G^{\prime}\right)=G$. Now we compute that
    $G^{\prime}$ is the desired element $r$ in Condition~\ref{condition:Hensel-analytic}:
    \begin{equation*}
      \Otheta_{\inf}^{\prime}\left(J_{P_{i}}\left(L_{i}^{\prime}\right)G^{\prime}-p^{q}\right)
      =J_{P_{i}}G-p^{q}=0,
    \end{equation*}
    that is $J_{P_{i}}\left(L_{i}^{\prime}\right)G^{\prime}-p^{q}\in\ker\Otheta_{\inf}$,
    as desired.
  \item[(b)] It suffices to check that $p\in\A_{\dR}^{q}\left(R,R^{+}\right)\left\<\frac{Z_{1},\dots,Z_{d}}{p^{q}}\right\>$
    is not a zero-divisor for all $q\geq 2$.
    Now simply apply Lemma~\ref{AdR-ptorsionfree-reconstructionpaper}.
  \item[(c)] Indeed, we have
  \begin{equation*}
    \Otheta_{\inf}\left(P_{ij}^{\prime}\left(L_{i}^{\prime}\right)\right)
    =P_{ij}\left(L_{i}\right)
    =0,
  \end{equation*}
  that is $P_{ij}^{\prime}\left(L_{i}^{\prime}\right)\equiv0\mod\ker\Otheta_{\inf}^{\prime}$,
  for all $j=1,\dots,n_{i}$.
\end{itemize}
We have thus verified Condition~\ref{condition:Hensel-analytic}. Therefore,
Theorem~\ref{thm:Hensel-analytic} applies. This gives us a desired choice of
$\widetilde{L}_{i}^{\prime}$, which \emph{loc. cit.} denotes by $\widetilde{s}$.
Theorem~\ref{thm:Hensel-analytic} also implies that the $\widetilde{L}_{ij}^{\prime}$
are power-bounded.

    However, these methods only yield uniqueness with respect
    to the condition in Theorem~\ref{thm:Hensel-analytic}(iii).
    Now we prove the stronger uniquness to finish the proof of Claim~\ref{claim:applicationHensel}.
    
    Since $\BB_{\dR}^{\widetilde{q}_{i},+}\left(R,R^{+}\right)\subseteq\BB_{\dR}^{+}\left(R,R^{+}\right)$,
    cf. Proposition~\ref{prop:BdRq+-to-BdR+-inj},
    Claim~\ref{claim:applicationHensel} follows from the following Lemma~\ref{lem:applicationHenselBdR+}: 
\end{proof}

  \begin{claim}\label{lem:applicationHenselBdR+}
    There exists a unique tuple
    \begin{equation*}
      \widetilde{L}_{i}^{\prime}=\left(\widetilde{L}_{i1}^{\prime},\dots,\widetilde{L}_{in}^{\prime}\right)
      \in\prod_{j=1}^{n_{i}}\BB_{\dR}^{+}\left(R,R^{+}\right)\left\llbracket Z_{1},\dots,Z_{d} \right\rrbracket
    \end{equation*}
    such that $P_{ij}^{\prime}\left(\widetilde{L}_{ij}^{\prime}\right)=0$ and
    $\Ovartheta_{\dR}^{,+,\prime}\left(\widetilde{L}_{ij}^{\prime}\right)= L_{ij}$ for all $j=1,\dots n_{i}$.
    Here, $\Ovartheta_{\dR}^{+}$ is the map
    $\BB_{\dR}^{+}\left(R,R^{+}\right)\left\llbracket Z_{1},\dots,Z_{d} \right\rrbracket\to R$,
    $\sum_{\alpha\in\NN^{d}}b_{\alpha}Z^{\alpha}\mapsto\vartheta_{\dR}^{+}\left(b_{0}\right)$,
    where $\vartheta_{\dR}^{+}\colon\BB_{\dR}^{+}\left(R,R^{+}\right)\to R$ is Fontaine's map.
  \end{claim}  
  
  \begin{proof}
    This follows from an application of a multidimensional version of Hensel's Lemma
    (see, for example, \cite[Corollary 4.5.2]{BourbakiCommAlgebra1972}) and is explained
    in the~\cite[proof of Proposition 6.10]{Sch13pAdicHodge}.
  \end{proof}
  
  Back to the proof of Lemma~\ref{lem:imagetildenu-Ainf}.
  Fix $\widetilde{L}_{i}^{\prime}$ as in Claim~\ref{claim:applicationHensel}.
  This allows us to define the morphism of $k_{0}$-Banach algebras
  \begin{equation*}
    \epsilon_{i}
    \colon
    R_{i}= k_{0}\left\<T_{1}^{\pm},\dots,T_{d}^{\pm}, L_{i1},\dots,L_{in_{i}}\right\> / \left( P_{i1},\dots,P_{in_{i}}\right)
      \maps \BB_{\dR}^{\widetilde{q}_{i},+}\left(R,R^{+}\right)\left\< \frac{Z_{1},\dots,Z_{d}}{p^{\widetilde{q}_{i}}} \right\>.
  \end{equation*}
  via $T_{l}\mapsto \left[T_{l}^{\flat}\right]+Z_{l}$ for all $l=1,\dots,d$
  and $L_{ij}\mapsto\widetilde{L}_{ij}^{\prime}$ for all $j=1,\dots,n_{i}$.
  Claim~\ref{claim:applicationHensel}
  immediately implies that the diagram~(\ref{lem:imagetildenu-Ainf})
  commutes, and that $\epsilon_{i}$ is unique with this property.
  This finishes the proof of Lemma~\ref{lem:imagetildenu-Ainf}.
\end{proof}

\begin{notation}\label{notation:Riplus-algebra-Ainf}
  Introduce the following bounded linear map
  \begin{alignat*}{3}
    \delta_{i}\colon R_{i}
    \stackrel{\epsilon_{i}^{+}}{\longrightarrow}
    \BB_{\dR}^{\widetilde{q}_{i}}\left(R,R^{+}\right)\left\<\frac{Z_{1},\dots,Z_{d}}{p^{\widetilde{q}_{i}}}\right\>&
    &\longrightarrow
    &\BB_{\dR}^{\widetilde{q}_{i}}\left(R,R^{+}\right) \\
    \sum_{\alpha\in\NN^{d}}b_{\alpha}Z^{\alpha}&
    &\longmapsto
    &b_{0}.
  \end{alignat*}
\end{notation}

The map $\delta_{i}$ gives $\BB_{\dR}^{\widetilde{q}_{i}}\left(R,R^{+}\right)$
the structure of an $R_{i}$-Banach algebra. We want
exhibit $\phi_{i}$ as a base-change of a certain $\tau_{i}$
along $\delta_{i}$, following the proof of the local description
of $\OB_{\dR}^{+}$~\cite{Sch13pAdicHodgeErratum}.

\begin{notation}
  Denote the multiplication $R_{i}\widehat{\otimes}_{k_{0}}R_{i}\to R_{i}$
  by $\mu_{i}$.
\end{notation}

\begin{lem}\label{lem:localdescription-of-OBla-defineG}
  The subset
  \begin{equation*}
    G:=\left(\id_{R_{i}}\widehat{\otimes}_{k_{0}}\delta_{i}\right)\left(\ker\mu_{i}\right)
    \cup\left\{1\widehat{\otimes}\xi\right\}
    \subseteq
    R_{i}\widehat{\otimes}_{k_{0}}\BB_{\dR}^{\widetilde{q}_{i}}\left(R,R^{+}\right)
  \end{equation*}
  generates the ideal $\ker\Ovartheta_{\dR}^{\widetilde{q}_{i}}$.
\end{lem}

\begin{proof}
  Drop the subscripts $k_{0}$ for clarity.
  $\ker\Ovartheta_{\dR}^{\widetilde{q}_{i}}$ is the composition
  \begin{equation*}
    R_{i}\widehat{\otimes}_{W(\kappa)}\BB_{\dR}^{\widetilde{q}_{i}}\left(R,R^{+}\right)
    \stackrel{\id\widehat{\otimes}\vartheta_{\dR}^{\widetilde{q}_{i}}}{\longrightarrow}
    R_{i}\widehat{\otimes}R
    \stackrel{\mu}{\longrightarrow}
    R,
  \end{equation*}
  where $\mu$ is the multiplication.
  Clearly, $1\widehat{\otimes}\xi\in\ker\Ovartheta_{\dR}^{\widetilde{q}_{i}}$. For every
  $g\in\ker\mu_{i}$,
  \begin{equation}\label{eq:localdescription-of-OBla-defineG}
  \begin{split}
    \Ovartheta_{\dR}^{\widetilde{q}_{i}}\left(\left(\id_{R_{i}}\widehat{\otimes}\delta_{i}\right)\left(g\right)\right)
    &=\left(
    \mu
    \circ\left( \id_{R_{i}}\widehat{\otimes}\vartheta_{\dR}^{\widetilde{q}_{i}}\right)
    \circ\left(\id_{R_{i}}\widehat{\otimes}\delta_{i}\right)
    \right)\left( g\right) \\
    &=\left(
    \mu
    \circ\left( \id_{R_{i}}\widehat{\otimes}\left(\vartheta_{\dR}^{\widetilde{q}_{i}}\circ\delta_{i}\right)\right)
    \right)\left( g\right) \\
    &\stackrel{\text{\ref{lem:imagetildenu-Ainf}}}{=}\left(
    \mu
    \circ\left( \id_{R_{i}}\widehat{\otimes}\iota_{i}\right)
    \right)\left( g\right) \\
    &=0,
  \end{split}
  \end{equation}
  where $\iota_{i}\colon R_{i}\to R$ is the canonical map.
  This proves $\left(G\right)\subseteq\ker\Ovartheta_{\dR}^{\widetilde{q}_{i}}$.
  We aim to show $\supseteq$ via computing that $\Ovartheta_{\dR}^{\widetilde{q}_{i}}$
  induces an isomorphism
  \begin{equation*}
    \left. \left(R_{i}\widehat{\otimes}\BB_{\dR}^{\widetilde{q}_{i}}\left( R,R^{+}\right)\right)
    \middle/ \left( G\right) \right.
    \stackrel{\cong}{\longrightarrow}
    R.
  \end{equation*}
  By the third isomorphism theorem, this breaks into two parts. First,
  \begin{equation}\label{eq:localdescription-of-OBla-defineG-1}
      \left. \left(R_{i}\widehat{\otimes}\BB_{\dR}^{\widetilde{q}_{i}}\left( R,R^{+}\right)\right)
      \middle/ \left( 1\widehat{\otimes}\xi\right) \right.
      \stackrel{\cong}{\longrightarrow}
      R_{i}\widehat{\otimes}R;
  \end{equation}
  this isomorphism is induced by $\id_{R_{i}}\widehat{\otimes}\Ovartheta_{\dR}^{\widetilde{q}_{i}}$.
  Second, $\mu$ induces
  \begin{equation}\label{eq:localdescription-of-OBla-defineG-2}
      \left.\left(R_{i}\widehat{\otimes}R\right)
      \middle/
      \left(\left(\id_{R_{i}}\widehat{\otimes}\vartheta_{\dR}^{\widetilde{q}_{i}}\right)\left(
      \left(\id_{R_{i}}\widehat{\otimes}\delta_{i}\right)\left(\ker\mu_{i}\right)
      \right)
      \right)
      \right.
      \stackrel{\cong}{\longrightarrow}
      R.
  \end{equation}
  The first isomorphism
  comes from the strictly coexact sequence
  \begin{equation*}
    \BB_{\dR}^{\widetilde{q}_{i}}\left( R,R^{+}\right)
    \stackrel{\xi}{\longrightarrow}
    \BB_{\dR}^{\widetilde{q}_{i}}\left( R,R^{+}\right)
    \longrightarrow
    R
    \longrightarrow
    0.
  \end{equation*}
  Indeed, applying $R_{i}\widehat{\otimes}-$
  gives the sequence
  \begin{equation*}
    R_{i}\widehat{\otimes}\BB_{\dR}^{\widetilde{q}_{i}}\left( R,R^{+}\right)
    \stackrel{\xi}{\longrightarrow}
    R_{i}\widehat{\otimes}\BB_{\dR}^{\widetilde{q}_{i}}\left( R,R^{+}\right)
    \longrightarrow
    R_{i}\widehat{\otimes}R
    \longrightarrow
    0,
  \end{equation*}
  which is strictly coexact.
  This gives~(\ref{eq:localdescription-of-OBla-defineG-1}).
  Regarding the second isomorphism~(\ref{eq:localdescription-of-OBla-defineG-2}),
  \begin{equation*}
    \left(\left(\id_{R_{i}}\widehat{\otimes}\vartheta_{\dR}^{\widetilde{q}_{i}}\right)\left(
    \left(\id_{R_{i}}\widehat{\otimes}\delta_{i}\right)\left(\ker\mu_{i}\right)
    \right)
    \right)
    =
    \left( \id_{R_{i}}\widehat{\otimes}\iota_{i}\right)
    \left( \ker\mu_{i}\right),
  \end{equation*}
  by the computation~\ref{eq:localdescription-of-OBla-defineG} above.
  We may thus write~(\ref{eq:localdescription-of-OBla-defineG-2}) as
  \begin{equation}\label{eq:localdescription-of-OBla-defineG-3}
      \left.\left(R_{i}\widehat{\otimes}R^{+}\right)
      \middle/
      \left(\left( \id_{R_{i}}\widehat{\otimes}\iota_{i}\right)
    \left( \ker\mu_{i}
      \right)
      \right)
      \right.
      \stackrel{\cong}{\longrightarrow}
      R.
  \end{equation}
  We claim that the image of
  $\ker\mu_{i}\subseteq A:=R_{i}\widehat{\otimes}R_{i}$ in
  $B:=R_{i}\widehat{\otimes}R$ generates the kernel of
  the multiplication map as an ideal.
  Apply~\cite[Theorem 4.1.4]{EquivariantD2021}
  to get an exact sequence of finitely presented
  $A$-modules
  \begin{equation*}
    A^{n}
    \longrightarrow
    A
   \stackrel{\mu_{i}}{\longrightarrow}
   R_{i}
   \longrightarrow 0.
  \end{equation*}
  Apply the functor $-\widehat{\otimes}_{A}B$
  to get an exact sequence of finitely presented $B$-modules
  \begin{equation*}
    B^{n}
    \longrightarrow
    B
   \stackrel{\mu_{i}\widehat{\otimes}_{A}\id_{B}}{\longrightarrow}
   R_{i}\widehat{\otimes}_{A}B
   \longrightarrow 0.
  \end{equation*}
  The functors
  $-\widehat{\otimes}_{R_{i}}R$ and
  $-\widehat{\otimes}_{A}B$ are isomorphic on finitely
  generated $A$-modules. This is because both are suitably
  right exact, and both send $A$ to $B$, up to natural isomorphism.
  Hence we get the exact sequence
  \begin{equation*}
    B^{n}
    \longrightarrow
    B
   \stackrel{\upsilon}{\longrightarrow}
   R_{i}\widehat{\otimes}_{R_{i}}R
   \cong R
   \longrightarrow 0.
  \end{equation*}  
  One the other hand, $\mu\colon B\to R$
  kills the image of $A^{n}$ in $B$, giving a complex
  \begin{equation*}
    B^{n} \longrightarrow B \stackrel{\mu}{\longrightarrow} R \longrightarrow 0.
  \end{equation*}
  One checks that this complex is equal to the previous
  exact sequence: the only non-trivial part is to show that $\mu$
  and $\upsilon$ coincide. But $\upsilon$ is the composition
  \begin{equation*}
  \begin{tikzcd}[row sep=tiny]
    R_{i}\widehat{\otimes}R = B \arrow{r}{\cong}
    & A\widehat{\otimes}_{A}B \arrow{r}{\mu_{i}\widehat{\otimes}_{A}B}
    & R_{i}\widehat{\otimes}_{A}B \arrow{r}{\cong}
    & R_{i} \\
    r\widehat{\otimes}s \arrow[maps to]{r}
    & \left(r\widehat{\otimes}1\right)\widehat{\otimes}\left(1\widehat{\otimes}s\right) \arrow[maps to]{r}
    & r\widehat{\otimes}\left(1\widehat{\otimes}s\right) \arrow[maps to]{r}
    & rs,
  \end{tikzcd}
  \end{equation*}
  which shows the claim.
  It follows that the image of
  $\ker\mu_{i}\subseteq A$ in
  $B$ generates the multiplication map. This gives~(\ref{eq:localdescription-of-OBla-defineG-3}).
 \end{proof}

\begin{lem}\label{lem:coordinates-germ-diagonal}
  Let $A$ denote a $k_{0}$-affinoid algebra and fix a finite set of elements
  $s_{1},\dots,s_{n}\in A$. Set $I:=\left(s_{1},\dots,s_{n}\right)$ and consider
  \begin{equation*}
    \omega^{q}\colon A\left\<\frac{s_{1},\dots,s_{n}}{p^{q}}\right\>
    \to\varprojlim_{l}A/I^{l}
  \end{equation*}
  for all $q\in\NN$. Then the $\ker\omega^{q}$ are Banach spaces and
  $\varinjlim_{q}\ker\omega^{q}=0$ as a $k_{0}$-ind-Banach space.
\end{lem}

\begin{proof}
  Write
  %\begin{equation*}
    $I_{q}:=A\left\< \frac{s_{1},\dots,s_{n}}{\pi^{q}} \right\>I$.
  %\end{equation*}
  Then
  \begin{equation*}
    A\left\< \frac{s_{1},\dots,s_{n}}{\pi^{q}} \right\>/I_{q}^{l}
    \stackrel{\cong}{\longrightarrow} A/I^{l}
  \end{equation*}
  for all $l$, thus the kernel of $\omega^{q}$ is the intersection
  of the ideals $I_{q}^{l}$. But any ideal in an affinoid algebra is
  closed, cf.~\cite[\S 6.1.1 Proposition 3]{BGR84}. Therefore
  $\ker\omega^{q}$ is closed, thus complete. In particular, we may
  view it as a Banach space with the restriction of the norm
  on $A\left\< \frac{s_{1},\dots,s_{n}}{\pi^{q}} \right\>$.
  
  By the Krull intersection theorem, there exists an element
  $f\in I_{q}$ such that
  \begin{equation*}
    (1-f)\ker\omega_{m} = (1-f)\bigcap_{l}I_{q}^{l} = 0.
  \end{equation*}
  Now consider the commutative diagram
  \begin{equation*}
    \begin{tikzcd}
      A\left\< \frac{\zeta_{1},\dots,\zeta_{n}}{\pi^{q}} \right\>
      \arrow{r}{\upsilon^{q}}\arrow{d}{\widetilde{\iota}_{q,q^{\prime}}} &
      A\left\< \frac{s_{1},\dots,s_{n}}{\pi^{q}} \right\>\arrow{d}{\iota_{q,q^{\prime}}} \\
      A\left\< \frac{\zeta_{1},\dots,\zeta_{n}}{\pi^{q^{\prime}}} \right\>
      \arrow{r}{\upsilon^{q^{\prime}}} &
      A\left\< \frac{s_{1},\dots,s_{n}}{\pi^{q^{\prime}}} \right\>  
    \end{tikzcd}
  \end{equation*}
  for any $q^{\prime}\geq q$,
  where the horizontal maps send each $\zeta_{i}$ to $s_{i}$.
  Pick $\widetilde{f}$ such that $\upsilon^{q}\left( \widetilde{f} \right)=f$.
  Since $f\in I_{q}$, we may assume that
  $\widetilde{f}\in\left( \zeta_{1},\dots,\zeta_{n}\right)$.
  In particular,
  \begin{equation*}
    \|\iota_{q,q^{\prime}}(f)\|
    \leq\|\tau^{q^{\prime}}\left( \widetilde{\iota}_{q,q^{\prime}}\left( \widetilde{f} \right)\right)\|
    \leq\|\widetilde{\iota}_{q,q^{\prime}}\left( \widetilde{f} \right)\|
    \leq|\pi|^{q^{\prime}-q}\|\widetilde{f}\|.
  \end{equation*}
  That is, we may chose $q^{\prime}$ large enough such that
  $\|\iota_{q,q^{\prime}}(f)\|\leq 1$. But then
  $\iota_{q,q^{\prime}}\left(1-f\right)$ becomes a unit: its inverse is
  $\sum_{\alpha\geq0}\iota_{q,q^{\prime}}\left(f\right)^{\alpha}$.
  Thus $\iota_{q,q^{\prime}}\left( \ker\omega^{q} \right)$
  is killed by a unit, thus it is zero. In other words,
  %\begin{equation*}
    $\ker\omega^{q}
    \subseteq\ker\iota^{q,q^{\prime}}$.
  %\end{equation*}
  It follows that
  \begin{equation*}
    \varinjlim_{q\geq0}\ker\omega^{q}
    \subseteq\varinjlim_{q^{\prime}\geq q\geq0}\ker\iota_{q,q^{\prime}}
    =\ker\varinjlim_{q^{\prime}\geq q\geq0}\iota_{q,q^{\prime}}
    =0.
  \end{equation*}
  But $\iota:=\varinjlim_{m^{\prime}\geq m\geq0}\iota_{q,q^{\prime}}$ is
  by definition an automorphism of
  $\varinjlim_{m\geq0}A\left\< \frac{s_{1},\dots,s_{n}}{\pi^{q}} \right\>$,
  thus $\ker\iota=0$. This implies $\varinjlim_{q\geq0}\ker\omega^{q}=0$,
  as desired.
\end{proof}

\begin{prop}\label{prop:coordinates-germ-diagonal}
  Fix generators $\left( s_{1},\dots,s_{n}\right)=\ker\mu_{i}$.
  We have an isomorphism 
  \begin{equation*}
    \tau_{i}\colon R_{i}\left\< \frac{Z_{1},\dots,Z_{d}}{p^{\infty}}\right\>
    \stackrel{\cong}{\longrightarrow}
    \left(R_{i}\widehat{\otimes}_{k_{0}}R_{i}\right)
    \left\< \frac{s_{1},\dots,s_{n}}{p^{\infty}}\right\>
   \end{equation*}
   extending $r\mapsto 1\widehat{\otimes}r$ and sending
   $Z_{l}$ to $T_{l}\widehat{\otimes}1-1\widehat{\otimes}T_{l}$
   of $k$-ind-Banach algebras.
\end{prop}

\begin{proof}
  By~\cite[Lemma 2.3]{KisinLocalconstancy99} the canonical morphism
  \begin{equation*}
  \left(R_{i}\widehat{\otimes}_{k_{0}}R_{i}\right)
    \left\< \frac{s_{1},\dots,s_{n}}{p^{\infty}}\right\>
    \isomap\left(R_{i}\widehat{\otimes}_{k_{0}}R_{i}\right)_{U}
  \end{equation*}
  is an isomorphism of $k$-ind-Banach algebras.
  Here, the colimit on the left runs through the system of affinoid open
  neighbourhoods $U\supseteq\Delta\left(\Sp\left(R_{i}\right)\right)$
  and $( - )_{U}$ denotes the corresponding localisation.
  In particular, the constructions in the~\cite[proof of Proposition 6.31]{soor2024sixfunctorformalismquasicoherentsheaves}
  apply. \emph{Loc. cit.} constructs
  \footnote{\cite{soor2024sixfunctorformalismquasicoherentsheaves}
    denotes $\sigma_{i}$ by $\psi$ and $\tau_{i}$ by $\varphi$.}
  \begin{equation*}
      \sigma_{i}\colon
      \left(R_{i}\widehat{\otimes}_{k_{0}}R_{i}\right)
      \left\< \frac{s_{1},\dots,s_{n}}{p^{\infty}}\right\>
      \longrightarrow
      R_{i}\left\< \frac{Z_{1},\dots,Z_{d}}{p^{\infty}}\right\>,
  \end{equation*}
  such that $\sigma_{i}\circ\tau_{i}$ is the identity.
  ~\cite[Proposition 1.1.8]{Sch99} implies that $\sigma_{i}$ is a strict epimorphism,
  thus it remains to check that it is a monomorphism. To do this,
  we apply Lemma~\ref{lem:coordinates-germ-diagonal} for
  $A:=R_{i}\widehat{\otimes}_{k_{0}}R_{i}$ and $I:=\ker\mu_{i}$.
  \emph{Loc. cit.} considers the morphism $\omega_{i}$, which fit into the
  commutative diagrams
  \begin{equation}\label{cd:coordinates-germ-diagonal}
  \begin{tikzcd}
    \left(R_{i}\widehat{\otimes}_{k_{0}}R_{i}\right)
    \left\< \frac{s_{1},\dots,s_{n}}{p^{q}}\right\>
    \arrow{r}{\tau_{i}^{q}}
    \arrow{d}{\omega_{i}^{q}} &
    R_{i}\left\< \frac{Z_{1},\dots,Z_{d}}{p^{\infty}}\right\>
    \arrow[hook]{d} \\
    \varprojlim_{j}\left(R_{i}\widehat{\otimes}_{k_{0}}R_{i}\right)/\left(\ker\mu_{i}\right)^{j}
    \arrow{r}{\cong} &
    R_{i}\left\llbracket Z_{1},\dots,Z_{d} \right\rrbracket.
  \end{tikzcd}
  \end{equation}  
  The morphisms $\tau_{i}^{q}$ at the top are the compositions
  \begin{equation*}
    \left(R_{i}\widehat{\otimes}_{k_{0}}R_{i}\right)
    \left\< \frac{s_{1},\dots,s_{n}}{p^{q}}\right\>
    \longrightarrow
    \left(R_{i}\widehat{\otimes}_{k_{0}}R_{i}\right)
    \left\< \frac{s_{1},\dots,s_{n}}{p^{\infty}}\right\>
    \stackrel{\sigma_{i}}{\longrightarrow}
    R_{i}\left\< \frac{Z_{1},\dots,Z_{d}}{p^{\infty}}\right\>,
  \end{equation*}
  and the isomorphism at the bottom of the commutative diagram
  comes from~\cite[Lemma B.7 and Corollary B.10]{soor2024sixfunctorformalismquasicoherentsheaves}.
  The commutative diagram~(\ref{cd:coordinates-germ-diagonal})
  implies $\ker\tau_{i}^{q}=\ker\omega_{i}^{q}$. Together with the
  aforementioned Lemma~\ref{lem:coordinates-germ-diagonal}, we find
  \begin{equation*}
    \ker\tau_{i}
    =\varinjlim_{q}\ker\tau_{i}
    =\varinjlim_{q}\ker\omega_{i}^{q}
    =0,
  \end{equation*}
  as desired.
\end{proof}

\begin{remark}\label{rem:discreteval-prop:coordinates-germ-diagonal}
  The proof of Proposition~\ref{prop:coordinates-germ-diagonal}
  requires $k$ to be discretely valued. Indeed, it follows that
  $k$ is finite over $k_{0}$, thus
  $R_{i}\widehat{\otimes}_{k_{0}}R_{i}$ is an affinoid $k$-algebra
  and~\cite[Lemma 2.3]{KisinLocalconstancy99} applies.
\end{remark}

\begin{proof}[Proof of Theorem~\ref{thm:localdescription-of-OBla}]
  As explained at the beginning of \S\ref{subsec:proof-localdescrption-of-OBla},
  it suffices to check that the morphism~(\ref{eq:localdescription-of-OBla-phii})
  is an isomorphism.
  Recall that $\BB_{\dR}^{\dag,+}\left(R,R^{+}\right)$ is an $R_{i}$-Banach algebra
  via $\delta_{i}$, cf. Notation~\ref{notation:Riplus-algebra-Ainf}.
  See Proposition~\ref{prop:coordinates-germ-diagonal} for $\tau_{i}$ and write
  $\psi_{i}:=\id_{\BB_{\dR}^{\dag,+}\left(R,R^{+}\right)}\widehat{\otimes}_{R_{i}}\tau_{i}$.
  We claim
  \begin{equation*}
    \phi_{i}\cong\psi_{i}.
  \end{equation*}
  Because $\psi_{i}$ is an isomorphism, this would imply that
  $\phi_{i}$ is an isomorphism.
  
  \emph{Step 1.} The domain of $\psi_{i}$ coincides with the domain of $\phi_{i}$:
  \begin{align*}
    \BB_{\dR}^{\dag,+}\left( R,R^{+} \right)
    \widehat{\otimes}_{R_{i}}
    R_{i}\left\< \frac{Z_{1},\dots,Z_{d}}{p^{\infty}}\right\>
    \stackrel{\text{\ref{lem:restrictedpowerseries-from-hotimes}}}{\cong}
    \BB_{\dR}^{\dag,+}\left( R,R^{+} \right)
    \left\< \frac{Z_{1},\dots,Z_{d}}{p^{\infty}}\right\>.
  \end{align*}
   
  \emph{Step 2.} We compute the codomain of $\psi_{i}$.
  Recall the generating set $G$ of $\ker\Ovartheta_{\dR}^{\widetilde{q}_{i}}$
  considered in Lemma~\ref{lem:localdescription-of-OBla-defineG}.
  Additionally, fix a finite generating set $S=\left\{s_{1},\dots,s_{n}\right\}$ of $\ker\mu_{i}$. This implies
  \begin{equation*}
    G:\left(\id_{R_{i}}\widehat{\otimes}_{k_{0}}\delta_{i}\right)\left(s_{1},\dots,s_{n}\right)
    \cup\left\{1\widehat{\otimes}\xi\right\}
    \subseteq
    R_{i}\widehat{\otimes}_{k_{0}}\widehat{\BB}_{\inf}\left( R,R^{+} \right)
  \end{equation*}
  and compute, for every $\widetilde{q}_{i}\in\NN$ as in Lemma~\ref{lem:imagetildenu-Ainf},
  \begin{align*}
    &\BB_{\dR}^{q,+}\left( R,R^{+} \right)
    \widehat{\otimes}_{R_{i}}
    \left(R_{i}\widehat{\otimes}_{k_{0}}R_{i}\right)
        \left\< \frac{S}{p^{q}}\right\> \\
    &\stackrel{\text{\ref{lem:Blaqplus-local-complete-at-xi}}}{\cong}
    \coker\left( \widehat{\BB}_{\inf}\left( R,R^{+} \right)\left\< \frac{\zeta}{p^{q}} \right\>\left(p^{q}\frac{\zeta}{p^{q}}-\xi\right)
    \to\widehat{\BB}_{\inf}\left( R,R^{+} \right)\left\< \frac{\zeta}{p^{q}} \right\>\right) \\
    &\qquad\widehat{\otimes}_{R_{i}}
    \coker\left(
    \begin{aligned}
      \bigoplus_{j=1}^{n}\left(R_{i}\widehat{\otimes}_{k_{0}}R_{i}\right)
      \left\<\frac{\zeta_{s}}{p^{q}}\colon s\in S\right\>
      \left(p^{q}\frac{\zeta_{s_{j}}}{p^{q}}-s_{j}\right) \\
      \to
      \left(R_{i}\widehat{\otimes}_{k_{0}}R_{i}\right)
      \left\<\frac{\zeta_{s}}{p^{q}}\colon s\in S\right\>
    \end{aligned}
    \right) \\
    &\cong\coker\left(
    \begin{aligned}
      \bigoplus_{g\in G}\left(R_{i}\widehat{\otimes}_{k_{0}}\widehat{\BB}_{\inf}\left( R,R^{+} \right)\right)
      \left\<\frac{\zeta_{g}}{p^{q}}\colon g\in G\right\>
      \left(p^{q}\frac{\zeta_{g}}{p^{q}}-g\right) \\
      \to
      \left(R_{i}\widehat{\otimes}_{k_{0}}\widehat{\BB}_{\inf}\left( R,R^{+} \right)\right)
      \left\<\frac{\zeta_{g}}{p^{q}}\colon g\in G\right\>
   \end{aligned}
    \right) \\
    &= \left(R^{+}\widehat{\otimes}_{k_{0}}\widehat{\BB}_{\inf}\left( R,R^{+} \right)\right)\left\<\frac{G}{p^{\widetilde{q}_{i}}}\right\> \\
    &= \left(R_{i}^{+}\widehat{\otimes}_{W(\kappa)}\A_{\inf}\left( R,R^{+} \right)\right)
    \left\< \frac{\ker\Otheta_{\inf}}{p^{\widetilde{q}_{i}}}\right\>\widehat{\otimes}_{W(\kappa)}k_{0}
  \end{align*}
  is an isomorphism.
  Now pass to the colimit along $\widetilde{q}_{i}\to\infty$ to see that $\phi_{i}$ and $\psi_{i}$
  have the same codomain:
  \begin{align*}
    &\BB_{\dR}^{\dag,+}\left( R,R^{+} \right)
    \widehat{\otimes}_{R_{i}}
    \left(R_{i}\widehat{\otimes}_{k_{0}}R_{i}\right)
        \left\< \frac{S}{p^{\infty}}\right\> \\
    &\cong
    \left(R_{i}^{+}\widehat{\otimes}_{W(\kappa)}\A_{\inf}\left(R,R^{+}\right)\right)
    \left\<\frac{\ker\Otheta_{\inf}}{p^{\infty}}\right\>\widehat{\otimes}_{W(\kappa)}k_{0}.
  \end{align*}  

  \emph{Step 3.} Both $\phi_{i}$ and $\psi_{i}$ are colimits of
  completed localisations of morphisms
  \begin{equation*}
    \BB_{\dR}^{\widetilde{q}_{i}}\left( R,R^{+} \right)
    \left\< \frac{Z_{1},\dots,Z_{d}}{p^{\widetilde{q}_{i}}} \right\>
    \to\left(R_{i}^{+}\widehat{\otimes}_{W(\kappa)}\A_{\inf}(U)\right)
    \left\<\frac{\ker\Otheta_{\inf}}{p^{\widetilde{q}_{i}}}\right\>\widehat{\otimes}_{W(\kappa)}k_{0}
  \end{equation*}
  of $\BB_{\dR}^{\widetilde{q}_{i}}\left( R,R^{+} \right)$-Banach algebras.
  Therefore, it suffices to check that both $\phi_{i}$
  and $\psi_{i}$ coincide on the variables $Z_{1},\dots,Z_{d}$.
  This is an easy computation:
  \begin{align*}
    \phi_{i}\left(Z_{l}\right)
    %&=u_{l}, \text{ and}\\
    &=T_{l}\widehat{\otimes}_{W(\kappa)}1-1\widehat{\otimes}_{W(\kappa)}\left[T_{l}^{\flat}\right], \text{ and} \\
    \psi_{i}\left(Z_{l}\right)
    &=\left( \id_{R_{i}^{+}}\widehat{\otimes}_{W(\kappa)}\delta_{i}\right)
    \left(\tau_{i}\left(Z_{l}\right)\right) \\
    &=\left( \id_{R_{i}^{+}}\widehat{\otimes}_{W(\kappa)}\delta_{i}\right)
    \left(T_{l}\widehat{\otimes}_{W(\kappa)}1-1\widehat{\otimes}_{W(\kappa)}T_{l}\right) \\
    &=T_{l}\widehat{\otimes}_{W(\kappa)}1-1\widehat{\otimes}_{W(\kappa)}\left[T_{l}^{\flat}\right]% \\
%    &=u_{l}
  \end{align*}
  for all $l=1,\dots,d$.% See~(\ref{eq:defn-ul}) for the definition of $u_{l}$.
\end{proof}

%%%%%%%%%%%%%%%%%%%%%%%%%%%%%%%%%%%%%%%%%%%%%%%%%%%%%%%%%%%%
%%%%%%%%%%%%%%%%%%%%%%%%%%%%%%%%%%%%%%%%%%%%%%%%%%%%%%%%%%%%
% Proof of local description of OBla
%%%%%%%%%%%%%%%%%%%%%%%%%%%%%%%%%%%%%%%%%%%%%%%%%%%%%%%%%%%%
%%%%%%%%%%%%%%%%%%%%%%%%%%%%%%%%%%%%%%%%%%%%%%%%%%%%%%%%%%%%

\subsection{Proof of Theorem~\ref{thm:localdescription-of-OBqplus}}
\label{subsec:proof-localdescrption-of-OBla-finitelevel}

Fix the notation as in \S\ref{subsec:proof-localdescrption-of-OBla}
as well as $q\in\NN_{\geq2}$.
Similarly to~(\ref{eq:localdescription-of-OBla-phii}), we have the morphism
\begin{equation*}\label{eq:localdescription-of-OBla-phiiq}
  \phi_{i}^{q}\colon
  \BB_{\dR}^{q,+}\left( R,R^{+} \right)\left\<\frac{Z_{1},\dots,Z_{d}}{p^{q}}\right\>
  \to
  \left(
  %\mathcal{O}^{+}(U_{i})
  R_{i}^{+}\widehat{\otimes}_{W(\kappa)}\A_{\inf}\left( R,R^{+} \right)\right)
  \left\<\frac{\ker\Otheta_{\inf}}{p^{q}}\right\>\widehat{\otimes}_{k^{\circ}}k
\end{equation*}
of $\BB_{\dR}^{q,+}\left( R,R^{+} \right)$-Banach algebras.
Using Lemma~\ref{lem:imagetildenu-Ainf} and~\ref{lem:Banach-completions-extend}, we construct
a unique morphism
\begin{equation*}\label{eq:localdescription-of-OBla-phiiq}
  \widetilde{\chi}_{i}^{q}
  \colon
  \left(
  R_{i}^{+}\widehat{\otimes}_{W(\kappa)}\A_{\inf}\left( R,R^{+} \right)\right)
  \left\<\frac{\ker\Otheta_{\inf}}{p^{q}}\right\>\widehat{\otimes}_{k^{\circ}}k
  \to
  \BB_{\dR}^{q,+}\left( R,R^{+} \right)\left\<\frac{Z_{1},\dots,Z_{d}}{p^{q}}\right\>
\end{equation*}
of $\BB_{\dR}^{q,+}\left( R,R^{+} \right)$-Banach algebras
for $q\gg0$. (In the notation of Lemma~\ref{lem:imagetildenu-Ainf},
we choose $q\geq\widetilde{q}_{i}$.)

Next, $\normalOB_{i}$ is the coimage of
\begin{equation*}
  \left(
  %\mathcal{O}^{+}(U_{i})
  R_{i}^{+}\widehat{\otimes}_{W(\kappa)}\A_{\inf}\left( R,R^{+} \right)\right)
  \left\<\frac{\ker\Otheta_{\inf}}{p^{q}}\right\>\widehat{\otimes}_{k^{\circ}}k 
  \to
  \left(
  %\mathcal{O}^{+}(U_{i})
  R_{i}^{+}\widehat{\otimes}_{W(\kappa)}\A_{\inf}\left( R,R^{+} \right)\right)
  \left\<\frac{\ker\Otheta_{\inf}}{p^{\infty}}\right\>\widehat{\otimes}_{k^{\circ}}k.
\end{equation*}

\begin{lem}
  $\widetilde{\chi}_{i}^{q}$ factors canonically through the morphism
  of $\BB_{\dR}^{q}\left(R,R^{+}\right)$-Banach algebras
  \begin{equation*}
    \chi_{i}^{q}\colon
    \normalOB_{i}^{q}
    \to
    \BB_{\dR}^{q,+}\left( R,R^{+} \right)\left\<\frac{Z_{1},\dots,Z_{d}}{p^{q}}\right\>.
  \end{equation*}
\end{lem}

\begin{proof}
  This follows from the commutative diagram
  \begin{equation*}
  \begin{tikzcd}
    \left(
    R_{i}^{+}\widehat{\otimes}_{W(\kappa)}\A_{\inf}\left( R,R^{+} \right)\right)
    \left\<\frac{\ker\Otheta_{\inf}}{p^{q}}\right\>\widehat{\otimes}_{k^{\circ}}k
    \arrow{d}\arrow{r} &
    \BB_{\dR}^{q,+}\left( R,R^{+} \right)\left\<\frac{Z_{1},\dots,Z_{d}}{p^{q}}\right\>.
    \arrow[hook]{d} \\    
    \left(
    R_{i}^{+}\widehat{\otimes}_{W(\kappa)}\A_{\inf}\left( R,R^{+} \right)\right)
    \left\<\frac{\ker\Otheta_{\inf}}{p^{\infty}}\right\>\widehat{\otimes}_{k^{\circ}}k
    \arrow{r} &
    \BB_{\dR}^{\infty,+}\left( R,R^{+} \right)\left\<\frac{Z_{1},\dots,Z_{d}}{p^{\infty}}\right\>.
  \end{tikzcd}
  \end{equation*}
  Here, the morphism at the right-hand side
  is a monomorphism by Proposition~\ref{prop:BdRdagplus-bornology-countable-basis-reconstructionpaper}.
\end{proof}

\begin{lem}\label{lem:2sidedinverse1-thm:localdescription-of-OBqplus}
  $\chi_{i}^{q}\circ\phi_{i}^{q}$ is the identity.
  %on $\left(
  %\mathcal{O}^{+}(U_{i})
  %R_{i}^{+}\widehat{\otimes}_{W(\kappa)}\A_{\inf}\left( R,R^{+} \right)\right)
  %\left\<\frac{\ker\Otheta_{\inf}}{p^{q}}\right\>\widehat{\otimes}_{k^{\circ}}k$.
\end{lem}

\begin{proof}
  We may check that $\widetilde{\chi}_{i}^{q}\circ\phi_{i}^{q}$ is the identity.
  Both $\widetilde{\chi}_{i}^{q}$ and $\phi_{i}^{q}$ and are morphisms of
  of $\BB_{\dR}^{q}\left( R,R^{+} \right)$-Banach algebras. Therefore,
  it suffices to compute
  \begin{equation*}
    \left(\widetilde{\chi}_{i}^{q}\circ\phi_{i}^{q}\right)\left( Z_{l} \right)
    =\widetilde{\chi}_{i}^{q}\left(T_{l}\widehat{\otimes}1 - 1\widehat{\otimes}\left[ T_{l}^{\flat} \right]\right)
    \stackrel{\text{\ref{lem:imagetildenu-Ainf}}}{=}
      \left(\left[T_{l}^{\flat}\right]+Z_{l}\right) - \left[ T_{l}^{\flat} \right]
    =Z_{l}
  \end{equation*}
  for all $l=1,\dots,d$.
\end{proof}

\begin{lem}\label{lem:2sidedinverse2-thm:localdescription-of-OBqplus}
  $\phi_{i}^{q}\circ\chi_{i}^{q}$ is the identity.
\end{lem}

\begin{proof}
  Using Lemma~\ref{lem:2sidedinverse1-thm:localdescription-of-OBqplus},
  we see that
  $\chi_{i}^{q}\circ\left(\phi_{i}^{q}\circ\chi_{i}^{q}-\id\right)
    \left(\chi_{i}^{q}\circ\phi_{i}^{q}-\id\right)\circ\chi_{i}^{q}=0$.
  Therefore, it suffices to check that $\chi_{i}^{q}$ is a monomorphism.
  This follows from the commutative diagram
  \begin{equation*}
  \begin{tikzcd}
    \normalOB_{i}^{q}
    \arrow{d}\arrow{r}{\chi_{i}^{q}} &
    \BB_{\dR}^{q,+}\left( R,R^{+} \right)\left\<\frac{Z_{1},\dots,Z_{d}}{p^{q}}\right\>.
    \arrow[hook]{d} \\    
    \normalOB_{i}^{\infty}
    \arrow{r}{\chi_{i}^{\infty}} &
    \BB_{\dR}^{\infty,+}\left( R,R^{+} \right)\left\<\frac{Z_{1},\dots,Z_{d}}{p^{\infty}}\right\>.
  \end{tikzcd}
  \end{equation*}
  Here, the morphism at the right-hand side
  is a monomorphism by Proposition~\ref{prop:BdRdagplus-bornology-countable-basis-reconstructionpaper}.
\end{proof}

\begin{proof}[Proof of Theorem~\ref{thm:localdescription-of-OBqplus}]
  Fix $q\gg0$ as in the previous discussion.
  
  Compute the following coimage in the category of presheaves:
  \begin{equation*}
    \OB_{\dR}^{q,+,\psh}:=\coim\left( \widetilde{\OB}_{\dR}^{q,+,\psh}\to\widetilde{\OB}_{\dR}^{\infty,+,\psh} \right).
  \end{equation*}
  As sheafification is strongly exact, cf.~\cite[Lemma 3.15]{Bo21}, we may as well show that
  the morphism of presheaves of $\BB_{\dR}^{q,+}|_{\widetilde{X}}$-ind-Banach algebras
  \begin{equation*}
    \BB_{\dR}^{q,+}|_{\widetilde{X}}\left\<\frac{Z_{1},\dots,Z_{d}}{p^{q}}\right\>
    \stackrel{\cong}{\longrightarrow}\OB_{\dR}^{q,+,\psh}|_{\widetilde{X}},
    Z_{l}\mapsto u_{l}
  \end{equation*}
  is an isomorphism of presheaves on $X_{\proet,\affperfd}^{\fin}/\widetilde{X}$,
  cf. Definition~\ref{defn:XproetaffperfdU}. Consequently, fixing the notation from the previous discussion,
  we may as well check that
  \begin{equation*}
    \phi_{i}^{q}\colon\BB_{\dR}^{q,+}\left(R,R^{+}\right)\left\<\frac{Z_{1},\dots,Z_{d}}{p^{q}}\right\>
    \to\normalOB_{i}^{q}
  \end{equation*}
  is an isomorphism of $\BB_{\dR}^{q,+}\left(R,R^{+}\right)$-Banach algebras.
  This follows from Lemma~\ref{lem:2sidedinverse1-thm:localdescription-of-OBqplus}
  and~\ref{lem:2sidedinverse2-thm:localdescription-of-OBqplus}.
\end{proof}

%%%%%%%%%%%%%%%%%%%%%%%%%%%%%%%%%%%%%%%%%%%%%%%%%%%%%%%%%%%%
%%%%%%%%%%%%%%%%%%%%%%%%%%%%%%%%%%%%%%%%%%%%%%%%%%%%%%%%%%%%
%%%%%%%%%%%%%%%%%%%%%%%%%%%%%%%%%%%%%%%%%%%%%%%%%%%%%%%%%%%%
% Pushforward of OBla
%%%%%%%%%%%%%%%%%%%%%%%%%%%%%%%%%%%%%%%%%%%%%%%%%%%%%%%%%%%%
%%%%%%%%%%%%%%%%%%%%%%%%%%%%%%%%%%%%%%%%%%%%%%%%%%%%%%%%%%%%
%%%%%%%%%%%%%%%%%%%%%%%%%%%%%%%%%%%%%%%%%%%%%%%%%%%%%%%%%%%%

\chapter{Continuous Galois cohomology of overconvergent period rings}
\label{ch:Galoiscoh-recpaper}

In \S\ref{ch:Galoiscoh-recpaper}, we compute the continuous Galois cohomology of the
period rings $B_{\dR}^{\dag,+}:=\BB_{\dR}^{\dag,+}\left(C,\mathcal{O}_{C}\right)$,
$B_{\dR}^{\dag}:=\BB_{\dR}^{\dag}\left(C,\mathcal{O}_{C}\right)$,
$B_{\pdR}^{\dag,+}:=\BB_{\pdR}^{\dag,+}\left(C,\mathcal{O}_{C}\right)$, and
$B_{\pdR}^{\dag}:=\BB_{\pdR}^{\dag}\left(C,\mathcal{O}_{C}\right)$.
Our proofs also yield explicit descriptions of the Galois cohomology of
the solidifications $\underline{B}_{\dR}^{\dag,+}$, $\underline{B}_{\dR}^{\dag}$, $\underline{B}_{\pdR}^{\dag,+}$,
and $\underline{B}_{\pdR}^{\dag}$.

%%%%%%%%%%%%%%%%%%%%%%%%%%%%%%%%%%%%%%%%%%%%%%%%%%%%%%%%%%%%
%%%%%%%%%%%%%%%%%%%%%%%%%%%%%%%%%%%%%%%%%%%%%%%%%%%%%%%%%%%%
%%%%%%%%%%%%%%%%%%%%%%%%%%%%%%%%%%%%%%%%%%%%%%%%%%%%%%%%%%%%
% Pushforward of OBla
%%%%%%%%%%%%%%%%%%%%%%%%%%%%%%%%%%%%%%%%%%%%%%%%%%%%%%%%%%%%
%%%%%%%%%%%%%%%%%%%%%%%%%%%%%%%%%%%%%%%%%%%%%%%%%%%%%%%%%%%%
%%%%%%%%%%%%%%%%%%%%%%%%%%%%%%%%%%%%%%%%%%%%%%%%%%%%%%%%%%%%

\section{A variant of the $\eta$-operator}
\label{subsec:the-eta-operator-variants}

Fix a commutative ring $R$ and a non zero-divisor $r\in R$.
The $\eta$-operator kills $r$-torsion in the cohomology of a given complex of $R$-modules.
Lemma~\ref{lem:killtorsionincontrolledwayincplx} gives a construction which annihilates
$r$-torsion in a single degree.

  \begin{lem}\label{lem:killtorsionincontrolledwayincplx}
    Fix $n\in\NN$. Given the cochain complex $C^{\bullet}$ of $R$-modules
    without $r$-torsion
    \begin{equation*}
      C^{0}
      \stackrel{d^{0}}{\longrightarrow}
      C^{1}
      \stackrel{d^{1}}{\longrightarrow}
      C^{2}
      \stackrel{d^{2}}{\longrightarrow}
      \dots, 
    \end{equation*}
    we consider the cochain complex $D^{\bullet}$ as follows:
    \begin{equation*}
      C^{0} \times_{d^{0} , C^{1}} rC^{1}
      \stackrel{e^{0}}{\longrightarrow}
      rC^{1}
      \stackrel{e^{1}}{\longrightarrow}
      rC^{2}
      \stackrel{e^{2}}{\longrightarrow}
      \dots. 
    \end{equation*}
    Here, $e^{0}$ is $(x,y)\mapsto y$ and $e^{i}:=d^{i}|_{rC^{i}}$ for all $i\geq1$. Then we compute, for all $i\in\NN$,
    \begin{equation*}
      \Ho^{i}\left(D^{\bullet}\right)
      \cong
      \begin{cases}
        \Ho^{0}\left(C^{\bullet}\right) & \text{if $i=0$,} \\
        \Ho^{1}\left(C^{\bullet}\right) / \text{($r$-torsion)} & \text{if $i=1$,} \\
        r\Ho^{i}\left(C^{\bullet}\right) & \text{if $i\geq2$}.
      \end{cases}
    \end{equation*}
  \end{lem}
  
  \begin{proof}
    There is a canonical morphism $D^{\bullet}\to C^{\bullet}$,
    given by $(x,y)\mapsto x$ in degree zero. It induces an injective map
    $\Ho^{0}\left(D^{\bullet}\right)\to\Ho^{0}\left(C^{\bullet}\right)$.
    On the other hand, consider an element $c$ in its codomain,
    that is the kernel of $d^{0}$. Then it has a preimage
    $(d,0)\in\ker e^{0}=\Ho^{0}\left(D^{\bullet}\right)$,
    showing that $\Ho^{0}\left(\phi^{\bullet}\right)$ is surjective.
    $\Ho^{0}\left(D^{\bullet}\right)\cong\Ho^{0}\left(C^{\bullet}\right)$ follows.
     
    Next, consider the multiplication-by-$r$-map
    $\psi\colon\Ho^{1}\left(C^{\bullet}\right)\to\Ho^{1}\left(D^{\bullet}\right)$.
    It is surjective: given $[rc]\in\Ho^{1}\left(D^{\bullet}\right)$,
    $r$-torsion freeness implies $c\in\ker d^{1}$, and we find
    $[c]\in\Ho^{1}\left(C^{\bullet}\right)$. On the other hand,
    $[c]\in\ker\psi$ if and only there exists $\widetilde{c}\in C^{0}$ such that
    $d^{0}\left(\widetilde{c}\right)=rc$. That is, $[c]\in\ker\psi$
    precisely if $[c]$ is $r$-torsion.
    $\Ho^{1}\left(C^{\bullet}\right)/\text{($r$-torsion)}\cong\Ho^{1}\left(D^{\bullet}\right)$
    follows.
 
    The computation of $\Ho^{i}\left(D^{\bullet}\right)$ for $i\geq 2$ is clear.
  \end{proof}

%%%%%%%%%%%%%%%%%%%%%%%%%%%%%%%%%%%%%%%%%%%%%%%%%%%%%%%%%%%%
%%%%%%%%%%%%%%%%%%%%%%%%%%%%%%%%%%%%%%%%%%%%%%%%%%%%%%%%%%%%
% Cech cohomology III
%%%%%%%%%%%%%%%%%%%%%%%%%%%%%%%%%%%%%%%%%%%%%%%%%%%%%%%%%%%%
%%%%%%%%%%%%%%%%%%%%%%%%%%%%%%%%%%%%%%%%%%%%%%%%%%%%%%%%%%%%

\section{Classical computations revisited}

We reformulate classical computations
in terms of the continuous Galois cohomology as in Definition~\ref{defn:indBan-ctsRGamma-recpaper}.

 \begin{prop}\label{prop:Galois-cohomology-Banach-k0-finiteextensionofk}
  Fix a finite Galois extension $k^{\prime}$ of $k$, which we view as a $k$-Banach space. Then
  \begin{equation*}
    \Ho_{\cont}^{i}\left(\Gal\left( k^{\prime} / k \right) , k^{\prime} \right)
    \cong
    \begin{cases}
      \I\left(k\right), &\text{ if $i=0$ and} \\
      0, &\text{otherwise}.
    \end{cases}
  \end{equation*}
\end{prop}
 
\begin{proof}
  We have to check that the following complex of $k$-Banach spaces
  is strictly exact:
  \begin{equation}\label{eq:Galois-cohomology-Banach-k0-finiteextensionofk-thecomplex-reconstructionpaper}
    0\longrightarrow
    k\longrightarrow
    k^{\prime}
    \stackrel{d^{0}}{\longrightarrow}
    \intHom_{\cont}\left( \Gal\left( k^{\prime} / k \right)^{1} , k^{\prime}\right)
    \stackrel{d^{1}}{\longrightarrow}\dots.
  \end{equation}
  Indeed, in this case~\cite[Corollary 1.2.28]{Sch99}
  would imply that 
  \begin{equation*}
    0\longrightarrow
    \I\left(k\right)\longrightarrow
    \I\left(k^{\prime}\right)
    \stackrel{\I\left(d^{0}\right)}{\longrightarrow}
    \I\left(\intHom_{\cont}\left( \Gal\left( k^{\prime} / k \right)^{1} , k^{\prime}\right)\right)
    \stackrel{\I\left(d^{1}\right)}{\longrightarrow}\dots.
  \end{equation*}
  is exact, and Proposition~\ref{prop:Galois-cohomology-Banach-k0-finiteextensionofk}
  would follow from Lemma~\ref{lem:LH-vs-H-reconstructionpaper}.
  
  Consider again~(\ref{eq:Galois-cohomology-Banach-k0-finiteextensionofk-thecomplex-reconstructionpaper}).
  As the topology on $\Gal\left(k^{\prime}/k\right)$ is discrete and thanks
  to the open mapping theorem, it suffices to check that the following
  complex of abstract groups is exact:
  \begin{equation*}
    0\longrightarrow
    k\longrightarrow
    k^{\prime}
    \stackrel{d^{0}}{\longrightarrow}
    \Hom\left( \Gal\left( k^{\prime} / k \right)^{1} , k^{\prime}\right)
    \stackrel{d^{1}}{\longrightarrow}\dots.
  \end{equation*}
  In degree zero, this follows from a well-known result of
  Ax, cf.~\cite{Ax70Zerosofpolynomialsoverlocalfields}, that is
  $\Ho^{0}\left( \Gal\left( k^{\prime} / k \right) , k^{\prime}\right)=k$. In higher degrees,
  we may apply Hilbert's Theorem 90 as in~\cite[Theorem 6.3.7]{Wei94}
  \footnote{We refer the reader to \emph{loc. cit.} Example 6.4.8 for an interesting discussion regarding cyclic Galois extensions.}.
\end{proof}

In Corollary~\ref{cor:Galois-cohomology-Banach-k0-finiteextensionofk-solid-reconstructionpaper},
the continuous cohomology refers to the one in Notation~\ref{notation:solidctscoh-recpaper}.

\begin{cor}\label{cor:Galois-cohomology-Banach-k0-finiteextensionofk-solid-reconstructionpaper}
  Fix a finite Galois extension $k^{\prime}$ of $k$, which we view as a $k$-Banach space. Then
  \begin{equation*}
    \Ho_{\cont}^{i}\left(\Gal\left( k^{\prime} / k \right) , \underline{k^{\prime}} \right)
    \cong
    \begin{cases}
      \underline{k}, &\text{ if $i=0$ and} \\
      0, &\text{otherwise}.
    \end{cases}
  \end{equation*}
\end{cor}

\begin{proof}
  Apply Lemma~\ref{lem:contgpcoh-indban-vs-solid-reconstructionpaper}
  to Proposition~\ref{prop:Galois-cohomology-Banach-k0-finiteextensionofk}.
\end{proof}

\begin{lem}\label{lem:HomcontSintoBanachunitball-reconstructionpaper}
  Given a profinite set $S$ and a $k$-Banach space $V$ with unit ball $V^{\circ}$, we have
  \begin{equation}\label{eq:HomcontSintoBanachunitball-reconstructionpaper}
    \Hom_{\cont}\left(S,V^{\circ}\right)[1/p]
    \isomap\Hom_{\cont}\left(S,V\right),
  \end{equation}
  as $k$-normed spaces. Its domain is the $k$-normed space
  with unit ball $\Hom_{\cont}\left(S,V^{\circ}\right)$.
\end{lem}

\begin{proof}
  Boundedness and injectivity are clear.
  Given a continuous map $f\colon S\to V$, its image is compact because $S$ is compact.
  That is, we may write $f=f^{\prime}/p^{m}$ for a given continuous map $f^{\prime}\colon S\to V^{\circ}$,
  and the surjectivity of~(\ref{eq:HomcontSintoBanachunitball-reconstructionpaper}) follows.
  Finally, apply the open mapping theorem.
\end{proof}

$C$ is the completion of a fixed algebraic closure $\overline{k}$
of $k$, cf. \S\ref{subsec:conventions-reconstructionpaper}.

 \begin{prop}\label{prop:Galois-cohomology-Banach-C}
  We compute the continuous Galois cohomology of the $k$-Banach space $C$:
  \begin{equation*}
    \Ho_{\cont}^{i}\left(\Gal\left( \overline{k} / k \right) , C \right)
    \cong
    \begin{cases}
      \I\left(k\right), &\text{ if $i=0,1$ and} \\
      0, &\text{otherwise}.
    \end{cases}
  \end{equation*}
\end{prop}
 
\begin{proof}
  Consider the complex
  \begin{equation}\label{eq:Galois-cohomology-Banach-C-thecomplex-reconstructionpaper}
    0\longrightarrow
    k\longrightarrow
    C
    \stackrel{d^{0}}{\longrightarrow}
    \intHom_{\cont}\left( \Gal\left( \overline{k} / k \right)^{1} , C \right)
    \stackrel{d^{1}}{\longrightarrow}\dots.
  \end{equation}
  By~\cite[Corollary 1.2.28]{Sch99} and Proposition~\ref{lem:LH-vs-H-reconstructionpaper},
  we have to check that (i)
  (\ref{eq:Galois-cohomology-Banach-C-thecomplex-reconstructionpaper}) is strictly exact in all degrees $n\neq 1$, and
  (ii) its cohomology in degree $n=1$ is isomorphic to $\I\left(k\right)$.
  Regarding (i), the exactness follows from~\cite[Theorem 1]{tatepdivisiblegroups},
  see also~\cite[Theorem 4.0.3(2)]{BSSW2024_rationalizationoftheKnlocalsphere},
  and the strict exactness follows from the open mapping theorem.
  Next, we work in degree $n=1$. Restrict (\ref{eq:Galois-cohomology-Banach-C-thecomplex-reconstructionpaper}) to
  \begin{equation*}
    0\longrightarrow
    \cal{O}_{k}\longrightarrow
    \cal{O}_{C}
    \stackrel{e^{0}}{\longrightarrow}
    \intHom_{\cont}\left( \Gal\left( \overline{k} / k \right)^{1} , \cal{O}_{C} \right)
    \stackrel{e^{1}}{\longrightarrow}\dots.
  \end{equation*}
  Then~\cite[Theorem 4.0.3]{BSSW2024_rationalizationoftheKnlocalsphere}
  gives the complex
  \begin{equation*}
    \cal{O}_{C}
    \stackrel{e^{0}}{\longrightarrow}
    \ker e^{1}
    \longrightarrow
    \cal{O}_{k}
    \longrightarrow
    0,
  \end{equation*}
  which is exact up to bounded $p$-torsion. Invert $p$ to get the exact complex
  \begin{equation}\label{eq:Galois-cohomology-Banach-C-thecomplex2-reconstructionpaper}
    C
    \stackrel{d^{0}}{\longrightarrow}
    \ker d^{1}
    \longrightarrow
    k
    \longrightarrow
    0.
  \end{equation}  
  Here, we used Lemma~\ref{lem:HomcontSintoBanachunitball-reconstructionpaper}.
  From the construction, it follows that (\ref{eq:Galois-cohomology-Banach-C-thecomplex2-reconstructionpaper})
  is an exact complex of $k$-Banach spaces.
  The open mapping theorem implies that it is strictly exact, and we find
  \begin{equation*}
    \Ho_{\cont}^{1}\left(\Gal\left( \overline{k} / k \right) , C \right)
    \stackrel{\text{\ref{lem:LH-vs-H-reconstructionpaper}}}{\cong}\coker\left(\I\left(C\right)
    \stackrel{\I\left(d^{0}\right)}{\longrightarrow}
    \ker \I\left(d^{1}\right)\right)
    \cong\I\left(k\right).
  \end{equation*}
  Here, the second isomorphism follows from~\cite[Corollary 1.2.28]{Sch99}.
\end{proof}

 \begin{prop}\label{prop:Galois-cohomology-Banach-C-solid-reconstructionpaper}
  We compute the continuous Galois cohomology of the solid $k$-vector space $\underline{C}$:
  \begin{equation*}
    \Ho_{\cont}^{i}\left(\Gal\left( \overline{k} / k \right) , \underline{C} \right)
    \cong
    \begin{cases}
      \underline{k}, &\text{ if $i=0,1$ and} \\
      0, &\text{otherwise}.
    \end{cases}
  \end{equation*}
\end{prop}

\begin{proof}
  Apply Lemma~\ref{lem:contgpcoh-indban-vs-solid-reconstructionpaper}
  to Proposition~\ref{prop:Galois-cohomology-Banach-C}.
\end{proof}

%%%%%%%%%%%%%%%%%%%%%%%%%%%%%%%%%%%%%%%%%%%%%%%%%%%%%%%%%%%%
%%%%%%%%%%%%%%%%%%%%%%%%%%%%%%%%%%%%%%%%%%%%%%%%%%%%%%%%%%%%
% Cech cohomology III
%%%%%%%%%%%%%%%%%%%%%%%%%%%%%%%%%%%%%%%%%%%%%%%%%%%%%%%%%%%%
%%%%%%%%%%%%%%%%%%%%%%%%%%%%%%%%%%%%%%%%%%%%%%%%%%%%%%%%%%%%

\section{The Galois cohomology of $B_{\dR}^{\dag,+}$}

Recall the computation of the continuous Galois cohomology of $C$ as in
Proposition~\ref{prop:Galois-cohomology-Banach-C}.

\begin{thm}\label{thm:galois-cohomology-of-solidBdRdaggerplus-born}
  Fontaine's map $B_{\dR}^{\dag,+}\to C$ induces the isomorphism
  \begin{equation*}
    \Ho_{\cont}^{i}\left(\Gal\left( \overline{k} / k \right) , B_{\dR}^{\dag,+} \right)
    \cong
    \begin{cases}
      \I\left(k\right), &\text{ if $i=0,1$ and} \\
      0, &\text{otherwise}.
    \end{cases}
  \end{equation*}
\end{thm}

From now on, we work towards the proof of Theorem~\ref{thm:galois-cohomology-of-solidBdRdaggerplus-born}.

%%%%%%%%%%%%%%%%%%%%%%%%%%%%%%%%%%%%%%%%%%%%%%%%%%%%%%%%%%%%
% Overview of the proof of 
%%%%%%%%%%%%%%%%%%%%%%%%%%%%%%%%%%%%%%%%%%%%%%%%%%%%%%%%%%%%

\subsection{Overview of the proof of Theorem~\ref{thm:galois-cohomology-of-solidBdRdaggerplus-born}}

\begin{notation}\label{notation:galois-cohomology-of-BdRdaggerplus-kzero}
  Fix an intermediate extension $\overline{k} / k^{\prime} / k$ such that $k^{\prime} / k$ is finite Galois
  and $k^{\prime}$ contains a primitive $p$th root of unity $\zeta_{p}$.
\end{notation}

The proof of Theorem~\ref{thm:galois-cohomology-of-solidBdRdaggerplus-born}
proceeds in two steps: We compute the continuous Galois cohomology of $B_{\dR}^{\dag,+}$ with respect to
$\overline{k}/k^{\prime}$, and then with respect to $k^{\prime}/k$. The first step is achieved
by Proposition~\ref{prop:galois-cohomology-of-BdRdaggerplus-primitiverootofunity}.
The second one is Theorem~\ref{thm:galois-cohomology-of-solidBdRdaggerplus}.

\begin{remark}\label{remark:thmgalois-cohomology-of-solidBdRdaggerplus-indBansolidmethods}
  The first step of the computation is carried out in the setting of ind-Banach spaces,
  as this allows to work with associated gradeds. The second step, 
  that is the proof of of Theorem~\ref{thm:galois-cohomology-of-solidBdRdaggerplus},
  works with the solid formalism, where the Hochschild-Serre spectral sequence is available.
  From this, we are then able to deduce
  Theorem~\ref{thm:galois-cohomology-of-solidBdRdaggerplus-born} in the ind-Banach setting.
\end{remark}

%%%%%%%%%%%%%%%%%%%%%%%%%%%%%%%%%%%%%%%%%%%%%%%%%%%%%%%%%%%%
% On Proposition~\ref{prop:galois-cohomology-of-BdRdaggerplus-primitiverootofunity}
%%%%%%%%%%%%%%%%%%%%%%%%%%%%%%%%%%%%%%%%%%%%%%%%%%%%%%%%%%%%

\subsection{On Proposition~\ref{prop:galois-cohomology-of-BdRdaggerplus-primitiverootofunity}
  and its proof}\label{subsec:on:prop:galois-cohomology-of-BdRdaggerplus-primitiverootofunity}

\begin{prop}\label{prop:galois-cohomology-of-BdRdaggerplus-primitiverootofunity}
  Fontaine's map $B_{\dR}^{\dag,+}\to C$ induces the isomorphism
  \begin{equation*}
    \Ho_{\cont}^{i}\left(\Gal\left( \overline{k} / k^{\prime} \right) , B_{\dR}^{\dag,+} \right)
    \cong
    \begin{cases}
      \I\left(k^{\prime}\right), &\text{ if $i=0,1$ and} \\
      0, &\text{otherwise}.
    \end{cases}
  \end{equation*}
  Here, $k^{\prime}$ denotes the fixed extension of $k$ as in
  Notation~\ref{notation:galois-cohomology-of-BdRdaggerplus-kzero}.
\end{prop}

From now on, we work towards the proof of Proposition~\ref{prop:galois-cohomology-of-BdRdaggerplus-primitiverootofunity}.
For the convenience of the reader, we choose to give an overview of our arguments in
\S\ref{subsec:on:prop:galois-cohomology-of-BdRdaggerplus-primitiverootofunity}.

Write $\mathcal{G}:=\Gal\left( \overline{k} / k^{\prime} \right)$.
%and recall Definition~\ref{defn:Banachmodule-Ccontcomplex-reconstructionpaper}.
The following complex computes the Galois cohomology of $C$:
\begin{equation*}
  C
  \stackrel{d^{0}}{\longrightarrow}
  \intHom_{\cont}\left(\cal{G},C\right)
  \stackrel{d^{1}}{\longrightarrow}
  \intHom_{\cont}\left(\cal{G}^{2},C\right)
  \stackrel{d^{2}}{\longrightarrow}\dots.
\end{equation*}
By Lemma~\ref{lem:sheavesonXproet-over-affinoidperfectoid}
and~\ref{lem:hcalOUtimesS-is-HomctsScalO}, this complex
is isomorphic to
\begin{equation*}
  \widehat{\cal{O}}\left(*\right)
  \stackrel{d^{0}}{\longrightarrow}
  \widehat{\cal{O}}\left(\cal{G} \times *\right)
  \stackrel{d^{1}}{\longrightarrow}
  \widehat{\cal{O}}\left(\cal{G}^{2} \times *\right)
  \stackrel{d^{2}}{\longrightarrow}\dots,
\end{equation*}
where $*:=\Spa\left(C,\cal{O}_{C}\right)$. In the following, we denote this complex by
$\widehat{\cal{O}}\left(\cal{G}\times*\right)$.

Similarly, Theorem~\ref{thm:subsections-periodsheaves-affperfd}
and Proposition~\ref{prop:BdRdaggerplusUtimesS-isomapHomcontSAdRgreaterthanqU-reconstructionpaper}
imply that the following complex
\begin{equation*}%\label{eq:galois-cohomology-of-BdRdagger-thecomplexcomputingGaloiscoh}
  \BB_{\dR}^{\dag,+}\left(*\right)
  \stackrel{\delta^{0}}{\longrightarrow}
  \BB_{\dR}^{\dag,+}\left(\cal{G} \times *\right)
  \stackrel{\delta^{1}}{\longrightarrow}
  \BB_{\dR}^{\dag,+}\left(\cal{G}^{2} \times *\right)
  \stackrel{\delta^{2}}{\longrightarrow}\dots
\end{equation*}
computes the continuous Galois cohomology of $B_{\dR}^{\dag,+}$.
Denote it by $\BB_{\dR}^{\dag,+}\left(\cal{G}^{\bullet}\times*\right)$.

\begin{lem}\label{lem:galois-cohomology-of-BdRdaggerplus-primitiverootofunity-whatistheretoshow}
  Fontaine's maps induce a morphism of complexes of $k$-ind-Banach spaces
  \begin{equation}\label{eq:galois-cohomology-of-BdRdaggerplus-primitiverootofunity-whatistheretoshow}
    \BB_{\dR}^{\dag,+}\left(\cal{G}^{\bullet} \times *\right)
    \to\widehat{\cal{O}}\left(\cal{G}^{\bullet} \times *\right).
  \end{equation}
  If it is a quasi-isomorphism
  in the sense of Definition~\ref{defn:quasiabelian-quasiiso-recpaper},
  Proposition~\ref{prop:galois-cohomology-of-BdRdaggerplus-primitiverootofunity}
  follows.
\end{lem}

\begin{proof}
  This is clear from the previous discussion.
\end{proof}

We proceed by proving that~(\ref{eq:galois-cohomology-of-BdRdaggerplus-primitiverootofunity-whatistheretoshow})
is a quasi-isomorphism of cochain complexes of ind-$k_{0}$-Banach modules.
We give an overview of our arguments in the remainder
of this \S\ref{subsec:on:prop:galois-cohomology-of-BdRdaggerplus-primitiverootofunity}.

Since $\BB_{\dR}^{\dag,+}=\text{``}\varinjlim\text{''}_{q\in\NN}\BB_{\dR}^{>q,+}$,
one may aim to show that the morphisms
\begin{equation*}
  \BB_{\dR}^{>q,+}\left(\cal{G}^{\bullet} \times *\right)
  \to\widehat{\cal{O}}\left(\cal{G}^{\bullet} \times *\right),
\end{equation*}
induced by Fontaine's maps, are quasi-isomorphisms large $q\in\NN$.
Since filtered colimits are strictly exact, this would imply
that~(\ref{eq:galois-cohomology-of-BdRdaggerplus-primitiverootofunity-whatistheretoshow})
is a quasi-isomorphism. Unfortunately, this approach does not work.
We explain now why this is the case.

  $\BB_{\dR}^{>q,+}\left(\cal{G}^{\bullet} \times *\right)$ is the completed localisation
  of $\A_{\dR}^{>q,+}\left(\cal{G}^{\bullet} \times *\right)$. This complex is more
  accessible, as it carries a well-behaved filtration, namely the $\left(p,\xi/p^{q}\right)$-adic one.
  One could therefore aim to compute the cohomology by looking at the associated gradeds
  \begin{equation*}
    \gr \A_{\dR}^{>q,+}\left(\cal{G}^{\bullet} \times *\right)\cong
    \left(\widehat{\cal{O}}^{+}\left(\cal{G}^{\bullet} \times *\right)/p\right)\left[\sigma(p),\sigma\left(\frac{\xi}{p^{q}}\right)\right].
  \end{equation*}
  For the zeroth graded piece, this requires to study
  the cohomology of the complex
  $\widehat{\cal{O}}^{+}\left(\cal{G}^{\bullet} \times *\right)/p$,
  that is the Galois cohomology of $\cal{O}_{C}/p$.
  As we are not aware of such computations in the literature,
  we choose to take a different route:
  Instead of $\A_{\dR}^{>q,+}\left(\cal{G}^{\bullet} \times *\right)$,
  we consider
  $\widetilde{\A}_{\dR}^{>q,+}\left(\cal{G}^{\bullet} \times *\right)$.
  %cf. Definition~\ref{defn:Wkappatriv-tildeAdRgreaterthanq}.
  Both cochain complexes coincide as complexes of abstract $W(\kappa)$-modules, but the
  latter one carries the $p/\xi^{p}$-adic filtration. Then
  \begin{equation*}
    \gr\widetilde{\A}_{\dR}^{>q,+}\left(\cal{G}^{\bullet} \times *\right)
    \cong\widehat{\cal{O}}^{+}\left(\cal{G}^{\bullet} \times *\right)\left[\sigma\left(\frac{\xi}{p^{q}}\right)\right].
  \end{equation*}
  The cohomology of this complex can be computed in terms of
  the continuous Galois cohomology of Tate twists of the rings of integers $\cal{O}_{C}$ of $C$:
  \begin{equation}\label{eq:howtoproceedintheproofofGaloiscohoverkprime-reconstructionpaper}
    \Ho^{i}\left(\gr^{n}\widetilde{\A}_{\dR}^{>q,+}\left(\cal{G}^{\bullet} \times *\right)\right)
    \cong\Ho_{\cont}^{i}\left(\cal{G},
    \cal{O}_{C}(n)\right).
  \end{equation}  
  Here, we are crucially using that we are working over $k^{\prime}$,
  cf. the arguments at the end of the proof of
  Lemma~\ref{lem:galois-cohomology-of-BdRdagger-Dgreaterthanqbulletcomputedgrn}.
  The continuous cohomology groups~(\ref{eq:howtoproceedintheproofofGaloiscohoverkprime-reconstructionpaper})
  are well known, cf.~\cite[Theorem 4.0.4 and 4.0.5]{BSSW2024_rationalizationoftheKnlocalsphere}.
  To be precise, \emph{loc. cit.} computes the cohomology up $p^{N}$-torsion,
  where $N$ is explicit. One may therefore hope that
  one could deduce an explicit description of the cohomology
  of $\widetilde{\A}_{\dR}^{>q,+}\left(\cal{G}^{\bullet} \times *\right)$
  up to $p^{N}$-torsion, say with Lemma~\ref{lem:decalage-commutes-gr},  
  and thus a computation of the
  desired cohomology of $\BB_{\dR}^{>q,+}\left(\cal{G}^{\bullet} \times *\right)$.
  
  This hope turns out to be unjustified since $N$ depends on $n$.
  But luckily, $N$ depends linearly on $n$. This motivates the definition of a slight modification
  $D^{>q,\bullet}$ of $\A_{\dR}^{>q,+}\left(\cal{G}^{\bullet} \times *\right)$,
  cf. \S\ref{subsubsec:galois-cohomology-of-BdRdaggerplus-primitiverootofunity1}.
  Similarly, we consider the variant
  $\widetilde{D}^{>q,\bullet}$ of $\widetilde{\A}_{\dR}^{>q,+}\left(\cal{G}^{\bullet} \times *\right)$
  in \S\ref{subsubsec:galois-cohomology-of-BdRdaggerplus-primitiverootofunity2}.
  In \S\ref{subsubsec:galois-cohomology-of-BdRdaggerplus-primitiverootofunity3},
  we observe that this modification indeed kills the junk torsion in the cohomology of
  each degree of the associated graded~(\ref{eq:howtoproceedintheproofofGaloiscohoverkprime-reconstructionpaper}).
  We then finish the proof of
  Proposition~\ref{prop:galois-cohomology-of-BdRdaggerplus-primitiverootofunity}
  in \S\ref{subsubsec:galois-cohomology-of-BdRdaggerplus-primitiverootofunity6}
  after setting up further technical machinery in
  \S\ref{subsubsec:galois-cohomology-of-BdRdaggerplus-primitiverootofunity4}
  and \S\ref{subsubsec:galois-cohomology-of-BdRdaggerplus-primitiverootofunity5}.

%%%%%%%%%%%%%%%%%%%%%%%%%%%%%%%%%%%%%%%%%%%%%%%%%%%%%%%%%%%%
% On Proposition~\ref{prop:galois-cohomology-of-BdRdaggerplus-primitiverootofunity}
%%%%%%%%%%%%%%%%%%%%%%%%%%%%%%%%%%%%%%%%%%%%%%%%%%%%%%%%%%%%

\subsection{The cochain complex $D^{>q,\bullet}$}\label{subsubsec:galois-cohomology-of-BdRdaggerplus-primitiverootofunity1}

\begin{notation}\label{defn:galois-cohomology-of-BdRdaggerplus-primitiverootofunity-Dgreaterthanqbullet-somedifferentials}
  We introduce the following notation for the differentials in the cochain complex
  \begin{equation*}
    \A_{\dR}^{>q}\left(*\right)
    \stackrel{\delta^{>q,0}}{\longrightarrow}
    \A_{\dR}^{>q}\left(\cal{G} \times *\right)
    \stackrel{\delta^{>q,1}}{\longrightarrow}
    \A_{\dR}^{>q}\left(\cal{G}^{2} \times *\right)
    \stackrel{\delta^{>q,2}}{\longrightarrow}\dots.
  \end{equation*}
\end{notation}

\begin{defn}\label{defn:galois-cohomology-of-BdRdaggerplus-primitiverootofunity-Dgreaterthanqbullet}
For every $q\in\NN_{\geq1}$, denote the following sequence of maps by $D^{>q,\bullet}$:
\begin{equation*}
  \A_{\dR}^{>q}\left(*\right)
  \times_{\delta^{>q,0} , \A_{\dR}^{>q}\left(\cal{G} \times *\right) }
  \A_{\dR}^{>q-1}\left(\cal{G} \times *\right)
  \stackrel{\epsilon^{>q,0}}{\longrightarrow}
  \A_{\dR}^{>q-1}\left(\cal{G} \times *\right)
  \stackrel{\epsilon^{>q,1}}{\longrightarrow}
  \A_{\dR}^{>q-1}\left(\cal{G}^{2} \times *\right)
  \stackrel{\epsilon^{>q,2}}{\longrightarrow}\dots.
\end{equation*}
Here, $\epsilon^{>q,0}$ is $(a,b)\mapsto b$ and $\epsilon^{>q,i}:=\delta^{>q-1,i}$
for all $i>0$.
\end{defn}

\begin{lem}\label{lem:galois-cohomology-of-BdRdagger-Dgreaterthanqbullet-cochaincomplex}
  $D^{>q,\bullet}$ is a cochain complex if $q\geq3$.
\end{lem}

\begin{proof}
  Clearly, $\epsilon^{>q,i+1}\circ\epsilon^{>q,i}=\delta^{>q-1,i+1}\circ\delta^{>q-1,i}=0$ for all $i\geq 1$.
  To check $\epsilon^{>q,1}\circ\epsilon^{>q,0}=0$, consider
  an arbitrary $\left(a,b\right)\in\A_{\dR}^{>q}\left(*\right)
  \times_{\delta^{>q,0} , \A_{\dR}^{>q}\left(\cal{G} \times *\right) }
  \A_{\dR}^{>q-1}\left(\cal{G} \times *\right)$.
  Then
  \begin{align*}
    \iota^{>q-1}\left(\left(\epsilon^{>q,1}\circ\epsilon^{>q,0}\right)\left(a,b\right)\right)
    &=\iota^{>q-1}\left(\delta^{>q-1,1}(b)\right) \\
    &=\delta^{>q,1}\left(\iota^{>q-1}(b)\right) \\
    &=\delta^{>q,1}\left(\delta^{>q,0}(a)\right) \\
    &=0
  \end{align*}
  where $\iota^{q-1}$ denotes any of the canonical maps
  $\A_{\dR}^{>q-1}\left(\cal{G}^{i} \times *\right)\to\A_{\dR}^{>q}\left(\cal{G}^{i} \times *\right)$.
  Lemma~\ref{lem:AdRgreaterthanq-toAdRgreaterthanqplus1-mono},
  which applies thanks to Theorem~\ref{thm:BdR>qplus+-sections-over-affperfd-recpaper},
  implies $\left(\epsilon^{>q,1}\circ\epsilon^{>q,0}\right)\left(a,b\right)=0$.
  %The result follows because the choice of $(a,b)$ was arbitrary.
  Lemma~\ref{lem:galois-cohomology-of-BdRdagger-Dgreaterthanqbullet-cochaincomplex} follows.
\end{proof}

Fix $q\in\NN_{\geq3}$ and consider the commutative diagram
\begin{equation*}
  \begin{tikzcd}
      \A_{\dR}^{>q}\left(*\right)
      \arrow{r}{\delta^{>q,0}} &
      \A_{\dR}^{>q}\left(\cal{G} \times*\right)
      \arrow{r}{\delta^{>q,1}} &
      \A_{\dR}^{>q}\left(\cal{G}^{2} \times*\right)
      \arrow{r}{\delta^{>q,2}} &
      \dots \\
      \A_{\dR}^{>q}\left(*\right)
      \times_{\delta^{>q,0} , \A_{\dR}^{>q,+}\left(\cal{G} \times *\right) }
      \A_{\dR}^{>q-1}\left(\cal{G} \times *\right)
      \arrow{r}{\epsilon^{>q,0}}\arrow{u}{\tau^{0}} &
      \A_{\dR}^{>q-1}\left(\cal{G} \times*\right)
      \arrow{r}{\epsilon^{>q,1}}\arrow{u} &
      \A_{\dR}^{>q-1}\left(\cal{G}^{2} \times*\right)
      \arrow{r}{\epsilon^{>q,2}}\arrow{u} &
      \dots  
  \end{tikzcd}
\end{equation*}
where $\tau^{0}$ is $(a,b)\mapsto a$ and all other vertical maps are canonical.
It defines a morphism
\begin{equation}\label{eq:Phigreatanqbullet}
  \Phi^{>q,\bullet}\colon D^{>q,\bullet} \to \A_{\dR}^{>q}\left(\cal{G}^{\bullet}\times*\right)
\end{equation}
of cochain complexes.

\begin{lem}\label{lem:galois-cohomology-of-BdRdagger-Dgreaterthanqbullet-isotoBdRdagplusGbulletstar}
  $\text{``}\varinjlim\text{"}_{q\in\NN_{\geq3}}\Phi^{>q,\bullet}\widehat{\otimes}_{W(\kappa)}\id_{k_{0}}$
  is a canonical isomorphism of cochain complexes
  \begin{equation*}
    \text{``}\varinjlim\text{"}_{q\in\NN_{\geq3}}D^{>q,\bullet}\widehat{\otimes}_{W(\kappa)}k_{0}
    \isomap\BB_{\dR}^{\dag,+}\left(\cal{G}^{\bullet}\times*\right).
  \end{equation*}
\end{lem}

\begin{proof}
    The following application of Lemma~\ref{lem:completed-localisation-preserves-fiberproducts}
    is allowed by Lemma~\ref{lem:AdRgreaterthanqptorsionfree}:
    \begin{align*}
      &\text{``}\varinjlim_{q\in\NN_{\geq1}}\text{"}
      \left(\A_{\dR}^{>q}\left(*\right)
      \times_{\delta^{>q,0} , \A_{\dR}^{>q}\left(\cal{G} \times *\right) }
      \A_{\dR}^{>q-1}\left(\cal{G} \times *\right)\right)\widehat{\otimes}_{W(\kappa)}k_{0} \\
      &\stackrel{\text{\ref{lem:completed-localisation-preserves-fiberproducts}}}{\cong}
      \text{``}\varinjlim_{q\in\NN_{\geq2}}\text{"}
      \left(\BB_{\dR}^{>q,+}\left(*\right)
      \times_{\delta^{>q,0} , \BB_{\dR}^{>q,+}\left(\cal{G} \times *\right) }
      \BB_{\dR}^{>q-1,+}\left(\cal{G} \times *\right)\right) \\
      &\cong
      \BB_{\dR}^{\dag,+}\left(*\right)
      \times_{\delta^{0} , \BB_{\dR}^{\dag,+}\left(\cal{G} \times *\right) }
      \BB_{\dR}^{\dag,+}\left(\cal{G} \times *\right).
    \end{align*}
    Therefore, we have the following depiction of
    $\text{``}\varinjlim\text{"}_{q\in\NN_{\geq1}}\Phi^{>q,\bullet}\widehat{\otimes}_{W(\kappa)}\id_{k_{0}}$, which
    implies Lemma~\ref{lem:galois-cohomology-of-BdRdagger-Dgreaterthanqbullet-isotoBdRdagplusGbulletstar}: %Lemma~\ref{lem:galois-cohomology-of-BdRdagger-Dgreaterthanqbullet-isotoBdRdagplusGbulletstar}:
    \begin{equation*}
    \begin{tikzcd}
      \BB_{\dR}^{\dag,+}\left(*\right)
      \arrow{r}{\delta^{0}} &
      \BB_{\dR}^{\dag,+}\left(\cal{G} \times *\right)
      \arrow{r}{\delta^{1}} &
      \BB_{\dR}^{\dag,+}\left(\cal{G}^{2} \times *\right)
      \arrow{r}{\delta^{2}} &
      \dots \\
      \BB_{\dR}^{\dag,+}\left(*\right)
      \times_{\delta^{0} , \BB_{\dR}^{\dag,+}\left(\cal{G} \times *\right) }
      \BB_{\dR}^{\dag,+}\left(\cal{G} \times *\right)
      \arrow{r}{\epsilon^{0}}\arrow{u}{\tau}\arrow[swap]{u}{\cong} &
      \BB_{\dR}^{\dag,+}\left(\cal{G}^{1} \times*\right)
      \arrow{r}{\delta^{1}}\arrow[equal]{u} &
      \BB_{\dR}^{\dag,+}\left(\cal{G}^{2} \times*\right)
      \arrow{r}{\delta^{2}}\arrow[equal]{u} &
      \dots.
    \end{tikzcd}
    \end{equation*}
    Here, $\tau$ and $\epsilon^{0}$ are the projections onto the first and second variables,
    respectively.
\end{proof}

Lemma~\ref{lem:galois-cohomology-of-BdRdagger-Dgreaterthanqbullet-isotoBdRdagplusGbulletstar}
says that one can compute the desired Galois cohomology
\begin{equation*}
  \Ho_{\cont}^{i}\left(\Gal\left( \overline{k} / k^{\prime} \right) , B_{\dR}^{\dag,+} \right)
\end{equation*}
via the cohomology of $D^{>q,\bullet}$.
More precisely, Fontaine's maps induce a morphism
\begin{equation*}
  \vartheta^{>q,\bullet}\colon D^{>q,\bullet} \to \widehat{\cal{O}}^{+}\left(\cal{G}^{\bullet}\times*\right)
\end{equation*}
of cochain complexes of $W(\kappa)$-Banach modules. Here, the right-hand side denotes
the complex
\begin{equation*}
    \widehat{\cal{O}}^{+}\left(*\right)
    \stackrel{d^{+,0}}{\longrightarrow}
    \widehat{\cal{O}}^{+}\left(\cal{G} \times *\right)
    \stackrel{d^{+,1}}{\longrightarrow}
    \widehat{\cal{O}}^{+}\left(\cal{G}^{2} \times *\right)
    \stackrel{d^{+,2}}{\longrightarrow}\dots.
\end{equation*}
We will see later on that $\vartheta^{>q,\bullet}$ is a quasi-isomorphism.

%%%%%%%%%%%%%%%%%%%%%%%%%%%%%%%%%%%%%%%%%%%%%%%%%%%%%%%%%%%%
% On Proposition~\ref{prop:galois-cohomology-of-BdRdaggerplus-primitiverootofunity}
%%%%%%%%%%%%%%%%%%%%%%%%%%%%%%%%%%%%%%%%%%%%%%%%%%%%%%%%%%%%

\subsection{The cochain complex $\widetilde{D}^{>q,\bullet}$}\label{subsubsec:galois-cohomology-of-BdRdaggerplus-primitiverootofunity2}

Recall the Definition~\ref{defn:galois-cohomology-of-BdRdaggerplus-primitiverootofunity-Dgreaterthanqbullet}
of the complex $D^{>q,\bullet}$. In the following, we introduce
the complex $\widetilde{D}^{>q,\bullet}$. It coincides with $D^{>q,\bullet}$
as a complex of abstract $W(\kappa)$-modules, but it carries different topologies.

Given a Banach ring $S$ and an $S$-Banach algebra $A$,
$|S|$ denotes the underlying ring and $|A|$
denotes the underlying $|S|$-algebra.
We recall that $\A_{\dR}^{>q}\left(\cal{G}^{i}\times*\right)$
carries the $\left(p,\xi/p^{q}\right)$-adic topology.

\begin{defn}\label{defn:Wkappatriv-tildeAdRgreaterthanqofU-reconstructionpaper}
  $\widetilde{\A}_{\dR}^{>q}\left(\cal{G}^{i}\times*\right)$ is the
  $|W(\kappa)|$-algebra $|\A_{\dR}^{>q}\left(\cal{G}^{i}\times*\right)|$
  with the $\left(\xi/p^{q}\right)$-adic filtration.
\end{defn}

Thanks to Theorem~\ref{thm:BdR>qplus+-sections-over-affperfd-recpaper},
Definition~\ref{defn:Wkappatriv-tildeAdRgreaterthanqofU-reconstructionpaper}
is compatible with Definition~\ref{defn:Wkappatriv-tildeAdRgreaterthanq},
at least for $q\geq 2$.

Recall Notation~\ref{defn:galois-cohomology-of-BdRdaggerplus-primitiverootofunity-Dgreaterthanqbullet-somedifferentials}.

\begin{defn}\label{defn:galois-cohomology-of-BdRdaggerplus-primitiverootofunity-Dgreaterthanqbullettilde}
We introduce the following sequences $\widetilde{D}^{>q,\bullet}$ for every $q\in\NN_{\geq1}$:
\begin{equation*}
  \widetilde{\A}_{\dR}^{>q}\left(*\right)
  \times_{\delta^{>q,0} , \widetilde{\A}_{\dR}^{>q}\left(\cal{G} \times *\right) }
  \widetilde{\A}_{\dR}^{>q-1}\left(\cal{G} \times *\right)
  \stackrel{\epsilon^{>q,0}}{\longrightarrow}
  \widetilde{\A}_{\dR}^{>q-1}\left(\cal{G} \times *\right)
  \stackrel{\epsilon^{>q,1}}{\longrightarrow}
  \widetilde{\A}_{\dR}^{>q-1}\left(\cal{G}^{2} \times *\right)
  \stackrel{\epsilon^{>q,2}}{\longrightarrow}\dots.
\end{equation*}
Here, $\epsilon^{>q,0}$ is $(a,b)\mapsto b$ and $\epsilon^{>q,i}:=\delta^{>q-1,i}$
for all $i>0$.
\end{defn}

\begin{lem}\label{lem:galois-cohomology-of-BdRdagger-Dgreaterthanqbullettilde-cochaincomplex}
  $\widetilde{D}^{>q,\bullet}$ is a cochain complex if $q\geq3$.
\end{lem}

\begin{proof}
  See Lemma~\ref{lem:galois-cohomology-of-BdRdagger-Dgreaterthanqbullet-cochaincomplex}.
\end{proof}

In the following, we mimic the previous discussion and equip
$\widehat{\cal{O}}^{+}\left(\cal{G}^{\bullet}\times*\right)$ with a similar filtration.

\begin{defn}\label{defn:puttrivialnormonring-reconstructionpaper}
  Given a commutative ring $S$, $S^{\triv}$ denotes $S$ equipped
  with the trivial filtration. That is, $\Fil^{n}S^{\triv}=S$ for $n\leq0$
  and $\Fil^{n}S^{\triv}=0$ otherwise.
\end{defn}

The operation $S\mapsto S^{\triv}$ is functorial. Thus we get the cochain complex
\begin{equation*}
    \widehat{\cal{O}}^{+}\left(*\right)^{\triv}
    \stackrel{d^{+,0}}{\longrightarrow}
    \widehat{\cal{O}}^{+}\left(\cal{G} \times *\right)^{\triv}
    \stackrel{d^{+,1}}{\longrightarrow}
    \widehat{\cal{O}}^{+}\left(\cal{G}^{2} \times *\right)^{\triv}
    \stackrel{d^{+,2}}{\longrightarrow}\dots,
\end{equation*}
of filtered $W(\kappa)^{\triv}$-modules which
which we denote by $\widehat{\cal{O}}^{+}\left(\cal{G}^{\bullet}\times*\right)^{\triv}$.

\begin{lem}\label{lem:needthistodefinewidetildevarthetagreaterthanq-reconstructionpaper}
  Fontaine's
  $\theta_{\dR}^{>q}\colon\widetilde{\A}_{\dR}^{>q}\left(\cal{G}^{i}\times*\right)
    \to\widehat{\cal{O}}^{+}\left(\cal{G}^{i}\times*\right)^{\triv}$
  for $q\in\NN_{\geq2}$
  preserve the filtrations.
\end{lem}

\begin{proof}
  This follows from $\theta_{\dR}^{>q}\left(\xi/p^{q}\right)=0$ and $\theta_{\dR}^{>q}\left(p\right)=p$.
\end{proof}

Lemma~\ref{lem:needthistodefinewidetildevarthetagreaterthanq-reconstructionpaper} implies
that Fontaine's maps induce morphisms
\begin{equation*}
  \widetilde{\vartheta}^{>q,\bullet}\colon
  \widetilde{D}^{>q,\bullet} \to \widehat{\cal{O}}^{+}\left(\cal{G}^{\bullet}\times*\right)^{\triv}
\end{equation*}
of cochain complexes of filtered $W(\kappa)^{\triv}$-modules.
We remark that
$\vartheta^{>q,\bullet}$, cf. \S\ref{subsubsec:galois-cohomology-of-BdRdaggerplus-primitiverootofunity1},
and $\widetilde{\vartheta}^{>q,\bullet}$ coincide
as morphisms of cochain complexes of abstract $W(\kappa)$-modules.

%%%%%%%%%%%%%%%%%%%%%%%%%%%%%%%%%%%%%%%%%%%%%%%%%%%%%%%%%%%%
% The cohomology of the assocaited graded of $\widetilde{D}^{>q,\bullet}$
%%%%%%%%%%%%%%%%%%%%%%%%%%%%%%%%%%%%%%%%%%%%%%%%%%%%%%%%%%%%

\subsection{The cohomology of $\gr\widetilde{D}^{>q,\bullet}$}
\label{subsubsec:galois-cohomology-of-BdRdaggerplus-primitiverootofunity3}

Fix $q\in\NN_{\geq3}$ and recall the Definition~\ref{defn:galois-cohomology-of-BdRdaggerplus-primitiverootofunity-Dgreaterthanqbullettilde}
of $\widetilde{D}^{>q,\bullet}$. We consider
\begin{equation*}
  \gr\widetilde{\vartheta}^{>q,\bullet}\colon
  \gr \widetilde{D}^{>q,\bullet}
  \to
  \gr\widehat{\cal{O}}^{+}\left(\cal{G}^{\bullet}\times*\right)^{\triv}
\end{equation*}
where $\widetilde{\vartheta}^{>q,\bullet}$ is the morphism
of cochain complexes of filtered $W(\kappa)^{\triv}$-modules
introduced in \S\ref{subsubsec:galois-cohomology-of-BdRdaggerplus-primitiverootofunity2}.
Now apply the $\eta$-operator
as in Definition~\ref{defn:algebraic-decalage} to get the
morphism~(\ref{eq:galois-cohomology-of-BdRdagger-whatistheretoshowquestionsmark-themorphism})
in Proposition~\ref{prop:galois-cohomology-of-BdRdagger-whatistheretoshowquestionsmark}.
This Proposition~\ref{prop:galois-cohomology-of-BdRdagger-whatistheretoshowquestionsmark}
is the key technical input in our proof of
Proposition~\ref{prop:galois-cohomology-of-BdRdaggerplus-primitiverootofunity}.
 
\begin{prop}\label{prop:galois-cohomology-of-BdRdagger-whatistheretoshowquestionsmark}
  There exists a constant $N\in\NN$ such that the morphism of abstract $W(\kappa)$-modules
  \begin{equation}\label{eq:galois-cohomology-of-BdRdagger-whatistheretoshowquestionsmark-themorphism}
    \eta_{p^{N}}\gr\widetilde{\vartheta}^{>q,\bullet}\colon
    \eta_{p^{N}}\gr \widetilde{D}^{>q,\bullet}
    \to
    \eta_{p^{N}}\gr\widehat{\cal{O}}^{+}\left(\cal{G}^{\bullet}\times*\right)^{\triv}
  \end{equation}
  induced by Fontaine's maps is a quasi-isomorphism.
\end{prop}

The remainder of this \S is devoted towards the proof of
Proposition~\ref{prop:galois-cohomology-of-BdRdagger-whatistheretoshowquestionsmark},
which itself is presented on page~\pageref{proof:prop:galois-cohomology-of-BdRdagger-whatistheretoshowquestionsmark}.

Given a $\cal{G}$-$W(\kappa)$-Banach module $M$,
$\Ho_{\cont}^{i}\left( \cal{G} , M \right)$ is by Definition~\ref{defn:indBan-ctsRGamma-recpaper}
an object of the left heart
$\Ban_{\I\left(W(\kappa)\right)}$ of the category of $W(\kappa)$-Banach modules.
To relate our notation to Tate's classical continuous cohomology as
in~\cite{Tate1976K2andGaloisCohomology}
or~\cite[\href{https://stacks.math.columbia.edu/tag/0DVG}{Tag 0DVG}]{stacks-project},
we introduce the following.

\begin{notation}\label{notation:Tatesclassicalgroupcohomology}
  Given a $\cal{G}$-$W(\kappa)$-Banach module $M$, we denote by
  $\Ho_{\clcont}^{i}\left( \cal{G} , M \right)$ Tate's classical continuous groups cohomology
  for all $i\in\NN$. We view these merely as abstract groups.
\end{notation}

\begin{lem}\label{lem:galois-cohomology-of-BdRdagger-Dgreaterthanqbulletcomputedgrn}
  For any $n\geq0$, $\cal{O}_{C}(n)$ denotes the $n$th Tate twist of $\cal{O}_{C}$. Compute
  \begin{equation*}
    \Ho^{i}\left(\gr^{n}\widetilde{D}^{>q,\bullet}\right) =
    \begin{cases}
      \Ho^{0}_{\clcont}\left(\cal{G},\cal{O}_{C}(n)\right) & \text{ if $i=0$}, \\
      \Ho^{1}_{\clcont}\left(\cal{G},\cal{O}_{C}(n)\right) / \text{($p^{n}$-torsion)} & \text{ if $i=1$, and} \\
      p^{n}\Ho^{i}_{\clcont}\left(\cal{G},\cal{O}_{C}(n)\right) & \text{ if $i\geq 2$.}
    \end{cases}
  \end{equation*}
\end{lem}

In the following, we make implicit use
of Lemma~\ref{lem:sheavesonXproet-over-affinoidperfectoid}
and Theorem~\ref{thm:BdR>qplus+-sections-over-affperfd-recpaper}.

\begin{proof}[Proof of Lemma~\ref{lem:galois-cohomology-of-BdRdagger-Dgreaterthanqbulletcomputedgrn}]
  In degree zero, we find
  \begin{align*}
    \gr^{n}\widetilde{D}^{>q,0}
    &\stackrel{\text{\ref{lem:gr:commutes-pullback}}}{\cong}
    \gr^{n}\widetilde{\A}_{\dR}^{>q}\left(*\right)
      \times_{\gr^{n}\delta^{0} , \gr^{n}\widetilde{\A}_{\dR}^{>q,+}\left(\cal{G} \times *\right) }
      \gr^{n}\widetilde{\A}_{\dR}^{>q-1}\left(\cal{G} \times *\right) \\
    &\stackrel{\text{\ref{lem:grAdRgreaterthanq-xioverpadicfiltration-reconstructionpaper}}}{\cong}
    \widehat{\cal{O}}^{+}\left(*\right)x^{n}
      \times_{\gr^{n}\delta^{0} , \widehat{\cal{O}}^{+}\left(\cal{G} \times *\right)x^{n} }
      \widehat{\cal{O}}^{+}\left(\cal{G} \times *\right)(px)^{n}
  \end{align*}
  where $x$ denotes the principal symbol of $\xi/p^{q}$;
  consequently, $px$ is the principal symbol of
  $p\cdot\xi/p^{q}=\xi/p^{q-1}$. Another application of Lemma~\ref{lem:grAdRgreaterthanq-xioverpadicfiltration-reconstructionpaper}
  now gives the following depiction of $\gr^{n} D^{>q,\bullet}$:
  \begin{equation}\label{eq:galois-cohomology-of-BdRdagger-Dgreaterthanqbulletcomputedgrn-needthislater}
  \begin{split}
    \widehat{\cal{O}}^{+}\left(*\right)x^{n}
    \times_{\gr^{n}\delta^{0} , \widehat{\cal{O}}^{+}\left(\cal{G} \times *\right)x^{n} }
    \widehat{\cal{O}}^{+}\left(\cal{G} \times *\right)(px)^{n}
    &\stackrel{\gr^{n}\epsilon^{>q,0}}{\longrightarrow}
    \widehat{\cal{O}}^{+}\left(\cal{G} \times *\right)(px)^{n} \\
    &\stackrel{\gr^{n}\epsilon^{>q,1}}{\longrightarrow}
    \widehat{\cal{O}}^{+}\left(\cal{G}^{2} \times *\right)(px)^{n}
    \stackrel{\gr^{n}\epsilon^{>q,2}}{\longrightarrow}\dots.
    \end{split}
  \end{equation}
  By Lemma~\ref{lem:killtorsionincontrolledwayincplx},
  \begin{equation}\label{eq:computeHoigrnDgreatahnqbullet-firstversion-reconstructionpaper}
    \Ho^{i}\left(\gr^{n}D^{>q,\bullet}\right) \cong
    \begin{cases}
      \Ho^{0}\left(\widehat{\cal{O}}^{+}\left(\cal{G}^{\bullet}\times*\right)x^{n}\right) & \text{ if $i=0$}, \\
      \Ho^{1}\left(\widehat{\cal{O}}^{+}\left(\cal{G}^{\bullet}\times*\right)x^{n}\right) / \text{($p^{n}$-torsion)} & \text{ if $i=1$, and} \\
      p^{n}\Ho^{i}\left(\widehat{\cal{O}}^{+}\left(\cal{G}^{\bullet}\times*\right)x^{n}\right) & \text{ if $i\geq 2$.}
    \end{cases}
  \end{equation}
  Next, we apply~\cite[Lemma 4.10 and Corollary 6.6]{Sch13pAdicHodge} to find
  \begin{equation*}
    \widehat{\cal{O}}^{+}\left(\cal{G}^{\bullet} \times *\right)x^{n}
    \cong\Hom_{\cont}\left(\cal{G}^{\bullet},\widehat{\cal{O}}^{+}\left(*\right)x^{n}\right)
    =\Hom_{\cont}\left(\cal{G}^{\bullet},\cal{O}_{C}x^{n}\right).
  \end{equation*}
  Thus, we may rewrite~(\ref{eq:computeHoigrnDgreatahnqbullet-firstversion-reconstructionpaper})
  as follows:
  \begin{equation*}%\label{eq:computeHoigrnDgreatahnqbullet-firstversion-reconstructionpaper}
    \Ho^{i}\left(\gr^{n}D^{>q,\bullet}\right) \cong
    \begin{cases}
      \Ho_{\clcont}^{0}\left(\cal{G} , \cal{O}_{C}x^{n}\right) & \text{ if $i=0$}, \\
      \Ho_{\clcont}^{1}\left(\cal{G} , \cal{O}_{C}x^{n}\right) / \text{($p^{n}$-torsion)} & \text{ if $i=1$, and} \\
      p^{n}\Ho_{\clcont}^{i}\left(\cal{G} , \cal{O}_{C}x^{n}\right) & \text{ if $i\geq 2$.}
    \end{cases}
  \end{equation*}  
  %Thus it remains to compute the $\cal{G}$-action on
  %$\cal{O}_{C}x^{n}\cong\gr^{n}\A_{\dR}^{>q}\left(*\right)$. More precisely, we check that
  It remains to compute that
  \begin{equation}\label{calOCxnconggrnAdRgreaterthanqstar}
    \cal{O}_{C}x^{n} \isomap \cal{O}_{C}(n),
    f x^{n} \mapsto \theta_{\dR}^{>q}(f)
  \end{equation}
  is an isomorphism of $\cal{G}$-representations.
  As it is clearly an isomorphism of Banach modules, it remains to show that it is
  equivariant. Since $\theta_{\dR}^{>q}$ is by definition
  equivariant, it remains to compute the $\cal{G}$-action on $x^{n}$.
  That is, we have to prove that
  $\alpha\cdot x^{n} = \chi(\alpha)^{n} x^{n}$
  for all $\alpha\in\cal{G}$,
  where $\chi\colon\cal{G}\to\ZZ_{p}^{\times}$ denotes the cyclomotic character.
  To do this, fix
  $\alpha\in\cal{G}$ and note that $\alpha$ fixes $\zeta_{p}\in k^{\prime}$. Then
  \begin{equation*}
    \left(\zeta_{p}-1\right)^{n}\alpha x^{n}
    =\alpha\left(\zeta_{p}-1\right)^{n}x^{n}
    \stackrel{\text{\ref{lem:principlesymbol-of-t}}}{=}
    \alpha\sigma\left(\frac{t}{p^{q}}\right)^{n}.
  \end{equation*}
  It is well known that the Galois action on $t$ is via the cyclomotic character
  $\chi$. This gives
  \begin{equation*}
    \left(\zeta_{p}-1\right)^{n}\alpha x^{n}
    =\chi\left(\alpha\right)^{n}\sigma\left(\frac{t}{p^{q}}\right)^{n}
    \stackrel{\text{\ref{lem:principlesymbol-of-t}}}{=}
    \left(\zeta_{p}-1\right)^{n}\chi\left(\alpha\right)^{n}x^{n}.
  \end{equation*}
  As $\left(\zeta_{p}-1\right)^{n}$ is not a zero-divisor, we find $\alpha x^{n}=\chi\left(\alpha\right)x^{n}$,
  as desired.
\end{proof}

\begin{lem}\label{lem:galois-cohomology-of-BdRdaggerplus-primitiverootofunity-morelemma1-reconstructionpaper}
  $\gr^{0}\widehat{\cal{O}}^{+}\left(\cal{G}^{\bullet}\times*\right)^{\triv}
  \cong
  \widehat{\cal{O}}^{+}\left(\cal{G}^{\bullet}\times*\right)$,
  and
  $\gr^{n}\widehat{\cal{O}}^{+}\left(\cal{G}^{\bullet}\times*\right)^{\triv}=0$ for $n>0$.
\end{lem}

\begin{proof}
  This is because the filtration on $\widehat{\cal{O}}^{+}\left(\cal{G}^{\bullet}\times*\right)$
  is trivial.
\end{proof}

Finally, here is the

\begin{proof}[Proof of Proposition~\ref{prop:galois-cohomology-of-BdRdagger-whatistheretoshowquestionsmark}]\label{proof:prop:galois-cohomology-of-BdRdagger-whatistheretoshowquestionsmark}
  We choose $N:=\max\left\{M,6\right\}$ where $M$ is as
  in~\cite[Theorem 4.4.3]{BSSW2024_rationalizationoftheKnlocalsphere}.
  \emph{Loc. cit.} as well as %Lemma~\ref{lem:galois-cohomology-of-BdRdagger-Dgreaterthanqbulletcomputedgr0},
  Lemma~\ref{lem:galois-cohomology-of-BdRdagger-Dgreaterthanqbulletcomputedgrn} and
  \ref{lem:galois-cohomology-of-BdRdaggerplus-primitiverootofunity-morelemma1-reconstructionpaper}
  now imply that in each degree $n$, $\gr^{n} \widetilde{D}^{>q,\bullet}\to
    \gr^{n}\widehat{\cal{O}}^{+}\left(\cal{G}^{\bullet}\times*\right)^{\triv}$
  is a quasi-isomorphism up to $p^{N}$-torsion.
  Thus Proposition~\ref{prop:galois-cohomology-of-BdRdagger-whatistheretoshowquestionsmark} follows
  from Lemma~\ref{lem:eta-operator-kills-torsion}.
\end{proof}

%%%%%%%%%%%%%%%%%%%%%%%%%%%%%%%%%%%%%%%%%%%%%%%%%%%%%%%%%%%%
% Applying the $\eta$-operator
%%%%%%%%%%%%%%%%%%%%%%%%%%%%%%%%%%%%%%%%%%%%%%%%%%%%%%%%%%%%

\subsection{Applying the $\eta$-operator}\label{subsubsec:galois-cohomology-of-BdRdaggerplus-primitiverootofunity4}

%We will now establish foundational lemmata that enable us to deduce
%Proposition~\ref{prop:galois-cohomology-of-BdRdaggerplus-primitiverootofunity}
%from Proposition~\ref{prop:galois-cohomology-of-BdRdagger-whatistheretoshowquestionsmark}.

Fix arbitrary $q\in\NN_{\geq3}$ and $N\in\NN$.
In the proof of Proposition~\ref{prop:galois-cohomology-of-BdRdaggerplus-primitiverootofunity},
we apply the $\eta$-operator as in Definition~\ref{defn:algebraic-decalage}
to specific cochain complexes of Banach modules. In this technical
\S\ref{subsubsec:galois-cohomology-of-BdRdaggerplus-primitiverootofunity4},
we check that these applications are well-behaved. The main results are
Corollary~\ref{cor:galois-cohomology-of-BdRdaggerplus-primitiverootofunity-aaaaaletaoperatorconditions-corollary}
and Lemma~\ref{lem:cohomologyofBdRdaggerplus-etagr-is-greta-forvarthetagreaterthanq-reconstructionpaper}.

\begin{lem}\label{lem:Dgreaterqisatisfiesconditionsformetaoperator-periodringcohomology-reconstructionpaper}
  For all $i\in\NN$, $p^{N}D^{>q,i}\subseteq D^{>q,i}$
  is a closed subset.
\end{lem}

\begin{proof}
  In degree $i=0$, we have the equality 
  \begin{equation}\label{eq:Dgreaterqisatisfiesconditionsformetaoperator-periodringcohomology-someequality-reconstructionpaper}
    p^{N}D^{>q,0}
    =
    p^{N}\A_{\dR}^{>q}\left(*\right)
    \times_{\delta^{>q,0} , p^{N}\A_{\dR}^{>q}\left(\cal{G} \times *\right) }
    p^{N}\A_{\dR}^{>q-1}\left(\cal{G} \times *\right).
  \end{equation}
  The inclusion $\subseteq$ is clear. To prove $\supseteq$,
  fix $a\in\A_{\dR}^{>q}\left(*\right)$ and
  $b\in\A_{\dR}^{>q-1}\left(\cal{G} \times *\right)$ such that
  \begin{equation*}
    \left(p^{N}a,p^{N}b\right)\in
    p^{N}\A_{\dR}^{>q}\left(*\right)
    \times_{\delta^{>q,0} , p^{N}\A_{\dR}^{>q}\left(\cal{G} \times *\right) }
    p^{N}\A_{\dR}^{>q-1}\left(\cal{G} \times *\right).
  \end{equation*}
  That is $\delta^{>q,0}\left(p^{N}a\right)=\iota^{>q}\left(p^{N}b\right)$,
  where $\iota^{>q}\colon\A_{\dR}^{>q-1}\left(\cal{G} \times *\right)\to\A_{\dR}^{>q}\left(\cal{G} \times *\right)$
  denotes the canonical map. But then Lemma~\ref{lem:AdRgreaterthanqptorsionfree}
  applies,
  giving $\delta^{>q,0}\left(a\right)=\iota^{>q}\left(b\right)$.
  That is $(a,b)\in D^{>q,0}$, thus $\left(p^{N}a,p^{N}b\right)\in p^{N}D^{>q,0}$. This proves $\supseteq$
  and~(\ref{eq:Dgreaterqisatisfiesconditionsformetaoperator-periodringcohomology-someequality-reconstructionpaper})
  follows.
  
  Now apply Lemma~\ref{lem:limpreservestrictmonomorphisms}
  and~\ref{lem:onAdR-multbyp-strictinjection-reconstructionpaper} to find that
  \begin{multline*}
    p^{N}D^{>q,0}
    \stackrel{\text{(\ref{eq:Dgreaterqisatisfiesconditionsformetaoperator-periodringcohomology-someequality-reconstructionpaper})}}{=}
    p^{N}\A_{\dR}^{>q}\left(*\right)
    \times_{\delta^{>q,0} , p^{N}\A_{\dR}^{>q}\left(\cal{G} \times *\right) }
    p^{N}\A_{\dR}^{>q-1}\left(\cal{G} \times *\right) \\
    \to
    \A_{\dR}^{>q}\left(*\right)
    \times_{\delta^{>q,0} , \A_{\dR}^{>q}\left(\cal{G} \times *\right) }
    \A_{\dR}^{>q-1}\left(\cal{G} \times *\right)
    =D^{>q,0}
  \end{multline*}
  is a strict monomorphism. In particular,
  $p^{N}D^{>q,0}\subseteq D^{>q,0}$ is closed.
  
  Next, let $i>0$. Then we apply Lemma~\ref{lem:onAdR-multbyp-strictinjection-reconstructionpaper} again
  to find that
  \begin{equation*}
    p^{N}D^{>q,i}=p^{N}\A_{\dR}^{>q}\left(\cal{G}^{i}\times*\right)
    \to\A_{\dR}^{>q}\left(\cal{G}^{i}\times*\right)=D^{>q,i}
  \end{equation*}
  is a strict monomorphism.
  Lemma~\ref{lem:Dgreaterqisatisfiesconditionsformetaoperator-periodringcohomology-reconstructionpaper}
  follows.
\end{proof}

\begin{lem}\label{lem:Dgreaterqtildeisatisfiesconditionsformetaoperator-periodringcohomology-reconstructionpaper}
  For all $i\in\NN$, $p^{N}\widetilde{D}^{>q,i}\subseteq\widetilde{D}^{>q,i}$
  is a closed subset.
\end{lem}

\begin{proof}
  We refer the reader to the proof of
  Lemma~\ref{lem:Dgreaterqisatisfiesconditionsformetaoperator-periodringcohomology-reconstructionpaper}
  for~(\ref{eq:Dgreaterqisatisfiesconditionsformetaoperator-periodringcohomology-someequality-reconstructionpaper}).
  \begin{multline*}
    p^{N}\widetilde{D}^{>q,0}
    \stackrel{\text{(\ref{eq:Dgreaterqisatisfiesconditionsformetaoperator-periodringcohomology-someequality-reconstructionpaper})}}{=}
    p^{N}\widetilde{\A}_{\dR}^{>q}\left(*\right)
    \times_{\delta^{>q,0} , p^{N}\widetilde{\A}_{\dR}^{>q}\left(\cal{G} \times *\right) }
    p^{N}\widetilde{\A}_{\dR}^{>q-1}\left(\cal{G} \times *\right) \\
    \to
    \widetilde{\A}_{\dR}^{>q}\left(*\right)
    \times_{\delta^{>q,0} , \A_{\dR}^{>q}\left(\cal{G} \times *\right) }
    \widetilde{\A}_{\dR}^{>q-1}\left(\cal{G} \times *\right)
    =\widetilde{D}^{>q,0}
  \end{multline*}
  is a strict monomorphism by Lemma~\ref{lem:limpreservestrictmonomorphisms}
  and~\ref{lem:galois-cohomology-of-Bla-pidealclosed-inAlaq}. In particular,
  $p^{N}\widetilde{D}^{>q,0}\subseteq \widetilde{D}^{>q,0}$ is closed.
  
  Next, let $i>0$. Then we apply Lemma~\ref{lem:onAdR-multbyp-strictinjection-reconstructionpaper} again
  to find that
  \begin{equation*}
    p^{N}\widetilde{D}^{>q,i}=p^{N}\widetilde{\A}_{\dR}^{>q}\left(\cal{G}^{i}\times*\right)
    \to\widetilde{\A}_{\dR}^{>q}\left(\cal{G}^{i}\times*\right)=\widetilde{D}^{>q,i}
  \end{equation*}
  is a strict monomorphism.
  Lemma~\ref{lem:Dgreaterqtildeisatisfiesconditionsformetaoperator-periodringcohomology-reconstructionpaper}
  follows.
\end{proof}

\begin{lem}\label{lem:galois-cohomology-of-BdRdaggerplus-primitiverootofunity-aaaaaletaoperatorconditions}
  Condition~\ref{cond:Banachmoduledecalage-reconstructionpaper} is satisfied for
  \begin{itemize}
    \item
      $R=W(\kappa)$, $r=p^{N}$, and $M^{\bullet}=D^{>q,\bullet}$, as well as
    \item
      $R=W(\kappa)$, $r=p^{N}$, and $M^{\bullet}=\widehat{\cal{O}}^{+}\left(\cal{G}^{\bullet}\times*\right)$,
    \iffalse %%% COMMENT BEGINNS
    \item 
      $R=W(\kappa)^{\triv}$, $r=p^{N}$, and $M^{\bullet}=\widetilde{D}^{>q,\bullet}$, as well as
    \item
      $R=W(\kappa)^{\triv}$, $r=p^{N}$, and $M^{\bullet}=\widehat{\cal{O}}^{+}\left(\cal{G}^{\bullet}\times*\right)^{\triv}$.
    \fi %%% COMMENT ENDS
  \end{itemize}
\end{lem}

\begin{proof}
  Firstly, we check the conditions for
  $R=W(\kappa)$, $r=p^{N}$, and $M^{\bullet}=D^{>q,\bullet}$:
  \begin{itemize}
    \item[(i)] $D^{>q,\bullet}$ is indeed concentrated in non-negative degrees.
    \item[(ii)] This is Lemma~\ref{lem:Dgreaterqisatisfiesconditionsformetaoperator-periodringcohomology-reconstructionpaper}.
    \item[(iii)] For every $i\in\NN$, $D^{>q,i}$ does not have $p^{N}$-torsion
      by Lemma~\ref{lem:AdRgreaterthanqptorsionfree}.
  \end{itemize}
  Secondly, we check the conditions for
  $R=W(\kappa)$, $r=p^{N}$, and $M^{\bullet}=\widehat{\cal{O}}^{+}\left(\cal{G}^{\bullet}\times*\right)$:
  \begin{itemize}
    \item[(i)] $\widehat{\cal{O}}^{+}\left(\cal{G}^{\bullet}\times*\right)$
      is indeed concentrated in non-negative degrees.
    \item[(ii)] This follows since every $\widehat{\cal{O}}^{+}\left(\cal{G}^{i}\times*\right)$ carries the $p$-adic topology.
    \item[(iii)] $\widehat{\cal{O}}^{+}\left(\cal{G}^{i}\times*\right)$
      does not have $p^{N}$-torsion for every $i\in\NN$.
  \end{itemize}
  \iffalse %%% COMMENT BEGINNS
  Thirdly, we check the conditions for
  $R=W(\kappa)^{\triv}$, $r=p^{N}$, and $M^{\bullet}=\widetilde{D}^{>q,\bullet}$:
  \begin{itemize}
    \item[(i)] $\widetilde{D}^{>q,\bullet}$ is indeed concentrated in non-negative degrees.
    \item[(ii)] This is Lemma~\ref{lem:Dgreaterqtildeisatisfiesconditionsformetaoperator-periodringcohomology-reconstructionpaper}, which requires $q\geq2$.
    \item[(iii)] For every $i\in\NN$, $\widetilde{D}^{>q,i}$ does not have
    $p^{N}$-torsion by Lemma~\ref{lem:AdRgreaterthanqptorsionfree} if $q\geq 2$.
  \end{itemize}
  Finally, we check the conditions for
  $R=W(\kappa)^{\triv}$, $r=p^{N}$, and $M^{\bullet}=\widehat{\cal{O}}^{+}\left(\cal{G}^{\bullet}\times*\right)^{\triv}$:
  \begin{itemize}
    \item[(i)] $\widehat{\cal{O}}^{+}\left(\cal{G}^{\bullet}\times*\right)^{\triv}$
      is indeed concentrated in non-negative degrees.
    \item[(ii)] This follows since every subset of any $\widehat{\cal{O}}^{+}\left(\cal{G}^{i}\times*\right)^{\triv}$
    is closed.
    \item[(iii)] $\widehat{\cal{O}}^{+}\left(\cal{G}^{i}\times*\right)^{\triv}$
      does not have $p^{N}$-torsion for every $i\in\NN$.
  \end{itemize}
  This proves Lemma~\ref{lem:galois-cohomology-of-BdRdaggerplus-primitiverootofunity-aaaaaletaoperatorconditions}.
  \fi %%% COMMENT ENDS
\end{proof}

\begin{cor}\label{cor:galois-cohomology-of-BdRdaggerplus-primitiverootofunity-aaaaaletaoperatorconditions-corollary}
  %Let $N$ be arbitrary and $q\geq 2$.
  The subcomplexes
  $\eta_{p^{N}} D^{>q,\bullet}\subseteq D^{>q,\bullet}$ and
  $\eta_{p^{N}}\widehat{\cal{O}}^{+}\left(\cal{G}^{\bullet}\times*\right)\subseteq\widehat{\cal{O}}^{+}\left(\cal{G}^{\bullet}\times*\right)$,
  equipped with the induced norms,
  are complexes $W(\kappa)$-Banach algebras.
  %The subcomplexes
  %$\eta_{p^{N}} \widetilde{D}^{>q,\bullet}\subseteq\widetilde{D}^{>q,\bullet}$
  %and $\eta_{p^{N}}\widehat{\cal{O}}^{+}\left(\cal{G}^{\bullet}\times*\right)^{\triv}\subseteq\widehat{\cal{O}}^{+}\left(\cal{G}^{\bullet}\times*\right)^{\triv}$,
  %equipped with the induced norms,
  %are complexes $W(\kappa)^{\triv}$-Banach algebras.
\end{cor}

\begin{proof}
  Thanks to Lemma~\ref{lem:galois-cohomology-of-BdRdaggerplus-primitiverootofunity-aaaaaletaoperatorconditions},
  this follows from Lemma~\ref{lem:decalage-defined-reconstructionpaper}.
\end{proof}

\begin{lem}\label{lem:galois-cohomology-of-BdRdaggerplus-primitiverootofunity-aaaaaletaoperatorconditionsfiltered}
  Condition~\ref{cond:filteredmodulesdecalage-reconstructionpaper} is satisfied for
  \begin{itemize}
    \item 
      $R=W(\kappa)^{\triv}$, $r=p^{N}$, and $M^{\bullet}=\widetilde{D}^{>q,\bullet}$, as well as
    \item
      $R=W(\kappa)^{\triv}$, $r=p^{N}$, and $M^{\bullet}=\widehat{\cal{O}}^{+}\left(\cal{G}^{\bullet}\times*\right)^{\triv}$.
  \end{itemize}
\end{lem}

\begin{proof}
  Firstly, we check the conditions for $R=W(\kappa)^{\triv}$, $r=p^{N}$, and
  $M^{\bullet}=\widetilde{D}^{>q,\bullet}$:
  \begin{itemize}
    \item[(i)] The cochain complex is indeed concentrated in nonnegative degrees.
    \item[(ii)] This is Lemma~\ref{lem:Dgreaterqtildeisatisfiesconditionsformetaoperator-periodringcohomology-reconstructionpaper}.
    \item[(iii)] See Lemma~\ref{lem:AdRgreaterthanqptorsionfree}.
    \item[(iv)] The description~(\ref{eq:galois-cohomology-of-BdRdagger-Dgreaterthanqbulletcomputedgrn-needthislater})
      of $\gr\widetilde{D}^{>q,\bullet}$ implies that it is $p^{i}$-torsion free in every degree $i\in\NN$.
  \end{itemize}
  Secondly, we check the conditions for $R=W(\kappa)^{\triv}$, $r=p^{N}$, and
  $M^{\bullet}=\widehat{\cal{O}}^{+}\left(\cal{G}^{\bullet}\times*\right)^{\triv}$:
  \begin{itemize}
    \item[(i)] The cochain complex is indeed concentrated in nonnegative degrees.
    \item[(ii)] This follows since every subset is closed.
    \item[(iii)] The $\widehat{\cal{O}}^{+}\left(\cal{G}^{i}\times*\right)^{\triv}$ are clearly $p$-torsion free.
    \item[(iv)] This follows from (iii), because
      $\gr \widehat{\cal{O}}^{+}\left(\cal{G}^{i}\times*\right)^{\triv}\cong\widehat{\cal{O}}^{+}\left(\cal{G}^{i}\times*\right)$
      for all $i\in\NN$.
  \end{itemize}
\end{proof}

\begin{cor}\label{cor:galois-cohomology-of-BdRdaggerplus-primitiverootofunity-aaaaaletaoperatorconditionsfiltered-corollary}
  %Let $N$ be arbitrary and $q\geq 2$.
  Consider the induced filtrations on the subcomplexes
  $\eta_{p^{N}} \widetilde{D}^{>q,\bullet}\subseteq \widetilde{D}^{>q,\bullet}$ and
  $\eta_{p^{N}}\widehat{\cal{O}}^{+}\left(\cal{G}^{\bullet}\times*\right)^{\triv}
  \subseteq\widehat{\cal{O}}^{+}\left(\cal{G}^{\bullet}\times*\right)^{\triv}$.
  These are exhaustive, separated, and complete.
\end{cor}

\begin{proof}
  Thanks to Lemma~\ref{lem:galois-cohomology-of-BdRdaggerplus-primitiverootofunity-aaaaaletaoperatorconditionsfiltered},
  this follows from Lemma~\ref{lem:decalagefilteredmodules-defined}.
\end{proof}

\begin{lem}\label{lem:cohomologyofBdRdaggerplus-etagr-is-greta-forvarthetagreaterthanq-reconstructionpaper}
  %Let $N\in\NN$ be arbitrary.
  We have the isomorphism of graded complexes
  \begin{equation*}
    \eta_{p^{N}}\gr\widetilde{\vartheta}^{>q,\bullet}\cong\gr\eta_{p^{N}}\widetilde{\vartheta}^{>q,\bullet}.
  \end{equation*}
\end{lem}

\begin{proof}
  Thanks to Lemma~\ref{lem:galois-cohomology-of-BdRdaggerplus-primitiverootofunity-aaaaaletaoperatorconditionsfiltered},
  this follows from Lemma~\ref{lem:decalage-commutes-gr}.
\end{proof}

%%%%%%%%%%%%%%%%%%%%%%%%%%%%%%%%%%%%%%%%%%%%%%%%%%%%%%%%%%%%
% From $\widetilde{D}^{>q,\bullet}$ to $D^{>q,\bullet}$
%%%%%%%%%%%%%%%%%%%%%%%%%%%%%%%%%%%%%%%%%%%%%%%%%%%%%%%%%%%%

\subsection{From $\widetilde{D}^{>q,\bullet}$ to $D^{>q,\bullet}$}\label{subsubsec:galois-cohomology-of-BdRdaggerplus-primitiverootofunity5}

Fix $q\in\NN_{\geq3}$.
So far, our goal was to study the cohomology of the complex $D^{>q,\bullet}$.
Instead, we chose to work with $\widetilde{D}^{>q,\bullet}$, as its associated graded
is more accessible. In this \S, we explain how to
recover $D^{>q,\bullet}$ from $\widetilde{D}^{>q,\bullet}$, therefore deducing
information about the former complex from the latter. This relies
on the operator $\p(-)$ introduced in Definition~\ref{defn:poperator}.

\begin{lem}\label{lem:poperator-commutes-with-pullbacks}
  Given a diagram $M\to T\leftarrow N$ of filtered groups $M$,
  \begin{equation*}
    \p\left( M \times_{T} N \right) = \p\left(M\right) \times_{p\left(T\right)} \p\left(N\right).
  \end{equation*}
\end{lem}

\begin{proof}
  Topological groups share the same underlying abstract groups, thus it remains to show that
  their topologies coincide. By definition, $\p\left( M \times_{T} N \right)$ has the open neighbourhood basis
  \begin{equation*}
    \left\{ p^{n}\left( M \times_{T} N\right) + \left(\Fil^{n}M\times_{T}\Fil^{n}M\right) \right\}_{n\geq0}.
  \end{equation*}
  This system of subsets coincides with the set
  \begin{equation*}
    \left\{ \left( p^{n}M + \Fil^{n}M \right) \times_{T} \left( p^{n}N + \Fil^{n}N \right) \right\}_{n\geq0}.
  \end{equation*}
  But this is an open neighborhood basis for $\p\left(M\right) \times_{p\left(T\right)} \p\left(N\right)$,
  as desired.
  %Thus the topologies on
  %$\p\left( M \times_{T} N \right)$ and $\p\left(M\right) \times_{p\left(T\right)} \p\left(N\right)$
  %coincide.
\end{proof}

In the following, we confuse $W(\kappa)$-Banach modules with their underlying topological groups.

\begin{lem}\label{lem:poperator-D-to-tildeD}
  $D^{>q,\bullet}=\p\left(\widetilde{D}^{>q,\bullet}\right)$.
\end{lem}

\begin{proof}
  This follows from Lemma~\ref{lem:poperator-commutes-with-pullbacks}
  and the observation
  \begin{equation}\label{eq:pwidetildeAdRisAdR-reconstructionpaper}
    \p\left(\widetilde{\A}_{\dR}^{>h}\left(\cal{G}^{i}\times*\right)\right) = \A_{\dR}^{>h}\left(\cal{G}^{i}\times*\right)
  \end{equation}
  for all $i\in\NN$ and $h\in\left\{q,q-1\right\}$,
  see also Lemma~\ref{lem:pwidetildeAdRisAdR-reconstructionpaper}.
\end{proof}

\begin{lem}\label{lem:galois-cohomology-of-BdRdaggerplus-primitiverootofunity-morelemma2-reconstructionpaper}
  $\cone\left(\eta_{p^{N}}\vartheta^{>q,\bullet}\right)=\p\left(\cone\left(\eta_{p^{N}}\widetilde{\vartheta}^{>q,\bullet}\right)\right)$
  for arbitrary $N\in\NN$.
%  \hfill
%  \begin{itemize}
%    \item[(i)] $D^{>q,\bullet}=\p\left(\widetilde{D}^{>q,\bullet}\right)$,
%    \item[(ii)] $\widehat{\cal{O}}^{+}\left(\cal{G}^{\bullet}\times*\right)=\p\left(\widehat{\cal{O}}^{+}\left(\cal{G}^{\bullet}\times*\right)^{\triv}\right)$, and
%    \item[(iii)] $\cone\left(\eta_{p^{N}}\vartheta^{>q,\bullet}\right)=\p\left(\cone\left(\eta_{p^{N}}\widetilde{\vartheta}^{>q,\bullet}\right)\right)$.
%  \end{itemize}
  \iffalse %%%
  Given the filtered complex $\widetilde{C}^{\bullet}:=\cone\left(\eta_{p^{N}}\widetilde{\vartheta}^{>q,\bullet}\right)$,
  we consider the complex of topological groups $C^{\bullet}$ as follows. Both complexes coincide
  as complexes of abstract groups. For every $i\in\ZZ$, $C^{i}$ carries the topology with respect to
  the open neighbourhood basis
  \begin{equation*}
    \left\{ p^{n}C^{i} + \Fil^{n}\right\}_{s\geq0}.
  \end{equation*}
  Here $\Fil^{n}\subseteq \widetilde{C}^{i}$ denotes the $n$th piece in the filtration.
  Then $C^{\bullet}$ coincides with the complex (of topological groups underlying the complex)
  of $W(\kappa)$-Banach modules
  $\cone\left(\eta_{p^{N}}\vartheta^{>q,\bullet}\right)$.
  \fi %%%
\end{lem}

\begin{proof}
  It suffices to check
  \begin{equation*}
  \begin{split}
    \eta_{p^{N}}D^{>q,\bullet}
    &=\p\left(\eta_{p^{N}}\widetilde{D}^{>q,\bullet}\right)
    \text{ and } \\
    \eta_{p^{N}}\widehat{\cal{O}}^{+}\left(\cal{G}^{\bullet}\times*\right)
    &=\p\left(\eta_{p^{N}}\widehat{\cal{O}}^{+}\left(\cal{G}^{\bullet}\times*\right)^{\triv}\right).
  \end{split}
  \end{equation*}
  The first equality follows from Lemma~\ref{lem:poperator-D-to-tildeD}
  and Definition~\ref{defn:decalage-for-Banachmodules}. The second
  one follows since every $\widehat{\cal{O}}^{+}\left(\cal{G}^{i}\times*\right)$ carries the $p$-adic topology
  and, again, by Definition~\ref{defn:decalage-for-Banachmodules}.
  \iffalse %%% 
  (i) follows from Lemma~\ref{lem:poperator-commutes-with-pullbacks}
  and the observation
  \begin{equation}\label{eq:pwidetildeAdRisAdR-reconstructionpaper}
    \p\left(\widetilde{\A}_{\dR}^{>h}\left(\cal{G}^{i}\times*\right)\right) = \A_{\dR}^{>h}\left(\cal{G}^{i}\times*\right)
  \end{equation}
  for all $i\in\NN$ and $h\in\left\{q,q-1\right\}$.
  (ii) holds, since $\widehat{\cal{O}}^{+}\left(\cal{G}^{i}\times*\right)$ carries the $p$-adic topology.
  Finally, to prove (iii), it suffices to check
  \begin{equation*}
  \begin{split}
    \eta_{p^{N}}D^{>q,\bullet}
    &=\p\left(\eta_{p^{N}}\widetilde{D}^{>q,\bullet}\right)
    \text{ and } \\
    \eta_{p^{N}}\widehat{\cal{O}}^{+}\left(\cal{G}^{\bullet}\times*\right)
    &=\p\left(\eta_{p^{N}}\widehat{\cal{O}}^{+}\left(\cal{G}^{\bullet}\times*\right)^{\triv}\right).
  \end{split}
  \end{equation*}
  But this follows directly from (i) and (ii), as well as Definition~\ref{defn:decalage-for-Banachmodules}. \fi
\end{proof}

%%%%%%%%%%%%%%%%%%%%%%%%%%%%%%%%%%%%%%%%%%%%%%%%%%%%%%%%%%%%
% From $\widetilde{D}^{>q,\bullet}$ to $D^{>q,\bullet}$
%%%%%%%%%%%%%%%%%%%%%%%%%%%%%%%%%%%%%%%%%%%%%%%%%%%%%%%%%%%%

\subsection{Proof of Proposition~\ref{prop:galois-cohomology-of-BdRdaggerplus-primitiverootofunity}}\label{subsubsec:galois-cohomology-of-BdRdaggerplus-primitiverootofunity6}

%We prove the following technical result before we can finally give the proof of
%Proposition~\ref{prop:galois-cohomology-of-BdRdaggerplus-primitiverootofunity}.

\begin{lem}\label{lem:galois-cohomology-of-BdRdaggerplus-primitiverootofunity-morelemma3-reconstructionpaper}
  Given $q\in\NN_{\geq3}$, $N\in\NN$, and
  the filtered complex $\widetilde{C}^{\bullet}:=\cone\left(\eta_{p^{N}}\widetilde{\vartheta}^{>q,\bullet}\right)$,
  we observe that $\widetilde{C}^{i}/\Fil^{n}$ does not have any $p$-power torsion for all $i,n\in\ZZ$.
  Here, $\Fil^{n}\subseteq \widetilde{C}^{i}$ denotes the $n$th piece in the filtration on $\widetilde{C}^{i}$.
\end{lem}

\begin{proof}
  This follows from Lemma~\ref{lem:AdRgreaterthanqUtimesS-isomapHomcontSAdRgreaterthanqU-assumptionsforlemmasatisfied-reconstructionpaper}
  and because $\widetilde{\cal{O}}^{+}\left(\cal{G}^{\bullet}\times*\right)^{\triv}$ carries the trivial filtration.
\end{proof}

\begin{proof}[Proof of Proposition~\ref{prop:galois-cohomology-of-BdRdaggerplus-primitiverootofunity}]
  Fix $N\in\NN$ as in Proposition~\ref{prop:galois-cohomology-of-BdRdagger-whatistheretoshowquestionsmark}.
  \emph{Loc. cit.} implies that
  \begin{equation*}
    \cone\left(\eta_{p^{N}}\gr\widetilde{\vartheta}^{>q,\bullet}\right)
    \stackrel{\text{\ref{lem:cohomologyofBdRdaggerplus-etagr-is-greta-forvarthetagreaterthanq-reconstructionpaper}}}{\cong}
    \cone\left(\gr\eta_{p^{N}}\widetilde{\vartheta}^{>q,\bullet}\right)
    =\gr\cone\left(\eta_{p^{N}}\widetilde{\vartheta}^{>q,\bullet}\right)
  \end{equation*}
  is acyclic. It follows from~\cite[Chapter I, \S 4.1, page 31-32, Theorem 4]{HuishiOystaeyen1996}
  that $\cone\left(\eta_{p^{N}}\widetilde{\vartheta}^{>q,\bullet}\right)$
  is strictly exact as a filtered complex;
  note that \emph{loc. cit.} applies because 
  $\cone\left(\eta_{p^{N}}\widetilde{\vartheta}^{>q,\bullet}\right)$ is
  separated and complete, which in turn follows from
  Corollary~\ref{cor:galois-cohomology-of-BdRdaggerplus-primitiverootofunity-aaaaaletaoperatorconditionsfiltered-corollary}.
  Now Lemma~\ref{lem:subsections-periodsheaves-affperfd-1}
  applies to $\cone\left(\eta_{p^{N}}\widetilde{\vartheta}^{>q,\bullet}\right)$
  thanks to Lemma~\ref{lem:galois-cohomology-of-BdRdaggerplus-primitiverootofunity-morelemma3-reconstructionpaper}.
  With
  Lemma~\ref{lem:galois-cohomology-of-BdRdaggerplus-primitiverootofunity-morelemma2-reconstructionpaper},
  it follows that that $\cone\left(\eta_{p^{N}}\vartheta^{>q,\bullet}\right)$
  is strictly exact as a complex of topological groups. Consequently, it
  is strictly exact as a complex of $W(\kappa)$-Banach modules.
  Because both $\eta_{p^{N}}D^{>q,\bullet}$ and $\widehat{\cal{O}}^{+}\left(\cal{G}^{\bullet}\times*\right)$
  are complexes of $p$-torsion free $W(\kappa)$-Banach modules,
  cf. Lemma~\ref{lem:AdRgreaterthanqptorsionfree} and
  Corollary~\ref{cor:galois-cohomology-of-BdRdaggerplus-primitiverootofunity-aaaaaletaoperatorconditions-corollary},
  Corollary~\ref{cor:completed-localisation-strictlyexact} applies,
  giving that $\cone\left(\eta_{p^{N}}\vartheta^{>q,\bullet}\right)\widehat{\otimes}_{W(\kappa)}k_{0}$
  is strictly exact. Now apply
  Lemma~\ref{lem:quasiiso-if-cone-strictlyexact-reconstructionpaper} to find that
  \begin{equation*}
    \eta_{p^{N}}\vartheta^{>q,\bullet}\widehat{\otimes}_{W(\kappa)}\id_{k_{0}}
    \colon\eta_{p^{N}}D^{>q,\bullet}\widehat{\otimes}_{W(\kappa)}k_{0}
    \to\eta_{p^{N}}\widehat{\cal{O}}^{+}\left(\cal{G}^{\bullet}\times*\right)\widehat{\otimes}_{W(\kappa)}k_{0}
  \end{equation*}
  is a quasi-isomorphism. The right-hand side is isomorphic to
  $\widehat{\cal{O}}\left(\cal{G}^{\bullet}\times*\right)$,
  cf. Lemma~\ref{lem:decalage-completed-localisation-is-just-completed-localisation},
  which applies by
  Corollary~\ref{lem:galois-cohomology-of-BdRdaggerplus-primitiverootofunity-aaaaaletaoperatorconditions}.
  Now pass to the colimit along $q\to\infty$.
  By Corollary~\ref{cor:filteredcol-inIndBan-stronglyexact},
  \begin{equation*}
    \text{``}\varinjlim_{q\in\NN_{\geq3}}\text{"}
    \eta_{p^{N}}\vartheta^{>q,\bullet}\widehat{\otimes}_{W(\kappa)}\id_{k_{0}}
    \colon\text{``}\varinjlim_{q\in\NN_{\geq3}}\text{"}\eta_{p^{N}}D^{>q,\bullet}\widehat{\otimes}_{W(\kappa)}k_{0}
    \to\widehat{\cal{O}}\left(\cal{G}^{\bullet}\times*\right)
  \end{equation*}
  is again a quasi-isomorphism. But the domain of this morphism is
  \begin{equation*}
    \text{``}\varinjlim_{q\in\NN_{\geq3}}\text{"}\eta_{p^{N}}D^{>q,\bullet}\widehat{\otimes}_{W(\kappa)}k_{0}
    \stackrel{\text{\ref{lem:decalage-completed-localisation-is-just-completed-localisation}}}{\cong}
    \text{``}\varinjlim_{q\in\NN_{\geq3}}\text{"}D^{>q,\bullet}\widehat{\otimes}_{W(\kappa)}k_{0}
    \stackrel{\text{\ref{lem:galois-cohomology-of-BdRdagger-Dgreaterthanqbullet-isotoBdRdagplusGbulletstar}}}{\cong}
    \BB_{\dR}^{\dag,+}\left(\cal{G}^{\bullet}\times*\right),
  \end{equation*}
  where Lemma~\ref{lem:decalage-completed-localisation-is-just-completed-localisation} applies
  because of Lemma~\ref{lem:galois-cohomology-of-BdRdaggerplus-primitiverootofunity-aaaaaletaoperatorconditions}.
  We have thus found
  that~(\ref{eq:galois-cohomology-of-BdRdaggerplus-primitiverootofunity-whatistheretoshow})
  is a quasi-isomorphism. %of cochain complexes of
  %$k_{0}$-ind-Banach spaces.
  Lemma~\ref{lem:galois-cohomology-of-BdRdaggerplus-primitiverootofunity-whatistheretoshow}
  implies Proposition~\ref{prop:galois-cohomology-of-BdRdaggerplus-primitiverootofunity}.
\end{proof}

%%%%%%%%%%%%%%%%%%%%%%%%%%%%%%%%%%%%%%%%%%%%%%%%%%%%%%%%%%%%
% From $\widetilde{D}^{>q,\bullet}$ to $D^{>q,\bullet}$
%%%%%%%%%%%%%%%%%%%%%%%%%%%%%%%%%%%%%%%%%%%%%%%%%%%%%%%%%%%%

\subsection{Proof of Theorem~\ref{thm:galois-cohomology-of-solidBdRdaggerplus-born}}\label{subsubsec:proofofthm:galois-cohomology-of-solidBdRdaggerplus-born}

Recall the computation of the continuous Galois cohomology of $\underline{C}$ as in
Proposition~\ref{prop:Galois-cohomology-Banach-C-solid-reconstructionpaper}.

\begin{cor}\label{cor:galois-cohomology-of-solidBdRdaggerplus-primitiverootofunity}
  Fontaine's $B_{\dR}^{\dag,+}\to C$ induces the isomorphism
  \begin{equation*}
    \Ho_{\cont}^{i}\left(\Gal\left( \overline{k} / k^{\prime} \right) , \underline{B}_{\dR}^{\dag,+} \right)
    \cong
    \begin{cases}
      \underline{k}^{\prime}, &\text{ if $i=0,1$ and} \\
      0, &\text{otherwise}.
    \end{cases}
  \end{equation*}
  Here, $k^{\prime}$ denotes the fixed extension of $k$ as in
  Notation~\ref{notation:galois-cohomology-of-BdRdaggerplus-kzero}.
\end{cor}

\begin{proof}
  Apply Lemma~\ref{lem:contgpcoh-indban-vs-solid-reconstructionpaper}
  to Proposition~\ref{prop:galois-cohomology-of-BdRdaggerplus-primitiverootofunity}.
\end{proof}

\begin{thm}\label{thm:galois-cohomology-of-solidBdRdaggerplus}
  Fontaine's $B_{\dR}^{\dag,+}\to C$ induces the isomorphism
  \begin{equation*}
    \Ho_{\cont}^{i}\left(\Gal\left( \overline{k} / k \right) , \underline{B}_{\dR}^{\dag,+} \right)
    \cong
    \begin{cases}
      \underline{k}, &\text{ if $i=0,1$ and} \\
      0, &\text{otherwise}.
    \end{cases}
  \end{equation*}
\end{thm}

\begin{proof}%[Proof of Theorem~\ref{thm:galois-cohomology-of-solidBdRdaggerplus}]\label{proof:thm:galois-cohomology-of-BdRdaggerplus}
  Fix the intermediate extension $\overline{k} / k^{\prime} / k$
  as in Notation~\ref{notation:galois-cohomology-of-BdRdaggerplus-kzero}
  and compute
  \begin{align*}
    \Ho^{0}\left(
      \Gal\left( \overline{k} / k\right) , B_{\dR}^{\dag,+}
    \right)
    &=\Ho^{0}\left(
      \Gal\left( k^{\prime} / k\right),
      \Ho^{0}\left(
        \Gal\left( \overline{k} / k^{\prime}\right) , B_{\dR}^{\dag,+}
      \right)\right) \\
    &\stackrel{\text{\ref{prop:galois-cohomology-of-BdRdaggerplus-primitiverootofunity}}}{\cong}
    \Ho^{0}\left(
      \Gal\left( k^{\prime} / k\right), k^{\prime}\right)
    \stackrel{\text{\ref{prop:Galois-cohomology-Banach-k0-finiteextensionofk}}}{\cong} k.
  \end{align*}
  Together with Lemma~\ref{lem:contgpcoh-indban-vs-solid-reconstructionpaper}, this
  gives the desired computation of the zeroth continuous cohomology.
  To compute the higher cohomology, we consider thes spectral sequence
  as in Lemma~\ref{lem:solidHochschild-Serre-reconstructiontheorem}:
  \begin{equation*}
    \Ho_{\cont}^{i}\left(
      \Gal\left( k^{\prime} / k\right) ,
      \Ho_{\cont}^{j}\left( \Gal\left(\overline{k}/k^{\prime}\right) , \underline{B}_{\dR}^{\dag,+}\right)
    \right)
    \implies 
    \Ho_{\cont}^{i+j}\left(
      \Gal\left( \overline{k} / k\right) , \underline{B}_{\dR}^{\dag,+}
    \right).
  \end{equation*}
  Corollary~\ref{cor:Galois-cohomology-Banach-k0-finiteextensionofk-solid-reconstructionpaper}
  and~\ref{cor:galois-cohomology-of-solidBdRdaggerplus-primitiverootofunity}
  imply the vanishing of
  $\Ho_{\cont}^{n}\left(
      \Gal\left( \overline{k} / k\right) , \underline{B}_{\dR}^{\dag,+}
    \right)$
  for $n\geq 2$.
  Furthermore, \emph{loc. cit.} gives an explicit description of the
  five-term exact sequence associated to the spectral sequence:
  \begin{equation*}
    0
    \to
    0
    \to
    \Ho_{\cont}^{1}\left(
      \Gal\left( \overline{k} / k\right) , \underline{B}_{\dR}^{\dag,+}
    \right)
    \to
    \underline{k}
    \to
    0.
  \end{equation*}
  We have thus found the desired isomorphism
  $\Ho_{\cont}^{1}\left(\Gal\left( \overline{k} / k \right) , \underline{B}_{\dR}^{\dag,+} \right)\cong\underline{k}$.
\end{proof}

\begin{proof}[Proof of Theorem~\ref{thm:galois-cohomology-of-solidBdRdaggerplus-born}]\label{proof:thm:galois-cohomology-of-BdRdaggerplus-born}
  Recall Definition~\ref{defn:indBanachmodule-Ccontcomplex-reconstructionpaper}.
  We have to check that
  \begin{equation*}
    C_{\cont}^{\bullet}\left(\Gal\left( \overline{k} / k\right),B_{\dR}^{\dag,+}\right)
    \to C_{\cont}^{\bullet}\left(\Gal\left( \overline{k} / k\right),C\right)
  \end{equation*}
  is a quasi-isomorphism of cochain complexes of $k_{0}$-ind-Banach spaces.
  By Lemma~\ref{lem:quasiiso-if-cone-strictlyexact-reconstructionpaper},
  this is equivalent to checking that its mapping cone $K^{\bullet}$ is strictly exact.
  We would like to use that
  \begin{equation*}
    \underline{K^{\bullet}}
    \stackrel{\text{\ref{lem:intHomunderlineSunderlineV-is-underlineHomcontSV-reconstructionpaper}}}{=}
    \cone\left(C_{\cont}^{\bullet}\left(\Gal\left( \overline{k} / k\right),\underline{B}_{\dR}^{\dag,+}\right)
      \to C_{\cont}^{\bullet}\left(\Gal\left( \overline{k} / k\right),\underline{C}\right)\right)
  \end{equation*}
  is exact, which follows from Theorem~\ref{thm:galois-cohomology-of-solidBdRdaggerplus}
  and the arguments in the proof of Lemma~\ref{lem:contgpcoh-indban-vs-solid-reconstructionpaper}.
  
  We note that for any profinite set $S$, the functor $\intHom_{\cont}\left(S,-\right)$ preserves injective maps between
  Banach spaces. Therefore,
  Proposition~\ref{prop:BdRdagplus-bornology-countable-basis-reconstructionpaper}
  implies that $K^{\bullet}$ is a cochain complex of bornological $k$-vector spaces whose bornology has a countable basis,
  cf. Definitions~\ref{defn:bornologicalspace-reconstructionpaper} and
  \ref{defn:bornology-countable-basis-reconstructionpaper}. Thus
  Proposition~\ref{prop:solidexact-implies-indBanachstrictlyexact-reconstructionpaper} applies,
  that is the exactness of $\underline{K^{\bullet}}$ implies the strict exactnes of
  $K^{\bullet}$.
\end{proof}

%%%%%%%%%%%%%%%%%%%%%%%%%%%%%%%%%%%%%%%%%%%%%%%%%%%%%%%%%%%%
%%%%%%%%%%%%%%%%%%%%%%%%%%%%%%%%%%%%%%%%%%%%%%%%%%%%%%%%%%%%
% Cech cohomology III
%%%%%%%%%%%%%%%%%%%%%%%%%%%%%%%%%%%%%%%%%%%%%%%%%%%%%%%%%%%%
%%%%%%%%%%%%%%%%%%%%%%%%%%%%%%%%%%%%%%%%%%%%%%%%%%%%%%%%%%%%

\section{The Galois cohomology of $B_{\dR}^{\dag}$}
\label{subsec:Galoiscoh-of-overconvergentdeRhamperiodring}

\begin{thm}\label{thm:galois-cohomology-of-solidBdRdagger-born}
  \begin{equation*}
    \Ho_{\cont}^{i}\left(\Gal\left( \overline{k} / k \right) , B_{\dR}^{\dag} \right)
    \cong
    \begin{cases}
      \I\left(k\right), &\text{ if $i=0,1$ and} \\
      0, &\text{otherwise}.
    \end{cases}
  \end{equation*}
\end{thm}

From now on, we work towards the proof of Theorem~\ref{thm:galois-cohomology-of-solidBdRdagger-born}.

%%%%%%%%%%%%%%%%%%%%%%%%%%%%%%%%%%%%%%%%%%%%%%%%%%%%%%%%%%%%
% Overview of the proof of 
%%%%%%%%%%%%%%%%%%%%%%%%%%%%%%%%%%%%%%%%%%%%%%%%%%%%%%%%%%%%

\subsection{Overview of the proof of Theorem~\ref{thm:galois-cohomology-of-solidBdRdagger-born}}

Fix an intermediate extension $\overline{k} / k^{\prime} / k$ as in
Notation~\ref{notation:galois-cohomology-of-BdRdaggerplus-kzero}.
The proof of Theorem~\ref{thm:galois-cohomology-of-solidBdRdagger-born}
proceeds in two steps: We compute Galois cohomology with respect to
$\overline{k}/k^{\prime}$, and then with respect to $k^{\prime}/k$. The first step is achieved
by Proposition~\ref{prop:galois-cohomology-of-BdRdaggerplus-kprime-injectintotminusj}.
The second one is given in \S\ref{subsubsec:proofthm:galois-cohomology-of-solidBdRdagger}.

\begin{remark}
  Just as in the proof of Theorem~\ref{thm:galois-cohomology-of-solidBdRdaggerplus-born},
  our proof of Theorem~\ref{thm:galois-cohomology-of-solidBdRdagger-born}
  uses techniques from both the ind-Banach and solid worlds.
  See also remark~\ref{remark:thmgalois-cohomology-of-solidBdRdaggerplus-indBansolidmethods}.
\end{remark}

%%%%%%%%%%%%%%%%%%%%%%%%%%%%%%%%%%%%%%%%%%%%%%%%%%%%%%%%%%%%
% On Proposition~\ref{prop:galois-cohomology-of-BdRdaggerplus-primitiverootofunity}
%%%%%%%%%%%%%%%%%%%%%%%%%%%%%%%%%%%%%%%%%%%%%%%%%%%%%%%%%%%%

\subsection{On Proposition~\ref{prop:galois-cohomology-of-BdRdaggerplus-kprime-injectintotminusj}
  and its proof}
\label{subsubsec:--prop:galois-cohomology-of-BdRdaggerplus-kprime-injectintotminusj--anditsproof-recpaper}
  
Recall Proposition~\ref{prop:galois-cohomology-of-BdRdaggerplus-primitiverootofunity}.

\begin{prop}\label{prop:galois-cohomology-of-BdRdaggerplus-kprime-injectintotminusj}
  Given $j\in\NN$, the canonical morphism
  $B_{\dR}^{\dag,+}\to t^{-j}B_{\dR}^{\dag,+}$
  induces the isomorphism
  \begin{equation*}
    \Ho_{\cont}^{i}\left(\Gal\left( \overline{k} / k^{\prime} \right) , t^{-j}B_{\dR}^{\dag,+} \right)
    \cong
    \begin{cases}
      \I\left(k^{\prime}\right), &\text{ if $i=0,1$ and} \\
      0, &\text{otherwise}.
    \end{cases}
  \end{equation*}
\end{prop}

Fix $j\in\NN$ for the remainder of \S\ref{subsec:Galoiscoh-of-overconvergentdeRhamperiodring}.
From now on, we work towards the proof of Proposition~\ref{prop:galois-cohomology-of-BdRdaggerplus-kprime-injectintotminusj}. We choose to give an overview of our arguments in
\S\ref{subsubsec:--prop:galois-cohomology-of-BdRdaggerplus-kprime-injectintotminusj--anditsproof-recpaper}.

Write $\mathcal{G}:=\Gal\left( \overline{k} / k^{\prime} \right)$.
Theorem~\ref{thm:subsections-periodsheaves-affperfd}
and Proposition~\ref{prop:BdRdaggerplusUtimesS-isomapHomcontSAdRgreaterthanqU-reconstructionpaper}
imply that the complex
\begin{equation*}%\label{eq:galois-cohomology-of-BdRdagger-thecomplexcomputingGaloiscoh}
  \BB_{\dR}^{\dag,+}\left(*\right)
  \stackrel{\delta^{0}}{\longrightarrow}
  \BB_{\dR}^{\dag,+}\left(\cal{G} \times *\right)
  \stackrel{\delta^{1}}{\longrightarrow}
  \BB_{\dR}^{\dag,+}\left(\cal{G}^{2} \times *\right)
  \stackrel{\delta^{2}}{\longrightarrow}\dots
\end{equation*}
computes the continuous Galois cohomology of $B_{\dR}^{\dag,+}$.
Denote it by $\BB_{\dR}^{\dag,+}\left(\cal{G}^{\bullet}\times*\right)$.
Similarly, the Galois cohomology of $t^{-j}B_{\dR}^{\dag,+}$
is computed by the complex
\begin{equation*}%\label{eq:galois-cohomology-of-BdRdagger-thecomplexcomputingGaloiscoh}
  t^{-j}\BB_{\dR}^{\dag,+}\left(*\right)
  \stackrel{\delta_{-j}^{0}}{\longrightarrow}
  t^{-j}\BB_{\dR}^{\dag,+}\left(\cal{G} \times *\right)
  \stackrel{\delta_{-j}^{1}}{\longrightarrow}
  t^{-j}\BB_{\dR}^{\dag,+}\left(\cal{G}^{2} \times *\right)
  \stackrel{\delta_{-j}^{2}}{\longrightarrow}\dots
\end{equation*}
which we denote by $t^{-j}\BB_{\dR}^{\dag,+}\left(\cal{G}^{\bullet}\times*\right)$.
The following result is clear from the previous discussion.

\begin{lem}\label{lem:galois-cohomology-of-BdRdaggerplus-primitiverootofunity-j-whatistheretoshow}
  $B_{\dR}^{\dag,+}\to t^{-j}B_{\dR}^{\dag,+}$ induces a canonical inclusion
  \begin{equation}\label{eq:galois-cohomology-of-BdRdaggerplus-primitiverootofunity-j-whatistheretoshow}
    \BB_{\dR}^{\dag,+}\left(\cal{G}^{\bullet} \times *\right)
    \to t^{-j}\BB_{\dR}^{\dag,+}\left(\cal{G}^{\bullet} \times *\right).
  \end{equation}
  If it is a quasi-isomorphism
  in the sense of Definition~\ref{defn:quasiabelian-quasiiso-recpaper},
  Proposition~\ref{prop:galois-cohomology-of-BdRdaggerplus-kprime-injectintotminusj}
  follows.
\end{lem}

We proceed by proving that~(\ref{eq:galois-cohomology-of-BdRdaggerplus-primitiverootofunity-j-whatistheretoshow})
is a quasi-isomorphism of cochain complexes of $k_{0}$-ind-Banach modules.
Along the way, we encounter similar technical difficulties as in the proof of
Proposition~\ref{prop:galois-cohomology-of-BdRdaggerplus-primitiverootofunity},
cf. the discussion following
Lemma~\ref{lem:galois-cohomology-of-BdRdaggerplus-primitiverootofunity-whatistheretoshow}.
Luckily, many ideas from the proof of Proposition~\ref{prop:galois-cohomology-of-BdRdaggerplus-primitiverootofunity}
can be recycled to overcome these complications. For example,
we introduced the modifications
$D^{>q,\bullet}$ of $\A_{\dR}^{>q,+}\left(\cal{G}^{\bullet} \times *\right)$ and
$\widetilde{D}^{>q,\bullet}$ of $\widetilde{\A}_{\dR}^{>q,+}\left(\cal{G}^{\bullet} \times *\right)$
in \S\ref{subsubsec:galois-cohomology-of-BdRdaggerplus-primitiverootofunity1}
and \S\ref{subsubsec:galois-cohomology-of-BdRdaggerplus-primitiverootofunity2}.
In \S\ref{subsubsec:proof:prop:galois-cohomology-of-BdRdaggerplus-kprime-injectintotminusj-1}
and \S\ref{subsubsec:proof:prop:galois-cohomology-of-BdRdaggerplus-kprime-injectintotminusj-2},
we introduce the modifications
$D_{-j}^{>q,\bullet}$ of $\left(t/p^{q}\right)^{-j}\A_{\dR}^{>q,+}\left(\cal{G}^{\bullet} \times *\right)$ and
$\widetilde{D}_{-j}^{>q,\bullet}$ of $\left(t/p^{q}\right)^{-j}\widetilde{\A}_{\dR}^{>q,+}\left(\cal{G}^{\bullet} \times *\right)$.
In \S\ref{subsubsec:proof:prop:galois-cohomology-of-BdRdaggerplus-kprime-injectintotminusj-3},
we observe that the latter modification kills the junk torsion in the cohomology of
each degree of the associated graded of
$\left(t/p^{q}\right)^{-j}\widetilde{\A}_{\dR}^{>q,+}\left(\cal{G}^{\bullet} \times *\right)$.
We then finish the proof of
Proposition~\ref{prop:galois-cohomology-of-BdRdaggerplus-kprime-injectintotminusj}
in \S\ref{subsubsec:proof:prop:galois-cohomology-of-BdRdaggerplus-kprime-injectintotminusj-6},
after setting up further technical machinery in
\S\ref{subsubsec:proof:prop:galois-cohomology-of-BdRdaggerplus-kprime-injectintotminusj-4}
and \S\ref{subsubsec:proof:prop:galois-cohomology-of-BdRdaggerplus-kprime-injectintotminusj-5}.

%%%%%%%%%%%%%%%%%%%%%%%%%%%%%%%%%%%%%%%%%%%%%%%%%%%%%%%%%%%%
% On Proposition~\ref{prop:galois-cohomology-of-BdRdaggerplus-primitiverootofunity}
%%%%%%%%%%%%%%%%%%%%%%%%%%%%%%%%%%%%%%%%%%%%%%%%%%%%%%%%%%%%

\subsection{The cochain complex $D_{-j}^{>q,\bullet}$}
\label{subsubsec:proof:prop:galois-cohomology-of-BdRdaggerplus-kprime-injectintotminusj-1}

Fix $q\in\NN_{\geq2}$ and $j\in\NN$. We continue to write $\cal{G}:=\Gal\left(\overline{k}/k^{\prime}\right)$.

\begin{notation}\label{notation:twisted-AdRgreaterthanq-recpaper}
  Let $i\in\NN$. Then
  $\A_{-j}^{>q}\left(\cal{G}^{i}\times*\right):=\left(t/p^{q}\right)^{-j}\A_{\dR}^{>q}\left(\cal{G}^{i}\times*\right)$,
  equipped with the norm
  \begin{equation*}
    \|\left(\frac{t}{p^{q}}\right)^{-j}a\|:=\|a\|
    \text{ for all $a\in\A_{\dR}^{>q}\left(\cal{G}^{i}\times*\right)$.}
  \end{equation*}
  This is well-defined as
  $\left( t / p^{q}\right)$ is not a zero-divisor, cf.
  Lemma~\ref{lem:multbytpq-onAdRgreaterthanq-strictmono-reconstructionpaper}
  and Theorem~\ref{thm:BdR>qplus+-sections-over-affperfd-recpaper}.
\end{notation}

$A_{-j}^{>q}:=\A_{-j}^{>q}\left(*\right)$ carries the continuous $\cal{G}$-action given by
\begin{equation*}
  \alpha\left(\left(\frac{t}{p^{q}}\right)^{-j}a\right)
  :=\left(\frac{t}{p^{q}}\right)^{-j}\left(\chi\left(\alpha\right)^{-j}\cdot\left(\alpha a\right)\right)
\end{equation*}
for all $\alpha\in\cal{G}$ and $a\in A_{\dR}^{>q}$.
Here, $\chi\colon\cal{G}\to\ZZ_{p}^{\times}$ is the cyclotomic character.
We are interested in the continuous cohomology
$\R\Gamma_{\cont}\left(\cal{G},A_{-j}^{>q}\right)$.
It is represented by the complex $C_{\cont}^{\bullet}\left(\cal{G},A_{-j}^{>q}\right)$, that is
\begin{equation*}
  A_{-j}^{>q}
  \longrightarrow
  \intHom_{\cont}\left(\cal{G} , A_{-j}^{>q} \right)
  \longrightarrow
  \intHom_{\cont}\left(\cal{G}^{2} , A_{-j}^{>q} \right)
  \longrightarrow
  \dots
\end{equation*}
where the differentials are as in Definition~\ref{defn:Banachmodule-Ccontcomplex-reconstructionpaper}
Proposition~\ref{prop:AdRgreaterthanqUtimesS-isomapHomcontSAdRgreaterthanqU-reconstructionpaper}
implies that it is isomorphic to
\begin{equation}\label{eq:notationbefore--defn:galois-cohomology-of-BdRdaggerplus-primitiverootofunity-Dgreaterthanqbullet-j}
  \A_{-j}^{>q}\left(*\right)
  \stackrel{\delta_{-j}^{>q,0}}{\longrightarrow}
  \A_{-j}^{>q}\left(\cal{G} \times *\right)
  \stackrel{\delta_{-j}^{>q,1}}{\longrightarrow}
  \A_{-j}^{>q}\left(\cal{G}^{2} \times *\right)
  \stackrel{\delta_{-j}^{>q,2}}{\longrightarrow}\dots.
\end{equation}
Denote this complex by $\A_{-j}^{>q}\left(\cal{G}^{\bullet} \times *\right)$.
It recovers the complex in Notation~\ref{defn:galois-cohomology-of-BdRdaggerplus-primitiverootofunity-Dgreaterthanqbullet-somedifferentials}
for $j=0$.

\begin{defn}\label{defn:galois-cohomology-of-BdRdaggerplus-primitiverootofunity-Dgreaterthanqbullet-j}
For every $q\in\NN_{\geq3}$, denote the following sequence of maps by $D_{-j}^{>q,\bullet}$:
\begin{equation*}
  \A_{-j}^{>q}\left(*\right)
  \times_{\delta_{-j}^{>q,0} , \A_{-j}^{>q}\left(\cal{G} \times *\right) }
  \A_{-j}^{>q-1}\left(\cal{G} \times *\right)
  \stackrel{\epsilon_{-j}^{>q,0}}{\longrightarrow}
  \A_{-j}^{>q-1}\left(\cal{G} \times *\right)
  \stackrel{\epsilon_{-j}^{>q,1}}{\longrightarrow}
  \A_{-j}^{>q-1}\left(\cal{G}^{2} \times *\right)
  \stackrel{\epsilon_{-j}^{>q,2}}{\longrightarrow}\dots.
\end{equation*}
Here, $\epsilon_{-j}^{>q,0}$ is $(a,b)\mapsto b$ and $\epsilon_{-j}^{>q,i}:=\delta_{-j}^{>q-1,i}$
for all $i>0$.
\end{defn}

\begin{lem}\label{lem:galois-cohomology-of-BdRdagger-Dgreaterthanqbullet-j-cochaincomplex}
  $D_{-j}^{>q,\bullet}$ is a cochain complex for all $q\in\NN_{\geq3}$.
\end{lem}

\begin{proof}
  The proof of Lemma~\ref{lem:galois-cohomology-of-BdRdagger-Dgreaterthanqbullet-cochaincomplex}
  goes through verbatim.
\end{proof}

Fix $q\in\NN_{\geq3}$ and consider the commutative diagram
\begin{equation*}
  \begin{tikzcd}
      \A_{-j}^{>q}\left(*\right)
      \arrow{r}{\delta_{-j}^{>q,0}} &
      \A_{-j}^{>q}\left(\cal{G} \times*\right)
      \arrow{r}{\delta_{-j}^{>q,1}} &
      \A_{-j}^{>q}\left(\cal{G}^{2} \times*\right)
      \arrow{r}{\delta_{-j}^{>q,2}} &
      \dots \\
      \A_{-j}^{>q}\left(*\right)
      \times_{\delta_{-j}^{>q,0} , \A_{-j}^{>q,+}\left(\cal{G} \times *\right) }
      \A_{-j}^{>q-1}\left(\cal{G} \times *\right)
      \arrow{r}{\epsilon_{-j}^{>q,0}}\arrow{u}{\tau_{-j}^{0}} &
      \A_{-j}^{>q-1}\left(\cal{G} \times*\right)
      \arrow{r}{\epsilon_{-j}^{>q,1}}\arrow{u} &
      \A_{-j}^{>q-1}\left(\cal{G}^{2} \times*\right)
      \arrow{r}{\epsilon_{-j}^{>q,2}}\arrow{u} &
      \dots  
  \end{tikzcd}
\end{equation*}
where $\tau_{-j}^{0}$ is $(a,b)\mapsto a$ and all other vertical maps are canonical.
It defines a morphism
\begin{equation*}
  \Phi_{-j}^{>q,\bullet}\colon D_{-j}^{>q,\bullet} \to \A_{-j}^{>q}\left(\cal{G}^{\bullet}\times*\right)
\end{equation*}
of cochain complexes. We recover $\Phi_{0}^{>q,\bullet}=\Phi^{>q,\bullet}$
as in~(\ref{eq:Phigreatanqbullet})

\begin{lem}\label{lem:galois-cohomology-of-BdRdagger-Dgreaterthanqbullet-j-isotoBdRdagplusGbulletstar}
  $\text{``}\varinjlim\text{"}_{q\in\NN_{\geq3}}\Phi_{-j}^{>q,\bullet}\widehat{\otimes}_{W(\kappa)}\id_{k_{0}}$
  is a canonical isomorphism of cochain complexes
  \begin{equation*}
    \text{``}\varinjlim\text{"}_{q\in\NN_{\geq3}}D_{-j}^{>q,\bullet}\widehat{\otimes}_{W(\kappa)}k_{0}
    \isomap t^{-j}\BB_{\dR}^{\dag,+}\left(\cal{G}^{\bullet}\times*\right).
  \end{equation*}
\end{lem}

\begin{proof}
  Again, the proof of Lemma~\ref{lem:galois-cohomology-of-BdRdagger-Dgreaterthanqbullet-isotoBdRdagplusGbulletstar}
  goes through verbatim.
\end{proof}

We aim to show that the canonical inclusion $B_{\dR}^{\dag,+}\to t^{-j}B_{\dR}^{\dag,+}$
induces an isomorphism on continuous Galois cohomology.
Lemma~\ref{lem:galois-cohomology-of-BdRdagger-Dgreaterthanqbullet-isotoBdRdagplusGbulletstar}
says that one can compute the Galois cohomology of $B_{\dR}^{\dag,+}$
via the cohomology of $D^{>q,\bullet}$. Similarly, $D_{-j}^{>q,\bullet}$ computes the cohomology of
$t^{-j}B_{\dR}^{\dag,+}$ by
Lemma~\ref{lem:galois-cohomology-of-BdRdagger-Dgreaterthanqbullet-j-isotoBdRdagplusGbulletstar}.
Therefore, we may compare the cohomology of $D^{>q,\bullet}$ and $D_{-j}^{>q,\bullet}$.
In fact, $B_{\dR}^{\dag,+}\to t^{-j}B_{\dR}^{\dag,+}$ comes from the
inclusions $A_{\dR}^{>q,+}\to A_{-j}^{>q,+}$ for varying $q\in\NN_{\geq3}$. These induce
the morphisms
\begin{equation*}
  \A_{\dR}^{>q,+}\left(\cal{G}^{\bullet}\times*\right)
  \to \A_{-j}^{>q,+}\left(\cal{G}^{\bullet}\times*\right).
\end{equation*}
They gives rise to the morphisms
\begin{equation*}
  \gamma_{-j}^{>q,\bullet}\colon D^{>q,\bullet} \to D_{-j}^{>q,\bullet}
\end{equation*}
of cochain complexes of $W(\kappa)$-Banach modules.
We see later on that the $\gamma_{-j}^{>q,\bullet}$ are quasi-isomorphisms.

%%%%%%%%%%%%%%%%%%%%%%%%%%%%%%%%%%%%%%%%%%%%%%%%%%%%%%%%%%%%
% On Proposition~\ref{prop:galois-cohomology-of-BdRdaggerplus-primitiverootofunity}
%%%%%%%%%%%%%%%%%%%%%%%%%%%%%%%%%%%%%%%%%%%%%%%%%%%%%%%%%%%%

\subsection{The cochain complex $\widetilde{D}_{-j}^{>q,\bullet}$}
\label{subsubsec:proof:prop:galois-cohomology-of-BdRdaggerplus-kprime-injectintotminusj-2}

We continue to fix $q\in\NN_{\geq3}$ and $j\in\NN$. Recall the
Definition~\ref{defn:galois-cohomology-of-BdRdaggerplus-primitiverootofunity-Dgreaterthanqbullet-j}
of the complex $D_{-j}^{>q,\bullet}$. In the following, we introduce the complex
$\widetilde{D}_{-j}^{>q,\bullet}$. It coincides with $D_{-j}^{>q,\bullet}$ as a complex of
abstract $W(\kappa)$-modules, but it carries different topologies. Our discussion
follows \S\ref{subsubsec:galois-cohomology-of-BdRdaggerplus-primitiverootofunity2}.

\begin{notation}\label{notation:A-jgreaterthanq-recpaper}
  Let $i\in\NN$. Then
  $\widetilde{\A}_{-j}^{>q}\left(\cal{G}^{i}\times*\right)$
  denotes the abstract $W(\kappa)$-module
  $|\A_{-j}^{>q}\left(\cal{G}^{i}\times*\right)|$,
  equipped with the filtration
  \begin{equation*}
    \Fil^{n}\widetilde{\A}_{-j}^{>q}\left(\cal{G}^{i}\times*\right)
    :=\begin{cases}
      |\A_{-j}^{>q}\left(\cal{G}^{i}\times*\right)| \cdot \left(\frac{\xi}{p^{q}}\right)^{n+j}
      &\text{ if $n\geq-j$, and} \\
      0 & \text{otherwise}.
    \end{cases}
  \end{equation*}
\end{notation}

In what follows, we use the notation for the differentials as
in~(\ref{eq:notationbefore--defn:galois-cohomology-of-BdRdaggerplus-primitiverootofunity-Dgreaterthanqbullet-j}).

\begin{defn}\label{defn:galois-cohomology-of-BdRdaggerplus-primitiverootofunity-tildeDgreaterthanqbullet-j}
Denote the following sequence of maps by $\widetilde{D}_{-j}^{>q,\bullet}$:
\begin{equation*}
  \widetilde{\A}_{-j}^{>q}\left(*\right)
  \times_{\delta_{-j}^{>q,0} , \widetilde{\A}_{-j}^{>q}\left(\cal{G} \times *\right) }
  \widetilde{\A}_{-j}^{>q-1}\left(\cal{G} \times *\right)
  \stackrel{\epsilon_{-j}^{>q,0}}{\longrightarrow}
  \widetilde{\A}_{-j}^{>q-1}\left(\cal{G} \times *\right)
  \stackrel{\epsilon_{-j}^{>q,1}}{\longrightarrow}
  \widetilde{\A}_{-j}^{>q-1}\left(\cal{G}^{2} \times *\right)
  \stackrel{\epsilon_{-j}^{>q,2}}{\longrightarrow}\dots.
\end{equation*}
Here, $\epsilon_{-j}^{>q,0}$ is $(a,b)\mapsto b$ and $\epsilon_{-j}^{>q,i}:=\delta_{-j}^{>q-1,i}$
for all $i>0$.
\end{defn}

\begin{lem}\label{lem:galois-cohomology-of-BdRdagger-tildeDgreaterthanqbullet-j-cochaincomplex}
  $\widetilde{D}_{-j}^{>q,\bullet}$ is a cochain complex.
\end{lem}

\begin{proof}
  This follows from Lemma~\ref{lem:galois-cohomology-of-BdRdagger-Dgreaterthanqbullet-j-cochaincomplex},
  because the complexes of abstract $W(\kappa)$-modules underlying
  $D_{-j}^{>q,\bullet}$ and $\widetilde{D}_{-j}^{>q,\bullet}$ coincide.
\end{proof}

The canonical inclusion
$\widetilde{A}_{\dR}^{>q,+}\to\widetilde{A}_{-j}^{>q,+}$
induces a morphism
\begin{equation*}
  \widetilde{\gamma}_{-j}^{>q,\bullet}\colon \widetilde{D}^{>q,\bullet} \to \widetilde{D}_{-j}^{>q,\bullet}
\end{equation*}
of cochain complexes of filtered modules. We remark that
the morphisms of complexes of abstract $W(\kappa)$-modules
underlying $\gamma_{-j}^{>q,\bullet}$,
cf. \S\ref{subsubsec:proof:prop:galois-cohomology-of-BdRdaggerplus-kprime-injectintotminusj-1},
and $\widetilde{\gamma}_{-j}^{>q,\bullet}$ coincide.

%%%%%%%%%%%%%%%%%%%%%%%%%%%%%%%%%%%%%%%%%%%%%%%%%%%%%%%%%%%%
% On Proposition~\ref{prop:galois-cohomology-of-BdRdaggerplus-primitiverootofunity}
%%%%%%%%%%%%%%%%%%%%%%%%%%%%%%%%%%%%%%%%%%%%%%%%%%%%%%%%%%%%

\subsection{The cohomology of $\gr\widetilde{D}_{-j}^{>q,\bullet}$}
\label{subsubsec:proof:prop:galois-cohomology-of-BdRdaggerplus-kprime-injectintotminusj-3}

We fix again $q\in\NN_{\geq3}$ and $j\in\NN$.
Recall the Definition~\ref{defn:galois-cohomology-of-BdRdaggerplus-primitiverootofunity-tildeDgreaterthanqbullet-j}
of $\widetilde{D}_{-j}^{>q,\bullet}$; see also
Lemma~\ref{lem:galois-cohomology-of-BdRdagger-tildeDgreaterthanqbullet-j-cochaincomplex}.
We observe that the morphism
$\widetilde{\gamma}_{-j}^{>q,\bullet}$ introduced in
\S\ref{subsubsec:proof:prop:galois-cohomology-of-BdRdaggerplus-kprime-injectintotminusj-2}
preserves the filtrations. 
In particular, it induces a morphism between the associated gradeds
\begin{equation*}
  \gr\widetilde{\gamma}_{-j}^{>q,\bullet}\colon
  \gr\widetilde{D}^{>q,\bullet}
  \to
  \gr\widetilde{D}_{-j}^{>q,\bullet}.
\end{equation*}
One now applies the $\eta$-operator
as in Definition~\ref{defn:algebraic-decalage} to get the
morphism~(\ref{eq:galois-cohomology-of-BdRdagger-whatistheretoshowquestionsmark-j-themorphism})
in Proposition~\ref{prop:galois-cohomology-of-BdRdagger-j-whatistheretoshowquestionsmark}.
%In fact, this Proposition is the key technical input for the proof of
%Proposition~\ref{prop:galois-cohomology-of-BdRdaggerplus-kprime-injectintotminusj}.
 
\begin{prop}\label{prop:galois-cohomology-of-BdRdagger-j-whatistheretoshowquestionsmark}
  There exists a constant $N_{-j}\in\NN$ such that
  \begin{equation}\label{eq:galois-cohomology-of-BdRdagger-whatistheretoshowquestionsmark-j-themorphism}
    \eta_{p^{N_{-j}}}\gr\widetilde{\gamma}_{-j}^{>q,\bullet}\colon
    \eta_{p^{N_{-j}}}\gr\widetilde{D}^{>q,\bullet}
    \to
    \eta_{p^{N_{-j}}}\gr\widetilde{D}_{-j}^{>q,\bullet}
  \end{equation}
  is a quasi-isomorphism of cochain complexes of
  abstract $W(\kappa)$-modules.
\end{prop}

\iffalse
\begin{remark}
  The constant $N$ in Proposition~\ref{prop:galois-cohomology-of-BdRdagger-j-whatistheretoshowquestionsmark}
  depends on $j$. In particular, it does not coincide with the constant
  $N$ as in Proposition~\ref{prop:galois-cohomology-of-BdRdagger-whatistheretoshowquestionsmark}.
\end{remark}
\fi

Recall Notation~\ref{notation:Tatesclassicalgroupcohomology}.
For any $s\in\ZZ$,
$\cal{O}_{C}(s)$ denotes the $s$th Tate twist of $\cal{O}_{C}$.
Our proof of
Proposition~\ref{prop:galois-cohomology-of-BdRdagger-j-whatistheretoshowquestionsmark}
relies on the following result:

\begin{lem}\label{lem:galois-cohomology-of-BdRdagger-injectintotminusj-Dgreaterthanqbulletcomputedgrn}
  For any $n\in\NN_{\geq-j}$,
  \begin{equation*}
    \Ho^{i}\left(\gr^{n}\widetilde{D}_{-j}^{>q,\bullet}\right) =
    \begin{cases}
      \Ho^{0}_{\clcont}\left(\cal{G},\cal{O}_{C}(n)\right) & \text{ if $i=0$}, \\
      \Ho^{1}_{\clcont}\left(\cal{G},\cal{O}_{C}(n)\right) / \text{($p^{n+j}$-torsion)} & \text{ if $i=1$, and} \\
      p^{n+j}\Ho^{i}_{\clcont}\left(\cal{G},\cal{O}_{C}(n)\right) & \text{ if $i\geq 2$.}
    \end{cases}
  \end{equation*}
\end{lem}

\begin{proof}
  The action of $\cal{G}$ on $t$ is given by the cyclotomic character
  $\chi\colon\cal{G}\to\ZZ_{p}^{\times}$.
  Consequently, the proof of Lemma~\ref{lem:galois-cohomology-of-BdRdagger-Dgreaterthanqbulletcomputedgrn}
  goes through verbatim, once we twist everything by $\chi^{-j}$.
\end{proof}

\begin{proof}[Proof of Proposition~\ref{prop:galois-cohomology-of-BdRdagger-j-whatistheretoshowquestionsmark}]\label{proof:prop:galois-cohomology-of-BdRdagger-j-whatistheretoshowquestionsmark}
  Fix $N$ as in Proposition~\ref{prop:galois-cohomology-of-BdRdagger-whatistheretoshowquestionsmark}.
  Then $\gr^{n}\widetilde{\gamma}_{-j}^{>q,\bullet}$%\colon
    %\gr^{n}\widetilde{D}^{>q,\bullet}
    %\to
    %\gr^{n}\widetilde{D}_{-j}^{>q,\bullet}$
  is a quasi-isomorphism up to $p^{N}$-torsion for all $n\geq-j$
  by Lemma~\ref{lem:galois-cohomology-of-BdRdagger-Dgreaterthanqbulletcomputedgrn}
  and~\ref{lem:galois-cohomology-of-BdRdagger-injectintotminusj-Dgreaterthanqbulletcomputedgrn},
  as well as \cite[Theorem 4.4.3]{BSSW2024_rationalizationoftheKnlocalsphere}.
  Furthermore, set $L:=\max\left\{L^{\prime},3\right\}$ where $L^{\prime}:=\max_{i=1,\dots,j}M+\nu_{p}(i)$,
  where $\nu_{p}$ denotes the $p$-adic valuation. \emph{Loc. cit.} proves that
  $\gr^{n}\widetilde{D}^{>q,\bullet}$ is acyclic up to $p^{L}$-torsion for all $n<0$.
  Since $\gr^{n}\widetilde{D}^{>q,\bullet}=0$ for such $n<0$,
  $\gr^{n}\widetilde{\gamma}_{-j}^{>q,\bullet}$ is a quasi-isomorphism up to
  $p^{L}$-torsion for $n<0$.
  
  We have thus checked that $\gr\widetilde{\gamma}_{-j}^{>q,\bullet}$%\colon
    %\gr\widetilde{D}^{>q,\bullet}
    %\to
    %\gr\widetilde{D}_{-j}^{>q,\bullet}$
  is a quasi-isomorphism up to $p^{N_{-j}}$-torsion, where
  $N_{-j}:=\max\left\{N,L\right\}$.
  Thus Proposition~\ref{prop:galois-cohomology-of-BdRdagger-j-whatistheretoshowquestionsmark} follows
  from Lemma~\ref{lem:eta-operator-kills-torsion}.
\end{proof}

%%%%%%%%%%%%%%%%%%%%%%%%%%%%%%%%%%%%%%%%%%%%%%%%%%%%%%%%%%%%
% Applying the $\eta$-operator
%%%%%%%%%%%%%%%%%%%%%%%%%%%%%%%%%%%%%%%%%%%%%%%%%%%%%%%%%%%%

\subsection{Applying the $\eta$-operator}\label{subsubsec:galois-cohomology-of-BdRdaggerplus-primitiverootofunity4}
\label{subsubsec:proof:prop:galois-cohomology-of-BdRdaggerplus-kprime-injectintotminusj-4}

%We will now establish foundational lemmata that enable us to deduce
%Proposition~\ref{prop:galois-cohomology-of-BdRdaggerplus-primitiverootofunity}
%from Proposition~\ref{prop:galois-cohomology-of-BdRdagger-whatistheretoshowquestionsmark}.

In the proof of Proposition~\ref{prop:galois-cohomology-of-BdRdaggerplus-kprime-injectintotminusj},
we apply the $\eta$-operator as in Definition~\ref{defn:algebraic-decalage}
to specific cochain complexes of Banach modules. In this technical \S, we check
that our constructions are well-behaved. %The main results are the
%Corollary~\ref{cor:galois-cohomology-of-BdRdaggerplus-primitiverootofunity-aaaaaletaoperatorconditions-corollary}
%and Lemma~\ref{lem:cohomologyofBdRdaggerplus-etagr-is-greta-forvarthetagreaterthanq-reconstructionpaper}.
Throughout, we fix arbitrary $q\in\NN_{\geq3}$ and $N\in\NN$.

\begin{lem}\label{lem:Dgreaterqisatisfiesconditionsformetaoperator-periodringcohomology-j-reconstructionpaper}
  For all $i\geq0$, $p^{N}D_{-j}^{>q,i}\subseteq D_{-j}^{>q,i}$
  is a closed subset.
\end{lem}

\begin{proof}
  The proof of
  Lemma~\ref{lem:Dgreaterqisatisfiesconditionsformetaoperator-periodringcohomology-reconstructionpaper}
  generalises straightforwardly.
\end{proof}

\begin{lem}\label{lem:Dgreaterqtildeisatisfiesconditionsformetaoperator-periodringcohomology-j-reconstructionpaper}
  For all $i\geq0$, $p^{N}\widetilde{D}_{-j}^{>q,i}\subseteq\widetilde{D}_{-j}^{>q,i}$
  is a closed subset.
\end{lem}

\begin{proof}
  The proof of
  Lemma~\ref{lem:Dgreaterqtildeisatisfiesconditionsformetaoperator-periodringcohomology-reconstructionpaper}
  generalises straightforwardly.
\end{proof}

\begin{lem}\label{lem:galois-cohomology-of-BdRdaggerplus-primitiverootofunity-j-aaaaaletaoperatorconditions}
  Condition~\ref{cond:Banachmoduledecalage-reconstructionpaper} is satisfied for
  $R=W(\kappa)$, $r=p^{N}$, and $M^{\bullet}=D_{-j}^{>q,\bullet}$.
\end{lem}

\begin{proof}
  We check:
  \begin{itemize}
    \item[(i)] $D_{-j}^{>q,\bullet}$ is indeed concentrated in non-negative degrees.
    \item[(ii)] This is Lemma~\ref{lem:Dgreaterqisatisfiesconditionsformetaoperator-periodringcohomology-j-reconstructionpaper}.
    \item[(iii)] For every $i\in\NN$, $D_{-j}^{>q,i}$ does not have $p^{N}$-torsion
      by Lemma~\ref{lem:AdRgreaterthanqptorsionfree}.
  \end{itemize}
\end{proof}

\begin{cor}\label{cor:galois-cohomology-of-BdRdaggerplus-primitiverootofunity-j-aaaaaletaoperatorconditions-corollary}
  The subcomplex
  $\eta_{p^{N}} D_{-j}^{>q,\bullet}\subseteq D_{-j}^{>q,\bullet}$
  equipped with the induced norms,
  is a complex of $W(\kappa)$-Banach algebras.
\end{cor}

\begin{proof}
  Thanks to Lemma~\ref{lem:galois-cohomology-of-BdRdaggerplus-primitiverootofunity-j-aaaaaletaoperatorconditions},
  this follows directly from Lemma~\ref{lem:decalage-defined-reconstructionpaper}.
\end{proof}

In the following, we use the notation as in Definition~\ref{defn:puttrivialnormonring-reconstructionpaper}.

\begin{lem}\label{lem:galois-cohomology-of-BdRdaggerplus-primitiverootofunity-j-aaaaaletaoperatorconditions-filtered}
  Condition~\ref{cond:filteredmodulesdecalage-reconstructionpaper}
  is satisfied for $R=W(\kappa)^{\triv}$, $r=p^{N}$, and
  $M^{\bullet}=\widetilde{D}_{-j}^{>q,\bullet}$.
\end{lem}

\begin{proof}
  We check:
  \begin{itemize}
    \item[(i)] $\widetilde{D}_{-j}^{>q,\bullet}$ is indeed concentrated in nonnegative degrees.
    \item[(ii)] This is Lemma~\ref{lem:Dgreaterqtildeisatisfiesconditionsformetaoperator-periodringcohomology-j-reconstructionpaper}.
    \item[(iii)] For every $i\in\NN$, $\widetilde{D}_{-j}^{>q,i}$ does not have $p^{N}$-torsion
      by Lemma~\ref{lem:AdRgreaterthanqptorsionfree}.
    \item[(iv)] $\gr\widetilde{D}_{-j}^{>q,i}$ is $p^{i}$-torsion free for every $i\in\ZZ$. This follows
    from a description similar to the
    one~(\ref{eq:galois-cohomology-of-BdRdagger-Dgreaterthanqbulletcomputedgrn-needthislater})
    of $\gr\widetilde{D}^{>q,i}$.
  \end{itemize}
\end{proof}

\begin{cor}\label{cor:galois-cohomology-of-BdRdaggerplus-primitiverootofunity-j-aaaaaletaoperatorconditions-filtered-corollary}
  Consider the induced filtration on the subcomplex
  $\eta_{p^{N}} \widetilde{D}_{-j}^{>q,\bullet}\subseteq \widetilde{D}_{-j}^{>q,\bullet}$.
  It is exhaustive, separated, and complete.
\end{cor}

\begin{proof}
  Thanks to
  Lemma~\ref{lem:galois-cohomology-of-BdRdaggerplus-primitiverootofunity-j-aaaaaletaoperatorconditions-filtered},
  this follows directly from Lemma~\ref{lem:decalagefilteredmodules-defined}.
\end{proof}

\begin{lem}\label{lem:cohomologyofBdRdaggerplus-etagr-is-greta-forvarthetagreaterthanq-j-reconstructionpaper}
  We have the isomorphism of graded complexes
  \begin{equation*}
    \eta_{p^{N}}\gr\widetilde{\gamma}_{-j}^{>q,\bullet}\cong\gr\eta_{p^{N}}\widetilde{\gamma}_{-j}^{>q,\bullet}.
  \end{equation*}
\end{lem}

\begin{proof}
  By Lemma~\ref{lem:galois-cohomology-of-BdRdaggerplus-primitiverootofunity-aaaaaletaoperatorconditionsfiltered}
  and~\ref{lem:galois-cohomology-of-BdRdaggerplus-primitiverootofunity-j-aaaaaletaoperatorconditions-filtered},  
  we can apply Lemma~\ref{lem:decalage-commutes-gr} to get the result.
\end{proof}

%%%%%%%%%%%%%%%%%%%%%%%%%%%%%%%%%%%%%%%%%%%%%%%%%%%%%%%%%%%%
% From $\widetilde{D}^{>q,\bullet}$ to $D^{>q,\bullet}$
%%%%%%%%%%%%%%%%%%%%%%%%%%%%%%%%%%%%%%%%%%%%%%%%%%%%%%%%%%%%

\subsection{From $\widetilde{D}_{-j}^{>q,\bullet}$ to $D_{-j}^{>q,\bullet}$}
\label{subsubsec:proof:prop:galois-cohomology-of-BdRdaggerplus-kprime-injectintotminusj-5}

Fix $q\in\NN_{\geq3}$ and $j\in\NN$.
So far, our goal was to study the cohomology of the complex $D_{-j}^{>q,\bullet}$.
Instead, we chose to work with $\widetilde{D}_{-j}^{>q,\bullet}$, as its associated graded
is more accessible. In this \S, we explain how to
recover $D_{-j}^{>q,\bullet}$ from $\widetilde{D}_{-j}^{>q,\bullet}$, therefore deducing
information about the former complex from the latter. This relies
on the operator $\p(-)$ introduced in Definition~\ref{defn:poperator}.

In the following, we often confuse $W(\kappa)$-modules with their underlying topological groups.

\begin{lem}\label{lem:poperator-D-j-to-tildeD-j}
  Compute $D_{-j}^{>q,\bullet}=\p\left(\widetilde{D}_{-j}^{>q,\bullet}\right)$.
\end{lem}

\begin{proof}
  This follows from Lemma~\ref{lem:poperator-commutes-with-pullbacks}
  and the observation
  \begin{equation}\label{eq:pwidetildeAdRisAdR-reconstructionpaper}
    \p\left(\widetilde{\A}_{-j}^{>h}\left(\cal{G}^{i}\times*\right)\right)
      =\A_{-j}^{>h}\left(\cal{G}^{i}\times*\right)
  \end{equation}
  for all $i\in\NN$ and $h\in\left\{q,q-1\right\}$,
  see also Lemma~\ref{lem:pwidetildeAdRisAdR-reconstructionpaper}.
\end{proof}

\begin{lem}\label{lem:galois-cohomology-of-BdRdaggerplus-primitiverootofunity-morelemma2-j-reconstructionpaper}
  Compute $\cone\left(\eta_{p^{N}}\gamma_{-j}^{>q,\bullet}\right)=\p\left(\cone\left(\eta_{p^{N}}\widetilde{\gamma}_{-j}^{>q,\bullet}\right)\right)$ for arbitrary $N\in\NN$.
%  \hfill
%  \begin{itemize}
%    \item[(i)] $D^{>q,\bullet}=\p\left(\widetilde{D}^{>q,\bullet}\right)$,
%    \item[(ii)] $\widehat{\cal{O}}^{+}\left(\cal{G}^{\bullet}\times*\right)=\p\left(\widehat{\cal{O}}^{+}\left(\cal{G}^{\bullet}\times*\right)^{\triv}\right)$, and
%    \item[(iii)] $\cone\left(\eta_{p^{N}}\vartheta^{>q,\bullet}\right)=\p\left(\cone\left(\eta_{p^{N}}\widetilde{\vartheta}^{>q,\bullet}\right)\right)$.
%  \end{itemize}
  \iffalse %%%
  Given the filtered complex $\widetilde{C}^{\bullet}:=\cone\left(\eta_{p^{N}}\widetilde{\vartheta}^{>q,\bullet}\right)$,
  we consider the complex of topological groups $C^{\bullet}$ as follows. Both complexes coincide
  as complexes of abstract groups. For every $i\in\ZZ$, $C^{i}$ carries the topology with respect to
  the open neighbourhood basis
  \begin{equation*}
    \left\{ p^{n}C^{i} + \Fil^{n}\right\}_{s\geq0}.
  \end{equation*}
  Here $\Fil^{n}\subseteq \widetilde{C}^{i}$ denotes the $n$th piece in the filtration.
  Then $C^{\bullet}$ coincides with the complex (of topological groups underlying the complex)
  of $W(\kappa)$-Banach modules
  $\cone\left(\eta_{p^{N}}\vartheta^{>q,\bullet}\right)$.
  \fi %%%
\end{lem}

\begin{proof}
  It suffices to check
  \begin{equation*}
  \begin{split}
    \eta_{p^{N}}D^{>q,\bullet}
    &=\p\left(\eta_{p^{N}}\widetilde{D}^{>q,\bullet}\right)
    \text{ and } \\
    \eta_{p^{N}}D_{-j}^{>q,\bullet}
    &=\p\left(\eta_{p^{N}}\widetilde{D}_{-j}^{>q,\bullet}\right)
  \end{split}
  \end{equation*}
  See the proof of
  Lemma~\ref{lem:galois-cohomology-of-BdRdaggerplus-primitiverootofunity-morelemma2-reconstructionpaper}
  for the first equality. The second
  one follows from Lemma~\ref{lem:poperator-D-j-to-tildeD-j}
  and Definition~\ref{defn:decalage-for-Banachmodules}.
\end{proof}

%%%%%%%%%%%%%%%%%%%%%%%%%%%%%%%%%%%%%%%%%%%%%%%%%%%%%%%%%%%%
% From $\widetilde{D}^{>q,\bullet}$ to $D^{>q,\bullet}$
%%%%%%%%%%%%%%%%%%%%%%%%%%%%%%%%%%%%%%%%%%%%%%%%%%%%%%%%%%%%

\subsection{Proof of Proposition~\ref{prop:galois-cohomology-of-BdRdaggerplus-kprime-injectintotminusj}}
\label{subsubsec:proof:prop:galois-cohomology-of-BdRdaggerplus-kprime-injectintotminusj-6}

%We prove the following technical result before we can finally give the proof of
%Proposition~\ref{prop:galois-cohomology-of-BdRdaggerplus-kprime-injectintotminusj}.

\begin{lem}\label{lem:galois-cohomology-of-BdRdaggerplus-primitiverootofunity-morelemma3-j-reconstructionpaper}
  Given the filtered complex $\widetilde{C}^{\bullet}:=\cone\left(\eta_{p^{N}}\widetilde{\gamma}^{>q,\bullet}\right)$,
  we observe that $C^{i}/\Fil^{n}$ does not have any $p$-power torsion for all $i\in\ZZ$ and all $n\in\ZZ$.
  Here, $\Fil^{n}\subseteq \widetilde{C}^{i}$ denotes the $n$th piece in the filtration on $C^{i}$.
\end{lem}

\begin{proof}
  This follows from Lemma~\ref{lem:AdRgreaterthanqUtimesS-isomapHomcontSAdRgreaterthanqU-assumptionsforlemmasatisfied-reconstructionpaper}.
\end{proof}

\begin{proof}[Proof of Proposition~\ref{prop:galois-cohomology-of-BdRdaggerplus-kprime-injectintotminusj}]
  Fix $N_{-j}\in\NN$ as in
  Proposition~\ref{prop:galois-cohomology-of-BdRdagger-j-whatistheretoshowquestionsmark}.
  \emph{Loc. cit.} shows that $\eta_{p^{N_{-j}}}\gr\widetilde{\gamma}_{-j}^{>q,\bullet}$
  is a quasi-isomorphism of abstract $W(\kappa)$-modules.
  With other words, its mapping cone
  \begin{equation*}
    \cone\left(\eta_{p^{N_{-j}}}\gr\widetilde{\gamma}_{-j}^{>q,\bullet}\right)
    \stackrel{\text{\ref{lem:cohomologyofBdRdaggerplus-etagr-is-greta-forvarthetagreaterthanq-j-reconstructionpaper}}}{\cong}
    \cone\left(\gr\eta_{p^{N_{-j}}}\widetilde{\gamma}_{-j}^{>q,\bullet}\right)
    =\gr\cone\left(\eta_{p^{N_{-j}}}\widetilde{\gamma}_{-j}^{>q,\bullet}\right)
  \end{equation*}
  is acyclic. It follows from~\cite[Chapter I, \S 4.1, page 31-32, Theorem 4]{HuishiOystaeyen1996}
  that $\cone\left(\eta_{p^{N_{-j}}}\widetilde{\gamma}_{-j}^{>q,\bullet}\right)$
  is strictly exact as a filtered complex;
  note that \emph{loc. cit.} applies by
  Corollary~\ref{cor:galois-cohomology-of-BdRdaggerplus-primitiverootofunity-aaaaaletaoperatorconditionsfiltered-corollary}
  and~\ref{cor:galois-cohomology-of-BdRdaggerplus-primitiverootofunity-j-aaaaaletaoperatorconditions-filtered-corollary}. 
  Now Lemma~\ref{lem:subsections-periodsheaves-affperfd-1}
  applies to $\cone\left(\eta_{p^{N_{-j}}}\widetilde{\gamma}^{>q,\bullet}\right)$
  by Lemma~\ref{lem:galois-cohomology-of-BdRdaggerplus-primitiverootofunity-morelemma3-j-reconstructionpaper}.
  Together with
  Lemma~\ref{lem:galois-cohomology-of-BdRdaggerplus-primitiverootofunity-morelemma2-j-reconstructionpaper},
  it implies that that $\cone\left(\eta_{p^{N_{-j}}}\gamma_{-j}^{>q,\bullet}\right)$
  is strictly exact as a complex of topological groups. Consequently, it
  is strictly exact as a complex of $W(\kappa)$-Banach modules.
  Because both $\eta_{p^{N_{-j}}}D^{>q,\bullet}$ and $\eta_{p^{N_{-j}}}D_{-j}^{>q,\bullet}$
  are complexes of $p$-torsion free $W(\kappa)$-Banach modules,
  cf. Lemma~\ref{lem:AdRgreaterthanqptorsionfree},
  Corollary~\ref{cor:galois-cohomology-of-BdRdaggerplus-primitiverootofunity-aaaaaletaoperatorconditions-corollary},
  and Corollary~\ref{cor:galois-cohomology-of-BdRdaggerplus-primitiverootofunity-j-aaaaaletaoperatorconditions-corollary},
  one can apply
  Corollary~\ref{cor:completed-localisation-strictlyexact}.
  This gives that $\cone\left(\eta_{p^{N_{-j}}}\gamma_{-j}^{>q,\bullet}\right)\widehat{\otimes}_{W(\kappa)}\id_{k_{0}}$
  is strictly exact. Now apply
  Lemma~\ref{lem:quasiiso-if-cone-strictlyexact-reconstructionpaper} to find that
  \begin{equation*}
    \eta_{p^{N_{-j}}}\gamma_{-j}^{>q,\bullet}\widehat{\otimes}_{W(\kappa)}\id_{k_{0}}
    \colon\eta_{p^{N_{-j}}}D^{>q,\bullet}\widehat{\otimes}_{W(\kappa)}k_{0}
    \to\eta_{p^{N_{-j}}}D_{-j}^{>q,\bullet}\widehat{\otimes}_{W(\kappa)}k_{0}
  \end{equation*}
  is a quasi-isomorphism.
  Pass to the colimit along $q\to\infty$.
  By Corollary~\ref{cor:filteredcol-inIndBan-stronglyexact},
  \begin{equation*}
    \text{``}\varinjlim_{q\in\NN_{\geq3}}\text{"}
    \eta_{p^{N_{-j}}}\gamma_{-j}^{>q,\bullet}\widehat{\otimes}_{W(\kappa)}\id_{k_{0}}
    \colon\text{``}\varinjlim_{q\in\NN_{\geq3}}\text{"}\eta_{p^{N_{-j}}}D^{>q,\bullet}\widehat{\otimes}_{W(\kappa)}k_{0}
    \to\text{``}\varinjlim_{q\in\NN_{\geq3}}\text{"}\eta_{p^{N_{-j}}}D_{-j}^{>q,\bullet}\widehat{\otimes}_{W(\kappa)}k_{0}
  \end{equation*}
  is again a quasi-isomorphism.
  But the domain of this morphism is
  \begin{equation*}
    \text{``}\varinjlim_{q\in\NN_{\geq3}}\text{"}\eta_{p^{N_{-j}}}D^{>q,\bullet}\widehat{\otimes}_{W(\kappa)}k_{0}
    \stackrel{\text{\ref{lem:decalage-completed-localisation-is-just-completed-localisation}}}{\cong}
    \text{``}\varinjlim_{q\in\NN_{\geq3}}\text{"}D^{>q,\bullet}\widehat{\otimes}_{W(\kappa)}k_{0}
    \stackrel{\text{\ref{lem:galois-cohomology-of-BdRdagger-Dgreaterthanqbullet-isotoBdRdagplusGbulletstar}}}{\cong}
    \BB_{\dR}^{\dag,+}\left(\cal{G}^{\bullet}\times*\right),
  \end{equation*}
  where Lemma~\ref{lem:decalage-completed-localisation-is-just-completed-localisation} applies
  because of Lemma~\ref{lem:galois-cohomology-of-BdRdaggerplus-primitiverootofunity-aaaaaletaoperatorconditions}.
  Similarly, the codomain is  
  \begin{equation*}
    \text{``}\varinjlim_{q\in\NN_{\geq1}}\text{"}\eta_{p^{N_{-j}}}D_{-j}^{>q,\bullet}\widehat{\otimes}_{W(\kappa)}k_{0}
    \stackrel{\text{\ref{lem:decalage-completed-localisation-is-just-completed-localisation}}}{\cong}
    \text{``}\varinjlim_{q\in\NN_{\geq1}}\text{"}D_{-j}^{>q,\bullet}\widehat{\otimes}_{W(\kappa)}k_{0}
    \stackrel{\text{\ref{lem:galois-cohomology-of-BdRdagger-Dgreaterthanqbullet-j-isotoBdRdagplusGbulletstar}}}{\cong}
    t^{-j}\BB_{\dR}^{\dag,+}\left(\cal{G}^{\bullet}\times*\right),
  \end{equation*}
  where Lemma~\ref{lem:decalage-completed-localisation-is-just-completed-localisation} applies
  by Lemma~\ref{lem:galois-cohomology-of-BdRdaggerplus-primitiverootofunity-j-aaaaaletaoperatorconditions}.
  We have thus found 
  that~(\ref{eq:galois-cohomology-of-BdRdaggerplus-primitiverootofunity-j-whatistheretoshow})
  is a quasi-isomorphism of cochain complexes of
  $k_{0}$-ind-Banach spaces. Proposition~\ref{prop:galois-cohomology-of-BdRdaggerplus-kprime-injectintotminusj}
  follows by Lemma~\ref{lem:galois-cohomology-of-BdRdaggerplus-primitiverootofunity-j-whatistheretoshow}.
\end{proof}

%%%%%%%%%%%%%%%%%%%%%%%%%%%%%%%%%%%%%%%%%%%%%%%%%%%%%%%%%%%%
% On Proposition~\ref{prop:galois-cohomology-of-BdRdaggerplus-primitiverootofunity}
%%%%%%%%%%%%%%%%%%%%%%%%%%%%%%%%%%%%%%%%%%%%%%%%%%%%%%%%%%%%

\subsection{Proof of Theorem~\ref{thm:galois-cohomology-of-solidBdRdagger-born}}\label{subsubsec:proofthm:galois-cohomology-of-solidBdRdagger}

Recall the computation of the continuous Galois cohomology of $\underline{B}_{\dR}^{\dag,+}$
as in Corollary~\ref{cor:galois-cohomology-of-solidBdRdaggerplus-primitiverootofunity}.

\begin{cor}\label{cor:galois-cohomology-of-BdRdaggerplus-kprime-injectintotminusj-solid}
  Given $j\in\NN$, the canonical morphism
  $B_{\dR}^{\dag,+}\to t^{-j}B_{\dR}^{\dag,+}$ induces
  \begin{equation*}
    \Ho_{\cont}^{i}\left(\Gal\left( \overline{k} / k^{\prime} \right) , \underline{t^{-j}B_{\dR}^{\dag,+}} \right)
    \cong
    \begin{cases}
      \underline{k^{\prime}}, &\text{ if $i=0,1$ and} \\
      0, &\text{otherwise}.
    \end{cases}
  \end{equation*}
\end{cor}

\begin{proof}
  Apply Lemma~\ref{lem:contgpcoh-indban-vs-solid-reconstructionpaper}
  to Proposition~\ref{prop:galois-cohomology-of-BdRdaggerplus-kprime-injectintotminusj}.
\end{proof}

\begin{cor}\label{cor:galois-cohomology-of-BdRdaggerplus-k-injectintotminusj-solid}
  Given $j\in\NN$, the canonical morphism
  $B_{\dR}^{\dag,+}\to t^{-j}B_{\dR}^{\dag,+}$
  induces
  \begin{equation*}
    \Ho_{\cont}^{i}\left(\Gal\left( \overline{k} / k \right) , \underline{t^{-j}B_{\dR}^{\dag,+}} \right)
    \cong
    \begin{cases}
      \underline{k}, &\text{ if $i=0,1$ and} \\
      0, &\text{otherwise}.
    \end{cases}
  \end{equation*}
\end{cor}

\begin{proof}
  Proceed as in the proof of Theorem~\ref{thm:galois-cohomology-of-solidBdRdaggerplus},
  applying Corollary~\ref{cor:galois-cohomology-of-BdRdaggerplus-kprime-injectintotminusj-solid}
  instead of~\ref{cor:galois-cohomology-of-solidBdRdaggerplus-primitiverootofunity}.
\end{proof}

\begin{thm}\label{thm:galois-cohomology-of-BdRdaggerplus-k-injectintotminusj}
  Given $j\in\NN$, the canonical morphism
  $B_{\dR}^{\dag,+}\to t^{-j}B_{\dR}^{\dag,+}$
  induces
  \begin{equation*}
    \Ho_{\cont}^{i}\left(\Gal\left( \overline{k} / k \right) , t^{-j}B_{\dR}^{\dag,+} \right)
    \cong
    \begin{cases}
      \I\left(k\right), &\text{ if $i=0,1$ and} \\
      0, &\text{otherwise}.
    \end{cases}
  \end{equation*}
\end{thm}

%The proof of Theorem~\ref{thm:galois-cohomology-of-BdRdaggerplus-k-injectintotminusj}
%is somewhat identical to our proof of Theorem~\ref{thm:galois-cohomology-of-BdRdaggerplus-born}
%one page~\pageref{proof:thm:galois-cohomology-of-BdRdaggerplus-born}.

\begin{proof}%[Proof of Theorem~\ref{thm:galois-cohomology-of-BdRdaggerplus-k-injectintotminusj}]
  Recall Definition~\ref{defn:indBanachmodule-Ccontcomplex-reconstructionpaper}
  and~\ref{defn:indBan-ctsRGamma-recpaper}.
  We have to check that
  \begin{equation*}
    C_{\cont}^{\bullet}\left(\Gal\left( \overline{k} / k\right),B_{\dR}^{\dag,+}\right)
    \to C_{\cont}^{\bullet}\left(\Gal\left( \overline{k} / k\right), t^{-j}B_{\dR}^{\dag,+}\right)
  \end{equation*}
  is a quasi-isomorphism of cochain complexes of $k_{0}$-ind-Banach spaces.
  By Lemma~\ref{lem:quasiiso-if-cone-strictlyexact-reconstructionpaper},
  this is equivalent to checking that its mapping cone $K^{\bullet}$ is strictly exact.
  We would like to use that
  \begin{equation*}
    \underline{K^{\bullet}}\stackrel{\text{\ref{lem:intHomunderlineSunderlineV-is-underlineHomcontSV-reconstructionpaper}}}{=}
    \cone\left(C_{\cont}^{\bullet}\left(\Gal\left( \overline{k} / k\right),\underline{B}_{\dR}^{\dag,+}\right)
      \to C_{\cont}^{\bullet}\left(\Gal\left( \overline{k} / k\right),\underline{t^{-j}B_{\dR}^{\dag,+}}\right)\right)
  \end{equation*}
  is exact, which follows from Corollary~\ref{cor:galois-cohomology-of-BdRdaggerplus-k-injectintotminusj-solid}
  and the arguments in the proof of Lemma~\ref{lem:contgpcoh-indban-vs-solid-reconstructionpaper}.
  
  We note that for any profinite set $S$, the functor $\intHom_{\cont}\left(S,-\right)$ preserves injective maps between
  Banach spaces. Therefore, the proof of Theorem~\ref{thm:BdRdagRRplus-bornological}
  one page~\pageref{proof:thm:BdRdagRRplus-bornological} implies
  that $K^{\bullet}$ is a cochain complex of bornological $k_{0}$-vector spaces whose bornology has a countable basis,
  cf. Definition~\ref{defn:bornologicalspace-reconstructionpaper} and
  \ref{defn:bornology-countable-basis-reconstructionpaper}.
  Therefore, Proposition~\ref{prop:solidexact-implies-indBanachstrictlyexact-reconstructionpaper}
  implies that $K^{\bullet}$ is strictly exact, as desired.
\end{proof}

\begin{proof}[Proof of Theorem~\ref{thm:galois-cohomology-of-solidBdRdagger-born}]\label{proof:thm:galois-cohomology-of-solidBdRdagger-born}
  Since $B_{\dR}^{\dag}=\text{``}\varinjlim\text{''}_{j\geq0}t^{-j}B_{\dR}^{\dag,+}$,
  Theorem~\ref{thm:galois-cohomology-of-BdRdaggerplus-k-injectintotminusj}
  and Lemma~\ref{lem:indbanachcontcohomology-commuteswith-filteredcolimits-reconstructionpaper}
  apply.
\end{proof}

Finally, we record the following version of Theorem~\ref{thm:galois-cohomology-of-solidBdRdagger-born}
in the solid setting.

\begin{thm}\label{thm:galois-cohomology-of-solidBdRdagger-solid}
  \begin{equation*}
    \Ho_{\cont}^{i}\left(\Gal\left( \overline{k} / k \right) , \underline{B}_{\dR}^{\dag} \right)
    \cong
    \begin{cases}
      \underline{k}, &\text{ if $i=0,1$ and} \\
      0, &\text{otherwise}.
    \end{cases}
  \end{equation*}
\end{thm}

\begin{proof}
  Since $\underline{B}_{\dR}^{\dag}=\varinjlim_{j\geq0}\underline{t^{-j}B_{\dR}^{\dag,+}}$,
  Corollary~\ref{cor:galois-cohomology-of-BdRdaggerplus-k-injectintotminusj-solid}
  and Lemma~\ref{lem:formalcolimits-and-solid-group-cohomology-reconstructionpaper}
  give the result.
\end{proof}

%%%%%%%%%%%%%%%%%%%%%%%%%%%%%%%%%%%%%%%%%%%%%%%%%%%%%%%%%%%%
%%%%%%%%%%%%%%%%%%%%%%%%%%%%%%%%%%%%%%%%%%%%%%%%%%%%%%%%%%%%
% Cech cohomology III
%%%%%%%%%%%%%%%%%%%%%%%%%%%%%%%%%%%%%%%%%%%%%%%%%%%%%%%%%%%%
%%%%%%%%%%%%%%%%%%%%%%%%%%%%%%%%%%%%%%%%%%%%%%%%%%%%%%%%%%%%

\section{The Galois cohomology of $B_{\pdR}^{\dag,+}$}
\label{subsec:galoiscoh-of-positiveoverconvalmostdRperiodring-recpaper}

\begin{thm}\label{thm:galois-cohomology-of-BpdRdaggerplus}
  The canonical morphism
  $k\isomap\R\Gamma_{\cont}\left(\Gal\left( \overline{k} / k \right) , B_{\pdR}^{\dag,+} \right)$
  is an isomorphism.
\end{thm}

We split the proof of Theorem~\ref{thm:galois-cohomology-of-BpdRdaggerplus} into two parts.
Firstly, we prove that it is an isomorphism in degree zero,
cf. \S\ref{subsubsec:proofthm:galois-cohomology-of-BpdRdaggerplus-degree0}.
In \S\ref{subsubsec:proofthm:galois-cohomology-of-BpdRdaggerplus-positivedegree},
we check the vanishing of $\Ho_{\cont}^{i}\left(\Gal\left( \overline{k} / k \right) , B_{\pdR}^{\dag,+} \right)$
for $i>0$.
Finally, everything comes together in \S\ref{subsubsec:proofthm:galois-cohomology-of-BpdRdaggerplus}.

\iffalse
\begin{thm}\label{thm:galois-cohomology-of-BpdRdaggerplus}
  The canonical morphism
  $\underline{k}\isomap\R\Gamma_{\cont}\left(\Gal\left( \overline{k} / k \right) , \underline{B}_{\pdR}^{\dag,+} \right)$
  is an isomorphism.
\end{thm}

We split the proof of Theorem~\ref{thm:galois-cohomology-of-BpdRdaggerplus} into two parts.
Firstly, we prove that it is an isomorphism in degree zero in \S\ref{subsubsec:proofthm:galois-cohomology-of-BpdRdaggerplus-degree0}.
In \S\ref{subsubsec:proofthm:galois-cohomology-of-BpdRdaggerplus-positivedegree},
we check the vanishing of $\Ho_{\cont}^{i}\left(\Gal\left( \overline{k} / k \right) , \underline{B}_{\pdR}^{\dag,+} \right)$
for $i>0$.
Finally, everything comes together in \S\ref{subsubsec:proofthm:galois-cohomology-of-BpdRdaggerplus}.
\fi

%%%%%%%%%%%%%%%%%%%%%%%%%%%%%%%%%%%%%%%%%%%%%%%%%%%%%%%%%%%%
%%%%%%%%%%%%%%%%%%%%%%%%%%%%%%%%%%%%%%%%%%%%%%%%%%%%%%%%%%%%
% Cech cohomology III
%%%%%%%%%%%%%%%%%%%%%%%%%%%%%%%%%%%%%%%%%%%%%%%%%%%%%%%%%%%%
%%%%%%%%%%%%%%%%%%%%%%%%%%%%%%%%%%%%%%%%%%%%%%%%%%%%%%%%%%%%

\subsection{Computations in degree 0}\label{subsubsec:proofthm:galois-cohomology-of-BpdRdaggerplus-degree0}

Fix $k^{\prime}$ as in
Notation~\ref{notation:galois-cohomology-of-BdRdaggerplus-kzero}
and write $\cal{G}:=\Gal\left(\overline{k}/k^{\prime}\right)$.

\begin{lem}\label{lem:GaloisinvarianceAdRgreathanqlogtninvertp-gives-GaloisinvarianceBdRgreathanqlogt-reconstructionpaper}
  Let $n\in\NN$. For all $q\geq 2$, we have the canonical isomorphism
  \begin{equation*}
    \left(A_{\dR}^{>q,+}\left[\log t\right]^{\leq n}\right)^{\cal{G}}\widehat{\otimes}_{W(\kappa)}k_{0}
    \isomap
    \left(B_{\dR}^{>q,+}\left[\log t\right]^{\leq n}\right)^{\cal{G}}.
  \end{equation*}
\end{lem}

\begin{proof}
  Following Definition~\ref{defn:Ginv-recpaper},
  \begin{equation*}
    \left(A_{\dR}^{>q,+}\left[\log t\right]^{\leq n}\right)^{\cal{G}}
    =\ker\left(A_{\dR}^{>q,+}\left[\log t\right]^{\leq n}\longrightarrow
      \intHom_{\cont}\left( \cal{G} , A_{\dR}^{>q,+}\left[\log t\right]^{\leq n} \right)\right).
  \end{equation*}
  Now apply Corollary~\ref{cor:completed-localisation-strictlyexact};
  this is allowed because of
  Lemma~\ref{lem:AdRgreaterthanqptorsionfree}.
  We find
  \begin{align*}
    &\left(A_{\dR}^{>q,+}\left[\log t\right]^{\leq n}\right)^{\cal{G}}\widehat{\otimes}_{W(\kappa)}k_{0} \\
    &\cong
      \ker\left(A_{\dR}^{>q,+}\left[\log t\right]^{\leq n} \widehat{\otimes}_{W(\kappa)}k_{0} \longrightarrow
      \intHom_{\cont}\left( \cal{G} , A_{\dR}^{>q,+}\left[\log t\right]^{\leq n} \right) \widehat{\otimes}_{W(\kappa)}k_{0} \right).
  \end{align*}
  Next, apply Lemma~\ref{lem:HomcontS-commutes-completed-localisation-overWkappa-reconstructionpaper};
  this is allowed, because of Lemma~\ref{lem:Agreaterthanq-multiplybyp-norm-reconstructionpaper}.
  We find
  \begin{align*}
    &\left(A_{\dR}^{>q,+}\left[\log t\right]^{\leq n}\right)^{\cal{G}}\widehat{\otimes}_{W(\kappa)}k_{0} \\
    &\cong
      \ker\left(A_{\dR}^{>q,+}\left[\log t\right]^{\leq n} \widehat{\otimes}_{W(\kappa)}k_{0} \longrightarrow
      \intHom_{\cont}\left( \cal{G} , A_{\dR}^{>q,+}\left[\log t\right]^{\leq n}\widehat{\otimes}_{W(\kappa)}k_{0} \right) \right) \\
    &\cong
      \ker\left(B_{\dR}^{>q,+}\left[\log t\right]^{\leq n} \longrightarrow
      \intHom_{\cont}\left( \cal{G} , B_{\dR}^{>q,+}\left[\log t\right]^{\leq n} \right) \right) \\
    &\cong\left(B_{\dR}^{>q,+}\left[\log t\right]^{\leq n}\right)^{\cal{G}},
  \end{align*}
  as desired.
\end{proof}

\begin{lem}\label{lem:galois-cohomology-of-BpdRdaggerplus-keylemma1-kprime-thefinallemma-reconstructionpaper}
  $\cal{O}_{C}(s)$ denotes the $s$th Tate twist of $\cal{O}_{C}$ for all $s\in\ZZ$. Let $n\in\NN$ and consider
  \begin{equation*}
    \psi_{s}\colon
    \cal{O}_{C}(s)\left[\log t\right]^{\leq n+1} \longrightarrow
    \cal{O}_{C}(s),
    \sum_{\alpha=0}^{n+1}\lambda_{\alpha}\left(\log t\right)^{\alpha}\mapsto \lambda_{n+1}.
  \end{equation*}
  If $\lambda\in\cal{O}_{C}(s)\left[\log t\right]^{\leq n+1}$ is
  $\cal{G}$-invariant, then $\psi_{s}\left(\lambda\right)=0$.
\end{lem}

\begin{proof}
  Writing $\lambda=\sum_{\alpha=0}^{n+1}\lambda_{\alpha}\left(\log t\right)^{\alpha}$,
  we have to check $\lambda_{n+1}=0$. We distinguish two cases:
  \begin{itemize}
    \item Let $s\neq0$. Because $\lambda$ is $\cal{G}$-invariant, $\lambda_{n+1}$
      is $\cal{G}$-invariant. By~\cite[Theorem 1 and 2]{tatepdivisiblegroups}, this implies $\lambda_{n+1}=0$.
    \item Let $s=0$. Note that $\lambda$ is an element of $C\left[\log t\right]^{\leq n+1}$,
    which~\cite{Fontaine2004Arithmetic} denotes by $C(0;n+1)$. Now apply \emph{loc. cit.} Proposition 2.15
    to deduce $\lambda=\lambda_{0}$. In particular, $\lambda_{n+1}=0$.
  \end{itemize}
  This finishes the proof of
  Lemma~\ref{lem:galois-cohomology-of-BpdRdaggerplus-keylemma1-kprime-thefinallemma-reconstructionpaper}.
\end{proof}

\begin{lem}\label{lem:galois-cohomology-of-BpdRdaggerplus-keylemma1-kprime}
  For any $n\in\NN$, the canonical
  $k^{\prime}\to B_{\dR}^{\dag,+}\hookrightarrow B_{\dR}^{\dag,+}\left[\log t\right]^{\leq n}$ induces
  the isomorphism
  \begin{equation*}
    k^{\prime}\isomap\left(B_{\dR}^{\dag,+}\left[\log t\right]^{\leq n}\right)^{\cal{G}}.
  \end{equation*}
\end{lem}

\begin{proof}
  We proceed via induction along $n$. Lemma~\ref{lem:galois-cohomology-of-BpdRdaggerplus-keylemma1-kprime} holds
  for $n=0$ by Theorem~\ref{thm:galois-cohomology-of-solidBdRdaggerplus-born},
  because of $B_{\dR}^{\dag,+}\left[\log t\right]^{\leq 0}=B_{\dR}^{\dag,+}$
  and Lemma~\ref{lem:indBanachcontgpcohzero-is-invariance},
  which applies thanks to the fully faithfulness of $\I$,
  cf.~\cite[Corollary 1.2.28]{Sch99}.
  Now assume the result holds for fixed $n\in\NN$. For all $q\in\NN$, we have the short exact sequence
  \begin{equation*}
    0 \longrightarrow
    B_{\dR}^{>q,+}\left[\log t\right]^{\leq n} \longrightarrow
    B_{\dR}^{>q,+}\left[\log t\right]^{\leq n+1} \stackrel{\varphi^{>q}}{\longrightarrow}
    B_{\dR}^{>q,+} \longrightarrow
    0
  \end{equation*}
  of $\cal{G}$-$k^{\prime}$-Banach modules, where
  $\varphi^{>q}\colon B_{\dR}^{>q,+}\left[\log t\right]^{\leq n+1}\to B_{\dR}^{>q,+}$,
  $\sum_{\alpha=0}^{n+1}b_{\alpha}\left(\log t\right)^{\alpha}\mapsto b_{n+1}$.
  This short exact sequence is strictly exact by the open mapping theorem.
  Now pass to the colimit along $q\to\infty$ and recall
  Corollary~\ref{cor:filteredcol-inIndBan-stronglyexact}
   to get the strictly exact sequence
   \begin{equation}\label{eq:galois-cohomology-of-BpdRdaggerplus-keylemma1-kprime-strictlyshortexactsequence}
      0 \longrightarrow
      B_{\dR}^{\dag,+}\left[\log t\right]^{\leq n} \longrightarrow
      B_{\dR}^{\dag,+}\left[\log t\right]^{\leq n+1} \stackrel{\varphi}{\longrightarrow}
      B_{\dR}^{\dag,+} \longrightarrow
      0.
  \end{equation}
  Next, apply Lemma~\ref{lem:indBanachinvariance-cocont-leftexact} to get the strictly exact sequence
  \begin{equation*}
    0 \longrightarrow
    \left(B_{\dR}^{\dag,+}\left[\log t\right]^{\leq n}\right)^{\cal{G}} \stackrel{f}{\longrightarrow}
    \left(B_{\dR}^{\dag,+}\left[\log t\right]^{\leq n+1}\right)^{\cal{G}} \stackrel{g}{\longrightarrow}
    \left(B_{\dR}^{\dag,+}\right)^{\cal{G}}
  \end{equation*}  
  of $k^{\prime}$-ind-Banach modules. We claim that $g$ is zero,
  as this would imply that $f$ is an isomorphism. In this case, the induction hypothesis would give
  \begin{equation*}
    k^{\prime}\cong\left(B_{\dR}^{\dag,+}\left[\log t\right]^{\leq n}\right)^{\cal{G}}
    \stackrel{f}{\cong}\left(B_{\dR}^{\dag,+}\left[\log t\right]^{\leq n+1}\right)^{\cal{G}},
  \end{equation*}
  which would finish the proof of Lemma~\ref{lem:galois-cohomology-of-BpdRdaggerplus-keylemma1-kprime}.

  To show that $g$ is zero,
  we note with Lemma~\ref{lem:indBanachinvariance-cocont-leftexact} that $g$ is the colimit of the morphisms
  \begin{equation*}
    g^{>q}\colon\left(B_{\dR}^{>q,+}\left[\log t\right]^{\leq n+1}\right)^{\cal{G}} \longrightarrow
    \left(B_{\dR}^{>q,+}\right)^{\cal{G}},
    \lambda\mapsto\varphi^{>q}\left(\lambda\right).
  \end{equation*}
  We may therefore check $g^{>q}=0$ for large $q$.
  By Lemma~\ref{lem:GaloisinvarianceAdRgreathanqlogtninvertp-gives-GaloisinvarianceBdRgreathanqlogt-reconstructionpaper},
  we may also prove that the morphisms
  \begin{equation*}
    \left(A_{\dR}^{>q}\left[\log t\right]^{\leq n+1}\right)^{\cal{G}} \longrightarrow
    \left(A_{\dR}^{>q}\right)^{\cal{G}},
  \end{equation*}  
  are zero, which are induced by the morphisms
  $\phi^{>q}\colon A_{\dR}^{>q}\left[\log t\right]^{\leq n+1}\to A_{\dR}^{>q}$,
  $\sum_{\alpha=0}^{n+1}a_{\alpha}\left(\log t\right)^{\alpha}\mapsto a_{n+1}$.
  This is a statement about abstract $\cal{G}$-representations, thus we can introduce filtrations
  which do not recover the topologies on $A_{\dR}^{>q}$ and $A_{\dR}^{>q}\left[\log t\right]^{\leq n+1}$.
  We choose to work with the filtered ring $\widetilde{A}_{\dR}^{>q}$, which coincides
  with $A_{\dR}^{>q}$ as an abstract $\cal{G}$-representation, but it carries the $\xi/p^{q}$-adic filtration.
  Consequently, may consider the $\cal{G}$-representation $\widetilde{A}_{\dR}^{>q}\left[\log t\right]^{\leq n+1}$,
  which coincides with $A_{\dR}^{>q}\left[\log t\right]^{\leq n+1}$ as an abstract $\cal{G}$-representation,
  but it carries the filtration
  \begin{equation*}
    \Fil^{i}\left(\widetilde{A}_{\dR}^{>q}\left[\log t\right]^{\leq n+1}\right)
    =\Fil^{i}\left(\widetilde{A}_{\dR}^{>q}\right)\left[\log t\right]^{\leq n+1} \,
    \text{ for all $i\in\ZZ$.}
  \end{equation*}
  Now we have to check that the morphism
  \begin{equation*}
    \widetilde{\phi}^{>q}\colon \widetilde{A}_{\dR}^{>q}\left[\log t\right]^{\leq n+1} \longrightarrow
    \widetilde{A}_{\dR}^{>q}, \sum_{\alpha=0}^{n+1}a_{\alpha}\left(\log t\right)^{\alpha}\mapsto a_{n+1}
  \end{equation*}
  sends any invariant
  $a=\sum_{\alpha=0}^{n+1}a_{\alpha}\left(\log t\right)^{\alpha}\in\widetilde{A}_{\dR}^{>q}\left[\log t\right]^{\leq n+1}$
  to zero.
  With other words, one has to check that $a_{n+1}=0$. As the filtration on
  $\widetilde{A}_{\dR}^{>q}$ is separated,
  cf. Lemma~\ref{lem:grAdRgreaterthanq-xioverpadicfiltration-reconstructionpaper}(i), this is the case if
  the principal symbol $\sigma\left(a_{n+1}\right)\in\gr\widetilde{A}_{\dR}^{>q}$ is zero.
  This happens if $\sigma(a)$ is in the kernel of
  \begin{equation*}
    \gr\widetilde{\varphi}^{>q}\colon
    \gr\widetilde{A}_{\dR}^{>q}\left[\log t\right]^{\leq n+1} \longrightarrow
    \gr\widetilde{A}_{\dR}^{>q}.
  \end{equation*}
  As $\sigma(a)$ is again $\cal{G}$-invariant, one may show the
  following: $\gr\widetilde{\varphi}^{>q}$ sends any $\cal{G}$-invariant element to zero. 
  To show this, we further describe the associated graded with
  Lemma~\ref{lem:grAdRgreaterthanq-xioverpadicfiltration-reconstructionpaper}(ii)
  to get
  \begin{equation*}
    \gr\widetilde{\varphi}^{>q}\colon
    \cal{O}_{C}\left[x\right]\left[\log t\right]^{\leq n+1} \longrightarrow
    \cal{O}_{C}\left[x\right],
    \sum_{\alpha=0}^{n+1}\lambda_{\alpha}\left(\log t\right)^{\alpha}\mapsto \lambda_{n+1}.
  \end{equation*}
  Here, $x$ denotes the principal symbol of $\xi/p^{q}$. From the
  isomorphism~(\ref{calOCxnconggrnAdRgreaterthanqstar})
  in the proof of Lemma~\ref{lem:galois-cohomology-of-BdRdagger-Dgreaterthanqbulletcomputedgrn},
  we get $\cal{O}_{C}\left[x\right]\cong\bigoplus_{s\geq0}\cal{O}_{C}(s)$.
  In particular, $\gr\widetilde{\varphi}^{>q}$ is the direct sum of the maps
  \begin{equation*}
    \psi_{s}\colon
    \cal{O}_{C}(s)\left[\log t\right]^{\leq n+1} \longrightarrow
    \cal{O}_{C}(s),
    \sum_{\alpha=0}^{n+1}\lambda_{\alpha}\left(\log t\right)^{\alpha}\mapsto \lambda_{n+1}.
  \end{equation*}
  of $\cal{G}$-representations. Now apply Lemma~\ref{lem:galois-cohomology-of-BpdRdaggerplus-keylemma1-kprime-thefinallemma-reconstructionpaper}.
  Thus $\gr\widetilde{\varphi}^{>q}$ kills the $\cal{G}$-invariant elements.
\end{proof}

\begin{lem}\label{lem:galois-cohomology-of-BpdRdaggerplus-keylemma1}
  For any $n\in\NN$, the canonical $k\to B_{\dR}^{\dag,+}\hookrightarrow B_{\dR}^{\dag,+}\left[\log t\right]^{\leq n}$ induces
  the isomorphism
  \begin{equation*}
    \I\left(k\right)\isomap\Ho_{\cont}^{0}\left(\Gal\left( \overline{k} / k \right) , B_{\dR}^{\dag,+}\left[\log t\right]^{\leq n} \right).
  \end{equation*}
\end{lem}

\begin{proof}
  Fix the intermediate extension $\overline{k} / k^{\prime} / k$
  as in Notation~\ref{notation:galois-cohomology-of-BdRdaggerplus-kzero}.
  Now compute
  \begin{align*}
    \left(B_{\dR}^{\dag,+}\left[\log t\right]^{\leq n}\right)^{\Gal\left( \overline{k} / k\right)}
    \cong
    \left(\left(B_{\dR}^{\dag,+}\left[\log t\right]^{\leq n}\right)^{\Gal\left( \overline{k} / k^{\prime}\right)}\right)^{\Gal\left( k^{\prime} / k\right)}
    \stackrel{\text{\ref{lem:galois-cohomology-of-BpdRdaggerplus-keylemma1-kprime}}}{\cong}
    \left(k^{\prime}\right)^{\Gal\left( k^{\prime} / k\right)}=k
  \end{align*}
  and apply Lemma~\ref{lem:indBanachcontgpcohzero-is-invariance}.
\end{proof}

We proceed with the solid formalism such that the long exact
sequence~(\ref{lem:galois-cohomology-of-BpdRdaggerplus-keylemma2-leq-recpaper}) below is available.

\begin{notation}\label{notation:underlineBdRdaggerpluslogtn-reconstructionpaper}
For any $n\in\NN$, write $\underline{B}_{\dR}^{\dag,+}\left[\log t\right]^{\leq n}:=\underline{B_{\dR}^{\dag,+}\left[\log t\right]^{\leq n}}$.
\end{notation}

\begin{prop}\label{prop:galois-cohomology-of-BpdRdaggerplus-keylemma1-solid}
  For any $n\in\NN$, the canonical $k\to B_{\dR}^{\dag,+}\hookrightarrow B_{\dR}^{\dag,+}\left[\log t\right]^{\leq n}$ induces
  \begin{equation*}
    \underline{k}\isomap\Ho_{\cont}^{0}\left(\Gal\left( \overline{k} / k \right) , \underline{B}_{\dR}^{\dag,+}\left[\log t\right]^{\leq n} \right).
  \end{equation*}
\end{prop}

\begin{proof}
  Apply Lemma~\ref{lem:contgpcoh-indban-vs-solid-reconstructionpaper}
  to Lemma~\ref{lem:galois-cohomology-of-BpdRdaggerplus-keylemma1}.
\end{proof}

%%%%%%%%%%%%%%%%%%%%%%%%%%%%%%%%%%%%%%%%%%%%%%%%%%%%%%%%%%%%
%%%%%%%%%%%%%%%%%%%%%%%%%%%%%%%%%%%%%%%%%%%%%%%%%%%%%%%%%%%%
% Cech cohomology III
%%%%%%%%%%%%%%%%%%%%%%%%%%%%%%%%%%%%%%%%%%%%%%%%%%%%%%%%%%%%
%%%%%%%%%%%%%%%%%%%%%%%%%%%%%%%%%%%%%%%%%%%%%%%%%%%%%%%%%%%%

\subsection{Computations in positive degree}\label{subsubsec:proofthm:galois-cohomology-of-BpdRdaggerplus-positivedegree}

Recall Notation~\ref{notation:underlineBdRdaggerpluslogtn-reconstructionpaper}.

\begin{lem}\label{lem:galois-cohomology-of-BpdRdaggerplus-keylemma2}
  For any $n\in\NN$, we compute the following:
  \begin{itemize}
    \item[(i)] $\Ho_{\cont}^{1}\left(\Gal\left( \overline{k} / k \right) , \underline{B}_{\dR}^{\dag,+}\left[\log t\right]^{\leq n} \right)=\underline{k}$, and
    \item[(ii)] $\Ho_{\cont}^{i}\left(\Gal\left( \overline{k} / k \right) , \underline{B}_{\dR}^{\dag,+}\left[\log t\right]^{\leq n} \right)=0$
    for all $i\geq2$.
  \end{itemize}
  Furthermore, the following canonical morphism is zero:
  \begin{equation*}
    \Ho_{\cont}^{1}\left(\Gal\left( \overline{k} / k \right) , \underline{B}_{\dR}^{\dag,+}\left[\log t\right]^{\leq n} \right)
    \to
    \Ho_{\cont}^{1}\left(\Gal\left( \overline{k} / k \right) , \underline{B}_{\dR}^{\dag,+}\left[\log t\right]^{\leq n+1} \right).
  \end{equation*}
\end{lem}

\begin{proof}
  We proceed via induction along $n$. Both (i) and (ii) hold for $n=0$ by
  Theorem~\ref{thm:galois-cohomology-of-solidBdRdaggerplus},
  because $\underline{B}_{\dR}^{\dag,+}\left[\log t\right]^{\leq 0}=\underline{B}_{\dR}^{\dag,+}$.
  Now assume (i) and (ii) hold for fixed $n$. We have the short exact sequence
  \begin{equation*}
    0 \longrightarrow
    B_{\dR}^{\dag,+}\left[\log t\right]^{\leq n} \longrightarrow
    B_{\dR}^{\dag,+}\left[\log t\right]^{\leq n+1} \stackrel{\varphi}{\longrightarrow}
    B_{\dR}^{\dag,+} \longrightarrow
    0
  \end{equation*}
  of $\Gal\left(\overline{k}/k\right)$-$k$-Banach modules,
  cf.~(\ref{eq:galois-cohomology-of-BpdRdaggerplus-keylemma1-kprime-strictlyshortexactsequence})
  in the proof of Lemma~\ref{lem:galois-cohomology-of-BpdRdaggerplus-keylemma1-kprime}.
  Now apply Lemma~\ref{lem:IndBan-to-Solid-strictlyexact-reconstructionpaper}
  to find the short exact sequence
  \begin{equation}\label{eq:galois-cohomology-of-BpdRdagger-keylemma2-someshortexactsequence}
    0 \longrightarrow
    \underline{B}_{\dR}^{\dag,+}\left[\log t\right]^{\leq n} \longrightarrow
    \underline{B}_{\dR}^{\dag,+}\left[\log t\right]^{\leq n+1} \stackrel{\underline{\varphi}}{\longrightarrow}
    \underline{B}_{\dR}^{\dag,+} \longrightarrow
    0
  \end{equation}  
  of solid $k$-vector spaces with continuous $G$-action.
  In this setting, the continuous group cohomology is defined
  as an honest derived functor. Therefore,~(\ref{eq:galois-cohomology-of-BpdRdagger-keylemma2-someshortexactsequence})
  gives rise to a long exact sequence 
  \begin{equation}\label{lem:galois-cohomology-of-BpdRdaggerplus-keylemma2-leq-recpaper}
  \begin{tikzcd}
    0 \arrow{r} &
    \Ho^{0}\left(\underline{B}_{\dR}^{\dag,+}\left[\log t\right]^{\leq n}\right) \arrow{r}{f^{0}} &
    \Ho^{0}\left(\underline{B}_{\dR}^{\dag,+}\left[\log t\right]^{\leq n+1}\right) \arrow{r}{g^{0}} &
    \Ho^{0}\left(\underline{B}_{\dR}^{\dag,+}\right) \\
    \empty \arrow{r}{\delta^{0}} &
    \Ho^{1}\left(\underline{B}_{\dR}^{\dag,+}\left[\log t\right]^{\leq n}\right) \arrow{r}{f^{1}} &
    \Ho^{1}\left(\underline{B}_{\dR}^{\dag,+}\left[\log t\right]^{\leq n+1}\right) \arrow{r}{g^{1}} &
    \Ho^{1}\left(\underline{B}_{\dR}^{\dag,+}\right) \\
    \empty \arrow{r}{\delta^{1}} &
    \Ho^{2}\left(\underline{B}_{\dR}^{\dag,+}\left[\log t\right]^{\leq n}\right) \arrow{r}{f^{2}} &
    \Ho^{2}\left(\underline{B}_{\dR}^{\dag,+}\left[\log t\right]^{\leq n+1}\right) \arrow{r}{g^{2}} &
    \Ho^{2}\left(\underline{B}_{\dR}^{\dag,+}\right) \\
    \empty \arrow{r}{\delta^{2}} &
    \dots, &&&
  \end{tikzcd}
  \end{equation}
  where we abbreviated $\Ho^{*}=\Ho_{\cts}^{*}\left(\Gal\left(\overline{k}/k\right),-\right)$.
  By the induction hypothesis, all the $\Ho^{i}$ vanish for $i\geq 2$. This gives (ii).
  Proposition~\ref{prop:galois-cohomology-of-BpdRdaggerplus-keylemma1-solid} implies that $f^{0}$ is an isomorphism,
  implying $g^{0}=0$. Thus $\delta^{0}$ is a monomorphism,
  and the induction hypothesis implies $\Ho^{1}\left(B_{\dR}^{\dag,+}\left[\log t\right]^{\leq n}\right)\cong k$,
  which gives $\delta^{0}$ is an isomorphism. Then $f^{1}$ is the zero map such that
  $g^{1}$ is a monomorphism. But $g^{1}$ is an epimorphism because the $\Ho^{2}$ vanish. Thus
  $g^{1}$ is an isomorphism. Apply Theorem~\ref{thm:galois-cohomology-of-solidBdRdaggerplus}
  to find (i).
  
  Along the way, we deduced $f^{1}=0$.
  This settles the final sentence of Lemma~\ref{lem:galois-cohomology-of-BpdRdaggerplus-keylemma2}.
\end{proof}

%%%%%%%%%%%%%%%%%%%%%%%%%%%%%%%%%%%%%%%%%%%%%%%%%%%%%%%%%%%%
%%%%%%%%%%%%%%%%%%%%%%%%%%%%%%%%%%%%%%%%%%%%%%%%%%%%%%%%%%%%
% Cech cohomology III
%%%%%%%%%%%%%%%%%%%%%%%%%%%%%%%%%%%%%%%%%%%%%%%%%%%%%%%%%%%%
%%%%%%%%%%%%%%%%%%%%%%%%%%%%%%%%%%%%%%%%%%%%%%%%%%%%%%%%%%%%

\subsection{Proof of Theorem~\ref{thm:galois-cohomology-of-BpdRdaggerplus}}\label{subsubsec:proofthm:galois-cohomology-of-BpdRdaggerplus}

We can now prove the main result of \S\ref{subsec:galoiscoh-of-positiveoverconvalmostdRperiodring-recpaper}.

\begin{proof}[Proof of Theorem~\ref{thm:galois-cohomology-of-BpdRdaggerplus}]
  We have to check that the following cochain complex is strictly exact:
  \begin{equation}\label{thm:galois-cohomology-of-BpdRdaggerplus-thisisstrictlyexact}
    0
    \to
    k
    \to\Hom_{\cont}\left( \Gal\left(\overline{k}/k\right) , B_{\pdR}^{\dag,+}\right)
    \to\Hom_{\cont}\left( \Gal\left(\overline{k}/k\right)^{2} , B_{\pdR}^{\dag,+}\right)
    \to\dots.
  \end{equation}
  As the functors
  $\Hom_{\cont}\left( \Gal\left(\overline{k}/k\right)^{i} , - \right)$
  preserve injections between $k_{0}$-Banach spaces,
  it follows
  from Lemma~\ref{lem:BpdRdagplus-bornology-countable-basis-reconstructionpaper}
  that~(\ref{thm:galois-cohomology-of-BpdRdaggerplus-thisisstrictlyexact})
  is a cochain complex of complete bornological $k_{0}$-vector spaces whose
  bornologies have countable basis, cf.
  Definitions~\ref{defn:bornologicalspace-reconstructionpaper}
  and~\ref{defn:bornology-countable-basis-reconstructionpaper}.
  Therefore, by Proposition~\ref{prop:solidexact-implies-indBanachstrictlyexact-reconstructionpaper}
  and Lemma~\ref{lem:intHomunderlineSunderlineV-is-underlineHomcontSV-reconstructionpaper},
  we may check that the colimit of the cochain complexes
  \begin{equation*}
    0
    \to
    \underline{k}
    \to\Hom_{\cont}\left( \Gal\left(\overline{k}/k\right) , \underline{B}_{\dR}^{\dag,+}\left[\log t\right]^{\leq n} \right)
    \to\Hom_{\cont}\left( \Gal\left(\overline{k}/k\right)^{2} , \underline{B}_{\dR}^{\dag,+}\left[\log t\right]^{\leq n} \right)
    \to\dots
  \end{equation*}
  along $n\in\NN$ is an exact cochain complex of solid $k$-vector spaces;
  see also Notation~\ref{notation:underlineBdRdaggerpluslogtn-reconstructionpaper}.
  This follows directly from
  Proposition~\ref{prop:galois-cohomology-of-BpdRdaggerplus-keylemma1-solid}
  and the second half of Lemma~\ref{lem:galois-cohomology-of-BpdRdaggerplus-keylemma2}.
\end{proof}

Finally, we record the following version of Theorem~\ref{thm:galois-cohomology-of-BpdRdaggerplus}
in the solid setting.

\begin{thm}\label{thm:galois-cohomology-of-solidBdRdaggerplus-solid}
  The canonical morphism
  $\underline{k}\isomap\R\Gamma_{\cont}\left(\Gal\left( \overline{k} / k \right) , \underline{B}_{\pdR}^{\dag,+} \right)$
  is an isomorphism.
\end{thm}

\begin{proof}
  This follows from Theorem~\ref{thm:galois-cohomology-of-BpdRdaggerplus}
  and Lemma~\ref{lem:contgpcoh-indban-vs-solid-reconstructionpaper}.
\end{proof}

%%%%%%%%%%%%%%%%%%%%%%%%%%%%%%%%%%%%%%%%%%%%%%%%%%%%%%%%%%%%
%%%%%%%%%%%%%%%%%%%%%%%%%%%%%%%%%%%%%%%%%%%%%%%%%%%%%%%%%%%%
% Cech cohomology III
%%%%%%%%%%%%%%%%%%%%%%%%%%%%%%%%%%%%%%%%%%%%%%%%%%%%%%%%%%%%
%%%%%%%%%%%%%%%%%%%%%%%%%%%%%%%%%%%%%%%%%%%%%%%%%%%%%%%%%%%%

\section{The Galois cohomology of $B_{\pdR}^{\dag}$}
\label{subsec:Galoiscoh-overconvalmostdeRhamperiodring-recpaper}

\begin{thm}\label{thm:galois-cohomology-of-BpdRdagger}
  The canonical morphism
  $k\isomap\R\Gamma_{\cont}\left(\Gal\left( \overline{k} / k \right) , B_{\pdR}^{\dag} \right)$
  is an isomorphism.
\end{thm}

We split the proof of Theorem~\ref{thm:galois-cohomology-of-BpdRdagger} into two parts.
Firstly, we prove that it is an isomorphism in degree zero in
\S\ref{subsubsec:proofthm:galois-cohomology-of-BpdRdagger-degree0}.
In \S\ref{subsubsec:proofthm:galois-cohomology-of-BpdRdagger-positivedegree},
we check the vanishing of $\Ho_{\cont}^{i}\left(\Gal\left( \overline{k} / k \right) , B_{\pdR}^{\dag,+} \right)$
for $i>0$.
Finally, everything comes together in \S\ref{subsubsec:proofthm:galois-cohomology-of-BpdRdagger}.

\iffalse
\begin{thm}\label{thm:galois-cohomology-of-BpdRdagger}
  The canonical morphism
  $\underline{k}\isomap\R\Gamma_{\cont}\left(\Gal\left( \overline{k} / k \right) , \underline{B}_{\pdR}^{\dag} \right)$
  is an isomorphism.
\end{thm}

We split the proof of Theorem~\ref{thm:galois-cohomology-of-BpdRdagger} into two parts.
Firstly, we prove that it is an isomorphism in degree zero in
\S\ref{subsubsec:proofthm:galois-cohomology-of-BpdRdagger-degree0}.
In \S\ref{subsubsec:proofthm:galois-cohomology-of-BpdRdagger-positivedegree},
we check the vanishing of $\Ho_{\cont}^{i}\left(\Gal\left( \overline{k} / k \right) , \underline{B}_{\pdR}^{\dag,+} \right)$
for $i>0$.
Finally, everything comes together in \S\ref{subsubsec:proofthm:galois-cohomology-of-BpdRdagger}.
\fi

%%%%%%%%%%%%%%%%%%%%%%%%%%%%%%%%%%%%%%%%%%%%%%%%%%%%%%%%%%%%
%%%%%%%%%%%%%%%%%%%%%%%%%%%%%%%%%%%%%%%%%%%%%%%%%%%%%%%%%%%%
% Cech cohomology III
%%%%%%%%%%%%%%%%%%%%%%%%%%%%%%%%%%%%%%%%%%%%%%%%%%%%%%%%%%%%
%%%%%%%%%%%%%%%%%%%%%%%%%%%%%%%%%%%%%%%%%%%%%%%%%%%%%%%%%%%%

\subsection{Computations in degree 0}\label{subsubsec:proofthm:galois-cohomology-of-BpdRdagger-degree0}

Fix $k^{\prime}$ as in
Notation~\ref{notation:galois-cohomology-of-BdRdaggerplus-kzero},
write $\cal{G}:=\Gal\left(\overline{k}/k^{\prime}\right)$,
and recall Notation~\ref{notation:twisted-AdRgreaterthanq-recpaper}.

\begin{lem}\label{lem:GaloisinvarianceAdRgreathanqlogtninvertp-gives-GaloisinvarianceBdRgreathanqlogt-j-reconstructionpaper}
  Let $n,j\in\NN$. For all $q\in\NN_{\geq3}$, we have the canonical isomorphism
  \begin{equation*}
    \left(A_{-j}^{>q,+}\left[\log t\right]^{\leq n}\right)^{\cal{G}}\widehat{\otimes}_{W(\kappa)}k_{0}
    \isomap
    \left(t^{-j}B_{\dR}^{>q,+}\left[\log t\right]^{\leq n}\right)^{\cal{G}}.
  \end{equation*}
\end{lem}

\begin{proof}
  The proof of Lemma~\ref{lem:GaloisinvarianceAdRgreathanqlogtninvertp-gives-GaloisinvarianceBdRgreathanqlogt-reconstructionpaper} generalises directly.
%%%
\end{proof}

The following Lemma~\ref{lem:galois-cohomology-of-BpdRdagger-j-keylemma1-kprime}
is a straightforward generalisation of Lemma~\ref{lem:galois-cohomology-of-BpdRdaggerplus-keylemma1-kprime}.
For the convenience of the reader, we include the full proof of below.

\begin{lem}\label{lem:galois-cohomology-of-BpdRdagger-j-keylemma1-kprime}
  For any $n,j\in\NN$, the canonical
  $k^{\prime}\to B_{\dR}^{\dag,+}\hookrightarrow t^{-j}B_{\dR}^{\dag,+}\left[\log t\right]^{\leq n}$ induces
  \begin{equation*}
    k^{\prime}\isomap\left(t^{-j}B_{\dR}^{\dag,+}\left[\log t\right]^{\leq n}\right)^{\cal{G}}.
  \end{equation*}
\end{lem}

\begin{proof}
  We proceed via induction along $n$. Lemma~\ref{lem:galois-cohomology-of-BpdRdagger-j-keylemma1-kprime} holds
  for $n=0$ by Proposition~\ref{prop:galois-cohomology-of-BdRdaggerplus-kprime-injectintotminusj},
  because $t^{-j}B_{\dR}^{\dag,+}\left[\log t\right]^{\leq 0}=t^{-j}B_{\dR}^{\dag,+}$. Indeed,
  here we use Lemma~\ref{lem:indBanachcontgpcohzero-is-invariance} and the fully faithfulness
  of $\I$, cf.~\cite[Corollary 1.2.28]{Sch99}.
  Now assume Lemma~\ref{lem:galois-cohomology-of-BpdRdagger-j-keylemma1-kprime}
  holds for fixed $n\in\NN$. For all $q\in\NN$, we have the short exact sequence
  \begin{equation*}
    0 \longrightarrow
    t^{-j}B_{\dR}^{>q,+}\left[\log t\right]^{\leq n} \longrightarrow
    t^{-j}B_{\dR}^{>q,+}\left[\log t\right]^{\leq n+1} \stackrel{\varphi_{-j}^{>q}}{\longrightarrow}
    t^{-j}B_{\dR}^{>q,+} \longrightarrow
    0
  \end{equation*}
  of $\cal{G}$-$k^{\prime}$-Banach modules, where
  $\varphi_{-j}^{>q}\colon t^{-j}B_{\dR}^{>q,+}\left[\log t\right]^{\leq n+1}\to t^{-j}B_{\dR}^{>q,+}$,
  $\sum_{\alpha=0}^{n+1}t^{-j}b_{\alpha}\left(\log t\right)^{\alpha}\mapsto t^{-j}b_{n+1}$.
  This short exact sequence is strictly exact by the open mapping theorem.
  Now pass to the colimit along $q\to\infty$ and recall
  Corollary~\ref{cor:filteredcol-inIndBan-stronglyexact}
   to get the strictly exact sequence
   \begin{equation}\label{eq:galois-cohomology-of-BpdRdaggerplus-keylemma1-kprime-strictlyshortexactsequence}
      0 \longrightarrow
      t^{-j}B_{\dR}^{\dag,+}\left[\log t\right]^{\leq n} \longrightarrow
      t^{-j}B_{\dR}^{\dag,+}\left[\log t\right]^{\leq n+1} \stackrel{\varphi_{-j}}{\longrightarrow}
      t^{-j}B_{\dR}^{\dag,+} \longrightarrow
      0
  \end{equation}
  Next, apply Lemma~\ref{lem:indBanachinvariance-cocont-leftexact} to get the strictly exact sequence
  \begin{equation*}
    0 \longrightarrow
    \left(t^{-j}B_{\dR}^{\dag,+}\left[\log t\right]^{\leq n}\right)^{\cal{G}} \stackrel{f_{-j}}{\longrightarrow}
    \left(t^{-j}B_{\dR}^{\dag,+}\left[\log t\right]^{\leq n+1}\right)^{\cal{G}} \stackrel{g_{-j}}{\longrightarrow}
    \left(t^{-j}B_{\dR}^{\dag,+}\right)^{\cal{G}}
  \end{equation*}  
  of $k^{\prime}$-ind-Banach modules. We claim that $g_{-j}$ is zero,
  as this would imply that $f_{-j}$ is an isomorphism. In this case, the induction hypothesis would give
  \begin{equation*}
    k^{\prime}\cong\left(t^{-j}B_{\dR}^{\dag,+}\left[\log t\right]^{\leq n}\right)^{\cal{G}}
    \stackrel{f_{-j}}{\cong}\left(t^{-j}B_{\dR}^{\dag,+}\left[\log t\right]^{\leq n+1}\right)^{\cal{G}},
  \end{equation*}
  which would finish the proof of Lemma~\ref{lem:galois-cohomology-of-BpdRdagger-j-keylemma1-kprime}.

  Next, we note with Lemma~\ref{lem:indBanachinvariance-cocont-leftexact} that $g$ is the colimit of the morphisms
  \begin{equation*}
    g_{-j}^{>q}\colon\left(t^{-j}B_{\dR}^{>q,+}\left[\log t\right]^{\leq n+1}\right)^{\cal{G}} \longrightarrow
    \left(t^{-j}B_{\dR}^{>q,+}\right)^{\cal{G}},
  \end{equation*}
  which are induced by the $\varphi_{-j}^{>q}$.
  We may therefore check that every $g_{-j}^{>q}$ is zero.
  By Lemma~\ref{lem:GaloisinvarianceAdRgreathanqlogtninvertp-gives-GaloisinvarianceBdRgreathanqlogt-j-reconstructionpaper},
  we may also check that the morphisms
  \begin{equation*}
    \left(A_{-j}^{>q}\left[\log t\right]^{\leq n+1}\right)^{\cal{G}} \longrightarrow
    \left(A_{-j}^{>q}\right)^{\cal{G}},
  \end{equation*}
  are zero, which are induced by
  $\phi_{-j}^{>q}\colon A_{-j}^{>q}\left[\log t\right]^{\leq n+1}\to A_{-j}^{>q}$,
  $\sum_{\alpha=0}^{n+1}\left(t/p^{q}\right)^{-j}a_{\alpha}\left(\log t\right)^{\alpha}\mapsto \left(t/p^{q}\right)^{-j}a_{n+1}$.
  This is a statement about abstract $\cal{G}$-representations, thus we can introduce filtrations
  which do not recover the topologies on
  $A_{-j}^{>q}\left[\log t\right]^{\leq n+1}$
  and $A_{-j}^{>q}$.
  We choose to work with $\widetilde{A}_{-j}^{>q}:=\widetilde{\A}_{-j}^{>q}\left(*\right)$,
  cf. Notation~\ref{notation:A-jgreaterthanq-recpaper}.
  Furthermore, we consider the $\cal{G}$-representation
  $\widetilde{A}_{-j}^{>q}\left[\log t\right]^{\leq n+1}$, carrying the filtration
  \begin{equation*}
    \Fil^{i}\left(\widetilde{A}_{-j}^{>q}\left[\log t\right]^{\leq n+1}\right)
    =\Fil^{i}\left(\widetilde{A}_{-j}^{>q}\right)\left[\log t\right]^{\leq n+1} \,
    \text{ for all $i\in\ZZ$.}
  \end{equation*}
  Now we have to check that the morphism
  \begin{equation*}
    \widetilde{\varphi}_{-j}^{>q}\colon
     \widetilde{A}_{-j}^{>q}\left[\log t\right]^{\leq n+1} \longrightarrow
     \widetilde{A}_{-j}^{>q},
     \sum_{\alpha=0}^{n+1}\left(\frac{t}{p^{q}}\right)^{-j}a_{\alpha}\left(\log t\right)^{\alpha}
      \mapsto \left(\frac{t}{p^{q}}\right)^{-j}a_{n+1}
  \end{equation*}
  sends any invariant
  $a=\sum_{\alpha=0}^{n+1}\left(t/p^{q}\right)^{-j}a_{\alpha}\left(\log t\right)^{\alpha}\in\widetilde{A}_{\dR}^{>q}\left[\log t\right]^{\leq n+1}$ to zero.
  With other words, one has to check $\left(t/p^{q}\right)^{-j}a_{n+1}=0$. As the filtration on
  $\widetilde{A}_{-j}^{>q}$ is separated,
  by Lemma~\ref{lem:grAdRgreaterthanq-xioverpadicfiltration-reconstructionpaper}(i), this is the case if
  the principal symbol
  $\sigma\left(\left(t/p^{q}\right)^{-j}a_{n+1}\right)\in\widetilde{A}_{-j}^{>q}$ is zero.
  This happens if $\sigma(a)$ is in the kernel of
  \begin{equation*}
    \gr\widetilde{\varphi}_{-j}^{>q}\colon
    \gr\widetilde{A}_{-j}^{>q,+}\left[\log t\right]^{\leq n+1} \longrightarrow
    \gr\widetilde{A}_{-j}^{>q,+}.
  \end{equation*}
  As $\sigma(a)$ is again $\cal{G}$-invariant, one may show the
  following: $\gr\widetilde{\varphi}^{>q}$ sends any $\cal{G}$-invariant element to zero. 
  To show this, we further describe the associated graded.
  Arguing as in the proof of Lemma~\ref{lem:galois-cohomology-of-BdRdagger-Dgreaterthanqbulletcomputedgrn},
  we find $\gr\widetilde{A}_{\dR}^{>q}\cong\bigoplus_{s\geq0}\cal{O}_{C}(s)$, where the
  $s$th Tate twist $\cal{O}_{C}(s)$ of $\cal{O}_{C}$ sits in degree $s$.
  Since $\cal{G}$ acts on $t/p^{q}$ via the cyclotomic character, we find
  $\gr\widetilde{A}_{-j}^{>q}\cong\bigoplus_{s\geq-j}\cal{O}_{C}(s)$.
  In particular, $\gr\widetilde{\varphi}_{-j}^{>q}$ is the direct sum of the maps
  \begin{equation*}
    \psi_{s}\colon
    \cal{O}_{C}(s)\left[\log t\right]^{\leq n+1} \longrightarrow
    \cal{O}_{C}(s),
    \sum_{\alpha=0}^{n+1}\lambda_{\alpha}\left(\log t\right)^{\alpha}\mapsto \lambda_{n+1}.
  \end{equation*}
  of $\cal{G}$-representations. Now apply Lemma~\ref{lem:galois-cohomology-of-BpdRdaggerplus-keylemma1-kprime-thefinallemma-reconstructionpaper}.
  Thus $\gr\widetilde{\varphi}_{-j}^{>q}$ kills the $\cal{G}$-invariant elements.
\end{proof}

\begin{lem}\label{lem:galois-cohomology-of-BpdRdagger-j-keylemma1}
  For any $n,j\in\NN$, the canonical $k\to B_{\dR}^{\dag,+}\hookrightarrow t^{-j}B_{\dR}^{\dag,+}\left[\log t\right]^{\leq n}$ induces
  \begin{equation*}
    \I\left(k\right)\isomap\Ho_{\cont}^{0}\left(\Gal\left( \overline{k} / k \right) ,
      t^{-j}B_{\dR}^{\dag,+}\left[\log t\right]^{\leq n} \right).
  \end{equation*}
\end{lem}

\begin{proof}
  Compute
  \begin{align*}
    \left(t^{-j}B_{\dR}^{\dag,+}\left[\log t\right]^{\leq n}\right)^{\Gal\left( \overline{k} / k\right)}
    \stackrel{\text{\ref{lem:galois-cohomology-of-BpdRdagger-j-keylemma1-kprime}}}{\cong}
    \left(\left(t^{-j}B_{\dR}^{\dag,+}\left[\log t\right]^{\leq n}\right)^{\Gal\left( \overline{k} / k^{\prime}\right)}\right)^{\Gal\left( k^{\prime} / k\right)}
    \cong\left(k^{\prime}\right)^{\Gal\left( k^{\prime} / k\right)}=k
  \end{align*}
  and apply Lemma~\ref{lem:indBanachcontgpcohzero-is-invariance}.
\end{proof}

We translate Lemma~\ref{lem:galois-cohomology-of-BpdRdagger-j-keylemma1}
into the solid language such that the long exact sequence in the proof of
Lemma~\ref{lem:galois-cohomology-of-BpdRdagger-keylemma2} becomes available.

\begin{notation}\label{notation:underlineBdRdaggerpluslogtn-j-reconstructionpaper}
For any $n,j\in\NN$, write
$t^{-j}\underline{B}_{\dR}^{\dag,+}\left[\log t\right]^{\leq n}:=\underline{t^{-j}B_{\dR}^{\dag,+}\left[\log t\right]^{\leq n}}$.
\end{notation}

\begin{prop}\label{prop:galois-cohomology-of-BpdRdaggerplus-j-keylemma1-solid}
  For any $n,j\in\NN$, the canonical
  $k\to B_{\dR}^{\dag,+}\hookrightarrow t^{-j}B_{\dR}^{\dag,+}\left[\log t\right]^{\leq n}$
  induces
  \begin{equation*}
    \underline{k}\isomap\Ho_{\cont}^{0}\left(\Gal\left( \overline{k} / k \right) ,
      t^{-j}\underline{B}_{\dR}^{\dag,+}\left[\log t\right]^{\leq n} \right).
  \end{equation*}
\end{prop}

\begin{proof}
  Apply Lemma~\ref{lem:contgpcoh-indban-vs-solid-reconstructionpaper}
  to Lemma~\ref{lem:galois-cohomology-of-BpdRdagger-j-keylemma1}.
\end{proof}

\iffalse
\begin{prop}\label{prop:galois-cohomology-of-BpdRdagger-keylemma1-solid}
  For any $n\in\NN$, the canonical
  $k\to B_{\dR}^{\dag,+}\hookrightarrow B_{\dR}^{\dag}\left[\log t\right]^{\leq n}$
  induces
  \begin{equation*}
    \underline{k}\isomap\Ho_{\cont}^{0}\left(\Gal\left( \overline{k} / k \right) , \underline{B}_{\dR}^{\dag}\left[\log t\right]^{\leq n} \right).
  \end{equation*}
\end{prop}

\begin{proof}
  $\underline{B}_{\dR}^{\dag}\left[\log t\right]^{\leq n}=\varinjlim_{j}t^{-j}\underline{B}_{\dR}^{\dag}\left[\log t\right]^{\leq n}$,
  thus one may apply
  Proposition~\ref{prop:galois-cohomology-of-BpdRdaggerplus-j-keylemma1-solid}
  and~\ref{lem:formalcolimits-and-solid-group-cohomology-reconstructionpaper}.
\end{proof}
\fi

%%%%%%%%%%%%%%%%%%%%%%%%%%%%%%%%%%%%%%%%%%%%%%%%%%%%%%%%%%%%
%%%%%%%%%%%%%%%%%%%%%%%%%%%%%%%%%%%%%%%%%%%%%%%%%%%%%%%%%%%%
% Cech cohomology III
%%%%%%%%%%%%%%%%%%%%%%%%%%%%%%%%%%%%%%%%%%%%%%%%%%%%%%%%%%%%
%%%%%%%%%%%%%%%%%%%%%%%%%%%%%%%%%%%%%%%%%%%%%%%%%%%%%%%%%%%%

\subsection{Computations in positive degree}\label{subsubsec:proofthm:galois-cohomology-of-BpdRdagger-positivedegree}

Recall Notation~\ref{notation:underlineBdRdaggerpluslogtn-j-reconstructionpaper}.

\begin{lem}\label{lem:galois-cohomology-of-BpdRdagger-keylemma2}
  For any $n,j\in\NN$, we compute the following:
  \begin{itemize}
    \item[(i)] $\Ho_{\cont}^{1}\left(\Gal\left( \overline{k} / k \right) , t^{-j}\underline{B}_{\dR}^{\dag,+}\left[\log t\right]^{\leq n} \right)=\underline{k}$, and
    \item[(ii)] $\Ho_{\cont}^{i}\left(\Gal\left( \overline{k} / k \right) , t^{-j}\underline{B}_{\dR}^{\dag,+}\left[\log t\right]^{\leq n} \right)=0$
    for all $i\geq2$.
  \end{itemize}
  Furthermore, the following canonical morphism is zero:
  \begin{equation*}
    \Ho_{\cont}^{i}\left(\Gal\left( \overline{k} / k \right) , t^{-j}\underline{B}_{\dR}^{\dag,+}\left[\log t\right]^{\leq n} \right)
    \to
    \Ho_{\cont}^{i}\left(\Gal\left( \overline{k} / k \right) , t^{-j}\underline{B}_{\dR}^{\dag,+}\left[\log t\right]^{\leq n+1} \right).
  \end{equation*}
\end{lem}

\begin{proof}
  Proceed as in the proof of Lemma~\ref{lem:galois-cohomology-of-BpdRdaggerplus-keylemma2},
  applying Corollary~\ref{cor:galois-cohomology-of-BdRdaggerplus-k-injectintotminusj-solid} instead of
  Theorem~\ref{thm:galois-cohomology-of-solidBdRdaggerplus}.
\end{proof}

%%%%%%%%%%%%%%%%%%%%%%%%%%%%%%%%%%%%%%%%%%%%%%%%%%%%%%%%%%%%
%%%%%%%%%%%%%%%%%%%%%%%%%%%%%%%%%%%%%%%%%%%%%%%%%%%%%%%%%%%%
% Cech cohomology III
%%%%%%%%%%%%%%%%%%%%%%%%%%%%%%%%%%%%%%%%%%%%%%%%%%%%%%%%%%%%
%%%%%%%%%%%%%%%%%%%%%%%%%%%%%%%%%%%%%%%%%%%%%%%%%%%%%%%%%%%%

\subsection{Proof of Theorem~\ref{thm:galois-cohomology-of-BpdRdagger}}\label{subsubsec:proofthm:galois-cohomology-of-BpdRdagger}

We can now prove the main result of \S\ref{subsec:Galoiscoh-overconvalmostdeRhamperiodring-recpaper}.

\begin{proof}[Proof of Theorem~\ref{thm:galois-cohomology-of-BpdRdagger}]
  We have to check that the following cochain complex is strictly exact:
  \begin{equation}\label{thm:galois-cohomology-of-BpdRdagger-thisisstrictlyexact}
    0
    \to
    k
    \to\Hom_{\cont}\left( \Gal\left(\overline{k}/k\right) , B_{\pdR}^{\dag}\right)
    \to\Hom_{\cont}\left( \Gal\left(\overline{k}/k\right)^{2} , B_{\pdR}^{\dag}\right)
    \to\dots.
  \end{equation}
  As the functors
  $\Hom_{\cont}\left( \Gal\left(\overline{k}/k\right)^{i} , - \right)$
  preserve injections between $k_{0}$-Banach spaces,
  it follows from Theorem~\ref{lem:BpdRdag-bornology-countable-basis-reconstructionpaper}  
  that~(\ref{thm:galois-cohomology-of-BpdRdagger-thisisstrictlyexact})
  is a cochain complex of complete bornological $k_{0}$-vector spaces whose
  bornologies have countable basis, cf.
  Definitions~\ref{defn:bornologicalspace-reconstructionpaper}
  and~\ref{defn:bornology-countable-basis-reconstructionpaper}.
  Therefore, by Proposition~\ref{prop:solidexact-implies-indBanachstrictlyexact-reconstructionpaper}
  and Lemma~\ref{lem:intHomunderlineSunderlineV-is-underlineHomcontSV-reconstructionpaper},
  we may check that the colimit of the complexes
  \begin{equation*}
    0
    \to
    \underline{k}
    \to\Hom_{\cont}\left( \Gal\left(\overline{k}/k\right) , t^{-j}\underline{B}_{\pdR}^{\dag,+}\left[\log t\right]^{\leq n} \right)
    \to\Hom_{\cont}\left( \Gal\left(\overline{k}/k\right)^{2} , t^{-j}\underline{B}_{\pdR}^{\dag,+}\left[\log t\right]^{\leq n} \right)
    \to\dots
  \end{equation*}
  along $n,j\in\NN$ is an exact complex of solid $k_{0}$-vector spaces;
  see also Notation~\ref{notation:underlineBdRdaggerpluslogtn-reconstructionpaper}.
  This follows directly from
  Proposition~\ref{prop:galois-cohomology-of-BpdRdaggerplus-keylemma1-solid}
  and the second half of Lemma~\ref{lem:galois-cohomology-of-BpdRdaggerplus-keylemma2}.
\end{proof}

Finally, we record the following version of Theorem~\ref{thm:galois-cohomology-of-BpdRdagger}
in the solid setting.

\begin{thm}\label{thm:galois-cohomology-of-solidBpdRdagger-solid}
  The canonical morphism
  $\underline{k}\isomap\R\Gamma_{\cont}\left(\Gal\left( \overline{k} / k \right) , \underline{B}_{\pdR}^{\dag} \right)$
  is an isomorphism.
\end{thm}

\begin{proof}
  This follows from Theorem~\ref{thm:galois-cohomology-of-BpdRdagger}
  and Lemma~\ref{lem:contgpcoh-indban-vs-solid-reconstructionpaper}.
\end{proof}

\iffalse
\begin{proof}[Proof of Theorem~\ref{thm:galois-cohomology-of-BpdRdagger}]
  For every $i\in\ZZ$, Lemma~\ref{lem:formalcolimits-and-solid-group-cohomology-reconstructionpaper}
  gives the isomorphism
  \begin{equation*}
    \Ho_{\cont}^{i}\left(\Gal\left( \overline{k} / k \right) , \underline{B}_{\pdR}^{\dag} \right)
=\varinjlim_{n}\Ho_{\cont}^{i}\left(\Gal\left( \overline{k} / k \right) , \underline{B}_{\dR}^{\dag}\left[\log t\right]^{\leq n} \right).
  \end{equation*}
  Thus Theorem~\ref{thm:galois-cohomology-of-BpdRdaggerplus} follows from
  Proposition~\ref{prop:galois-cohomology-of-BpdRdagger-keylemma1-solid}
  and Lemma~\ref{lem:galois-cohomology-of-BpdRdagger-keylemma2}.
\end{proof}
\fi

%%%%%%%%%%%%%%%%%%%%%%%%%%%%%%%%%%%%%%%%%%%%%%%%%%%%%%%%%%%%
%%%%%%%%%%%%%%%%%%%%%%%%%%%%%%%%%%%%%%%%%%%%%%%%%%%%%%%%%%%%
%%%%%%%%%%%%%%%%%%%%%%%%%%%%%%%%%%%%%%%%%%%%%%%%%%%%%%%%%%%%
% Pushforward of OBla
%%%%%%%%%%%%%%%%%%%%%%%%%%%%%%%%%%%%%%%%%%%%%%%%%%%%%%%%%%%%
%%%%%%%%%%%%%%%%%%%%%%%%%%%%%%%%%%%%%%%%%%%%%%%%%%%%%%%%%%%%
%%%%%%%%%%%%%%%%%%%%%%%%%%%%%%%%%%%%%%%%%%%%%%%%%%%%%%%%%%%%

\chapter{Derived pushforwards of overconvergent period structure sheaves}
\label{ch:pushfowardsofperiodstructuresheaves-reconstructionpaper}

Fix the conventions and notation as in \S\ref{subsec:conventions-reconstructionpaper}.
$X$ denotes an arbitrary smooth rigid-anaytic $k$-variety, and $\nu\colon X_{\proet}\to X$ is the canonical
projection. In \S\ref{ch:pushfowardsofperiodstructuresheaves-reconstructionpaper},
we compute the derived pushfowards $\R\nu_{*}\OB_{\dR}^{\dag}$ and
$\R\nu_{*}\OB_{\pdR}^{\dag}$
As a byproduct of our arguments, we also get explicit descriptions of the pushforwards
of the solidifications $\R\nu_{*}\underline{\OB}_{\dR}^{\dag}$ and
$\R\nu_{*}\underline{\OB}_{\pdR}^{\dag}$.

%%%%%%%%%%%%%%%%%%%%%%%%%%%%%%%%%%%%%%%%%%%%%%%%%%%%%%%%%%%%
%%%%%%%%%%%%%%%%%%%%%%%%%%%%%%%%%%%%%%%%%%%%%%%%%%%%%%%%%%%%
% Preparation
%%%%%%%%%%%%%%%%%%%%%%%%%%%%%%%%%%%%%%%%%%%%%%%%%%%%%%%%%%%%
%%%%%%%%%%%%%%%%%%%%%%%%%%%%%%%%%%%%%%%%%%%%%%%%%%%%%%%%%%%%

\section{Preparations}

The following will be used later on and should be skipped on a first reading.

\begin{lem}\label{lem:iso-after-modulo-torsion}
  Consider a split injection $\iota\colon M\hookrightarrow N$
  of modules over an abstract commutative ring $R$, whose cokernel
  is killed by $r\in R$. If $M$ has no $r$-torsion, then $\iota$ induces
  an isomorphism
  \begin{equation*}
    \overline{\iota} \colon M \isomap N/N[r].
  \end{equation*}
  Here, $N[r]\subseteq N$ denotes the $r$-torsion submodule. 
\end{lem}

\begin{proof}
  Fix a section $s$ of $\iota$.
  To see that $\overline{\iota}$ is injective, note that
  $\overline{\iota}(m)=0$ imples $\iota(m)\in N[r]$.
  That is, $\iota(rm)=r\iota(m)=0$, implying $rm=0$.
  Since $M$ does not have any $r$-torsion, $m=0$.

  It remains to check that any fixed $[n]\in N/N[r]$ lies in the
  image of $\overline{\iota}$. Since $r\coker\iota=0$, we
  find an $m$ such that $rn=\iota(m)$. Compute
  \begin{equation*}
    rn
    =\iota(m)
    =\iota(s\iota(m))
    =\iota(s(rm))
    =r\iota(s(n)).
  \end{equation*}
  This implies $r(n-\iota(s(n)))=0$. In particular,
  %\begin{equation*}
    $[n]=[\iota(s(n))]=\overline{\iota}(s(n))\in\im\overline{\iota}$.
  %\end{equation*}
\end{proof}

%%%%%%%%%%%%%%%%%%%%%%%%%%%%%%%%%%%%%%%%%%%%%%%%%%%%%%%%%%%%
%%%%%%%%%%%%%%%%%%%%%%%%%%%%%%%%%%%%%%%%%%%%%%%%%%%%%%%%%%%%
% Cech cohomology I
%%%%%%%%%%%%%%%%%%%%%%%%%%%%%%%%%%%%%%%%%%%%%%%%%%%%%%%%%%%%
%%%%%%%%%%%%%%%%%%%%%%%%%%%%%%%%%%%%%%%%%%%%%%%%%%%%%%%%%%%%

\section{The initial computation}\label{subsec:initialcomputation}

Fix a compatible system $1,\zeta_{p},\zeta_{p^{2}},\dots\in C$
of primitive $p$th roots of unity.
This gives rise to $\epsilon:=\left(1,\zeta_{p},\zeta_{p^{2}},\dots\right)\in C^{\flat}$.
Write $\mu:=[\epsilon]-1\in A_{\inf}:=\A_{\inf}\left(C,\cal{O}_{C}\right)$.
Let $d\in\NN_{\geq1}$. $S$ denotes a profinite set.
In the following, we frequently use Notation~\ref{examplenotation:torus-variable-index}.

\begin{notation}\label{notation:normalOBi-reconstructionpaper}
  $\normalOB_{i}
    :=\OB_{\dR}^{\dag}\left(\widetilde{\TT}_{C}^{\left\{1,\dots,i\right\}}\times S\right)
    \widehat{\otimes}_{k}\cal{O}\left(\TT^{\left\{i+1,\dots,d\right\}}\right)$
    for all $i=0,\dots,d$.
\end{notation}

\begin{remark}\label{remark--notation:normalOBi-reconstructionpaper}
  $\normalOB_{d}=\OB_{\dR}^{\dag}\left(\widetilde{\TT}_{C}^{d}\times S\right)$
  and
  $\normalOB_{0}=\BB_{\dR}^{\dag}\left(\Spa\left(C,\cal{O}_{C}\right)\times S\right)\widehat{\otimes}\cal{O}\left(\TT^{d}\right)$,
  cf. Notation~\ref{notatio:base-change-in-proet}.
\end{remark}

Fix $i=1,\dots,d$ and consider the $\ZZ_{p}(1)^{i}$-action on
$\widetilde{\TT}_{C}^{\left\{1,\dots,i\right\}}$ as in Lemma~\ref{lem:ZpGaloiscoverings}.
Concretely, we fix a $\ZZ_{p}$-basis $\gamma_{1},\dots,\gamma_{i}$ of $\ZZ_{p}(1)^{i}\cong\ZZ_{p}^{i}$
such that
\begin{equation*}
  \gamma_{l}\cdot T_{1}^{e_{1}}\cdots T_{i}^{e_{i}}=\zeta^{e_{l}}T_{1}^{e_{1}}\cdots T_{i}^{e_{i}}
\end{equation*}
for all $i=1,\dots,l$. Here, $\zeta^{e_{l}}=\zeta_{p^{j}}^{e_{l}p^{j}}$ whenever $e_{l}p^{j}\in\ZZ$.
This $\ZZ_{p}(1)^{i}$-action on $\widetilde{\TT}_{K}^{i}$
gives a $\ZZ_{p}(1)^{i}$-action on $\normalOB_{i}$
by functoriality. In particular, we have an automorphism
$\gamma_{i}\colon\normalOB_{i}\to\normalOB_{i}$.

On the other hand, we denote the canonical inclusion $\normalOB_{i-1}\to\normalOB_{i}$
by $\iota$. Because $\gamma_{i}$ fixes $T_{1},\dots,T_{i-1}$, we find
$\im\iota\subset\gamma_{i}-1$.
Thus we get the cochain complex~(\ref{eq:cont-group-coh-OBlatTTKd-the-sequence})
in Proposition~\ref{prop:cont-group-coh-OBlatTTKd}.

\begin{prop}\label{prop:cont-group-coh-OBlatTTKd}
  The sequence
  \begin{equation}\label{eq:cont-group-coh-OBlatTTKd-the-sequence}
    0
    \to\normalOB_{i-1}
    \stackrel{\iota}{\longrightarrow} \normalOB_{i}
    \stackrel{\gamma_{i}-1}{\longrightarrow}
    \normalOB_{i} \longrightarrow 0
  \end{equation}
  of $k$-ind-Banach spaces is strictly exact.
\end{prop}

We spend the remainder of \S\ref{subsec:initialcomputation} on the proof of
Proposition~\ref{prop:cont-group-coh-OBlatTTKd}.
In the following, we summarise our arguments for the convenience of the reader.

In \S\ref{subsubsec:initialcomputation-usefulpresentationsOBiOBi-1-reconstructionpaper},
we give descriptions of $\normalOB_{i}$ and $\normalOB_{i-1}$
which make the complex~(\ref{eq:cont-group-coh-OBlatTTKd-the-sequence})
explicit. However, at this point it is still difficult to give a direct proof of the strict
exactness of~(\ref{eq:cont-group-coh-OBlatTTKd-the-sequence}).
The issue is that both $\normalOB_{i}$ and $\normalOB_{i-1}$
do not carry useful filtrations, therefore we cannot argue via associated gradeds.
We resolve this by replacing $\normalOB_{i}$ and $\normalOB_{i-1}$
with certain Banach algebras $\normalOA_{i-1,i}^{>q-1,q}$,
$\normalOA_{i-1}^{>q-1}$, and $\normalOA_{i}^{>q-1}$;
these carry filtrations which are at least complete.
We do this in \S\ref{subsubsec:initialcomputation-normalOAi-1ietc},
and then build a family of cochain complexes $E^{>q,\bullet}$
in \S\ref{subsubsec:initialcomputation-thecomplexEgreathanqbullet}
which carries complete filtrations.
Their cohomology is directly related to the strict exactness of~(\ref{eq:cont-group-coh-OBlatTTKd-the-sequence}),
cf. Lemma~\ref{lem:etamupq-decelagemakessense}. One would then like to proceed
by consideration of the associated graded $\gr E^{>q,\bullet}$, which we failed to do.

Therefore, in \S\ref{subsubsec:initialcomputation-normalOtildeAi-1ietc},
we introduce variants
$\normalOtildeA_{i-1,i}^{>q-1,q}$,
$\normalOtildeA_{i-1}^{>q-1}$, and $\normalOtildeA_{i}^{>q-1}$
of $\normalOA_{i-1,i}^{>q-1,q}$,
$\normalOA_{i-1}^{>q-1}$, and $\normalOA_{i}^{>q-1}$,
carrying slightly different filtrations. We use these to build the variants
$\widetilde{E}^{>q,\bullet}$ of the cochain complexes $E^{>q,\bullet}$
in \S\ref{subsubsec:initialcomputation-thecomplextildeEgreathanqbullet}.
Both coincide as cochain complexes of abstract rings, but their filtrations differ.
Nevertheless, we are able to relate the cohomology of $\widetilde{E}^{>q,\bullet}$
to the strict exactness of~(\ref{eq:cont-group-coh-OBlatTTKd-the-sequence}),
cf. Lemma~\ref{lem:cont-group-coh-OBlatTTKd-Dbullet-reducetotildeE}.

In \S\ref{subsubsec:initialcomputation-associatedgradedofEgreaterthanqbullet-reconstructionpaper},
we continue by studying the cohomology of $\widetilde{E}^{>q,\bullet}$
through its associated graded. We relate it to the cohomology of a certain the complex $H^{\bullet}$,
whose cohomology we compute in \S\ref{subsubsec:initialcomputation-Ho-of-Hbullet}.
We then use these computations in \S\ref{subsubsec:initialcomputation-finallytheproof}
to complete the proof of Proposition~\ref{prop:cont-group-coh-OBlatTTKd}.

%%%%%%%%%%%%%%%%%%%%%%%%%%%%%%%%%%%%%%%%%%%%%%%%%%%%%%%%%%%%
%%%%%%%%%%%%%%%%%%%%%%%%%%%%%%%%%%%%%%%%%%%%%%%%%%%%%%%%%%%%
% Cech cohomology I
%%%%%%%%%%%%%%%%%%%%%%%%%%%%%%%%%%%%%%%%%%%%%%%%%%%%%%%%%%%%
%%%%%%%%%%%%%%%%%%%%%%%%%%%%%%%%%%%%%%%%%%%%%%%%%%%%%%%%%%%%

\subsection{A useful presentation of $\normalOB_{i}$}
\label{subsubsec:initialcomputation-usefulpresentationsOBiOBi-1-reconstructionpaper}

Theorem~\ref{thm:cohomology-OBdRdag-over-affperfd-reconstructionpaper}
yields the isomorphism
\begin{equation}\label{eq:aprioridescription-normalOBi-reconstructionpaper}
  \normalOB_{i}
  \cong
  \BB_{\dR}^{\dag}\left(\widetilde{\TT}_{C}^{\left\{1,\dots,i\right\}}\times S\right)
  \left\<\frac{Z_{1},\dots,Z_{i}}{p^{\infty}}\right\>
  \widehat{\otimes}_{k}
  \cal{O}\left(\TT^{\left\{i+1,\dots,d\right\}}\right).
\end{equation}
Instead, we would like to give a different description of
$\normalOB_{i}$ which interacts well with the $\gamma_{i}$-action.

\begin{notation}
  $T_{j}:=\left[T_{j}^{\flat}\right]+Z_{j}\in
  \A_{\inf}\left(\widetilde{\TT}_{C}^{\left\{1,\dots,i\right\}}\times S\right)
  \llbracket Z_{1},\dots,Z_{i} \rrbracket$
  for all $j=1,\dots,i$.
\end{notation}

\begin{equation*}
  \A_{\inf}\left(\widetilde{\TT}_{C}^{\left\{1,\dots,i\right\}}\times S\right)
  \llbracket Z_{1},\dots,Z_{i} \rrbracket
  \to\widehat{\cal{O}}^{+}\left(\widetilde{\TT}_{C}^{\left\{1,\dots,i\right\}}\times S\right),
  \sum_{\alpha\in\NN^{d}}a_{\alpha}Z^{\alpha}\mapsto\theta_{\inf}\left(a_{0}\right)
\end{equation*}
sends $T_{j}$ to $T_{j}\in \widehat{\cal{O}}^{+}\left(\widetilde{\TT}_{C}^{\left\{1,\dots,i\right\}}\times S\right)$,
which is a unit. Its kernel is $\left(\xi,Z_{1},\dots,Z_{d}\right)$, and
because $\A_{\inf}\left(\widetilde{\TT}_{C}^{\left\{1,\dots,i\right\}}\times S\right)
  \llbracket Z_{1},\dots,Z_{i} \rrbracket$
is complete with respect to the $\left(\xi,Z_{1},\dots,Z_{d}\right)$-adic topology,
cf. Proposition~\ref{prop:Scomplete-Spowerseriescomplete-ifKoszulregular},
\cite[\href{https://stacks.math.columbia.edu/tag/05GI}{Tag 05GI}]{stacks-project}
implies that $T_{j}$ is a unit in
$\A_{\inf}\left(\widetilde{\TT}_{C}^{\left\{1,\dots,i\right\}}\times S\right)
  \llbracket Z_{1},\dots,Z_{i} \rrbracket$,
for all $j=1,\dots,i$. Thus the following Notation~\ref{notation:Fj} makes sense.

\begin{notation}\label{notation:Fj}
  $F_{j}:=\sum_{n\geq 0}\frac{\left(-1\right)^{n}}{(n+1)T_{j}^{n}}Z_{j}^{n}
  \in\BB_{\inf}\left(\widetilde{\TT}_{C}^{\left\{1,\dots,i\right\}}\times S\right)
  \llbracket Z_{1},\dots,Z_{i} \rrbracket$
  for all $j=1,\dots,i$.
\end{notation}

\begin{notation}\label{notation:Uj}
  %\begin{equation*}
    $U_{j}:=Z_{j}F_{j}$
  %\end{equation*}
  for all $j=1,\dots,i$.
\end{notation}

\begin{lem}\label{lem:Ujconverges}
  For $q$ large enough,
  $U_{j}\in\BB_{\dR}^{q}\left(\widetilde{\TT}_{C}^{\left\{1,\dots,i\right\}}\times S\right)\left\< \frac{Z_{1},\dots,Z_{i}}{p^{q}}\right\>$
  for all $j=1,\dots,i$. 
\end{lem}

\begin{proof}
  It suffices to check
  $F_{j}\in\BB_{\dR}^{q}\left(\widetilde{\TT}_{C}^{\left\{1,\dots,i\right\}}\times S\right)\left\< \frac{Z_{1},\dots,Z_{i}}{p^{q}}\right\>$.
  If the multiplication on the Banach algebra is bounded by $C>0$, then
  \begin{equation*}
    p^{-nq}\|\frac{\left(-1\right)^{n}}{(n+1)T_{j}^{n}}\|
    \leq p^{-nq}\frac{C^{n}}{\|n+1\|\|T_{j}\|^{n}}\to0 \text{ for } n\to\infty
  \end{equation*}
  whenever $q$ is large enough.
\end{proof}

\begin{lem}\label{eq:aposteroridescription-normalOBi-reconstructionpaper}
  $\normalOB_{i}
    \cong
    \BB_{\dR}^{\dag}\left(\widetilde{\TT}_{C}^{\left\{1,\dots,i\right\}}\times S\right)
    \left\<
      \frac{U_{1},\dots,U_{i}}{p^{\infty}}
    \right\>
    \widehat{\otimes}_{k}
    \cal{O}\left(\TT^{\left\{i+1,\dots,d\right\}}\right)$.
\end{lem}
  
\begin{proof}
  Thanks to the
  isomorphism~(\ref{eq:aprioridescription-normalOBi-reconstructionpaper}),
  it suffices to check
  \begin{equation*}
    \BB_{\dR}^{q,+}\left(\widetilde{\TT}_{C}^{\left\{1,\dots,i\right\}}\times S\right)
    \left\<
      \frac{Z_{1},\dots,Z_{i}}{p^{q}}
    \right\>
    \cong
    \BB_{\dR}^{q,+}\left(\widetilde{\TT}_{C}^{\left\{1,\dots,i\right\}}\times S\right)
    \left\<
      \frac{U_{1},\dots,U_{i}}{p^{q}}
    \right\>
  \end{equation*}
  for large $q$.
  Because $U_{j}=Z_{j}F_{j}$
  we may check that the
  $F_{j}\in\BB_{\dR}^{q,+}\left(\widetilde{\TT}_{C}^{\left\{1,\dots,i\right\}}\times S\right)
    \left\<
      \frac{Z_{1},\dots,Z_{i}}{p^{q}}
    \right\>$
  are units. To do this, write
  \begin{equation}\label{eq:Fjrewritteninintegralsetting-reconstructionpaper}
    F_{j}=\sum_{n\geq 0}\frac{(-p)^{n}}{(n+1)T^{n}}\left(\frac{Z}{p}\right)^{n}.
  \end{equation}
  The fractions $(-p)^{n}/(n+1)$ are $p$-adic integers, thus (\ref{eq:Fjrewritteninintegralsetting-reconstructionpaper})
  implies that $F_{j}$ lies in the formal power series ring
  $R:=\A_{\inf}\left(\widetilde{\TT}_{C}^{\{j\}}\times S\right)\llbracket Z_{j}/p\rrbracket$.
  This ring $R$ is $Z_{j}/p$-adically complete
  and the reduction of $F_{j}$ modulo $Z_{j}/p$ is $1$,
  thus~\cite[\href{https://stacks.math.columbia.edu/tag/05GI}{Tag 05GI}]{stacks-project}
  implies that $F_{j}$ is a unit in $R$. Consequently, $F_{j}$ is a unit in
  $\BB_{\dR}^{q,+}\left(\widetilde{\TT}_{C}^{\left\{1,\dots,i\right\}}\times S\right)
    \left\<
      \frac{Z_{1},\dots,Z_{i}}{p^{q}}
    \right\>$.
\end{proof}

For future reference, we record the following variant of
Lemma~\ref{eq:aposteroridescription-normalOBi-reconstructionpaper}.

\begin{lem}\label{eq:aposteroridescription-normalOBi-1-reconstructionpaper}
  $\normalOB_{i-1}
    \cong
    \BB_{\dR}^{\dag}\left(\widetilde{\TT}_{C}^{\left\{1,\dots,i-1\right\}}\times S\right)
    \left\<
      \frac{U_{1},\dots,U_{i-1}}{p^{\infty}}
    \right\>
    \widehat{\otimes}_{k}
    \cal{O}\left(\TT^{\left\{i,\dots,d\right\}}\right)$.
\end{lem}

Using Lemma~\ref{eq:aposteroridescription-normalOBi-reconstructionpaper},
we determine the $\gamma_{i}$-action on both $\normalOB_{i}$ and $\normalOB_{i-1}$.
Everything is clear, except the action on the variables $U_{1},\dots,U_{i}$,
which we describe in Corollary~\ref{cor:gammai-acts-on-Uj} below.

\begin{lem}\label{lem:Ujislogarithm}
  $U_{j}=\log\left( \left[T_{j}^{\flat}\right] / T_{j} \right)$ for all $j=1,\dots,i$.
\end{lem}

\begin{proof}
  $U_{j}=Z_{j}F_{j}
  =\sum_{n \geq0}\frac{\left(-1\right)^{n}}{n+1}\left(\frac{-Z_{j}}{T_{j}}\right)^{n}
  =\sum_{n \geq0}\frac{\left(-1\right)^{n}}{n+1}\left(\frac{\left[T_{j}^{\flat}\right]}{T_{j}}-1\right)^{n}
  =\log\left( \left[T_{j}^{\flat}\right] / T_{j} \right)$.
\end{proof}

\begin{cor}\label{cor:gammai-acts-on-Uj}
  $\gamma_{i}\cdot U_{i}= t + U_{i} $ and $\gamma_{i}\cdot U_{j}=U_{j}$ for all $j=1,\dots,i-1$.
\end{cor}

\begin{proof}
  The first identity follows from
  \begin{equation*}
    \gamma_{i}\cdot U_{i}
    \stackrel{\text{\ref{lem:Ujislogarithm}}}{=}
    \log\left(\frac{\left[\gamma_{i}T_{i}^{\flat}\right]}{T_{i}}\right)
    =\log\left(\frac{\left[\epsilon\right]\left[T_{i}^{\flat}\right]}{T_{i}}\right)
    =\log\left(\left[\epsilon\right]\right)+\log\left(\frac{\left[T_{i}^{\flat}\right]}{T_{i}}\right)
    \stackrel{\text{\ref{lem:Ujislogarithm}}}{=}    
    t+U_{i}.
  \end{equation*}
  The second identity follows because
  $\gamma_{i}$ fixes $T_{j}^{1/p^{e}}$ for all $e\in\NN$ and $j=1,\dots,i-1$.
\end{proof}

%%%%%%%%%%%%%%%%%%%%%%%%%%%%%%%%%%%%%%%%%%%%%%%%%%%%%%%%%%%%
%%%%%%%%%%%%%%%%%%%%%%%%%%%%%%%%%%%%%%%%%%%%%%%%%%%%%%%%%%%%
% Cech cohomology I
%%%%%%%%%%%%%%%%%%%%%%%%%%%%%%%%%%%%%%%%%%%%%%%%%%%%%%%%%%%%
%%%%%%%%%%%%%%%%%%%%%%%%%%%%%%%%%%%%%%%%%%%%%%%%%%%%%%%%%%%%

\subsection{The Banach algebras $\normalOA_{i-1,i}^{>q-1,q}$,
$\normalOA_{i-1}^{>q-1}$, and $\normalOA_{i}^{>q-1}$}\label{subsubsec:initialcomputation-normalOAi-1ietc}

The difficulty in computing the strict exactness of~(\ref{eq:cont-group-coh-OBlatTTKd-the-sequence})
is that this complex does not carry a useful filtration. We would like to change this
by recovering~(\ref{eq:cont-group-coh-OBlatTTKd-the-sequence}) from certain filtered complexes
$E^{>q,\bullet}$, which we can then study through their associated gradeds.
These complexes are defined in the following
\S\ref{subsubsec:initialcomputation-thecomplexEgreathanqbullet}.
Here, we introduce their building blocks
$\normalOA_{i-1,i}^{>q-1,q}$, $\normalOA_{i-1}^{>q-1}$, and $\normalOA_{i}^{>q-1}$.

\begin{notation}\label{notation:initialcomputation-thecomplexEgreathanqbullet-largeformalpowerseriesrings}
  Let $q\in\NN_{\geq1}$.
  \begin{itemize}
    \item[(i)]
      $\A_{\dR}^{>q}\left(\widetilde{\TT}_{C}^{\left\{1,\dots,i\right\}}\times S\right)
      \left\llbracket
        \frac{U_{1},\dots,U_{i-1}}{p^{q-1}},\frac{U_{i}}{p^{q}}
      \right\rrbracket$
      is the $\left(p,\xi/p^{q},U_{1} / p^{q-1},\dots,U_{i-1} / p^{q-1},U_{i} / p^{q}\right)$-adic
      completion of
      $\A_{\dR}^{q}\left(\widetilde{\TT}_{C}^{\left\{1,\dots,i\right\}}\times S\right)
      \left\<
        \frac{U_{1},\dots,U_{i-1}}{p^{q-1}}\right\>\left\<\frac{U_{i}}{p^{q}}
      \right\>$.
    \item[(ii)]
      $\A_{\dR}^{>q}\left(\widetilde{\TT}_{C}^{\left\{1,\dots,i-1\right\}}\times S\right)
      \left\llbracket
        \frac{U_{1},\dots,U_{i-1}}{p^{q-1}}
      \right\rrbracket$
      is the $\left(p,\xi/p^{q},U_{1} / p^{q-1},\dots,U_{i-1} / p^{q-1}\right)$-adic
      completion of
      $\A_{\dR}^{q}\left(\widetilde{\TT}_{C}^{\left\{1,\dots,i-1\right\}}\times S\right)
      \left\<
        \frac{U_{1},\dots,U_{i-1}}{p^{q-1}}\right\>$.
    \item[(iii)]
      $\A_{\dR}^{>q}\left(\widetilde{\TT}_{C}^{\left\{1,\dots,i\right\}}\times S\right)
      \left\llbracket
        \frac{U_{1},\dots,U_{i}}{p^{q-1}}
      \right\rrbracket$
      is the $\left(p,\xi/p^{q},U_{1} / p^{q-1},\dots,U_{i} / p^{q-1}\right)$-adic
      completion of \\
      $\A_{\dR}^{q}\left(\widetilde{\TT}_{C}^{\left\{1,\dots,i\right\}}\times S\right)
      \left\<
        \frac{U_{1},\dots,U_{i}}{p^{q-1}}\right\>$.
  \end{itemize}
\end{notation}

Write $A_{\dR}^{>1}:=\A_{\dR}^{>1}\left(C,\cal{O}_{C}\right)$.
We view the rings introduced in
Notation~\ref{notation:initialcomputation-thecomplexEgreathanqbullet-largeformalpowerseriesrings}
as seminormed $A_{\dR}^{>1}$-algebras equipped with the
\begin{itemize}
  \item[(i)] $\left(p,\xi/p^{q},U_{1} / p^{q-1},\dots,U_{i-1} / p^{q-1},U_{i} / p^{q}\right)$-adic seminorm,
  \item[(ii)] $\left(p,\xi/p^{q},U_{1} / p^{q-1},\dots,U_{i-1} / p^{q-1}\right)$-adic seminorm, and
  \item[(iii)] $\left(p,\xi/p^{q},U_{1} / p^{q-1},\dots,U_{i} / p^{q-1}\right)$-adic seminorm,
\end{itemize}
respectively. In fact, this defines $A_{\dR}^{>1}$-Banach algebras
by~\cite[\href{https://stacks.math.columbia.edu/tag/05GG}{Tag 05GG}]{stacks-project}.

\begin{lem}\label{lem:grwithpnormalOAiq-etc}
  Let $q\in\NN_{\geq2}$.
  Equip the rings in
  Notation~\ref{notation:initialcomputation-thecomplexEgreathanqbullet-largeformalpowerseriesrings}
  with the
  \begin{itemize}
  \item[(i)] $\left(p,\xi/p^{q},U_{1} / p^{q-1},\dots,U_{i-1} / p^{q-1},U_{i} / p^{q}\right)$-adic filtration,
  \item[(ii)] $\left(p,\xi/p^{q},U_{1} / p^{q-1},\dots,U_{i-1} / p^{q-1}\right)$-adic filtration, and
  \item[(iii)] $\left(p,\xi/p^{q},U_{1} / p^{q-1},\dots,U_{i} / p^{q-1}\right)$-adic filtration,
\end{itemize}
respectively. Then we compute
  \begin{equation*}  
  \begin{split}
  &\gr\A_{\dR}^{>q}\left(\widetilde{\TT}_{C}^{\left\{1,\dots,i\right\}}\times S\right)
      \left\llbracket
        \frac{U_{1},\dots,U_{i-1}}{p^{q-1}},\frac{U_{i}}{p^{q}}
      \right\rrbracket \\
    &\quad\cong \left(\widehat{\cal{O}}^{+}\left( \widetilde{\TT}_{C}^{\left\{1,\dots,i\right\}}\times S \right)/p\right)\left[
      \sigma\left(p\right),
      \sigma\left(\frac{\xi}{p^{q}}\right),
      \sigma\left(\frac{U_{1}}{p^{q-1}}\right),
      \dots
      \sigma\left(\frac{U_{i-1}}{p^{q-1}}\right),
      \sigma\left(\frac{U_{i}}{p^{q}}\right)
      \right], \\
    &\gr\A_{\dR}^{>q}\left(\widetilde{\TT}_{C}^{\left\{1,\dots,i-1\right\}}\times S\right)
      \left\llbracket
        \frac{U_{1},\dots,U_{i-1}}{p^{q-1}}
      \right\rrbracket \\
    &\quad\cong \left(\widehat{\cal{O}}^{+}\left( \widetilde{\TT}_{C}^{\left\{1,\dots,i-1\right\}}\times S \right)/p\right)\left[
      \sigma\left(p\right),
      \sigma\left(\frac{\xi}{p^{q}}\right),
      \sigma\left(\frac{U_{1}}{p^{q-1}}\right),
      \dots
      \sigma\left(\frac{U_{i-1}}{p^{q-1}}\right)
      \right], \text{ and}\\
    &\gr\A_{\dR}^{>q}\left(\widetilde{\TT}_{C}^{\left\{1,\dots,i\right\}}\times S\right)
      \left\llbracket
        \frac{U_{1},\dots,U_{i}}{p^{q-1}}
      \right\rrbracket \\
    &\quad\cong \left(\widehat{\cal{O}}^{+}\left( \widetilde{\TT}_{C}^{\left\{1,\dots,i\right\}}\times S \right)/p\right)\left[
      \sigma\left(p\right),
      \sigma\left(\frac{\xi}{p^{q}}\right),
      \sigma\left(\frac{U_{1}}{p^{q-1}}\right),
      \dots
      \sigma\left(\frac{U_{i}}{p^{q-1}}\right)
      \right].
    \end{split}
    \end{equation*}
    All the principal symbols $\sigma(-)$ are of degree one.
  \end{lem}

\begin{proof}
  Here, we only establish the first isomorphism. The other two follow similarly.
  
  Equip
  $\A_{\dR}^{q}\left(\widetilde{\TT}_{C}^{\left\{1,\dots,i\right\}}\times S\right)
    \left\<
      \frac{U_{1},\dots,U_{i-1}}{p^{q-1}}\right\>\left\<\frac{U_{i}}{p^{q}}
    \right\>$
  with the $\left(p,\xi/p^{q},U_{1} / p^{q-1},\dots,U_{i-1} / p^{q-1},U_{i} / p^{q}\right)$-adic
  filtration. By definition, the canonical morphism
  \begin{equation*}
    \gr\A_{\dR}^{q}\left(\widetilde{\TT}_{C}^{\left\{1,\dots,i\right\}}\times S\right)
    \left\<
      \frac{U_{1},\dots,U_{i-1}}{p^{q-1}}\right\>\left\<\frac{U_{i}}{p^{q}}
    \right\>
    \stackrel{\cong}{\longrightarrow}
    \gr\A_{\dR}^{>q}\left(\widetilde{\TT}_{C}^{\left\{1,\dots,i\right\}}\times S\right)
      \left\llbracket
        \frac{U_{1},\dots,U_{i-1}}{p^{q-1}},\frac{U_{i}}{p^{q}}
      \right\rrbracket
  \end{equation*}
  is an isomorphism. Therefore, we may compute
  $\gr\A_{\dR}^{q}\left(\widetilde{\TT}_{C}^{\left\{1,\dots,i\right\}}\times S\right)
  \left\<
    \frac{U_{1},\dots,U_{i-1}}{p^{q-1}}\right\>\left\<\frac{U_{i}}{p^{q}}
  \right\>$.
  Note that
  \begin{equation*}
      \frac{U_{i}}{p^{q}},\frac{U_{i-1}}{p^{q-1}},\dots,\frac{U_{1}}{p^{q-1}},\frac{\xi}{p^{q}},p
      \in
      \A_{\dR}^{q}\left(\widetilde{\TT}_{C}^{\left\{1,\dots,i\right\}}\times S\right)
      \left\<
        \frac{U_{1},\dots,U_{i-1}}{p^{q-1}}\right\>\left\<\frac{U_{i}}{p^{q}}
      \right\>
  \end{equation*}
  is a regular sequence. This follows because the
  $U_{1} / p^{q-1},\dots,U_{i-1} / p^{q-1},U_{i} / p^{q}$ are formal variables,
  and because $\xi/p^{q},p$ is a regular sequence
  by Lemma~\ref{lem:xipq-p-regularseq-inAdRq-recpaper}.
  Therefore, by~\cite[Exercise 17.16]{Ei95}, it remains to compute the zeroth graded piece.
  This can be done with Lemma~\ref{lem:Fontaines-map-for-Ala}.
\end{proof}

\begin{lem}\label{lem:OAiq-etc-formalpowerseriesdescription}
  Fix $q\in\NN_{\geq2}$ and formal variables $\zeta_{1},\dots,\zeta_{i},\eta$.
  We have the isomorphisms
  \begin{equation}\label{eq:themaps--lem:OAiq-etc-formalpowerseriesdescription}
  \begin{split}
      \A_{\dR}^{>q}\left(\widetilde{\TT}_{C}^{\left\{1,\dots,i\right\}}\times S\right)
      \left\llbracket
        \zeta_{1},\dots,\zeta_{i-1},\eta
      \right\rrbracket      
      &\stackrel{\cong}{\longrightarrow}
      \A_{\dR}^{>q}\left(\widetilde{\TT}_{C}^{\left\{1,\dots,i\right\}}\times S\right)
      \left\llbracket
        \frac{U_{1},\dots,U_{i}}{p^{q-1}},\frac{U_{i}}{p^{q}}
      \right\rrbracket, \\
        \zeta_{j}&\mapsto\frac{U_{j}}{p^{q-1}} \text{ of all $j=1,\dots,i-1$,} \\
        \eta&\mapsto\frac{U_{i}}{p^{q}} \\
      %%%%%
      \A_{\dR}^{>q}\left(\widetilde{\TT}_{C}^{\left\{1,\dots,i-1\right\}}\times S\right)
      \left\llbracket
        \zeta_{1},\dots,\zeta_{i-1}
      \right\rrbracket      
      &\stackrel{\cong}{\longrightarrow}
      \A_{\dR}^{>q}\left(\widetilde{\TT}_{C}^{\left\{1,\dots,i-1\right\}}\times S\right)
      \left\llbracket
        \frac{U_{1},\dots,U_{i-1}}{p^{q-1}}
      \right\rrbracket, \text{ and} \\
      \zeta_{j}&\mapsto\frac{U_{j}}{p^{q-1}} \text{ of all $j=1,\dots,i-1$} \\
      %%%%%
      \A_{\dR}^{>q}\left(\widetilde{\TT}_{C}^{\left\{1,\dots,i\right\}}\times S\right)
      \left\llbracket
        \zeta_{1},\dots,\zeta_{i}
      \right\rrbracket      
      &\stackrel{\cong}{\longrightarrow}
      \A_{\dR}^{>q}\left(\widetilde{\TT}_{C}^{\left\{1,\dots,i\right\}}\times S\right)
      \left\llbracket
        \frac{U_{1},\dots,U_{i}}{p^{q-1}}
      \right\rrbracket \\
      \zeta_{j}&\mapsto\frac{U_{j}}{p^{q-1}} \text{ of all $j=1,\dots,i$}
  \end{split}
  \end{equation}
  of $\A_{\dR}^{>q}\left(\widetilde{\TT}_{C}^{\left\{1,\dots,l\right\}}\times S\right)$-Banach algebras,
  $l\in\{i-1,i\}$.
  Here, the domains of the maps carry the
  \begin{itemize}
    \item[(i)] $\left(p,\xi/p^{q},\zeta_{1},\dots,\zeta_{i-1},\eta\right)$-adic seminorm,
    \item[(ii)] $\left(p,\xi/p^{q},\zeta_{1},\dots,\zeta_{i-1}\right)$-adic seminorm, and
    \item[(iii)] $\left(p,\xi/p^{q},\zeta_{1},\dots,\zeta_{i}\right)$-adic seminorm.
  \end{itemize}
\end{lem}

\begin{proof}
  Firstly, the aforementioned algebras at the left-hand side are indeed Banach algebras,
  because $\A_{\dR}^{>q}\left(\widetilde{\TT}_{C}^{\left\{1,\dots,i\right\}}\times S\right)$
  is a Banach algebra carrying the $\left(p,\xi/p^{q}\right)$-adic norm
  and Proposition~\ref{prop:Scomplete-Spowerseriescomplete-ifKoszulregular}.
  \emph{Loc. cit.} applies because $\xi/p^{q},p$ is a regular sequence,
  which follows from Lemma~\ref{lem:Fontaines-map-for-Alagreatq}.

  The isomorphisms~(\ref{eq:themaps--lem:OAiq-etc-formalpowerseriesdescription}) are established
  by similar arguments. We thus give details only for
  \begin{equation*}
    \A_{\dR}^{>q}\left(\widetilde{\TT}_{C}^{\left\{1,\dots,i\right\}}\times S\right)
    \left\llbracket
      \zeta_{1},\dots,\zeta_{i-1},\eta
    \right\rrbracket      
    \stackrel{\cong}{\longrightarrow}
    \A_{\dR}^{>q}\left(\widetilde{\TT}_{C}^{\left\{1,\dots,i\right\}}\times S\right)
    \left\llbracket
      \frac{U_{1},\dots,U_{i-1}}{p^{q-1}},\frac{U_{i}}{p^{q}}
    \right\rrbracket.
  \end{equation*}
  It exists because the
  $U_{1} / p^{q-1},\dots,U_{i-1} / p^{q-1},U_{i} / p^{q}\in
    \A_{\dR}^{>q}\left(\widetilde{\TT}_{C}^{\left\{1,\dots,i\right\}}\times S\right)
    \left\llbracket
      \frac{U_{1}}{p^{q-1}},\dots,\frac{U_{i-1}}{p^{q-1}},\frac{U_{i}}{p^{q}}
    \right\rrbracket$
  are topologically nilpotent units. To show that it is an isomorphism,
  it suffices to check that their associated gradeds coincide, where
  the domain carries the
  $\left(p,\xi/p^{q},\zeta_{1},\dots,\zeta_{i-1},\eta\right)$-adic filtration,
  and the codomain carries the
  $\left(p,\xi/p^{q},U_{1} / p^{q-1},\dots,U_{i-1} / p^{q-1},U_{i} / p^{q}\right)$-adic filtration;
  this follows from~\cite[Chapter I, \S 4.2 page 31-32, Theorem 4(5)]{HuishiOystaeyen1996}.
  Compute with~\cite[Exercise 17.16]{Ei95}
  \begin{equation*}
    \begin{split}
    &\gr\A_{\dR}^{>q}\left(\widetilde{\TT}_{C}^{\left\{1,\dots,i\right\}}\times S\right)
    \left\llbracket
      \zeta_{1},\dots,\zeta_{i-1},\eta
    \right\rrbracket \\
    &\quad\cong
     \left(\widehat{\cal{O}}^{+}\left( \widetilde{\TT}_{C}^{\left\{1,\dots,i\right\}}\times S \right)/p\right)\left[
      \sigma\left(p\right),
      \sigma\left(\frac{\xi}{p^{q}}\right),
      \sigma\left(\zeta_{1}\right),
      \dots
      \sigma\left(\zeta_{i-1}\right),
      \sigma\left(\eta\right),
      \right].
    \end{split}
  \end{equation*}  
  \emph{Loc. cit.} is using that $\eta,\zeta_{i-1},\dots,\zeta_{1},\xi/p^{q},p$
  is a regular sequence, which follows because the
  $\zeta_{1},\dots,\zeta_{i-1},\eta$ are formal variables,
  and because $\xi/p^{q},p$ is a regular sequence
  by Lemma~\ref{lem:xipq-p-regularseq-inAdRq-recpaper}.
  It then remains to compute the zeroth graded piece, which can be
  done with Lemma~\ref{lem:Fontaines-map-for-Alagreatq}.
  Now use the description of the other associated graded as in Lemma~\ref{lem:grwithpnormalOAiq-etc}
  to finish this proof.
\end{proof}

\begin{notation}
  Let $q\in\NN_{\geq1}$ and define
  \begin{equation*}  
  \begin{split}
    \normalOA_{i-1,i}^{>q-1,q}
      &:=\A_{\dR}^{>q}\left(\widetilde{\TT}_{C}^{\left\{1,\dots,i\right\}}\times S\right)
      \left\llbracket
        \frac{U_{1},\dots,U_{i-1}}{p^{q-1}},\frac{U_{i}}{p^{q}}
      \right\rrbracket
      \widehat{\otimes}_{k^{\circ}}
      \cal{O}\left( \TT^{\left\{i+1,\dots,d\right\}}\right)^{\circ}, \\
    \normalOA_{i-1}^{>q-1}
      &:=\A_{\dR}^{>q}\left(\widetilde{\TT}_{C}^{\left\{1,\dots,i-1\right\}}\times S\right)
      \left\llbracket
        \frac{U_{1},\dots,U_{i-1}}{p^{q-1}}
      \right\rrbracket
      \widehat{\otimes}_{k^{\circ}}
      \cal{O}\left( \TT^{\left\{i,\dots,d\right\}}\right)^{\circ}, \text{ and}\\
    \normalOA_{i}^{>q-1}
      &:=\A_{\dR}^{>q}\left(\widetilde{\TT}_{C}^{\left\{1,\dots,i\right\}}\times S\right)
      \left\llbracket
        \frac{U_{1},\dots,U_{i}}{p^{q-1}}
      \right\rrbracket
      \widehat{\otimes}_{k^{\circ}}
      \cal{O}\left( \TT^{\left\{i+1,\dots,d\right\}}\right)^{\circ}.
  \end{split}
  \end{equation*}
\end{notation}

\begin{lem}\label{lem:calOTTecirc-preservesstrictmono-reconstructionpaper}
  For every $e\in\NN$,
  $-\widehat{\otimes}_{k^{\circ}}\cal{O}\left( \TT^{e} \right)^{\circ}\colon\Ban_{k^{\circ}}\to\Ban_{k^{\circ}}$
  is strongly exact.
\end{lem}

\begin{proof}
  We refer the reader to~\cite[chapter 10]{Sch02}
  for the definition of the Banach spaces $c_{0}(\Omega)$ for a given set $\Omega$.
  \emph{Loc. cit.} Proposition 10.1 implies that
  $\cal{O}\left( \TT^{e} \right)^{\circ}$ is in fact isomorphic to a
  $k$-Banach space $c_{0}(\Omega)$, for some set $\Omega$ which we fix.
  In the notation of \S\ref{subsec:Tubular},
  we observe
  \begin{equation*}
    c_{0}(\Omega)^{\circ}
    \cong k\left\<\zeta_{\omega}\colon \omega\in \Omega\right\>^{\circ}
    \cong k^{\circ}\left\<\zeta_{\omega}\colon \omega\in \Omega\right\>
    \cong{\coprod_{\alpha\in\NN^{(\Omega)}}}^{\leq 1} k^{\circ} \zeta^{\alpha}
  \end{equation*}
  Here, we utilised \emph{loc. cit.} Remark~\ref*{rem:restrictedpowerseries-is-c0space}
  and Lemma~\ref*{lem:infiniteTatealgebrasascontractingcoproductandc0}.
  Therefore, Lemma~\ref{lem:calOTTecirc-preservesstrictmono-reconstructionpaper}
  follows from the strong flatness of
  the $k^{\circ}$-Banach module
  ${\coprod_{\alpha\in\NN^{(\Omega)}}}^{\leq 1} k^{\circ} \zeta^{\alpha}$,
  cf.~\cite[Proposition 5.1.16]{benbassat2024perspectivefoundationsderivedanalytic}.
\end{proof}

\begin{lem}\label{lem:etamupq-decelagemakessense-checkconditionii}
  For all $q\in\NN_{\geq2}$, the following inclusions are strict monomorphisms:
  \begin{equation}\label{lem:etamupq-decelagemakessense-checkconditionii-theinclusions}
  \begin{split}
    \frac{t}{p^{q}}\normalOA_{i-1,i}^{>q-1,q} &\hookrightarrow \normalOA_{i-1,i}^{>q-1,q}, \\
    \frac{t}{p^{q}}\normalOA_{i-1}^{>q-1} &\hookrightarrow \normalOA_{i-1}^{>q-1}, \text{ and} \\
    \frac{t}{p^{q}}\normalOA_{i}^{>q-1} &\hookrightarrow \normalOA_{i}^{>q-1}.
  \end{split}
  \end{equation}
\end{lem}

\begin{proof}
  The strict exactness of the first map would follows from the strict exactness of
  \begin{equation*}
    \frac{t}{p^{q}}\times\colon\normalOA_{i-1,i}^{>q-1,q} \to \normalOA_{i-1,i}^{>q-1,q}.
  \end{equation*}
  By Lemma~\ref{lem:calOTTecirc-preservesstrictmono-reconstructionpaper},
  it suffices to check that the multiplication-by-$t/p^{q}$ map
  \begin{equation*}
    \frac{t}{p^{q}}\times\colon
    \A_{\dR}^{>q}\left(\widetilde{\TT}_{C}^{\left\{1,\dots,i\right\}}\times S\right)
    \left\llbracket
      \frac{U_{1},\dots,U_{i-1}}{p^{q-1}},\frac{U_{i}}{p^{q}}
    \right\rrbracket
    \to
    \A_{\dR}^{>q}\left(\widetilde{\TT}_{C}^{\left\{1,\dots,i\right\}}\times S\right)
    \left\llbracket
      \frac{U_{1},\dots,U_{i-1}}{p^{q-1}},\frac{U_{i}}{p^{q}}
    \right\rrbracket    
  \end{equation*}
  is a strict monomorphism. Thanks to
  Lemma~\ref{lem:OAiq-etc-formalpowerseriesdescription},
  it suffices to check that
  \begin{equation*}
    \frac{t}{p^{q}}\times\colon
    \A_{\dR}^{>q}\left(\widetilde{\TT}_{C}^{\left\{1,\dots,i\right\}}\times S\right)
    \left\llbracket
      \zeta_{1},\dots,\zeta_{i-1},\eta
    \right\rrbracket
    \to
    \A_{\dR}^{>q}\left(\widetilde{\TT}_{C}^{\left\{1,\dots,i\right\}}\times S\right)
    \left\llbracket
      \zeta_{1},\dots,\zeta_{i-1},\eta
    \right\rrbracket    
  \end{equation*}
  is a strict monomorphism.
  This is Lemma~\ref{lem:multbytpq-onAdRgreaterthanq-strictmono-reconstructionpaper}.
  
  The strict exactness of the second and third inclusion
  in~(\ref{lem:etamupq-decelagemakessense-checkconditionii-theinclusions})
  follows from similar arguments.
\end{proof}

\begin{lem}\label{lem:etamupq-decelagemakessense-checkconditioniii}
  $\normalOA_{i-1,i}^{>q-1,q}$,
  $\normalOA_{i-1}^{>q-1}$, and $\normalOA_{i}^{>q-1}$
  are $t/p^{q}$-torsion free.
\end{lem}

\begin{proof}
  We check that $\normalOA_{i-1,i}^{>q-1,q}$ is $t/p^{q}$-torsion free.
  Thanks to Lemma~\ref{lem:calOTTecirc-preservesstrictmono-reconstructionpaper},
  it remains to check that
  $\A_{\dR}^{>q}\left(\widetilde{\TT}_{C}^{\left\{1,\dots,i\right\}}\times S\right)
    \left\llbracket
      \frac{U_{1},\dots,U_{i-1}}{p^{q-1}},\frac{U_{i}}{p^{q}}
    \right\rrbracket$
  is $t/p^{q}$-torsion free. By Lemma~\ref{lem:OAiq-etc-formalpowerseriesdescription},
  it suffices to check that
  $\A_{\dR}^{>q}\left(\widetilde{\TT}_{C}^{\left\{1,\dots,i\right\}}\times S\right)$
  is $t/p^{q}$-torsion free, which we have already checked:
  see Lemma~\ref{lem:canetaoperatorAdRgreaterthanqRRplus-ii}.

  One proceeds similarly to check that $\normalOA_{i-1}^{>q-1}$,
  and $\normalOA_{i}^{>q-1}$ are $t/p^{q}$-torsion free.
\end{proof}

Write $B_{\dR}^{>1}:=\BB_{\dR}^{>1}\left(C,\cal{O}_{C}\right)$.

\begin{lem}\label{lem:normalOAila-and-normalOBila}
  We have the canonical isomorphisms of $k$-ind-Banach spaces
  \begin{equation*}  
  \begin{split}
    \varinjlim_{q\in\NN_{\geq1}}
    \normalOA_{i-1,i}^{>q-1,q}\widehat{\otimes}_{A_{\dR}^{>1}}B_{\dR}^{>1}
      &\stackrel{\cong}{\longrightarrow}\normalOB_{i}, \\
    \varinjlim_{q\in\NN_{\geq1}}
    \normalOA_{i-1}^{>q-1}\widehat{\otimes}_{A_{\dR}^{>1}}B_{\dR}^{>1}
      &\stackrel{\cong}{\longrightarrow}\normalOB_{i-1}, \text{ and} \\
    \varinjlim_{q\in\NN_{\geq1}}
    \normalOA_{i}^{>q-1}\widehat{\otimes}_{A_{\dR}^{>1}}B_{\dR}^{>1}
      &\stackrel{\cong}{\longrightarrow}\normalOB_{i}.
  \end{split}
  \end{equation*}
\end{lem}
  
\begin{proof}
  Apply Lemma~\ref{lem:BBdR-fromAAdR-and-BdRgreaterthan1-reconstructionpaper}
  to find the canonical isomorphism
  \begin{equation}\label{eq:normalOAila-and-normalOBila-iso1}
    \varinjlim_{q\in\NN_{\geq1}}
    \left(\A_{\dR}^{q}\left(\widetilde{\TT}_{C}^{\left\{1,\dots,i\right\}}\times S\right)
      \left\<
        \frac{U_{1},\dots,U_{i-1}}{p^{q-1}}\right\>\left\<\frac{U_{i}}{p^{q}}
      \right\>\widehat{\otimes}_{k^{\circ}}
      \cal{O}\left( \TT^{\left\{i+1,\dots,d\right\}}\right)^{\circ}\right)
      \widehat{\otimes}_{A_{\dR}^{>1}}B_{\dR}^{>1}
    \isomap\normalOB_{i}.
  \end{equation}
  Next, we observe that the canonical maps
  \begin{equation*}
    \iota^{q}\colon
    \A_{\dR}^{q}\left(\widetilde{\TT}_{C}^{\left\{1,\dots,i\right\}}\times S\right)
      \left\<
        \frac{U_{1},\dots,U_{i-1}}{p^{q-1}}\right\>\left\<\frac{U_{i}}{p^{q}}
      \right\>
    \to
    \A_{\dR}^{q+1}\left(\widetilde{\TT}_{C}^{\left\{1,\dots,i\right\}}\times S\right)
      \left\<
        \frac{U_{1},\dots,U_{i-1}}{p^{q}}\right\>\left\<\frac{U_{i}}{p^{q+1}}
      \right\>
  \end{equation*}
  send all elements $p,\xi/p^{q},U_{1} / p^{q-1},\dots,U_{i-1} / p^{q-1}$, and $U_{i} / p^{q}$
  to topologically nilpotent elements. Therefore, the $\iota^{q}$
  factor canonically through bounded linear maps
  \begin{equation*}
    \A_{\dR}^{>q}\left(\widetilde{\TT}_{C}^{\left\{1,\dots,i\right\}}\times S\right)
      \left\llbracket
        \frac{U_{1},\dots,U_{i-1}}{p^{q-1}}\frac{U_{i}}{p^{q}}
      \right\rrbracket
    \to
    \A_{\dR}^{q+1}\left(\widetilde{\TT}_{C}^{\left\{1,\dots,i\right\}}\times S\right)
      \left\<
        \frac{U_{1},\dots,U_{i-1}}{p^{q}}\right\>\left\<\frac{U_{i}}{p^{q+1}}
      \right\>.
  \end{equation*}  
  This allows to construct a two-sided inverse of the canonical morphism
  \begin{equation}\label{eq:normalOAila-and-normalOBila-iso2}  
    \varinjlim_{q\in\NN_{\geq1}}
      \A_{\dR}^{q}\left(\widetilde{\TT}_{C}^{\left\{1,\dots,i\right\}}\times S\right)
      \left\<
        \frac{U_{1},\dots,U_{i-1}}{p^{q-1}}\right\>\left\<\frac{U_{i}}{p^{q}}
      \right\>
    \isomap
    \varinjlim_{q\in\NN_{\geq1}}
    \A_{\dR}^{>q}\left(\widetilde{\TT}_{C}^{\left\{1,\dots,i\right\}}\times S\right)
      \left\llbracket
        \frac{U_{1},\dots,U_{i-1}}{p^{q-1}}\frac{U_{i}}{p^{q}}
      \right\rrbracket,
  \end{equation}
  which is thus an isomorphism. We can now put everything together:
  \begin{align*}
    \normalOB_{i}
    &\stackrel{\text{(\ref{eq:normalOAila-and-normalOBila-iso1})}}{\cong}
    \varinjlim_{q\in\NN_{\geq1}}
    \left(\A_{\dR}^{q}\left(\widetilde{\TT}_{C}^{\left\{1,\dots,i\right\}}\times S\right)
      \left\<
        \frac{U_{1},\dots,U_{i-1}}{p^{q-1}}\right\>\left\<\frac{U_{i}}{p^{q}}
      \right\>\widehat{\otimes}_{k^{\circ}}
      \cal{O}\left( \TT^{\left\{i+1,\dots,d\right\}}\right)^{\circ}\right)
      \widehat{\otimes}_{A_{\dR}^{>1}}B_{\dR}^{>1} \\
    &\cong
    \left(\varinjlim_{q\in\NN_{\geq1}}\A_{\dR}^{q}\left(\widetilde{\TT}_{C}^{\left\{1,\dots,i\right\}}\times S\right)
      \left\<
        \frac{U_{1},\dots,U_{i-1}}{p^{q-1}}\right\>\left\<\frac{U_{i}}{p^{q}}
      \right\>\widehat{\otimes}_{k^{\circ}}
      \cal{O}\left( \TT^{\left\{i+1,\dots,d\right\}}\right)^{\circ}\right)
      \widehat{\otimes}_{A_{\dR}^{>1}}B_{\dR}^{>1} \\
    &\stackrel{\text{(\ref{eq:normalOAila-and-normalOBila-iso2})}}{\cong}
    \left(
        \varinjlim_{q\in\NN_{\geq1}}
    \A_{\dR}^{>q}\left(\widetilde{\TT}_{C}^{\left\{1,\dots,i\right\}}\times S\right)
      \left\llbracket
        \frac{U_{1},\dots,U_{i-1}}{p^{q-1}}\frac{U_{i}}{p^{q}}
      \right\rrbracket
    \widehat{\otimes}_{k^{\circ}}
      \cal{O}\left( \TT^{\left\{i+1,\dots,d\right\}}\right)^{\circ}\right)
      \widehat{\otimes}_{A_{\dR}^{>1}}B_{\dR}^{>1} \\ 
      &\cong\varinjlim_{q\in\NN_{\geq1}}
    \normalOA_{i-1,i}^{>q-1,q}\widehat{\otimes}_{A_{\dR}^{>1}}B_{\dR}^{>1}.
  \end{align*}
  This establishes the first isomorphism in Lemma~\ref{lem:normalOAila-and-normalOBila}.
  The other two follow similarly.
\end{proof}

%%%%%%%%%%%%%%%%%%%%%%%%%%%%%%%%%%%%%%%%%%%%%%%%%%%%%%%%%%%%
%%%%%%%%%%%%%%%%%%%%%%%%%%%%%%%%%%%%%%%%%%%%%%%%%%%%%%%%%%%%
% Cech cohomology I
%%%%%%%%%%%%%%%%%%%%%%%%%%%%%%%%%%%%%%%%%%%%%%%%%%%%%%%%%%%%
%%%%%%%%%%%%%%%%%%%%%%%%%%%%%%%%%%%%%%%%%%%%%%%%%%%%%%%%%%%%

\subsection{The cochain complexes $E^{>q,\bullet}$}\label{subsubsec:initialcomputation-thecomplexEgreathanqbullet}

We would like to recover~(\ref{eq:cont-group-coh-OBlatTTKd-the-sequence})
from complexes $E^{>q,\bullet}$ built out of $\normalOA_{i-1,i}^{>q-1,q}$,
$\normalOA_{i-1}^{>q-1}$, and $\normalOA_{i}^{>q-1}$, which we
introduced in the previous \S\ref{subsubsec:initialcomputation-normalOAi-1ietc}.
These objects have the advantage that they carry complete filtrations,
thus we could proceed by computing the associated gradeds.
We already defined the building blocks for the objects of $E^{>q,\bullet}$.
Now we work towards the definition of their differentials.

\begin{lem}\label{lem:liftmapstoTatealgebras-reconstructionpaper}
  Consider a morphism $\phi\colon A_{1}\to A_{2}$ of $A_{\dR}^{>1}$-Banach algebras.
  Let $\zeta_{1},\dots,\zeta_{j}$ denote formal variables, $h\in\NN$,
  and fix power-bounded elements $a_{1},\dots,a_{j}\in A_{2}$.
  Then there exists a unique morphism
  $A_{1}\left\<\frac{\zeta_{1},\dots,\zeta_{j}}{p^{h}}\right\>\to A_{2}$
  of $A_{\dR}^{>1}$-Banach algebras as follows:
   \begin{itemize}
     \item[(i)] It restricts to $\phi$ on $A_{1}$, and
     \item[(ii)] $\zeta_{l}/p^{h} \mapsto a_{l}$ for all $l=1,\dots,j$.
   \end{itemize} 
\end{lem}

\begin{proof}
  Uniqueness is clear. To show existence, we recall
  \begin{equation*}
    A_{1}\left\<\frac{\zeta_{1},\dots,\zeta_{j}}{p^{h}}\right\>
    =A_{1}\left\<\zeta_{1},\dots,\zeta_{j}\right\>\left\<\eta_{1},\dots,\eta_{j}\right\> /
    \overline{\left( p^{h}\eta_{1}-\zeta_{1} , \dots , p^{h}\eta_{j}-\zeta_{j} \right)}.
  \end{equation*}
  It remains to construct a morphism
  $\psi\colon A_{1}\left\<\eta_{1},\dots,\eta_{j}\right\>\to A_{2}$
  of $A_{\dR}^{>1}$-Banach algebras as follows:
   \begin{itemize}
     \item[(i)] It restricts to $\phi$ on $A_{1}$, and
     \item[(ii)] $\eta_{l} \mapsto a_{l}$ for all $l=1,\dots,j$.
   \end{itemize}
   (i) and (ii) force
   $\psi\left(\sum_{\alpha\in\NN^{j}}\lambda_{\alpha}\zeta^{\alpha}\right):=
   \sum_{\alpha\in\NN^{j}}\phi\left(\lambda_{\alpha}\right)a^{\alpha}$,
   where we have used the multi-index notation
   $\eta=\left(\eta_{1},\dots,\eta_{j}\right)$ and
   $a=\left(a_{1},\dots,a_{j}\right)$. We check well-definedness:
   \begin{equation*}
     \|\phi\left(\lambda_{\alpha}\right)a^{\alpha}\|
     \leq C \|\phi\|\left(\lambda_{\alpha}\right)\|\|a^{\alpha}\|
     \leq C^{j} \|\phi\|\|\lambda_{\alpha}\|\|a_{1}^{\alpha_{1}}\|\cdots\|a_{j}^{\alpha_{j}}\|
     \to 0 \text{ for } |\alpha|\to\infty,
   \end{equation*}
   where $C$ denotes a bound on the multiplication on $A_{2}$.
   Furthermore, this computation shows
   \begin{equation*}
     \|\psi\| \leq C^{j} \|\phi\|\|a_{1}^{\alpha_{1}}\|\cdots\|a_{j}^{\alpha_{j}}\|,
   \end{equation*}
   that is the boundedness of $\psi$.
   It follows that $\psi$ is a morphism of $A_{\dR}^{>1}$-algebras.
\end{proof}

Recall that we fixed a suitable $\ZZ_{p}$-basis
$\gamma_{i},\dots,\gamma_{i}$ of $\ZZ_{p}(1)^{i}\cong\ZZ_{p}^{i}$,
cf. the discussion following Remark~\ref{remark--notation:normalOBi-reconstructionpaper}.
By functoriality, the action of $\gamma_{i}$ on $\widetilde{\TT}_{C}^{\{1,\dots,i\}}$
gives rise to an endomorphism
\begin{equation*}
  \gamma_{i}\colon
  \A_{\dR}^{q}\left(\widetilde{\TT}_{C}^{\left\{1,\dots,i\right\}}\times S\right)
  \to\A_{\dR}^{q}\left(\widetilde{\TT}_{C}^{\left\{1,\dots,i\right\}}\times S\right)
\end{equation*}
for every $q\in\NN_{\geq1}$.
It appears in the following Lemma~\ref{lem:definegammai-on-theintegraltatering1}.

\begin{lem}\label{lem:definegammai-on-theintegraltatering1}
  Let $q\in\NN_{\geq1}$.
  The $A_{\dR}^{>1}$-Banach algebra
  $\A_{\dR}^{q}\left(\widetilde{\TT}_{C}^{\left\{1,\dots,i\right\}}\times S\right)
    \left\<
    \frac{U_{1},\dots,U_{i-1}}{p^{q-1}}\right\>\left\<\frac{U_{i}}{p^{q}}
    \right\>$
     admits a unique endomorphism
     which is determined as follows:
   \begin{itemize}
     \item[(i)] It restricts to $\gamma_{i}$ on $\A_{\dR}^{q}\left(\widetilde{\TT}_{C}^{\left\{1,\dots,i\right\}}\times S\right)$, and
     \item[(ii)] $U_{i}/p^{q} \mapsto t/p^{q} + U_{i}/p^{q}$ as well as $U_{j}/p^{q-1} \mapsto U_{j}/p^{q-1}$ for all $j=1,\dots,i-1$.
   \end{itemize} 
\end{lem}

\begin{proof}
  By Lemma~\ref{lem:liftmapstoTatealgebras-reconstructionpaper},
  we have to check that $t/p^{q} + U_{i}/p^{q}$ and $U_{j}/p^{q-1}$ for all $j=1,\dots,i-1$
  are power-bounded. For $t/p^{q} + U_{i}/p^{q}$, the follows from the power boundedness of
  $t/p^{q}$ and $U_{i}/p^{q}$, since
  \begin{equation*}
    \|\left(\frac{t}{p^{q}} + \frac{U_{i}}{p^{q}}\right)^{n}\|
    =\|\sum_{m=0}^{n}\binom{n}{m}\left(\frac{t}{p^{q}}\right)^{m}\left(\frac{U_{i}}{p^{q}}\right)^{n-m} \|
    \leq\max_{m=0,\dots,n}\max\left\{\|\left(\frac{t}{p^{q}}\right)^{m}\| , \|\left(\frac{U_{i}}{p^{q}}\right)^{n-m}\|\right\}.
  \end{equation*}
  Clearly, $U_{j}/p^{q-1}$ is power-bounded whenever $j=1,\dots,i$.
\end{proof}

\begin{cor}\label{cor:definegammai-on-theintegralformalpowerseriesring1}
  Let $q\in\NN_{\geq1}$.
  The $A_{\dR}^{>1}$-Banach algebra
  $\A_{\dR}^{>q}\left(\widetilde{\TT}_{C}^{\left\{1,\dots,i\right\}}\times S\right)
    \left\llbracket
    \frac{U_{1},\dots,U_{i-1}}{p^{q-1}}, \frac{U_{i}}{p^{q}}\right\rrbracket$
  admits a unique endomorphism which is determined as follows:
   \begin{itemize}
     \item[(i)] Its composition with
     $\A_{\dR}^{q}\left(\widetilde{\TT}_{C}^{\left\{1,\dots,i\right\}}\times S\right)
       \to\A_{\dR}^{>q}\left(\widetilde{\TT}_{C}^{\left\{1,\dots,i\right\}}\times S\right)
    \left\llbracket
    \frac{U_{1},\dots,U_{i-1}}{p^{q-1}}, \frac{U_{i}}{p^{q}}\right\rrbracket$
     is $a\mapsto\gamma_{i}(a)$. %given by $\gamma_{i}$.
     \item[(ii)] $U_{i}/p^{q} \mapsto t/p^{q} + U_{i}/p^{q}$ as well as $U_{j}/p^{q-1} \mapsto U_{j}/p^{q-1}$ for all $j=1,\dots,i-1$.
   \end{itemize} 
\end{cor}

\begin{proof}
  Uniqueness is clear. To show existence, consider the composition
  \begin{multline}\label{multline:definegammai-on-theintegralformalpowerseriesring1-eq1}
    \A_{\dR}^{q}\left(\widetilde{\TT}_{C}^{\left\{1,\dots,i\right\}}\times S\right)
    \left\<
    \frac{U_{1},\dots,U_{i-1}}{p^{q-1}}\right\>\left\<\frac{U_{i}}{p^{q}}
    \right\>
    \to
    \A_{\dR}^{q}\left(\widetilde{\TT}_{C}^{\left\{1,\dots,i\right\}}\times S\right)
    \left\<
    \frac{U_{1},\dots,U_{i-1}}{p^{q-1}}\right\>\left\<\frac{U_{i}}{p^{q}}
    \right\> \\
    \to
    \A_{\dR}^{>q}\left(\widetilde{\TT}_{C}^{\left\{1,\dots,i\right\}}\times S\right)
    \left\llbracket
    \frac{U_{1},\dots,U_{i-1}}{p^{q-1}}, \frac{U_{i}}{p^{q}}\right\rrbracket
  \end{multline}
  of the endomorphism described by Lemma~\ref{lem:definegammai-on-theintegraltatering1}
  and the canonical map. It follows from \emph{loc. cit.} that it sends
  $p,t/p^{q},U_{1}/p^{q-1},\dots,U_{i-1}/p^{q-1}$, and $U_{i}/p^{q}$
  to topologically nilpotent elements. Therefore,
  we may extend~(\ref{multline:definegammai-on-theintegralformalpowerseriesring1-eq1})
  by continuity to get
  \begin{equation*}
    \A_{\dR}^{>q}\left(\widetilde{\TT}_{C}^{\left\{1,\dots,i\right\}}\times S\right)
    \left\llbracket
    \frac{U_{1},\dots,U_{i-1}}{p^{q-1}}, \frac{U_{i}}{p^{q}}\right\rrbracket
    \to
    \A_{\dR}^{>q}\left(\widetilde{\TT}_{C}^{\left\{1,\dots,i\right\}}\times S\right)
    \left\llbracket
    \frac{U_{1},\dots,U_{i-1}}{p^{q-1}}, \frac{U_{i}}{p^{q}}\right\rrbracket,
  \end{equation*}
  the desired endomorphism.
\end{proof}

%%% fi

\begin{defn}\label{defn:gammai-on-overlinenormalOAigreaterthanq-reconstructionpaper}
  Fix $q\in\NN_{\geq1}$.
  Let $\gamma_{i}^{\prime}$ denote the endomorphism described in
  Corollary~\ref{cor:definegammai-on-theintegralformalpowerseriesring1}.
  Then
  \begin{equation*}
    \gamma_{i}:=\gamma_{i}^{\prime}\widehat{\otimes}_{k^{\circ}}\id_{\cal{O}\left(\TT^{\left\{ i+1 , \dots, d \right\}}\right)^{\circ}}
  \end{equation*}
  is an endomorphism of the $A_{\dR}^{>1}$-Banach algebra $\normalOA_{i-1,i}^{>q-1,q}$.
\end{defn}

\begin{remark}
  Definition~\ref{defn:gammai-on-overlinenormalOAigreaterthanq-reconstructionpaper} is not
  an abuse of notation. To be precise, the diagram
  \begin{equation*}
  \begin{tikzcd}
    \normalOB_{i}
    \arrow{r}{\gamma_{i}} &
    \normalOB_{i} \\
    \normalOA_{i-1,i}^{>q-1,q}
    \arrow{u}
    \arrow{r}{\gamma_{i}} &
    \normalOA_{i-1,i}^{>q-1,q}
    \arrow{u},
  \end{tikzcd}
  \end{equation*}
  where the vertical maps are the canonical ones,
  is commutative. This follows from Corollary~\ref{cor:gammai-acts-on-Uj}.
\end{remark}

Here is the main definition of this \S\ref{subsubsec:initialcomputation-thecomplexEgreathanqbullet}:

\begin{defn}\label{defn:Egreaterthanqbullet}
  For any $q\in\NN_{\geq1}$, $E^{>q,\bullet}$ denotes the sequence of maps
  of $A_{\dR}^{>1}$-Banach algebras
  \begin{multline*}%\label{eq:cont-group-coh-OBlatTTKd-the-sequence-toshowisexact}
%  \begin{split}
    0
    \longrightarrow
    \normalOA_{i-1}^{>q-1}\times_{\normalOA_{i-1,i}^{>q-1,q}}
    \left(
      \normalOA_{i-1,i}^{>q-1,q} \times_{\gamma_{i}-1,\normalOA_{i-1,i}^{>q-1,q}}\normalOA_{i}^{>q-1}
    \right) \\
    \stackrel{\phi_{0}}{\longrightarrow}
    \normalOA_{i-1,i}^{>q-1,q} \times_{\gamma_{i}-1,\normalOA_{i-1,i}^{>q-1,q}}\normalOA_{i}^{>q-1}
    \stackrel{\phi_{1}}{\longrightarrow}
    \normalOA_{i}^{>q} \longrightarrow
    0.
%  \end{split}
  \end{multline*}
  Both $\phi_{0}$ and $\phi_{1}$ are given by $\left(f , g\right)\mapsto g$.
  See Definition~\ref{defn:gammai-on-overlinenormalOAigreaterthanq-reconstructionpaper}
  for $\gamma_{i}$.
\end{defn}

\begin{lem}\label{lem:Egreaterthanqbullet-cochaincomplex-reconstructionpaper}
  $E^{>q,\bullet}$ is a cochain complex of $A_{\dR}^{>1}$-Banach modules
  for all $q\in\NN_{\geq1}$.
\end{lem}

\begin{proof}
  We have to show $\left(\phi_{1}\circ\phi_{0}\right)\left(\left(f , \left( g , h\right)\right)\right)=0$ for all elements
  \begin{equation*}
    \left(f , \left( g , h\right)\right) \in
    \normalOA_{i-1}^{>q-1}\times_{\normalOA_{i-1,i}^{>q-1,q}}
    \left(
      \normalOA_{i-1,i}^{>q-1,q} \times_{\gamma_{i}-1,\normalOA_{i-1,i}^{>q-1,q}}\normalOA_{i}^{>q-1}
    \right).
  \end{equation*}
  That is, we have to check $h=0$. To do this, we note that
  the restriction of $\gamma_{i}$ to $\normalOA_{i-1}^{>q-1}$
  is the identity. This is the case because $\gamma_{i}$
  fixes $\A_{\dR}^{>q}\left(\widetilde{\TT}_{K}^{\left\{1,\dots,i-1\right\}}\right)$
  and by Corollary~\ref{cor:gammai-acts-on-Uj}.
  But $f\in\normalOA_{i-1}^{>q-1}$, therefore
  $h=\left(\gamma_{i}-1\right)(g)=\left(\gamma_{i}-1\right)(f)=0$.
\end{proof}

We introduce Convention~\ref{conv:Egreaterthanqbullet-concentratedindegree012}
to apply the $\eta$-operator to $E^{>q,\bullet}$,
cf. Definition~\ref{defn:algebraic-decalage}.

\begin{conv}\label{conv:Egreaterthanqbullet-concentratedindegree012}
  Let $q\in\NN_{\geq1}$. $E^{>q,\bullet}$ is, as a cochain complex,
  concentrated in degrees $0$, $1$ and $2$.
\end{conv}

See the beginning of \S\ref{subsec:initialcomputation}
for the definition $\mu:=[\epsilon]-1\in A_{\inf}$. We denote its image
under the canonical map $A_{\inf}\to A_{\dR}^{>1}$ again by $\mu$.
  
\begin{lem}\label{lem:etamupq-decelagemakessense}
  $\eta_{\mu/p^{q}}E^{>q,\bullet}\subseteq E^{>q,\bullet}$,
  equipped with the induced norms, is a cochain complex
  of $A_{\dR}^{>1}$-Banach modules.
\end{lem}
  
\begin{proof}
  $E^{>q,\bullet}$ is a cochain complex by Lemma~\ref{lem:Egreaterthanqbullet-cochaincomplex-reconstructionpaper}.
  Next, we check that Condition~\ref{cond:Banachmoduledecalage-reconstructionpaper}
  is satisfied for
  $R=A_{\dR}^{>1}$,
  $r=\mu/p^{q}$, and
  $M^{\bullet}=E^{>q,\bullet}$:
  \begin{itemize}
    \item[(i)] $E^{>q,\bullet}$ is concentrated in non-negative degrees by
      Convention~\ref{conv:Egreaterthanqbullet-concentratedindegree012}.
    \item[(ii)] As $\mu/p^{q}$ and $t/p^{q}$ coincide up to a unit in $A_{\dR}^{>1}$,
      cf. the proof of Lemma~\ref{Bla-independent-of-epsilon},
      Lemma~\ref{lem:etamupq-decelagemakessense-checkconditionii} implies that
      the inclusions
      \begin{align*}
        \frac{\mu}{p^{q}}\normalOA_{i-1,i}^{>q-1,q} &\hookrightarrow \normalOA_{i-1,i}^{>q-1,q}, \\
        \frac{\mu}{p^{q}}\normalOA_{i-1}^{>q-1} &\hookrightarrow \normalOA_{i-1}^{>q-1}, \text{ and} \\
        \frac{\mu}{p^{q}}\normalOA_{i}^{>q-1} &\hookrightarrow \normalOA_{i}^{>q-1}.
      \end{align*}    
      are strict monomorphisms. Because taking pullbacks preserves strict monomorphisms,
      cf.~\cite[Lemma 3.7]{BBKFrechetModulesDescent}, this gives (ii).
    \item[(iii)] This is Lemma~\ref{lem:etamupq-decelagemakessense-checkconditioniii},
      again because
      as $\mu/p^{q}$ and $t/p^{q}$ coincide up to a unit in $A_{\dR}^{>1}$.
  \end{itemize}
  Now apply Lemma~\ref{lem:decalage-defined-reconstructionpaper}.
\end{proof}

Lemma~\ref{lem:cont-group-coh-OBlatTTKd-Dbullet}
reduces the proof of Proposition~\ref{prop:cont-group-coh-OBlatTTKd}
to a study of the cochain complexes $\eta_{\mu/p^{q}}E^{>q,\bullet}$.
  
  \begin{lem}\label{lem:cont-group-coh-OBlatTTKd-Dbullet}
   Proposition~\ref{prop:cont-group-coh-OBlatTTKd}
   follows if the $\eta_{\mu/p^{q}}E^{>q,\bullet}$ are strictly exact for $q\in\NN$ large enough.
  \end{lem}
  
  \begin{proof}
  The upper row of the following commutative diagram is~(\ref{eq:cont-group-coh-OBlatTTKd-the-sequence}):
  \begin{equation*}
  \begin{tikzcd}
    0 \arrow{r} &
    \normalOB_{i-1} \arrow{r}{\iota}\arrow{d}{\tau_{0}} & %%%\arrow{r}{\iota}
    \normalOB_{i} \arrow{r}{\gamma_{i}-1}\arrow{d}{\tau_{1}} &
    \normalOB_{i} \arrow{r}\arrow[equal]{d} &
    0 \\
    0 \arrow{r} &
      \normalOB_{i-1}\times_{\normalOB_{i}}
      \left(
        \normalOB_{i} \times_{\gamma_{i}-1,\normalOB_{i}}\normalOB_{i}
      \right)
      \arrow{r}{\Phi_{0}} &
    \normalOB_{i} \times_{\gamma_{i}-1,\normalOB_{i}}\normalOB_{i}
      \arrow{r}{\Phi_{1}} &
    \normalOB_{i} \arrow{r} &
    0.
  \end{tikzcd}
  \end{equation*}
  Furthermore, we have the morphisms
  \begin{alignat*}{3}
    \tau_{0}\colon f
      &\mapsto \left(f , \left( \iota\left( f\right) , 0 \right)\right),& \quad
    \Phi_{0}\colon \left(f , g\right)
      &\mapsto g \\
    \tau_{1}\colon f
      &\mapsto \left(f , \left(\gamma_{i}-1\right)(f) \right),& \quad
    \Phi_{1}\colon \left(f , g\right)
      &\mapsto g.
  \end{alignat*}
  Both $\tau_{0}$ and $\tau_{1}$ are isomorphisms.
  This implies that~(\ref{eq:cont-group-coh-OBlatTTKd-the-sequence}) is, as a
  cochain complex, isomorphic to
  \begin{equation*}
    0 \longrightarrow
    \normalOB_{i-1}\times_{\normalOB_{i}}
    \left(
      \normalOB_{i} \times_{\gamma_{i}-1,\normalOB_{i}}\normalOB_{i}
    \right)
    \stackrel{\Phi_{0}}{\longrightarrow}
    \normalOB_{i} \times_{\gamma_{i}-1,\normalOB_{i}}\normalOB_{i}
      \stackrel{\Phi_{1}}{\longrightarrow}
    \normalOB_{i} \longrightarrow
    0.
  \end{equation*}  
  By Lemmata~\ref{lem:completed-localisation-preserves-fiberproducts}
  and~\ref{lem:normalOAila-and-normalOBila}, this complex is
  isomorphic to
  \begin{equation*}\label{eq:claim:cont-group-coh-OBlatTTKd-Dbullet-final-computation}
  \begin{split}
    \varinjlim_{q\in\NN_{\geq1}}\left(E^{>q,\bullet}
      \widehat{\otimes}_{A_{\dR}^{>1}}B_{\dR}^{>1}\right)
    &\cong
    \varinjlim_{q\in\NN_{\geq1}}\left(E^{>q,\bullet}
      \widehat{\otimes}_{A_{\dR}^{>1}}\text{``}\varinjlim_{t\times}\text{''}A_{\dR}^{>1}\right) \\
    &\stackrel{\text{\ref{lem:decalage-indBanach-localisation-is-just-completed-localisation}}}{\cong}
    \varinjlim_{q\in\NN_{\geq1}}\left(\eta_{\mu/p^{q}}E^{>q,\bullet}
      \widehat{\otimes}_{A_{\dR}^{>1}}\text{``}\varinjlim_{t\times}\text{''}A_{\dR}^{>1}\right)
    \cong    
    \varinjlim_{q\in\NN_{\geq1}}\varinjlim_{t\times}\eta_{\mu/p^{q}}E^{>q,\bullet}.
  \end{split}
  \end{equation*}
  Here, Lemma~\ref{lem:decalage-indBanach-localisation-is-just-completed-localisation}
  applies because
  Condition~\ref{cond:Banachmoduledecalage-reconstructionpaper}
  is satisfied for
  $R=A_{\dR}^{>1}$,
  $r=\mu/p^{q}$, and
  $M^{\bullet}=E^{>q,\bullet}$,
  cf. the proof of Lemma~\ref{lem:etamupq-decelagemakessense}.
  Now apply Corollary~\ref{cor:filteredcol-inIndBan-stronglyexact}
  to find that $\varinjlim_{q\in\NN_{\geq1}}\varinjlim_{t\times}\eta_{\mu/p^{q}}E^{>q,\bullet}$,
  that is~(\ref{eq:cont-group-coh-OBlatTTKd-the-sequence}), is strictly exact.
\end{proof}

%%%%%%%%%%%%%%%%%%%%%%%%%%%%%%%%%%%%%%%%%%%%%%%%%%%%%%%%%%%%
%%%%%%%%%%%%%%%%%%%%%%%%%%%%%%%%%%%%%%%%%%%%%%%%%%%%%%%%%%%%
% Cech cohomology I
%%%%%%%%%%%%%%%%%%%%%%%%%%%%%%%%%%%%%%%%%%%%%%%%%%%%%%%%%%%%
%%%%%%%%%%%%%%%%%%%%%%%%%%%%%%%%%%%%%%%%%%%%%%%%%%%%%%%%%%%%

\subsection{The filtered rings $\normalOtildeA_{i-1,i}^{>q-1,q}$,
$\normalOtildeA_{i-1}^{>q-1}$, and $\normalOtildeA_{i}^{>q-1}$}\label{subsubsec:initialcomputation-normalOtildeAi-1ietc}

Fix $q\in\NN_{\geq2}$.
In \S\ref{subsubsec:initialcomputation-normalOAi-1ietc},
we studied the Banach algebras $\normalOA_{i-1,i}^{>q-1,q}$,
$\normalOA_{i-1}^{>q-1}$, and $\normalOA_{i}^{>q-1}$.
In thus \S\ref{subsubsec:initialcomputation-normalOtildeAi-1ietc},
we equip these algebras with filtrations which do \emph{not}
recover the topologies induced by the norms. We will therefore denote
these filtered algebras by
$\normalOtildeA_{i-1,i}^{>q-1,q}$,
$\normalOtildeA_{i-1}^{>q-1}$, and $\normalOtildeA_{i}^{>q-1}$

Given a Banach algebra $R$, $|R|$ is its underlying abstract algebra.
$\widetilde{A}_{\dR}^{>1}$ is the filtered ring whose underlying abstract ring
is $|A_{\dR}^{>1}|$, equipped with the $\xi/p^{q}$-adic filtration.

\begin{defn}
  We define the following three filtered $\widetilde{A}_{\dR}^{>1}$-algebras:
\begin{itemize}
  \item[(i)] $\normalOtildeA_{i-1,i}^{>q-1,q}$ is $|\normalOA_{i-1,i}^{>q-1,q}|$,
  carrying the $\left(\xi/p^{q},U_{1} / p^{q-1},\dots,U_{i-1} / p^{q-1},U_{i} / p^{q}\right)$-adic filtration.
  \item[(ii)] $\normalOtildeA_{i-1}^{>q-1}$ is
  $|\normalOA_{i-1}^{>q-1}|$, carrying the
  $\left(\xi/p^{q},U_{1} / p^{q-1},\dots,U_{i-1} / p^{q-1} \right)$-adic filtration.
  \item[(iii)] $\normalOtildeA_{i}^{>q-1}$ is
  $|\normalOA_{i}^{>q-1}|$, carrying the
  $\left(\xi/p^{q},U_{1} / p^{q-1},\dots,U_{i} / p^{q-1}\right)$-adic filtration.  
  \end{itemize}
\end{defn}

%Recall Definition~\ref{defn:puttrivialnormonring-reconstructionpaper}.

\begin{lem}\label{lem:normalOtildeA-allofthem-Banachalgebra-recpaper}
  $\normalOtildeA_{i-1,i}^{>q-1,q}$, $\normalOtildeA_{i-1}^{>q-1}$,
  and $\normalOtildeA_{i}^{>q-1}$ are separated and complete.
\end{lem}

\begin{proof}
  $\normalOA_{i-1,i}^{>q-1,q}$ carries the
  $\left( p , \xi/p^{q} , U_{1} / p^{q-1} , \dots , U_{i-1} / p^{q-1} , U_{i} / p^{q} \right)$-adic norm
  by Lemma~\ref{lem:thenormonAhotimesB-Iadic-piadic-reconstructionpaper}.
  The separated completeness of $\normalOA_{i-1,i}^{>q-1,q}$
  and~\cite[\href{https://stacks.math.columbia.edu/tag/05GG}{Tag 05GG}]{stacks-project}
  imply the separated completeness of $\normalOtildeA_{i-1,i}^{>q-1,q}$.
  Proceed similarly for $\normalOtildeA_{i-1}^{>q-1}$
  and $\normalOtildeA_{i}^{>q-1}$.
\end{proof}

We continue by computing the associated gradeds of
$\normalOtildeA_{i-1,i}^{>q-1,q}$,
$\normalOtildeA_{i-1}^{>q-1}$, and $\normalOtildeA_{i}^{>q-1}$.

\begin{lem}\label{lem:xipq-regularelements-onnormalOtildeAi-1ietc-reconstructionpaper}
  The mulitplication-by-$\xi/p^{q}$ maps
  \begin{equation*}  
    \begin{split}
    \frac{\xi}{p^{q}}\colon
      \normalOtildeA_{i-1,i}^{>q-1,q} &\to \normalOtildeA_{i-1,i}^{>q-1,q} \\
    \frac{\xi}{p^{q}}\colon
      \normalOtildeA_{i-1}^{>q-1} &\to \normalOtildeA_{i-1}^{>q-1} \\
    \frac{\xi}{p^{q}}\colon
      \normalOtildeA_{i}^{>q-1} &\to \normalOtildeA_{i}^{>q-1}
    \end{split}
    \end{equation*}  
  are injective.
\end{lem}

\begin{proof}
  We only consider the first map
  $\normalOtildeA_{i-1,i}^{>q-1,q} \to \normalOtildeA_{i-1,i}^{>q-1,q}$
  and leave the other ones to the reader. Firstly,
  Lemma~\ref{lem:calOTTecirc-preservesstrictmono-reconstructionpaper}
  implies that it suffices to check that
  \begin{equation*}
    \frac{\xi}{p^{q}}\colon
    \A_{\dR}^{>q}\left(\widetilde{\TT}_{C}^{\left\{1,\dots,i\right\}}\times S\right)
      \left\llbracket
        \frac{U_{1},\dots,U_{i-1}}{p^{q-1}},\frac{U_{i}}{p^{q}}
      \right\rrbracket
    \to\A_{\dR}^{>q}\left(\widetilde{\TT}_{C}^{\left\{1,\dots,i\right\}}\times S\right)
      \left\llbracket
        \frac{U_{1},\dots,U_{i-1}}{p^{q-1}},\frac{U_{i}}{p^{q}}
      \right\rrbracket
  \end{equation*}  
  is injective. By Lemma~\ref{lem:OAiq-etc-formalpowerseriesdescription},
  it remains to check that
  \begin{equation*}
    \frac{\xi}{p^{q}}\colon
    \A_{\dR}^{>q}\left(\widetilde{\TT}_{C}^{\left\{1,\dots,i\right\}}\times S\right)
    \to\A_{\dR}^{>q}\left(\widetilde{\TT}_{C}^{\left\{1,\dots,i\right\}}\times S\right)
  \end{equation*}
  is injective. This follows from Lemma~\ref{lem:Fontaines-map-for-Alagreatq}.
\end{proof}

\begin{defn}\label{defn:defn-Rpluswidetildej}
  For all $j=0,\dots,d$,
  $R^{+,\widetilde{j}}
    :=
      \widehat{\cal{O}}^{+}\left(\widetilde{\TT}_{K}^{\left\{1,\dots,j\right\}}\times S\right)
      \widehat{\otimes}_{k^{\circ}}\cal{O}\left(\TT^{\left\{j+1,\dots,d\right\}}\right)^{\circ}$.
\end{defn}
    
  \begin{lem}\label{lem:grnormalOtildeAiq-etc}
   Compute
    \begin{equation*}  
    \begin{split}
    \gr\normalOtildeA_{i-1,i}^{>q-1,q}
      &\cong R^{+,\widetilde{i}}\left[
        \sigma\left(\frac{\xi}{p^{q}}\right),
        \sigma\left(\frac{U_{1}}{p^{q-1}}\right),
        \dots
        \sigma\left(\frac{U_{i-1}}{p^{q-1}}\right),
        \sigma\left(\frac{U_{i}}{p^{q}}\right)
        \right], \\
    \gr\normalOtildeA_{i-1}^{>q-1}
      &\cong R^{+,\widetilde{i-1}}\left[
        \sigma\left(\frac{\xi}{p^{q}}\right),
        \sigma\left(\frac{U_{1}}{p^{q-1}}\right),
        \dots
        \sigma\left(\frac{U_{i-1}}{p^{q-1}}\right)
        \right], \\
    \gr\normalOtildeA_{i}^{>q-1}
      &\cong R^{+,\widetilde{i}}\left[
        \sigma\left(\frac{\xi}{p^{q-1}}\right),
        \sigma\left(\frac{U_{1}}{p^{q-1}}\right),
        \dots
        \sigma\left(\frac{U_{i}}{p^{q-1}}\right),
      \right].
    \end{split}
    \end{equation*}  
    All the principal symbols $\sigma(-)$ are of degree one.
  \end{lem}
  
  \begin{proof}    
    \iffalse %%%
    $\widetilde{\TT}_{K}^{i}\times S$ is affinoid perfectoid because
    $\widetilde{\TT}_{K}^{i}$ is affinoid perfectoid, cf.~\cite[Corollary 6.6]{Sch13pAdicHodge}.
    By a theorem in my thesis,
    \begin{equation*}
      \A_{\dR}^{>q}\left(\widetilde{\TT}_{K}^{i}\times S\right)
      =\A_{\dR}^{>q}\left(
        R,R^{+}\right),
    \end{equation*}
    where
    \begin{equation*}
      \left(R,R^{+}\right)
      =\left(\widehat{\cal{O}}\left(\widetilde{\TT}_{K}^{i}\times S\right),
        \widehat{\cal{O}}^{+}\left(\widetilde{\TT}_{K}^{i}\times S\right)\right)
    \end{equation*}
    is an affinoid perfectoid $\left(K,K^{+}\right)$-algebra. Thus a result
    from my thesis applies, giving that
    $\xi/p^{q}\in\A_{\dR}^{>q}\left(\widetilde{\TT}_{K}^{i}\times S\right)$
    is not a zero-divisor.
    \fi %%% comment ends
    First, we consider $\gr\normalOA_{i-1,i}^{>q-1,q}$. We find that
    $U_{i}/p^{q},U_{i-1}/p^{q-1},\dots,U_{1}/p^{q-1},\xi/p^{q}$
    is a regular sequence in $\normalOA_{i-1}^{>q-1}$;
    this follows because the
    $U_{i}/p^{q},U_{i-1}/p^{q-1},\dots,U_{1}/p^{q-1}$    
    are formal elements and because of
    Lemma~\ref{lem:xipq-regularelements-onnormalOtildeAi-1ietc-reconstructionpaper}.
    Thus~\cite[Exercise 17.16]{Ei95} applies, so that it remains to compute
    \begin{equation*}
    \begin{split}
      \gr^{0}\normalOtildeA_{i-1,i}^{>q-1,q}
      &=\normalOtildeA_{i-1,i}^{>q-1,q} / 
        \left(\xi/p^{q}, U_{1}/p^{q-1}, \dots, U_{i-1}/p^{q-1}, U_{i}/p^{q}\right) \\
      &\cong\widehat{\cal{O}}^{+}\left(\widetilde{\TT}_{K}^{\left\{1,\dots,i\right\}}\times S\right)
         \widehat{\otimes}_{k^{\circ}}\cal{O}\left(\TT_{k}^{\left\{i+1,\dots,d\right\}}\right) 
       =R^{+,\widetilde{i}},
    \end{split}
    \end{equation*}
    as desired. Proceed similarly to compute
    $\gr\normalOtildeA_{i-1}^{>q-1}$ and $\gr\normalOtildeA_{i}^{>q-1}$.
  \end{proof}

\begin{lem}\label{lem:etamupq-decelagemakessense-tilde-checkconditionii}
  The following inclusions are strict monomorphisms of filtered rings:
  \begin{equation}\label{lem:etamupq-decelagemakessense-tilde-checkconditionii-theinclusions}
  \begin{split}
    \frac{t}{p^{q}}\normalOtildeA_{i-1,i}^{>q-1,q} &\hookrightarrow \normalOtildeA_{i-1,i}^{>q-1,q}, \\
    \frac{t}{p^{q}}\normalOtildeA_{i-1}^{>q-1} &\hookrightarrow \normalOtildeA_{i-1}^{>q-1}, \text{ and} \\
    \frac{t}{p^{q}}\normalOtildeA_{i}^{>q-1} &\hookrightarrow \normalOtildeA_{i}^{>q-1}.
  \end{split}
  \end{equation}
\end{lem}

\begin{proof}
  By~\cite[Chapter I, \S 4.2 page 31-32, Theorem 4(5)]{HuishiOystaeyen1996},
  which applies by Lemma~\ref{lem:normalOtildeA-allofthem-Banachalgebra-recpaper},
  the strict exactness of the first map would follow if
  \begin{equation*}
    \sigma\left(\frac{t}{p^{q}}\right)\times
    \colon\gr\normalOtildeA_{i-1,i}^{>q-1,q} \hookrightarrow\gr\normalOtildeA_{i-1,i}^{>q-1,q}
  \end{equation*}
  is injective. This follows directly from Lemma~\ref{lem:grnormalOtildeAiq-etc}
  and the description of $\sigma\left(t/p^{q}\right)$ as in Lemma~\ref{lem:principlesymbol-of-t}.
  Similarly, one checks that the second and third inclusion
  in~(\ref{lem:etamupq-decelagemakessense-tilde-checkconditionii-theinclusions})
  are strict monomorphisms.
\end{proof}

\begin{lem}\label{lem:cont-group-coh-OBlatTTKd-Dbullet-reducetotildeE-neededlemma}
  $\normalOtildeA_{i-1,i}^{>q-1,q}/\Fil^{s}$,
  $\normalOtildeA_{i-1}^{>q-1}/\Fil^{s}$, and $\normalOtildeA_{i}^{>q-1}/\Fil^{s}$
  have no $p$-power torsion for $s\in\NN$.
\end{lem}

\begin{proof}
  As the arguments for
  $\normalOtildeA_{i-1}^{>q-1}/\Fil^{s}$ and $\normalOtildeA_{i}^{>q-1}/\Fil^{s}$
  are similar, we only show here that $\normalOtildeA_{i-1,i}^{>q-1,q}/\Fil^{s}$
  has no $p$-power torsion.
  Lemma~\ref{lem:OAiq-etc-formalpowerseriesdescription}
  gives the isomorphism
  \begin{equation*}
    \normalOtildeA_{i-1,i}^{>q-1,q}/\Fil^{s}
    \cong
    \bigoplus\left(\A_{\dR}^{>q}\left(\widetilde{\TT}_{C}^{\left\{1,\dots,i\right\}}\times S\right) \middle/ \left(\frac{\xi}{p^{q}}\right)^{\alpha_{0}}\right)
      \zeta_{1}^{\alpha_{1}}\cdots\zeta_{i-1}^{\alpha_{i-1}}\eta^{\alpha_{i}},
  \end{equation*}
  where the direct sum runs over all
  $\alpha_{0},\dots,\alpha_{i}\in\NN$ such that $\sum_{j=0}^{i}\alpha_{j}\leq s$.
  Apply
  Lemma~\ref{lem:AdRgreaterthanqUtimesS-isomapHomcontSAdRgreaterthanqU-assumptionsforlemmasatisfied-reconstructionpaper}.
\end{proof}

%%%%%%%%%%%%%%%%%%%%%%%%%%%%%%%%%%%%%%%%%%%%%%%%%%%%%%%%%%%%
%%%%%%%%%%%%%%%%%%%%%%%%%%%%%%%%%%%%%%%%%%%%%%%%%%%%%%%%%%%%
% Cech cohomology I
%%%%%%%%%%%%%%%%%%%%%%%%%%%%%%%%%%%%%%%%%%%%%%%%%%%%%%%%%%%%
%%%%%%%%%%%%%%%%%%%%%%%%%%%%%%%%%%%%%%%%%%%%%%%%%%%%%%%%%%%%

\subsection{The cochain complexes $\widetilde{E}^{>q,\bullet}$}
\label{subsubsec:initialcomputation-thecomplextildeEgreathanqbullet}

Fix $q\in\NN_{\geq2}$ and recall the Definition~\ref{defn:Egreaterthanqbullet}
of the complex $E^{>q,\bullet}$. In the following, we introduce the
filtered complex
$\widetilde{E}^{>q,\bullet}$. It coincides with $E^{>q,\bullet}$
as a complex of abstract $A_{\dR}^{>1}$-modules, but the topologies induced
by the filtration and norms do not coincide. It is built out
of the filtered rings $\normalOtildeA_{i-1,i}^{>q-1,q}$,
$\normalOtildeA_{i-1}^{>q-1}$, and $\normalOtildeA_{i}^{>q-1}$,
which we introduced in the previous \S\ref{subsubsec:initialcomputation-normalOtildeAi-1ietc}.
It remains to construct the differentials of $\widetilde{E}^{>q,\bullet}$.

\begin{lem}\label{lem:tildeEgreaterthanqbullet-definetildegamma-recpaper}
  The morphism
  \begin{equation*}
    \widetilde{\gamma}_{i}\colon\normalOtildeA_{i-1,i}^{>q-1,q}\to\normalOtildeA_{i-1,i}^{>q-1,q},
    x\mapsto\widetilde{\gamma}_{i}(x):=\gamma_{i}(x),
  \end{equation*}
  where $\gamma_{i}$ is as in
  Definition~\ref{defn:gammai-on-overlinenormalOAigreaterthanq-reconstructionpaper},
  is a morphism of filtered $\widetilde{A}_{\dR}^{>1}$-algebras.
\end{lem}

\begin{proof}
  By Lemma~\ref{lem:bounded-map-adic-rings},
  we may check that the following inclusion holds:
  \begin{equation*}
    \widetilde{\gamma}_{i}\left( \left( \frac{\xi}{p^{q}}, \frac{U_{1}}{p^{q-1}} , \dots , \frac{U_{i-1}}{p^{q-1}} , \frac{U_{i}}{p^{q}} \right) \right)
    \subseteq\left( \frac{\xi}{p^{q}}, \frac{U_{1}}{p^{q-1}} , \dots , \frac{U_{i-1}}{p^{q-1}} , \frac{U_{i}}{p^{q}} \right).
  \end{equation*}
  Everything is clear, except
  \begin{equation*}
    \frac{U_{i}}{p^{q}} + \frac{t}{p^{q}}=\widetilde{\gamma}_{i}\left( \frac{U_{i}}{p^{q}} \right)
    \in\left( \frac{\xi}{p^{q}}, \frac{U_{1}}{p^{q-1}} , \dots , \frac{U_{i-1}}{p^{q-1}} , \frac{U_{i}}{p^{q}} \right).
  \end{equation*}
  But $t/p^{q}$ is in the kernel of Fontaine's $\theta_{\dR}^{q}\colon A_{\dR}^{q}\to \cal{O}_{C}$.
  Therefore, Lemma~\ref{lem:Fontaines-map-for-Ala} implies that
  it is divisible by $\xi/p^{q}$, as desired.
\end{proof}

Fix the notation as in Lemma~\ref{lem:tildeEgreaterthanqbullet-definetildegamma-recpaper}.
Here is the main definition of \S\ref{subsubsec:initialcomputation-thecomplextildeEgreathanqbullet}.

\begin{defn}\label{defn:tildeEgreaterthanqbullet}
  $\widetilde{E}^{>q,\bullet}$ denotes the sequence of maps
  of filtered $\widetilde{A}_{\dR}^{>1}$-modules
  \begin{multline*}%\label{eq:cont-group-coh-OBlatTTKd-the-sequence-toshowisexact}
%  \begin{split}
    0
    \longrightarrow
    \normalOtildeA_{i-1}^{>q-1}\times_{\normalOtildeA_{i-1,i}^{>q-1,q}}
    \left(
      \normalOtildeA_{i-1,i}^{>q-1,q} \times_{\widetilde{\gamma}_{i}-1,\normalOtildeA_{i-1,i}^{>q-1,q}}\normalOtildeA_{i}^{>q-1}
    \right) \\
    \stackrel{\widetilde{\phi}_{0}}{\longrightarrow}
    \normalOtildeA_{i-1,i}^{>q-1,q} \times_{\widetilde{\gamma}_{i}-1,\normalOtildeA_{i-1,i}^{>q-1,q}}\normalOtildeA_{i}^{>q-1}
    \stackrel{\widetilde{\phi}_{1}}{\longrightarrow}
    \normalOtildeA_{i}^{>q} \longrightarrow
    0.
%  \end{split}
  \end{multline*}
  Both $\widetilde{\phi}_{0}$ and $\widetilde{\phi}_{1}$ are given by $\left(f , g\right)\mapsto g$.
\end{defn}

\begin{lem}\label{lem:tildeEgreaterthanqbullet-cochaincomplex-reconstructionpaper}
  $\widetilde{E}^{>q,\bullet}$ is a cochain complex of filtered $\widetilde{A}_{\dR}^{>1}$-modules
  for all $q\in\NN_{\geq1}$.
\end{lem}

\begin{proof}
  $\widetilde{E}^{>q,\bullet}$ is a sequence of filtered $\widetilde{A}_{\dR}^{>1}$-modules
  by Lemma~\ref{lem:tildeEgreaterthanqbullet-definetildegamma-recpaper}.
  It is a cochain complex because its underlying complex
  of abstract $\widetilde{A}_{\dR}^{>1}$-modules is a cochain complex by
  Lemma~\ref{lem:Egreaterthanqbullet-cochaincomplex-reconstructionpaper}.
\end{proof}

We introduce Convention~\ref{conv:tildeEgreaterthanqbullet-concentratedindegree012}
to apply the $\eta$-operator to $\widetilde{E}^{>q,\bullet}$,
cf. Definition~\ref{defn:algebraic-decalage}.

\begin{conv}\label{conv:tildeEgreaterthanqbullet-concentratedindegree012}
  $\widetilde{E}^{>q,\bullet}$ is, as a cochain complex,
  concentrated in degrees $0$, $1$, and $2$.
\end{conv}

We recover $E^{>q,\bullet}$ from $\widetilde{E}^{>q,\bullet}$ as follows.

\begin{lem}\label{lem:poperator-to-widetildeEgreaterthanq-recopaper}
  Recall Definition~\ref{defn:poperator}
  and compute $E^{>q,\bullet}=\p\left(\widetilde{E}^{>q,\bullet}\right)$.
\end{lem}

\begin{proof}
  %$\A_{\dR}^{>q}\left(\widetilde{\TT}_{K}^{\left\{1,\dots,i\right\}}\times S\right)\left\llbracket
  %  \frac{U_{1}}{p^{q-1}},\dots,\frac{U_{i-1}}{p^{q-1}},\frac{U_{i}}{p^{q}}
  %  \right\rrbracket$ carries the $\left( p , \xi/p^{q} , U_{1} / p^{q-1} , \dots , U_{i-1} / p^{q-1} , U_{i} / p^{q} \right)$-adic norm.
  $\normalOA_{i-1,i}^{>q-1,q}$ carries the
  $\left( p , \xi/p^{q} , U_{1} / p^{q-1} , \dots , U_{i-1} / p^{q-1} , U_{i} / p^{q} \right)$-adic norm
  by Lemma~\ref{lem:thenormonAhotimesB-Iadic-piadic-reconstructionpaper}.
  Therefore, we get the identity $\normalOA_{i-1,i}^{>q-1,q} = \p\left(\normalOtildeA_{i-1,i}^{>q-1,q}\right)$
  of topological abelian groups immediately from Definition~\ref{defn:poperator}.
  Similarly, $\normalOA_{i-1}^{>q-1} = \p\left(\normalOtildeA_{i-1}^{>q-1}\right)$
  and $\normalOA_{i}^{>q-1} = \p\left(\normalOtildeA_{i}^{>q-1}\right)$.
  Finally, apply Lemma~\ref{lem:poperator-commutes-with-pullbacks}.
\end{proof}

As a consequence, we prove the following
Lemma~\ref{lem:cont-group-coh-OBlatTTKd-Dbullet-reducetotildeE}.
It reduces the proof of Proposition~\ref{prop:cont-group-coh-OBlatTTKd}
to a study of the cochain complexes $\eta_{\mu/p^{q}}\tilde{E}^{>q,\bullet}$.
Here, we follow the notation in \S\ref{subsubsec:initialcomputation-thecomplexEgreathanqbullet},
and denote the image of $\mu=[\epsilon]-1$
under the canonical map $A_{\inf}\to A_{\dR}^{>1}$ again by $\mu$.
  
\begin{lem}\label{lem:cont-group-coh-OBlatTTKd-Dbullet-reducetotildeE}
  Proposition~\ref{prop:cont-group-coh-OBlatTTKd}
  follows if the $\eta_{\mu/p^{q}}\widetilde{E}^{>q,\bullet}$ are strictly exact for all $q\in\NN_{\geq 2}$.
\end{lem}

\begin{proof}
  Let $q\in\NN_{\geq 2}$ be arbitrary.
  By Lemma~\ref{lem:cont-group-coh-OBlatTTKd-Dbullet-reducetotildeE-neededlemma},
  $\widetilde{E}^{>q,\bullet}/\Fil^{s}$ has no $p$-power torsion
  for all $s\in\NN$. Therefore, we can
  apply Lemma~\ref{lem:subsections-periodsheaves-affperfd-1}, giving
  the strict exactness of
  \begin{equation*}
    \p\left(\eta_{\mu/p^{q}}\widetilde{E}^{>q,\bullet}\right)=\eta_{\mu/p^{q}}E^{>q,\bullet};
  \end{equation*}
  this equality follows from Lemma~\ref{lem:poperator-to-widetildeEgreaterthanq-recopaper}
  and Definition~\ref{defn:decalage-for-Banachmodules}.
  
  As the choice of $q$ was arbitrary, we have verified that the complexes
  $\eta_{\mu/p^{q}}E^{>q,\bullet}$ are strictly exact for all $q\in\NN_{\geq2}$.
  Now apply Lemma~\ref{lem:cont-group-coh-OBlatTTKd-Dbullet}.
\end{proof}

%%%%%%%%%%%%%%%%%%%%%%%%%%%%%%%%%%%%%%%%%%%%%%%%%%%%%%%%%%%%
%%%%%%%%%%%%%%%%%%%%%%%%%%%%%%%%%%%%%%%%%%%%%%%%%%%%%%%%%%%%
% Cech cohomology I
%%%%%%%%%%%%%%%%%%%%%%%%%%%%%%%%%%%%%%%%%%%%%%%%%%%%%%%%%%%%
%%%%%%%%%%%%%%%%%%%%%%%%%%%%%%%%%%%%%%%%%%%%%%%%%%%%%%%%%%%%

\subsection{On the associated graded of $\widetilde{E}^{>q,\bullet}$}\label{subsubsec:initialcomputation-associatedgradedofEgreaterthanqbullet-reconstructionpaper}

Fix $q\in\NN_{\geq2}$. Motivated by Lemma~\ref{lem:cont-group-coh-OBlatTTKd-Dbullet-reducetotildeE},
we aim to check that $\eta_{\mu/p^{q}}\widetilde{E}^{>q,\bullet}$ is strictly exact. We proceed by consideration
of the associated graded.
   
\begin{lem}\label{lem:gretatEgreaterthanqbullet-isetaGbullet-conditionsatisfied}
  Condition~\ref{cond:filteredmodulesdecalage-reconstructionpaper}
  is satisfied for $R=\widetilde{A}_{\dR}^{>1}$, $r=\mu/p^{q}$, and
  $M^{\bullet}=\widetilde{E}^{>q,\bullet}$.
\end{lem}

\begin{proof}
  We freely use that $\mu/p^{q}$ and $t/p^{q}$ coincide up to a unit in $A_{\dR}^{>1}$,
      cf. the proof of Lemma~\ref{Bla-independent-of-epsilon}.
  \begin{itemize}
    \item[(i)] $\widetilde{E}^{>q,\bullet}$ is indeed concetrated in nonnegative
      degrees, cf. Convention~\ref{conv:tildeEgreaterthanqbullet-concentratedindegree012}.
    \item[(ii)]
      Lemma~\ref{lem:etamupq-decelagemakessense-tilde-checkconditionii} implies that
      the inclusions
      \begin{align*}
        \frac{\mu}{p^{q}}\normalOtildeA_{i-1,i}^{>q-1,q} &\hookrightarrow \normalOtildeA_{i-1,i}^{>q-1,q}, \\
        \frac{\mu}{p^{q}}\normalOtildeA_{i-1}^{>q-1} &\hookrightarrow \normalOtildeA_{i-1}^{>q-1}, \text{ and} \\
        \frac{\mu}{p^{q}}\normalOtildeA_{i}^{>q-1} &\hookrightarrow \normalOtildeA_{i}^{>q-1}.
      \end{align*}
      are strict monomorphisms. Because taking pullbacks preserves strict monomorphisms,
      cf.~\cite[Lemma 3.7]{BBKFrechetModulesDescent}, this gives (ii).
    \item[(iii)] This is Lemma~\ref{lem:etamupq-decelagemakessense-checkconditioniii}.
    \item[(iv)] There is nothing to check for $\gr\widetilde{E}^{>q,0}$. Next, we compute
      \begin{equation*}
        \gr\widetilde{E}^{>q,1}
        \stackrel{\text{\ref{lem:gr:commutes-pullback}}}{=}
        \gr\normalOtildeA_{i-1,i}^{>q-1,q}
          \times_{\gr\left(\gamma_{i}-1\right),\gr\normalOtildeA_{i-1,i}^{>q-1,q}}
          \gr\normalOtildeA_{i}^{>q-1}.
      \end{equation*}
      From Lemma~\ref{lem:grnormalOtildeAiq-etc}
      and the description of the $\sigma\left(t/p^{q}\right)$ as in Lemma~\ref{lem:principlesymbol-of-t},
      we find that $\gr\widetilde{E}^{>q,1}$ is
      $\sigma\left(\mu/p^{q}\right)=\sigma\left(t/p^{q}\right)$-torsion free.
      Proceed similarly for
      $\gr\widetilde{E}^{>q,2}$.
  \end{itemize}
  We have thus checked Condition~\ref{cond:filteredmodulesdecalage-reconstructionpaper}.
\end{proof}
    
\begin{notation}
  $G^{\bullet}:=\gr\widetilde{E}^{>q,\bullet}$.
\end{notation}
  
\begin{cor}\label{cor:gretatEgreaterthanqbullet-isetaGbullet}
  We have the canonical isomorphism
  $\gr\eta_{\mu/p^{q}}\widetilde{E}^{>q,\bullet}\cong\eta_{\sigma(\mu/p^{q})}G^{\bullet}$.
\end{cor}
  
\begin{proof}
  %\begin{equation*}
    $\gr\eta_{\mu/p^{q}}\widetilde{E}^{>q,\bullet}
    \stackrel{\text{\ref{lem:decalage-commutes-gr}}}{\cong}
    \eta_{\sigma(\mu/p^{q})}\gr\widetilde{E}^{>q,\bullet}
    =\eta_{\sigma(\mu/p^{q})}G^{\bullet}$;
  %\end{equation*}
  Lemma~\ref{lem:decalage-commutes-gr} applies
  by Lemma~\ref{lem:gretatEgreaterthanqbullet-isetaGbullet-conditionsatisfied}.
\end{proof}
  
To show that $\eta_{\mu/p^{q}}\widetilde{E}^{>q,\bullet}$
is strictly exact, we may check that its associated graded is exact.
By Corollary~\ref{cor:gretatEgreaterthanqbullet-isetaGbullet},
this reduces the computation to showing that $G^{\bullet}$
is strictly exact up to $\sigma(\mu/p^{q})$-torsion. Therefore, we
give the following explicit description of $G^{\bullet}$.
Recall Definition~\ref{defn:defn-Rpluswidetildej}.

\begin{lem}\label{lem:initialcomputation-describe-Gbullet}
  $G^{\bullet}$ is isomorphic to the cochain complex
  \begin{equation}\label{eqlarge:initialcomputation-describe-Gbullet}
  \begin{tikzcd}
      0
      \arrow{d} \\
      \begin{multlined}
        R^{+,\widetilde{i-1}}\left[x,py_{1}.\dots,py_{i-1}\right] \\
        \times_{R^{+,\widetilde{i}}\left[x,py_{1}.\dots,py_{i-1},y_{i}\right]}
        \left(
          \begin{array}{l}
          R^{+,\widetilde{i}}\left[x,py_{1},\dots,py_{i-1},y_{i}\right] \hfill \\
          \quad\times_{\gamma_{i}-1,R^{+,\widetilde{i}}\left[x,py_{1},\dots,py_{i-1},y_{i}\right]}
          R^{+,\widetilde{i}}\left[px,py_{1},\dots,py_{i}\right]
          \end{array}
          \right)
       \end{multlined}
    \arrow{d}{\varphi_{0}} \\
      R^{+,\widetilde{i}}\left[x,py_{1},\dots,py_{i-1},y_{i}\right]
        \times_{\gamma_{i}-1,R^{+,\widetilde{i}}\left[x,py_{1},\dots,py_{i-1},y_{i}\right]}
      R^{+,\widetilde{i}}\left[px,py_{1},\dots,py_{i}\right]
    \arrow{d}{\varphi_{1}} \\
      R^{+,\widetilde{i}}\left[x,py_{1}.\dots,py_{i}\right]
    \arrow{d} \\
      0.
  \end{tikzcd}
  \end{equation}  
  Both $\varphi_{0}$ and
  $\varphi_{1}$ are given by $(f,g) \mapsto g$.
  Furthermore,
    \begin{itemize}
      \item[(i)] the action of $\gamma_{i}$ on $R^{+,\widetilde{i}}$
        is given as in Lemma~\ref{lem:ZpGaloiscoverings}.
        It is thus the identity on
        $R^{+,\widetilde{i-1}}\subseteq R^{+,\widetilde{i}}$.
      \item[(ii)] $\gamma_{i}$ fixes $x,y_{1},\dots,y_{i-1}$, and
      \item[(iii)] $\gamma_{i}\left(y_{i}\right)=y_{i}+\left(\zeta_{p}-1\right)x$.
    \end{itemize}
    Furthermore, with respect to this description of $G^{\bullet}$,
    $\eta_{\sigma(\mu/p^{q})}G^{\bullet}=\eta_{\left(\zeta_{p}-1\right)x}G^{\bullet}$.
\end{lem}
  
\begin{proof}
  Apply Lemma~\ref{lem:gr:commutes-pullback} to see that
  $G^{\bullet}:=\gr\widetilde{E}^{>q,\bullet}$ is isomorphic to
  \begin{multline*}
      0
    \longrightarrow
      \gr\normalOA_{i-1}^{>q-1}\times_{\gr\normalOA_{i-1,i}^{>q-1,q}}
      \left(
        \gr\normalOA_{i-1,i}^{>q-1,q} \times_{\gr\left(\gamma_{i}-1\right),\gr\normalOA_{i-1,i}^{>q-1,q}}
        \gr\normalOA_{i}^{>q-1}
      \right) \\
    \stackrel{\gr\phi_{0}}{\longrightarrow}
      \gr\normalOA_{i-1,i}^{>q-1,q} \times_{\gr\left(\gamma_{i}-1\right),\gr\normalOA_{i-1,i}^{>q-1,q}}
      \gr\normalOA_{i}^{>q-1}
    \stackrel{\gr\phi_{1}}{\longrightarrow}
      \gr\normalOA_{i}^{>q}
    \longrightarrow
      0.
  \end{multline*}
  Using Lemma~\ref{lem:grnormalOtildeAiq-etc},
  one sees that this cochain complex coincides with
  (\ref{eqlarge:initialcomputation-describe-Gbullet}).
  Here we identify
  $x$ with $\sigma\left(\xi / p^{q}\right)$ and
  $y_{l}$ with $\sigma\left(U_{l} / p^{q}\right)$ for all $l=1,\dots,i$.
  Also, $\varphi_{0}:=\gr\phi_{0}$, $\varphi_{1}:=\gr\phi_{1}$,
  and, by abuse of notation, $\gamma_{i}:=\gr\left(\gamma_{i}\right)$.
  The description of $\gamma_{i}$ as in (i),(ii), and (iii) comes directly from
  the Definition~\ref{defn:gammai-on-overlinenormalOAigreaterthanq-reconstructionpaper}
  and Lemma~\ref{lem:principlesymbol-of-t}.
  
  Finally, $\eta_{\sigma(\mu/p^{q})}G^{\bullet}=\eta_{\left(\zeta_{p}-1\right)x}G^{\bullet}$
  follows because $\mu/p^{q}$ and $t/p^{q}$ coincide up to a unit $v\in A_{\dR}^{1}$,
  cf. the proof of Lemma~\ref{Bla-independent-of-epsilon}. Indeed, \emph{loc. cit.} also shows that
  $v$ is congruent to $1$ modulo $\xi/p^{q}$, which gives the identity at the left-hand side here:
  \begin{equation*}
    \sigma\left( \frac{\mu}{p^{q}} \right)
    =\sigma\left( \frac{t}{p^{q}} \right)
    =\left(\zeta_{p}-1\right)\sigma\left( \frac{\xi}{p^{q}} \right)
    =\left(\zeta_{p}-1\right)x
  \end{equation*}
  The second equality in this computation comes from Lemma~\ref{lem:principlesymbol-of-t}.
\end{proof}

%%% COMMENT ENDS

Next, we would like to relate $G^{\bullet}$ to the following sequence of maps which we denote by
$H^{\bullet}$:
  \begin{equation*}
  \begin{tikzcd}
      0
      \arrow{d} \\
      \begin{multlined}
        R^{+,\widetilde{i-1}}\left[x,py_{1}.\dots,py_{i-1}\right] \\
        \times_{R^{+,\widetilde{i-1}}\left[x,py_{1}.\dots,py_{i-1},y_{i}\right]}
        \left(\begin{array}{l}
          R^{+,\widetilde{i-1}}\left[x,py_{1},\dots,py_{i-1},y_{i}\right] \\
          \quad\times_{\gamma_{i}-1,R^{+,\widetilde{i-1}}\left[x,py_{1},\dots,py_{i-1},y_{i}\right]}
        R^{+,\widetilde{i-1}}\left[px,py_{1},\dots,py_{i}\right]
        \end{array}
        \right)
      \end{multlined}
    \arrow{d}{\varphi_{0}} \\
      R^{+,\widetilde{i-1}}\left[x,py_{1},\dots,py_{i-1},y_{i}\right]
        \times_{\gamma_{i}-1,R^{+,\widetilde{i-1}}\left[x,py_{1},\dots,py_{i-1},y_{i}\right]}
      R^{+,\widetilde{i-1}}\left[px,py_{1},\dots,py_{i}\right]
    \arrow{d}{\varphi_{1}} \\
      R^{+,\widetilde{i-1}}\left[x,py_{1}.\dots,py_{i}\right]
    \arrow{d} \\
      0.
  \end{tikzcd}
\end{equation*}
Here, both $\varphi_{0}$ and
$\varphi_{1}$ are given by $(f,g) \mapsto g$. Also,
$\gamma_{i}$ is defined by (i),(ii), and (iii) as in Lemma~\ref{lem:initialcomputation-describe-Gbullet}.

Heuristically, $H^{\bullet}$ is $G^{\bullet}$ but without the $p$th power roots
of $T_{i}^{\pm}$; that is, we swapped all the $R^{+,\widetilde{i}}$ to $R^{+,\widetilde{i-1}}$.
Therefore, $H^{\bullet}\subseteq G^{\bullet}$. This inclusion gives that $H^{\bullet}$ is a cochain complex
as well. Furthermore, this inclusion induces the morphism in Lemma~\ref{lem:etaC-qiso-etatC} below.
  
\begin{conv}\label{conv:GbulletHbullet-degrees-012}
  View
  $H^{\bullet}$ and $G^{\bullet}$ as cochain complexes concentrated in degrees $0$, $1$, and $2$.
  This allows to apply the $\eta$-operator as in Definition~\ref{defn:algebraic-decalage}.
\end{conv}
  
  \begin{lem}\label{lem:etaC-qiso-etatC}
    The canonical map
    $\eta_{\left(\zeta_{p}-1\right)x}H^{\bullet}
    \to\eta_{\left(\zeta_{p}-1\right)x}G^{\bullet}$
    is a quasi-isomorphism.
  \end{lem}

  \begin{proof}
    We introduce the following descending filtrations: For all $s\in\ZZ$,
    \begin{align*}
      \Fil^{s}R^{+,\widetilde{i-1}}\left[x,py_{1},\dots,py_{i-1}\right]
      &:= 
      \begin{cases}
        R^{+,\widetilde{i-1}}\left[x,py_{1},\dots,py_{i-1}\right], & \text{for all $s\leq 0$,} \\
        0, & \text{for all $s>0$,}
      \end{cases}
      \\
      \Fil^{s}R^{+,\widetilde{i}}\left[x,py_{1},\dots,py_{i-1},y_{i}\right]
      &:=
      \left\{f\colon
        \begin{multlined}
          \text{$\deg(f)\leq -s$, where we view} \\
          \text{$f$ as a polynomial in $y_{i}$}
        \end{multlined}\right\}, \text{ and}
        \\
      \Fil^{s}R^{+,\widetilde{i}}\left[x,py_{1},\dots,py_{i}\right]
      &:=
      \left\{f\colon
          \begin{multlined}
            \text{$\deg(f)\leq -s$, where we view} \\
            \text{$f$ as a polynomial in $py_{i}$}
        \end{multlined}\right\}.
    \end{align*}
    This gives rise to filtrations on both
    $G^{\bullet}$ and $H^{\bullet}$.
    These induce filtrations on the subcomplexes
    $\eta_{\left(\zeta_{p}-1\right)x}G^{\bullet}\subseteq G^{\bullet}$
    and
    $\eta_{\left(\zeta_{p}-1\right)x}H^{\bullet}\subseteq H^{\bullet}$.
    As these filtrations induce the discrete topology,
    one checks directly that Condition~\ref{cond:filteredmodulesdecalage-reconstructionpaper}
    is satisfied in the following two instances:
    \begin{itemize}
      \item $R=k^{\circ}[x]$, equipped with the trivial filtration, $r=\left(\zeta_{p}-1\right)x$, and $M^{\bullet}=G^{\bullet}$.
      \item $R=k^{\circ}[x]$, equipped with the trivial filtration, $r=\left(\zeta_{p}-1\right)x$, and $M^{\bullet}=H^{\bullet}$.    
    \end{itemize}
    Therefore, Lemma~\ref{lem:decalage-commutes-gr} applies. This gives
    \begin{equation}\label{eq:graded-by-filtrations-by-degree-commutes-with-eta}
    \begin{split}
      \gr\eta_{\left(\zeta_{p}-1\right)x}G^{\bullet}
      &\cong
        \eta_{\left(\zeta_{p}-1\right)x}\gr G^{\bullet},\text{ and} \\
      \gr\eta_{\left(\zeta_{p}-1\right)x}H^{\bullet}
      &\cong
        \eta_{\left(\zeta_{p}-1\right)x}\gr H^{\bullet}.
    \end{split}
    \end{equation}
    \iffalse %%% I DON'T THINK I REALLY NEED THESE DETAILS ANMORE
    Here, Lemma~\ref{lem:decalage-commutes-gr} applies because
    the conditions (i), (ii), and (iii) \emph{loc. cit.} are satisfied.
    To see this, let $S\in\left\{R^{+,\widetilde{i-1}}\left[x,py_{1},\dots,py_{i-1}\right],
      R^{+,\widetilde{i}}\left[x,py_{1},\dots,py_{i-1},y_{i}\right],
      R^{+,\widetilde{i}}\left[x,py_{1},\dots,py_{i}\right]\right\}$
      and observe:
    \begin{itemize}
      \item[(i)] $\left(\zeta_{p}-1\right)^{l}x^{l}\cdot S\subseteq S$ is a closed subset
        for all $l\in\NN$. Indeed, the topology induced by the filtration on $S$ is the discrete topology.
      \item[(ii)] $S$ does not have $\left(\zeta_{p}-1\right)x$-torsion, and
      \item[(iii)] $\gr S$ does not have $\sigma\left(\left(\zeta_{p}-1\right)x\right)$-torsion; here,
        $\sigma\left(\left(\zeta_{p}-1\right)x\right)\in\gr R$ denotes the principal symbol of $\left(\zeta_{p}-1\right)x$.
  \end{itemize}
    \fi %%% comment ends
    Recall that we aim to show that
    $\eta_{\left(\zeta_{p}-1\right)x}H^{\bullet}
    \to\eta_{\left(\zeta_{p}-1\right)x}G^{\bullet}$
    is a quasi-isomorphism. This morphism is filtered
    because $H^{\bullet}\to G^{\bullet}$ is filtered.
    As the filtrations are bounded below
    and exhaustive, \cite[Comparison Theorem 5.2.12]{Wei94}
    applies, giving that it suffices to check that
    $\gr\eta_{\left(\zeta_{p}-1\right)x}H^{\bullet}
      \to\gr\eta_{\left(\zeta_{p}-1\right)x}G^{\bullet}$
    is a quasi-isomorphism. By~(\ref{eq:graded-by-filtrations-by-degree-commutes-with-eta}),
    we may check that 
    $\eta_{\left(\zeta_{p}-1\right)x}\gr H^{\bullet}
      \to\eta_{\left(\zeta_{p}-1\right)x}\gr G^{\bullet}$
    is a quasi-isomorphism. This is what we will do, via an explicit computation.
    Firstly, $\gr H^{\bullet}$ is
    \begin{equation*}
    \begin{tikzcd}
      0
      \arrow{d} \\
      \begin{multlined}
        R^{+,\widetilde{i-1}}\left[x,py_{1}.\dots,py_{i-1}\right] \\
        \times_{R^{+,\widetilde{i-1}}\left[x,py_{1}.\dots,py_{i-1}\right]\left[\sigma\left(y_{i}\right)\right]}
        \left(\begin{array}{l}
          R^{+,\widetilde{i-1}}\left[x,py_{1},\dots,py_{i-1}\right]\left[\sigma\left(y_{i}\right)\right] \\
          \quad
          \times_{\gr\left(\gamma_{i}\right)-1,R^{+,\widetilde{i-1}}\left[x,py_{1},\dots,py_{i-1}\right]\left[\sigma\left(y_{i}\right)\right]}
          R^{+,\widetilde{i-1}}\left[px,py_{1},\dots,py_{i-1}\right]\left[\sigma\left(y_{i}\right)\right]
          \end{array}
        \right)
      \end{multlined}
    \arrow{d}{d_{0}} \\
        \left(
          R^{+,\widetilde{i-1}}\left[x,py_{1},\dots,py_{i-1}\right]\left[\sigma\left(y_{i}\right)\right]
          \times_{\gr\left(\gamma_{i}\right)-1,R^{+,\widetilde{i-1}}\left[x,py_{1},\dots,py_{i-1}\right]\left[\sigma\left(y_{i}\right)\right]}
        R^{+,\widetilde{i-1}}\left[px,py_{1},\dots,py_{i-1}\right]\left[\sigma\left(y_{i}\right)\right]
        \right)
    \arrow{d}{d_{1}} \\
      R^{+,\widetilde{i-1}}\left[px,py_{1},\dots,py_{i-1}\right]\left[\sigma\left(y_{i}\right)\right]
    \arrow{d} \\
      0.
  \end{tikzcd}
  \end{equation*}
  Here, $\sigma\left(y_{i}\right)$ denotes the principal symbol of $y_{i}$.
  Both $d_{0}:=\gr\varphi_{0}$ and
  $d_{1}:=\gr\varphi_{0}$ are given by
  %\begin{equation*}
    $(f,g) \mapsto g$.
  %\end{equation*} 
  From the description in Lemma~\ref{lem:initialcomputation-describe-Gbullet},
  we find that
  \begin{itemize}
    \item[(i)] $\gr\left(\gamma_{i}\right)\colon R^{+,\widetilde{i}}\to R^{+,\widetilde{i}}$
      is given by $\gamma_{i}$ as in Lemma~\ref{lem:ZpGaloiscoverings}.
      It is thus the identity on $R^{+,\widetilde{i-1}}\subseteq R^{+,\widetilde{i}}$.
    \item[(ii)] $\gr\left(\gamma_{i}\right)$ fixes $x,y_{1},\dots,y_{i-1}$.
    \item[(iii)] $\gr\left(\gamma_{i}\right)$ fixes $\sigma\left(y_{i}\right)$ as well.
  \end{itemize}  
  Similarly, $\gr G^{\bullet}$ is
    \begin{equation*}
    \begin{tikzcd}
      0
      \arrow{d} \\
      \begin{multlined}
        R^{+,\widetilde{i-1}}\left[x,py_{1}.\dots,py_{i-1}\right] \\
        \times_{R^{+,\widetilde{i-1}}\left[x,py_{1}.\dots,py_{i-1}\right]\left[\sigma\left(y_{i}\right)\right]}
        \left(\begin{array}{l}
          R^{+,\widetilde{i}}\left[x,py_{1},\dots,py_{i-1}\right]\left[\sigma\left(y_{i}\right)\right] \\
          \quad
          \times_{\gr\left(\gamma_{i}\right)-1,R^{+,\widetilde{i}}\left[x,py_{1},\dots,py_{i-1}\right]\left[\sigma\left(y_{i}\right)\right]}
          R^{+,\widetilde{i}}\left[px,py_{1},\dots,py_{i-1}\right]\left[\sigma\left(y_{i}\right)\right]
          \end{array}
        \right)
      \end{multlined}
    \arrow{d}{d_{0}} \\
        \left(
          R^{+,\widetilde{i}}\left[x,py_{1},\dots,py_{i-1}\right]\left[\sigma\left(y_{i}\right)\right]
          \times_{\gr\left(\gamma_{i}\right)-1,R^{+,\widetilde{i}}\left[x,py_{1},\dots,py_{i-1}\right]\left[\sigma\left(y_{i}\right)\right]}
        R^{+,\widetilde{i}}\left[px,py_{1},\dots,py_{i-1}\right]\left[\sigma\left(y_{i}\right)\right]
        \right)
    \arrow{d}{d_{1}} \\
      R^{+,\widetilde{i}}\left[px,py_{1},\dots,py_{i-1}\right]\left[\sigma\left(y_{i}\right)\right]
    \arrow{d} \\
      0
  \end{tikzcd}
  \end{equation*}
  where $d_{0}$ and
  $d_{1}$ are $(f,g) \mapsto g$.
  and $\gr\left(\gamma_{i}\right)$ is given as before.

  In the following, we denote $\gr\gamma_{i}$ by $\gamma_{i}$.
  This is not an abuse of notation, by the explicit description of $\gr\gamma_{i}$ given above.
  Following Convention~\ref{conv:GbulletHbullet-degrees-012},
  we now compute cohomology:
  \begin{equation}\label{eq:ho-grHbullet}
  \begin{split}
    \Ho^{0}\left(\gr H^{\bullet}\right)
    &=0, \\
    \Ho^{1}\left(\gr H^{\bullet}\right)
    &=
    \frac{\ker\left( %%% frac beginns
        R^{+,\widetilde{i-1}}\left[x,py_{1},\dots,py_{i-1}\right]\left[\sigma\left(y_{i}\right)\right]
        \stackrel{\gamma_{i}-1}{\longrightarrow}
          R^{+,\widetilde{i-1}}\left[x,py_{1},\dots,py_{i-1}\right]\left[\sigma\left(y_{i}\right)\right]\right)}
      {R^{+,\widetilde{i-1}}\left[px,py_{1},\dots,py_{i-1}\right]} \\ %%% frac ends
    &=
    \frac{ %%% frac beginns
          \ker\left( R^{+,\widetilde{i-1}}\stackrel{\gamma_{i}-1}{\longrightarrow}
            R^{+,\widetilde{i-1}}\right)\left[x,py_{1},\dots,py_{i-1}\right]\left[\sigma\left(y_{i}\right)\right]}
      {R^{+,\widetilde{i-1}}\left[px,py_{1},\dots,py_{i-1}\right]} \\ %%% frac ends
    &=
    \frac{ %%% frac beginns
          \Ho_{\cont}^{0}\left( \ZZ_{p}\gamma_{i} , R^{+,\widetilde{i-1}}\right)
          \left[x,py_{1},\dots,py_{i-1}\right]\left[\sigma\left(y_{i}\right)\right]}
      {R^{+,\widetilde{i-1}}\left[px,py_{1},\dots,py_{i-1}\right]}, \\ %%% frac ends
    \Ho^{2}\left(\gr H^{\bullet}\right)
    &=
    \frac{R^{+,\widetilde{i-1}}\left[px,py_{1},\dots,py_{i-1}\right]\left[p\sigma\left(y_{i}\right)\right]}{
      \left(
      \begin{multlined}
        \im\left(
        R^{+,\widetilde{i-1}}\left[x,py_{1},\dots,py_{i-1}\right]\left[\sigma\left(y_{i}\right)\right]
        \stackrel{\gamma_{i}-1}{\longrightarrow}
        R^{+,\widetilde{i-1}}\left[x,py_{1},\dots,py_{i-1}\right]\left[\sigma\left(y_{i}\right)\right]\right) \\
        \cap R^{+,\widetilde{i-1}}\left[px,py_{1},\dots,py_{i-1}\right]\left[p\sigma\left(y_{i}\right)\right]
      \end{multlined}
      \right)} \\
    &=\Ho_{\cont}^{1}\left( \ZZ_{p}\gamma_{i} , R^{+,\widetilde{i-1}}\right)
      \left[px,py_{1},\dots,py_{i-1}\right]\left[p\sigma\left(y_{i}\right)\right].
  \end{split}
  \end{equation}
  Similarly,
  \begin{equation}\label{eq:ho-grGbullet}
  \begin{split}
    \Ho^{0}\left(\gr G^{\bullet}\right)
    &=0, \\
    \Ho^{1}\left(\gr G^{\bullet}\right)
    &=
    \frac{ %%% frac beginns
          \Ho_{\cont}^{0}\left( \ZZ_{p}\gamma_{i} , R^{+,\widetilde{i}}\right)
          \left[x,py_{1},\dots,py_{i-1}\right]\left[\sigma\left(y_{i}\right)\right]}
      {R^{+,\widetilde{i-1}}\left[px,py_{1},\dots,py_{i-1}\right]}, \\ %%% frac ends
    \Ho^{2}\left(\gr G^{\bullet}\right)
    &=\Ho_{\cont}^{1}\left( \ZZ_{p}\gamma_{i} , R^{+,\widetilde{i}}\right)
      \left[px,py_{1},\dots,py_{i-1}\right]\left[p\sigma\left(y_{i}\right)\right].
  \end{split}
  \end{equation}
  Recall that we wanted to check that
  $\eta_{\left(\zeta_{p}-1\right)x}\gr H^{\bullet}
      \to\eta_{\left(\zeta_{p}-1\right)x}\gr G^{\bullet}$
  is a quasi-isomorphism. By Lemma~\ref{lem:eta-operator-kills-torsion},
  we have to check that the induced maps
  \begin{equation*}
  \begin{split}
    \tau_{0}\colon\Ho^{0}\left(\gr H^{\bullet}\right)
    &\longrightarrow
    \Ho^{0}\left(\gr G^{\bullet}\right) \\
    \tau_{1}\colon
    \Ho^{1}\left(\gr H^{\bullet}\right)
    &\longrightarrow
    \Ho^{1}\left(\gr G^{\bullet}\right) \\
    \tau_{2}\colon
    \Ho^{2}\left(\gr H^{\bullet}\right)
    &\longrightarrow
    \Ho^{2}\left(\gr G^{\bullet}\right)  
  \end{split}
  \end{equation*}
  are isomorphisms up to $\left(\zeta_{p}-1\right)x$-torsion.
  From the computations~(\ref{eq:ho-grHbullet}) and~(\ref{eq:ho-grGbullet}) above,
  we see that this boils down to checking the following two facts:
  \begin{itemize}
    \item[(i)] $\Ho_{\cont}^{0}\left( \ZZ_{p}\gamma_{i} , R^{+,\widetilde{i-1}}\right)
    \isomap\Ho_{\cont}^{0}\left( \ZZ_{p}\gamma_{i} , R^{+,\widetilde{i}}\right)$
      is injective with cokernel killed $\zeta_{p}-1$.
    \item[(ii)] $\Ho_{\cont}^{1}\left( \ZZ_{p}\gamma_{i} , R^{+,\widetilde{i-1}}\right)
    \to\Ho_{\cont}^{1}\left( \ZZ_{p}\gamma_{i} , R^{+,\widetilde{i}}\right)$
    is injective with cokernel killed by $\zeta_{p}-1$. Indeed,
    Lemma~\ref{lem:iso-after-modulo-torsion} then implies that $\tau_{1}$ is an isomorphism
    up to $\left(\zeta_{p}-1\right)x$-torsion.
  \end{itemize}
  Both (i) and (ii) follow from Lemma~\ref{lem:Rt-S-ZZp-cohomology-i} and~\ref{lem:Rt-S-ZZp-cohomology-ii},
  as the $\ZZ_{p}\gamma_{i}$-action on
  $R^{+,\widetilde{i-1}}$ is trivial.
  \end{proof}

%%%%%%%%%%%%%%%%%%%%%%%%%%%%%%%%%%%%%%%%%%%%%%%%%%%%%%%%%%%%
%%%%%%%%%%%%%%%%%%%%%%%%%%%%%%%%%%%%%%%%%%%%%%%%%%%%%%%%%%%%
% A curious complex
%%%%%%%%%%%%%%%%%%%%%%%%%%%%%%%%%%%%%%%%%%%%%%%%%%%%%%%%%%%%
%%%%%%%%%%%%%%%%%%%%%%%%%%%%%%%%%%%%%%%%%%%%%%%%%%%%%%%%%%%%

\subsection{On the cohomology of $H^{\bullet}$}\label{subsubsec:initialcomputation-Ho-of-Hbullet}

Fix $q\in\NN_{\geq2}$.
Recall that, because of Lemma~\ref{lem:cont-group-coh-OBlatTTKd-Dbullet-reducetotildeE},
we aim to check that $\eta_{\mu/p^{q}}\widetilde{E}^{>q,\bullet}$ is strictly exact.
In \S\ref{subsubsec:initialcomputation-associatedgradedofEgreaterthanqbullet-reconstructionpaper},
we reduced this aim to showing that $\eta_{\mu/p^{q}}G^{\bullet}$ is exact.
By Lemma~\ref{lem:etaC-qiso-etatC}, we may as well check consider the cohomology of
$H^{\bullet}$. This is what we do in this \S\ref{subsubsec:initialcomputation-Ho-of-Hbullet}.

Write
$S := R^{+,\widetilde{i-1}}\left[py_{1},\dots,py_{i-1}\right]$,
$z := y_{i}$, $u := \zeta_{p}-1 \in S$, and
$\delta := \gamma_{i}$.
Recall that $\delta$ acts as the identity on $S[x]$ and $\delta(z)=ux+z$.
  With this notation, $H^{\bullet}$ is the cochain complex
  \begin{multline}\label{eq:cont-group-coh-OBlatTTKd-finalversionofCbullet}
    0\to
    S[x] \times_{S[x,z]}\left( S[x,z]\times_{\delta-1,S[x,z]}S[px,pz] \right) \\
    \stackrel{d^{0}}{\to}
    \left( S[x,z]\times_{\delta-1,S[x,z]}S[px,pz] \right)
    \stackrel{d^{1}}{\to}
    S[px,pz]
    \to
    0,
  \end{multline}
  concentrated in degree $0,1$, and $2$,
  cf. Convention~\ref{conv:GbulletHbullet-degrees-012}.
  Here, $d^{0}$ and $d^{1}$ are both given by $(f,g)\mapsto g$.

  \begin{lem}\label{prop:curiouscomplex-defn-Cbullet-etacomplex-exact-0}
    $\Ho^{0}\left(H^{\bullet}\right)=0$.
  \end{lem}

  \begin{proof}
    Given $(f,g)\in\ker d^{0}$, we find $g=0$. This forces $f=0$ as well.
  \end{proof}

  \begin{lem}\label{prop:curiouscomplex-defn-Cbullet-etacomplex-exact-1}
    The canonical inclusion
      $S[x] \isomap \ker\left(
        \delta-1 \colon S[x,z] \to S[x,z] \right)$
    is an isomorphism.
  \end{lem}

  \begin{proof}
    Injectivity is clear, thus it remains to prove surjectivity.
    Consider $f\in\ker(\delta-1)$.
    Because $\delta-1$ preserves the degree of the polynomials,
    %(in the variables $x$ and $z$),
    we may assume $f=sx^{n}z^{m}$ for
    $s\in S$ and $n,m\in\NN$. Now compute
    \begin{equation*}
    \begin{split}
      0=\left(\delta-1\right)\left(sx^{n}z^{m}\right)
      &=sx^{n}\left(\delta\left(z^{m}\right)-z^{m}\right) \\
      &=sx^{n}\left(\left(ux+z\right)^{m}-z^{m}\right) \\
      &=sx^{n}\left(\sum_{v=0}^{m}\binom{m}{v}(ux)^{m-v}z^{v}-z^{m}\right) \\
      &=sx^{n}\left(\sum_{v=0}^{m-1}\binom{m}{v}(ux)^{m-v}z^{v}\right) \\
      &=\left(\sum_{v=0}^{m-1}\left(\binom{m}{v}su^{m-v}\right)x^{m-v+n}z^{v}\right).
    \end{split}
    \end{equation*}
    This implies $\binom{m}{v}su^{m-v}=0$ for all $v=0,\dots,m-1$.
    Because $S$ is of characteristic zero and $u$ is not a zero-divisor, this gives
    $m=0$. That is $f=sx^{n}\in S[x]$.
  \end{proof}

  \begin{lem}\label{prop:curiouscomplex-defn-Cbullet-etacomplex-exact-2}
    Given $g\in S\left[px,pz\right]$, there exists an
    $f\in S[x,z]$ such that $\left(\delta-1\right)(f)=ux\cdot g$
  \end{lem}

  \begin{proof}
    First, we prove the following fact:
    \begin{equation}\label{eq:Ho1Cbullet-is-torsion-neededfact}
      (n+1)!uxz^{n}\in\im\left(\delta-1\right) \text{ for all } n\geq0.
    \end{equation}
    We proceed by induction along $n$.
    The case $n=0$ is clear as $\left(\delta-1\right)(z)=ux$.  
    Now assume~(\ref{eq:Ho1Cbullet-is-torsion-neededfact}) is known for
    $n$; this implies
    \begin{equation*}
    \begin{split}
      \left(\delta-1\right)\left(n!z^{n+1}\right)
      =\left(n+1\right)!uxz^{n}
      +\sum_{i=1}^{n-1}
      \underbrace{\binom{n+1}{i}\frac{n!}{(i+1)!}(ux)^{(n+1)-i}(i+1)z^{i}}_{\in\im\left(\delta-1\right) \text{ for all } i=1,\dots,n-1}.
    \end{split}
    \end{equation*}
    In particular, $(n+1)!uxz^{n}$ is in the image of $\delta-1$.
    This proves~(\ref{eq:Ho1Cbullet-is-torsion-neededfact}).
    
    To prove the lemma,
    we may assume $g=\left(pz\right)^{n}$ because $\delta-1$ is $S\left[px\right]$-linear.
    \begin{equation*}
      \left(\delta-1\right)(z)
      =ux=ux \cdot \left(pz\right)^{0}
    \end{equation*}
    verifies the case $n=0$.
    It remains to consider $n\geq1$.
    But then we apply~(\ref{eq:Ho1Cbullet-is-torsion-neededfact}) to find
    \begin{equation*}
      ux \cdot \left(pz\right)^{n}
      =\frac{p^{n}}{(n+1)!}\left(\left(n+1\right)! ux z^{n}\right)
      \in\im\left(\delta-1\right).
    \end{equation*}
    Here we are also using that $p^{n}/(n+1)!$ is a $p$-adic integer.
    \end{proof}

\begin{prop}\label{prop:curiouscomplex-defn-Cbullet-etacomplex-exact}
  The complex $\eta_{\left(\zeta_{p}-1\right)x}H^{\bullet}=\eta_{ux}H^{\bullet}$ is exact.
\end{prop}

\begin{proof}
  Lemma~\ref{prop:curiouscomplex-defn-Cbullet-etacomplex-exact-0}
  says $\Ho^{0}\left(H^{\bullet}\right)=0$,
  Lemma~\ref{prop:curiouscomplex-defn-Cbullet-etacomplex-exact-1}
  implies $\Ho^{1}\left(H^{\bullet}\right)=0$, and
  Lemma~\ref{prop:curiouscomplex-defn-Cbullet-etacomplex-exact-2}
  gives $ux\cdot\Ho^{2}\left(H^{\bullet}\right)=0$.
  Thus Proposition~\ref{prop:curiouscomplex-defn-Cbullet-etacomplex-exact} follows
  from Lemma~\ref{lem:eta-operator-kills-torsion}
\end{proof}

%%%%%%%%%%%%%%%%%%%%%%%%%%%%%%%%%%%%%%%%%%%%%%%%%%%%%%%%%%%%
%%%%%%%%%%%%%%%%%%%%%%%%%%%%%%%%%%%%%%%%%%%%%%%%%%%%%%%%%%%%
% A curious complex
%%%%%%%%%%%%%%%%%%%%%%%%%%%%%%%%%%%%%%%%%%%%%%%%%%%%%%%%%%%%
%%%%%%%%%%%%%%%%%%%%%%%%%%%%%%%%%%%%%%%%%%%%%%%%%%%%%%%%%%%%

\subsection{Proof of Proposition~\ref{prop:cont-group-coh-OBlatTTKd}}\label{subsubsec:initialcomputation-finallytheproof}
  
    By Lemma~\ref{lem:cont-group-coh-OBlatTTKd-Dbullet-reducetotildeE},
    we have to check that the
    $\eta_{\mu/p^{q}}\widetilde{E}^{>q,\bullet}$ are strictly exact for all $q\in\NN_{\geq 2}$.
    To do this, we fix such $q$. Lemma~\ref{lem:gretatEgreaterthanqbullet-isetaGbullet-conditionsatisfied}
    and~\ref{lem:decalagefilteredmodules-defined} imply that the filtration topology on
    $\eta_{\mu/p^{q}}\widetilde{E}^{>q,\bullet}$ is separated and complete.
    Therefore~\cite[Chapter I, \S 4.2 page 31-32, Theorem 4(5)]{HuishiOystaeyen1996} applies:
    thus it remains to check that $\gr\eta_{\mu/p^{q}}\widetilde{E}^{>q,\bullet}$ is exact.
    By Corollary~\ref{cor:gretatEgreaterthanqbullet-isetaGbullet}, we have to prove that
    $\eta_{\sigma\left(\mu/p^{q}\right)}G^{\bullet}$ is exact. The last sentence of
    Lemma~\ref{lem:initialcomputation-describe-Gbullet} implies that this complex coincides with
    $\eta_{\left(\zeta_{p}-1\right)x}G^{\bullet}$. It is exact by Lemma~\ref{lem:etaC-qiso-etatC}
    and Proposition~\ref{prop:curiouscomplex-defn-Cbullet-etacomplex-exact}.
    \qed

%%%%%% comment ends

%%%%%%%%%%%%%%%%%%%%%%%%%%%%%%%%%%%%%%%%%%%%%%%%%%%%%%%%%%%%
%%%%%%%%%%%%%%%%%%%%%%%%%%%%%%%%%%%%%%%%%%%%%%%%%%%%%%%%%%%%
% Cech cohomology II
%%%%%%%%%%%%%%%%%%%%%%%%%%%%%%%%%%%%%%%%%%%%%%%%%%%%%%%%%%%%
%%%%%%%%%%%%%%%%%%%%%%%%%%%%%%%%%%%%%%%%%%%%%%%%%%%%%%%%%%%%

\section{On the cohomology over $X_{C}$}

%%%%%%%%%%%%%%%%%%%%%%%%%%%%%%%%%%%%%%%%%%%%%%%%%%%%%%%%%%%%
% Cech cohomology II
%%%%%%%%%%%%%%%%%%%%%%%%%%%%%%%%%%%%%%%%%%%%%%%%%%%%%%%%%%%%

\subsection{Preparations}

We continue to fix a compatible system $1,\zeta_{p},\zeta_{p^{2}},\dots\in C$
of primitive $p$th roots of unity and a profinite set $S$.
Next, we recall the $\ZZ_{p}(1)^{d}$-action on
$\widetilde{\TT}_{C}^{d}$ as in Lemma~\ref{lem:ZpGaloiscoverings}.
Concretely, we fix a $\ZZ$-basis $\gamma_{1},\dots,\gamma_{d}$ of $\ZZ_{p}(1)^{d}\cong\ZZ_{p}^{d}$
such that
\begin{equation*}
  \gamma_{i}\cdot T_{1}^{e_{1}}\cdots T_{d}^{e_{d}}=\zeta^{e_{i}}T_{1}^{e_{1}}\cdots T_{d}^{e_{d}}
\end{equation*}
for all $i=1,\dots,d$. Here, $\zeta^{e_{i}}=\zeta_{p^{j}}^{e_{i}p^{j}}$ whenever $e_{i}p^{j}\in\ZZ$.
Bu functoriality, this gives a $\ZZ_{p}(1)^{d}$-action on $\OB_{\dR}^{\dag}\left(\widetilde{\TT}_{K}^{d}\times S\right)$.
In particular, we have, for all $i=1,\dots,d$, the automorphisms
\begin{equation*}
  \gamma_{i}\colon
  \OB_{\dR}^{\dag}\left(\widetilde{\TT}_{C}^{d}\times S\right)\to\OB_{\dR}^{\dag}\left(\widetilde{\TT}_{C}^{d}\times S\right).
\end{equation*}

In the following, the symbol $\simeq$ denotes a quasi-isomorphism.
Furthermore, we use Notation~\ref{notatio:base-change-in-proet} such that
$\BB_{\dR}^{\dag}\left(\Spa\left(C,\cal{O}_{C}\right)\times S\right)$
in Proposition~\ref{prop:cont-group-coh-OBlatTTKd-S}
below is well-defined.

\begin{prop}\label{prop:cont-group-coh-OBlatTTKd-S}
  We have the canonical quasi-isomorphism
  \begin{equation*}
    \cal{O}\left(\TT^{d}\right)\widehat{\otimes}_{k}\BB_{\dR}^{\dag}\left(\Spa\left(C,\cal{O}_{C}\right)\times S\right)
    \stackrel{\simeq}{\longrightarrow}
    \Kos_{\OB_{\dR}^{\dag}\left(\widetilde{\TT}_{C}^{d}\times S\right)}\left(\gamma_{1}-1,\dots,\gamma_{d}-1\right).
  \end{equation*}
\end{prop}

\begin{proof}  
  Recall Proposition~\ref{notation:normalOBi-reconstructionpaper}. We aim to check that the canonical morphisms
  \begin{equation*}
    \normalOB_{i-1}
    \stackrel{\simeq}{\longrightarrow}
    \Kos_{\normalOB_{d}}\left(\gamma_{i}-1,\dots,\gamma_{d}-1\right)
  \end{equation*}
  are quasi-isomorphisms for all $i=1,\dots,d$. This is more general than
  Proposition~\ref{prop:cont-group-coh-OBlatTTKd-S}, which itself is the special case
  $i=1$, cf. Remark~\ref{notation:normalOBi-reconstructionpaper}.
  We proceed via induction along $i=d,\dots,1$.
  The claim is already proven for $i=d$:
  \begin{equation*}
    \normalOB_{d-1}\stackrel{\simeq}{\longrightarrow}\Kos_{\normalOB_{d}}\left(\gamma_{d}-1\right)
  \end{equation*}
  is a quasi-isomorphism by Proposition~\ref{prop:cont-group-coh-OBlatTTKd}.
  Now consider the induction hypothesis for $i+1$. This gives the second quasi-isomorphism here:
  \begin{equation*}
  \begin{split}
    \normalOB_{i-1}
    &\simeq\cone\left(\normalOB_{i}[-1]\stackrel{\gamma_{i}-1}{\longrightarrow}\normalOB_{i}\right) \\
    &\simeq\cone\left(
    \Kos_{\normalOB_{d}}\left(\gamma_{i+1}-1,\dots,\gamma_{d}-1\right)
    \stackrel{\gamma_{i}-1}{\longrightarrow}
    \Kos_{\normalOB_{d}}\left(\gamma_{i+1}-1,\dots,\gamma_{d}-1\right)
    \right) \\
    &=\Kos_{\normalOB_{d}}\left(\gamma_{i}-1,\dots,\gamma_{d}-1\right).
  \end{split}
  \end{equation*}
  The first quasi-isomorphism in this computation comes again from
  Proposition~\ref{prop:cont-group-coh-OBlatTTKd}.
\end{proof}

One may utilise Proposition~\ref{prop:cont-group-coh-OBlatTTKd-S}
to compute the cohomology of $\OB_{\dR}^{\dag}$ over $\widetilde{\TT}_{C}^{d}\times S$.
The following Proposition~\ref{prop:OBlatT-otimes-OTOX-is-OBlaX}
will enable us to generalise to the setting of an arbitrary
smooth affinoid rigid-analytic $k$-variety $X$, equipped with a fixed étale morphism
$X\to\TT^{d}$.

\begin{prop}\label{prop:OBlatT-otimes-OTOX-is-OBlaX}
  We have the following canonical isomorphism in $\D\left(\IndBan_{k}\right)$:
  \begin{equation*}
    \OB_{\dR}^{\dag}\left( \widetilde{\TT}_{C}^{d}\times S\right)
    \widehat{\otimes}_{\cal{O}\left(\TT^{d}\right)}^{\rL}
    \cal{O}(X)
    \isomap
    \OB_{\dR}^{\dag}\left(\widetilde{X}_{C}\times S\right).
  \end{equation*}
\end{prop}

We complete the proof of Proposition~\ref{proofprop:OBlatT-otimes-OTOX-is-OBlaX}
on page~\pageref{proofprop:OBlatT-otimes-OTOX-is-OBlaX}. We start with two lemmata.

\begin{lem}\label{lem:OBlatT-otimes-OTOX-is-OBlaX-1}
  Let $Y$ denote a smooth rigid-analytic $k$-variety, equipped with an étale morphism
  $Y\to\TT^{d}$. For any affinoid perfectoid $U\in Y_{\proet}$,
  we find that the canonical morphism
  $\left(\nu^{-1}\cal{O}\right)(U)\to\OB_{\dR}^{\dag}(U)$ lifts
  to the isomorphism of $k$-ind-Banach spaces
  \begin{equation}\label{eq:OBlatT-otimes-OTOX-is-OBlaX-1-themap}
    \widehat{\cal{O}}(U)\left\<\frac{\zeta,Z_{1},\dots,Z_{d}}{p^{\infty}}\right\>\to\OB_{\dR}^{\dag,+}(U).
  \end{equation}
  where $\zeta\mapsto\xi$ and $Z_{i}\mapsto Z_{i}$ for all $i=1,\dots,d$,
  cf. Theorem~\ref{thm:cohomology-OBpdRdag-over-affperfd-reconstructionpaper}.
\end{lem}

\begin{proof}
  Construct the map in the opposite direction with Lemma~\ref{lem:imagetildenu-Ainf}.
\end{proof}

Let $B_{\dR}^{\dag}:=\BB_{\dR}^{\dag}\left(C,\cal{O}_{C}\right)$.
  
\begin{lem}\label{lem:OBlatT-otimes-OTOX-is-OBlaX-2}
  Let $Y$ denote a smooth rigid-analytic $k$-variety, equipped with an étale morphism
  $Y\to\TT^{d}$. For any affinoid perfectoid $U\in Y_{\proet}/Y_{C}$, we have the
  following canonical isomorphism %in $\D\left(\Vect_{k}^{\solid}\right)$:
  \begin{equation*}
    B_{\dR}^{\dag}\widehat{\otimes}_{B_{\dR}^{\dag,+}}^{\rL}\OB_{\dR}^{\dag,+}\left(U\right)
    \isomap\OB_{\dR}\left(U\right).
  \end{equation*}
\end{lem}  

\begin{proof}
  Write $\widehat{U}=\Spa\left(R,R^{+}\right)$ and compute
  \begin{align*}
    B_{\dR}^{\dag}\widehat{\otimes}_{B_{\dR}^{\dag,+}}^{\rL}\OB_{\dR}^{\dag,+}\left(U\right)
    &=\varinjlim_{t\times}B_{\dR}^{\dag,+}\widehat{\otimes}_{B_{\dR}^{\dag,+}}^{\rL}\OB_{\dR}^{\dag,+}\left(U\right) \\
    &\cong\varinjlim_{t\times}\OB_{\dR}^{\dag,+}\left(U\right) \\
    &\cong
      \varinjlim_{t\times}\BB_{\dR}^{\dag,+}\left(R,R^{+}\right)\left\< \frac{Z_{1},\dots,Z_{d}}{p^{\infty}}\right\> \\
     &\cong\BB_{\dR}^{\dag}\left(R,R^{+}\right)\left\< \frac{Z_{1},\dots,Z_{d}}{p^{\infty}}\right\> \\
     &\stackrel{\text{\ref{thm:cohomology-OBdRdag-over-affperfd-reconstructionpaper}}}{\cong}\OB_{\dR}^{\dag}(U).
  \end{align*}
  Here, we used the local description
  of $\OB_{\dR}^{\dag,+}$ as in Corollary~\ref{cor:localdescription-of-subsections-OBla}.
\end{proof}

\begin{proof}[Proof of Proposition~\ref{prop:OBlatT-otimes-OTOX-is-OBlaX}]\label{proofprop:OBlatT-otimes-OTOX-is-OBlaX}
  Proceed via a direct computation:
  \begin{equation*}
  \begin{split}
    &\OB_{\dR}^{\dag}\left( \widetilde{X}_{C}\times S\right)
    \widehat{\otimes}^{\rL}_{\cal{O}\left(\TT^{d}\right)}
    \cal{O}(X) \\
    &\stackrel{\text{\ref{lem:OBlatT-otimes-OTOX-is-OBlaX-2}}}{\cong}
    B_{\dR}^{\dag}\widehat{\otimes}_{B_{\dR}^{\dag,+}}^{\rL}
    \OB_{\dR}^{\dag,+}\left( \widetilde{X}_{C} \times S\right)
    \widehat{\otimes}_{\cal{O}\left(\TT^{d}\right)}^{\rL}
    \cal{O}(X) \\
    &\stackrel{\text{\ref{lem:OBlatT-otimes-OTOX-is-OBlaX-1}}}{\cong}
    B_{\dR}^{\dag}\widehat{\otimes}_{B_{\dR}^{\dag,+}}^{\rL}
    \widehat{\cal{O}}\left(\widetilde{X}_{C} \times S\right)\left\<\frac{\zeta,Z_{1},\dots,Z_{d}}{p^{\infty}}\right\>
    \widehat{\otimes}_{\cal{O}\left(\TT^{d}\right)}^{\rL}
    \cal{O}(X) \\
    &\cong
    B_{\dR}^{\dag}\widehat{\otimes}_{B_{\dR}^{\dag,+}}^{\rL}
    \left(\widehat{\cal{O}}\left(\widetilde{\TT}_{K}^{d}\times S\right)\widehat{\otimes}_{k_{0}}
    k_{0}\left\<\frac{\zeta,Z_{1},\dots,Z_{d}}{p^{\infty}}\right\>\right)
    \widehat{\otimes}_{\cal{O}\left(\TT^{d}\right)}^{\rL}
    \cal{O}(X) \\
    &\cong
    B_{\dR}^{\dag}\widehat{\otimes}_{B_{\dR}^{\dag,+}}^{\rL}
    \left(\left(\widehat{\cal{O}}\left(\widetilde{\TT}_{K}^{d}\times S\right)
    \widehat{\otimes}_{\cal{O}\left(\TT^{d}\right)}^{\rL}
    \cal{O}(X)\right)
    \widehat{\otimes}_{k_{0}}^{\rL}k_{0}\left\<\frac{\zeta,Z_{1},\dots,Z_{d}}{p^{\infty}}\right\>
    \right) \\
    &\stackrel{\text{\ref{lem:hOtildeTKdtimesShotimesOTdOX-isomaphOtildeXKtimesS-reconstructionpaper}}}{\cong}
    B_{\dR}^{\dag}\widehat{\otimes}_{B_{\dR}^{\dag,+}}^{\rL}
    \left(\widehat{\cal{O}}\left(\widetilde{X}_{K}\times S\right)
    \widehat{\otimes}_{k_{0}}^{\rL}k_{0}\left\<\frac{\zeta,Z_{1},\dots,Z_{d}}{p^{\infty}}\right\>
    \right) \\
    &\cong
    B_{\dR}^{\dag}\widehat{\otimes}_{B_{\dR}^{\dag,+}}^{\rL}
    \left(\widehat{\cal{O}}\left(\widetilde{X}_{K}\times S\right)
    \widehat{\otimes}_{k_{0}}k_{0}\left\<\frac{\zeta,Z_{1},\dots,Z_{d}}{p^{\infty}}\right\>
    \right) \\
    &\cong
    B_{\dR}^{\dag}\widehat{\otimes}_{B_{\dR}^{\dag,+}}^{\rL}
    \left(\widehat{\cal{O}}\left(\widetilde{X}_{K}\times S\right)\left\<\frac{\zeta,Z_{1},\dots,Z_{d}}{p^{\infty}}\right\>
    \right) \\
    &\stackrel{\text{\ref{lem:OBlatT-otimes-OTOX-is-OBlaX-1}}}{\cong}
    B_{\dR}^{\dag}\widehat{\otimes}_{B_{\dR}^{\dag,+}}^{\rL}\OB_{\dR}^{\dag,+}\left(\widetilde{X}_{K}\times S\right) \\
    &\stackrel{\text{\ref{lem:OBlatT-otimes-OTOX-is-OBlaX-2}}}{\cong}
    \OB_{\dR}^{\dag}\left(\widetilde{X}_{K}\times S\right).
  \end{split}
  \end{equation*}
  Here, we used Corollary~\ref{cor:tensorproductfieldexact}
  and Lemma~\ref{lem:restrictedpowerseries-from-hotimes}.
\end{proof}

In the following, we make frequent use of Notation~\ref{notatio:base-change-in-proet}
without further reference.

\begin{prop}\label{prop:cont-group-coh-OBlattX-S}
  We have the canonical quasi-isomorphism
  \begin{equation*}
    \cal{O}\left(X\right)\widehat{\otimes}_{k}\BB_{\dR}^{\dag}\left(\Spa\left(C,\cal{O}_{C}\right)\times S\right)
    \stackrel{\simeq}{\longrightarrow}
    \Kos_{\OB_{\dR}^{\dag}\left(\widetilde{X}_{C}\times S\right)}\left(\gamma_{1}-1,\dots,\gamma_{d}-1\right).
  \end{equation*}
\end{prop}

\begin{proof}
  Every $k$-ind-Banach space is flat by Corollary~\ref{cor:tensorproductfieldexact}.
  This enables us to compute
  \begin{equation*}
    \begin{split}
    &\Kos_{\OB_{\dR}^{\dag}\left(\widetilde{X}_{C} \times S\right)}\left(\gamma_{1}-1,\dots,\gamma_{d}-1\right) \\
    &\stackrel{\text{\ref{prop:OBlatT-otimes-OTOX-is-OBlaX}}}{\cong}
    \Kos_{\OB_{\dR}^{\dag}\left(\widetilde{\TT}_{C}^{d} \times S\right)}\left(\gamma_{1}-1,\dots,\gamma_{d}-1\right)
    \widehat{\otimes}_{\cal{O}\left(\TT^{d}\right)}^{\rL}\cal{O}(X) \\
    &\stackrel{\text{\ref{prop:cont-group-coh-OBlatTTKd-S}}}{\cong}
    \left(\cal{O}\left(\TT^{d}\right)\widehat{\otimes}_{k}\BB_{\dR}^{\dag}\left(\Spa\left(C,\cal{O}_{C}\right) \times S \right)\right)
    \widehat{\otimes}_{\cal{O}\left(\TT^{d}\right)}^{\rL}\cal{O}(X) \\
    &\cong
    \left(\cal{O}\left(\TT^{d}\right)\widehat{\otimes}_{k}^{\rL}\BB_{\dR}^{\dag}\left(\Spa\left(C,\cal{O}_{C}\right) \times S \right)\right)
    \widehat{\otimes}_{\cal{O}\left(\TT^{d}\right)}^{\rL}\cal{O}(X) \\
    &\cong
    \BB_{\dR}^{\dag}\left(\Spa\left(C,\cal{O}_{C}\right) \times S \right)\widehat{\otimes}_{k}\cal{O}(X).
    \end{split}
  \end{equation*}
  in the derived category of $\IndBan_{k}$.
\end{proof}

%%%%%%%%%%%%%%%%%%%%%%%%%%%%%%%%%%%%%%%%%%%%%%%%%%%%%%%%%%%%
% Cech cohomology II
%%%%%%%%%%%%%%%%%%%%%%%%%%%%%%%%%%%%%%%%%%%%%%%%%%%%%%%%%%%%

\subsection{Theorems in the solid world}
\label{subsubsec:coh-over-XC-solid-reconstructionpaper}

%Denote the solidifications of $\BB_{\dR}^{\dag}$ and
%$\OB_{\dR}^{\dag}$ by
Write $\underline{\BB}_{\dR}^{\dag}:=\underline{\BB_{\dR}^{\dag}}$
and $\underline{\OB}_{\dR}^{\dag}:=\underline{\OB_{\dR}^{\dag}}$.

\begin{thm}\label{thm:solidcont-group-coh-OBlatX}
  Let $X$ denote an arbitrary smooth rigid-analytic $k$-variety,
  equipped with an étale morphism $X\to\TT^{d}$. Then,
  we have the following canonical isomorphism in $\D\left(\Vect_{k}^{\solid}\right)$:
  \begin{equation*}
    \underline{\cal{O}}\left(X\right)\otimes_{k}^{\blacksquare}\underline{\BB}_{\dR}^{\dag}\left(\Spa\left(C,\cal{O}_{C}\right)\times S\right)
  \isomap\R\Gamma\left(X_{C} \times S,\underline{\OB}_{\dR}^{\dag}\right)
  \end{equation*}
\end{thm}

\begin{proof}
  We compute the cohomology via \v{C}ech cohomology with respect to the
  covering $\widetilde{X}_{C}\to X_{C}$. This is allowed, because the cohomology
  of $\underline{\OB}_{\dR}^{\dag}$ vanishes over objects of the form
  \begin{equation*}
    \widetilde{X}_{C}\times_{X_{C}}\dots\times_{X_{C}}\widetilde{X}_{C}
  \end{equation*}
  as they are affinoid perfectoid, cf. Theorem~\ref{cor:cohomology-solidOBdRdag-over-affperfd-reconstructionpaper}.
  This gives the first isomorphism here:
  \begin{equation}\label{eq:solidcont-group-coh-OBlatX-thecomputation}
  \begin{split}
    \R\Gamma\left(X_{C} \times S,\underline{\OB}_{\dR}^{\dag}\right)
    &\cong
    \check{C}^{\bullet}\left(\widetilde{X}_{C}\times S \to X_{C} \times S,\underline{\OB}_{\dR}^{\dag}\right) \\
    &\stackrel{\text{\ref{lem:sectionsofunderlineF}}}{\cong}
    \underline{\check{C}^{\bullet}\left(\widetilde{X}_{C}\times S \to X_{C} \times S,\OB_{\dR}^{\dag}\right)} \\
    &\stackrel{\text{\ref{cor:OBdRdag-overUtimesS-reconstructionpaper}}}{\cong}
    \R\Gamma_{\cont}\left( \ZZ_{p}(1)^{d} , \underline{\OB_{\dR}^{\dag}\left(\widetilde{X}_{C} \times S\right)} \right) \\
    &\stackrel{\text{\ref{lem:solidcontinuous-cohomology-over-ZZpd-via-Koszul}}}{\cong}
    \Kos_{\underline{\OB_{\dR}^{\dag}\left(\widetilde{X}_{C} \times S\right)}}\left(\gamma_{1}-1,\dots,\gamma_{d}-1\right) \\
    &\cong
    \underline{\Kos_{\OB_{\dR}^{\dag}\left(\widetilde{X}_{C} \times S\right)}\left(\gamma_{1}-1,\dots,\gamma_{d}-1\right)} \\ 
    &\stackrel{\text{\ref{prop:cont-group-coh-OBlattX-S}}}{\cong}
    \underline{\cal{O}\left(X\right)\widehat{\otimes}_{k}\BB_{\dR}^{\dag}\left(\Spa\left(C,\cal{O}_{C}\right)\times S\right)} \\
    &\stackrel{\text{\ref{lem:IndBan-to-Solid-canonicallystronglymonoidal-reconstructionpaper}}}{\cong}
    \underline{\cal{O}\left(X\right)}\otimes_{k}^{\blacksquare}\underline{\BB_{\dR}^{\dag}\left(\Spa\left(C,\cal{O}_{C}\right)\times S\right)} \\
    &\stackrel{\text{\ref{lem:sectionsofunderlineF}}}{\cong}
    \underline{\cal{O}}\left(X\right)\otimes_{k}^{\blacksquare}\underline{\BB}_{\dR}^{\dag}\left(\Spa\left(C,\cal{O}_{C}\right)\times S\right).
  \end{split}
  \end{equation}
  Here, we used that $\widetilde{X}_{C} \times S\to X_{C}\times S$ is a
  pro-étale $\ZZ_{p}(1)^{d}$-torsor, cf. Lemma~\ref{lem:ZpGaloiscoverings},
  and we fixed a suitable  basis $\gamma_{1},\dots,\gamma_{d}$ for $\ZZ_{p}(1)^{d}\cong\ZZ_{p}^{d}$
  as in \emph{loc. cit.}.
  In the second and last step of the computation~(\ref{eq:solidcont-group-coh-OBlatX-thecomputation}),
  we view both $\BB_{\dR}^{\dag}$ as well as $\OB_{\dR}^{\dag}$
  as sheaves on $X_{\proet,\affperfd}^{\fin}$,
  cf. Lemma~\ref{lem:sheaves-on-Xproet-and-Xproetaffperfdfin},
  and $\cal{O}$ as a sheaf on the site $X_{w}$ of affinoid subdomains of $X$ equipped with the
  weak Grothendieck topology.
  Both $X_{\proet,\affperfd}^{\fin}$
  and $X_{w}$ only admit finite coverings, thus we can apply Lemma~\ref{lem:sectionsofunderlineF}.
\end{proof}

Write $B_{\dR}^{\dag}:=\BB_{\dR}^{\dag}\left(C,\cal{O}_{C}\right)$ and denote its solidification by
$\underline{B}_{\dR}^{\dag}:=\underline{B_{\dR}^{\dag}}$.

\begin{cor}\label{thm:solidcont-group-coh-OBlatX-Spoint}
  Let $X$ denote an arbitrary smooth rigid-analytic $k$-variety,
  equipped with an étale morphism $X\to\TT^{d}$. Then,
  we have the following canonical isomorphism in $\D\left(\Vect_{k}^{\solid}\right)$:
  \begin{equation*}
    \underline{\cal{O}}\left(X\right)\otimes_{k}^{\blacksquare}\underline{B}_{\dR}^{\dag}
    \isomap\R\Gamma\left(X_{C},\underline{\OB}_{\dR}\right)
  \end{equation*}
\end{cor}

\begin{proof}
  The proof of Theorem~\ref{thm:solidcont-group-coh-OBlatX} gives
  \begin{equation*}
    \underline{\cal{O}}\left(X\right)\otimes_{k}^{\blacksquare}\underline{\BB_{\dR}^{\dag}\left(\Spa\left(C,\cal{O}_{C}\right)\times S\right)}
    \isomap\R\Gamma\left(X_{C}\times S,\underline{\OB}_{\dR}\right)
  \end{equation*}
  for an abitrary profinite set $S$. Now set $S=*$ and apply
  Theorem~\ref{thm:derivedlocalsectionsofBdRdagplus-recostructionpaper}.
\end{proof}

Denote the solidifications of $\BB_{\pdR}^{\dag}$ and
$\OB_{\pdR}^{\dag}$ by $\underline{\BB}_{\pdR}^{\dag}:=\underline{\BB_{\pdR}^{\dag}}$
and $\underline{\OB}_{\pdR}^{\dag}:=\underline{\OB_{\pdR}^{\dag}}$.

\begin{thm}\label{thm:solidcont-group-coh-OBpdRdagtTTKd}
  Let $X$ denote an arbitrary smooth rigid-analytic $k$-variety,
  equipped with an étale morphism $X\to\TT^{d}$. Then,
  we have the following canonical isomorphism in $\D\left(\Vect_{k}^{\solid}\right)$:
  \begin{equation*}
    \underline{\cal{O}}\left(X\right)\otimes_{k}^{\blacksquare}\underline{\BB}_{\pdR}^{\dag}\left(\Spa\left(C,\cal{O}_{C}\right)\times S\right)
  \isomap\R\Gamma\left(X_{C} \times S,\underline{\OB}_{\pdR}^{\dag}\right).
  \end{equation*}
\end{thm}

\begin{proof}
  We compute the cohomology via \v{C}ech cohomology with respect to the
  covering $\widetilde{X}_{C}\to X_{C}$. This is allowed, because the cohomology
  of $\underline{\OB}_{\dR}^{\dag}$ vanishes over objects of the form
  \begin{equation*}
    \widetilde{X}_{C}\times_{X_{C}}\dots\times_{X_{C}}\widetilde{X}_{C}
  \end{equation*}
  as they are affinoid perfectoid, cf. Corollary~\ref{cor:cohomology-solidOBpdRdag-over-affperfd-reconstructionpaper}.
  This gives the first isomorphism here:
  \begin{align*}
    \R\Gamma\left( X_{C} \times S,\underline{\OB}_{\pdR}^{\dag}\right)
    &\cong
    \check{C}^{\bullet}\left(\widetilde{X}_{C}\times S \to X_{C} \times S,\underline{\OB}_{\pdR}^{\dag}\right) \\
    &\stackrel{\text{\ref{lem:sectionsofunderlineF}}}{\cong}
    \underline{\check{C}^{\bullet}\left(\widetilde{X}_{C}\times S \to X_{C} \times S,\OB_{\pdR}^{\dag}\right)} \\
    &\stackrel{\text{\ref{cor:cohomology-solidOBpdRdag-over-affperfd-reconstructionpaper}}}{\cong}
    \R\Gamma_{\cont}\left( \ZZ_{p}(1)^{d} , \underline{\OB_{\pdR}^{\dag}\left( \widetilde{X}_{C} \times S\right)} \right) \\
    &\stackrel{\text{\ref{lem:solidcontinuous-cohomology-over-ZZpd-via-Koszul}}}{\cong}
    \Kos_{\underline{\OB_{\pdR}^{\dag}\left(\widetilde{X}_{C} \times S\right)}}\left(\gamma_{1}-1,\dots,\gamma_{d}-1\right) \\
    &\stackrel{\text{\ref{lem:solidOBpdRdag-directsum-of-solidOBdRdag}}}{\cong}
    \bigoplus_{\alpha\geq0}\underline{\Kos_{\OB_{\dR}^{\dag}\left(\widetilde{X}_{C} \times S\right)}\left(\gamma_{1}-1,\dots,\gamma_{d}-1\right)}\left(\log t\right)^{\alpha} \\ 
    &\cong
    \bigoplus_{\alpha\geq0}\R\Gamma\left(\widetilde{X}_{C} \times S,\underline{\OB}_{\dR}^{\dag}\right)\left(\log t\right)^{\alpha} \\
    &\stackrel{\text{\ref{thm:solidcont-group-coh-OBlatX}}}{\cong}
    \bigoplus_{\alpha\geq0}\underline{\cal{O}}\left(X\right)\otimes_{k}^{\blacksquare}\underline{\BB}_{\dR}^{\dag}\left(\Spa\left(C,\cal{O}_{C}\right)\times S\right)\left(\log t\right)^{\alpha} \\
    &\cong
    \underline{\cal{O}}\left(X\right)\otimes_{k}^{\blacksquare}\bigoplus_{\alpha\geq0}\underline{\BB}_{\dR}^{\dag}\left(\Spa\left(C,\cal{O}_{C}\right)\times S\right)\left(\log t\right)^{\alpha} \\
    &\stackrel{\text{(\ref{lem:describeBpdRdagoverXC-intererstingcomputation})}}{\cong}
    \underline{\cal{O}}\left(X\right)\otimes_{k}^{\blacksquare}\underline{\BB}_{\pdR}^{\dag}\left(\Spa\left(C,\cal{O}_{C}\right)\times S\right).
  \end{align*}  
  Here, we used that $\widetilde{X}_{C} \times S\to X_{C}\times S$ is a
  pro-étale $\ZZ_{p}(1)^{d}$-torsor, cf. Lemma~\ref{lem:ZpGaloiscoverings},
  and we fixed a suitable  basis $\gamma_{1},\dots,\gamma_{d}$ for $\ZZ_{p}(1)^{d}\cong\ZZ_{p}^{d}$
  as in \emph{loc. cit.}.
  In the second step of the computation~(\ref{eq:solidcont-group-coh-OBlatX-thecomputation}),
  we view $\OB_{\pdR}^{\dag}$ as a sheaf
  on $X_{\proet,\affperfd}^{\fin}$,
  cf. Lemma~\ref{lem:sheaves-on-Xproet-and-Xproetaffperfdfin}.
  This site only admits finite coverings, thus we are allowed apply Lemma~\ref{lem:sectionsofunderlineF}.
  Finally, see the proof of Lemma~\ref{lem:describeBpdRdagoverXC}
  for~(\ref{lem:describeBpdRdagoverXC-intererstingcomputation}).  
\end{proof}

Let $B_{\pdR}^{\dag}:=\BB_{\pdR}^{\dag}\left(C,\cal{O}_{C}\right)$ and denote its solidification by
$\underline{B}_{\pdR}^{\dag}:=\underline{B_{\pdR}^{\dag}}$.

\begin{cor}\label{thm:solidcont-group-coh-OBpdRdagtX-Spoint}
  Let $X$ denote an arbitrary smooth rigid-analytic $k$-variety,
  equipped with an étale morphism $X\to\TT^{d}$. Then,
  we have the following canonical isomorphism in $\D\left(\Vect_{k}^{\solid}\right)$:
  \begin{equation*}
    \underline{\cal{O}}\left(X\right)\otimes_{k}^{\blacksquare}\underline{B}_{\pdR}^{\dag}
    \isomap\R\Gamma\left(X_{C},\underline{\OB}_{\pdR}\right).
  \end{equation*}
\end{cor}

\begin{proof}
  The proof of Theorem~\ref{thm:solidcont-group-coh-OBpdRdagtTTKd} gives
  \begin{equation*}
    \underline{\cal{O}}\left(X\right)\otimes_{k}^{\blacksquare}\underline{\BB_{\pdR}^{\dag}\left(\Spa\left(C,\cal{O}_{C}\right)\times S\right)}
    \isomap\R\Gamma\left(X_{C}\times S,\underline{\OB}_{\pdR}\right)
  \end{equation*}
  for an abitrary profinite set $S$. Now set $S=*$ and apply
  Corollary~\ref{cor:derivedlocalsectionsofBpdRdag-recostructionpaper}.
\end{proof}

%%%%%%%%%%%%%%%%%%%%%%%%%%%%%%%%%%%%%%%%%%%%%%%%%%%%%%%%%%%%
% Cech cohomology II
%%%%%%%%%%%%%%%%%%%%%%%%%%%%%%%%%%%%%%%%%%%%%%%%%%%%%%%%%%%%

\subsection{Theorems in the ind-Banach world}

We deduce ind-Banach versions of the results
in \S\ref{subsubsec:coh-over-XC-solid-reconstructionpaper}.
Recall Definitions~\ref{defn:bornologicalspace-reconstructionpaper}
and~\ref{defn:bornology-countable-basis-reconstructionpaper}.

\begin{lem}\label{lem:bornological-thm:borncont-group-coh-OBlatX}
  Let $X$ denote an arbitrary smooth rigid-analytic $k$-variety,
  equipped with an étale morphism $X\to\TT^{d}$. Then, 
  for any affinoid perfectoid $U\in X_{\proet}/\widetilde{X}_{C}$
  and any affinoid $k$-algebra $A$,
  $\OB_{\dR}^{\dag}(U)$ and
  $A\widehat{\otimes}_{k}\BB_{\dR}^{\dag}\left(U\right)$
  are bornological spaces whose bornology has countable basis.
\end{lem}

\begin{proof}
  $\OB_{\dR}^{\dag}(U)$ is a bornological space whose bornology has countable basis
  by the Theorems~\ref{thm:cohomology-OBdRdag-over-affperfd-reconstructionpaper}
  and~\ref{thm:BdRdagRRplus-bornological}.
  Thanks to Theorem~\ref{thm:derivedlocalsectionsofBdRdagplus-recostructionpaper}
  and Theorem~\ref{thm:BdRdagRRplus-bornological}, it remains to check that
  $A\widehat{\otimes}_{k}-$ preserves injective maps between Banach spaces.
  But by~\cite[Proposition 10.1]{Sch02}, $A\cong c_{0}\left(\Omega\right)$ for some set $\Omega$,
  thus~\cite[Proposition 5.1.16]{benbassat2024perspectivefoundationsderivedanalytic} implies by \emph{loc. cit.}
  Remark~\ref*{rem:restrictedpowerseries-is-c0space} and
  Lemma~\ref*{lem:infiniteTatealgebrasascontractingcoproductandc0}.
\end{proof}

In the following, we work with sheaves valued in the left heart 
$\LH\left(\IndBan_{k}\right)=\IndBan_{\I\left(k\right)}$. As explained in
\S\ref{subsec:sheavesindbanspaces-reconstructionpaper},
this category has enough injectives. Thus the derived functor
in Theorem~\ref{thm:borncont-group-coh-OBlatX} below is well-defined.

\begin{thm}\label{thm:borncont-group-coh-OBlatX}
  Let $X$ denote an arbitrary smooth rigid-analytic $k$-variety,
  equipped with an étale morphism $X\to\TT^{d}$. Then,
  we have the following canonical isomorphism in $\D\left(\IndBan_{\I\left(k\right)}\right)$:
  \begin{equation*}
    \I\left(\cal{O}\left(X\right)
      \widehat{\otimes}_{k}\BB_{\dR}^{\dag}\left(\Spa\left(C,\cal{O}_{C}\right)\times S\right)\right)
    \isomap\R\Gamma\left(X_{C} \times S,\I\left(\OB_{\dR}^{\dag}\right)\right).
  \end{equation*}
  In particular,
  $\cal{O}\left(X\right)
    \widehat{\otimes}_{k}\BB_{\dR}^{\dag}\left(\Spa\left(C,\cal{O}_{C}\right)\times S\right)
    \isomap\OB_{\dR}^{\dag}\left(X_{C} \times S\right)$.
\end{thm}

\begin{proof}
  We compute the cohomology via \v{C}ech cohomology with respect to the
  covering $\widetilde{X}_{C}\to X_{C}$. This is allowed, because the cohomology
  of $\I\left(\OB_{\dR}^{\dag}\right)$ vanishes over objects of the form
  \begin{equation*}
    \widetilde{X}_{C}\times_{X_{C}}\dots\times_{X_{C}}\widetilde{X}_{C}
  \end{equation*}  
  as they are affinoid perfectoid, cf. Theorem~\ref{thm:cohomology-OBdRdag-over-affperfd-reconstructionpaper}.
  This gives the following isomorphism
  \begin{equation*}
    \R\Gamma\left(X_{C} \times S,\I\left(\OB_{\dR}^{\dag}\right)\right)
    \cong
    \check{C}^{\bullet}\left(\widetilde{X}_{C}\times S \to X_{C} \times S,\I\left(\OB_{\dR}^{\dag}\right)\right).
  \end{equation*}
  Consequently, it remains to check that
  \begin{equation*}
    \cal{O}\left(X\right)
      \widehat{\otimes}_{k}\BB_{\dR}^{\dag}\left(\Spa\left(C,\cal{O}_{C}\right)\times S\right)
      \to
      \check{C}^{\bullet}\left(\widetilde{X}_{C}\times S \to X_{C} \times S,\OB_{\dR}^{\dag}\right)
  \end{equation*}
  is a quasi-isomorphism, cf. Lemma~\ref{lem:LH-vs-H-reconstructionpaper}.
  But we know from the proof of Theorem~\ref{thm:solidcont-group-coh-OBlatX} that
  \begin{equation*}
    \underline{\cal{O}\left(X\right)
      \widehat{\otimes}_{k}\BB_{\dR}^{\dag}\left(\Spa\left(C,\cal{O}_{C}\right)\times S\right)}
      \to
      \underline{\check{C}^{\bullet}\left(\widetilde{X}_{C}\times S \to X_{C} \times S,\OB_{\dR}^{\dag}\right)}
  \end{equation*}
  is a quasi-isomorphism.
  By Lemma~\ref{lem:bornological-thm:borncont-group-coh-OBlatX}, we can now apply
  Proposition~\ref{prop:solidexact-implies-indBanachstrictlyexact-reconstructionpaper}.
\end{proof}

\begin{cor}\label{thm:borncont-group-coh-OBlatX-Spoint}
  Let $X$ denote an arbitrary smooth rigid-analytic $k$-variety,
  equipped with an étale morphism $X\to\TT^{d}$. Then,
  we have the following canonical isomorphism in $\D\left(\IndBan_{\I\left(k\right)}\right)$:
  \begin{equation*}
    \I\left(\cal{O}\left(X\right)
      \widehat{\otimes}_{k}B_{\dR}^{\dag}\right)
    \isomap\R\Gamma\left(X_{C},\I\left(\OB_{\dR}^{\dag}\right)\right).
  \end{equation*}
  In particular,
  $\cal{O}\left(X\right)\widehat{\otimes}_{k}B_{\dR}^{\dag}
    \isomap\OB_{\dR}^{\dag}\left(X_{C}\right)$.
\end{cor}

\begin{proof}
  This is Theorem~\ref{thm:borncont-group-coh-OBlatX} for $S=*$.
\end{proof}

\begin{lem}\label{lem:bornological-thm:borncont-group-coh-OBpdRdag}
  Let $X$ denote an arbitrary smooth rigid-analytic $k$-variety,
  equipped with an étale morphism $X\to\TT^{d}$. Then, 
  for any affinoid perfectoid $U\in X_{\proet}/\widetilde{X}_{C}$
  and any affinoid $k$-algebra $A$,
  $\OB_{\pdR}^{\dag}(U)$ and
  $A\widehat{\otimes}_{k}\BB_{\pdR}^{\dag}\left(U\right)$
  are bornological spaces whose bornology has countable basis.
\end{lem}

\begin{proof}
  $\OB_{\pdR}^{\dag}(U)$ is a bornological space whose bornology has countable basis
  by the Theorems~\ref{thm:cohomology-OBpdRdag-over-affperfd-reconstructionpaper}
  and~\ref{lem:BpdRdag-bornology-countable-basis-reconstructionpaper}.
  Thanks to Corollary~\ref{cor:derivedlocalsectionsofBpdRdag-recostructionpaper}
  and Theorem~\ref{lem:BpdRdag-bornology-countable-basis-reconstructionpaper}, it remains to check that
  $A\widehat{\otimes}_{k}-$ preserves injective maps between Banach spaces.
  For this, proceed as in the proof of Lemma~\ref{lem:bornological-thm:borncont-group-coh-OBlatX}.
\end{proof}

\begin{thm}\label{thm:borncont-group-coh-OBpdRdagtTTKd}
  Let $X$ denote an arbitrary smooth rigid-analytic $k$-variety,
  equipped with an étale morphism $X\to\TT^{d}$. Then,
  we have the following canonical isomorphism in $\D\left(\IndBan_{\I\left(k\right)}\right)$:
  \begin{equation*}
    \I\left(\cal{O}\left(X\right)\widehat{\otimes}_{k}\BB_{\pdR}^{\dag}\left(\Spa\left(C,\cal{O}_{C}\right)\times S\right)\right)
  \isomap\R\Gamma\left(X_{C} \times S,\I\left(\OB_{\pdR}^{\dag}\right)\right).
  \end{equation*}
  In particular,
  $\cal{O}\left(X\right)
    \widehat{\otimes}_{k}\BB_{\pdR}^{\dag}\left(\Spa\left(C,\cal{O}_{C}\right)\times S\right)
    \isomap\OB_{\pdR}^{\dag}\left(X_{C} \times S\right)$.
\end{thm}

\begin{proof}
  We compute the cohomology via \v{C}ech cohomology with respect to the
  covering $\widetilde{X}_{C}\to X_{C}$. This is allowed, because the cohomology
  of $\I\left(\OB_{\dR}^{\dag}\right)$ vanishes over objects of the form
  \begin{equation*}
    \widetilde{X}_{C}\times_{X_{C}}\dots\times_{X_{C}}\widetilde{X}_{C}
  \end{equation*}
  as they are affinoid perfectoid, cf. Theorem~\ref{thm:cohomology-OBdRdag-over-affperfd-reconstructionpaper}.
  This gives the following isomorphism
  \begin{equation*}
    \R\Gamma\left(X_{C} \times S,\I\left(\OB_{\pdR}^{\dag}\right)\right)
    \cong
    \check{C}^{\bullet}\left(\widetilde{X}_{C}\times S \to X_{C} \times S,\I\left(\OB_{\pdR}^{\dag}\right)\right).
  \end{equation*}
  Consequently, it remains to check that
  \begin{equation*}
    \cal{O}\left(X\right)
      \widehat{\otimes}_{k}\BB_{\pdR}^{\dag}\left(\Spa\left(C,\cal{O}_{C}\right)\times S\right)
      \to
      \check{C}^{\bullet}\left(\widetilde{X}_{C}\times S \to X_{C} \times S,\OB_{\pdR}^{\dag}\right)
  \end{equation*}
  is a quasi-isomorphism, cf. Lemma~\ref{lem:LH-vs-H-reconstructionpaper}.
  But we know from the proof of Theorem~\ref{thm:solidcont-group-coh-OBpdRdagtTTKd} that
  \begin{equation*}
    \underline{\cal{O}\left(X\right)
      \widehat{\otimes}_{k}\BB_{\pdR}^{\dag}\left(\Spa\left(C,\cal{O}_{C}\right)\times S\right)}
      \to
      \underline{\check{C}^{\bullet}\left(\widetilde{X}_{C}\times S \to X_{C} \times S,\OB_{\pdR}^{\dag}\right)}
  \end{equation*}
  is a quasi-isomorphism.
  By Lemma~\ref{lem:bornological-thm:borncont-group-coh-OBpdRdag}, we can now apply
  Proposition~\ref{prop:solidexact-implies-indBanachstrictlyexact-reconstructionpaper}.
\end{proof}

\begin{cor}\label{thm:borncont-group-coh-OBpdRdagtX-Spoint}
  Let $X$ denote an arbitrary smooth rigid-analytic $k$-variety,
  equipped with an étale morphism $X\to\TT^{d}$. Then,
  we have the following canonical isomorphism in $\D\left(\IndBan_{\I\left(k\right)}\right)$:
  \begin{equation*}
    \I\left(\cal{O}\left(X\right)\widehat{\otimes}_{k}B_{\pdR}^{\dag}\right)
  \isomap\R\Gamma\left(X_{C},\I\left(\OB_{\pdR}^{\dag}\right)\right).
  \end{equation*}
  In particular,
  $\cal{O}\left(X\right)\widehat{\otimes}_{k}B_{\pdR}^{\dag}
    \isomap\OB_{\pdR}^{\dag}\left(X_{C}\right)$.
\end{cor}

\begin{proof}
  This is Theorem~\ref{thm:borncont-group-coh-OBpdRdagtTTKd} for $S=*$.
\end{proof}

%%%%%%%%%%%%%%%%%%%%%%%%%%%%%%%%%%%%%%%%%%%%%%%%%%%%%%%%%%%%
%%%%%%%%%%%%%%%%%%%%%%%%%%%%%%%%%%%%%%%%%%%%%%%%%%%%%%%%%%%%
% Cech cohomology II
%%%%%%%%%%%%%%%%%%%%%%%%%%%%%%%%%%%%%%%%%%%%%%%%%%%%%%%%%%%%
%%%%%%%%%%%%%%%%%%%%%%%%%%%%%%%%%%%%%%%%%%%%%%%%%%%%%%%%%%%%

\section{On the cohomology over $X$}

Let $X$ denote an arbitrary smooth rigid-analytic $k$-variety.

%%%%%%%%%%%%%%%%%%%%%%%%%%%%%%%%%%%%%%%%%%%%%%%%%%%%%%%%%%%%
% Cech cohomology II
%%%%%%%%%%%%%%%%%%%%%%%%%%%%%%%%%%%%%%%%%%%%%%%%%%%%%%%%%%%%

\subsection{Theorems in the ind-Banach world}

In the following, we consider the adjunct
\begin{equation}\label{eq:O-to-nustar-OBdRdag}
  \cal{O}\to\nu_{*}\OB_{\dR}^{\dag,+}
\end{equation}
of the canonical map $\nu^{-1}\cal{O}\to\OB_{\dR}^{\dag,+}$.

\begin{thm}\label{thm:borncohomologyOBdRdagX-reconstructionpaper}
  Suppose $X$ is affinoid and admits an étale morphism $X\to\TT^{d}$.
\begin{itemize}
  \item[(i)] (\ref{eq:O-to-nustar-OBdRdag}) induces the canonical isomorphism
    $\cal{O}(X)\isomap \Ho^{0}\left(X,\OB_{\dR}^{\dag}\right)$.
  \item[(ii)] $\Ho^{1}\left(X,\I\left(\OB_{\dR}^{\dag}\right)\right)\cong\I\left(\cal{O}(X)\right)$, and
  \item[(iii)] $\Ho^{i}\left(X,\I\left(\OB_{\dR}^{\dag}\right)\right)=0$ for all $i\geq2$.
\end{itemize}
\end{thm}

\begin{proof}
  We compute the cohomology via \v{C}ech cohomology with respect to the
  covering $X_{C}\to X$. This is allowed, because the cohomology
  of $\I\left(\OB_{\dR}^{\dag}\right)$ vanishes over objects of the form
  \begin{equation*}
    X_{C}\times_{X}\dots\times_{X}X_{C} \cong X_{C}\times\cal{G}^{j}
  \end{equation*}
  where $\cal{G}:=\Gal\left(\overline{k}/k\right)$,
  cf. Corollary~\ref{thm:borncont-group-coh-OBlatX-Spoint}.
  It follows that each $\Ho^{i}\left(X,\I\left(\OB_{\dR}^{\dag}\right)\right)$
  is the $i$th cohomology of the \v{C}ech complex
  $\check{C}^{\bullet}\left(X_{C}\to X , \I\left(\OB_{\dR}^{\dag}\right) \right)$.
  By Lemma~\ref{lem:LH-vs-H-reconstructionpaper},
  this coincides with the $i$th cohomology of
  \begin{equation}\label{eq:borncohomologyOBdRdagX-computation-reconstructionpaper}
  \begin{split}
    &\check{C}^{\bullet}\left(X_{C}\to X , \OB_{\dR}^{\dag} \right) \\
    &\cong\left(
      \OB_{\dR}^{\dag}\left(X_{C}\right)
      \to\OB_{\dR}^{\dag}\left(X_{C}\times\cal{G}\right)
      \to\dots
    \right) \\
    &\stackrel{\text{\ref{thm:borncont-group-coh-OBlatX}}}{\cong}
    \left(
      \cal{O}\left(X\right)
        \widehat{\otimes}_{k}\BB_{\dR}^{\dag}\left(\Spa\left(C,\cal{O}_{C}\right)\right)
      \to
      \cal{O}\left(X\right)
        \widehat{\otimes}_{k}\BB_{\dR}^{\dag}\left(\Spa\left(C,\cal{O}_{C}\right)\times\cal{G}\right)
      \to\dots
    \right) \\
    &\simeq
    \cal{O}\left(X\right)\widehat{\otimes}_{k}
      \left(
      \BB_{\dR}^{\dag}\left(\Spa\left(C,\cal{O}_{C}\right)\right)
      \to
      \BB_{\dR}^{\dag}\left(\Spa\left(C,\cal{O}_{C}\right)\times\cal{G}\right)
      \to\dots
    \right) \\
    &\stackrel{\text{\ref{cor:BdRdaggerUtimesS-isomapHomcontSAdRgreaterthanqU-reconstructionpaper}}}{\simeq} 
      \cal{O}\left(X\right)\widehat{\otimes}_{k}
      \R\Gamma_{\cont}\left(\cal{G},B_{\dR}^{\dag}\right)
  \end{split}
  \end{equation}
  Here, $\simeq$ refers to a quasi-isomorphism. This quasi-isomorphism
  in the previous computation does indeed exist, as Corollary~\ref{cor:tensorproductfieldexact}
  gives the flatness of $\cal{O}(X)$ with respect to $\widehat{\otimes}$.
  From the discussion above and~(\ref{eq:borncohomologyOBdRdagX-computation-reconstructionpaper}),
  we now get
  \begin{equation*}
    \R\Gamma\left(X,\I\left(\OB_{\dR}^{\dag}\right)\right)
    \cong
    \I\left(\cal{O}\left(X\right)\right)\widehat{\otimes}_{\I\left(k\right)}
      \R\Gamma_{\cont}\left(\cal{G},B_{\dR}^{\dag}\right).
  \end{equation*}
  Using again the flatness of $\cal{O}\left(X\right)$,
  we apply Theorem~\ref{thm:galois-cohomology-of-solidBdRdagger-born} to finish the proof.
\end{proof}

\begin{cor}\label{cor:borncohomologyOBdRdagX-reconstructionpaper}
Let $X$ denote an arbitrary smooth rigid-analytic $k$-variety.
\begin{itemize}
  \item[(i)] The canonical morphism $\cal{O}\isomap\nu_{*}\OB_{\dR}^{\dag}$,
  cf.~(\ref{eq:O-to-nustar-OBdRdag}), is an isomorphism,
  \item[(ii)] $\R^{1}\nu_{*}\I\left(\OB_{\dR}^{\dag}\right)\cong\I\left(\cal{O}\right)$, and
  \item[(iii)] $\R^{i}\nu_{*}\I\left(\OB_{\dR}^{\dag}\right)=0$ for all $i\geq2$.
\end{itemize}
\end{cor}

\begin{proof}
  We may assume that $X$ is affinoid and equipped with an étale morphism
  $X\to\TT^{d}$. For all $i\in\NN$,
  $\R^{i}\nu_{*}\I\left(\OB_{\dR}^{\dag}\right)$ is the sheafification of the presheaf
  \begin{equation*}
    U\mapsto\Ho^{i}\left(U,\I\left(\OB_{\dR}^{\dag}\right)\right),
  \end{equation*}
  thus everything follows from Theorem~\ref{thm:borncohomologyOBdRdagX-reconstructionpaper}.
\end{proof}

Compose~(\ref{eq:O-to-nustar-OBdRdag})
with the canonical map $\nu_{*}\OB_{\dR}^{\dag}\to\nu_{*}\OB_{\pdR}^{\dag}$ to get
\begin{equation}\label{eq:O-to-nustar-OBpdRdag}
  \cal{O}\to\nu_{*}\OB_{\pdR}^{\dag}.
\end{equation}

\begin{thm}\label{thm:borncohomologyOBpdRdagX-reconstructionpaper}
  Suppose $X$ is affinoid and admits an étale morphism $X\to\TT^{d}$.
  Then the canonical morphism
  $\I\left(\cal{O}(X)\right)\isomap \R\Gamma\left(X,\I\left(\OB_{\pdR}^{\dag}\right)\right)$
  induced by~(\ref{eq:O-to-nustar-OBpdRdag}) is an isomorphism
  in $\D\left(\IndBan_{k}\right)$.
\end{thm}

\begin{proof}
  Proceed as in the proof of Theorem~\ref{thm:borncohomologyOBdRdagX-reconstructionpaper}
  to deduce
  \begin{equation*}
    \R\Gamma\left(X,\I\left(\OB_{\pdR}^{\dag}\right)\right)
    \cong
    \I\left(\cal{O}\left(X\right)\right)\widehat{\otimes}_{\I\left(k\right)}
      \R\Gamma_{\cont}\left(\cal{G},B_{\pdR}^{\dag}\right),
  \end{equation*}
  but apply Theorem~\ref{thm:borncont-group-coh-OBpdRdagtTTKd}
  instead of Theorem~\ref{thm:borncont-group-coh-OBlatX},
  and Corollary~\ref{cor:BpdRdaggerUtimesS-isomapHomcontSAdRgreaterthanqU-reconstructionpaper}
  instead of Corollary~\ref{cor:BdRdaggerUtimesS-isomapHomcontSAdRgreaterthanqU-reconstructionpaper}.
  By the flatness of $\cal{O}\left(X\right)$,
  cf. the proof of Theorem~\ref{thm:borncohomologyOBdRdagX-reconstructionpaper},
  it remains to apply Theorem~\ref{thm:galois-cohomology-of-BpdRdagger}.
\end{proof}

\begin{cor}\label{cor:borncohomologyOBpdRdagX-reconstructionpaper}
  Let $X$ denote an arbitrary smooth rigid-analytic $k$-variety.
  The canonical morphism $\I\left(\cal{O}\right)\to\R\nu_{*}\I\left(\OB_{\pdR}^{\dag}\right)$
  induced by~(\ref{eq:O-to-nustar-OBpdRdag})
  is an isomorphism.
\end{cor}

\begin{proof}
  We may assume that $X$ is affinoid and equipped with an étale morphism
  $X\to\TT^{d}$. For all $i\in\NN$,
  $\R^{i}\nu_{*}\I\left(\OB_{\pdR}^{\dag}\right)$ is the sheafification of the presheaf
  \begin{equation*}
    U\mapsto\Ho^{i}\left(U,\I\left(\OB_{\pdR}^{\dag}\right)\right),
  \end{equation*}
  thus everything follows from Theorem~\ref{thm:borncohomologyOBpdRdagX-reconstructionpaper}.
\end{proof}

%%%%%%%%%%%%%%%%%%%%%%%%%%%%%%%%%%%%%%%%%%%%%%%%%%%%%%%%%%%%
% Cech cohomology II
%%%%%%%%%%%%%%%%%%%%%%%%%%%%%%%%%%%%%%%%%%%%%%%%%%%%%%%%%%%%

\subsection{Theorems in the solid world}

Consider again~(\ref{eq:O-to-nustar-OBdRdag}).
Apply the solidification functor and compose with the canonical
morphism $\underline{\nu_{*}\OB_{\dR}^{\dag}}\to\nu_{*}\underline{\OB}_{\dR}^{\dag}$
to get the canonical morphism
\begin{equation}\label{eq:underlinecalO-to-nustar-underlineOBdRdag}
  \underline{\cal{O}}\to\nu_{*}\underline{\OB}_{\dR}^{\dag}.
\end{equation}

\begin{thm}\label{thm:cohomologyOBdRdagX-reconstructionpaper}
  Suppose $X$ is affinoid and admits an étale morphism $X\to\TT^{d}$.
\begin{itemize}
  \item[(i)] The canonical morphism
    $\underline{\cal{O}}(X)\isomap \Ho^{0}\left(X,\underline{\OB}_{\dR}^{\dag}\right)$
    induced by~(\ref{eq:underlinecalO-to-nustar-underlineOBdRdag}) is an isomorphism,
  \item[(ii)] $\Ho^{1}\left(X,\underline{\OB}_{\dR}^{\dag}\right)\cong\underline{\cal{O}}(X)$, and
  \item[(iii)] $\Ho^{i}\left(X,\underline{\OB}_{\dR}^{\dag}\right)=0$ for all $i\geq2$.
\end{itemize}
\end{thm}

\begin{proof}
  We compute the cohomology via \v{C}ech cohomology with respect to the
  covering $X_{C}\to X$. This is allowed, because the cohomology
  of $\underline{\OB}_{\dR}^{\dag}$ vanishes over objects of the form
  \begin{equation*}
    X_{C}\times_{X}\dots\times_{X}X_{C} \cong X_{C}\times\cal{G}^{j}
  \end{equation*}
  where $\cal{G}:=\Gal\left(\overline{k}/k\right)$,
  cf. Theorem~\ref{thm:solidcont-group-coh-OBlatX}.
  It follows that each $\Ho^{i}\left(X,\underline{\OB}_{\dR}^{\dag}\right)$
  is the $i$th cohomology of the \v{C}ech complex
    
  \begin{equation*}
    \check{C}^{\bullet}\left(X_{C}\to X , \underline{\OB}_{\dR}^{\dag} \right) \\
    \cong\left(
      \underline{\OB}_{\dR}^{\dag}\left(X_{C}\right)
      \to\underline{\OB}_{\dR}^{\dag}\left(X_{C}\times\cal{G}\right)
      \to\dots
    \right)
  \end{equation*}
  Thus Theorem~\ref{thm:cohomologyOBdRdagX-reconstructionpaper}
  follows from the proof of
  Theorem~\ref{thm:borncohomologyOBdRdagX-reconstructionpaper}
  and the exactness of the functor
  $\IndBan_{\I\left(k\right)}\to\Vect_{k}^{\solid}$, $E\mapsto\underline{E}$,
  cf. the discussion around~(\ref{eq:underlineIVisunderlineV-reconstructionpaper})
  on page~\pageref{eq:underlineIVisunderlineV-reconstructionpaper}.
%%% comment ends
\end{proof}

\begin{cor}\label{cor:cohomologyOBdRdagX-reconstructionpaper}
Let $X$ denote an arbitrary smooth rigid-analytic $k$-variety.
\begin{itemize}
  \item[(i)] The canonical morphism $\underline{\cal{O}}\isomap\nu_{*}\underline{\OB}_{\dR}^{\dag}$,
  cf.~(\ref{eq:underlinecalO-to-nustar-underlineOBdRdag}), is an isomorphism,
  \item[(ii)] $\R^{1}\nu_{*}\underline{\OB}_{\dR}^{\dag}\cong\underline{\cal{O}}$, and
  \item[(iii)] $\R^{i}\nu_{*}\underline{\OB}_{\dR}^{\dag}=0$ for all $i\geq2$.
\end{itemize}
\end{cor}

\begin{proof}
  We may assume that $X$ is affinoid and equipped with an étale morphism
  $X\to\TT^{d}$. For all $i\in\NN$,
  $\R^{i}\nu_{*}\underline{\OB}_{\dR}^{\dag}$ is the sheafification of the presheaf
  \begin{equation*}
    U\mapsto\Ho^{i}\left(U,\underline{\OB}_{\dR}^{\dag}\right),
  \end{equation*}
  thus everything follows from Theorem~\ref{thm:cohomologyOBdRdagX-reconstructionpaper}.
\end{proof}

Compose~(\ref{eq:underlinecalO-to-nustar-underlineOBdRdag})
with the canonical map $\nu_{*}\underline{\OB}_{\dR}^{\dag}\to\nu_{*}\underline{\OB}_{\pdR}^{\dag}$ to get
\begin{equation}\label{eq:underlinecalO-to-nustar-underlineOBpdRdag}
  \underline{\cal{O}}\to\nu_{*}\underline{\OB}_{\pdR}^{\dag}.
\end{equation}

\begin{thm}\label{thm:cohomologyOBpdRdagX-reconstructionpaper}
  Suppose $X$ is affinoid and admits an étale morphism $X\to\TT^{d}$.
  Then the canonical morphism
  $\underline{\cal{O}}(X)\isomap \R\Gamma\left(X,\underline{\OB}_{\pdR}^{\dag}\right)$
  induced by~(\ref{eq:underlinecalO-to-nustar-underlineOBpdRdag}) is an isomorphism
  in $\D\left(\Vect_{k}^{\solid}\right)$.
\end{thm}

\begin{proof}
  Proceed as in the proof of Theorem~\ref{thm:cohomologyOBdRdagX-reconstructionpaper}
  to deduce
  \begin{equation*}
    \R\Gamma\left(X,\underline{\OB}_{\pdR}^{\dag}\right)
    \cong
    \underline{\cal{O}}\left(X\right)\otimes_{k}^{\blacksquare}
      \R\Gamma_{\cont}\left(\cal{G},\underline{B}_{\pdR}^{\dag}\right),
  \end{equation*}
  but apply Theorem~\ref{thm:solidcont-group-coh-OBpdRdagtTTKd}
  instead of Theorem~\ref{thm:solidcont-group-coh-OBlatX},
  and Corollary~\ref{cor:BpdRdaggerUtimesS-isomapHomcontSAdRgreaterthanqU-reconstructionpaper}
  instead of Corollary~\ref{cor:BdRdaggerUtimesS-isomapHomcontSAdRgreaterthanqU-reconstructionpaper}.
  By the flatness of $\underline{\cal{O}}\left(X\right)$,
  cf. the proof of Theorem~\ref{thm:cohomologyOBdRdagX-reconstructionpaper},
  it remains to apply Theorem~\ref{thm:galois-cohomology-of-BpdRdagger}.
\end{proof}

\begin{cor}\label{cor:cohomologyOBpdRdagX-reconstructionpaper}
  Let $X$ denote an arbitrary smooth rigid-analytic $k$-variety.
  The canonical morphism $\underline{\cal{O}}\to\R\nu_{*}\underline{\OB}_{\pdR}^{\dag}$
  induced by~(\ref{eq:underlinecalO-to-nustar-underlineOBpdRdag})
  is an isomorphism.
\end{cor}

\begin{proof}
  We may assume that $X$ is affinoid and equipped with an étale morphism
  $X\to\TT^{d}$. For all $i\in\NN$,
  $\R^{i}\nu_{*}\underline{\OB}_{\pdR}^{\dag}$ is the sheafification of the presheaf
  \begin{equation*}
    U\mapsto\Ho^{i}\left(U,\underline{\OB}_{\pdR}^{\dag}\right),
  \end{equation*}
  thus everything follows from Theorem~\ref{thm:cohomologyOBpdRdagX-reconstructionpaper}.
\end{proof}

\bibliography{RecThm}
\bibliographystyle{amsplain}  %use the plain bibliography style

\end{document}